\documentclass[11pt]{amsart}

\usepackage[article]{robertpreamble}
\usepackage{geometry}
\makeindex 

\begin{document}
\title[Structure Theory of Parabolic Nodal and Singular Sets]{Structure Theory of Parabolic Nodal and Singular Sets}

\author{Max Hallgren}
\address{Department of Mathematics, Rutgers University}
\email{mh1564@scarletmail.rutgers.edu}

\author{Robert Koirala}
\address{Department of Mathematics, University of California San Diego}
\email{rkoirala@ucsd.edu}

\author{Zilu Ma}
\address{Department of Mathematics, University of Tennessee Knoxville}
\email{zma21@utk.edu}

\date{\today}
\begin{abstract}
    We establish new estimates for the size and structure of the nodal set $\{u=0\}$ and the singular set $\{u=\verts{\cd u}=0\}$ of solutions $u$ to parabolic inequalities with parabolic Lipschitz coefficients. In particular, we show that almost all of the nodal and singular sets are covered by \textit{regular} parabolic Lipschitz graphs with estimates, and that both sets satisfy parabolic Minkowski estimates depending only on a doubling quantity at a point. Many of our results are new even for the heat equation on $\mathbb{R}^{n}\times \R$. 
\end{abstract}

\maketitle 
\tableofcontents


\section{Introduction}\label{introduction}
Let $u$ be a non-trivial solution to a parabolic inequality with parabolic Lipschitz coefficients defined on an open subset of $\R^n\times \R$:
\begin{align}
   \begin{split}
        \verts{\pdt u- \div (a\cdot \cd u)}&\leq M (\verts{u}+\verts{\cd u}),\\
        (1+M)^{-1}\delta^{ij}&\le a^{ij}\leq (1+M)\delta^{ij},\\
        \verts{a^{ij}(x,t)-a^{ij}(y,s)}&\leq M \max \{\verts{x-y},\sqrt{\verts{t-s}}\},\label{eq-parabolic-equation}
    \end{split}
\end{align}
where $a^{ij}=a^{ji}$ and $M\geq 0$. 

In this paper, we study the geometric and measure-theoretic structure of the nodal and singular sets of $u$,
\begin{align}
    Z(u)\coloneqq \{u=0\},\qquad S(u)\coloneqq \{u=\verts{\cd u}=0\}.
\end{align}
Our analysis assumes one of several doubling conditions on $u$, which excludes pathological cases such as solutions that vanish identically on a time slab (see, for instance, \cite{delsanto}). Under these assumptions, we show that both $Z(u)$ and $S(u)$ can be covered, up to negligible subsets, by regular parabolic Lipschitz graphs over affine subspaces of dimensions $(n+1)$ and $n$, respectively. Furthermore, we establish sharp $n+1$ and $n$-dimensional parabolic Minkowski content estimates for $Z(u)$ and $S(u)$, respectively, depending on the doubling condition. As a consequence, we obtain new $(n-2)$-dimensional bounds for $S(u)$ on almost every time slice. Precise statements of our results are presented in Section \ref{main-results}, while Section \ref{outline-of-the-proof} provides an outline of the proofs and highlights the main new ideas.

\subsection{Background}
To place our setting in context, consider first the case $M=0$, when $u$ satisfies the heat equation
\begin{align} \label{eq:heateq}
    \partial_t u = \Delta u.
\end{align}
Such functions $u$ are called \textit{caloric.} If $u$ is globally defined and satisfies appropriate growth bounds, then $u$ is real analytic in both space and time variables. Consequently, both $Z(u)$ and $S(u)$ are real analytic subsets of $\mathbb{R}^n \times \mathbb{R}$ of dimension $n$ and $n-1$, respectively. In particular ({cf.} \cite[\S 3.4.10]{federer-2014-geometric-measure-theory}), each time slice $Z(u)\cap (\mathbb{R}^n \times \{t\})$ has locally finite $(n-1)$-dimensional Hausdorff measure for any $t\in \mathbb{R}$, and $Z(u)$ has locally finite $n$-dimensional Hausdorff measure. Similarly, $S(u)$ has locally finite $(n-1)$-dimensional Hausdorff measure.

Quantitative bounds for these measures, however, are substantially more difficult to obtain. For caloric polynomials of degree $m$, the Hausdorff measure of their nodal and singular sets grows polynomially on $m$ (see, for example, \cite{wongkew-1993-volumes}). For general solutions $u$, one introduces quantities such as the \textit{parabolic frequency} and the \textit{doubling index} (see Definition \ref{definition-energy-functionals}), which play roles analogous to the degree of a caloric polynomial. Han and Lin in \cite{han-lin-1994-nodal-sets-ii} obtained $n$-dimensional Hausdorff estimates, depending on the doubling index, for nodal sets of solutions to more general parabolic equations. In \cite{arya2025sharp, kukavicaparabolic, lin-1991-nodal-sets}, the authors established estimates for the time slices $Z(u) \cap (\mathbb{R}^n \times \{t\})$ of the nodal sets of solutions to parabolic equations with real-analytic coefficients. See \cite{arya2023space, angenent1991nodal, colding-mninicozzi-2011-lower, colding-minicozzi-2020-parabolic, cheeger-naber-valtorta-2015-critical, donnelly-fefferman-1988-nodal-sets, donnelly-fefferman-1990-nodal, dong-1992-nodal, han-hardt-lin-1998-geometric, hardt-hoffman-nadirashvili-1999-critical-sets, huang-jiang-2023-volume, han-lin-1994-nodal-sets-ii, han-lin-1004-on-the-geometric, hardt-simon-1989-nodal, hezari-sogge-2012-a-natural, kenig-zhu-zhuge-2022-doubling, lin-1991-nodal-sets, lin-liu2016-betti-numbers-of-level-sets, logunov-2018-nodal, logunov-2018-nodal-sets, logunov-malinnikova-2018-nodal-sets, logunov-malinnikova-2020-review, logunov-malinnikova-2021-the-sharp, lin-shen-2019-nodal-sets, liu-tian-yang-2024-measure, naber-valtorta-2017-volume-estimtates-of-critical-sets-of-pde, sogge-zelditch-2012-lower-bounds, zelditch-2013-eigenfunctions-and-nodal-sets, zelditch-2015-hausdorff} and the references therein for a non-exhaustive list of related works. More recently, Huang and Jiang in \cite{huang-jiang-2024nodal} extended these results to obtain Hausdorff measure estimates for the time slices of the nodal sets of general solutions to \eqref{eq-parabolic-equation}, following the elliptic framework developed by Naber and Valtorta in \cite{naber-valtorta-2017-volume-estimtates-of-critical-sets-of-pde}.

Further extensions of the work of Naber and Valtorta have remained open even in the caloric case. In the setting of \eqref{eq-parabolic-equation}, it is natural to formulate estimates using the parabolic metric $d_{\cP}$ on $\mathbb{R}^n\times \mathbb{R}$ induced by the ``norm"
\begin{align*}
    |(x,t)|^2_{\mathcal{P}}\coloneqq \max\{|x|^2 ,|t|\},
\end{align*}
which better reflects the underlying space-time geometry. With respect to $d_{\cP}$, $\R^n\times \R$ has parabolic dimension $n+2$. Until now, parabolic Minkowski estimates for the nodal and singular sets have been unavailable; in particular, even Hausdorff estimates for the singular sets were unknown.

A related question concerns parabolic rectifiability. Following \cite{mattila-2022-parabolic-rectifiability}, one may ask whether $Z(u)$ and $S(u)$ are covered, up to negligible sets, by a union of parabolic Lipschitz graphs. In the PDE literature, a stronger notion of Lipschitz property called the \textit{regular} parabolic Lipschitz graphs exists, which requires estimates on the half-time derivative. Such graphs play for parabolic problems the same role as Euclidean Lipschitz graphs play for elliptic equations. In fact, it was shown in \cite{hofmanPDE} that if $\Omega$ has boundary locally given by a Lipschitz graph, then the $L^p$ Dirichlet boundary problem is well-posed if and only if the Lipschitz graphs are regular.

In this paper, we prove that $Z(u)$ and $S(u)$ satisfy the stronger notion of rectifiability as well as the aforementioned parabolic Minkowski estimates.

\subsection{Main results}\label{main-results}

\subsubsection{General setting}
We now state our main results, first assuming that $u$ solves \eqref{eq-parabolic-equation} on $P(\mathbf{0},5)$, where 
\begin{align*}
    P(\mathbf{x},r)\coloneqq B(x,r)\times (t-r^2,t+r^2)=\{\mathbf{y}\in\mathbb{R}^{n+1}\colon |\mathbf{y}-\mathbf{x}|_{\mathcal{P}}<r\}
\end{align*}
denotes the parabolic ball centered at $\mathbf{x}=(x,t)$. For $A\subseteq \R^{n+1}$, we write $P(A,r)$ for the $r$-neighborhood of $A$ with respect to $d_{\mathcal{P}}$. We also assume that $u$ satisfies the \textit{doubling condition}: 
\begin{align}
    \Theta \coloneqq \sup_{t\in [-2,2]} \frac{\int_{B(0,5)\times [-25,t]}u^2\,d\mathcal{H}_{\mathcal{P}}^{n+2}}{\int_{B(0,\frac{1}{2})}u^2(x,t)\,dx}<\infty.\label{eq-intro-doubling-condition}
\end{align}
Asking that $\Theta$ is finite is equivalent to asking that $u$ does not vanish identically on any time slice (see \cite{AlVe,Fernandez}).

For convenience, we set
\begin{align*}
    \mathscr{C}^k(u)\coloneqq \begin{cases}
        Z(u) & \text{ when }k=n+1,\\
        S(u) & \text{ when }k=n.
    \end{cases}
\end{align*}

We first establish that the nodal and singular sets are almost everywhere covered by regular parabolic Lipschitz graphs, with quantitative control on their content and on the mean oscillation of the half-time derivative. (See Definitions~\ref{def of grassmann} and~\ref{def:regularlipschitz} for the precise notions.)
\begin{theorem} \label{thm:regularity-general}
    Suppose $u$ is a solution to \eqref{eq-parabolic-equation} on $P(\mathbf{0},5)$ satisfying \eqref{eq-intro-doubling-condition}. For any $k\in \{n,n+1\}$ and any $\epsilon\in (0,1]$, there exist points $\mathbf{x}_i \in P(\mathbf{0},\tfrac{1}{2})$, vertical planes $V_i \in \operatorname{Aff}_{\mathcal{P}}(k)$, radii $r_i \in (0,1]$, and $\epsilon$-regular parabolic Lipschitz functions $F_i\colon V_i \to V_i^{\perp}$ such that 
    \begin{align*}
        \sum_{i} r_i^{k} \leq C(\epsilon,M,\Theta),
    \end{align*}
    and the graphs $G_i \coloneqq \{ \mathbf{x}+ F_i(\mathbf{x})\colon \mathbf{x} \in V_i \cap P(\mathbf{x}_i,10 r_i)\}$ satisfy
    \begin{align*}
        \mathcal{H}_{\mathcal{P}}^{k}\left( (P(\mathbf{0},\tfrac{1}{4})\cap \mathscr{C}^k(u))\setminus \bigcup_{i}G_i \right)=0.
    \end{align*}
    Here $\cH_\cP^{k}$ denotes the $k$-dimensional parabolic Hausdorff measure induced by $d_{\mathcal{P}}$.
\end{theorem}

\begin{remark}
    \begin{enumerate}
        \item Hence $\mathscr{C}^k(u)$ is vertically $k$-rectifiable (see Definition \ref{definition-rectifiability-b}) in the sense of \cite[Definition 3.7]{mattila-2022-parabolic-rectifiability}.

        \item Regular parabolic Lipschitz graphs form a strictly smaller class than parabolic rectifiable or parabolic Lipschitz graphs. In fact, \cite[Example 8.2]{mattila-2022-parabolic-rectifiability} constructs $\mathcal{H}^{n+1}$-measurable parabolic rectifiable subsets $E \subseteq \mathbb{R}^{n+1}$ with $\mathcal{H}_{\mathcal{P}}^{n+1}(E) \in (0,\infty)$ yet $\mathcal{H}_{\mathcal{P}}^{n+1}(E\cap G) = 0$ for every regular parabolic Lipschitz graph $G$. 

        \item There is no analog of regular Lipschitz graphs in the Euclidean case: Euclidean Lipschitz functions are differentiable almost everywhere with uniformly bounded gradients, while parabolic Lipschitz functions need not admit half-time derivatives even almost everywhere. 
    \end{enumerate}
\end{remark}

As a consequence of the proof of Theorem \ref{thm:regularity-general}, we get a qualitative regularity of the nodal and singular sets of solutions to \eqref{eq-parabolic-equation}. 

\begin{corollary} \label{cor:qualitative}
    Let $u$ be a solution to \eqref{eq-parabolic-equation} on an open subset $\Omega\subset \R^n\times \R$ such that $u(\cdot,t)$ does not vanish identically on a non-empty connected component of $\Omega \cap (\R^n\times \{t\})$. Then there exists a countable collection of regular parabolic Lipschitz graphs $G_i$ such that 
    \begin{align*}
        \mathcal{H}_{\mathcal{P}}^k \left( \mathscr{C}^k(u)\setminus \bigcup_i G_i \right)=0.
    \end{align*}
    In particular, $\mathscr{C}^k(u)$ is vertically parabolic $k$-rectifiable with locally finite $\cH_{\cP}^k$ measure.
\end{corollary}

We next quantify the size of $\mathscr{C}^k$.
\begin{theorem}
\label{main-theorem-general-setting}
    Let $u$ satisfy \eqref{eq-parabolic-equation} and \eqref{eq-intro-doubling-condition}. Then, for any  $0<r\leq \ol{r}(M, \Theta)$,
    \begin{align*}
        \cH_{\cP}^{n+2}(P(\mathscr{C}^k(u),r)\cap P(\mathbf{0},\tfrac{1}{4}))&\leq  C(M,\Theta)r^{n+2-k}.
    \end{align*}
    Hence the parabolic Minkowski dimension of $\mathscr{\cC}^k$ is at most $k$.
\end{theorem}

This theorem follows from stronger volume estimates for quantitative strata ({cf.} Theorem \ref{thm: caloric vol part 2}), which are themselves parabolically rectifiable (see Definition \ref{definition-rectifiability-b}).

As a consequence of the proof of Theorem \ref{main-theorem-general-setting}, we also have estimates for almost every time slice.
\begin{corollary}\label{corollary-general-setting}
    Let $k\in \{n,n+1\}$.
    \begin{enumerate}
        \item \label{thm:mainpart3-general} For $\mathcal{H}^\frac{1}{2}$-almost every $t\in (-1,1)$, $\mathscr{C}^k_t(u)\coloneqq \mathscr{C}^k(u)\cap (B(0^n,1) \times \{t\})$ satisfies
        \begin{align*}
            \mathcal{H}^{k-1}(\mathscr{C}^k_t(u))=0.
        \end{align*}
        
        \item \label{thm:mainpart4-general} For $\mathcal{H}^{1}$-almost every $t\in (-1,1)$, $\mathscr{C}^k_t(u)$ is $(k-2)$-rectifiable, and 
        \begin{align*}
            \mathcal{H}^{k-2}(\mathscr{C}^k_t(u))\le C(M,\Theta).
        \end{align*}
    \end{enumerate}
\end{corollary}

\begin{remark} 
    The stated estimates for $S_t(u)$ cannot hold for every time slice. For instance, $u(x,t)=x^2+2t$ satisfies $S_0(u)=Z_0(u)=\{(0,0)\}$. 
\end{remark}

\subsubsection{Caloric setting}
We next state our results in the caloric setting. Here the estimates are sharper and depend only on the parabolic frequency at a point, defined using a backward heat kernel measure as follows. For a base point $\mathbf{x}_0\coloneqq (x_0,t_0)\in \mathbb{R}^n \times \mathbb{R}$ and $t<t_0$, set
\begin{align*}
    d\nu_{\mathbf{x}_0;t}(x)\coloneqq \frac{1}{(4\pi(t_0-t))^{\frac{n}{2}}} \exp \left( -\frac{|x-x_0|^2}{4(t_0-t)} \right) \,dx.
\end{align*}
When $\mathbf{x}_0=\mathbf{0}\coloneqq(0^n,0)$, we write $d\nu_t$ for short. For any $\tau\coloneqq t_0-t>0$, define the \textit{frequency} (introduced in \cite{poon-1996-unique-continuation}) by
\begin{align*}
    N_{\mathbf{x}_0}(\tau)\coloneqq \frac{2\tau \int_{\R^n} \verts{\cd u}^2 \, d\nu_{\mathbf{x}_0;t}}{\int_{\R^n} u^2 \, d\nu_{\mathbf{x}_0;t}},
\end{align*}
and write $N(\tau)\coloneqq N_{\mathbf{x}_0}(\tau)$ when $\mathbf{x}_0=\mathbf{0}=(0^n,0)$.

We assume that $u \coloneqq \mathbb{R}^n \times (t_0,t_1) \to \mathbb{R}$ is a solution to \eqref{eq:heateq} satisfying the following \textit{mild growth condition}
\begin{align} 
    \sup_{s\in [t_0+\epsilon,t_1-\epsilon]}\int_{\R^n} u^2 \,d\nu_{0,t_1;s} <\infty,\label{intro-mild-growth}
\end{align} 
for all $\epsilon>0$. 

In this setting, we get estimates for the following effective nodal and singular sets.
\begin{definition}
    For $r>0$, define
    \begin{subequations}
        \begin{align*}
            Z_r(u)&\coloneqq \left\{ \mathbf{x} \in P(\mathbf{0},1)\colon \inf_{P(\mathbf{x},\frac{s}{16})} |u|^2 \leq \frac{1}{8} \int_{\mathbb{R}^n} |u|^2 \, d\nu_{\mathbf{x};t-s^2} 
            \text{ for all }s\in[r,1]\right\},\\
            S_r(u)&\coloneqq \left\{ \mathbf{x}\in P(\mathbf{0},1)\colon  \inf_{P(\mathbf{x},\frac{s}{16})} \left(\verts{u}^2 +2s^2\verts{\cd u}^2 \right) \leq \frac{1}{16} \int_{\R^n} \verts{u}^2 \, d\nu_{\mathbf{x};t-s^2} \text{ for all }s\in[r,1]\right\}.
        \end{align*}
    \end{subequations}
\end{definition}
Observe that $Z(u)\subseteq Z_r(u)$ and $S(u)\subseteq S_r(u)$ for all $r>0$. 
We define $\mathscr{C}^k_r$ analogously.

We now state the quantitative volume estimate of effective nodal and singular sets.
\begin{theorem}\label{main-theorem-caloric-setting}
    Let $k\in \{n,n+1\}$ and let $u$ be a caloric function with $N(10^5)\leq \Lambda$. Then, for any $r \in (0,1]$,
    \begin{align*}
        \cH_{\cP}^{n+2}(P(\mathscr{C}^k(u),r)\cap P(\mathbf{0},1))\leq \cH_{\cP}^{n+2}(P(\mathscr{C}^k_r(u),r)\cap P(\mathbf{0},1))&\leq  C^{\Lambda^4}r^{n+2-k}.
    \end{align*}
\end{theorem}

\begin{remark}
    We do not expect that the stated dependence of the right hand side on $\Lambda$ is sharp. For example, in the elliptic setting, it is known that the dependence on $\Lambda$ is polynomial (see \cite{logunov-malinnikova-2020-review,logunov-malinnikova-2021-the-sharp,logunov-2018-nodal} and the references therein).
\end{remark}

We also get estimates at almost every time.
\begin{corollary}\label{theorem time slice}
    Let $k\in \{n,n+1\}$ and let $u$ be a caloric function with $N(10^5)\leq \Lambda$. 
    \begin{enumerate}
        \item \label{thm:mainpart3} For $\mathcal{H}^\frac{1}{2}$-almost every $t\in (-1,1)$, $\mathscr{C}^k_t(u)\coloneqq \mathscr{C}^k(u)\cap (B(0^n,1) \times \{t\})$ satisfies
        \begin{align*}
            \mathcal{H}^{k-1}(\mathscr{C}^k_t(u))=0.
        \end{align*}
        
        \item \label{thm:mainpart4} For $\mathcal{H}^{1}$-almost every $t\in (-1,1)$, $\mathscr{C}^k_t(u)$ is $(k-2)$-rectifiable, and 
        \begin{align*}
            \mathcal{H}^{k-2}(\mathscr{C}^k_t(u))\le C^{\Lambda^4}.
        \end{align*}
    \end{enumerate}
\end{corollary}

\subsection{Outline of the proof}\label{outline-of-the-proof}
For simplicity, we describe the proof in the caloric setting. Our approach follows the general strategy of \textit{dimension reduction}, first introduced by Federer in \cite{federer-1970-singular-sets}, later quantified using \textit{quantitative stratification} by Cheeger and Naber in \cite{cheeger-naber-2013-lower}, and further refined through the concept of \textit{neck regions} introduced in \cite{jiang-naber-2021-l2-curvature,naber-valtorta-YM}. Similar techniques have been used for other problems in both the elliptic \cite{cheeger-jiang-naber-2021-Sharp-quantitative,huang-jiang-2023-volume,naberreifenbergnotes, naber-valtorta-2024energy} and parabolic
\cite{fang-li-RF, fang-li-2025volume, fu2025stratification, gianniotis-diameter,gianniotis-L1,huang-jiang-2024nodal,huang-jiang-MCF} settings. However, our regularity result Theorem \ref{thm:regularity-general} is new in the parabolic setting, which necessitates new techniques, as explained below.

An essential ingredient in our analysis is a monotone quantity: the frequency $N_{\mathbf{x}}(r^2)$. For each base point $\mathbf{x}$, the frequency is an increasing function of the scale $r$. The monotonicity and rigidity of the frequency allows us to obtain quantitative uniqueness of tangent flows (Lemma \ref{lemma-monotonicity-formulae-for-the-energy-functionals} and Theorem \ref{theorem-almost-frequency-cone-implies-unique-geometric-cone}) and quantitative cone splitting (Theorem \ref{thm: sym split equiv}) at multiple scales and locations.

Using these tools, we obtain a \textit{finite-resolution neck decomposition} (Theorem \ref{theorem-neck-decomposition2}) of $P(\mathbf{0},1)$:
\begin{align*}
    P(\mathbf{0},1)\subseteq \bigcup_a \mathcal{N}^a \cup \bigcup_b P(\mathbf{x}_b,r_b) \cup \bigcup_f P(\mathbf{x}_f,r_f), 
\end{align*}
into $k$-\textit{neck regions} $\cN^a \subseteq P(\mathbf{x}_a,r_a)$ ({cf.} Definition \ref{definition neck region}), almost $(k+1)$-\textit{symmetric regions} $P(\mathbf{x}_b,r_b)$, and residual region $P(\mathbf{x}_r,r_f)$ with a $k$-content estimate:
\begin{align*}
    \sum_a r_a^{k} + \sum_b r_b^{k}+\sum_f r_f^{k}\leq C,
\end{align*}
where $r \leq r_a,r_b$ and $r_f\approx r$, see Theorem \ref{theorem-neck-decomposition2}. In previous settings, the residual region could be ignored by approximating solutions by solutions with empty singular sets. However, in our setting, this is impossible, so the finite-resolution neck decomposition becomes crucial for proving Minkowski estimates. To prove the content estimate, we first obtain a quantitative dimension reduction (Proposition \ref{prop: extra symmetry}) for which we have to argue at multiple scales depending on the proximity to the time slab of the symmetry plane. Using this, $\epsilon$-regularity (Proposition \ref{proposition-containment}), and the neck structure theorem (Theorem \ref{theorem-neck-structure}), we reduce the volume bound for $\mathscr{C}^k_r$ to the above content estimate, thereby proving the parabolic Minkowski estimates in the caloric case (Theorem \ref{main-theorem-caloric-setting}). Once the parabolic Minkowski estimates are established, we use a disintegration of $\cH_{\cP}^{n+2}$ (Lemma \ref{lem:fubinilike}) and the neck structure theorem to obtain Hausdorff estimates for almost every time slice (Theorem \ref{theorem time slice}). An important distinction from \cite{huang-jiang-2024nodal} is that our neck regions are constructed in space–time rather than on individual time slices. This is essential for obtaining estimates on 
$S(u)$, since such estimates fail on individual time slices.

Although our Minkowski estimates follow the overall framework of \cite{cheeger-jiang-naber-2021-Sharp-quantitative,jiang-naber-2021-l2-curvature}, the proof of Theorem \ref{thm:regularity-general} requires a fundamentally new argument. The lack of any control on the half-time derivative under a parabolic Lipschitz assumption creates a genuine parabolic obstruction with no analogue in the elliptic theory, and the existing strategies cannot be fully adapted to prove the strong neck structure theorem (Theorem \ref{thm:strongneckstructure}). This theorem will apply to necks in a refined decomposition (Theorem \ref{theorem-neck-decomposition3}):
\begin{align*}
    P(\mathbf{0},1)\subseteq \bigcup_a \mathcal{N}^a
    \cup  \bigcup_g P(\mathbf{x}_g,r_g)\cup\left( \widetilde{\mathcal{C}} \cup \bigcup_a \mathcal{C}_{a,0} \right).
\end{align*}
The packing measures of these neck regions now satisfy an Ahlfors regularity condition (Theorem \ref{theorem-neck-structure} and Proposition \ref{prop:ahlforsreg}), which guarantees that the average frequency drop with respect to the packing measures satisfies a Carleson-type estimate (Lemma \ref{lem:carlesoncondition}). To apply this, we prove a new sharp cone-splitting inequality (Theorem \ref{theorem-cone-splitting-inequality}) based on commutator identities \eqref{eq-spatial-cone-splitting-homogeneous} and \eqref{eq-temporal-cone-splitting-homogeneous} for space-time vector fields, rather than linear combinations of radial fields as in the elliptic case (\cite[Theorem 2.20]{naber-valtorta-2024energy}). Using the sharp cone splitting, we show that the parabolic Jones $\beta$-numbers at points in the center set of $\mathcal{N}^a$ are controlled by the average frequency drop (Lemma \ref{lem:beta}). Combining this with the Carleson--type estimate for the average frequency drop, we conclude that the $\beta$-numbers also satisfy a Carleson-type estimate (Lemma \ref{lem:carlesoncondition}). We then use the Carleson-type estimate for $\beta$-numbers to prove the strong neck structure Theorem (\ref{thm:strongneckstructure}). To achieve this, we use a Whitney-type covering (Lemma \ref{lem:itwasacoverup}) to construct a parabolic Lipschitz graph whose neighborhood contains the center set of $\mathcal{N}^a$ (Lemma \ref{lemma:strongneckclaim1} and Lemma \ref{lemma:newlipschitz}). A corresponding Carleson-type estimate for a variant of $\beta$-number called $\kappa$-number provides bounds on the mean oscillation of the half-time derivative of the graph (Proposition \ref{prop:criterionforregularity}). We show that such a Carleson-type estimate for the $\kappa$-numbers of our extended graph follows from the Carleson estimate for $\beta$-numbers of the packing measure (Lemma \ref{lemma:newkappa}). Combining these with the containment of the nodal and singular sets in the center set, we show that the nodal and singular sets are contained in a countable union of regular parabolic Lipschitz graphs up to negligible subsets.

Our estimates in the more general setting of parabolic inequalities are analogous to the caloric setting, but contain scale-dependent error terms. A significant new technical challenge is to ensure that these error terms are summably small across scales. While the methods of \cite{huang-jiang-2024nodal} suffice for volume estimates of nodal sets, they do not guarantee the summability needed for the strong neck structure theorem, and overcoming this requires substantially new arguments. Along the way, we also establish sharp quantitative uniqueness (Theorem \ref{theorem quantitative uniqueness in general}) via a linear stability analysis of the operator $\Delta_f+\frac{m}{2}=\Delta -\frac{x}{2\tau}\cdot \cd +\frac{m}{2}$, following ideas of Simon \cite[\S II.3]{simongeneral}. The operator $\Delta_f+\frac{m}{2}$ appears when $u$ is rescaled dynamically. See \cite[Theorem 4.1]{huang-jiang-2024nodal} for a related quantitative uniqueness result, proved by a different method and in a non-sharp form. With these tools, we can extend the results from the caloric case to general solutions to \eqref{eq-parabolic-equation}, proving Theorem \ref{thm:regularity-general}, Corollary \ref{cor:qualitative}, Theorem \ref{main-theorem-general-setting}, and Corollary \ref{corollary-general-setting}.

\subsection*{Notation} Unless otherwise specified, we omit the dependence of constants on the dimension $n$. We write $\N_0=\{0,1,2\dots\}$. Usually, $C<\infty$ denotes a large constant greater than $1$, while $c$ denotes a small positive constant less than $1$, and we do not rename constants from line to line. When the context is clear we write $\verts{\cdot}\coloneqq \verts{\cdot}_{\cP}$. A statement ``if $\delta\le\bar\delta(\Lambda,\epsilon)$, then \dots'' means ``for any $\Lambda>0$ and any $\epsilon>0$, there exists $\bar\delta>0$ depending on $\Lambda$ and $\epsilon$ such that if $\delta\le\bar\delta$, then \dots". A statement ``Apply Lemma L with $\alpha\leftarrow A$, $\beta\leftarrow B$ \dots" means ``Apply Lemma L with the parameters $\alpha$ replaced by $A$, $\beta$ replaced by $B$ \dots".

\subsection*{Acknowledgments} The authors would like to thank Bennett Chow, Dennis Kriventsov, and Alec Payne for helpful comments and suggestions. They are especially grateful to Aaron Naber for numerous insightful discussions and recommendations. MH was supported in part by the National Science Foundation under Grant No. DMS-2202980. ZM was supported by an AMS–Simons Travel Grant.

\section{Frequency}\label{parabolic frequency}
In the following five sections, we prove Theorem \ref{main-theorem-caloric-setting} and Corollary \ref{theorem time slice}, thereby estimating the parabolic Minkowski content of the effective nodal and singular sets of a \textit{caloric function} $u$, i.e., a solution to 
\begin{align*}
    \pdt u-\Delta u=0.
\end{align*}
We treat caloric functions separately for multiple reasons. First, the results and techniques are more effective and cleaner than in the general setting. Second, many estimates in Section \ref{more-general-parbolic-equations} are proved using a caloric approximation and appealing to the corresponding results in the caloric framework.

\subsection{Conjugate heat kernel}\label{conjugate-heat-kernel}
The \textit{conjugate heat kernel measure} based at $\mathbf{x}_0\coloneqq (x_0,t_0)\in \mathbb{R}^n \times \mathbb{R}$ is
\begin{align*}
    d\nu_{\mathbf{x}_0;t}(x) \coloneqq \frac{1}{(4\pi(t_0-t))^{\frac{n}{2}}} \exp \left( -\frac{|x-x_0|^2}{4(t_0-t)} \right) dx.
\end{align*}
We set
\begin{align*}
    f_{\mathbf{x}_0}(x,t)\coloneqq \frac{\verts{x-x_0}^2}{4(t_0-t)}.
\end{align*}
When $\mathbf{x}_0=\mathbf{0}\coloneqq (0^n,0)$, we write $\nu_{t}\coloneqq \nu_{\mathbf{x}_0;t}$ and $f\coloneqq f_{\mathbf{x}_0}$.

\noindent\textbf{Convention:} We assume that a caloric function $u\in C^\infty (\R^n\times (t_0,t_1))$ has \textit{mild growth}, i.e., 
\begin{align} 
    \sup_{s\in [t_0+\epsilon,t_1-\epsilon]}\int_{\R^n} u^2 \,d\nu_{0,t_1;s} <\infty.\label{eq mild growth}
\end{align}
We impose such a bound for the following three reasons. First, to ensure uniqueness of solutions by ruling out Tychonoff's example \cite[\S2.3 Theorem 7]{evans-2010-pde}. Second, to freely integrate by parts (see Lemma \ref{IBPlemma}). Finally, such a bound allows us to have a spectral decomposition \eqref{eq spectral decomposition} of $u$ into an infinite sum of eigenfunctions of the drift Laplacian where the sum converges locally smoothly and in weighted $L^2$, see \cite[Proposition 3.1]{metafune-Pallara-Priola-2002-spectrum}.

For any function $\phi\in C^1(\R^{n+1})$, we define the \textit{drift Laplacian} $\Delta_\phi$ as
\begin{align*}
    \Delta_\phi u\coloneqq \Delta u-\angles{\cd \phi, \cd u}.
\end{align*}

\noindent \textbf{Convention:} In the remainder of the paper, we will use the following convention for ease of notation. Given any function $v: \mathbb{R}^n \times I \to \mathbb{R}$, where $I$ is an interval containing $t\in \mathbb{R}$, and $v(\cdot,t) \in L^1(\mathbb{R}^n,\nu_{\mathbf{x}_0;t})$, we write 
\begin{align*}
    \int_{\mathbb{R}^n} v\, d\nu_{\mathbf{x}_0;t} \coloneqq \int_{\mathbb{R}^n} v(x,t)\,d\nu_{\mathbf{x}_0;t}(x).
\end{align*}

We will need the following comparison of conjugate heat kernels based at nearby points.
\begin{lemma}\label{lemma-change-of-base-point}
    The following holds for any $\sigma\in(0,1]$ and $\alpha_0,\alpha$ such that $\alpha-1<\alpha_0<\alpha<1$. If $\mathbf{x}_0=(x_0,t_0),\mathbf{x}_1=(x_1,t_1)$ are points in $\mathbb{R}^n \times \mathbb{R}$ satisfying $\verts{x_0-x_1}<1$ and $\verts{t_0-t_1}\le \sigma$, then, for $\tau \geq \frac{2\sigma}{\alpha-\alpha_0}$,
    \begin{align*}
        e^{\alpha_0 f_{\mathbf{x}_0}}d\nu_{\mathbf{x}_0;t_1-\tau} \le C^{\frac{1}{\sigma}} e^{\alpha f}\,d\nu_{\mathbf{x}_1;t_1-\tau}.
    \end{align*}
\end{lemma}

\begin{proof}
    By spacetime translation, we may assume that $\mathbf{x}_1 =\mathbf{0}$. Set $\tau\coloneqq -t$ and $\tau_0\coloneqq t_0-t$. If $2\sigma\leq (\alpha-\alpha_0)\tau$, then
    \begin{align}
        1-\frac{\alpha-\alpha_0}{2}\leq \frac{\tau_0}{\tau}\leq 1+\frac{\alpha-\alpha_0}{2}.\label{eq bounded ratio of tau0 and tau1}
    \end{align}
    Since $\alpha_0<\alpha<1$, we have
    \begin{align*}
        \tau\left(1-\frac{\tau_0(1-\alpha)}{\tau(1-\alpha_0)}\right)\geq \frac{\tau(\alpha-\alpha_0)}{2}\geq \sigma.
    \end{align*}
    Therefore,
    \begin{align*}
        &((1-\alpha_0)f_{\mathbf{x}_0}-(1-\alpha) f)(x,t) \\
        &=\left(\frac{1-\alpha_0}{4\tau_0}-\frac{1-\alpha}{4\tau}\right)\left|x-x_0\frac{1-\alpha_0}{4\tau_0\left(\frac{1-\alpha_0}{4\tau_0}-\frac{1-\alpha}{4\tau}\right)}\right|^2-\frac{1-\alpha}{4\tau\left(1-\frac{\tau_0(1-\alpha)}{\tau(1-\alpha_0)}\right)}\verts{x_0}^2\geq -\frac{1}{4\sigma}.
    \end{align*}
    This together with \eqref{eq bounded ratio of tau0 and tau1} gives
    \begin{align*}
        (4\pi\tau_0)^{-\frac{n}{2}}e^{\alpha_0 f_{\mathbf{x}_0}}e^{-f_{\mathbf{x}_0}} \leq (4\pi(1-\tfrac{\alpha-\alpha_0}{2})\tau)^{-\frac{n}{2}}e^{\frac{1}{4\sigma}+\alpha f}e^{-f}.
    \end{align*}
\end{proof}

We recall the following hypercontractivity estimate, which is a well-known consequence of the Gaussian log Sobolev inequality. 
\begin{proposition} \label{proposition-hypercontractivity}
    Suppose $u$ is a caloric function. For any $0<\tau_1<\tau_2$ and $1<q\leq p<\infty$ with $\frac{\tau_1}{\tau_2} \geq \frac{p-1}{q-1}$, the following holds:
    \begin{align*}
        \left(\int_{\R^n} \verts{u}^p \, d\nu_{\mathbf{x}_0;t_0-\tau_1}\right)^{\frac{1}{p}}\leq \left(\int_{\R^n} \verts{u}^q \, d\nu_{\mathbf{x}_0;t_0-\tau_2}\right)^{\frac{1}{q}}.
    \end{align*}
\end{proposition}

\begin{proof}
    See \cite{gross-1975-log-sobolev}.
\end{proof}

As an application, we can compare the $L^2$ norm of caloric functions based at nearby points.
\begin{lemma}\label{lemma-comparison-of-caloric-energy} 
    Suppose $u$ is a caloric function.
    For any $\theta\in(0,\frac{1}{4}]$, $\sigma\in(0,1]$, $r>0$, if $|\mathbf{x}_1-\mathbf{x}_0|<r$ and $|t_1-t_0|\le\sigma r^2$, then, for any $\tau\geq \frac{6\sigma}{\theta}r^2$
    \begin{align*}
        \int_{\R^n} u^2 \,d\nu_{\mathbf{x}_0;t_0-\tau}\leq C^{\frac{1}{\sigma}}\int_{\R^n} u^2 \,d\nu_{\mathbf{x}_1;t_1-(1+\theta)\tau}.
    \end{align*}
    The same estimate holds for $\pdt u$ and $\cd u\cdot z$ for any $z\in \R^n$.
\end{lemma}

\begin{proof}
    After a space-time translation and parabolic rescaling, we assume that $\mathbf{x}_0=\mathbf{0}$, and $r=1$. Take $\alpha=\frac{\theta}{3}$ and $p=\frac{1+\frac{7}{12}\theta}{1+\frac{1}{6}\theta}$. If $\tau\ge \frac{6\sigma}{\theta}$, then $|t_1|\le \sigma\le \frac{\theta}{6}\tau$, and
    \begin{align*}
        \frac{(1+\theta)\tau}{t_1+\tau}\geq \frac{(1+\theta)\tau}{\frac{\theta}{6}\tau+\tau}= 2p-1,\qquad 
        \frac{\alpha p}{p-1} = \frac{\theta}{3}\frac{1+\frac{7}{12}\theta}{\frac{5}{12}\theta}<.92<1.
    \end{align*}
    Apply Lemma \ref{lemma-change-of-base-point} with $\mathbf{x}_0 \leftarrow \mathbf{x}_1$, $t \leftarrow -\tau$ and $\alpha_0\leftarrow 0$, H\"older's inequality, and then Proposition \ref{proposition-hypercontractivity} with $q \leftarrow 2$ and $p \leftarrow 2p$ to obtain
    \begin{align*}
        \int_{\R^n} u^2 \,d\nu_{-\tau}&\leq C^{\frac{1}{\sigma}}\int_{\R^n} u^2 e^{\alpha f_{\mathbf{x}_1}}\,d\nu_{\mathbf{x}_1;-\tau}\leq C^{\frac{1}{\sigma}}\left(\int_{\R^n} u^{2p}\, d\nu_{\mathbf{x}_1;t_1-(t_1+\tau)}\right)^{\frac{1}{p}} \left(\int_{\R^n}e^{ \frac{\alpha p}{p-1}f_{\mathbf{x}_1}}\, d\nu_{\mathbf{x}_1;-\tau}\right)^{\frac{1}{2}} \\
        &\leq C^{\frac{1}{\sigma}}\int_{\R^n} u^2 \, d\nu_{\mathbf{x}_1;t_1-(1+\theta)\tau}.
    \end{align*}
    The second claim follows because $\pdt u$ and $\cd u\cdot z$ are caloric.
\end{proof}

\subsection{Energy functionals}
\label{energy-functionals}
We record the definitions of pointed energy functionals ({cf.} Definition \ref{definition-energy-functionals}) and their associated derivative formulae ({cf.} Lemma \ref{lemma-monotonicity-formulae-for-the-energy-functionals} and Lemma \ref{frequencydirectionalestimate}) for functions on $\R^n\times \R$.
\begin{definition}\label{definition-energy-functionals}
    Given a function $w\in C^\infty(\R^n\times I)$, $\tau>0$, and a base point $\mathbf{x}_0=(x_0,t_0)$, we define 
    \begin{subequations}
        \begin{align*}
            H^w_{\mathbf{x}_0}(\tau)&\coloneqq  \int_{\R^n} w^2 \, d\nu_{\mathbf{x}_0;t},
            &&E^w_{\mathbf{x}_0}(\tau)\coloneqq 2\tau \int_{\R^n} \verts{\cd w}^2 \, d\nu_{\mathbf{x}_0;t},\\
            N^w_{\mathbf{x}_0}(\tau)&\coloneqq \frac{E_{\mathbf{x}_0}^w(\tau)}{H_{\mathbf{x}_0}^w(\tau)},&&D^w_{\mathbf{x}_0}(\tau)\coloneqq \log_2 \frac{H_{\mathbf{x}_0}^w(2\tau)}{H^w_{\mathbf{x}_0}(\tau)}.
        \end{align*}
    \end{subequations}
\end{definition}
In the literature, $D^w_{\mathbf{x}_0}(\tau)$ and $N^w_{\mathbf{x}_0}(\tau)$ are called the \textit{doubling index} and the \textit{(parabolic) frequency} based at $\mathbf{x}_0$ at scale $\sqrt{\tau}$, respectively. When the context is clear, we omit $w$ in the superscript and write $E_{\mathbf{x}_0}(\tau)\coloneqq E^w_{\mathbf{x}_0}(\tau)$ and so on. Furthermore, we write $E(\tau)=E_{\mathbf{x}_0}(\tau)$ and so on when $\mathbf{x}_0=\mathbf{0}\coloneqq (0^n,0)$.

For simplicity, we will state our results at $\mathbf{0}$ but they hold for any other base point $\mathbf{x}_0$ after a space-time translation. We now recall the monotonicity formulae. We write
\begin{align*}
    \Box \coloneqq \pdt-\Delta, \qquad \Box^*=-\pdt -\Delta.
\end{align*}
\begin{lemma}\label{lemma-monotonicity-formulae-for-the-energy-functionals}
    Suppose $w\in C^\infty(\R^n\times \R)$ and its derivatives have mild growth as in \eqref{eq mild growth} and let $\tau>0$. Then
    \begin{subequations}
        \begin{align*}
            D'(\tau)&=\frac{N(2\tau)-N(\tau)}{\tau \log2}+ \frac{2}{\log 2} \bigg(\int_{\R^n}w\Box w \,d\nu_{-\tau}-2\int_{\R^n}w\Box w \,d\nu_{-2\tau}\bigg),\\
            E'(\tau)&=4\tau\int_{\R^n}(\Delta_f w)^2 +(\Delta_f w)(\Box w)\,d\nu_{-\tau},\\
            H'(\tau)&=\frac{E(\tau)}{\tau}-2\int_{\R^n}w\Box w \,d\nu_{-\tau},\\
            N'(\tau)&=\frac{4\tau}{H(\tau)}\int_{\R^n} \left( \bigg(\Delta_f w+ \frac{N(\tau)}{2\tau}w\bigg)^2+\bigg(\Delta_f w+ \frac{N(\tau)}{2\tau}w\bigg)\Box w \right) \,d\nu_{-\tau}.
        \end{align*}
    \end{subequations}
\end{lemma}

\begin{proof}
    The derivative formulae for $H$ and $E$ are standard (see \cite{poon-1996-unique-continuation}), and follow from the evolution equations
    \begin{align*}
        \Box w^2=2w\Box w-2\verts{\cd w}^2, \quad \Box \verts{\cd w}^2=2\cd w\cdot \cd \Box w-2\verts{\cd^2 w}^2,
    \end{align*}
    and integrating by parts the $f$-Bochner formula
    \begin{align}
        \Delta_f |\nabla w|^2 = 2|\nabla^2 w|^2 +2\langle \nabla w,\nabla \Delta_f w\rangle + \frac{1}{\tau}|\nabla w|^2.\label{eq: f-Bochner formula}
    \end{align}
    To compute $N'$, we use $N(\tau)=\frac{E(\tau)}{H(\tau)}$ to get
    \begin{align*}
        H^2N'&= 4\tau \int_{\R^n} w^2 \, d\nu_{-\tau}  \int_{\R^n}(\Delta_f w)^2 \, d\nu_{-\tau}- 4\tau\left(\int_{\R^n} w\Delta_f w \, d\nu_{-\tau} \right)^2\\
        &\quad+ 4\tau H\int_{\R^n}(\Delta_f w)(\Box w)\,d\nu_{-\tau}+2E\int_{\R^n}w\Box w \,d\nu_{-\tau}\\
        &=4\tau H \int_{\R^n}\left(\Delta_f w+ \frac{N(\tau)}{2\tau}w\right)^2\, d\nu_{-\tau}+4\tau H\int_{\R^n} \bigg(\Delta_f w+ \frac{N(\tau)}{2\tau}w\bigg)\Box w\,d\nu_{-\tau}.
    \end{align*}

    Finally, $N(\tau)=\frac{E(\tau)}{H(\tau)}=\frac{\tau (H'(\tau)+2\int_{\R^n}w\Box w \,d\nu_{-\tau})}{H(\tau)}$ implies that
    \begin{align}
        \log \frac{H(2\tau)}{H(\tau)}=\int_{\tau}^{2\tau} \frac{N(\ol{\tau})}{\ol{\tau}}\, d\ol{\tau} -2\int_{\tau}^{2\tau}\int_{\R^n}w\Box w \,d\nu_{-\ol{\tau}}d\ol{\tau}.\label{eq-doubling-index-and-frequency}
    \end{align}
    In particular,
    \begin{align*}
        D'(\tau)=\frac{N(2\tau)-N(\tau)}{\tau \log2}+ \frac{2}{\log 2} \bigg(\int_{\R^n}w\Box w \,d\nu_{-\tau}-2\int_{\R^n}w\Box w \,d\nu_{-2\tau}\bigg).
    \end{align*}
\end{proof}

To compare the frequency at a fixed scale at nearby points, for $\mathbf{x}=(x,t)$, $s>0$, we write $H_s(\mathbf{x})=H_{\mathbf{x}}(s)$ and so on. We also define the restricted functional with respect to a $k$-plane $L^k\subset \R^n$ and with respect to the temporal direction as:
\begin{align*}
    N_{s;L}(\mathbf{x})\coloneqq \frac{2s\int_{\R^n} \verts{\pi_L \cd w}^2\,d\nu_{\mathbf{x};t-s}}{H_s(\mathbf{x})}, \qquad T_{s}(\mathbf{x})\coloneqq \frac{2s^2\int_{\R^n} \verts{ \pdt w}^2\,d\nu_{\mathbf{x};t-s}}{H_s(\mathbf{x})}.
\end{align*}

\begin{lemma}\label{frequencydirectionalestimate}
    Suppose $w\in C^\infty (\R^n\times I)$ and $L\subset \R^n$ is a $k$-plane. Then 
    \begin{align*}
        s\verts{\pi_L \cd D_s}^2\leq C (N_{2s;L}+N_{s;L}),\qquad \qquad s^2\verts{\pdt D_s}^2\leq  C(T_{s}+T_{2s}).
    \end{align*}
\end{lemma}

\begin{proof}
    After a rotation, we may assume $L\ni v = e_1$, so $\nabla w\cdot v = \partial_{x_1} w$. At $\mathbf{x}=(x,t)$, we compute
    \begin{align*}
        \pd_{x_1} H_s(\mathbf{x})&= (4\pi s)^{-\frac{n}{2}}\int_{\R^n} w^2(y,t-s) \pd_{x_1} e^{-\frac{\verts{x-y}^2}{4s}}\,dy= -(4\pi s)^{-\frac{n}{2}}\int_{\R^n} w^2(y,t-s) \pd_{y_1}e^{-\frac{\verts{x-y}^2}{4s}}\,dy\\
        &=\int_{\R^n} 2w(y)\pd_{y_1}w(y) \, d\nu_{\mathbf{x};t-s}(y)
        \leq 2\parens*{\int_{\R^n} w^2 \, d\nu_{\mathbf{x};t-s}}^{\frac{1}{2}}\parens*{\int_{\R^n} \verts{\pi_{L}\cd w}^2 \, d\nu_{\mathbf{x};t-s}}^{\frac{1}{2}}.
    \end{align*}
    Then, 
    \begin{align*}
        s(\log 2)^2\verts{\pd_{x_1} D_s}^2&\leq 2s\left(\verts{\pd_{x_1} \log H_{2s}}^2+\verts{\pd_{x_1} \log H_{s}}^2\right) \leq 4(N_{2s;L}+N_{s;L}).
    \end{align*}
    The second inequality is proved similarly.
\end{proof}

We often use the following $L^2(\nu_{-\tau})$-orthogonal decomposition of a function $w \in L^2(\nu_{-\tau})$
\begin{align}
    w = \sum_{j=0}^\infty p_j,\label{eq spectral decomposition}
\end{align}
where $p_j$ is the projection of $w$ onto the $L^2(\nu_{t})$-eigenspace of $2\tau\Delta_{f}$ with eigenvalue $j\in \N_0$, i.e., $2\tau \Delta_f p_j+jp_j=0$ and is a \textit{homogeneous caloric polynomial} (at $\mathbf{0}$) of degree $j$ in the following sense. When $u$ is caloric, the same decomposition \eqref{eq spectral decomposition} holds for all time. The existence of such decomposition follows from a spectral analysis of $2\tau \Delta_f$, see \cite[Proposition 3.1]{metafune-Pallara-Priola-2002-spectrum}.

\begin{definition}\label{definition homogeneous polynomials}
    A caloric polynomial $p$ is \textit{homogeneous} at $\mathbf{x}_0$ of degree $m$ if for any $\lambda>0$ we have
    \begin{align*}
        p(x_0+\lambda (x-x_0), t_0+\lambda^2 (t-t_0)) = \lambda^m p(x,t).
    \end{align*}
\end{definition}
Denote $\cP_m$ to be the space of homogeneous caloric polynomials of degree $m$.

In the following lemma, we prove that if $w(\cdot,\tau)$ is almost an eigenfunction of $2\tau \Delta_f$ with eigenvalue $N(\tau)$, then $w(\cdot,\tau)$ can be approximated by a caloric homogeneous polynomial and $N(\tau)$ is almost an integer.
\begin{lemma}\label{lemma almost eigenvalue equation}
    For any $w\in C^\infty (\R^n\times I)$ such that $w(\cdot,-\tau)\in L^2(\nu_{-\tau})$, $\tau>0$, suppose $m\in \arg\min_{k\in \N_0}\verts{N(\tau)-k}$ and $N(\tau)\in [j,j+1)$ for some $j\in \N_0$. Let $p_m\in \cP_m$ be the $L^2(\mathbb{R}^n,\nu_t)$-orthogonal projection onto the $\frac{m}{2}$-eigenspace of $-\Delta_{f_t}$. Then
    \begin{align*}
        (N(\tau)-j)(j+1-N(\tau))H(\tau)+\int_{\R^n} (w-p_m)^2\, d\nu_{-\tau}\leq 20\tau^2\int_{\R^n}\left(\Delta_f w+ \frac{N(\tau)}{2\tau}w\right)^2\, d\nu_{-\tau}.
    \end{align*}
\end{lemma}

\begin{proof}
    Note that 
    \begin{align*}
        \max_{a\in \R}\min_{k\in \N_0}\left\{(N(\tau)-k)(a-k)\right\}=(N(\tau)-j)(j+1-N(\tau)),
    \end{align*}
    which is achieved when $a=2j+1-N(\tau)$. Using the spectral decomposition $w=\sum_{j=0}^\infty p_j$, and the $L^2(\nu_{-\tau})$-orthogonality of $p_j$, we get
    \begin{align*}
        \int_{\R^n} (2\tau \Delta_{f} w+N(\tau)w)^2\,d\nu_{-\tau}&= \int_{\mathbb{R}^n} (2\tau \Delta_f w+N(\tau)w)(2\tau\Delta_f w + aw)\,d\nu_{-\tau} \\
        &=\sum_j (N(\tau)-j)(a-j) \int_{\R^n} p_j^2 \,d\nu_{-\tau}\\
        &\geq (N(\tau)-j)(j+1-N(\tau)) H(\tau),
    \end{align*}
    where we used $\int_{\mathbb{R}^n}(2\tau\Delta_f u+N(\tau)u)ud\,\nu_{-\tau} =0$ for the first equality. Using this again, the conclusion follows by combining the following:
    \begin{align*}
        \int_{\R^n} (w-p_m)^2\,d\nu_{-\tau}&=\sum_{k\neq m} \int_{\R^n}p_k^2\, d\nu_{-\tau}\leq 4
        \sum_{k} (N(\tau)-k)(N(\tau)-k)\int_{\R^n}p_k^2\, d\nu_{-\tau}.
    \end{align*}
\end{proof}

\subsection{Monotonicity and quantitative uniqueness in caloric setting}
For the rest of the section, we will restrict our discussion to the case of a caloric function $u$. Since $\Box u=0$, the functionals $D$, $E$, $H$, $N$ are monotone by Lemma \ref{lemma-monotonicity-formulae-for-the-energy-functionals}. We list some consequences of the monotonicity.

First, we get the following version of interior estimates on two-sided parabolic balls that we need later. Note that the standard interior estimates hold on backward parabolic balls; see, for example, \cite[\S 2.3]{evans-2010-pde}.
\begin{lemma} \label{lemma-interior-estimates}
   If $u$ is a caloric function, then
    \begin{align*}
        \sup_{\mathbf{x}\in P(\mathbf{0},\frac{r}{8})} r^{2(2l+j)} \verts{\pdt^l\cd^j u}^2(\mathbf{x}) \leq C_{jl}H(4r^2).
    \end{align*}
\end{lemma}

\begin{proof}
    After parabolic scaling, we may assume $r=1$. Since $\pd_t^l \cd^j u=\cd^j \Delta^l u$ is caloric, using the monotonicity formula in Lemma \ref{lemma-monotonicity-formulae-for-the-energy-functionals} and the comparison of energy in Lemma \ref{lemma-comparison-of-caloric-energy}, we deduce that for all $\mathbf{x}\in P(\mathbf{0},\frac{1}{8})$
    \begin{align*}
        \verts{\pd_t^l \cd^j u}^2(\mathbf{x})\leq \int_{\R^n} \verts{\pdt^l \cd^j u}^2 d\nu_{\mathbf{x};t-\frac{1}{2^{2l+j}}}\leq 
        C_{jl}\int_{\R^n}u^2 \, d\nu_{\mathbf{x};t-1}
        \leq C_{jl}H_{\mathbf{0}}(3).
    \end{align*}
\end{proof}

We now prove that the frequency at nearby points at smaller scale is bounded above by the frequency at a reference point at a larger scale ({cf.} Theorem 2.2.8 in \cite{han-lin-1994-nodal-sets-ii}, Lemma 2.7 in \cite{cheeger-naber-valtorta-2015-critical}, or Theorem 3.14 in \cite{naber-valtorta-2017-volume-estimtates-of-critical-sets-of-pde}). This allows us to assume a bound on the frequency just at one point.
\begin{lemma}\label{lemma-frequency-uniform-bound}
    Suppose $u$ is caloric. Then
    \begin{align*}
        \sup_{\mathbf{x}\in P(\mathbf{0},\frac{1}{8})} N_{\mathbf{x}}(1)\leq C(N(4)+1).
    \end{align*}
\end{lemma}

\begin{proof}
    Applying Lemma \ref{lemma-comparison-of-caloric-energy} with $\theta=\frac{1}{4}$, for $\mathbf{x}\in P(\mathbf{0},\frac{1}{8})$, we get
    \begin{align*}
        H_{\mathbf{0}}\left(\tfrac{1}{1+\theta}\right) \le C H_{\mathbf{x}}(1),\qquad 
        H_{\mathbf{x}}\left(2\right) \le C H_{\mathbf{0}}(2+2\theta).
    \end{align*}
    
    Therefore,
    \begin{align*}
        N_{\mathbf{x}}(1)\le D_{\mathbf{x}}(1) = \log_2 \frac{H_{\mathbf{x}}(2)}{H_{\mathbf{x}}(1)}
        \le C+ \frac{1}{\log 2}\int_{\frac{1}{1+\theta}}^{2+2\theta} \frac{N(\tau)}{\tau}\,d\tau
        \le C + (1+2\theta)N^u_{\mathbf{0}}(4).
    \end{align*}
\end{proof}

So far, we have used the monotonicity of the frequency. However, by Lemma \ref{lemma-monotonicity-formulae-for-the-energy-functionals} and Lemma \ref{lemma almost eigenvalue equation}, we get refined monotonicity
\begin{align}
   \label{ineq: N refined}
    \tau N'(\tau)\geq (N(\tau)-m)(m+1-N(\tau)),
\end{align}
when $N(\tau)\in [m,m+1)$ for some $m\in \mathbb{N}_0$. In the lemma below we use \eqref{ineq: N refined} to prove a definite drop in the frequency and pinching of the frequency near integers.
\begin{lemma}\label{lemma-refined-monotonicity-of-frequency}
    The following hold:
    \begin{enumerate}[label=(\arabic*)]
        \item \label{lemma: frequency drop} If $N(\tau)\le m+1-\epsilon$ for some $\epsilon\in(0,\frac{1}{4})$, then
        \begin{align*}
            N\left(\left(\frac{\epsilon}{1-\epsilon}\right)^2\tau\right) \le m+\epsilon.
        \end{align*}

        \item\label{lemma: pinching int-1} If $5\epsilon=\min_{k\in \mathbb{N}_0}|N(\tau)-k|$, then $N(\tau)-N\left(\frac{\tau}{2}\right)\geq \epsilon$.

        \item \label{lemma: pinching int-2}Further, if $N(\tau)-N\left(\frac{\tau}{2}\right)\leq \delta$, then there exists $m\in\mathbb{N}_0$ such that for $s\in\left[\frac{\tau}{2},\tau\right]$,
        \begin{align*}
            |N(s)-m|\le 6\delta.
        \end{align*}
    \end{enumerate}
\end{lemma}

\begin{proof}
    \ref{lemma: frequency drop} We use \eqref{ineq: N refined} to argue as in \cite[Lemma 3.16]{naber-valtorta-2017-volume-estimtates-of-critical-sets-of-pde}. We may assume that $\tau=1$ and $N(1)\in[m,m+1-\epsilon]$. Then $x(\tau)=N(\tau)-m$ satisfies $\tau x'\ge x(1-x)$ as long as $x>0$. Let $z(\tau)=\frac{a\tau}{1+a\tau}$, where $a=\frac{1-\delta}{\delta}$ for some $\delta<\epsilon$ so that $z(1)=1-\delta$. Then $\tau z'=z(1-z)$ and $\tau \frac{d}{d\tau} \log\!\left(\frac{x}{z}\right)\ge z-x$. By an ODE argument, $x(\tau)\le z(\tau)$ as long as $x(\tau)> 0$ holds. The conclusion follows by taking $\delta\to \epsilon$.

    \ref{lemma: pinching int-1} We may assume $\epsilon>0$. 
    Let $m=[N(\tau)]$ be the largest integer not greater than $N(\tau)$. Then $m+5\epsilon\le N(\tau)\le m+1-5\epsilon$. For $s\le \tau$, as long as $N(s)\in [m+3\epsilon, m+1-5\epsilon]$, by Lemma \ref{lemma-refined-monotonicity-of-frequency},
    \begin{align*}
        sN'(s) \ge (N-m)(m+1-N)\ge 3\epsilon(1-3\epsilon) \ge 2\epsilon.
    \end{align*}
    If $N\left(\frac{\tau}{2}\right)<m+3\epsilon$, then we are done. Otherwise, integrating the differential inequality yields
    \begin{align*}
        N(\tau)-N(\tfrac{\tau}{2}) \ge 2\epsilon\log2  > \epsilon.
    \end{align*}
    
    \ref{lemma: pinching int-2} Let $5\epsilon=\min_k|N(\tau)-k|=|N(\tau)-m|$. By \ref{lemma: pinching int-1}, $\epsilon \le N(\tau)-N(\tfrac{\tau}{2})\le \delta$. The conclusion follows by the monotonicity of the frequency.
\end{proof}

We prove that the frequency is almost constant at smaller scales using the monotonicity and a bound on the frequency at a larger scale. 
\begin{lemma}\label{lem: pinched scale}
     For any $\epsilon\in(0,\frac{1}{10})$, if
     \begin{align*}
         N(r_2^2)-N(r_1^2)\le \Lambda,\qquad  0< r_1\le \epsilon^{4\Lambda+10}r_2,
     \end{align*}
     then there is $s\in (r_1,r_2)$ such that
     \begin{align*}
         N(\epsilon^{-2}s^2)-N\left(\epsilon^{2}s^2\right) < \epsilon.
     \end{align*}
\end{lemma}

\begin{proof}
    By scaling, we may assume $r_2=1$. Define $\tau_k=\epsilon^{4k}$. 
    Let $m$ be the smallest integer such that $N(\tau_0)<m+\frac{1}{2}\epsilon$. If $N(\tau_1)\ge m-\frac{1}{2}\epsilon$, then we can take $s=\epsilon^2$. Otherwise, $N(\tau_1)<m-\frac{1}{2}\epsilon$, and by Lemma \ref{lemma-refined-monotonicity-of-frequency}, $N(\tau_2)\le N(\frac{1}{4}\epsilon^2\tau_1)<m-1+\frac{1}{2}\epsilon$. If $N(\tau_3)\ge m-1+\frac{1}{2}\epsilon$, then we can take $s=\epsilon^5$. Otherwise, we can proceed as above. 
    If we fail to find such $s$ in $k=[\Lambda]+2$ steps, then 
    \begin{align*}
        N(r_1^2)\le N(\tau_{2k})< m-k+\tfrac{1}{2}\epsilon\le N(1)+1-k \le N(r_1^2)+\Lambda-[\Lambda]-1,
    \end{align*}
    a contradiction.
\end{proof}

The following theorem shows that when the frequency is pinched at $m$, $u$ is close to $p_m$ where $u = \sum_{j=0}^\infty p_j$ is the $L^2(\nu_{-\tau})$-orthogonal decomposition in terms of homogeneous caloric polynomials.
\begin{theorem}\label{theorem-almost-frequency-cone-implies-unique-geometric-cone}
    Fix $0\leq \tau_1<\frac{\tau_2}{2}$. Assume that $N(\tau_2)-N(\tau_1)\le\delta<\frac{1}{10}$, and let $m\in\mathbb{N}_0$ given by Lemma \ref{lemma-refined-monotonicity-of-frequency}\ref{lemma: pinching int-2} with $\tau \leftarrow \tau_2$. Then, for any $l,j\in \N_0$ and $0<\tau\leq \frac{\tau_1}{2}$, we have
    \begin{align*}
        \tau^{2l+j}\int_{\R^n} \verts{\pdt^l\cd^j (u-p_m)}^2\, d\nu_{-\tau} \leq
        C_{j,l}
        \delta H(\tau_2),
    \end{align*}
    where $p_m$ is the projection of $u$ onto the eigenspace of $2\tau \Delta_f$ with eigenvalue $m$. 
\end{theorem}

\begin{proof}
    First assume that $j=l=0$. By Lemma \ref{lemma-monotonicity-formulae-for-the-energy-functionals} and Lemma \ref{lemma almost eigenvalue equation}, we know that
    \begin{align*}
        \int_{\tau_1}^{\tau_2} \frac{1}{\tau H(\tau)}\int_{\R^n} \verts{u-p_m}^2\,d\nu_{-\tau}\,d\tau
        =\ 4\int_{\tau_1}^{\tau_2} N'(\tau) \,d\tau
        = 4(N(\tau_2)-N(\tau_1))
        \le 4\delta.
    \end{align*}
    Since $u-p_m$ is caloric, by the monotonicity {of $H$},
    \begin{align*}
        \int_{\R^n} \verts{u-p_m}^2 \,d\nu_{-\tau_1}\leq \frac{4\tau_2}{\tau_{2}-\tau_1}\delta H(\tau_2).
    \end{align*}

    To prove the general case, we can use the monotonicity formula to get 
    \begin{align*}
        \frac{d}{dt}\int_{\R^n} \verts{u-p_m}^2\,d\nu_{-\tau}= -2 \int_{\R^n} |\nabla (u-p_m)|^2\,d\nu_{-\tau},
    \end{align*}
    where $\tau=-t$. By the monotonicity {of $\tau \mapsto \int_{\mathbb{R}^n}|\nabla(u-p_m)|^2 d\nu_{-\tau} $},
    \begin{align*}
        \frac{\tau}{2}\int_{\R^n} |\nabla(u-p_m)|^2\,d\nu_{-\frac{\tau}{2}} \le \int_{\frac{\tau}{2}}^{\tau} \int_{\R^n} |\nabla(u-p_m)|^2\,d\nu_{-s} ds \le \frac{1}{2} \int_{\R^n} (u-p_m)^2\,d\nu_{-\tau}.
    \end{align*}
    Furthermore, since $\pdt^l\cd^j(u-p_m)=\cd^j\Delta^l (u-p_m)$ is caloric, we can prove the general case by induction.
\end{proof}

As a consequence of Theorem \ref{theorem-almost-frequency-cone-implies-unique-geometric-cone}, we show that if the frequencies based at two nearby points are almost constant at a given scale, then they should be pinched at the same integer.
\begin{corollary}\label{cor: nearby same m}
    Suppose $N_{\mathbf{x}_1}(8)\leq \Lambda$ and $\mathbf{x}_2\in P(\mathbf{x}_1,\frac{1}{8})$ such that for both $i=1,2$, we have $N_{\mathbf{x}_i}(8)-N_{\mathbf{x}_i}(\tfrac{1}{8})<\delta$. If $\delta\le C^{-\Lambda}$, then there exists $m\in \N_0$ such that
    \begin{align*}
        |N_{\mathbf{x}_i}(\tau)-m|<C\delta 
    \end{align*}
    for all $\tau\in \left[\frac{1}{8},8\right]$ and both $i=1,2$.
\end{corollary}

\begin{proof}
    By Lemma \ref{lemma-refined-monotonicity-of-frequency}\ref{lemma: pinching int-2}, there exist $m_1,m_2\in \mathbb{N}_0$ such that $|N_{\mathbf{x}_i}(\tau)-m_i|<C\delta$ for both $i=1,2$ and all $\tau\in\left[\frac{1}{8},8\right]$, so it suffices to show $m_1=m_2$. Suppose by way of contradiction that $m_1<m_2$. 

    Let $p_{m_i}(\cdot,t)$ be the projection of $u(\cdot,t)$ onto the $L^2(\nu_{\mathbf{x}_i;t})$-eigenspace of $2\tau_i \Delta_{f_{\mathbf{x}_i}}$ with eigenvalue $m_i$, which is a polynomial at $x_i$ of degree $m_i$. Using $[\Delta_f,\nabla]=\frac{1}{2(t_i-t)}\nabla$ and integrating $\frac{1}{2}\Delta_{f_i} |\nabla^k p_{m_i}|^2 = |\nabla^{k+1}p_{m_i}|^2 + \langle \nabla^k p_{m_i},\Delta_{f_i} \nabla^k p_{m_i}\rangle$ yields 
    \begin{align*}
        \int_{\mathbb{R}^n} |\nabla^{m_2} p_{m_2}|^2 d\nu_{\mathbf{x}_2;t_2-\tau} = \frac{m_2!}{(2\tau)^{m_2}} \int_{\mathbb{R}^n} p_{m_2}^2 d\nu_{\mathbf{x}_2;t_2-\tau}
    \end{align*}
    for any $\tau>0$. Thus Lemma \ref{lemma-comparison-of-caloric-energy} and Theorem \ref{theorem-almost-frequency-cone-implies-unique-geometric-cone} give the following, where $C_0 \coloneqq  m_2! 4^{m_2} \geq C^{\Lambda^2}$: 
    \begin{align*}
        C_{0}(1-C^\Lambda \delta)H_{\mathbf{x}_2}(1/8) &\leq C_{0}\int_{\R^n}\verts{p_{m_2}}^2 \, d\nu_{\mathbf{x}_2; t_2-\frac{1}{8}}=\int_{\R^n} \verts{\cd^{m_2}p_{m_2}}^2 \, d\nu_{\mathbf{x}_2; t_2-\frac{1}{8}} \\
        &=\int_{\R^n} \verts{\cd^{m_2}p_{m_1}-\cd^{m_2}p_{m_2}}^2\, d\nu_{\mathbf{x}_2; t_2-\frac{1}{8}}\\
        &\leq 2\int_{\R^n} \verts{\cd^{m_2}p_{m_1}-\cd^{m_2}u}^2\, d\nu_{\mathbf{x}_2; t_2-\frac{1}{8}}\\
        &\quad+2\int_{\R^n} \verts{\cd^{m_2}u-\cd^{m_2}p_{m_2}}^2\, d\nu_{\mathbf{x}_2; t_2-\frac{1}{8}}\\
        &\leq C^{\Lambda} \delta H_{\mathbf{x}_2}(1/8),
    \end{align*}
    which is a contradiction if $\delta \leq C^{-\Lambda}$. 
\end{proof}

\section{Symmetry and Cone Splitting}\label{symmetry-and-cone-splitting}
In this section, we define $k$-pinching and quantitative $k$-symmetry. In Theorem \ref{thm: sym split equiv}, we show that almost $k$-symmetry is equivalent to smallness of the $k$-pinching by establishing a cone splitting inequality in Theorem \ref{theorem-cone-splitting-inequality}. In Lemma \ref{newlineup}, we prove the existence of a uniform (across scales and locations) plane of symmetry. Finally, in Lemma \ref{lem:beta}, we prove that Jones $\beta$-numbers are controlled by the average frequency drop.

\subsection{Quantitative symmetry}\label{quantitative-symmetry-and-cone-splitting-principle}
Before we define symmetry, we define a quantitative notion of linear independence of space-time points, following \cite[Definition 4.5]{naber-valtorta-2017-rectifiable-for-harmonic-maps} from the elliptic setting.  
\begin{definition} \label{def of grassmann}
    Let $\text{Gr}_{\mathcal{P}}(k)$ be the set of linear subspaces of $\mathbb{R}^{n+1}$ of one of the following types:
    \begin{enumerate}[label={(\roman*)}]
       \item $L \times \{0\}$, where $L\subseteq \mathbb{R}^n$ is a $k$-dimensional subspace,
       
       \item $L \times \mathbb{R}$, where $L\subseteq \mathbb{R}^n$ is a $(k-2)$-dimensional subspace.
    \end{enumerate} 
    Note that $V\in \text{Gr}_{\mathcal{P}}(k)$ are the only linear subspaces that are invariant under parabolic scaling. If $V=L\times\{0\}\in {\rm Gr}_{\mathcal{P}}(k)$, define $V^\perp = L^\perp\times\R$. If $V=L\times\R\in {\rm Gr}_{\mathcal{P}}(k)$, define $V^\perp = L^\perp\times\{0\}$. Let ${\rm Aff}_{\mathcal{P}}(k)$ be the set of affine subspaces of the form $W=\mathbf{x}+V$, where $\mathbf{x}\in \mathbb{R}^{n+1}$ and $V\in {\rm Gr}_{\mathcal{P}}(k)$. For $W=\mathbf{x}+V\in {\rm Aff}_{\cP}(k)$, we write $\hat W=V$ and $W^\perp=V^\perp$.    
\end{definition}

\begin{definition}[Independence]\label{definition-independence}
   We say $S\subseteq \mathbb{R}^{n+1}$ is \textit{$(k,\alpha)$-independent} if there is no  $W\in {\rm Aff}_{\cP}(k-1)$ such that $S\subseteq P(W,\alpha)$. Given a $(k,\alpha)$-independent $S$, we say $S$ is  \textit{$(k,\alpha)$-spatially independent} if $|t_1-t_2|<\alpha^2$ for any $\mathbf{x}_i=(x_i,t_i)\in S$, $i=1,2$; otherwise we say $S$ is \textit{$(k,\alpha)$-temporally independent.}
\end{definition}

As a result of Gram-Schmidt orthogonalization, we can prove that independent points give rise to a basis.
\begin{lemma}\label{lemma-linear-independence}
    Let $S=\{\mathbf{x}_j\}_{j=0}^K$ be \textit{$(k,\alpha r)$-independent}. Suppose either $K=k$ and $S$ is spatial, or $K=k-2$ and $S$ is temporal. Denote
    \begin{align*}
        L\coloneqq {\rm Span} \{x_1-x_0,\cdots, x_K-x_0\}.
    \end{align*}
    Let $V\coloneqq L\times\{0\}$ if $S$ is spatial, and $V\coloneqq L\times \mathbb{R}$ if $S$ is temporal.
    Then, for all $\mathbf{y}\in (\mathbf{x}_0+V)\cap P(
    \mathbf{x}_0,2r)$, there exists a unique set $\{q_j\}_{j=1}^K \subseteq \mathbb{R}$ such that
    \begin{align*}
        y=x_0+\sum_{j=1}^K q_j(x_j-x_0), && \verts{q_i}\leq C \alpha^{-n} \frac{\verts{y-x_0}}{r}.
    \end{align*}
\end{lemma}

\begin{proof}
    Apply the argument of \cite[Lemma 4.6]{naber-valtorta-2017-rectifiable-for-harmonic-maps} to the set $\{x_j\}_{j=0}^K$.
\end{proof}

\begin{definition}[Frequency pinching]\label{definition-frequency-pinching}
    Given a caloric function $u$, define
    \begin{align*}
        \mathcal{E}_{r}(\mathbf{x})\coloneqq \mathcal{E}_{r}(\mathbf{x};u)\coloneqq N_{\mathbf{x}}(8r^2) - N_{\mathbf{x}}\left(\tfrac{1}{8}r^2\right), &&\mathcal{E}_{r}(\{\mathbf{x}_i \}_{i=0}^{k})\coloneqq \max_{i} \mathcal{E}_{r}(\mathbf{x}_i).
    \end{align*}
    Define the $(k,\alpha r)$-\textit{pinching} at $\mathbf{x}_0$
    \begin{align*}
        \mathcal{E}^{k,\alpha}_{r}(\mathbf{x}_0)\coloneqq \mathcal{E}^{k,\alpha}_{r}(\mathbf{x}_0;u) \coloneqq \inf \mathcal{E}_{r}(\{\mathbf{x}_i \}_{i=0}^{K}),
    \end{align*}
     where the infimum is over $(k,\frac{1}{20}\alpha r)$-independent subsets $\{\mathbf{x}_i\}_{i=0}^K \subseteq P(\mathbf{x}_0,\frac{1}{10}r)$. 
\end{definition}

\begin{definition}\label{definition symmetry}
    A caloric function $u$ is \textit{$(k,\epsilon,r)$-symmetric} at $\mathbf{x}=(x,t)$ (with respect to $V$), if one of the following holds:
    \begin{enumerate}[label={(\arabic*)}]
        \item \label{def: spatial symmetry} There is a $k$-plane $L \subseteq \mathbb{R}^n$ such that $V=L\times \{0\}$ satisfies
        \begin{align*}
            r^2 \int_{\mathbb{R}^n} |\pi_{L} \nabla u|^2 \, d\nu_{\mathbf{x};t-r^2} \leq \epsilon \int_{\mathbb{R}^n} |u|^2 \, d\nu_{\mathbf{x};t-r^2}.
        \end{align*}
        
        \item There is a $(k-2)$-plane $L \subseteq \mathbb{R}^n$ such that $V=L\times \mathbb{R}$ satisfies
        \begin{align*}
            r^2 \int_{\mathbb{R}^n} |\pi_{L} \nabla u|^2 \, d\nu_{\mathbf{x};t-r^2} + r^4 \int_{\mathbb{R}^n} |\partial_t u|^2 \, d\nu_{\mathbf{x};t-r^2} \leq \epsilon \int_{\mathbb{R}^n} |u|^2 \, d\nu_{\mathbf{x};t-r^2}.
        \end{align*}
    \end{enumerate}
    We say the symmetry is \textit{spatial} if \ref{def: spatial symmetry} holds and \textit{temporal} otherwise.
\end{definition}

In the following, we roughly show that $\mathcal{E}_r^{k,\alpha}(\mathbf{x})$ is small if and only if $\mathbf{x}$ is frequency-pinched and almost $k$-symmetric at scale $r$. 
\begin{theorem} \label{thm: sym split equiv}
    Let $u$ be a caloric function with $N_{\mathbf{x}}(10^5r^2)\le \Lambda$. 
    \begin{enumerate}[label={(\arabic*)}]
        \item $u$ is $(k,C^{\Lambda}\alpha^{-n}\mathcal{E}^{k,\alpha}_r(\mathbf{x}),r)$-symmetric at $\mathbf{x}$.\label{thm: sym split equiv-1}
        
        \item \label{thm: sym split equiv-2} Suppose $u$ is $(k,\epsilon,10^2r)$-symmetric at $\mathbf{x}$ with respect to $V\in{\rm Gr}_{\mathcal{P}}(k)$ and
        \begin{align*}
            |N_{\mathbf{x}}(10^{-2} \kappa^2 r^2)-N_{\mathbf{x}}(10^{2}r^2)| \leq \epsilon.
        \end{align*}
        Then
        \begin{align}\label{freqatnearbypoints}
            \verts*{N_{\mathbf{v}}(50 r^2)-N_{\mathbf{v}}\left(50^{-1}\kappa^2r^2\right)}\leq C(\kappa)^{\Lambda}\sqrt{\epsilon}
        \end{align}
        for any $\mathbf{v} \in (\mathbf{x}+V)\cap P(\mathbf{x},10r)$. As a consequence, $\mathcal{E}^{k,1}_s(\mathbf{v})\le C(\kappa)^{\Lambda}\sqrt{\epsilon}$, and
        $u$ is $(k,C(\kappa)^{\Lambda}\sqrt{\epsilon},s)$-symmetric at $\mathbf{v}$ with respect to $V$ for all $s\in [\kappa r,r]$.
    \end{enumerate}
\end{theorem}

We first prove \ref{thm: sym split equiv-2} and defer the proof of \ref{thm: sym split equiv-1} to the Section \ref{cone-splitting-inequality}.
\begin{proof}[{Proof of Theorem \ref{thm: sym split equiv}\ref{thm: sym split equiv-2}}]
    By parabolic rescaling and translation, we may assume $\mathbf{x}=\mathbf{0}$ and $r=1$.
    First, suppose the symmetry is spatial: $V=L^k\times \{0\}$ and fix an arbitrary $\mathbf{v}=(v,0)\in V\cap P(\mathbf{0},10)$. 
    By the monotonicity and Lemma \ref{lemma-comparison-of-caloric-energy} with $\theta \leftarrow 10^{-1},\sigma\leftarrow 10^{-5}$, $t_0 = t_1 \leftarrow 0$, $\tau \leftarrow 90$, $r\leftarrow 10$, if $w$ is caloric, then 
    \begin{align*}
        \int_{\mathbb{R}^n} w^2\,d\nu_{\mathbf{v};-s^2} \le  \int_{\mathbb{R}^n} w^2\,d\nu_{\mathbf{v};-90}
        \le C \int_{\mathbb{R}^n} w^2\,d\nu_{-10^4},
    \end{align*}
    for all $s\in[\kappa,9]$. Applying this to $w=\nabla u \cdot v_i$ for an orthonormal basis $\{v_i\}_{i=1}^k$ of $L$, and using the assumption that $u$ is $(k,\epsilon,10^2)$-symmetric at $\mathbf{0}$ with respect to $V$ yields
    \begin{align*}
        \int_{\mathbb{R}^n} |\pi_L \nabla u|^2 d\nu_{\mathbf{v};-s^2} &\leq C \int_{\mathbb{R}^n} |\pi_L \nabla u|^2 d\nu_{-10^4} \leq C\epsilon \int_{\mathbb{R}^n} u^2 d\nu_{-10^4} \\
        &\leq C\epsilon \int_{\mathbb{R}^n} u^2 d\nu_{\mathbf{v};-2\cdot 10^4} \leq C\kappa^{-2\Lambda} \epsilon \int_{\mathbb{R}^n} u^2 d\nu_{\mathbf{v};-s^2}
    \end{align*}
    for all $s\in [\kappa,9]$,
    where for the last two inequalities, we applied Lemma \ref{lemma-comparison-of-caloric-energy} and the assumption $N_{\mathbf{0}}(10^5) \leq \Lambda$. In other words, $N_{s^2;L}(\mathbf{w})\leq C(\kappa)^{\Lambda}\epsilon$ for all $\mathbf{w} \in P(\mathbf{0},10)\cap V$ and $s\in [\kappa,9]$. By Lemma \ref{frequencydirectionalestimate}, we have 
    \begin{align*}
        &\left|\frac{d}{d\sigma}D_{s^2}(\sigma v,0)\right|^2
        \le Cs^{-2} (N_{2s^2;L}+N_{s^2;L})(\sigma v,0)
        \le C(\kappa)^{\Lambda}\epsilon
    \end{align*}
    for all $s\in [\kappa,4]$ and $\sigma \in [0,1]$. Thus
    \begin{align}
        N_{\mathbf{v}}(s^2) \leq D_{\mathbf{v}}(s^2) \leq D_{\mathbf{0}}(s^2)+C(\kappa)^{\Lambda}\sqrt{\epsilon} \leq N_{\mathbf{0}}(2s^2)+C(\kappa)^{\Lambda}\sqrt{\epsilon},\label{eq new three point one}
    \end{align}
    and similarly $N_{\mathbf{v}}(s^2)\geq N_{\mathbf{x}}(\frac{1}{2}s^2)-C(\kappa)^{\Lambda}\sqrt{\epsilon}$ for all $s\in [\kappa,4]$. In the above computation, we used
    \begin{align}
        N(\tau)\leq D(\tau)\leq N(2\tau),\label{eq-comparison-of-frequency-and-doubling-index}
    \end{align}
    which follows from the monotonicity of the frequency and \eqref{eq-doubling-index-and-frequency}. 
    
    Since $|N_{\mathbf{0}}(10^{-2}\kappa^2)-N_{\mathbf{0}}(10^2)|\leq \epsilon$, \eqref{freqatnearbypoints} holds for any $\mathbf{v} \in V\cap P(\mathbf{0},10)$ because of \eqref{eq new three point one}, \eqref{eq-comparison-of-frequency-and-doubling-index} and the monotonicity of the frequency. As a consequence, $\mathcal{E}_s^{k,1}(\mathbf{v}) \leq C(\kappa)^{\Lambda}\sqrt{\epsilon}$, and $u$ is $(k,C(\kappa)^\Lambda \sqrt{\epsilon}, s)$-symmetric at $\mathbf{v}$ with respect to $V$ by the proof of \ref{thm: sym split equiv-1}.

    If instead $V=L^k \times \mathbb{R}$,  Lemma \ref{lemma-comparison-of-caloric-energy} with $r \leftarrow 10$, $\sigma = 1$, $\tau \leftarrow 9 \cdot 10^3$, and $\theta \leftarrow 10^{-1}$, we have
    \begin{align*}
        \int_{\mathbb{R}^n} |\partial_t u|^2 d\nu_{0,t;t-s^2} \leq \int_{\mathbb{R}^n} |\partial_t u|^2 d\nu_{0,t;t-9 \cdot10^3} \leq C \int_{\mathbb{R}^n} |\partial_t u|^2 d\nu_{\mathbf{0};-10^4}
    \end{align*}
    for all $s\in [\kappa,9]$ and $t\in [-100,100]$. Given this, the proof of \eqref{freqatnearbypoints} proceeds as the case $V=L^k \times \{0\}$, using Lemma \ref{frequencydirectionalestimate}.
\end{proof}

\subsection{Propagation of symmetry}
In the following lemma, we prove that almost symmetry propagates to a smaller scale, with some loss in the estimates. Furthermore, if we assume that the frequency is already pinched, then almost symmetry propagates to a larger scale as well.
\begin{lemma}\label{lemma: propagation of symmetry and pinching}
    For any $\alpha>0$, $0<\beta\leq 1$, and $\delta>0$, the following statements hold. Suppose $u$ is a caloric function and $r>0$ is such that $N(10^{5}r^2)\leq \Lambda$. Fix $\mathbf{x}\in P(\mathbf{0},r)$. 
    \begin{enumerate}[label={(\arabic*)}]
        \item \label{propagation of symmetry}If $u$ is $(k,\delta,r)$-symmetric with respect to $V\in {\rm Gr}_{\mathcal{P}}(k)$ at $\mathbf{x}$, then $u$ is $(k,\beta^{-2\Lambda}\delta ,s)$-symmetric with respect to $V$ at $\mathbf{x}$ for all $s\in [\beta r,r]$.

        \item \label{upwardpropagation} If {$0<r_1 \leq r$, $|N_{\mathbf{x}}(r_1^2)-N_{\mathbf{x}}(4r^2)|< \delta<C^{-\Lambda}$, and for some $r_2 \in [r_1,r]$, $u$ is $(k,\delta,r_2)$-symmetric at $\mathbf{x}$ with respect to $V\in{\rm Gr}_{\cP}(k)$}, then, for all $s\in [r_1,r]$, $u$ is $(k,C^\Lambda\delta,s)$-symmetric at $\mathbf{x}$ with respect to $V$. 
    \end{enumerate}
\end{lemma}

\begin{proof} 
    \ref{propagation of symmetry} We may assume $\mathbf{x}=\mathbf{0}$ and $r=1$. Because 
    $\partial_tu$ and $\nabla u$ are caloric, if $u$ is $(k,\delta,1)$-symmetric with respect to $L\times \mathbb{R}$, then, for any $s\in [\beta,1]$, we have 
    \begin{align*}
        \int_{\mathbb{R}^n} (|\pi_L \nabla u|^2 + |\partial_t u|^2)\,d\nu_{-s^2}\leq \delta \int_{\mathbb{R}^n}|u|^2 \,d\nu_{-1} \leq \delta \beta^{-2\Lambda}\int_{\mathbb{R}^n}|u|^2 \,d\nu_{-s^2}.
    \end{align*}
    The spatially symmetric case is identical.
    
    \ref{upwardpropagation} We may assume $\mathbf{x}=\mathbf{0}$ and $r_2=1$. We may also assume $u$ is $(k,\delta,1)$-symmetric with respect to $V=W\times\{0\}$ for some $k$-plane $W$, as the other case can be dealt with similarly. Then  
    \begin{align}
        \int_{\mathbb{R}^n} |\pi_{W} \nabla u|^2 \, d\nu_{-1} \leq \delta \int_{\mathbb{R}^n} u^2 \, d\nu_{-1}.\label{lemma-s-independence-spatial-symmetry1}
    \end{align}
    Theorem \ref{theorem-almost-frequency-cone-implies-unique-geometric-cone} implies that
    \begin{align}
        s^2\int_{\mathbb{R}^n} |\nabla (u-p_{m})|^2 \, d\nu_{-s^2} + s^4 \int_{\mathbb{R}^n} |\partial_t (u-p_{m})|^2 d\nu_{-s^2} \leq C^{\Lambda}\delta \int_{\mathbb{R}^n} u^2 \, d\nu_{-s^2},\label{lemma-s-independence-spatial-symmetry-comparison1}
    \end{align}
    for any {$s\in [r_1,r]$,}
    where $p_{m}$ is a caloric homogeneous polynomial of degree $m$, which is independent of $s$ and satisfies 
    \begin{align}
        \frac{1}{2}\int_{\mathbb{R}^n} u^2 \, d\nu_{-s^2}\leq \int_{\mathbb{R}^n} p_{m}^2 \, d\nu_{-s^2} \leq \int_{\mathbb{R}^n} u^2 \, d\nu_{-s^2} \label{lemma-s-independence-bound-of-homogeneous-part1}
    \end{align}
    for any {$s\in [r_1,r]$}.
    Combining \eqref{lemma-s-independence-spatial-symmetry1}, \eqref{lemma-s-independence-spatial-symmetry-comparison1}, and \eqref{lemma-s-independence-bound-of-homogeneous-part1}, we obtain
    \begin{equation*} 
        \int_{\mathbb{R}^n} |\pi_{W}\nabla p_{m}|^2 \, d\nu_{-1} \leq C^\Lambda \delta \int_{\mathbb{R}^n} p_{m}^2 \, d\nu_{-1}.
    \end{equation*}
    Because $p_m$ is homogeneous, it follows that
    \begin{equation} \label{eq: planeforhomogeneous}
        s^2\int_{\mathbb{R}^n} |\pi_{W}\nabla p_{m}|^2 \, d\nu_{-s^2} \leq C^\Lambda \delta \int_{\mathbb{R}^n} p_{m}^2 \, d\nu_{-s^2}
    \end{equation}
    for any $s\in [r_1,r]$. By \eqref{lemma-s-independence-spatial-symmetry-comparison1} and \eqref{eq: planeforhomogeneous}, for any {$s\in [r_1,r]$}
    \begin{align*}
        s^2\int_{\R^n} |\pi_{W}\nabla u|^2\,d\nu_{-s^2}
        \le C^\Lambda\delta\int_{\R^n} u^2\,d\nu_{-s^2}.
    \end{align*}
\end{proof}

\subsection{Sharp cone splitting inequality}\label{cone-splitting-inequality}
In this subsection, we prove a sharp \textit{cone splitting inequality} ({cf.} Theorem \ref{theorem-cone-splitting-inequality}) which implies Theorem \ref{thm: sym split equiv}\ref{thm: sym split equiv-1}. To motivate the ideas, we first consider the exact case. Suppose the frequency is constant at two distinct points $\mathbf{x}_1\neq \mathbf{x}_2$. By Lemma \ref{lemma-monotonicity-formulae-for-the-energy-functionals} and Corollary \ref{cor: nearby same m}, $u$ is an eigenfunction of $2\tau_i\Delta_{f_i}$ with the same eigenvalue $m\in \N_0$, i.e.,
\begin{align*}
    2\tau_i\Delta_{f_i} u+ mu=0,
\end{align*}
where $f_i = f_{\mathbf{x}_i}$, $\tau_i=t_i-t$. The following computation shows that $u$ is spatially invariant in the direction of $(x_2-x_1)$ if $x_2\neq x_1$; similarly, $u$ is static if $t_1\neq t_2:$
\begin{subequations}
    \begin{align}
        \cd u \cdot (x_2-x_1)&=4\tau_2 \Delta_{f_2}u-4\tau_1 \Delta_{f_1}u-[2\tau_1 \Delta_{f_1},2\tau_2 \Delta_{f_2}]u=0,\label{eq-spatial-cone-splitting-homogeneous}\\
        2(t_2-t_1)\pdt u&=[2\tau_1 \Delta_{f_1},2\tau_2 \Delta_{f_2}]u+2 \tau_1 \Delta_{f_1}u-2\tau_2 \Delta_{f_2}u=0.\label{eq-temporal-cone-splitting-homogeneous}
    \end{align}
\end{subequations}
 
We shall make the arguments above effective. We first prove that homogeneity at a point arises from the smallness of the frequency pinching at that point. 

\begin{lemma}\label{lemma-detecing-homogeneity}
    Let $u$ be a caloric function with $N(\tau_2)-N(\tau_1)\le \delta$ for some $0<\tau_1<\frac{\tau_2}{2}$, and let $m\in \N_0$ be as in Lemma \ref{lemma-refined-monotonicity-of-frequency}\ref{lemma: pinching int-2}. Then, for any $\tau\in (\tau_1,\tau_2)$,
    \begin{align*}
        \int_{\R^n} v^2\,d\nu_{-\tau} \le \frac{C\tau_2}{\tau_2-\tau}\delta \int_{\R^n} u^2\,d\nu_{-\tau_2},
    \end{align*}
    where $v \coloneqq 2\tau\Delta_f u + mu$.
\end{lemma}

\begin{proof}
    By parabolic rescaling, we may assume $\tau_2=1$. Since $v$ is caloric, we may apply the monotonicity formula in Lemma \ref{lemma-monotonicity-formulae-for-the-energy-functionals} so that
    for any $s\in(-1,-\tau_1)$, we obtain
    \begin{align*}
        \int_{\R^n} v^2\,d\nu_{s}&\leq  \frac{1}{s+1}\int_{-1}^s\int_{\R^n} v^2\,d\nu_{t}dt\leq \frac{2}{s+1}\int_{-1}^s\int_{\R^n} \verts{2\tau\Delta_fu+N(\tau)u}^2 +(N(\tau)-m)^2 u^2 \,d\nu_{t} dt\\
        &\leq \frac{2}{s+1} H(1)\int_{-1}^s \tau N'(\tau)\,dt
        + C\delta^2 H(1)\leq \frac{C}{s+1} \delta H(1).
    \end{align*}
\end{proof}

The following lemma allows us to get rid of one of the $2\tau_i\Delta_{f_i}$ in $2\tau_1\Delta_{f_1}(2\tau_2\Delta_{f_2})$ in the commutator identities \eqref{eq-spatial-cone-splitting-homogeneous} and \eqref{eq-temporal-cone-splitting-homogeneous}.
\begin{lemma}\label{lemma-integral-f-Laplacian-control}
    Suppose $u$ is a caloric function. Then, for any $0<\tau <\tau_1$, we have 
    \begin{align*}
        \int_{\R^n} (2\tau \Delta_{f}u)^2 \, d\nu_{-\tau}\leq \frac{4\tau \tau_1}{ (\tau_1-\tau)^2} \int_{\R^n}  u^2 \, d\nu_{-\tau_1}.
    \end{align*}
\end{lemma}

\begin{proof}
    By parabolic rescaling, we may assume $\tau_1=1$. Pick $-1<s<t< 0$. Since $2\tau\Delta_fu$ is caloric, by the monotonicity formula in Lemma \ref{lemma-monotonicity-formulae-for-the-energy-functionals},
    \begin{align*}
        \int_{\R^n} (2\tau \Delta_{f}u)^2 \,d\nu_{-\tau}&\leq \frac{1}{t-s} \int_{s}^{t}\int_{\R^n} (2\tau \Delta_{f}u)^2 \,d\nu_{r}dr\leq\ \frac{1}{t-s} \int_{s}^{t} |r| E'(|r|) \,dr\\
        &
        \leq \frac{-s}{t-s} E(-s)= \frac{2s^2}{t-s} \int_{\R^n} \verts{\cd u}^2 \, d\nu_{s}\leq\ \frac{2s^2}{(t-s)(s+1)} \int_{-1}^{s}\int_{\R^n} \verts{\cd u}^2 \,d\nu_{t} dt\\
        &\leq \frac{s^2}{(t-s)(s+1)}  \int_{\R^n}  u^2 \, d\nu_{-1}.
    \end{align*}
    Noting that $s=\frac{2t}{1-t}$ minimizes the coefficient above, we get the desired inequality.
\end{proof} 

In the following proposition, we prove a cone splitting inequality for two points.
\begin{proposition}\label{proposition-spatial-splitting-two-points}
    Let $u$ be a caloric function. Fix $\mathbf{x}_2 \in P(\mathbf{x}_1,\tfrac{1}{10}r)$.  
    If $N_{\mathbf{x}_j}(r^2)\le \Lambda$, and $\delta\coloneqq \mathcal{E}_r(\{\mathbf{x}_1,\mathbf{x}_2\})\le\bar\delta(\Lambda)=C^{-\Lambda}$,
    then, for $j=1,2$, we have
    \begin{align*}
        \int_{\R^n} |(x_2-x_1)\cdot\nabla u|^2+|(t_2-t_1)\cdot\partial_t u|^2\,d\nu_{\mathbf{x}_j;t_j-r^2}
        \le C(1+\Lambda^2) \delta \int_{\R^n} u^2\,d\nu_{\mathbf{x}_j;t_j-2r^2}.
    \end{align*}
\end{proposition}

\begin{proof}
    Without loss of generality, we assume that $r=1$. We only prove the first inequality for $j=1$. 
    By Corollary \ref{cor: nearby same m}, if $\delta\le C^{-\Lambda}$, there is a common $m\in \N_0$ such that 
    \begin{align*}
        \sup_{\tau\in\left[\frac{1}{8},8\right]}|N_{\mathbf{x}_j}(\tau)-m|<C\delta.
    \end{align*}
    For $i=1,2$, denote $f_i=f_{\mathbf{x}_i},\tau_i=t_i-t$, $v_i=2(t_i-t)\Delta_{f_i}u+mu$. Recall that
    \begin{align}
        \begin{split}
            \cd u \cdot (x_2-x_1)&=4\tau_2 \Delta_{f_2}u-4\tau_1 \Delta_{f_1}u-[2\tau_1 \Delta_{f_1},2\tau_2 \Delta_{f_2}]u\\
            &=-2\tau_1\Delta_{f_1}v_2+2\tau_2\Delta_{f_2}v_1-(2-m)v_1 +(2-m)v_2.\label{eq-spatial-splitting-proof-spatial-derivative}
        \end{split}
    \end{align}
    Let $\eta=\frac{9}{8}$. Since $v_i$ is caloric, by using Lemma \ref{lemma-integral-f-Laplacian-control}, changing the base points (Lemma \ref{lemma-comparison-of-caloric-energy} with $\theta=\frac{1}{8}$), and using Lemma \ref{lemma-detecing-homogeneity}, we get
    \begin{align*}
        \int_{\R^n} \verts{2\tau_1\Delta_{f_1}v_2}^2 \, d\nu_{\mathbf{x}_1;t_1-1}
        &\le C \int_{\R^n} v_2^2  \, d\nu_{\mathbf{x}_1;t_1-\eta}
        \le C \int_{\R^n} v_2^2\, d\nu_{\mathbf{x}_2;t_2-\eta^2} \le C\delta
        \int_{\R^n} u^2\,d\nu_{\mathbf{x}_2;t_2-\eta^3}\\  
        &\le  C\delta\int_{\R^n} u^2\,d\nu_{\mathbf{x}_1;t_1-2}.
    \end{align*}
    Using similar techniques, we can estimate
    \begin{align*}
        \int_{\R^n} \verts{2\tau_2\Delta_{f_2}v_1}^2 \,d\nu_{\mathbf{x}_1;t_1-1}  
        &\le C \delta \int_{\R^n} u^2\,d\nu_{\mathbf{x}_1;t_1-2};\\
        \int_{\R^n} (v_1^2+v_2^2)\,d\nu_{\mathbf{x}_1;t_1-1}
        &\le C\delta \int_{\R^n} u^2\,d\nu_{\mathbf{x}_1;t_1-2}.
    \end{align*}
    Combining the estimates above and using \eqref{eq-spatial-splitting-proof-spatial-derivative}, we get the desired estimates.
\end{proof}

Applying the above Proposition multiple times, we get the following cone splitting inequality. 
\begin{theorem}\label{theorem-cone-splitting-inequality}
    Let $u$ be a caloric function with $N_{\mathbf{x}_0}(10^5r^2)\le \Lambda$. If $\{\mathbf{x}_i\}_{i=0}^K\subset P(\mathbf{x}_0, \frac{1}{10}r)$ is $(k,\alpha r)$-independent and $K=k$, then
    \begin{align*}
        r^2 \int_{\R^n} |\pi_{L}\nabla u|^2\,d\nu_{\mathbf{x}_0;t_0-r^2}
        \le C^{\Lambda} \alpha^{-n} \mathcal{E}_r(\{\mathbf{x}_i\}) \int_{\R^n} u^2\,d\nu_{\mathbf{x}_0;t_0-r^2},
    \end{align*}
    where $L={\rm Span}\{x_1-x_0,\cdots,x_K-x_0\}$. If $\{\mathbf{x}_i\}_{i=0}^K$ is $(k,\alpha r)$-temporally independent and $K=k-2$, then
    \begin{align*}
        \int_{\R^n} r^2|\pi_{L}\nabla u|^2+r^4|\partial_t u|^2\,d\nu_{\mathbf{x}_0;t_0-r^2}
        \le C^{\Lambda} \alpha^{-n} \mathcal{E}_r(\{\mathbf{x}_i\}) \int_{\R^n} u^2\,d\nu_{\mathbf{x}_0;t_0-r^2}. 
    \end{align*}
\end{theorem}

\begin{proof}
    The inequalities hold by applying Proposition \ref{proposition-spatial-splitting-two-points} for $\mathbf{x}_1\leftarrow \mathbf{x}_0$, $\mathbf{x}_2\leftarrow \mathbf{x}_j$ for $j=1,\cdots,K$, and by Lemma \ref{lemma-linear-independence} and \eqref{eq-doubling-index-and-frequency}.
\end{proof}

\begin{proof}[Proof of Theorem \ref{thm: sym split equiv}\ref{thm: sym split equiv-1}]
    Apply Theorem \ref{theorem-cone-splitting-inequality}.
\end{proof}

\subsection{Existence of a plane of symmetry}
In the lemma below, we prove that there exists a plane of symmetry that is uniform across scales and locations. This will be the main tool for proving the neck decomposition theorem ({cf.} Theorem \ref{cballcovering}).
\begin{lemma} \label{newlineup} 
    For any $\delta,\eta,\kappa \in (0,1]$ and $\Lambda \in [1,\infty)$, the following holds for any $r_0>0$ and any caloric function $u$ satisfying $N(10^5 \kappa^{-2} r_0^2) \leq \Lambda$. Suppose $\mathcal{C} \subseteq P(\mathbf{0},r_0)$ and $\mathbf{r}_{\bullet}:\mathcal{C} \to [0,r_0]$ is a function satisfying the following for some $m \in \mathbb{N}_0$:
    \begin{enumerate}[label=(\alph*)]
        \item \label{hypothesisplaneofsymmetrya} there exists $\mathbf{x}_0 \in \mathcal{C}$ and $s_0 \in [\mathbf{r}_{\mathbf{x}_0},r_0]$ such that $u$ is $(k,\delta,s_0)$-symmetric with respect to $V$; \label{plane of symmetry a}
        
        \item \label{hypothesisplaneofsymmetryb}
        $\sup_{\mathbf{x}\in \mathcal{C}} \sup_{s\in [10^{-3}\kappa \mathbf{r}_{\mathbf{x}},10^3 \kappa^{-1}r_0]} 
        |N_{\mathbf{x}}(s^2)-m|<\delta$.
    \end{enumerate}
    Then the following statements hold for any $\mathbf{x} \in \mathcal{C}$:
    \begin{enumerate}[label={(\arabic*)}]
        \item \label{existence of plane of symmetry} for any $s\in [\kappa \mathbf{r}_{\mathbf{x}},\kappa^{-1} r_0]$, $u$ is $(k,C^{\Lambda} \delta,s)$-symmetric with respect to $V$ at $\mathbf{x}$;
    
        \item \label{plane of symmetry inside pinched points} for any $s\in [\mathbf{r}_{\mathbf{x}},r_0]$,
        \begin{align*}
            (\mathbf{x}+V)\cap P(\mathbf{x},10s) \subset \{\mathbf{y}\in P(\mathbf{x}, 10s)\colon |N_{\mathbf{y}}(\kappa^2 s^2)-m|<C(\kappa)^{\Lambda}\sqrt{\delta}\};
        \end{align*}
        
        \item \label{containment of plane of symmetry} if in addition we assume that $u$ is not $(k+1,\eta,s_0)$-symmetric at $\mathbf{x}_0$, then, for any $s\in [\mathbf{r}_{\mathbf{x}},10^{-2}r_0]$ and $\zeta \geq C^{\Lambda}\sqrt{\delta}$,
        \begin{align*}
            \{\mathbf{y} \in P(\mathbf{x},10s): \mathcal{E}_{100s}(\mathbf{y})< \zeta\}\subseteq  P(\mathbf{x}+ V , C^{\Lambda} \eta^{-\frac{1}{n}}\zeta^{\frac{1}{n}} s).
        \end{align*}
    \end{enumerate}
\end{lemma}

\begin{proof} 
    By parabolic rescaling, we may assume $r_0=1$. 
    
    \ref{existence of plane of symmetry} By Lemma \ref{lemma: propagation of symmetry and pinching}\ref{upwardpropagation} with $r \leftarrow 100$, $r_1 \leftarrow \mathbf{r}_{\mathbf{x}_0}$, $r_2 \leftarrow s_0$, $u$ is $(k,C^{\Lambda}\delta,100)$-symmetric at $\mathbf{x}_0$ with respect to $V$. By Lemma \ref{lemma-comparison-of-caloric-energy}, $u$ is $(k,C^{\Lambda}\delta,1)$-symmetric with respect to $V$ at any $\mathbf{x} \in P(\mathbf{0},1)$. The claim then follows by applying Lemma \ref{lemma: propagation of symmetry and pinching}\ref{upwardpropagation} to each point $\mathbf{x}\in\mathcal{C}$. 

    \ref{plane of symmetry inside pinched points} 
    Apply \ref{existence of plane of symmetry} and  Theorem \ref{thm: sym split equiv}\ref{thm: sym split equiv-2}.
    
    \ref{containment of plane of symmetry} Suppose $\mathbf{x} \in \mathcal{C}$, $s\in [\mathbf{r}_{\mathbf{x}},10^{-2}]$, $\mathbf{y} \in P(\mathbf{x},10s)$, and $\cE_{100s}(\mathbf{y})<\zeta$ for some $\zeta \geq C^{\Lambda}\sqrt{\delta}$. Set $\epsilon\coloneqq  C^{\Lambda}\eta^{-\frac{1}{n}}\zeta^{\frac{1}{n}}$, and assume by way of contradiction that $\mathbf{y} \notin P(\mathbf{x}+V,\epsilon s)$. By applying \ref{plane of symmetry inside pinched points} twice with $s \leftarrow s$ and $s\leftarrow 100s$, there exists a $(k,100 s)$-independent set $\{\mathbf{v}_i\}_{i \in I} \subseteq (\mathbf{x}+V)\cap P(\mathbf{x},10 s)$ containing $\mathbf{x}$ and satisfying $\mathcal{E}_{100 s}(\mathbf{v}_i)< C^{\Lambda}\sqrt{\delta} \leq \zeta$ for all $i\in I$. Thus $\mathcal{E}_{100s}^{k+1,\epsilon}(\mathbf{x})<\zeta$, so by Theorem \ref{thm: sym split equiv}\ref{thm: sym split equiv-1} and Lemma \ref{lemma: propagation of symmetry and pinching}\ref{upwardpropagation}, $u$ is $(k+1,C^{-\Lambda}\eta,s)$-symmetric at $\mathbf{x}$. By Lemma \ref{lemma: propagation of symmetry and pinching}\ref{upwardpropagation}, it follows that $u$ is $(k+1,C^{-\Lambda}\eta,100)$-symmetric at $\mathbf{x}$, hence by Lemma \ref{lemma-comparison-of-caloric-energy}, $u$ is also $(k+1,C^{-\Lambda}\eta,1)$-symmetric at $\mathbf{x}_0$. Another application of Lemma \ref{lemma: propagation of symmetry and pinching}\ref{upwardpropagation} finally yields that $u$ is $(k+1,\eta,s_0)$-symmetric at $\mathbf{x}_0$, a contradiction. 
\end{proof}

\subsection{\texorpdfstring{$L^2$}{L2}-subspace approximation}
In this subsection, we make precise use of Theorem \ref{theorem-cone-splitting-inequality} to prove Lemma \ref{lem:beta}. The lemma states that the average frequency pinching bounds from above the Jones beta number (defined below), which quantifies, in the $L^2$ sense, how well the support of a measure $\mu$ can be approximated by an affine subspace. The strategy of the proof is from \cite[Section 7]{naber-valtorta-2017-rectifiable-for-harmonic-maps}.
\begin{definition} 
    For any measure $\mu$ on $\mathbb{R}^{n+1}$, an integer $0\leq k\leq n+2$, and any $\mathbf{x}\in \mathbb{R}^{n+1}$, we define the \textit{parabolic Jones $\beta$-number} by
    \begin{align*}
        \beta_{\mathcal{P},k}^2(\mathbf{x},r)\coloneqq \beta_{\mathcal{P},k}^2(\mathbf{x},r;\mu)\coloneqq \inf_{V\in \text{Aff}_{\mathcal{P}}(k)} \frac{1}{r^{k}} \int_{P(\mathbf{x},r)} \left( \frac{d_{\mathcal{P}}(\mathbf{y},V)}{r} \right)^2 d\mu(\mathbf{y}).
    \end{align*}
\end{definition}

\begin{lemma} \label{lem:beta}
    Let $k\in\{n,n+1\}$. For any finite measure $\mu$ supported in $P(\mathbf{0},1)$ on $\mathbb{R}^{n+1}$ and any caloric function $u$ satisfying $N_{\mathbf{0}}(10^{5})\leq\Lambda$, if $u$ is $(k,\epsilon, 2r)$-symmetric but not $(k+1,\eta,r)$-symmetric at $\mathbf{0}$, then, for any $P(\mathbf{x},r)\subseteq P(\mathbf{0},2)$, if $\epsilon \leq C^{-\Lambda} \eta$, we have
    \begin{align*}
        \beta_{\mathcal{P},k}^{2}(\mathbf{x},r;\mu)\leq \frac{C^\Lambda}{r^k \eta}\left(\int_{P(\mathbf{x},r)}\mathcal{E}_{20r}(\mathbf{y})d\mu(\mathbf{y})\right).
    \end{align*}
\end{lemma}

We prove the lemma first by characterizing $\beta$ in terms of the eigenvalues of the \textit{covariance matrix} $Q$ of $\mu$ defined below and then estimating the eigenvalues.
\begin{definition}\label{defn: center of mass and covariance}
    For a probability measure $\mu$ supported on $P(\mathbf{0},1)$, we define the corresponding center of mass and covariance matrix by
    \begin{align} \label{def:cmandvariance}
        x_{\rm cm}\coloneqq \int_{P(\mathbf{0},1)}x\,d\mu(x,t),\quad 
        Q \coloneqq \int_{P(\mathbf{0},1)}(x-x_{\rm cm})(x-x_{\rm cm})^T\,d\mu(x,t),
    \end{align}
    respectively. Let $\lambda_1\le\cdots\le\lambda_n$ be the eigenvalues of $Q$, with corresponding unit eigenvectors $v_1,\dots,v_n$. 
\end{definition} 

\begin{lemma}\label{lemma beta as sum of eigenvalues}
    Let $\mu$ be a probability measure supported in $P(\mathbf{0},1)$. Then
    \begin{align}
        \beta_{\mathcal{P},k}^2(\mathbf{0},1;\mu)\leq \lambda_1+\cdots+\lambda_{n+2-k}.
        \label{eq beta as sum of eigenvalues}
    \end{align}
\end{lemma}

\begin{proof}
    For any $L \in \operatorname{Gr}(k-2)$, we have $d_{\mathcal{P}}(y,x_{\operatorname{cm}}+V) = |\pi_V^{\perp}(\mathbf{y}-\mathbf{x}_{\operatorname{cm}})|$, so that
    \begin{align*}
        \beta_{\mathcal{P},k}^2(\mathbf{0},1)
        &\leq  \inf_{L\in {\rm Gr}(k-2)}
        \int_{P(\mathbf{0},1)} |\pi_L^{\perp}(y-x_{\rm cm})|^2\,d\mu(\mathbf{y}) = \inf_{L\in {\rm Gr}(n+2-k)}
        \int_{P(\mathbf{0},1)} |\pi_L(y-x_{\rm cm})|^2\,d\mu(\mathbf{y}) \\
        &\leq \lambda_1+\cdots+\lambda_{n+2-k}. 
    \end{align*}
\end{proof}

In the lemma below, we estimate the eigenvalues of $Q$ in terms of the average frequency pinching.
\begin{lemma}\label{lemma eigenvalue estimate}
    Let $v_j$ be the eigenvector of the covariance matrix $Q$ with eigenvalue $\lambda_j$. Then
    \begin{align*}
        \lambda_j \int_{\R^n} ( v_j \cdot \nabla u)^2\,d\nu_{-400}
        \le C^\Lambda\int_{P(\mathbf{0},1)}\mathcal{E}_{20}(\mathbf{x})\,d\mu(\mathbf{x})\cdot\int_{\R^n} u^2\,d\nu_{-400}.
    \end{align*}
\end{lemma}

\begin{proof}
    For any $z\in \R^n,t_*<0$,
    \begin{align*}
        \lambda_j  v_j \cdot \nabla u(z,t_*)
        &= Qv_j \cdot \nabla u(z,t_*) 
        = \int_{P(\mathbf{0},1)} \left( (x - x_{\operatorname{cm}})\cdot v_j \right)\nabla u(z,t_*) \cdot (x-x_{\operatorname{cm}})\,d\mu(x,t)\\
        &= \int_{P(\mathbf{0},1)}\int_{P(\mathbf{0},1)}\left( (x - x_{\operatorname{cm}})\cdot v_j \right) \nabla u(z,t_{\ast})\cdot ( x-y) \,d\mu(x,t)d\mu(y,s).   
    \end{align*}
    Integration against $\nu_{t_{\ast}}$ and Cauchy--Schwarz yield
    \begin{align}
    \begin{split} \label{eq:L2bestplane1}
        &\lambda_j^2 \int_{\R^n}  ( v_j \cdot \nabla u(z,t_*))^2\,d\nu_{t_*}(z) \\
        &= \int_{\R^n}\left(\int_{P(\mathbf{0},1)}\int_{P(\mathbf{0},1)}\left( (x - x_{\operatorname{cm}})\cdot v_j \right) \nabla u(z,t_{\ast})\cdot ( x-y) \,d\mu(\mathbf{x})d\mu(\mathbf{y})\right)^2 d\nu_{t_*}(z)\\ 
        &\le \left( \int_{P(\mathbf{0},1)} ((x-x_{\operatorname{cm}})\cdot v_j)^2d\mu(\mathbf{x}) \right) \int_{\R^n}  \int_{P(\mathbf{0},1)}\int_{P(\mathbf{0},1)} \left( \nabla u(z,t_{\ast})\cdot (x-y) \right)^2  \,d\mu(\mathbf{x})d\mu(\mathbf{y})  d\nu_{t_*}(z) \\
        &= \lambda_j \int_{\R^n}  \int_{P(\mathbf{0},1)}\int_{P(\mathbf{0},1)} \left( \nabla u(z,t_{\ast})\cdot (x-y) \right)^2  \,d\mu(\mathbf{x})d\mu(\mathbf{y})  d\nu_{t_*}(z).
        \end{split}
    \end{align}
    On the other hand, for any fixed $\mathbf{x},\mathbf{y}\in P(\mathbf{0},1)$, Lemma \ref{lemma-comparison-of-caloric-energy} and Proposition \ref{proposition-spatial-splitting-two-points} with $r\leftarrow  20$ yield
    \begin{align*}
        \int_{\R^n} (\nabla u(z,-400)\cdot ( x-y ) )^2 \,d\nu_{-400}(z)
        &\le \int_{\R^n} (\nabla u(z,-400)\cdot(x-y))^2 \,d\nu_{\mathbf{x};t-440}(z)
        \\& \le C^\Lambda \mathcal{E}_{20}(\{\mathbf{x},\mathbf{y}\}) \int_{\R^n} u^2\, d\nu_{-400}.
    \end{align*}
    Integrating over $\mathbf{x},\mathbf{y} \in P(\mathbf{0},1)$ and combining the result with \eqref{eq:L2bestplane1} yields
    \begin{align*}
        \lambda_j \int_{\R^n} ( v_j \cdot \nabla u)^2\,d\nu_{-400} 
        &\le \int_{P(\mathbf{0},1)}\int_{P(\mathbf{0},1)}\int_{\R^n}((x-y)\cdot \nabla u(z,-400))^2 d\nu_{-400}(z)\,d\mu(\mathbf{x})d\mu(\mathbf{y}) \\
        &\le C^\Lambda \int_{P(\mathbf{0},1)}\int_{P(\mathbf{0},1)}
        (\mathcal{E}_{20}(\mathbf{x})+\mathcal{E}_{20}(\mathbf{y}))\,d\mu(\mathbf{x})d\mu(\mathbf{y})
        \cdot \int_{\R^n} u^2\,d\nu_{-400}\\ 
        &\le C^\Lambda \int_{P(\mathbf{0},1)}
        \mathcal{E}_{20}(\mathbf{x})\,d\mu(\mathbf{x})
        \cdot \int_{\R^n} u^2\,d\nu_{-400}.
    \end{align*}
\end{proof}

\begin{proof}[Proof of Lemma \ref{lem:beta}]
    By translation and rescaling, we may assume $\mathbf{x}=\mathbf{0}$ and $r=1$. By rescaling $\mu$, we may also assume that $\mu$ is a probability measure. By shifting $\mu$ we may assume $\mathbf{x}_{\rm cm}=\mathbf{0}$. 

    Recall that $u$ is $(k,\epsilon,2)$-symmetric. We first consider the case where the symmetry is temporal. Because $u$ is not $(k+1,\eta,1)$-symmetric at $\mathbf{0}$,
    \begin{align*}
        \int_{\R^n} |\pi_L\nabla u|^2 \,d\nu_{-1}\ge (\eta-\epsilon) \int_{\R^n} u^2\,d\nu_{-1},
    \end{align*}
    where $L\coloneqq \operatorname{span}(v_{n+2-k},\dots,v_n)$. Thus there exists $j \in \{n+2-k,\dots,n\}$ such that
    \begin{align*}
        \int_{\mathbb{R}^n} |v_j \cdot \nabla u|^2 \,d\nu_{-1} \geq C^{-\Lambda}(\eta - \epsilon)\int_{\mathbb{R}^n}u^2 \,d\nu_{-1}.
    \end{align*}
    Combining this with Lemma \ref{lemma beta as sum of eigenvalues} and \ref{lemma eigenvalue estimate}, we get
    \begin{align*}
        \beta^2_{\mathcal{P},k}(\mathbf{0},1)
        \le \lambda_{n+2-k} \le C^\Lambda (\eta-\epsilon)^{-1} \int_{P(\mathbf{0},1)} \mathcal{E}_{20}(\mathbf{x})\,d\mu(\mathbf{x}).
    \end{align*}
    Suppose instead that the symmetry is spatial (so that $k=n$). Then
    \begin{align*}
        \int_{\R^n} \verts{\pd_t u}^2 \, d\nu_{-1}=\int_{\R^n} \verts{\Delta u}^2 \, d\nu_{-1}\leq C\int_{\R^n} \verts{\cd^2 u}^2 \, d\nu_{-1}\leq \int_{\R^n} \verts{\cd u}^2 \,d\nu_{-4}\leq \epsilon C^\Lambda \int_{\R^n} u^2\, d\nu_{-1},
    \end{align*}
    which contradicts our assumption that $u$ is not $(n+2,\eta,1)$-symmetric at $\mathbf{0}$. 
\end{proof}

\section{Neck Regions}\label{neck-structure-and-decomposition}
In this section, we define neck regions and study their structure. The neck regions (see Definition \ref{definition neck region}) serve as transition zones between quantitative singular regions and regular (i.e., highly symmetric) regions. They capture the presence of \textit{multi-scale symmetry} and \textit{homogeneity}. In the neck structure theorem (Theorem \ref{theorem-neck-structure}), we show that the singular sets associated with these neck regions align into 
$k$-dimensional structures across scales and locations, a property that is essential for establishing the content estimates.

Throughout this section, fix $\gamma \coloneqq10^{-10n}$. Define $\ol{P}(\mathcal{C},\mathbf{r}_{\bullet})\coloneqq \cup_{\mathbf{x}\in\mathcal{C}} \ol{P}(\mathbf{x},\mathbf{r}_{\mathbf{x}})$ when $\mathbf{r}_{\bullet}:\mathcal{C}\to [0,1]$. 

Our definitions of neck region and strong neck region are adaptations of \cite[Definition 7.10]{jiang-naber-2021-l2-curvature} and \cite[Definition 2.4]{cheeger-jiang-naber-2021-Sharp-quantitative} to the parabolic setting. However, in our setting we can use quantitative uniqueness (through Lemma \ref{newlineup}) to show that the approximating planes of symmetry can be taken independent of both scales and locations. This greatly simplifies our proof of the weak neck structure theorem (Theorem \ref{theorem-neck-structure}).

\begin{definition}[Neck Region]\label{definition neck region}
    Let $\mathcal{C}\subseteq P(\mathbf{x}_0,2r)$ be a closed subset such that $\mathbf{x}_0\in \cC$, $\mathbf{r}_{\bullet}:\mathcal{C}\to [0,\gamma r]$ be a continuous function and $k,m\in \N_0$, $\delta,\eta>0$. We say that the set
    \begin{align*}
        \mathcal{N} \coloneqq P(\mathbf{x}_0,2r)\setminus \overline{P}(\mathcal{C},\mathbf{r}_{\bullet})
    \end{align*}
    is an \textit{$(m,k,\delta,\eta)$-neck region (of scale $r$)} modeled on $V\in \text{Gr}_{\mathcal{P}}(k)$ if the following hold:
    \begin{enumerate}[label={(n\arabic*)}]
        \item $\{P(\mathbf{x},\gamma^{2} \mathbf{r}_{\mathbf{x}}) \}_{\mathbf{x}\in \mathcal{C}}$ is pairwise disjoint;\label{neck-vitali-covering} 

        \item $\sup_{s\in [\mathbf{r}_{\mathbf{x}},\gamma^{-3}r]} |N_{\mathbf{x}}(s^2)-m|< \delta$ for all $\mathbf{x} \in \mathcal{C}$; \label{neck-frequency-pinching}
        
        \item For all $s\in[\mathbf{r}_{\mathbf{x}},\gamma^{-3}r]$, $u$ is $(k,\delta,s)$-symmetric at $\mathbf{x}$ with respect to $V$ but not $(k+1,\eta,s)$-symmetric;
        \label{neck-k-sym}
       
        \item \label{neck-Hausdorff} For all $\mathbf{x}\in \mathcal{C}$ and $s\in [\mathbf{r}_{\mathbf{x}},\gamma^{-3}r]$ such that $P(\mathbf{x},s)\subseteq P(\mathbf{x}_0,2r)$, we have
        \begin{subequations}
            \begin{align} \label{neck-n2-CinV}
                \mathcal{C}\cap P(\mathbf{x},s) &\subseteq P(\mathbf{x}+V,\delta s), \tag{n4.a}\\
                (\mathbf{x}+V)\cap P(\mathbf{x},s)& \label{neck-n2-VinC} \subseteq \bigcup_{\mathbf{y}\in \mathcal{C}} P(\mathbf{y},10\gamma (s+\mathbf{r}_{\mathbf{y}})). \tag{n4.b}
            \end{align}
        \end{subequations}
    \end{enumerate} 
    We say $\mathcal{N}$ is a \textit{strong $(m,k,\delta,\eta)$-neck region} if \ref{neck-vitali-covering}-\ref{neck-k-sym} and \eqref{neck-n2-CinV} hold, and if in addition we have the following:
    \begin{enumerate}[label={(n\arabic*)}, ref=(n\arabic*)]\setcounter{enumi}{4}
        \item[(n4')]  For all $\mathbf{x} \in \mathcal{C}$ and $s\in [\mathbf{r}_{\mathbf{x}},\gamma^{-3}r]$ such that $P(\mathbf{x},s)\subseteq P(\mathbf{x}_0,2r)$, we have \label{neck-improved-Hausdorff}
        \begin{equation} \label{eq:improvedCinVinC}
            (\mathbf{x}+V)\cap P(\mathbf{x},s)\subseteq P(\mathcal{C},\gamma s);\tag{n4.b'}
        \end{equation}
        
        \item \label{neck-lipschitz} $\mathbf{r}_{\bullet}:\mathcal{C} \to [0,\gamma r]$ is parabolically $\delta$-Lipschitz.
    \end{enumerate} 
\end{definition}

\begin{example}
    If $p$ is a homogeneous caloric polynomial of degree $m$ that is invariant with respect to $V\in {\rm Gr}_{\cP}(k)$, then defining $\cC=\cC_0\coloneqq V$ and $\mathbf{r}\equiv 0$ yields an $(m,k,0,\eta)$-neck region modeled on $V$.
\end{example}

We refer to $\cC$ as the \textit{center set} associated to $\cN$, and define the subsets
\begin{align*}
    \cC_+\coloneqq \{\mathbf{x}\in \cC\colon \mathbf{r_x}>0\}, \qquad \cC_0\coloneqq \{\mathbf{x}\in \cC\colon \mathbf{r_x}=0\}.
\end{align*}
We also define the \textit{packing measure} associated to the center set $\cC$.
\begin{align}
    \mu \coloneqq \mu_0 + \mu_+ \coloneqq \mathcal{H}_{\mathcal{P}}^{n+1}|_{\mathcal{C}_0} + \sum_{\mathbf{x}\in \mathcal{C}_+} \mathbf{r}_{\mathbf{x}}^{n+1} \delta_{\mathbf{x}}.\label{eq-packing measure}
\end{align}

In the following lemma, we prove that restriction of a strong neck region still defines a strong neck region.
\begin{lemma}\label{lemma restriction of neck region}
    Suppose $\cN\coloneqq P(\mathbf{x}_0,2r)\setminus \ol{P}(\cC,\mathbf{r}_{\bullet})$ is a strong $(m,k,\delta,\eta)$-neck region with $\mathbf{r}_{\bullet} \leq \frac{1}{4}\gamma r$, and with \eqref{eq:improvedCinVinC} replaced by
    \begin{align}\label{eq:improvedCinVinCdoubleprime}
        (\mathbf{x}+V)\cap P(\mathbf{x},s)\subseteq P(\mathcal{C},\tfrac{\gamma}{10} s).\tag{n4.b''}
    \end{align}
    Then $\mathcal{N}'\coloneqq P(\mathbf{x}_0,\frac{r}{2})\setminus \bigcup_{\mathbf{y} \in \mathcal{C}\cap P(\mathbf{x}_0,\frac{r}{2})}\ol{P}(\mathbf{y},\mathbf{r}_{\mathbf{y}})$ is a strong $(m,k,\delta,\eta)$-neck region of scale $\frac{r}{4}$.
\end{lemma}

\begin{proof} 
    We may assume $\mathbf{x}_0 = \mathbf{0}$ and $r=1$. It suffices to justify \eqref{eq:improvedCinVinC} since the other statements hold because $\cN$ is a neck region. Fix $\mathbf{x} \in \widetilde{\mathcal{C}}\coloneqq \mathcal{C} \cap P(\mathbf{0},\frac{1}{2})$. Suppose $s\in [\mathbf{r}_{\mathbf{x}},\frac{1}{4}\gamma^{-3}]$ such that $P(\mathbf{x},s)\subset P(\mathbf{0},\frac{1}{2})$. Fix $\mathbf{z} \in (\mathbf{x}+ V)\cap P(\mathbf{x},s)$. Let $\mathbf{z}' \in (\mathbf{x}+V)\cap P(\mathbf{x},(1-\frac{\gamma}{2})s)$ such that $|\mathbf{z}-\mathbf{z}'|=\frac{\gamma}{2}s$. By \eqref{eq:improvedCinVinCdoubleprime}, there exists $\mathbf{y} \in \mathcal{C}$ such that $|\mathbf{y}-\mathbf{z}'|<\frac{\gamma}{10}s$. Moreover, $|\mathbf{y}-\mathbf{z}| \leq \gamma s$ and
    \begin{align*}
        |\mathbf{y}|\leq |\mathbf{y}-\mathbf{z}'|+|\mathbf{z}'-\mathbf{x}|+|\mathbf{x}| \leq \frac{\gamma}{10}s+\left(1-\frac{\gamma}{2}\right)s+(\tfrac{1}{2}-s) <\frac{1}{2}.
    \end{align*}
    Thus $\mathbf{y}' \in \widetilde{\mathcal{C}}$, and \eqref{eq:improvedCinVinC} follows.
\end{proof}

\subsection{Bi-Lipschitz property and Ahlfors regularity}
In this subsection, we prove that the center set $\cC$ of a neck region is bi-Lipschitz equivalent to the plane $V$ with respect to which the neck region is defined. As a consequence, the $k$-dimensional content of $\cC$ is bounded above. Moreover, we prove Ahlfors regularity for the packing measure $\mu$ associated with the center set. The proof of the lower Ahlfors regularity uses the strategy of \cite[Theorem 3.24]{jiang-naber-2021-l2-curvature}.

We record the definition of parabolically bi-Lipschitz maps and some properties of Lipschitz graphs.
\begin{definition}[Projection]\label{definition-projection}
    Given $V\in \text{Gr}_{\mathcal{P}}(k)$, we define the \textit{orthogonal projection} $\pi_V$ onto $V$ as follows.
    \begin{itemize}
        \item If $V = L \times \{0\}$ for some $k$-plane $L\subseteq \mathbb{R}^n$, then 
        \begin{align*}
            \pi_V(x,t)=(\pi_L(x),0),
        \end{align*}
        where $\pi_L: \mathbb{R}^n \to L$ is the standard orthogonal projection.

        \item If instead $V= L\times \mathbb{R}$ for some $(k-2)$-plane $L \subseteq \mathbb{R}^n$, then
        \begin{align*}
            \pi_V(x,t) = (\pi_L(x),t).
        \end{align*}
    \end{itemize}
    In both cases, we define the \textit{orthogonal complement projection}
    \begin{align*}
        \pi_V^{\perp} \coloneqq \text{id}_{\mathbb{R}^n\times \mathbb{R}} -\pi_V,
    \end{align*}
    so that for any $\mathbf{y} \in \mathbb{R}^n \times \mathbb{R}$, the parabolic distance to $V$ satisfies
    \begin{align*}
        d_{\mathcal{P}}(\mathbf{y},V)= |\pi_V^{\perp}(\mathbf{y})|.
    \end{align*}

    Given $V\in {\rm Gr}(k)$, for any $\mathbf{y}\in V$ and $s\geq 0$, we define the restricted ball as
    \begin{align}
        P^V(\mathbf{y},s)\coloneqq P(\mathbf{y},s)\cap V.\label{eq restricted ball}
    \end{align}
\end{definition}

\begin{definition}[{\cite[Definition 3.1]{mattila-2022-parabolic-rectifiability}}]\label{definition-rectifiability-a}
    We say that $G\subseteq \mathbb{R}^{n+1}$ is a \textit{$(k,\ell)$-Lipschitz graph} if there exist an affine $k$-plane $V\in {\rm Aff}_{\cP}(k)$, a subset $A\subseteq V$, and an $\ell$-Lipschitz function (with respect to $d_{\mathcal{P}}$) $\mathbf{f}\colon A\to V^\perp$ such that 
    \begin{align*}
        G=\{\mathbf{v}+\mathbf{f}(\mathbf{v})\colon \mathbf{v}\in A\}.
    \end{align*}
\end{definition}

We now prove the following structural result for $\cC$.
\begin{theorem}[Neck Structure Theorem]\label{theorem-neck-structure}
    Let $\mathcal{N}=P(\mathbf{x}_{0},2r)\setminus\overline{P}(\mathcal{C},\mathbf{r}_{\bullet})$ be a (strong) $(m,k,\delta,\eta)$-neck region modeled on $V\in {\rm Gr}_{\cP}(k)$. Then the following statements hold if $\delta \leq \overline{\delta}$:
    \begin{enumerate}[label={(\roman*)}]
        \item \label{theorem-neck-structure-bi-lipschitz} For all $\mathbf{x},\mathbf{y} \in \mathcal{C}$, we have
        \begin{align}
            (1-C\delta)|\mathbf{x} - \mathbf{y}| \leq | \pi_V(\mathbf{x})-\pi_V(\mathbf{y}) | \leq |\mathbf{x} - \mathbf{y}|.\label{eq-theorem-neck-structure-bi-lipschitz}
        \end{align}
        
        \item $\mathcal{C}$ is a $(k,C\delta)$-graph over $V$, 
        and satisfies the Ahlfors regularity upper bound 
        \begin{align*}
            \mu(P(\mathbf{x},s))\leq Cs^{k}
        \end{align*}
        for all $\mathbf{x} \in \mathcal{C}$ and $s\in [\mathbf{r}_{\mathbf{x}},r]$. 
        \label{theorem-neck-structure-graph}
    \end{enumerate}
\end{theorem}

\begin{proof}
    By rescaling and translation, we can assume $r=1$ and $\mathbf{x}_0=\mathbf{0}$. 
    
    \ref{theorem-neck-structure-bi-lipschitz} The upper bound in \eqref{eq-theorem-neck-structure-bi-lipschitz} 
    holds trivially. 
    To prove the lower bound, choose any $\mathbf{x}_1,\mathbf{x}_2\in \mathcal{C}$. By \ref{neck-vitali-covering}, $s\coloneqq \gamma^{-2}|\mathbf{x}_1-\mathbf{x}_2|\ge \mathbf{r}_{\mathbf{x}_1}+\mathbf{r}_{\mathbf{x}_2}$. By \eqref{neck-n2-CinV}, $\mathbf{x}_2\in \mathcal{C}\cap P(\mathbf{x}_1,s)\subseteq P(\mathbf{x}_1+V,\delta s)$. It follows that
    \begin{align*}
        |\pi_{V}^\perp(\mathbf{x}_1-\mathbf{x}_2)| = d(\mathbf{x}_2,\mathbf{x}_1+V) < \delta s = \delta \gamma^{-2}|\mathbf{x}_1-\mathbf{x}_2|,
    \end{align*}
    which implies \eqref{eq-theorem-neck-structure-bi-lipschitz}.
    
    \ref{theorem-neck-structure-graph} 
    Set $A\coloneqq \pi_V(\mathcal{C})$.
    By \ref{theorem-neck-structure-bi-lipschitz}, we can define $\mathbf{f}:A\to V^\perp$ so that for any $\mathbf{x}\in \mathcal{C}$, $\mathbf{v}+\mathbf{f}(\mathbf{v})=\mathbf{x}$ if $\mathbf{v}=\pi_V(\mathbf{x})$. 
    For any $\mathbf{x}_i\in \mathcal{C}$, $\mathbf{v}_i=\pi_V(\mathbf{x}_i)$, $i=1,2$, 
    \begin{align*}
        |\mathbf{f}(\mathbf{v}_1)-\mathbf{f}(\mathbf{v}_2)|=|(\mathbf{x}_1-\mathbf{v}_1)-(\mathbf{x}_2-\mathbf{v}_2)|
        = |\pi_V^\perp(\mathbf{x}_1-\mathbf{x}_2)| \le C\delta |\mathbf{x}_1-\mathbf{x}_2|.
    \end{align*}
    Thus, $\mathcal{C}=\mathbf{f}(A)$ is a $(k,C\delta)$-graph.

    Fix $\mathbf{x} \in \mathcal{C}$ and $s\in [\mathbf{r}_{\mathbf{x}},1]$. By \eqref{eq-theorem-neck-structure-bi-lipschitz} and recalling \eqref{eq restricted ball}, we have
    \begin{align*}
        \mu_{0}(P(\mathbf{x},s)) \leq \mathcal{H}_{\mathcal{P}}^{k}\left((\pi_V|_{\mathcal{C}})^{-1}(P^V(\pi_V(\mathbf{x}),2s))\right)\leq C\mathcal{H}_{\mathcal{P}}^{k}(P^V(\pi_V(\mathbf{x}),2s))\leq Cs^{k}.
    \end{align*}
    By \ref{neck-vitali-covering} and \ref{theorem-neck-structure-bi-lipschitz},  $\{P^V(\pi_V(\mathbf{y}),\frac{1}{2}\gamma^2 \mathbf{r}_{\mathbf{y}})\}_{\mathbf{y}\in \mathcal{C}_+ \cap P(\mathbf{x},s)}$ is pairwise disjoint in $P^V(\mathbf{x},2s)$ so that
    \begin{align*}
        \mu_{+}(P(\mathbf{x},s))=\sum_{\mathbf{x}\in\mathcal{C}_{+}\cap P(\mathbf{x},s)}\mathbf{r}_{\mathbf{x}}^{k}\leq C\mathcal{H}_{\mathcal{P}}^{k}\left(P^V(\pi_V(\mathbf{x}),\frac{1}{2}\gamma^{2}\mathbf{r}_{\mathbf{x}})\right)\leq C\mathcal{H}_{\mathcal{P}}^{k}(P^V(\pi_V(\mathbf{x}),2s))\leq Cs^{k}.
    \end{align*}
    Combining these estimates, we obtain $\mu(P(\mathbf{x},s))\leq Cs^{k}$. 
\end{proof}

\begin{remark}
    Note that for any $\mathbf{x} \in \mathcal{C}$ and $r\in [\mathbf{r}_{\mathbf{x}},\gamma^{-1}]$ satisfying $P(\mathbf{x},2r) \subseteq P(\mathbf{0},2)$, we use \eqref{neck-n2-CinV} and Theorem \ref{theorem-neck-structure}\ref{theorem-neck-structure-graph} to get
    \begin{align} \label{eq:Visdece}
        \beta_{\mathcal{P},k}^2(\mathbf{x},r) \leq \frac{\delta^2}{r^k} \mu(P(\mathbf{x},r)) \leq C\delta^2.
    \end{align}
\end{remark}

We now turn to proving Ahlfors regularity estimates for packing measure of a strong neck region. As a first step, using the bi-Lipschitz property, \eqref{eq-theorem-neck-structure-bi-lipschitz}, of the orthogonal projection
\begin{align*}
    \pi\coloneqq \pi_V|_{\mathcal{C}}\colon \mathcal{C}\to V,
\end{align*}
we will prove in Lemma \ref{lem:itwasacoverup} that the restriction of $\cC\cap P(\mathbf{z},\frac{9}{5}s)$ at $\mathbf{z}\in \cC$ discretely looks like a ball in $V$ centered at $\pi(\mathbf{z})$.
\begin{lemma} \label{lem:itwasacoverup}
        Let $\mathcal{N}=P(\mathbf{x}_{0},2r)\setminus\overline{P}(\mathcal{C},\mathbf{r}_{\bullet})$ be a strong $(m,k,\delta,\eta)$-neck region modeled on $V\in {\rm Gr}_{\cP}(k)$, where $\delta \leq \overline{\delta}$. For all $\mathbf{x} \in \mathcal{C}$ and $s\in [\gamma^{-1}\mathbf{r}_{\mathbf{x}},r]$ such that $P(\mathbf{x},\frac{3}{2}s)\subseteq P(\mathbf{x}_0,2r)$,
        \begin{align*} 
            P^V(\pi(\mathbf{x}),\tfrac{7}{4}s) \subseteq \bigcup_{\mathbf{z} \in \mathcal{C} \cap P(\mathbf{x},\frac{9}{5}s)}\overline{P}^V(\pi(\mathbf{z}),\mathbf{r}_{\mathbf{z}}).
        \end{align*}
\end{lemma}
\begin{proof} 
    By translation and rescaling, we may assume that $\mathbf{x}_0=\mathbf{0}$ and $r=1$. Fix $\mathbf{x} \in \mathcal{C}$ and $s\in (0,1]$ such that $P(\mathbf{x},2s) \subseteq P(\mathbf{0},2)$.

    First, we prove that \eqref{eq:improvedCinVinC} is preserved under the projection map.
    
    \begin{claim} \label{claim:VinCprojected}
          If $\delta\leq \ol{\delta}$, the following holds. For any $\mathbf{z} \in \mathcal{C} \cap P(\mathbf{x},\frac{19}{10}s)$, $s'\in [\mathbf{r}_{\mathbf{z}},\gamma^{-1}]$ such that $P(\mathbf{z},s') \subseteq P(\mathbf{x},\frac{19}{10}s)$, we have
        \begin{align} 
            P^V(\pi(\mathbf{z}),s') \subseteq P^V(\pi(\mathcal{C}\cap P(\mathbf{z},s')),7\gamma s').
        \end{align}
    \end{claim}
    \begin{proof}
        Note that for any $\mathbf{y}\in \cC$, if $P(\mathbf{y},\gamma s')\cap P(\mathbf{z},(1-2\gamma)s')\neq \emptyset$, then $\verts{\mathbf{y}-\mathbf{z}}\leq (1-\gamma)s'$. It follows that
        \begin{align}
            P(\mathcal{C},\gamma s')\cap P(\mathbf{z},(1-2\gamma)s') \subset P(\cC \cap P(\mathbf{z},s'), \gamma s').\label{eq containment of restricted cetners}
        \end{align}
        By \eqref{eq:improvedCinVinC} with $\mathbf{x} \leftarrow \mathbf{z}$ and $s\leftarrow (1-2\gamma)s'$ we get $(\mathbf{z}+V)\cap P(\mathbf{z},(1-2\gamma)s') \subseteq P(\mathcal{C},\gamma s')\cap P(\mathbf{z},(1-2\gamma)s') $, which combined with \eqref{eq containment of restricted cetners} yields
        \begin{align*}
            (\mathbf{z}+V)\cap P(\mathbf{z},(1-2\gamma)s') \subset P(\cC \cap P(\mathbf{z},s'), \gamma s').
        \end{align*}
        Applying $\pi$ to both sides, using \eqref{eq-theorem-neck-structure-bi-lipschitz} with $\delta\leq \ol{\delta}$ as well as
        \begin{align*}
             P^V(\pi(\mathbf{z}),s') \subseteq P^V\left(P^V(\pi(\mathbf{z}),(1-3\gamma)s'),3\gamma s' \right) 
        \end{align*}
        yield
        \begin{align*}
            P^V(\pi(\mathbf{z}),s') &\subseteq P^V\left(P^V(\pi(\mathbf{z}),(1-3\gamma)s'),3\gamma s' \right) \subseteq P^V \left( P^V(\pi(\cC \cap P(\mathbf{z},s')),2\gamma s'), 5\gamma s' \right) \\
            & =P^V(\pi(\mathcal{C}\cap P(\mathbf{z},s')),7\gamma s').
        \end{align*}
    \end{proof}
    
    Suppose by way of contradiction that there exists 
    \begin{align*}
        \mathbf{\overline{y}} \in P^V(\pi(\mathbf{x}),\tfrac{7}{4}s) \setminus \bigcup_{\mathbf{z} \in \mathcal{C} \cap P(\mathbf{x},\frac{9}{5}s)}\overline{P}^V(\pi(\mathbf{z}),\mathbf{r}_{\mathbf{z}}).
    \end{align*}
    Since $\mathbf{r_x}\leq \gamma s<\frac{7}{4}s<\frac{19}{10}s$, we can apply Claim \ref{claim:VinCprojected} with $\mathbf{z} \leftarrow \mathbf{x}$ and $s' \leftarrow \frac{7}{4}s$ to get
    \begin{align}
        P^V(\pi(\mathbf{x}),\tfrac{7}{4}s) \subseteq P^V(\pi(\mathcal{C} \cap \overline{P}(\mathbf{x},\tfrac{9}{5}s)),20\gamma s),\label{eq containment claim 1}
    \end{align}
    Choose $\mathbf{y} \in \arg\min d_{\mathcal{P}}(\ol{\mathbf{y}},\pi(\mathcal{C}\cap \overline{P}(\mathbf{x},\frac{9}{5}s)))$. Since $\ol{\mathbf{y}}\in P^V(\pi(\mathbf{x}),\frac{7}{4}s)$ we can use \eqref{eq containment claim 1} to obtain $\verts{\mathbf{y}-\ol{ \mathbf{y}}}<20\gamma s$. Therefore, from \eqref{eq-theorem-neck-structure-bi-lipschitz}, it follows that $P(\pi^{-1}(\mathbf{y}),2|\mathbf{y}-\overline{\mathbf{y}}|) \subseteq P(\mathbf{x},\frac{19}{10}s)$. By the definition of $\ol{\mathbf{y}}$, we know that $\mathbf{r}_{\pi^{-1}(\mathbf{y})}< |\mathbf{y}-\overline{\mathbf{y}}|$. Hence we can apply Claim \ref{claim:VinCprojected} with $\mathbf{z} \leftarrow \pi^{-1}(\mathbf{y})$ and $s'\leftarrow 2|\mathbf{y}-\ol{\mathbf{y}}|$ to obtain
    \begin{align*}
        P^V(\mathbf{y},2|\mathbf{y}-\overline{\mathbf{y}}|) \subseteq P^V\parens*{\pi\parens*{\mathcal{C}\cap P\parens*{\pi^{-1}(\mathbf{y}),2|\mathbf{y}-\overline{\mathbf{y}}|}},14 \gamma |\mathbf{y}-\mathbf{\overline{y}}| }.
    \end{align*}
    In particular, there exists $\mathbf{y}' \in \pi\parens*{\mathcal{C}\cap P\parens*{\pi^{-1}(\mathbf{y}),2|\mathbf{y}-\overline{\mathbf{y}}|}}$ such that
    \begin{align}
        |\mathbf{y}'-\overline{\mathbf{y}}|< 14\gamma |\mathbf{y}-\mathbf{\overline{y}}|< |\mathbf{y}-\mathbf{\overline{y}}|.\label{eq minimality of y'}
    \end{align}
    On the other hand, by \eqref{eq-theorem-neck-structure-bi-lipschitz} if $\delta \leq \ol{\delta}$, we have 
    \begin{align*} 
         |\mathbf{x}-\pi^{-1}(\mathbf{y}')| \leq (1+\gamma)(|\pi(\mathbf{x})-\overline{\mathbf{y}}|+|\mathbf{\overline{y}}-\mathbf{y}'|) \leq (1+\gamma)(\tfrac{7}{4}+280\gamma^2)s \leq \tfrac{9}{5}s. 
     \end{align*}
    Hence $\pi^{-1}(\mathbf{y}') \in \mathcal{C} \cap \overline{P}(\mathbf{x},\frac{9}{5}s)$, which together with \eqref{eq minimality of y'} contradicts the choice of $\mathbf{y}$.
\end{proof}

Using Lemma \ref{lem:itwasacoverup}, we prove a lower bound $\mu(P(\mathbf{z},s))\geq C^{-1}s^k$, proving Ahlfors regularity.
\begin{proposition} \label{prop:ahlforsreg}
    Let $\mathcal{N}=P(\mathbf{x}_{0},2r)\setminus\overline{P}(\mathcal{C},\mathbf{r}_{\bullet})$ be a strong $(m,k,\delta,\eta)$-neck region modeled on $V\in {\rm Gr}_{\cP}(k)$, where $\delta \leq \overline{\delta}$. Then the packing measure $\mu$ of $\mathcal{C}$ defined in \eqref{eq-packing measure} satisfies the Ahlfors regularity condition
    \begin{equation} \label{eq:Ahlfors}
       C^{-1}s^k \leq \mu (P(\mathbf{x},s)),
    \end{equation}
    for all $\mathbf{x} \in \mathcal{C}$ and $s\in (0,r]$ such that $P(\mathbf{x},\frac{3}{2}s)\subseteq P(\mathbf{x}_0,2r)$, where $C$ is universal. 
\end{proposition}

\begin{proof}
    By translation and rescaling, we may assume that $\mathbf{x}_0=\mathbf{0}$ and $r=1$. Fix $\mathbf{x} \in \mathcal{C}$ and $s\in (0,1]$ such that $P(\mathbf{x},2s) \subseteq P(\mathbf{0},2)$. If $s<\gamma^{-1}\mathbf{r}_{\mathbf{x}}$, then $\mathbf{x}\in\mathcal{C}_{+}$. Thus, \eqref{eq:Ahlfors} follows because
    \begin{align*}
        \mu(P(\mathbf{x},s))\geq\mathbf{r}_{\mathbf{x}}^{k}\geq C^{-1}s^{k}.
    \end{align*}
    
    Suppose instead $s\in[\gamma^{-1}\mathbf{r}_{\mathbf{x}},1]$. By Lemma \ref{lem:itwasacoverup}, we have
    \begin{align*}
        C^{-1}s^{k}\leq\sum_{\mathbf{y}\in\mathcal{C}\cap P(\mathbf{x},\frac{9}{5}s)}\mathbf{r}_{\mathbf{y}}^{k}\le \mu(P(\mathbf{x},\tfrac{9}{5}s)).
    \end{align*}
    Thus \eqref{eq:Ahlfors} follows by replacing $s$ in the above computation with $\frac{5}{9}s$. 
\end{proof}

\subsection{Best subspaces on a temporal strong neck region}
Suppose $\mathcal{N}=P(\mathbf{0},2)\setminus\bigcup_{\mathbf{x}\in\mathcal{C}}\ol{P}(\mathbf{x},\mathbf{r}_{\mathbf{x}})$
is a temporal strong $(m,k,\delta,\eta)$-neck region modeled on $V \in \operatorname{Gr}_{\mathcal{P}}(k)$. In general, $V$ is not the best plane that approximates $\cC$, and we often rely on the best plane. In this subsection, we provide tilting estimates for the best planes across different locations and scales. The proofs in this section are modifications of \cite[Section 11.1]{naberreifenbergnotes} from the elliptic setting.

For each $\mathbf{x}\in \mathcal{C}$ with $P(\mathbf{x},r) \subseteq P(\mathbf{0},2)$, we define a best vertical plane at $\mathbf{x}$ at scale $r$ by
\begin{align} \label{eq:bestplanedef}
    V_{\mathbf{x},r}\in\arg\min_{W\in\text{Aff}_{\mathcal{P},\operatorname{v}}(k)}\frac{1}{r^{k}}\int_{P(\mathbf{x},r)}\left(\frac{d_{\mathcal{P}}(\mathbf{y},W)}{r}\right)^{2}d\mu(\mathbf{y}),
\end{align}
where $\operatorname{Aff}_{\mathcal{P},\operatorname{v}}(k)$ is the set of all vertical planes in $\operatorname{Aff}_{\mathcal{P}}(k)$. Let $\widetilde{V}_{\mathbf{x},r}\in\text{Gr}_{\mathcal{P}}(k)$ be the corresponding linear subspace, and write $\pi_{\mathbf{x},r}\coloneqq \pi_{\widetilde{V}_{\mathbf{x},r}}$, $\pi_{\mathbf{x},r}^{\perp}\coloneqq \pi_{\widetilde{V}_{\mathbf{x},r}}^{\perp}$. 

In the lemma below, we estimate the distance between $V_{\mathbf{x},r}$ and the points in $r$-neighborhood of $\mathbf{x}$.
\begin{lemma} \label{lem:xclosetoV} 
    For any $\mathbf{x} \in \mathcal{C}$ and $r>0$ such that $P(\mathbf{x},r)\subseteq P(\mathbf{0},2)$,
    if $\delta\leq \ol{\delta}$, we have
    \begin{align*} 
        \sup_{\mathbf{w} \in P(\mathbf{x},\frac{9}{10}r)} d_{\mathcal{P}}(\mathbf{w},V_{\mathbf{x},r}) \leq C\delta^{\frac{2}{k+2}}r.
    \end{align*}
\end{lemma}

\begin{proof} 
    Fix $\mathbf{w} \in P(\mathbf{x},\frac{9}{10}r)$ and set $\rho\coloneqq d_{\mathcal{P}}(\mathbf{w},V_{\mathbf{x},r})$, $s=\min\{r,\rho\}$.
    \begin{align} 
        C\delta^2 \geq \frac{1}{r^{k+2}}\int_{P(\mathbf{w},\frac{1}{10}s) }d_{\mathcal{P}}^2(\mathbf{y},V_{\mathbf{x},r})d\mu(\mathbf{y}) \geq \frac{s^k}{Cr^{k+2}} \rho^2,\label{eq nice computation}
    \end{align}
    where we used \eqref{eq:Visdece} in the first inequality and Ahlfors regularity \eqref{eq:Ahlfors} of $\mu$ in the second inequality. Now suppose $\min\{r,\rho\}=r$. In this case,
    \begin{align*}
        C\delta^2\geq \frac{\rho^2}{Cr^2}\geq \frac{1}{C}.
    \end{align*}
    If $\delta \leq \ol{\delta}$, then this case cannot occur. We may thus assume $\min\{r,\rho\}=\rho$, so that \eqref{eq nice computation} implies $\rho \leq C\delta^{\frac{2}{k+2}}r$. Therefore, the claim follows.
\end{proof}

For $\mathbf{x} \in \mathcal{C}$ and $r>0$, we define the \textit{parabolic mean:}
\begin{align*}
    \mathbf{z}_{\mathbf{x},r}\coloneqq\frac{1}{\mu(P(\mathbf{x},r))}\int_{P(\mathbf{x},r)}\mathbf{z}\,d\mu(\mathbf{z}).
\end{align*}

In the lemma below, we estimate the distance between $V_{\mathbf{x},r}$ and a parabolic mean at different locations and scales.
\begin{lemma} \label{lemma:centerofmass}
    If $\mathbf{y} \in \mathcal{C}$ and $P(\mathbf{y},s)\subseteq P(\mathbf{x},r)$ and $P(\mathbf{x},2r) \subseteq P(\mathbf{0},2)$, then
    \begin{align*}
        d_{\mathcal{P}}^{2}(\mathbf{z}_{\mathbf{y},s},V_{\mathbf{x},r})\leq\frac{r^{k+2}}{\mu(P(\mathbf{y},s))}\beta_{\mathcal{P},k}^{2}(\mathbf{x},r).
    \end{align*}
\end{lemma}

\begin{proof}
    Recall that for any choice of $\mathbf{x}_{0}\in V_{\mathbf{x},r}$, we have $d_{\mathcal{P}}(\mathbf{z},V_{\mathbf{x},r})=|\pi_{\mathbf{x},r}^{\perp}(\mathbf{z}-\mathbf{x}_{0})|$
    for all $\mathbf{z}\in\mathbb{R}^{n+1}$. Thus Jensen's inequality gives
    \begin{align*}
        d_{\mathcal{P}}^{2}(\mathbf{z}_{\mathbf{y},s},V_{\mathbf{x},r})= & \left|\frac{1}{\mu(P(\mathbf{y},s))}\int_{P(\mathbf{y},s)}\pi_{\mathbf{x},r}^{\perp}(\mathbf{z}-\mathbf{x}_{0})\,d\mu(\mathbf{z})\right|^{2}\leq  \frac{1}{\mu(P(\mathbf{y},s))}\int_{P(\mathbf{y},s)}|\pi_{\mathbf{x},r}^{\perp}(\mathbf{z}-\mathbf{x}_{0})|^{2}\,d\mu(\mathbf{z})\\
        \leq & \frac{1}{\mu(P(\mathbf{y},s))}\int_{P(\mathbf{x},r)}d_{\mathcal{P}}^{2}(\mathbf{z},V_{\mathbf{x},r})\,d\mu(\mathbf{z})\leq \frac{r^{k+2}}{\mu(P(\mathbf{y},s))}\beta_{\mathcal{P},k}^{2}(\mathbf{x},r).
    \end{align*}
\end{proof}

Below, we prove that the subspaces $V_{\mathbf{x},r}$ are comparable for nearby points and scales.
\begin{proposition} \label{prop:tilting} 
    Suppose $\mathcal{N}=P(\mathbf{0},2)\setminus\bigcup_{\mathbf{x}\in\mathcal{C}}P(\mathbf{x},\mathbf{r}_{\mathbf{x}})$ is a temporal strong $(m,k,\delta,\eta)$-neck region modeled on $V\in \mathrm{Gr}_{\mathcal{P}}(k)$. There exists $C$ such that the following holds if $\delta\leq\overline{\delta}(\eta)$.
    \begin{enumerate}[label=(\arabic*)]
        \item For any $\mathbf{x},\mathbf{y} \in \mathcal{C}$ and $s>0$ such that $P(\mathbf{x},10^{-2}r) \subseteq P(\mathbf{y},s) \subseteq P(\mathbf{x},10^2r)$, we have
        \begin{align} \label{eq:tiltingcontrol}
            d_{\mathcal{P},H}(V_{\mathbf{x},r} \cap P(\mathbf{x},10^3 r),V_{\mathbf{y},s} \cap P(\mathbf{x},10^3 r))\le C\beta_{\mathcal{P},k}(\mathbf{x},10^{4}r).
        \end{align}\label{prop:tilting1}

        \item Moreover, for any $\mathbf{x}\in \mathcal{C}$ and $r \in [\gamma^{-2}\mathbf{r}_{\mathbf{x}},\gamma^{-1}]$ such that $P(\mathbf{x},10^3 r) \subseteq P(\mathbf{0},2)$, we have
        \begin{align} \label{eq:courseplanetilt}
            |\pi_V^{\perp}-\pi_{\mathbf{x},r}^{\perp}|\leq C\delta^{\frac{2}{k+2}}.
        \end{align} \label{prop:tilting2}
    \end{enumerate}
\end{proposition}

\begin{proof}
    \ref{prop:tilting1} Fix $\mathbf{x},\mathbf{y}\in \mathcal{C}$ and $r,s>0$ such that 
    \begin{align*}
        P(\mathbf{x},10^{-2}r) \subseteq P(\mathbf{y},s) \subseteq P(\mathbf{x},10^2r)\subseteq P(\mathbf{0},2), \qquad P(\mathbf{x},10^3 r)\subseteq P(\mathbf{0},2).
    \end{align*}
    By \eqref{eq:improvedCinVinC}, we have 
    \begin{align*}
        (\mathbf{y}+V)\cap P(\mathbf{y},s)\subseteq P(\mathcal{C},\gamma s).
    \end{align*}
    Since $V$ is vertical, we can find a spatially $(k-2,\frac{s}{2})$-independent subset $\{\mathbf{y}_{i}\}_{i=0}^{k-2}\subseteq\mathcal{C}\cap P(\mathbf{y},s)$ with $\mathbf{y}_0=\mathbf{y}$. Because $P(\mathbf{y}_i,\frac{3}{2}\gamma s) \subseteq P(\mathbf{y},2s) \subseteq P(\mathbf{0},2)$, Proposition \ref{prop:ahlforsreg} gives
    \begin{align*}
        \mu(P(\mathbf{y}_{i},\gamma s))\geq C^{-1}s^{k}.
    \end{align*}
    Fix a collection $\{\mathbf{y}_{i}\}_{i=0}^{k-2}\subseteq\mathcal{C}\cap P(\mathbf{y},s)$. Set $\mathbf{z}_{0}\coloneqq \mathbf{z}_{\mathbf{y},\gamma s}$
    and $\mathbf{z}_{i}\coloneqq \mathbf{z}_{\mathbf{y}_{i},\gamma s}$. Since $P(\mathbf{y}_i,\gamma s) \subseteq P(\mathbf{y},s) \subseteq P(\mathbf{x},10^2 r)$, Lemma \ref{lemma:centerofmass} with $\mathbf{y} \leftarrow \mathbf{y}_i$, $s\leftarrow \gamma s$, $\mathbf{x} \leftarrow \mathbf{x}$, and $r\leftarrow 10^2 r$ implies
    \begin{align} \label{eq:tilting1}
        d_{\mathcal{P}}^{2}(\mathbf{z}_{i},V_{\mathbf{x},10^{2}r})&\leq\frac{Cr^{k+2}}{\mu(P(\mathbf{y}_{i},\gamma s))}\beta_{\mathcal{P},k}^{2}(\mathbf{x},10^2r)\leq C\beta_{\mathcal{P},k}^{2}(\mathbf{x},10^2 r)r^{2},
    \end{align}
    where we used \eqref{eq:Ahlfors} in the second inequality. Similarly, we apply Lemma \ref{lemma:centerofmass} with $\mathbf{y} \leftarrow \mathbf{y}_i$, $s\leftarrow \gamma s$, $\mathbf{x} \leftarrow \mathbf{y}$, and $r\leftarrow s$ to obtain
    \begin{align} \label{eq:tilting2}
        d_{\mathcal{P}}^{2}(\mathbf{z}_{i},V_{\mathbf{y},s})&\leq\frac{s^{k+2}}{\mu(P(\mathbf{y}_i,\gamma s))}\beta_{\mathcal{P},k}^{2}(\mathbf{y},s)\leq C\beta_{\mathcal{P},k}^{2}(\mathbf{y},s)r^{2}.
    \end{align}
    From $P(\mathbf{y},s)\subseteq P(\mathbf{x},10^2 r)$, we also have 
    \begin{align} \label{eq:tilting3}
        \beta_{\mathcal{P},k}^2(\mathbf{y},s) \leq \frac{1}{s^{k+2}}\int_{P(\mathbf{x},10^2 r)} d_{\mathcal{P}}^2(\mathbf{z},V_{\mathbf{x},10^2r}) d\mu(\mathbf{z}) \leq C\beta_{\mathcal{P},k}^2(\mathbf{x},10^2 r).
    \end{align}
    Because $\mathbf{z}_{i}\in P(\mathbf{y}_{i},\gamma s)$, we know
    $\{z_{i}\}_{i=0}^{k-2}\subseteq\mathbb{R}^{n}$ is $(k-2,\frac{1}{4}s)$-independent. By \eqref{eq:Visdece}, \eqref{eq:tilting1}, \eqref{eq:tilting2}, and \eqref{eq:tilting3}, we have
    \begin{align} 
        d_{\mathcal{P}}(\mathbf{z}_i,V_{\mathbf{x},10^2 r})+d_{\mathcal{P}}(\mathbf{z}_i,V_{\mathbf{y},s}) \leq C\beta_{\mathcal{P},k}(\mathbf{x},10^2 r)r \leq C\delta r. \label{eq distance between planes and points} 
    \end{align}
    On the other hand, we apply Lemma \ref{lem:xclosetoV} with $\mathbf{w} \leftarrow \mathbf{x}$ and $r\leftarrow 10^2 r$ to obtain
    \begin{align} \label{eq:xnottoofarfromV}
        d_{\mathcal{P}}(\mathbf{x},V_{\mathbf{x},10^2r}) \leq C\delta^{\frac{2}{k+2}}r.
    \end{align}
    Following the proof of \cite[Lemma 11.8]{naberreifenbergnotes}, we can use \eqref{eq distance between planes and points} and \eqref{eq:xnottoofarfromV} to obtain
    \begin{align} \label{eq:tiltingalmostdone1} 
        d_{\mathcal{P},H}(V_{\mathbf{x},10^2 r} \cap P(\mathbf{x},10^3 r),V_{\mathbf{y},s} \cap P(\mathbf{x},10^3 r)) \leq C \beta_{\mathcal{P},k}(\mathbf{x},10^2 r)r.
    \end{align}
    We can argue as above and take $\mathbf{y} \leftarrow \mathbf{x}$ and $s\leftarrow r$ in \eqref{eq:tiltingalmostdone1} to get
    \begin{align} \label{eq:tiltingalmostdone2}
        d_{\mathcal{P},H}(V_{\mathbf{x},10^2 r} \cap P(\mathbf{x},10^3 r),V_{\mathbf{x},r} \cap P(\mathbf{x},10^3 r)) \leq C \beta_{\mathcal{P},k}(\mathbf{x},10^2 r)r.
    \end{align}
    Combining \eqref{eq:tiltingalmostdone1} and \eqref{eq:tiltingalmostdone2} yields \eqref{eq:tiltingcontrol}.
    
    \ref{prop:tilting2} First, we apply Lemma \ref{lem:xclosetoV} with $\mathbf{w} \leftarrow \mathbf{y}_i$ and $r \leftarrow s$ to get 
    \begin{align} \label{eq:tilting4} 
        d_{\mathcal{P}}(\mathbf{y}_i,V_{\mathbf{y},s}) \leq C\delta^{\frac{2}{k+2}}s.
    \end{align}
    Further, \eqref{neck-n2-CinV} implies $\mathbf{y}_i\in \operatorname{supp}(\mu)\cap P(\mathbf{y},s) \subseteq P(\mathbf{y}+V,\delta s)$, hence
    \begin{align} \label{eq:tilting5} 
        d_{\mathcal{P}}(\mathbf{y}_i,\mathbf{y}+V) \leq \delta s.
    \end{align}
    Because $\{y_i\}_{i=0}^{k-2} \subseteq \mathbb{R}^n$ is $(k-2,\frac{s}{4})$-independent, we can apply \eqref{eq:tilting4}, \eqref{eq:tilting5}, and again appeal to the proof of \cite[Lemma 11.8]{naberreifenbergnotes} to obtain
    \begin{align*}
        d_{\mathcal{P},H}((\mathbf{y}+V)\cap P(\mathbf{y},10^3 s),V_{\mathbf{y},s} \cap P(\mathbf{y},10^3 s)) \leq C\delta^{\frac{2}{k+2}}s.
    \end{align*}
    This implies that the underlying linear subspace $\widetilde{V}_{\mathbf{y},s}$ of $V_{\mathbf{y},s}$ satisfies
    \begin{align*}
        d_{\mathcal{P},H}(V\cap P(\mathbf{0},1),\widetilde{V}_{\mathbf{y},s} \cap P(\mathbf{0},1)) \leq C\delta^{\frac{2}{k+2}}.
    \end{align*}
    We may therefore follow the proof of \cite[Lemma 4.3 and Lemma 4.4]{naber-valtorta-2017-rectifiable-for-harmonic-maps} to obtain \eqref{eq:courseplanetilt}.
\end{proof}

\subsection{Strong neck structure theorem}
In Theorem \ref{thm:strongneckstructure}, we upgrade the $C\delta$-graphicality of $\cC_0$ from Theorem \ref{theorem-neck-structure} under the assumption that the neck region is strong and $k\in \{n,n+1\}$. More precisely, we prove that $\cC$ is contained in a $C\delta \mathbf{r}$ neighborhood of a regular parabolic Lipschitz graph defined below. Some of the methods of this section are inspired by \cite{parabolicwhitney,DavidSemmes}.

For a compactly supported function $\phi :\mathbb{R}^n \times \mathbb{R} \to \mathbb{R}^{n+2-k}$, we define
\begin{align} \label{eq:defofhalfderivative}
    \partial_t^{\frac{1}{2}}\phi(\mathbf{x})\coloneqq \breve{c} \int_{\mathbb{R}} \frac{\phi(x,s)-\phi(x,t)}{|s-t|^{\frac{3}{2}}}\,ds
\end{align}
for $\mathbf{x}=(x,t)\in \mathbb{R}^n \times \mathbb{R}$, where $\breve{c} \in \mathbb{R}$ is a normalizing constant, chosen so that for any Schwartz function $\psi \in C^{\infty}(\mathbb{R})$, if $\cF{\psi}$ denotes the Fourier transform of $\psi$, then
\begin{align*}
    (\cF\partial_t^{\frac{1}{2}}\psi)(\tau)=|\tau|^{\frac{1}{2}}(\cF\psi)(\tau)
\end{align*}
for all $\tau \in \mathbb{R}$. We also define the \textit{parabolic {\rm BMO} norm} on $V\in {\rm Gr}(k)$ as
\begin{align}
    \Verts{\phi}_{{\rm BMO}_{\mathcal{P}}(V)}\coloneqq \sup_{\mathbf{x} \in V} \sup_{r>0} \fint_{P^V(\mathbf{x},r)} \left| \phi(\mathbf{y}) -\fint_{P^V(\mathbf{x},r)} \phi \,d\mathcal{H}_{\mathcal{P}}^{k} \right| \,d\mathcal{H}_{\mathcal{P}}^{k}(\mathbf{y}).\label{eq def of BMO}
\end{align}

\begin{definition} \label{def:regularlipschitz} 
    Given a vertical $V\in \operatorname{Gr}_{\mathcal{P}}(k)$, we say that a compactly supported function $\phi:V\to V^\perp$ is $\delta$-\textit{regular parabolic Lipschitz} if it is $\delta$-Lipschitz with respect to $d_{\mathcal{P}}$ and $\Verts{\partial_t^{\frac{1}{2}}\phi}_{\operatorname{BMO}_{\mathcal{P}}(V)}<\delta$.

    For any function $\phi\colon V\to V^{\perp}$, we let $\mathcal{G}(\phi)\coloneqq  \{\mathbf{y}+\phi(\mathbf{y}) \colon \mathbf{y} \in V\}$ denote the \textit{graph of $\phi$.}
\end{definition}

We now state our main result on the regularity of $\cC$.
\begin{theorem}[Strong Neck Structure Theorem] \label{thm:strongneckstructure} 
    Let $k\in \{n,n+1\}$ and let $\mathcal{N}=P(\mathbf{x}_{0},2r)\setminus\overline{P}(\mathcal{C},\mathbf{r}_{\bullet})$ be an $(m,k,\delta,\eta)$-neck region modeled on $V\in {\rm Gr}_{\cP}(k)$. Then there exists a compactly supported $C\delta^{\frac{1}{k+2}}$-regular parabolic Lipschitz function $G: V\to V^{\perp}$ such that for all $\mathbf{x} \in \mathcal{C}\cap P(\mathbf{x}_0,\frac{4}{3}r)$
    \begin{align}
        d_{\mathcal{P}}(\mathbf{x},\mathcal{G}(G)) < C(\eta,\Lambda)\delta \mathbf{r}_{\mathbf{x}},\label{eq graph is close to centers}
    \end{align}
    and in particular, $\mathcal{C}_0 \cap P(\mathbf{x}_0,\frac{4}{3}r)\subseteq \mathcal{G}(G)$.
\end{theorem}

In our setting, since $\cC$ is a bi-Lipschitz graph over $V$ by Theorem \ref{theorem-neck-structure}, we can substitute the submanifold approximation theorem used in neck structure theorems, for instance \cite{naberreifenbergnotes}, with a regular parabolic Lipschitz extension problem. The main idea of the proof is to construct a Lipschitz almost extension $F$ of $\pi^{-1}$, which is Lipschitz by Theorem \ref{theorem-neck-structure}. At large scales $F$ will look like $\pi^{-1}$, while at scales smaller than $\mathbf{r}_\bullet$ near $\cC_+$, $F$ is well approximated by its tangent plane.

Let $\mu$ be the packing measure of $\mathcal{N}$ as defined in \eqref{eq-packing measure}. Set
\begin{align*} 
    \widehat{\cC}\coloneqq \pi(\mathcal{C}), \qquad \widehat{\cC}_0\coloneqq \pi(\mathcal{C}_0), \qquad \widehat{\cC}_+ \coloneqq \pi (\mathcal{C}_+),\qquad \widehat{\mu}\coloneqq  \pi_{\ast}\mu.
\end{align*}
Observe that $\widehat{\cC}$ and $\widehat{\cC}_0$ are closed subsets of $V$ by \eqref{eq-theorem-neck-structure-bi-lipschitz}. Further, \eqref{eq-theorem-neck-structure-bi-lipschitz} and the proof of Theorem \ref{theorem-neck-structure} imply that the map 
\begin{align*}
    F^{\ast}\colon \widehat{\cC}\to \mathbb{R},\quad \mathbf{x}\mapsto (\pi|_{\mathcal{C}}^{-1}(\mathbf{x})-\mathbf{x})\cdot e_1
\end{align*}
is well-defined and $C\delta$-Lipschitz. By definition, $F^{\ast}(\mathbf{x})e_1+\mathbf{x}=\pi^{-1}(\mathbf{x}) \in \mathcal{C}$ for all $\mathbf{x} \in \widehat{\cC}$. 
Define 
\begin{align}
    \widehat{\mathbf{r}}_{\mathbf{x}}\coloneqq \mathbf{r}_{\pi^{-1}(\mathbf{x})}, \qquad \widehat{\beta}_{\mathcal{P},k}(\mathbf{x},s)\coloneqq \beta_{\mathcal{P},k}(\pi^{-1}(\mathbf{x}),s), \qquad  \widehat{V}_{\mathbf{x},r} \coloneqq  V_{\pi^{-1}(\mathbf{x}),r} \qquad \text{ for } \mathbf{x} \in \widehat{\cC}.\label{eq hat calculus}
\end{align}
By Theorem \ref{theorem-neck-structure} and Proposition \ref{prop:ahlforsreg}, $\widehat{\mu}$ satisfies the Ahlfors regularity condition
\begin{align} \label{eq:projectedAhlfors} 
    C^{-1}s^{k} \leq \widehat{\mu} ( P^V(\mathbf{x},s)) \leq Cs^{k}
\end{align}
for all $\mathbf{x} \in \widehat{\cC}$ and all $s\in [\widehat{\mathbf{r}}_{\mathbf{x}},\gamma^{-1}]$ such that $P^V(\mathbf{x},3s) \subseteq P^V(\mathbf{0},2)$.

In the following lemma, we prove that $F^*$ can be well approximated at scales larger than $\mathbf{r}$ by Lipschitz affine functions with vanishing time derivative and small Lipschitz constant. Furthermore, the approximating functions vary slowly across scales.
\begin{lemma} \label{lemma:strongneckclaim1} 
    There exist affine functions $\ell_{\mathbf{x},r}:V \to \mathbb{R}^{n+2-k}$ for any $\mathbf{x} \in \widehat{\cC}$ and $r \in [\widehat{\mathbf{r}}_{\mathbf{x}},\gamma^{-1}]$ with $P^V(\mathbf{x},10^3r) \subseteq P^V(\mathbf{0},2)$, such that the following hold:
    \begin{enumerate}[label=(\arabic*)]
        \item \label{lemma:strongneckclaim1.0} For any $\mathbf{x}_0\in \widehat{\cC}$ and $r_0 \in [\widehat{\mathbf{r}}_{\mathbf{x}_0},1]$, $\partial_t \ell_{\mathbf{x}_0,r_0}\equiv 0$.
        
        \item \label{lemma:strongneckclaim1.1} We have
        \begin{align*}
            \frac{1}{r^{k}}\int_{P(\mathbf{x},r)} \left( \frac{|F^{\ast}(\mathbf{y})-\ell_{\mathbf{x},r}(\mathbf{y})|}{r} \right)^2 d\widehat{\mu}(\mathbf{y})\leq C\widehat{\beta}_{\mathcal{P},k}^2(\mathbf{x},2r).
        \end{align*}
    
        \item \label{lemma:strongneckclaim1.2} For any $\mathbf{y}\in P^V(\mathbf{0},2)$ and $s \geq \widehat{\mathbf{r}}_{\mathbf{y}}$ satisfying
        $P^V(\mathbf{x},10^{-2} r) \subseteq P^V(\mathbf{y},s) \subseteq P^V(\mathbf{x},10^2r)$,
        \begin{align*}
            \sup_{ P^V(\mathbf{x},10^{3}r)}|\ell_{\mathbf{x},r}-\ell_{\mathbf{y},s}|\leq C\widehat{\beta}_{\mathcal{P},k}(\mathbf{x},10^4 r)r.
        \end{align*}
    
        \item \label{lemma:strongneckclaim1.3} For each $\mathbf{x}\in \widehat{\mathcal{C}}$ and $r\in [\widehat{\mathbf{r}}_{\mathbf{x}},\gamma^{-1}]$, the function $\ell_{\mathbf{x},r}$ is $C\delta^{\frac{2}{k+2}}$-Lipschitz.
    \end{enumerate}
\end{lemma}

\begin{proof} 
    By rotation and translation, we may assume that $V=\{0\} \times \mathbb{R}^{k-2} \times \mathbb{R}$. If $k=n+1$, we define $\ell_{\mathbf{x},r}:V\to \mathbb{R}$, by
    \begin{align*}
        \ell_{\mathbf{x},r}(\mathbf{y})\coloneqq \left( (\pi_V|_{\widehat{V}_{\mathbf{x},r}})^{-1}(\mathbf{y})-\mathbf{y} \right) \cdot e_1,
    \end{align*}
    so that $\mathbf{y}+\ell_{\mathbf{x},r}(\mathbf{y})e_{1}\in \widehat{V}_{\mathbf{x},r}$ for $\mathbf{y}\in V$. Otherwise, we define $\ell_{\mathbf{x},r}:V\to \mathbb{R}^2$ by 
    \begin{align*}
        \ell_{\mathbf{x},r}(\mathbf{y})\coloneqq \left( \left( (\pi_V|_{\widehat{V}_{\mathbf{x},r}})^{-1}(\mathbf{y})-\mathbf{y} \right) \cdot e_1,\left( (\pi_V|_{\widehat{V}_{\mathbf{x},r}})^{-1}(\mathbf{y})-\mathbf{y} \right) \cdot e_2\right). 
    \end{align*}
    For ease of exposition, we restrict our attention to $k=n+1$ since the other case is similar.
    
    \ref{lemma:strongneckclaim1.0} This holds because $\widehat{V}_{\mathbf{x},r}$ are vertical.  
    
    \ref{lemma:strongneckclaim1.1} 
    Notice that $\pi^{-1}(\mathbf{x})=\mathbf{x}+F^{\ast}(\mathbf{x})e_1 \in \mathcal{C}$. For all $\mathbf{z} \in \mathcal{C}$, we have $\pi_V^{\perp}(\mathbf{z}-\pi(\mathbf{z}))=(F^{\ast}\circ \pi)(\mathbf{z})$ and $\pi(\mathbf{z})+(\ell_{\mathbf{x},r}\circ \pi)(\mathbf{z}) \in \widehat{V}_{\mathbf{x},r}$. From this and $\pi_V^{\perp}(e_1)=e_1$, we have 
    \begin{align} \label{eq:planepicture}
        \begin{aligned}
            d_{\mathcal{P}}(\mathbf{z},\widehat{V}_{\mathbf{x},r}) =& \left|\pi_{\pi^{-1}(\mathbf{x}),r}^{\perp}\left(\mathbf{z}- \left(\pi(\mathbf{z})+(\ell_{\mathbf{x},r} \circ \pi)(\mathbf{z}) \right)\right) \right| \\
            \geq & |\pi_{V}^{\perp}(\mathbf{z}-\pi(\mathbf{z})-(\ell_{\mathbf{x},r}\circ\pi)(\mathbf{z})e_{1})|-|(\pi_{\pi^{-1}(\mathbf{x}),r}^{\perp}-\pi_{V}^{\perp})(\mathbf{z}-\pi(\mathbf{z})-(\ell_{\mathbf{x},r}\circ\pi)(\mathbf{z}))|\\
            \ge & (1-C\delta^{\frac{2}{n+3}})|(F^{\ast} \circ \pi)(\mathbf{z})-(\ell_{\mathbf{x},r}\circ\pi_{V})(\mathbf{z})|
        \end{aligned}
    \end{align}
    for $\mathbf{z} \in \mathcal{C}$, where we used \eqref{eq:courseplanetilt} to obtain the last inequality.
    
    Since $\pi(P(\pi^{-1}(\mathbf{x}),2r))\supseteq P^V(\mathbf{x},r)$, we can use \eqref{eq:planepicture} to estimate
    \begin{align*}
        &\frac{1}{r^{n+1}}\int_{ P^V(\mathbf{x},r)}\left(\frac{|F^{\ast}(\mathbf{y})-\ell_{\mathbf{x},r}(\mathbf{y})|}{r}\right)^{2}d\widehat{\mu}(\mathbf{y})\\
        &\leq  \frac{C}{r^{n+1}}\int_{\pi(P(\pi^{-1}(\mathbf{x}),2r))}\left(\frac{|F^{\ast}(\mathbf{y})-\ell_{\mathbf{x},r}(\mathbf{y})|}{r}\right)^{2}d\widehat{\mu}(\mathbf{y})\\
        &=  \frac{C}{r^{n+1}}\int_{P(\pi^{-1}(\mathbf{x}),2r)}\left(\frac{|(F^{\ast} \circ \pi)(\mathbf{z})-(\ell_{\mathbf{x},r}\circ\pi)(\mathbf{z})|}{r}\right)^{2}d\mu(\mathbf{z})\\
        &\leq  \frac{C}{r^{n+1}}\int_{P(\pi^{-1}(\mathbf{x}),2r)}\left(\frac{d_{\mathcal{P}}(\mathbf{z},\widehat{V}_{\mathbf{x},r})}{r}\right)^{2}d\mu(\mathbf{z})\\
        &= C\widehat{\beta}_{\mathcal{P}}^{2}(\mathbf{x},2r).
    \end{align*}
    Thus \ref{lemma:strongneckclaim1.1} holds.
    
    \ref{lemma:strongneckclaim1.2} Now assume $\mathbf{y} \in P^V(\mathbf{0},2)$ and $s \geq \widehat{\mathbf{r}}_{\mathbf{y}}$ satisfy $P^V(\mathbf{x},10^{-2} r) \subseteq P^V(\mathbf{y},s) \subseteq P^V(\mathbf{x},10^2r)$.
    For any $\mathbf{w} \in P^V(\mathbf{0},2)$, we have $\mathbf{w} + \ell_{\mathbf{x},r}(\mathbf{w})e_1 \in \widehat{V}_{\mathbf{x},r}$, so that from \eqref{eq:courseplanetilt}, we have
    \begin{align*} 
        |(\ell_{\mathbf{x},r}-\ell_{\mathbf{y},s})(\mathbf{w})| &=  \left |\pi_V^{\perp} \left(  (\ell_{\mathbf{x},r}(\mathbf{w})-\ell_{\mathbf{y},s}(\mathbf{w}))e_1 \right) \right| \\ &\leq  |\pi_V^{\perp}-\pi_{\pi^{-1}(\mathbf{x}),r}^{\perp}|\cdot |(\ell_{\mathbf{x},r}-\ell_{\mathbf{y},s})(\mathbf{w})| \\
        &\qquad +\left| \pi_{\pi^{-1}(\mathbf{x}),r}^{\perp}\left( ( \mathbf{w}+\ell_{\mathbf{x},r}(\mathbf{w})e_1)-(\mathbf{w}+\ell_{\mathbf{y},s}(\mathbf{w})e_1)  \right) \right| \\
        &\leq  \frac{1}{2}|(\ell_{\mathbf{x},r}-\ell_{\mathbf{y},s})(\mathbf{w})|+d_{\mathcal{P}}(\mathbf{w}+\ell_{\mathbf{y},s}(\mathbf{w})e_1,\widehat{V}_{\mathbf{x},r})
    \end{align*}
    if $\delta \leq \overline{\delta}$. If in addition, $\mathbf{w} \in P(\mathbf{x},10^3 r)$, then \eqref{eq:tiltingcontrol} gives
    \begin{align*} d_{\mathcal{P}}(\mathbf{w}+\ell_{\mathbf{y},s}(\mathbf{w})e_1,\widehat{V}_{\mathbf{x},r}) & \leq d_{\mathcal{P},H}(\widehat{V}_{\mathbf{x},r}\cap P(\mathbf{x},10^3 r) ,\widehat{V}_{\mathbf{y},r}\cap P(\mathbf{x},10^3 r)) \\
    &\leq C\widehat{\beta}_{\mathcal{P},k}(\mathbf{x},10^4 r),
    \end{align*}
    and the claim follows by combining expressions.
    
    \ref{lemma:strongneckclaim1.3} For any $\mathbf{x}_1,\mathbf{x}_2 \in V$, we have $\mathbf{x}_i+\ell_{\mathbf{x},r}(\mathbf{x}_i)e_1 \in \widehat{V}_{\mathbf{x},r}$, so that 
    \begin{align*} 
        \mathbf{w} \coloneqq  (\mathbf{x}_1 + \ell_{\mathbf{x},r}(\mathbf{x}_1)e_1)-(\mathbf{x}_2+\ell_{\mathbf{x},r}(\mathbf{x}_2)e_1) \in \widehat{V}_{\mathbf{x},r}.
    \end{align*}
    Thus $\pi_{\pi^{-1}(\mathbf{x}),r}^{\perp}(\mathbf{w})=0$, hence \eqref{eq:courseplanetilt} yields
    \begin{align*} 
        |\ell_{\mathbf{x},r}(\mathbf{x}_1)-\ell_{\mathbf{x},r}(\mathbf{x}_2)| =& |\pi_V^{\perp}(\mathbf{w})| = |(\pi_V^{\perp}-\pi_{\pi^{-1}(\mathbf{x}),r}^{\perp})(\mathbf{w})| \leq C\delta^{\frac{2}{n+3}}|\mathbf{w}|\\ \leq&  C\delta^{\frac{2}{n+3}}(|\mathbf{x}_1-\mathbf{x}_2|+ |\ell_{\mathbf{x},r}(\mathbf{x}_1)-\ell_{\mathbf{x},r}(\mathbf{x}_2)|),
    \end{align*}
    and the claim follows assuming $\delta \leq \overline{\delta}$.
\end{proof}

Now we will extend $F^*$ to $P^V(\mathbf{0},\frac{7}{4})$. We will smoothly patch together the linear approximations of $F^*$ around each point in $\widehat{\cC}_+$ at scale comparable to $\mathbf{r}_{\bullet}$. To do so, we have the following Whitney-type decomposition of $P^V(\mathbf{0},\frac{7}{4})$ as an application of Lemma \ref{lem:itwasacoverup} with $\mathbf{x} \leftarrow 0$ and $s\leftarrow 1$:
\begin{align} \label{eq:usefulcover} 
    P^V(\mathbf{0},\tfrac{7}{4})\setminus \widehat{\cC}_0 \subseteq \bigcup_{\mathbf{y} \in \widehat{\cC}_+ \cap P(\mathbf{0},\frac{9}{5})} P(\mathbf{y},2\widehat{\mathbf{r}}_{\mathbf{y}}),
\end{align}

\begin{remark}\label{rem:doperemark}
    \begin{enumerate}[label=(\arabic*)]
        \item \label{rem:comparableradii}For any $\mathbf{z} \in \widehat{\cC}_+$ and $\mathbf{y}\in P^V(\mathbf{z},10\widehat{\mathbf{r}}_{\mathbf{z}})$, we use the $C\delta$-Lipschitz property of $\widehat{\mathbf{r}}_{\bullet}$ to obtain
        \begin{align*}
            \widehat{\mathbf{r}}_{\mathbf{z}}< \widehat{\mathbf{r}}_{\mathbf{y}}+ C\delta \verts{\mathbf{y}-\mathbf{z}}< \widehat{\mathbf{r}}_{\mathbf{y}}+C\delta\widehat{\mathbf{r}}_{\mathbf{z}}. 
        \end{align*}
        In particular, $\widehat{\mathbf{r}}_{\mathbf{z}} \leq (1+C\delta)\widehat{\mathbf{r}}_{\mathbf{y}}\leq  (1+C\delta)C\delta \verts{\mathbf{x}-\mathbf{y}}$, so that $|\mathbf{y}-\mathbf{z}|\leq 11\widehat{\mathbf{r}}_{\mathbf{y}}$, hence we have
        \begin{align} \label{eq:comparableradii} 
            (1-C\delta) \widehat{\mathbf{r}}_{\mathbf{y}} \leq \widehat{\mathbf{r}}_{\mathbf{z}} \leq (1+C\delta)\widehat{\mathbf{r}}_{\mathbf{y}}.
        \end{align}

        \item \label{rem:overlaps}
        There is a dimensional constant $C$ such that the following holds when $\delta \leq \overline{\delta}$. 
        For any $\mathbf{x} \in P^V(\mathbf{0},\frac{7}{4})$, there exist at most $C$ points $\mathbf{y} \in \widehat{\cC}_+$ satisfying $\mathbf{x} \in P^V(\mathbf{y},20\widehat{\mathbf{r}}_{\mathbf{y}})$. To see this, let $\widehat{\cC}_{\mathbf{x}}$ be the set of $\mathbf{y} \in P^V(\mathbf{0},\frac{7}{4})$ such that $\mathbf{x} \in P^V(\mathbf{y},20\widehat{\mathbf{r}}_{\mathbf{y}})$. Fix $\mathbf{y}_0 \in \widehat{\cC}_{\mathbf{x}}$. For any other $\mathbf{y} \in \widehat{\cC}_{\mathbf{x}}$, we then have $|\mathbf{y}-\mathbf{y}_0|\leq 2(\mathbf{\widehat{r}}_{\mathbf{y}}+\widehat{\mathbf{r}}_{\mathbf{y}_0})$, so by arguing as in \eqref{eq:comparableradii}, we have
        \begin{align*}
            (1-C\delta)\widehat{\mathbf{r}}_{\mathbf{y}_0} \leq \widehat{\mathbf{r}}_{\mathbf{y}} \leq (1+C\delta)\widehat{\mathbf{r}}_{\mathbf{y}_0}.
        \end{align*}
        By \ref{neck-vitali-covering}, it follows that $\{P(\mathbf{y},\gamma^3 \widehat{\mathbf{r}}_{\mathbf{y}_0})\}_{\mathbf{y} \in \widehat{\cC}_{\mathbf{x}}}$ is a pairwise-disjoint collection whose union is contained in $P(\mathbf{x},100\widehat{\mathbf{r}}_{\mathbf{y}_0})$. It follows that the cardinality of $\widehat{\cC}_{\mathbf{x}}$ is bounded above by $C$.
    \end{enumerate}
\end{remark}

We also need to construct a partition of unity subordinate to the cover \eqref{eq:usefulcover}.
\begin{lemma}[Partition of unity]\label{lem:POU} 
    There exist nonnegative $\psi_{\mathbf{z}} \in C_c^{\infty}(P^V(\mathbf{z},10\widehat{\mathbf{r}}_{\mathbf{z}})\cap P(\mathbf{0},\frac{7}{4}))$ for $\mathbf{z} \in \widehat{\mathcal{C}}_+ \cap P^V(\mathbf{0},\frac{9}{5})$ such that the following hold:
    \begin{enumerate}[label=(\arabic*)]
        \item \label{lem:POU1} $\sum_{\mathbf{z}\in \widehat{\mathcal{C}}_+\cap P^V(\mathbf{0},\frac{9}{5}) }\psi_{\mathbf{z}} =1$ on $P^V(\mathbf{0},\frac{7}{4})$.
    
        \item \label{lem:POU2} There exists a universal constant $C$ such that for any $\mathbf{x} \in P^V(\mathbf{0},\frac{7}{4})$, there are at most $C$ distinct $\mathbf{z} \in \widehat{\cC}_+$ such that $\mathbf{x} \in \operatorname{supp}(\psi_{\mathbf{z}})$.

        \item \label{lem:POU3} For any $k,\ell \in \mathbb{N}_0$, there exist universal constants $C_{k,\ell}$ such that
        \begin{align*}    \widehat{\mathbf{r}}_{\mathbf{z}}^{k+2\ell}|\nabla^k \partial_t^{\ell}\psi_{\mathbf{z}}|\leq C_{k,\ell}.
        \end{align*}
    \end{enumerate}
\end{lemma}

\begin{proof} 
    Fix $\phi \in C_c^{\infty}(P^V(\mathbf{0},10))$ such that $0\leq \phi \leq 1$ and $\phi|_{P^V(\mathbf{0},2)}\equiv 1$. For $\mathbf{z}=(z,t_z)$ and $\mathbf{x} = (x,t) \in P^V(\mathbf{0},\frac{7}{4})$, we define
    \begin{align*}
        \phi_{\mathbf{z}}(\mathbf{x})\coloneqq \phi(\mathbf{\widehat{r}}_{\mathbf{z}}^{-1}(x-z),\mathbf{\widehat{r}}_{\mathbf{z}}^{-2}(t-t_z)),\qquad \psi_{\mathbf{z}}(\mathbf{x}) \coloneqq  \frac{\phi_{\mathbf{z}}(\mathbf{x})}{\sum_{\mathbf{w} \in \widehat{\mathcal{C}}_+\cap P^V(\mathbf{0},\frac{9}{5})}\phi_{\mathbf{w}}(\mathbf{x})}.
    \end{align*}
    Using the cover \eqref{eq:usefulcover}, we know that
    \begin{align}
        \sum_{\mathbf{w} \in \widehat{\mathcal{C}}_+\cap P^V(\mathbf{0},\frac{9}{5})}\phi_{\mathbf{w}}(\mathbf{x})\geq 1\label{eq sum is big}
    \end{align}
    for every $\mathbf{x}\in P^V(\mathbf{0},\frac{7}{4})$. In particular, $\psi_{\mathbf{z}}$ is well-defined and smooth.
    
    \ref{lem:POU1} This claim holds by construction.

    \ref{lem:POU2} Fix $\mathbf{x}\in P^V(\mathbf{0},\frac{7}{4})$. By Remark \ref{rem:doperemark}\ref{rem:overlaps}, there is a universal constant $C$ such that $\mathbf{x} \in \operatorname{supp}(\phi_{\mathbf{z}})$ for at most $C$ distinct $\mathbf{z} \in \widehat{\cC}_+$. In particular, \ref{lem:POU2} holds.

    \ref{lem:POU3} This follows from \eqref{eq sum is big}, Remark \ref{rem:doperemark}\ref{rem:comparableradii} and the proof of \ref{lem:POU2}.
\end{proof}

We now define the candidate almost-extension $F:V \to \mathbb{R}$ of $F^*$ as 
\begin{align*}
    F(\mathbf{x})\coloneqq \begin{cases} \sum_{\mathbf{z} \in \widehat{\cC}_+} \psi_{\mathbf{z}}(\mathbf{x})\ell_{\mathbf{z},\gamma^{-1}\widehat{\mathbf{r}}_{\mathbf{z}}}(\mathbf{x})& \text{ if } \mathbf{x} \notin \widehat{\cC}_0,\\
    F^{\ast}(\mathbf{x}) & \text{ if }\mathbf{x} \in \widehat{\cC}_0.
    \end{cases}
\end{align*}
In the lemma below, we prove that $F$ is a $C\delta^{\frac{2}{k+2}}$-Lipschitz almost extension of $F^*$.
\begin{lemma} \label{lemma:newlipschitz} 
    \begin{enumerate}[label=(\arabic*)]
        \item \label{lemma:newlipschitz1} For any $\mathbf{y} \in \widehat{\cC}\cap P^V(\mathbf{0},\frac{7}{4})$,
        \begin{align}  
            |F(\mathbf{y})-F^{\ast}(\mathbf{y})|\leq C \delta \widehat{\mathbf{r}}_{\mathbf{y}}.
        \end{align}
        In particular, $F$ is continuous. 
        
        \item \label{lemma:newlipschitz2} $F$ is $C\delta^{\frac{2}{k+2}}$-Lipschitz function on $P^V(\mathbf{0},\frac{7}{4})$.  
    \end{enumerate}
\end{lemma}

\begin{proof}
    \ref{lemma:newlipschitz1} Because $F|_{\widehat{\cC}_0}=F^{\ast}|_{\widehat{\cC}_0}$, it suffices to consider $\mathbf{y} \in \widehat{\cC}_+\cap P^V(\mathbf{0},\frac{7}{4})$. For any $\mathbf{z} \in \widehat{\cC}_+$ such that $\psi_{\mathbf{z}}(\mathbf{y})>0$, we use \eqref{eq:comparableradii}, Lemma \ref{lemma:strongneckclaim1}\ref{lemma:strongneckclaim1.1}, and \eqref{eq:Visdece} to get
    \begin{align}
        \begin{split}
            C^{-1}\frac{|F^{\ast}(\mathbf{y})-\ell_{\mathbf{z},\gamma^{-1}\widehat{\mathbf{r}}_{\mathbf{z}}}(\mathbf{y})|^2}{\widehat{\mathbf{r}}_{\mathbf{y}}^2}
            &\leq \frac{1}{\widehat{\mathbf{r}}_{\mathbf{z}}^{n+1}}\int_{P(\mathbf{z},\gamma^{-1}\widehat{\mathbf{r}}_{\mathbf{z}})} \left( \frac{|F^{\ast}(\mathbf{w})-\ell_{\mathbf{z},\gamma^{-1}\widehat{\mathbf{r}}_{\mathbf{z}}}(\mathbf{w})|}{\widehat{\mathbf{r}}_{\mathbf{z}}} \right)^2 d\widehat{\mu}(\mathbf{w}) \\
            &\leq \frac{\widehat{\mathbf{r}}_{\mathbf{y}}^{n+1}}{\widehat{\mathbf{r}}_{\mathbf{z}}^{n+1}}C\beta_{\mathcal{P}}^2(\mathbf{z},2\gamma^{-1}\widehat{\mathbf{r}}_{\mathbf{z}}) \leq C\delta^2.\label{eq dope computation}
        \end{split}
    \end{align}
    Therefore, we can estimate
    \begin{align*}
        \verts{F^{\ast}(\mathbf{y})-F(\mathbf{y})}&\leq \sum_{\mathbf{z}\in \widehat{\cC}_+} \psi_{\mathbf{z}}(\mathbf{y}) \verts{\ell_{\mathbf{z},\gamma^{-1}\widehat{\mathbf{r}}_{\mathbf{z}}}(\mathbf{y})-F^*(\mathbf{y})}\leq C\delta \widehat{\mathbf{r}}_{\mathbf{y}}.
    \end{align*}
    
    \ref{lemma:newlipschitz2} First, suppose $\mathbf{x},\mathbf{y} \in \widehat{\cC}_0$. Then the claim follows from the corresponding property for $F^{\ast}$. 
    
    Second, let $\mathbf{x} \in V\setminus \widehat{\cC}_0$ and $\mathbf{y}\in \widehat{\cC}_0$. For any $\mathbf{z} \in \widehat{\cC}_+$ such that $\psi_{\mathbf{z}}(\mathbf{x})>0$, we use \eqref{eq:comparableradii}, Lemma \ref{lemma:strongneckclaim1}\ref{lemma:strongneckclaim1.3} and $C\delta$-Lipschitz property of $\widehat{\mathbf{r}}_{\bullet}$ to obtain
    \begin{align*}
        \verts{\ell_{\mathbf{z},\gamma^{-1}\widehat{\mathbf{r}}_{\mathbf{z}}}(\mathbf{x})-\ell_{\mathbf{z},\gamma^{-1}\widehat{\mathbf{r}}_{\mathbf{z}}}(\mathbf{z})}\leq C\delta^{\frac{2}{k+2}}\verts{\mathbf{x}-\mathbf{z}}< C\delta^{\frac{2}{k+2}} \widehat{\mathbf{r}}_{\mathbf{z}}< C\delta^{\frac{2}{k+2}} \widehat{\mathbf{r}}_{\mathbf{x}}\leq C\delta^{\frac{2}{k+2}} \delta \verts{\mathbf{x}-\mathbf{y}}.
    \end{align*}
    Using this, the bi-Lipschitz property of $F^{\ast}$ and $\ell_{\mathbf{z},\gamma^{-1}\widehat{\mathbf{r}}_{\mathbf{z}}}$, we can estimate
    \begin{align*}
        \verts{F(\mathbf{x})-F(\mathbf{y})}&=\verts{F(\mathbf{x})-F^*(\mathbf{y})}\leq \sum_{\mathbf{z}\in \widehat{\cC}_+} \psi_{\mathbf{z}}(\mathbf{x}) \verts{\ell_{\mathbf{z},\gamma^{-1}\widehat{\mathbf{r}}_{\mathbf{z}}}(\mathbf{x})-F^*(\mathbf{y})}\\
        &\leq C\delta |\mathbf{x}-\mathbf{y}|+\sum_{\mathbf{z} \in \widehat{\cC}_+}\psi_{\mathbf{z}}(\mathbf{x}) \verts{\ell_{\mathbf{z},\gamma^{-1}\widehat{\mathbf{r}}_{\mathbf{z}}}(\mathbf{z})-F^*(\mathbf{z})}\\
        &\leq C\delta \verts{\mathbf{x}-\mathbf{y}},
    \end{align*}
    where we used that $|\mathbf{x}-\mathbf{z}|\leq |\mathbf{x}-\mathbf{y}|+|\mathbf{y}-\mathbf{z}| \leq (1+C\delta)|\mathbf{x}-\mathbf{y}|$ in the second line and we used $\widehat{\mathbf{r}}_{\mathbf{y}}\leq \delta \verts{\mathbf{x}-\mathbf{y}}$, \eqref{eq:comparableradii} and a computation similar to \eqref{eq dope computation} in the third line. This implies the claim when $\mathbf{x}\in V\setminus \widehat{\cC}_0$ and $\mathbf{y}\in \widehat{\cC}_0$.
    
    Finally, we consider the case $\mathbf{x_1},\mathbf{x}_2 \in V\setminus \widehat{\cC}_0$. We will break up the proof into two cases depending on whether the points are far or near compared to their corresponding $\mathbf{r}$. In the first case, $F$ can be approximated by $F^*$ evaluated at points in $\widehat{\cC}_+$ that are nearby $\mathbf{x}_1,\mathbf{x}_2$. Then we can use Lipschitz property of $F^*$. In the latter case, since $F$ is approximated by $\ell$, we can use its derivative estimates. More precisely, suppose $|\mathbf{x}_1-\mathbf{x}_2|\geq 100\max \{\widehat{\mathbf{r}}_{\mathbf{x}_1},\widehat{\mathbf{r}}_{\mathbf{x}_2}\}$. By \eqref{eq:usefulcover}, there exist $\mathbf{y}_1,\mathbf{y}_2 \in \widehat{\cC}_+$ such that $\mathbf{x}_i \in P(\mathbf{y}_i,2\widehat{\mathbf{r}}_{\mathbf{y}_i})$. Now using the computation in the preceding paragraph, we have
    \begin{align*} 
        |F(\mathbf{x}_i) - F^{\ast}(\mathbf{y}_i)| \leq C\delta \verts{\mathbf{x}_i-\mathbf{y}_i}\leq  C\delta \widehat{\mathbf{r}}_{\mathbf{y}_i}.
    \end{align*}
    Therefore,
    \begin{align*}
        |F(\mathbf{x}_1) - F(\mathbf{x}_2)|&\leq |F(\mathbf{x}_1) - F^{\ast}(\mathbf{y}_1)|+|F(\mathbf{x}_2) - F^{\ast}(\mathbf{y}_2)|+\verts{F^{\ast}(\mathbf{y}_1)-F^{\ast}(\mathbf{y}_2)}\\
        &\leq C\delta (\widehat{\mathbf{r}}_{\mathbf{y}_1}+\widehat{\mathbf{r}}_{\mathbf{y}_2})+C\delta (\verts{\mathbf{y}_1-\mathbf{x}_1}+\verts{\mathbf{x}_1-\mathbf{x}_2}+\verts{\mathbf{x}_2-\mathbf{y}_2})\\
        &\leq  C\delta \verts{\mathbf{x}_1-\mathbf{x}_2}.
    \end{align*}
    Now we consider the case $|\mathbf{x}_1-\mathbf{x}_2|\leq 100 \max \{\widehat{\mathbf{r}}_{\mathbf{x}_1},\widehat{\mathbf{r}}_{\mathbf{x}_2}\}$. Without loss of generality, let $\widehat{\mathbf{r}}_{\mathbf{x}_1}=\max \{\widehat{\mathbf{r}}_{\mathbf{x}_1},\widehat{\mathbf{r}}_{\mathbf{x}_2}\}$. Then using the $C\delta$-Lipschitz property of $\widehat{\mathbf{r}}$, we get $\widehat{\mathbf{r}}_{\mathbf{x}_1}\leq \widehat{\mathbf{r}}_{\mathbf{x}_2}+C\delta \verts{\mathbf{x}_1-\mathbf{x}_2}\leq \widehat{\mathbf{r}}_{\mathbf{x}_2}+ C\delta \widehat{\mathbf{r}}_{\mathbf{x}_1}$. In particular,
    \begin{align}
        \widehat{\mathbf{r}}_{\mathbf{x}_2}\leq \widehat{\mathbf{r}}_{\mathbf{x}_1}\leq (1+C\delta) \widehat{\mathbf{r}}_{\mathbf{x}_2}.\label{eq comparability of radius function}
    \end{align}
    Let $\mathbf{x}'$ denote a point in the line segment connecting $\mathbf{x}_1$ and $\mathbf{x}_2$. Since $\mathbf{x}'\in P^V(\mathbf{0},\frac{7}{4})$, using \eqref{eq:usefulcover}, we can choose $\mathbf{y}' \in \widehat{\cC}_+$ satisfying $\mathbf{x}' \in P(\mathbf{y}',2\widehat{\mathbf{r}}_{\mathbf{y}'})$. Then Lemma \ref{lemma:strongneckclaim1}\ref{lemma:strongneckclaim1.3},  and Lemma \ref{lemma:strongneckclaim1}\eqref{lemma:strongneckclaim1.2} with $\mathbf{x} \leftarrow \mathbf{y}'$, $\mathbf{x}' \leftarrow \mathbf{z}$ where $\mathbf{z}$ satisfies $\nabla \psi_{\mathbf{z}}(\mathbf{x}')\neq 0$, $r \leftarrow \gamma^{-1}\widehat{\mathbf{r}}_{\mathbf{y}'}$, $r' \leftarrow \gamma^{-1}\widehat{\mathbf{r}}_{\mathbf{z}}$ (where the hypotheses are satisfied due to \eqref{eq:comparableradii}) give
    \begin{align*}
        |\nabla F(\mathbf{x}')| &\leq C\delta^{\frac{2}{k+2}} + \left| \sum_{\mathbf{z} \in \widehat{\cC}_+} \nabla \psi_{\mathbf{z}}(\mathbf{x}')(\ell_{\mathbf{z},\gamma^{-1}\widehat{\mathbf{r}}_{\mathbf{z}}}(\mathbf{x}')-\ell_{\mathbf{y}',\gamma^{-1}\widehat{\mathbf{r}}_{\mathbf{y}'}}(\mathbf{x}'))\right| \leq C\delta^{\frac{2}{k+2}},\\
        |\partial_t  F(\mathbf{x}')| &\leq  \left| \sum_{\mathbf{z} \in \widehat{\cC}_+} \partial_t \psi_{\mathbf{z}}(\mathbf{x}')(\ell_{\mathbf{z},\gamma^{-1}\widehat{\mathbf{r}}_{\mathbf{z}}}(\mathbf{x}')-\ell_{\mathbf{y}',\gamma^{-1}\widehat{\mathbf{r}}_{\mathbf{y}'}}(\mathbf{x}'))\right| \leq C\frac{\delta}{\widehat{\mathbf{r}}_{\mathbf{y}'}}.
    \end{align*}
    Integrating this along a spacetime path from $\mathbf{x}_1$ to $\mathbf{x}_2$ yields the claim. In fact, the error coming from spatial direction is bounded above by $C\delta^{\frac{2}{k+2}}\verts{\mathbf{x}_1-\mathbf{x}_2}$. To estimate the error coming from integrating in temporal direction, note that 
    \begin{align*}
        \verts{\mathbf{y}'-\mathbf{x}_i}\leq \verts{\mathbf{y}'-\mathbf{x}'}+\verts{\mathbf{x}'-\mathbf{x}_i} \leq 2\widehat{\mathbf{r}}_{\mathbf{y}'}+\verts{\mathbf{x}_1-\mathbf{x}_2}.
    \end{align*}
    Using this and the $C\delta$-Lipschitz property of $\widehat{\mathbf{r}}$, we get
    \begin{align*}
        \widehat{\mathbf{r}}_{\mathbf{y}'}\geq \widehat{\mathbf{r}}_{\mathbf{x}_i}-C\delta \verts{\mathbf{y}'-\mathbf{x}_i} \geq \widehat{\mathbf{r}}_{\mathbf{x}_i}-C\delta \widehat{\mathbf{r}}_{\mathbf{x}_i} \geq \frac{1}{2} \widehat{\mathbf{r}}_{\mathbf{x}_i}
    \end{align*}
    if $\delta \leq \ol{\delta}$, where we used \eqref{eq comparability of radius function} in the second inequality. Therefore, the contribution of integrating in the temporal direction is bounded above by
    \begin{align*}
        C\frac{\delta\verts{t_1-t_2}}{\widehat{\mathbf{r}}_{\mathbf{x}_i}}\leq C\frac{\delta\verts{\mathbf{x}_1-\mathbf{x}_2}^2}{\widehat{\mathbf{r}}_{\mathbf{x}_i}}\leq C\delta\verts{\mathbf{x}_1-\mathbf{x}_2}\frac{ \max \{\widehat{\mathbf{r}}_{\mathbf{x}_1},\widehat{\mathbf{r}}_{\mathbf{x}_2}\}}{\widehat{\mathbf{r}}_{\mathbf{x}_i}}\leq C\delta\verts{\mathbf{x}_1-\mathbf{x}_2},
    \end{align*}
    where we used \eqref{eq comparability of radius function} in the last inequality.
\end{proof}

We now work towards establishing the BMO estimates for $\pdt^{\frac{1}{2}}F$. By Proposition \ref{prop:criterionforregularity}, this follows from a Carleson estimate for an analog of the $\beta$-number for measures $\mu$ supported on a Lipschitz graph. For $\mathbf{x}\in V$ and $r>0$, we define 
\begin{align*}
    \kappa^2(\mathbf{x},r)\coloneqq \kappa^2_k(\mathbf{x},r;F)\coloneqq \inf_{\ell}\frac{1}{r^{k}}\int_{ P^V(\mathbf{x},r)}\left(\frac{|F(\mathbf{y})-\ell(\mathbf{y})|}{r}\right)^{2}\mathcal{H}_{\mathcal{P}}^{k}(\mathbf{y}),
\end{align*}
where the infimum is taken over all affine functions $\ell\colon V\to \mathbb{R}^{n+2-k}$ satisfying $\partial_{t}\ell\equiv 0$. 

Inspired by \cite[\S 19]{DavidSemmes}  (see also \cite[Appendix]{parabolicwhitney}), we first prove that the Carleson norm of $\kappa^2$ is bounded above by that of $\beta^2$.
\begin{lemma} \label{lemma:newkappa} 
    For any $\mathbf{x} \in P(\mathbf{0},\frac{3}{2})$ and $r\in (0,\frac{1}{100}]$,
    \begin{align}
        \int_{P^V(\mathbf{x},r)}  \int_{0}^r\kappa^2(\mathbf{y},s)\,\frac{ds}{s}d\mathcal{H}_{\mathcal{P}}^{k}(\mathbf{y}) \leq C\delta^{\frac{2}{k+2}} r^k +\int_{P^V(\mathbf{x},10r)} \int_{\frac{1}{8}\widehat{\mathbf{r}}_{\mathbf{y}}}^r\widehat{\beta}_{\mathcal{P},k}^2(\mathbf{y},10^5s)\,\frac{ds}{s}d\widehat{\mu}(\mathbf{y}).\label{eq: kappa and beta}
    \end{align}
\end{lemma}

The idea is to break up the integral over scale $s$ into two ranges depending on whether it is smaller or larger than $\mathbf{r}_{\bullet}$. In the case where the scale is smaller than $\mathbf{r}$, we just use higher derivative estimates of $F$. When $s$ is large compared to $\mathbf{r}_{\bullet}$, we can approximate $\cH_{\cP}^k$ by $\widehat{\mu}$ using the cover \eqref{eq:usefulcover}. On such large scales, $F$ is well approximated by an affine function given by Lemma \ref{lemma:strongneckclaim1}. This gives rise to two terms involving $\widehat{\beta}$ and $\widehat{\mu}$, where the second term appears because we use Lipschitz approximation so that we can apply Lemma \ref{lemma:strongneckclaim1}. To estimate the term involving $\widehat{\mu}$, we use its Ahlfors regularity \eqref{eq:projectedAhlfors}.

\begin{proof}
    First, we assume that $\mathbf{x}\in P(\mathbf{0},\frac{3}{2})\cap \widehat{\cC}$ and $r\in (0,\frac{1}{100}]$. Note that
    \begin{align*}
        \int_{P^V(\mathbf{x},r)} \int_{0}^r \kappa^2(\mathbf{y},s)\,\frac{ds}{s}d\mathcal{H}_{\mathcal{P}}^{k}(\mathbf{y})&= \int_{P^V(\mathbf{x},r)}\int_{0}^{\widehat{\mathbf{r}}_{\mathbf{y}}} \kappa^2(\mathbf{y},s)\,\frac{ds}{s} d\mathcal{H}_{\mathcal{P}}^{k}(\mathbf{y})+ \int_{P^V(\mathbf{x},r)}\int_{\widehat{\mathbf{r}}_{\mathbf{y}}}^{r} \kappa^2(\mathbf{y},s)\,\frac{ds}{s} d\mathcal{H}_{\mathcal{P}}^{k}(\mathbf{y})\\
        &\eqqcolon I_1+I_2,
    \end{align*}
    so that the integrand in $I_1$ vanishes when $\mathbf{y} \in \widehat{\cC}_0$.
    
    We now estimate $I_2$. Fix $\mathbf{y}\in P^V(\mathbf{x},r)\subset P^V(\mathbf{0},\frac{7}{4})$. Using \eqref{eq:usefulcover}, we choose $\ol{\mathbf{y}} \in \widehat{\cC}_+\cap P^V(\mathbf{0},\frac{9}{5})$ such that $\mathbf{y}\in P^V(\ol{\mathbf{y}},2\widehat{\mathbf{r}}_{\ol{\mathbf{y}}})$. Now fix $s\in [\widehat{\mathbf{r}}_{\mathbf{y}},r]$. Let $\ell_{\ol{\mathbf{y}},s}$ be as in Lemma \ref{lemma:strongneckclaim1}. By the $C\delta^{\frac{2}{k+2}}$-Lipschitz properties of $\ell$ and $F$ in Lemma \ref{lemma:strongneckclaim1} and Lemma \ref{lemma:newlipschitz}, respectively, for any $\mathbf{z}\in \widehat{\cC}_+\cap P^V(\mathbf{0},\frac{9}{5})$, we have
    \begin{align*} 
        &\int_{P^V(\mathbf{z}, 2\widehat{\mathbf{r}}_{\mathbf{z}})} |F(\mathbf{w})-\ell_{\mathbf{\overline{y}},s}(\mathbf{w})|^2\,d\mathcal{H}_{\mathcal{P}}^{k}(\mathbf{w}) \\
        &\leq  C\int_{P^V(\mathbf{z},2\widehat{\mathbf{r}}_{\mathbf{z}})} \left( |F(\mathbf{w})-F(\mathbf{z})|^2 + |\ell_{\mathbf{\overline{y}},s}(\mathbf{w})-\ell_{\mathbf{\overline{y}},s}(\mathbf{z})|^2\right) \,d\mathcal{H}_{\mathcal{P}}^k(\mathbf{w}) \\ &\qquad+C\widehat{\mathbf{r}}_{\mathbf{z}}^k |\ell_{\mathbf{\overline{y}},s}(\mathbf{z})-F^{\ast}(\mathbf{z})|^2 + C\widehat{\mathbf{r}}_{\mathbf{z}}^k |F^{\ast}(\mathbf{z})-F(\mathbf{z})|^2 \\
        &\leq C\delta^{\frac{2}{k+2}} \widehat{\mathbf{r}}_{\mathbf{z}}^{k+2} +C\widehat{\mathbf{r}}_{\mathbf{z}}^k |\ell_{\mathbf{\overline{y}},s}(\mathbf{z})-F^{\ast}(\mathbf{z})|^2\\ 
        &\leq C\delta^{\frac{2}{k+2}} \widehat{\mathbf{r}}_{\mathbf{z}}^{k+2}+C\int_{P(\mathbf{z},2\widehat{\mathbf{r}}_{\mathbf{z}})} |\ell_{\mathbf{\overline{y}},s}(\mathbf{z})-F^{\ast}(\mathbf{z})|^2 \,d\widehat{\mu}(\mathbf{w}) \\
        &\leq  C \delta^{\frac{2}{k+2}} \widehat{\mathbf{r}}_{\mathbf{z}}^{k+2} + C\int_{P(\mathbf{z},2 \widehat{\mathbf{r}}_{\mathbf{z}})} |\ell_{\mathbf{\overline{y}},s}(\mathbf{w})-F^{\ast}(\mathbf{w})|^2 \,d\widehat{\mu}(\mathbf{w}),
    \end{align*}
    where we used Ahlfors regularity \eqref{eq:projectedAhlfors} in the third inequality and the Lipschitz properties in the last line. Using the cover \eqref{eq:usefulcover} and summing over $\mathbf{z}\in \widehat{\cC}_+\cap P^V(\mathbf{x},s)$, we get 
    \begin{align*} 
        s^{k+2} \kappa^2(\mathbf{y},s) &\leq C\delta^{\frac{2}{k+2}}\sum_{\mathbf{z} \in \widehat{\cC}_+\cap P^V(\ol{\mathbf{y}},4s)} \widehat{\mathbf{r}}_{\mathbf{z}}^{k+2} + C\sum_{\mathbf{z} \in \widehat{\cC}_+\cap P^V(\ol{\mathbf{y}},4s)} \int_{P(\mathbf{z},2\widehat{\mathbf{r}}_{\mathbf{z}})} |\ell_{\mathbf{\overline{y}},s}(\mathbf{w})-F^{\ast}(\mathbf{w})|^2 d\widehat{\mu}(\mathbf{w}) \\
        &\leq C\delta^{\frac{2}{k+2}}\sum_{\mathbf{z} \in \widehat{\cC}_+\cap P^V(\ol{\mathbf{y}},4s)} \widehat{\mathbf{r}}_{\mathbf{z}}^{k+2} + C \int_{P(\ol{\mathbf{y}},8s)} |\ell_{\mathbf{\overline{y}},s}(\mathbf{w})-F^{\ast}(\mathbf{w})|^2 \,d\widehat{\mu}(\mathbf{w}) \\
        &\leq C\delta^{\frac{2}{k+2}}\sum_{\mathbf{z} \in \widehat{\cC}_+\cap P^V(\ol{\mathbf{y}},4s)} \widehat{\mathbf{r}}_{\mathbf{z}}^{k+2} + C \int_{P(\ol{\mathbf{y}},8s)} |\ell_{\mathbf{\overline{y}},8s}(\mathbf{w})-F^{\ast}(\mathbf{w})|^2 \,d\widehat{\mu}(\mathbf{w})\\
        &\qquad+ Cs^{k+2}\widehat{\beta}_{\cP,k}^2(\ol{\mathbf{y}}, 8\cdot 10^4 s)\\
        &\leq C\delta^{\frac{2}{k+2}}\sum_{\mathbf{z} \in \widehat{\cC}_+\cap P^V(\ol{\mathbf{y}},4s)} \widehat{\mathbf{r}}_{\mathbf{z}}^{k+2} + C s^{k+2}\widehat{\beta}_{\mathcal{P},k}^2(\ol{\mathbf{y}},8 \cdot 10^4 s),
    \end{align*}
    where we used \eqref{eq:comparableradii} and boundedness of intersection number in the covering \eqref{eq:usefulcover} in the second line (see Remark \ref{rem:doperemark}\ref{rem:overlaps}), and we used Lemma \ref{lemma:strongneckclaim1} twice in the last two inequalities. On the other hand,
    \begin{align*}
        \int_{P^V(\mathbf{x},r)} \int_{\mathbf{\widehat{r}}_{\mathbf{y}}}^r \left( \sum_{\mathbf{z} \in \widehat{\cC}_{+}\cap P^V(\mathbf{\ol{y}},4s)} \widehat{\mathbf{r}}_{\mathbf{z}}^{k+2} \right) \,\frac{ds}{s^{k+3}}d\mathcal{H}_{\mathcal{P}}^{k}(\mathbf{y})&\leq \int_{P^V(\mathbf{x},r)} \int_{\mathbf{\widehat{r}}_{\mathbf{y}}}^r \left( \sum_{\mathbf{z} \in \widehat{\cC}_{+}\cap P^V(\mathbf{x},4r)} \widehat{\mathbf{r}}_{\mathbf{z}}^{k+2} \right) \,\frac{ds}{s^{k+3}}d\mathcal{H}_{\mathcal{P}}^{k}(\mathbf{y}) \\
        &\leq C\sum_{\mathbf{z} \in \widehat{\cC}_{+}\cap P^V(\mathbf{x},4r)} \widehat{\mathbf{r}}_{\mathbf{z}}^{k+2}\int_{P^V(\mathbf{z},2\mathbf{r}_{\mathbf{z}})} \int_{\mathbf{\widehat{r}}_{\mathbf{y}}}^r \,\frac{ds}{s^{k+3}}d\mathcal{H}_{\mathcal{P}}^{k}(\mathbf{y}) \\
        &\leq C\sum_{\mathbf{z} \in \widehat{\cC}_{+}\cap P^V(\mathbf{x},4r)} \widehat{\mathbf{r}}_{\mathbf{z}}^{k+2} \int_{P^V(\mathbf{z},2\mathbf{r}_{\mathbf{z}})} \frac{1}{\widehat{\mathbf{r}}_{\mathbf{z}}^{k+2}}\,d\mathcal{H}_{\mathcal{P}}^{k}(\mathbf{y})\\
        &\leq  C\sum_{\mathbf{z} \in \widehat{\cC}_{+}\cap P^V(\mathbf{x},4r)} \widehat{\mathbf{r}}_{\mathbf{z}}^{k}\leq Cr^k,
    \end{align*}
    where we used \eqref{eq:usefulcover} in the second line, \eqref{eq:comparableradii} in the third line and Ahlfors regularity \eqref{eq:projectedAhlfors} in the last line. We can use Ahlfors regularity because $\widehat{\mathbf{r}}_{\mathbf{x}}\leq \widehat{\mathbf{r}}_{\mathbf{y}}+C\delta r\leq 4r$ if $\delta \leq \ol{\delta}$. Similarly,
    \begin{align*}
        \int_{P^V(\mathbf{x},r)} \int_{\mathbf{\widehat{r}}_{\mathbf{y}}}^r \kappa^2(\mathbf{y},s) \, \frac{ds}{s}d\cH_{\cP}^k(\mathbf{y})&\leq C \sum_{\mathbf{z} \in \widehat{\cC}_{+}\cap P^V(\mathbf{x},r)} \int_{P^V(\mathbf{z},2\widehat{\mathbf{r}}_{\mathbf{z}})} \int_{\widehat{\mathbf{r}}_{\mathbf{y}}}^r \kappa^2(\mathbf{y},s)\, \frac{ds}{s}d\cH_{\cP}^k(\mathbf{y})\\
        &\leq Cr^k+ \sum_{\mathbf{z} \in \widehat{\cC}_{+}\cap P^V(\mathbf{x},r)} \int_{P^V(\mathbf{z},2\widehat{\mathbf{r}}_{\mathbf{z}})} \int_{\widehat{\mathbf{r}}_{\mathbf{y}}}^r\widehat{\beta}_{\cP,k}^2(\mathbf{z},\cdot 10^5s)\, \frac{ds}{s}d\cH_{\cP}^k(\mathbf{y})\\
        &\leq Cr^k+ \sum_{\mathbf{z} \in \widehat{\cC}_{+}\cap P^V(\mathbf{x},r)} \widehat{\mathbf{r}}_{\mathbf{z}}^k \int_{\frac{1}{2}\widehat{\mathbf{r}}_{\mathbf{z}}}^r\widehat{\beta}_{\cP,k}^2(\mathbf{z},10^5s)\, \frac{ds}{s}\\
        &\leq Cr^k+C\sum_{\mathbf{z} \in \widehat{\cC}_{+}\cap P^V(\mathbf{x},r)} \int_{P^V(\mathbf{z},2\widehat{\mathbf{r}_{\mathbf{z}}})}\int_{\frac{1}{4}\widehat{\mathbf{r}}_{\mathbf{w}}}^r\widehat{\beta}_{\cP,k}^2(\mathbf{w},10^5s)\,d\widehat{\mu}(\mathbf{w}) \frac{ds}{s}\\
        &\leq Cr^k+C \int_{P^V(\mathbf{x},2r)}\int_{\frac{1}{4}\widehat{\mathbf{r}}_{\mathbf{w}}}^r\widehat{\beta}_{\cP,k}^2(\mathbf{w},10^5s)\,d\widehat{\mu}(\mathbf{w}) \frac{ds}{s}.
    \end{align*}

    Now we estimate $I_1$. Suppose $\mathbf{y} \in V\setminus \widehat{\cC}_0$. For $s\in (0,\widehat{\mathbf{r}}_{\mathbf{y}}]$ and $\mathbf{w} \in P^V(\mathbf{y},s)$, using \eqref{eq:usefulcover}, we choose $\mathbf{y}' \in \widehat{\cC}_+\cap P^V(\mathbf{0},\frac{9}{5})$ such that $\mathbf{w} \in P(\mathbf{y}',2\widehat{\mathbf{r}}_{\mathbf{y}'})$. Then using Lemma \ref{lemma:strongneckclaim1}, Lemma \ref{lem:POU}\ref{lem:POU3} and \eqref{eq:comparableradii},  we get
    \begin{align}\label{eq:parametrizationhess}
        \begin{split}
            |\nabla^2F(\mathbf{w})| \leq & \left| \sum_{\mathbf{z}\in \widehat{\cC}_+} (\ell_{\mathbf{z},\gamma^{-1}\widehat{\mathbf{r}}_{\mathbf{z}}}(\mathbf{w})- \ell_{\mathbf{y}',\gamma^{-1}\widehat{\mathbf{r}}_{\mathbf{y}'}}(\mathbf{w}))\nabla^2 \psi_{\mathbf{z}}(\mathbf{w})\right| \\
            &\quad+ \sum_{\mathbf{z}\in \widehat{\cC}_+} |\nabla \psi_{\mathbf{z}}(\mathbf{w})|\cdot |\nabla(\ell_{\mathbf{z},\gamma^{-1}\widehat{\mathbf{r}}_{\mathbf{z}}}- \ell_{\mathbf{y}',\gamma^{-1}\widehat{\mathbf{r}}_{\mathbf{y}'}})(\mathbf{w})| \\ 
            \leq & C\widehat{\mathbf{r}}_{\mathbf{w}}^{-1}\delta^{\frac{2}{k+2}}.
        \end{split}
    \end{align}

    Fix $\mathbf{y}\in P^V(\mathbf{x},r)$ and $s\in (0,\mathbf{r}_{\mathbf{y}})$. Suppose $\mathbf{y}\in P^V(\mathbf{z},2\widehat{\mathbf{r}}_{\mathbf{z}})$ for some $\mathbf{z}\in \widehat{\cC}_+\cap P^V(\mathbf{0},\frac{9}{5})$. Then $\mathbf{r}_{\mathbf{y}}\leq 2\mathbf{r}_{\mathbf{z}}$. Therefore,
    \begin{align*}
        \int_{P^V(\mathbf{x},r)}\int_0^{\widehat{\mathbf{r}}_{\mathbf{y}}} \kappa^2(\mathbf{y},s)\,\frac{ds}{s}d\cH_{\cP}^k(\mathbf{y})\leq C\sum_{\mathbf{z}\in \widehat{\cC}_+\cap P^V(\mathbf{0},\frac{9}{5})} \int_{P^V(\mathbf{z},2\widehat{\mathbf{r}}_{\mathbf{z}})} \int_0^{\widehat{\mathbf{r}}_{\mathbf{y}}} \kappa^2(\mathbf{y},s)\,\frac{ds}{s}d\cH_{\cP}^k(\mathbf{y}).
    \end{align*}
    If $\ell$ is the linear component of the Taylor expansion at any $\mathbf{y}\in P^V(\mathbf{z},2\widehat{\mathbf{r}}_{\mathbf{z}})$, then we know that
    \begin{align*}
        \sup_{\mathbf{w}\in P^V(\mathbf{y},s)}\verts{F(\mathbf{w})-\ell(\mathbf{w})}^2\leq Cs^4 \sup_{\mathbf{w}\in P^V(\mathbf{y},s)}(\verts{\cd^2 F(\mathbf{w})}^2 +\verts{\pdt F(\mathbf{w})}^2)\leq C s^2 \delta^{\frac{2}{k+2}}.
    \end{align*}
    Thus,
    \begin{align*}
        \kappa^2(\mathbf{y},s)\leq \frac{1}{s^{k+2}}\int_{P^V(\mathbf{y},s)} \verts{F(\mathbf{w})-\ell(\mathbf{w})}^2\,d\cH_{\cP}^k(\mathbf{w})\leq C\frac{s^2}{\widehat{\mathbf{r}}_{\mathbf{z}}^2}\delta^{\frac{2}{k+2}}.
    \end{align*}
    Therefore,
    \begin{align*}
        \int_{0}^{2\widehat{\mathbf{r}}_{\mathbf{z}}}\int_{P^V(\mathbf{z},\widehat{\mathbf{r}}_{\mathbf{z}})} \kappa^2(\mathbf{y},s)\,d\cH_{\cP}^k(\mathbf{y})\frac{ds}{s}\leq C\int_{0}^{2\widehat{\mathbf{r}}_{\mathbf{z}}}\int_{P^V(\mathbf{z},\widehat{\mathbf{r}}_{\mathbf{z}})} \frac{s^2}{\widehat{\mathbf{r}}_{\mathbf{z}}^2}\delta^{\frac{2}{k+2}}\,d\cH_{\cP}^k(\mathbf{y})\frac{ds}{s}\leq C\widehat{\mathbf{r}}_{\mathbf{z}}^k\delta^{\frac{2}{k+2}}.
    \end{align*}
    Now summing over $\mathbf{z}$ we get the desired estimate.

    Second we consider arbitrary $\mathbf{x}\in P^V(\mathbf{0},\frac{3}{2})$ and $r\in (0,\frac{1}{10}]$. If $r\leq \widehat{\mathbf{r}}_{\mathbf{x}}$, then the claim follows by an estimate similar to that for $I_1$. Suppose instead $\widehat{\mathbf{r}}_{\mathbf{x}}<r$. Then choose $\ol{\mathbf{x}}\in P(\mathbf{0},\frac{9}{5})$ such that $\mathbf{x}\in P^V(\ol{\mathbf{x}},\widehat{\mathbf{r}}_{\ol{\mathbf{x}}})$. Then $P(\mathbf{x},r)\subset P(\ol{\mathbf{x}},4r)$ because $\widehat{\mathbf{r}}_{\ol{\mathbf{x}}}\leq 2\widehat{\mathbf{r}}_{\mathbf{x}}<2r$. Thus, we are back to the case $\mathbf{x}\in P(\mathbf{0},\frac{3}{2})\cap \widehat{\cC}$.
\end{proof}

To ensure the hypothesis of Proposition \ref{prop:criterionforregularity}, we now prove a Carleson estimate for the $\beta_{\mathcal{P},k}^2$ term appearing in the right-hand side of \eqref{eq: kappa and beta}. The proof is inspired by \cite[(5.10)]{cheeger-jiang-naber-2021-Sharp-quantitative}.

\begin{lemma} \label{lem:carlesoncondition} 
    For any $\mathbf{x}\in P(\mathbf{0},\frac{3}{2})$ and $r>0$ such that $P(\mathbf{x},20r)\subset P(\mathbf{0},\frac{7}{4})$,
    \begin{align*}
        \int_{P^V(\mathbf{x},10r)} \int_{\frac{1}{8}\widehat{\mathbf{r}}_{\mathbf{y}}}^r\widehat{\beta}_{\mathcal{P},k}^2(\mathbf{y},10^5s)\,\frac{ds}{s}d\widehat{\mu}(\mathbf{y}) \leq C\delta r^{k}.
    \end{align*}
\end{lemma}
\begin{proof} 
    Set $r_{i}\coloneqq 2^{-i}$ for $i \in \mathbb{N}$. Applying 
    Fubini's theorem, Lemma \ref{lem:beta}, and \eqref{eq:projectedAhlfors}, we obtain
    \begin{align*} 
        &\int_{P(\mathbf{x},10r)} \int_{\frac{1}{8}\widehat{\mathbf{r}}_{\mathbf{y}}}^r \widehat{\beta}_{\mathcal{P},k}^2(\mathbf{y},10^5s)\,\frac{ds}{s}d\widehat{\mu}(\mathbf{y}) \\
        &\leq C\int_{P(\mathbf{x},10r)}\sum_{\frac{1}{8}\widehat{\mathbf{r}}_{\mathbf{y}}\leq r_i \leq 10^5r} \widehat{\beta}_{\mathcal{P},k}^2(\mathbf{y},r_{i})d\widehat{\mu}(\mathbf{y})\\
        &\leq C\int_{P(\mathbf{x},10r)} \sum_{i\in \mathbb{N}} \frac{\1_{\frac{1}{8}\{\widehat{\mathbf{r}}_{\mathbf{y}}\leq r_i \leq 10^5r\} }}{r_i^{k}} \left( \int_{P(\mathbf{y},r_i)} \mathcal{E}_{20r_i}(\mathbf{z})d\widehat{\mu} (\mathbf{z}) \right) d\widehat{\mu}(\mathbf{y}) \\
        &\leq  C \sum_{i \in \mathbb{N}} \frac{1}{r_i^{k}}\int_{P(\mathbf{x},20r)} \1_{\{ \frac{1}{16}\widehat{\mathbf{r}}_{\mathbf{z}}\leq r_i \leq 10^5r\} } \mathcal{E}_{20r_i}(\mathbf{z}) \left( \int_{P(\mathbf{z},r_i)} d\widehat{\mu}(\mathbf{y}) \right) d\widehat{\mu}(\mathbf{z}) \\
        &\leq  C\sum_{i\in \mathbb{N}} \int_{P(\mathbf{x},20r)} \1_{\{ \frac{1}{16}\widehat{\mathbf{r}}_{\mathbf{z}}\leq r_i \leq 10^5r\} } \mathcal{E}_{20r_i}(\mathbf{z})d\widehat{\mu}(\mathbf{z})\leq  C\int_{P(\mathbf{x},20r)} \left( N_{\mathbf{z}}(10^{12}r^2)-N_{\mathbf{z}}(\widehat{\mathbf{r}}_{\mathbf{z}}^2) \right) d\widehat{\mu}(\mathbf{z})\\
        &\leq  C\delta \widehat{\mu}(P(\mathbf{x},20r)) \leq C\delta r^{k},
    \end{align*}
    where we used $\widehat{\mathbf{r}}_{\mathbf{x}}\leq \widehat{\mathbf{r}}_{\mathbf{y}}+10\delta r\leq 10r$ and Ahlfors regularity \eqref{eq:projectedAhlfors} in the last line.
\end{proof}

\begin{proof}[Proof of Theorem \ref{thm:strongneckstructure}]
    By translation and rescaling, we may assume that $\mathbf{x}_0=\mathbf{0}$ and $r=1$. Choose $\varphi\in C_c^\infty (P^V(\mathbf{0},\frac{5}{4}))$ such that $\varphi\equiv 1$ in $P^V(\mathbf{0},\frac{6}{5})$ and $0\leq \varphi\leq 1$ and estimates on $\verts{\pdt^i\cd^j \varphi}$ for $2i+j\leq 2$. We claim that $G\coloneqq \varphi F$ is the desired function $G$. Note that \eqref{eq graph is close to centers} follows from Lemma \ref{lemma:newlipschitz}\ref{lemma:newlipschitz1} and the fact that $F|_{P^V(\mathbf{0},\frac{6}{5})}=G|_{P^V(\mathbf{0},\frac{6}{5})}$. Because $\mathbf{0} \in \mathcal{C}$, we know $F(\mathbf{0})=\mathbf{0}$, so from Lemma \ref{lemma:newlipschitz}\ref{lemma:newlipschitz2}, it follows that 
    \begin{align*}
        \sup_{P^V(\mathbf{0},\frac{7}{4})}|F|\leq C\delta^{\frac{2}{k+2}}.
    \end{align*}
    Again using Lemma \ref{lemma:newlipschitz}\ref{lemma:newlipschitz2}, we conclude that $G$ is $C\delta^{\frac{2}{k+2}}$-Lipschitz. Thus, it remains to prove a BMO estimate for $\partial_t^{\frac{1}{2}}G$.
    
    In light of Proposition \ref{prop:criterionforregularity}, it suffices to establish a Carleson estimate for $\kappa_k^2(\mathbf{x},r;G)$:
    \begin{align}
        \int_{0}^{r} \int_{V\cap P(\mathbf{x},r)}\kappa^{2}(\mathbf{y},s;G)\,d\mathcal{H}_{\mathcal{P}}^{k}(\mathbf{y})\frac{ds}{s}<C\delta^{\frac{2}{k+2}} r^{k}.\label{eq carleson for appendix}
    \end{align}
    for any $\mathbf{x} \in V$ and $r\in (0,100]$. To see \eqref{eq carleson for appendix}, we break the integral depending on location, scale and resolution. In the support of $\varphi$ at small scales, we use Taylor expansion of $\varphi$ to find a comparison linear function to estimate $\kappa^2$. Moreover, there is no contribution from outside of the support of $\varphi$ at small scales. Finally, at larger scales, we can just use a crude global bound on $G$. More precisely, we have
    \begin{align*}
        \int_{P^V(\mathbf{x},r)} \int_{0}^{r} \kappa^{2}(\mathbf{y},s;G)\,\frac{ds}{s}d\mathcal{H}_{\mathcal{P}}^{k}(\mathbf{y})&= \int_{P^V(\mathbf{x},r)\cap P(\mathbf{0},\frac{4}{3})} \int_{0}^{\min\{r,\frac{1}{100}\}
        } \kappa^{2}(\mathbf{y},s;G)\,\frac{ds}{s}d\mathcal{H}_{\mathcal{P}}^{k}(\mathbf{y})\\
        &\qquad+ \int_{P^V(\mathbf{x},r)\setminus P(\mathbf{0},\frac{4}{3})} \int_{0}^{\min\{r,\frac{1}{100}\}} \kappa^{2}(\mathbf{y},s;G)\,\frac{ds}{s}d\mathcal{H}_{\mathcal{P}}^{k}(\mathbf{y})\\
        &\qquad+\int_{P^V(\mathbf{x},r)} \int_{\min\{r,\frac{1}{100}\}
        }^r \kappa^{2}(\mathbf{y},s;G)\,\frac{ds}{s}d\mathcal{H}_{\mathcal{P}}^{k}(\mathbf{y})\\
        &\eqqcolon J_1+J_2+J_3.
    \end{align*}
    Note that $J_2=0$. In fact, for any $\mathbf{y}\in P^V(\mathbf{x},r)\setminus P(\mathbf{0},\frac{4}{3})$ and $s\leq \min\{r,\frac{1}{100}\}$, we note that $P(\mathbf{y},s)\cap P(\mathbf{0},\frac{5}{4})=\emptyset$. Thus, $G\equiv 0$ on $P(\mathbf{y},s)$. Because $\kappa^2(\mathbf{y},s;G) \leq C\delta^{\frac{2}{k+2}}$ when $s\geq \frac{1}{100}$, we also obtain $J_3 \leq C\delta^{\frac{2}{k+2}}r^k$. 

    To estimate $J_1$, let $\mathbf{y}\in P^V(\mathbf{0},\frac{4}{3})$ and $s\in (0,\frac{1}{100}]$. Define
    \begin{align*}
        \ol{\ell}\in \arg\min_{\ell}\frac{1}{s^k} \int_{ P^V(\mathbf{y},s)}\left(\frac{|F(\mathbf{w})-\ell(\mathbf{w})|}{s}\right)^{2}\mathcal{H}_{\mathcal{P}}^{k}(\mathbf{w}),
    \end{align*}
    where $\ell$ is a function $\ell\colon V\to \mathbb{R}^{n+2-k}$ satisfying $\partial_{t}\ell\equiv 0$. Define 
    \begin{align*} 
        \widehat{\ell}(\mathbf{w})& \coloneqq \varphi(\mathbf{y})\overline{\ell}(\mathbf{w})+F(\mathbf{y})\nabla \varphi(\mathbf{y})\cdot(\mathbf{w}-\mathbf{y}),\\
        q(\mathbf{w})& \coloneqq \varphi(\mathbf{w})-\varphi(\mathbf{y})-\nabla \varphi(\mathbf{y})\cdot (\mathbf{w}-\mathbf{y}).
    \end{align*}
    Using that $\varphi$ is Lipschitz, $|F|\leq C\delta^{\frac{2}{k+2}}$, $0\leq \varphi \leq 1$, and $|q(\mathbf{w})|\leq C|\mathbf{w}-\mathbf{y}|^2$, we obtain
    \begin{align} \label{eq:intermediate}
    \begin{aligned}
        \kappa_k^2(\mathbf{y},s;G) &\leq  \frac{1}{s^{k}}\int_{ P^V(\mathbf{y},s)}\left(\frac{|\varphi(\mathbf{w})F(\mathbf{w})-\widehat{\ell}(\mathbf{w})|}{s}\right)^{2}\,d\mathcal{H}_{\mathcal{P}}^{k}(\mathbf{w}) \\
        &\leq  \frac{1}{s^{k+2}}\int_{ P^V(\mathbf{y},s)}2\bigg(|\varphi(\mathbf{y})(F(\mathbf{w})-\ol{\ell}(\mathbf{w}))|^2+|q(\mathbf{w})F(\mathbf{w})|^2 \\
        &\qquad \qquad\qquad\qquad\qquad+|(F(\mathbf{w})-F(\mathbf{y}))\cd \varphi(\mathbf{y})\cdot (\mathbf{w}-\mathbf{y})|^2\bigg)\,d\mathcal{H}_{\mathcal{P}}^{k}(\mathbf{w})\\
        &\leq 2\kappa_k^2(\mathbf{y},s;F)+\frac{2}{s^{k+2}}s^k \cdot C \delta^{\frac{2}{k+2}} \cdot  Cs^4 \\
        & \leq 2\kappa_k^2(\mathbf{y},s;F)+C\delta^{\frac{2}{k+2}}s^2.
    \end{aligned}
    \end{align}
    
    If $J_1 \neq 0$, then $P^V(\mathbf{x},\frac{1}{100})\cap P(\mathbf{0},\frac{4}{3})\neq \emptyset$, hence $\mathbf{x} \in P(\mathbf{0},\frac{3}{2})$. We may therefore combine \eqref{eq:intermediate}, Lemma \ref{lemma:newkappa}, Lemma \ref{lem:carlesoncondition} to obtain $J_1 \leq C\delta^{\frac{2}{k+2}}r^k$. Thus, \eqref{eq carleson for appendix} holds.

    Therefore, the BMO estimate for $\partial_t^{\frac{1}{2}}G$ follows by using Proposition \ref{prop:criterionforregularity} for each component of $G$.
\end{proof}

\section{Neck Decomposition}\label{neck decomposition}
In this section, we prove Theorem \ref{theorem-neck-decomposition}, Theorem \ref{theorem-neck-decomposition2}, and Theorem \ref{theorem-neck-decomposition3}, which decompose $P(\mathbf{0},1)$ into $k$-neck regions, $(k+1)$-symmetric balls, and a residual set. The strategy for these decomposition comes from \cite[Section 7]{jiang-naber-2021-l2-curvature} (see also \cite{naber-valtorta-YM}). 

In the first version of the neck decomposition, the residual set has $k$-Hausdorff measure zero. Furthermore, the neck regions and $(k+1)$-symmetric balls have summable Minkowski content.
\begin{theorem}[Neck Decomposition] \label{theorem-neck-decomposition}
    Suppose $u$ is a caloric function satisfying $N(\gamma^{-6})\leq \Lambda$. For any $\epsilon,\eta\in (0,1]$, there exists a decomposition
    \begin{align*}
        P(\mathbf{0},1)\subseteq \bigcup_a \mathcal{N}^a\cup \bigcup_b P(\mathbf{x}_b,r_b) \cup \left( \widetilde{\mathcal{C}} \cup \bigcup_a \mathcal{C}_{a,0} \right),
    \end{align*}
    where the following hold:
    \begin{enumerate}[label=(\alph*)]
        \item \label{neckdecomposition:aballs} $\mathcal{N}^a = P(\mathbf{x}_a,2r_a)\setminus \overline{P}(\mathcal{C}_a,\mathbf{r}_a)$ is an $(m_a,k,\epsilon ,C^{-\Lambda}\eta)$-neck region for some positive integer $m_a \leq C\Lambda$, where $\cC_a$ is the center set associated to $\cN^a$;
    
        \item \label{neckdecomposition:bballs} $\mathcal{E}_{100 r_b}^{k+1,1}(\mathbf{y}_b) \leq \eta$ for some $\mathbf{y}_b\in P(\mathbf{x}_b,4r_b)$;

        \item \label{neckdecomposition:contentestimate} A $k$-dimensional parabolic Minkowski content estimate holds:
        \begin{align*}
            \sum_a r_a^k + \sum_b r_b^k\leq C^{\Lambda^2}(\epsilon \eta)^{-C\Lambda};
        \end{align*}

        \item \label{neckdecomposition:Minkowskiofcenters} Further, $\mathcal{H}_{\mathcal{P}}^k(\widetilde{\mathcal{C}})=0$. 
    \end{enumerate}
\end{theorem}

In the second version of the neck decomposition, given a ``resolution" $r_\ast$, the residual set is of scale commensurable to $r_\ast$. Commensurability will be crucial in proving the volume estimates for the quantitative nodal and singular sets at resolution $r_\ast$. Such a consideration is necessary because we cannot approximate our solutions $u$ by solutions whose neck regions all satisfy $\mathcal{C}_0=\emptyset$ unlike the settings of \cite{jiang-naber-2021-l2-curvature,naber-valtorta-YM}.

\begin{theorem}[Finite-Resolution Neck Decomposition]\label{theorem-neck-decomposition2}
    Suppose $u$ is a caloric function satisfying $N(\gamma^{-6})\leq \Lambda$. For any $\epsilon,\eta,r_*\in (0,1]$, there exists a decomposition
    \begin{align*}
        P(\mathbf{0},1)\subseteq \bigcup_a \mathcal{N}^a\cup \bigcup_b P(\mathbf{x}_b,r_b) \cup  \bigcup_f P(\mathbf{x}_f,r_f),
    \end{align*}
    where the following hold:
    \begin{enumerate}[label=(\alph*)]
        \item $\mathcal{N}^a = P(\mathbf{x}_a,2r_a)\setminus \overline{P}(\mathcal{C}_a,\mathbf{r}_a)$ is an $(m_a,k,\epsilon, C^{-\Lambda}\eta)$-neck region for some positive integer $m_a \leq C\Lambda$, where $\cC_a$ is the center set associated to $\cN^a$ and $\mathbf{r}_\bullet \geq r_\ast$;

        \item $\mathcal{E}_{100 r_b}^{k+1,1}(\mathbf{y}_b) \leq \eta$ for some $\mathbf{y}_b\in P(\mathbf{x}_b,4r_b)$;

        \item \label{neckdecomposition:contentestimate2} A $k$-dimensional parabolic Minkowski content estimate holds:
        \begin{align*}
            \sum_a r_a^k + \sum_b r_b^k +\sum_f r_f^k\leq C^{\Lambda^2}(\epsilon \eta)^{-C\Lambda},
        \end{align*}
        such that $r_a,r_b\geq r_\ast$ and $r_f\in [r_\ast, C^{\Lambda} \eta^{-\frac{2}{n}}\epsilon^{-20n^2}r_\ast]$.
    \end{enumerate}
\end{theorem}

In the third version, we strengthen the neck decomposition to include strong neck regions when $k\in \{n,n+1\}$. This version is useful in proving stronger regularity of the nodal and singular sets. 
\begin{theorem}[Refined Neck Decomposition]\label{theorem-neck-decomposition3}
    Let $k\in\{n,n+1\}$. Suppose $u$ is a caloric function satisfying $N_{\mathbf{0}}(\gamma^{-100})\leq \Lambda$. For any $\epsilon\in (0,1]$ and $\eta \leq \ol{\eta}(\Lambda)$, there exists a decomposition
    \begin{align*}
        P(\mathbf{0},1)\subseteq \bigcup_a \mathcal{N}^a 
        \cup  \bigcup_g P(\mathbf{x}_g,r_g)\cup\left( \widetilde{\mathcal{C}} \cup \bigcup_a \mathcal{C}_{a,0} \right),
    \end{align*}
    where the following hold:
    \begin{enumerate}[label=(\alph*)]
        \item $\mathcal{N}^a = P(\mathbf{x}_a,2r_a)\setminus \overline{P}(\mathcal{C}_a,\mathbf{r}_a)$ is an $(m_a,k,\epsilon,C(\Lambda)^{-1}\eta)$-strong neck region of scale $r_a$ for some integer $m_a \in[m_\ast, C\Lambda]$, where
         \begin{align} \label{eq def of mstar}
            m_{\ast} \coloneqq
            \begin{cases} 
                1 & \text{ if } k=n+1, \\ 
                2 & \text{ if } k=n.
            \end{cases}
        \end{align}
        Here, $\cC_a$ is the center set associated to $\cN^a$. Further, $\cN^a=\widehat{\cN}^{a}\cap P(\mathbf{x}_a,2r_a)$ where $\widehat{\cN}^{a}$ is an $(m_a,k,\epsilon,C(\Lambda)^{-1}\eta)$ strong neck region of scale $2r_a;$

        \item For all $g$,
        \begin{align*}
            \sup_{\mathbf{x}\in P(\mathbf{x}_g,r_g)} N_{\mathbf{x}}(\epsilon^{-2}r_g^2)\leq m_\ast-1+C(\Lambda)^{-1}\epsilon;
        \end{align*} 

        \item \label{neckdecomposition:contentestimate3} A $k$-dimensional parabolic Minkowski content estimate holds:
        \begin{align*}
            \sum_a r_a^k  +\sum_g r_g^k\leq C(\Lambda)(\epsilon \eta)^{-C(\Lambda)};
        \end{align*}

        \item \label{neckdecomposition:Minkowskiofcenters3} Further, $\mathcal{H}_{\mathcal{P}}^k(\widetilde{\mathcal{C}})=0$. 
    \end{enumerate}
\end{theorem}

We prove all neck decomposition theorems by an argument involving double induction. Namely, we first induct on scales keeping the frequency fixed. Then we induct on the frequency, which drops by $1$ at each step ({cf.} Lemma \ref{lemma-refined-monotonicity-of-frequency}). Eventually, the frequency of the remaining balls is sufficiently small so that $u$ is very symmetric. 
\begin{definition}
    For $\delta>0$, $m\in \mathbb{N}_0$, and $r\in (0,1]$, define the \textit{pinched set}
    \begin{align*}
        \mathcal{V}_{\delta,m,r}^u(\mathbf{x})\coloneqq \left\{ \mathbf{y}\in P(\mathbf{x},4r)\colon |N_{\mathbf{y}}^u(s^2)-m|\leq \delta \text{ for any } s\in [\delta r,\delta^{-1}r] \right\}.
    \end{align*}
\end{definition}
To carry out the induction procedure, given $\mathbf{x}\in P(\mathbf{0},1)$, $m,k\in \mathbb{N}_0$, and $r,\epsilon,\eta,\alpha>0$, we classify $P(\mathbf{x},r)$ into the following categories:
\begin{enumerate}[label=(\alph*)]
    \item A ball $P(\mathbf{x}_a,r_a)$ is associated with an $(m,k,\epsilon ,\eta)$-neck region $\mathcal{N}^a = P(\mathbf{x}_a,2r_a)\setminus \overline{P}(\mathcal{C}^a,\mathbf{r}^a)$;\label{a-balls}

    \item A ball $P(\mathbf{x}_b,r_b)$ such that $\mathcal{E}_{100 r_b}^{k+1,1}(\mathbf{y}_b) \le \eta$ for some $\mathbf{y}_b\in P(\mathbf{x}_b,4r_b)$;\label{b-balls} 

    \item A ball $P(\mathbf{x}_c,r_c)$ is not a \ref{b-balls}-ball, and $\mathcal{V}_{\delta,m,r_c}(\mathbf{x}_c)$ is $(k,\alpha r_c)$-independent;\label{c-balls}

    \item A ball $P(\mathbf{x}_d,r_d)$ is not a \ref{b-balls}-ball, and $\mathcal{V}_{\delta,m,r_d}(\mathbf{x}_d) \neq \emptyset$ is not $(k,\alpha r_d)$-independent;\label{d-balls}

    \item A ball $P(\mathbf{x}_e,r_e)$ is not a \ref{b-balls}-ball, and $\mathcal{V}_{\delta,m,r_e}(\mathbf{x}_e)=\emptyset$.\label{e-balls}
\end{enumerate}

\subsection{Covering of \ref{c-balls}-balls}

We first show that any \ref{c-balls}-ball can be decomposed into a neck region, \ref{b-balls}-balls, \ref{d-balls}-balls, and \ref{e-balls}-balls. Our proof is similar to \cite[Section 7.3]{jiang-naber-2021-l2-curvature}, except that using Lemma \ref{newlineup}, we can show that the approximating planes in our inductive covering can be taken independent of scale and location.

\begin{proposition}[Covering of \ref{c-balls}-balls]\label{cballcovering} 
    For any $\alpha, \epsilon, \eta, \Lambda>0$, the following statements hold when $\delta \leq C(\alpha)^\Lambda \eta^{\frac{1}{n}} \epsilon^{10n^2}$. Let $u$ be a caloric function satisfying 
    \begin{align}
        \sup_{\mathbf{x}\in P(\mathbf{0},20r)} N_{\mathbf{x}}^u(\delta^{-2}r^2) \leq m+\delta,\label{eq: c ball decomposition frequency pinching}
    \end{align}
    for some non-negative integer $m\leq \Lambda$. Suppose $\mathcal{V}\coloneqq \mathcal{V}^u_{\delta,m,r}(\mathbf{0})$ is a $(k,\alpha r)$-independent set and that $\mathcal{E}_{100r}^{k+1,1}(\mathbf{y})> \eta$ for all $\mathbf{y}\in P(\mathbf{0},2r)$. Then the following hold:
    \begin{enumerate}[label=(\arabic*)]
        \item \label{c-ball decomposition-continuous} There exists a decomposition
        \begin{align*}
            P(\mathbf{0},r)\subseteq (\mathcal{C}_0 \cup \mathcal{N}) \cup \bigcup_{b\in B} P(\mathbf{x}_b,r_b) \cup \bigcup_{d\in D} P(\mathbf{x}_d,r_d) \cup \bigcup_{e\in E} P(\mathbf{x}_e,r_e),
        \end{align*}
        where $\mathcal{N}$ is an $(m,k,\epsilon,C(\alpha)^{-\Lambda}\eta)$-neck of scale $10r$, and we have a $k$-dimensional content estimate
        \begin{align*}
            \mathcal{H}_{\mathcal{P}}^k(\mathcal{C}_0) + \sum_{b\in B}r_b^k + \sum_{d\in D}r_d^k + \sum_{e\in E}r_e^k \leq Cr^k.
        \end{align*}

        \item \label{c-ball decomposition-discrete} For any $r_\ast\in (0,r]$, there exists a decomposition
        \begin{align*}
            P(\mathbf{0},r)\subseteq \mathcal{N} \cup \bigcup_{b\in B} P(\mathbf{x}_b,r_b) \cup \bigcup_{d\in D} P(\mathbf{x}_d,r_d) \cup \bigcup_{e\in E} P(\mathbf{x}_e,r_e) \cup \bigcup_{f\in F} P(\mathbf{x}_f,r_f),
        \end{align*}
        where $\mathcal{N}$ is an $(m,k,\epsilon,C(\alpha)^{-\Lambda}\eta)$-neck of scale $10r$ with $\mathbf{r}_\bullet \geq r_\ast$, and we have a $k$-dimensional content estimate
        \begin{align*}
             \sum_{b\in B}r_b^k + \sum_{d\in D}r_d^k + \sum_{e\in E}r_e^k + \sum_{f\in F}r_f^k \leq Cr^k,
        \end{align*}
        $r_e,r_b,r_d\geq r_\ast$, and $r_f \in[r_\ast,\gamma^{-1} r_\ast]$.
    \end{enumerate}
\end{proposition}

\begin{remark}
    In Proposition \ref{cballcovering}\ref{c-ball decomposition-continuous}, $\cC=\cC_0\cup \{\mathbf{x}_b\}\cup \{\mathbf{x}_d\}\cup \{\mathbf{x}_e\}$ and $\mathbf{r}_{\mathbf{x}_b}=r_b$, $\mathbf{r}_{\mathbf{x}_d}=r_d$, $\mathbf{r}_{\mathbf{x}_e}=r_e$. A similar description exists for the center set in Proposition \ref{cballcovering}\ref{c-ball decomposition-discrete}.
\end{remark}

\begin{proof}
    By parabolic rescaling, we may assume $r=1$. We will first construct a sequence $(\mathcal{N}^j)_{j\in \mathbb{N}}$ of $(m,k,\epsilon ,C(\alpha)^{-\Lambda}\eta)$-neck regions inductively.

    \ref{c-ball decomposition-continuous} Because $\mathcal{V} \neq \emptyset$, we can choose $\mathbf{x}_a \in \mathcal{V}$. Set $s_0\coloneqq 10^4$. We claim that $u$ is not $(k+1,C^{-\Lambda}\eta^2,s_0)$ symmetric at $\mathbf{x}_a$ if $\delta \leq C^{-\Lambda}\eta^2$. Suppose not. Then we may apply Lemma \ref{newlineup}\ref{plane of symmetry inside pinched points} with $\cC\leftarrow \{\mathbf{x}_a\}$, $\delta\leftarrow C^{-\Lambda}\eta^2$, $s_0\leftarrow s_0$, $r_0\leftarrow s_0$, $\kappa\leftarrow \gamma$, $\mathbf{r}_{\mathbf{x}_a}\leftarrow 1$, and $s\leftarrow 100$ so that $\mathcal{E}_{100}^{k+1,1}(\mathbf{x}_a)<\eta$, which is a contradiction.

    Since $\cV$ is $(k,\alpha)$-independent, $\cE_{s_0}^{k,s_0^{-1}\alpha} (\mathbf{x}_a)\leq \delta$. By Theorem \ref{thm: sym split equiv}\ref{thm: sym split equiv-1}, $u$ is $(k, C(\alpha)^{\Lambda} \delta,s_0)$-symmetric at $\mathbf{x}_a$. From this and \eqref{eq: c ball decomposition frequency pinching}, we can apply Lemma \ref{newlineup} with $\cC\leftarrow \{\mathbf{x}_a\}$, $r_0\leftarrow s_0$, $\kappa \leftarrow  10^3\gamma^{-1}\delta$, $\delta \leftarrow C(\alpha)^\Lambda \delta$, and $\mathbf{r}_{\bullet} \leftarrow \gamma$ to obtain $V\in {\rm Gr}_\cP(k)$ such that:
    \begin{enumerate}[label=(1.\Alph*), ref=\Alph*]
    
        \item \label{def of V-2} for all $s\in [\gamma,\gamma^{-5}]$, $V\cap P(\mathbf{x}_a,10s)\subseteq \{\mathbf{y}\in P(\mathbf{x}_a,10s)\colon |N_{\mathbf{y}}(10^6\delta^2 s^2)-m|< C(\alpha)^{\Lambda}\sqrt{\delta}\};$
        
        \item \label{def of V-3} for all $s\in [\gamma,1]$, for any $\zeta \geq C(\alpha)^{\Lambda}\sqrt{\delta}$, we have 
        \begin{align*}
            \{ \mathbf{y} \in P(\mathbf{x}_a,10s): |N_{\mathbf{y}}(s^2)-m|< \zeta\} \subseteq P(V, C(\alpha)^{\Lambda} \eta^{-\frac{1}{n}}\zeta^{\frac{1}{n}}s).
        \end{align*}
    \end{enumerate} 
    Let $\{\mathbf{x}_i^1\}_{i=1}^{N_1}$ be a maximal subset of $V \cap P(\mathbf{x}_a,20)$ satisfying $\mathbf{x}_a\in \{\mathbf{x}_i^1\}_{i=1}^{N_1}$ and $|\mathbf{x}_i^1-\mathbf{x}_j^1|\ge 2\gamma^3$ when $i\neq j$. Each ball $P(\mathbf{x}_i^1,\gamma)$ must be of the type \ref{b-balls}, \ref{c-balls}, \ref{d-balls} or \ref{e-balls}. Let $\mathcal{C}^1 \coloneqq \{ \mathbf{x}_i^1 \}_{i=1}^{N_1}$, and set $\mathbf{r}_{\mathbf{x}}^1\coloneqq \gamma$ for all $\mathbf{x}\in \mathcal{C}^1$. Defining $\mathcal{N}^1 \coloneqq P(\mathbf{x}_a,20)\setminus \overline{P}(\mathcal{C}^1,\mathbf{r}^1)$, we then have
    \begin{align*}
        P(\mathbf{0},1) \subseteq \mathcal{N}^1 \cup \bigcup_{b\in B^1} P(\mathbf{x}_b,r_b) \cup \bigcup_{c\in C^1} P(\mathbf{x}_c,r_c) \cup \bigcup_{d\in D^1} P(\mathbf{x}_d,r_d) \cup \bigcup_{e\in E^1} P(\mathbf{x}_e,r_e),
    \end{align*}
    where $r_b=r_c=r_d=r_e=\gamma$.
    \begin{claim}\label{claim neck region base case}
        If $\delta\leq C(\alpha)^{-\Lambda}\min \{ \epsilon^2,\eta\}$, $\mathcal{N}^1$ is an $(m,k,\epsilon ,C(\alpha)^{-\Lambda}\eta)$-neck region of scale $10$.
    \end{claim}
    \begin{proof}[Proof of Claim \ref{claim neck region base case}]
        Properties \ref{neck-vitali-covering} and \ref{neck-Hausdorff} hold by construction. Property \ref{neck-frequency-pinching} holds by \itemref{1}{def of V-2} if $\delta \leq C(\alpha)^{-\Lambda}\epsilon^2$. Suppose $\mathbf{x} \in \mathcal{C}^1$ and $s\in [\gamma,\gamma^{-3}]$. Applying Lemma \ref{newlineup}\ref{existence of plane of symmetry} with $\mathcal{C} \leftarrow \mathcal{C}^1$, $r_0 \leftarrow 1$, $\kappa \leftarrow \gamma^3$, $\delta \leq C(\alpha)^{\Lambda}\sqrt{\delta}$, and $\mathbf{r}_{\bullet} \leftarrow \gamma$ yields that $u$ is $(k,C(\alpha)^{\Lambda}\sqrt{\delta},s)$-symmetric at $\mathbf{x}$ with respect to $V$. We claim that $u$ is not $(k+1,C(\alpha)^{-\Lambda}\eta,s)$-symmetric at $\mathbf{x}$. Suppose not. Then Lemma \ref{newlineup}\ref{existence of plane of symmetry} with $\mathcal{C} \leftarrow \mathcal{C}^1$, $r_0 \leftarrow 1$ $\kappa \leftarrow \gamma^3$, $s_0 \leftarrow s$, $\delta \leftarrow C(\alpha)^{-\Lambda}\eta$, $\mathbf{r}_{\bullet} \leftarrow \gamma$ implies that $u$ is $(k+1,C(\alpha)^{\Lambda}\delta,s_0)$-symmetric at $\mathbf{x}_a$, which is a contradiction. Thus \ref{neck-k-sym} holds.
    \end{proof} 
    Set $\epsilon' \coloneqq C(\alpha)^{\Lambda}\eta^{-\frac{1}{n}}\delta^{\frac{1}{2n}}$, $\eta'\coloneqq C(\alpha)^{-\Lambda}\eta$. Suppose by induction that we have defined collections of centers
    \begin{align*}
        \mathcal{C}^{\ell} \coloneqq  \{ \mathbf{x}_b\}_{b\in B^{\ell}} \cup \{ \mathbf{x}_c\}_{c\in C^{\ell}} \cup \{ \mathbf{x}_d\}_{d\in D^{\ell}} \cup \{ \mathbf{x}_e\}_{e\in E^{\ell}}
    \end{align*}
    for $\ell = 1,\dots,j$ with
    \begin{align*}
        P(\mathbf{0},1) \subseteq \mathcal{N}^{\ell} \cup \bigcup_{b\in B^{\ell}} P(\mathbf{x}_b,r_b) \cup \bigcup_{c\in C^{\ell}} P(\mathbf{x}_c,r_c)\cup \bigcup_{d\in D^{\ell}} P(\mathbf{x}_d,r_d) \cup \bigcup_{e\in E^{\ell}} P(\mathbf{x}_e,r_e),
    \end{align*}
    where $\mathcal{N}^{\ell} = P(\mathbf{x}_a,20) \setminus \overline{P}(\mathcal{C}^{\ell},\mathbf{r}^{\ell})$ is an $(m,k,\epsilon ,\eta')$-neck region with corresponding plane $V\in \text{Gr}_{\mathcal{P}}(k)$, such that additionally the following hold for $1\leq \ell \leq j$:
    \begin{enumerate}[label=(j.\roman*), ref=\roman*]
        \item \label{induction hypothesis i} for $b\in B^{\ell}\setminus B^{\ell-1}$, $c\in C^{\ell}$, $d\in D^{\ell} \setminus D^{\ell-1}$, and $e\in E^{\ell} \setminus E^{\ell-1}$, we have $r_b=r_c=r_d=\gamma^{\ell}$, 
       
        \item \label{induction hypothesis ii} there exist $\mathbf{y}_c \in \mathcal{V}_{\delta,m,r_c}(\mathbf{x}_c)$ 
        for each $c\in C^{\ell-1}$ such that $\mathbf{x}_c \in P(\mathbf{y}_c+V, \epsilon'\gamma^{\ell})$ and
        \begin{align*}
            \mathcal{C}^{\ell}\setminus \mathcal{C}^{\ell-1} \subseteq \bigcup_{c\in C^{\ell-1}} \left( (\mathbf{y}_c+V)\cap P(\mathbf{x}_c,r_c) \right) \subseteq \bigcup_{\mathbf{z} \in \mathcal{C}^{\ell}} P(\mathbf{z},2\gamma^{2}\mathbf{r}_{\mathbf{z}}^{\ell}),
        \end{align*}
    
        \item \label{induction hypothesis iii} if $\mathbf{x} \in \mathcal{C}^{\ell}\setminus \mathcal{C}^{\ell-1}$, then $\mathbf{r}_{\mathbf{x}}^{\ell}=\gamma^{\ell}$; otherwise, we have $\mathbf{r}_{\mathbf{x}}^{\ell}=\mathbf{r}_{\mathbf{x}}^{\ell-1}$, 
    
        \item \label{induction hypothesis iv} $|N_{\mathbf{x}}(s^2)- m|< \epsilon'$ for $\mathbf{x}\in\mathcal{C}^\ell$ and $ s\in [10^{3}\gamma^{-1}\delta\mathbf{r}_{\mathbf{x}}^\ell,\gamma^{-5}]$,
        
        \item \label{induction hypothesis v} for any $\mathbf{x}\in \mathcal{C}^{j}$, and $\ell < j$, there exists $c_{\ell}\in C^{\ell}$ such that $P(\mathbf{x},\gamma^j)\subseteq P(\mathbf{x}_{c_{\ell}},\frac{3}{2}\gamma^{\ell})$.
    \end{enumerate}
    Note that \itemref{1}{induction hypothesis i}, \itemref{1}{induction hypothesis ii}, \itemref{1}{induction hypothesis iii}, and \itemref{1}{induction hypothesis v} are satisfied by construction. Further, \itemref{1}{induction hypothesis iv} holds by \itemref{1}{def of V-2}.
    
    Given the information at $j$th step, we proceed to the $(j+1)$th step as follows. For each $c\in C^j$, choose $\mathbf{y}_c \in  \mathcal{V}_{\delta,m,\gamma^{j}}(\mathbf{x}_c)$, which is non-empty by assumption. Because $u$ is $(k,C(\alpha)^{\Lambda}\delta,s_0)$-symmetric with respect to $V$ at $\mathbf{x}_a$, but not $(k+1,C^{-\Lambda}\eta,s_0)$-symmetric at $\mathbf{x}_a$, we can apply \eqref{eq: c ball decomposition frequency pinching} and Lemma \ref{newlineup} with $\mathcal{C} \leftarrow \{ \mathbf{y}_c,\mathbf{x}_a\}$, $\mathbf{r}_{\mathbf{x}_a} \leftarrow \gamma $, $\kappa \leftarrow 10^{3}\gamma^{-1}\delta$, $s_0 \leftarrow s_0$, $\mathbf{r}_{\mathbf{y}_c} \leftarrow \gamma^{j}$, $r_0\leftarrow s_0 $, $\delta \leftarrow C(\alpha)^{\Lambda}\delta$, $\eta \leftarrow C^{-\Lambda}\eta$ to conclude the following for all $c\in C^{j}$: 
    \begin{enumerate}[label=(j.\Alph*), ref=\Alph*]
        \item \label{def j+1 of V-2} for all $s\in [\gamma^j,\gamma^{-5}]$, $(\mathbf{y}_c+V)\cap P(\mathbf{y}_c,10s)\subseteq \{\mathbf{y}\in P(\mathbf{y}_c,10s)\colon |N_{\mathbf{y}}(10^{6}\delta^2s^2)-m|< C(\alpha)^{\Lambda}\sqrt{\delta}\};$
    
        \item \label{def j+1 of V-3} for all $s\in [\gamma^j,1]$ and $\zeta \in [C(\alpha)^{\Lambda}\sqrt{\delta},1]$, we have 
        \begin{align*}
            \{\mathbf{y} \in P(\mathbf{y}_c,10 s): |N_{\mathbf{y}}(s^2)-m|<\zeta \} \subseteq  P(\mathbf{y}_c+V,C(\alpha)^{\Lambda}\eta^{-\frac{1}{n}}\zeta^{\frac{1}{n}} s).
        \end{align*}
    \end{enumerate}
    
    Let $\{\mathbf{x}_i^{j+1} \}_{i=1}^{N_{j+1}}$ be a maximal subset of $\cup_{c\in C^j} ((\mathbf{y}_{c} + V)\cap P(\mathbf{x}_c,\gamma^j))$ such that 
    \begin{align*}
        |\mathbf{x}_{i_1}^{j+1} -\mathbf{x}_{i_2}^{j+1}|&\geq 2\gamma^{(j+1)+2} \text{ for } 1\leq i_1 \neq i_2 \leq N_{j+1},\\
        |\mathbf{x}_i^{j+1}-\mathbf{x}|&\geq \gamma^{2}(\gamma^{j+1}+\mathbf{r}_{\mathbf{x}}^j) \text{ for } 1\leq i \leq N_{j+1} \text{ and } \mathbf{x} \in \mathcal{C}^{j}\setminus \{\mathbf{x}_c\}_{c\in C^j}.
    \end{align*}
    Define $\mathcal{C}^{j+1} \coloneqq (\mathcal{C}^j \setminus \{\mathbf{x}_c\}_{c\in C^j}) \cup \{\mathbf{x}_i^{j+1} \}_{i=1}^{N_{j+1}}$, and set
    \begin{align*}
        \mathbf{r}_{\mathbf{x}}^{j+1}\coloneqq 
        \begin{cases}
             \mathbf{r}_{\mathbf{x}}^j & \mathbf{x}\in \mathcal{C}^j \setminus \{ \mathbf{x}_c \}_{c\in C^j}, \\
             \gamma^{j+1} & \text{otherwise.} 
        \end{cases}
    \end{align*}
    
    Categorize $\{P(\mathbf{x}_i^{j+1},\gamma^{j+1})\}_{i=1}^{N_{j+1}}$ into balls of type \ref{b-balls}, \ref{c-balls}, \ref{d-balls} and \ref{e-balls}. Let $B^{j+1},D^{j+1},E^{j+1}$ be the union of $B^j,D^j,E^j$ with the new balls of type \ref{b-balls}, \ref{d-balls}, \ref{e-balls}, respectively. Let $C^{j+1}$ be the new balls of type \ref{c-balls}. Define $r_b$, $r_c$, $r_d$ so that induction hypothesis \itemref{j$+$1}{induction hypothesis i} holds.
    
    \begin{claim}\label{claim induction hypothesis at j+1}
        The statements {\rm \itemref{j$+$1}{induction hypothesis i}}-{\rm \itemref{j$+$1}{induction hypothesis v}} hold.
    \end{claim}
    \begin{proof}[Proof of Claim \ref{claim induction hypothesis at j+1}]
        Statements \itemref{j+1}{induction hypothesis i}, \itemref{j+1}{induction hypothesis iii} and \itemref{j+1}{induction hypothesis v} hold by construction. For $c\in C^{j+1}$, the statement \itemref{j+1}{induction hypothesis iv} for $\mathbf{x}=\mathbf{x}_c$ follows from \itemref{j}{def j+1 of V-2} and $\mathbf{x}_c \in \cup_{c'\in C^j} ((\mathbf{y}_{c'} + V)\cap P(\mathbf{y}_{c'},5\gamma^j))$.
        Finally, we prove \itemref{j+1}{induction hypothesis ii}. By our construction, there exists $c'\in C^{j-1}$ such that 
        \begin{align*}
            \mathbf{x}_{c}\in (\mathbf{y}_{c'}+V)\cap P(\mathbf{x}_{c'},\gamma^{j-1})\subset (\mathbf{y}_{c'}+V)\cap P(\mathbf{y}_{c'},2\gamma^{j-1}).
        \end{align*} 
       By \itemref{j$-$1}{def j+1 of V-2} with $s\leftarrow \gamma^{j-1}$, we thus have $|N_{\mathbf{x}_c}(\gamma^{2j})-m|<C^{\Lambda}\sqrt{\delta}$. Using this, and $\mathbf{x}_c \in P(\mathbf{y}_c,10\gamma^j)$, \itemref{j}{def j+1 of V-3} with $s\leftarrow \gamma^j$ and $\zeta \leftarrow C(\alpha)^{\Lambda}\sqrt{\delta}$ gives $\mathbf{x}_c \in P(\mathbf{y}_c+V,C(\alpha)^{\Lambda}\eta^{-\frac{1}{n}}\delta^{\frac{1}{2n}}\gamma^j)$. The other claims of \itemref{j+1}{induction hypothesis ii} follow by construction.
    \end{proof}
    
    \begin{claim} 
    \label{claim: weak closeness}
        For any $\ell \leq j+1$ and $\mathbf{x} \in \mathcal{C}^{\ell}$, there exists $\mathbf{y} \in \mathcal{C}^{j+1}$ such that $|\mathbf{x}-\mathbf{y}|<2\gamma^2 (\gamma^{\ell}+\mathbf{r}_{\mathbf{y}}^{j+1})$.
    \end{claim}
    
    \begin{proof}[Proof of Claim \ref{claim: weak closeness}]
        We prove by induction on $\ell=j+1$ down to $1$. If $\ell=j+1$, we take $\mathbf{y}=\mathbf{x}$. Suppose the claim holds for $\ell+1$, where $\ell \leq j$, and consider $\mathbf{x}\in \mathcal{C}^\ell$. If $\mathbf{x}\in \mathcal{C}^{j+1}$, then we can take $\mathbf{y}=\mathbf{x}$. Otherwise, we have $\mathbf{x}=\mathbf{x}_c$ for some $c\in C^{\ell}$. By \itemref{$\ell+1$}{induction hypothesis ii}, we have $\mathbf{x}_c \in P(\mathbf{y}_c+V,\epsilon' \gamma^{\ell+1})$. Again using \itemref{$\ell+1$}{induction hypothesis ii}, we have $(\mathbf{y}_c +V)\cap P(\mathbf{x}_c,r_c) \subseteq P(\mathbf{z},2\gamma^2 \mathbf{r}_{\mathbf{z}}^{\ell})$ for some $\mathbf{z} \in \mathcal{C}^{\ell+1}$. Hence, by the triangle inequality, $|\mathbf{x}-\mathbf{z}|<\epsilon'\gamma^{\ell+1}+2\gamma^{2}\mathbf{r}_{\mathbf{z}}^{\ell+1}$. If $\mathbf{z}\in \mathcal{C}^{j+1}$, then we take $\mathbf{y}=\mathbf{z}$, and we are done. Otherwise, $\mathbf{z} \in \{\mathbf{x}_{c'}\}_{c'\in C^{\ell+1}}$, so that $\mathbf{r}_{\mathbf{z}}^{\ell+1}= \gamma^{\ell+1}$, and by the induction hypothesis, there is $\mathbf{y}\in\mathcal{C}^{j+1}$ such that $|\mathbf{z}-\mathbf{y}|<2\gamma^2 (\gamma^{\ell+1}+\mathbf{r}_{\mathbf{y}}^{j+1})$. By the triangle inequality,
        \begin{align*}
            |\mathbf{x}-\mathbf{y}|\le \verts{\mathbf{x}-\mathbf{z}}+\verts{\mathbf{z}-\mathbf{y}}\leq \epsilon'\gamma^{\ell+1}+2\gamma^{2}\gamma^{\ell+1}+2\gamma^2 (\gamma^{\ell+1}+\mathbf{r}_{\mathbf{y}}^{j+1})<2\gamma^{2}(\gamma^{\ell}+\mathbf{r}_{\mathbf{y}}^{j+1}),
        \end{align*}
        if $\epsilon'<\gamma^2$ and we have proved the claim.
    \end{proof}
    
    \begin{claim}\label{claim: neck region induction}
        If $\delta< C(\alpha)^{-\Lambda} \min\{\eta^{\frac{1}{n}}\epsilon^{8n^2},\eta'\}$, then
        \begin{align*}
            \mathcal{N}^{j+1} \coloneqq P(\mathbf{x}_a,20) \setminus P(\mathcal{C}^{j+1},\mathbf{r}^{j+1})
        \end{align*}
        is an $(m,k,\epsilon,\eta')$-neck region.
    \end{claim}
    
    \begin{proof}[Proof of Claim \ref{claim: neck region induction}]
        Condition \ref{neck-vitali-covering} holds because $\mathcal{N}^j$ is an $(m,k,\epsilon,\eta')$-neck region, and by our construction of $\mathcal{C}^{j+1}$. For any $\mathbf{x}\in \mathcal{C}^j \cap \mathcal{C}^{j+1}$,  \ref{neck-frequency-pinching} and \ref{neck-k-sym} hold by the fact that $\mathcal{N}^j$ is an $(m,k,\epsilon,\eta')$-neck region. If instead $\mathbf{x} \in \mathcal{C}^{j+1}\setminus \mathcal{C}^j$, then \ref{neck-frequency-pinching} holds by \itemref{j$+$1}{induction hypothesis iv}. We now prove \ref{neck-k-sym}. Let $\mathbf{x}\in \cC^j\cap \cC^{j+1}$. Using $\mathbf{x} \in \cup_{c\in C^j}((\mathbf{y}_c+V)\cap P(\mathbf{y}_c,2\gamma^j))$ and \itemref{j}{def j+1 of V-2} we know that 
        \begin{align*}
            \verts{N_{\mathbf{x}}(\gamma^{18}s^2)-m}< C(\alpha)^\Lambda \sqrt{\delta}
        \end{align*}
        for any $s\in [\gamma^{j},s_0]$. We now apply Lemma \ref{newlineup}\ref{existence of plane of symmetry} with $\cC\leftarrow\{\mathbf{x},\mathbf{x}_a\}$, $\delta\leftarrow C(\alpha)^\Lambda\sqrt{\delta}$, $\mathbf{r}\leftarrow \mathbf{r}^{j+1}$, $r_0\leftarrow \gamma^{-5}$, $s_0\leftarrow s_0$, $\kappa\leftarrow \gamma^7$, $\mathbf{x}\leftarrow \mathbf{x}$, and $\mathbf{x}_0\leftarrow \mathbf{x}_a$ so that for any $s\in [\gamma^{j+8}, \gamma^{-5}]$, $u$ is $(k,C^\Lambda(\alpha)\sqrt{\delta},s)$-symmetric at $\mathbf{x}$. We claim that $u$ is not $(k+1,\eta',s)$-symmetric at $\mathbf{x}$. Suppose not. Then we apply Lemma \ref{newlineup}\ref{existence of plane of symmetry} with $\cC\leftarrow\{\mathbf{x},\mathbf{x}_a\}$, $\delta\leftarrow \eta'$, $\mathbf{r}\leftarrow \mathbf{r}^{j+1}$, $r_0\leftarrow \gamma^{-5}$, $s_0\leftarrow s$, $\kappa\leftarrow \gamma^7$, $\mathbf{x}_0\leftarrow \mathbf{x}$, $\mathbf{x}\leftarrow \mathbf{x}_a$, and $s\leftarrow s_0$ so that $u$ is $(k+1,C^\Lambda \eta', s_0)$-symmetric which is a contradiction since $C^\Lambda \eta'<\eta$.
        
        We now prove \eqref{neck-n2-CinV}. Fix $\mathbf{x} \in \mathcal{C}^{j+1}$ and $s\in [\mathbf{r}_{\mathbf{x}}^{j+1},\gamma^{-3}]$ such that $P(\mathbf{x},s)\subseteq P(\mathbf{x}_a,20)$. Suppose $\mathbf{z} \in \mathcal{C}^{j+1} \cap P(\mathbf{x},s)$. By \ref{neck-vitali-covering}, we then have $\mathbf{r}_{\mathbf{z}}^{j+1} \leq \gamma^{-2}|\mathbf{x}-\mathbf{z}| \leq \gamma^{-2}s$, hence $|N_{\mathbf{z}}(s^2)-m|<\epsilon'$ and $|N_{\mathbf{x}}(s^2)-m|<\epsilon'$ by \itemref{j$+$1}{induction hypothesis iv}.  Choose $\ell \leq j$ such that $s \in [\gamma^{\ell+1},\gamma^{\ell}]$. First assume that $\ell \geq 0$. By \itemref{j$+1$}{induction hypothesis v}, there exists $c\in C^{\ell}$ such that $\mathbf{x} \in P(\mathbf{y}_c,2\gamma^{\ell})$. This, along with $|\mathbf{x}-\mathbf{z}| < s \leq \gamma^{\ell}$ and \eqref{eq: c ball decomposition frequency pinching} implies
        \begin{align*}
            \mathbf{x},\mathbf{z} \in \{ \mathbf{y} \in P(\mathbf{y}_c,10\gamma^{\ell}) \colon |N_{\mathbf{y}}(\gamma^{2\ell})-m|<\epsilon'\}.
        \end{align*}
        By \itemref{$\ell$}{def j+1 of V-3} with $s \leftarrow \gamma^{\ell}$, and $\zeta \leftarrow C^{\Lambda}\sqrt{\delta}$, we thus have $\mathbf{x},\mathbf{z} \in P(\mathbf{y}_c+V,C(\alpha)^{\Lambda}\eta^{-\frac{2}{n}}\delta^{\frac{1}{4n^2}}s)$, hence $\mathbf{x}-\mathbf{z} \in P(V,C(\alpha)^{\Lambda}\eta^{-\frac{2}{n}}\delta^{\frac{1}{4n^2}}s)$. If $-3\leq \ell \leq -1$, then we can take $\mathbf{y}_c=\mathbf{x}_a$ and the same argument works.
        
        We now prove \eqref{neck-n2-VinC}. Again fix $\mathbf{x} \in \mathcal{C}^{j+1}$ and $s\in [\mathbf{r}_{\mathbf{x}}^{j+1},\gamma^{-3}]$ such that $P(\mathbf{x},s)\subseteq P(\mathbf{x}_a,10)$, and now suppose that $\mathbf{z} \in (\mathbf{x}+V)\cap P(\mathbf{x},s)$. Choose $\ell \leq j$ such that $s\in [\gamma^{\ell+1},\gamma^{\ell}]$. Recall that $\mathbf{x}\in (\mathbf{y}_c+V)\cap P(\mathbf{x}_c,\gamma^j)$ for some $c\in C^j$. Therefore, $\mathbf{z}\in (\mathbf{y}_c+V)\cap P(\mathbf{y}_c, 10\gamma^{\ell})$. Then applying \itemref{j}{def j+1 of V-2} with $s\leftarrow \gamma^\ell$ yields $\verts{N_{\mathbf{z}}(\gamma^{6+2\ell})-m}\leq C(\alpha)^\Lambda \sqrt{\delta}$. Now we use \itemref{1}{def of V-3} with $\zeta\leftarrow C(\alpha)^{\Lambda}\sqrt{\delta}$, $s\leftarrow 1$ to get $\mathbf{z} \in P(V, C(\alpha)^{\Lambda} \eta^{-\frac{1}{n}}\delta^{\frac{1}{2n}})$. From \itemref{1}{induction hypothesis ii} and $\verts{\mathbf{z}}\leq 2$, there exists $\mathbf{x}_1 \in \mathcal{C}^1$ such that $\mathbf{z} \in P(\mathbf{x}_1,3\gamma^3)$. If $P(\mathbf{x}_1,\gamma)$ is not a \ref{c-balls}-ball, then $\mathbf{x}_1 \in \mathcal{C}^{j+1}$ and $\mathbf{r}_{\mathbf{x}_1}^{j+1}=\gamma$, hence $\mathbf{z} \in P(\mathbf{x}_1,\gamma(\mathbf{r}_{\mathbf{x}_1}^{j+1}+s))$ and the claim follows in this case. Suppose instead that $P(\mathbf{x}_1,\gamma)$ is a \ref{c-balls}-ball, so that $\mathbf{x}_1 = \mathbf{x}_{c_1}$ for some $c_1 \in C^1$. Then \itemref{1}{def of V-3} gives $\mathbf{z} \in P(\mathbf{y}_{c_1}+V,C(\alpha)^{\Lambda}\eta^{-\frac{1}{n}}\delta^{\frac{1}{2n}}\gamma)\cap P(\mathbf{x}_{c_1},\gamma)$, hence by \itemref{2}{induction hypothesis ii}, there exists $\mathbf{x}_2 \in \mathcal{C}^2$ such that $\mathbf{z} \in P(\mathbf{x}_2,3\gamma^2 \mathbf{r}_{\mathbf{x}_2}^2)$. If $P(\mathbf{x}_2,\gamma^2)$ is not a \ref{c-balls}-ball, then $\mathbf{x}_2\in \cC^{j+1}$ and $\mathbf{r}_{\mathbf{x}_2}^{j+1}=\mathbf{r}_{\mathbf{x}_2}^{2}=\gamma^2$, hence $\mathbf{z} \in P(\mathbf{x}_2,\gamma(s+\mathbf{r}_{\mathbf{x}_2}^{j+1}))$, and claim follows in this case. We can repeat this argument to obtain $\mathbf{x}_{\ell} \in \mathcal{C}^{\ell}$ such that $\mathbf{z} \in P(\mathbf{x}_{\ell},3\gamma^2\mathbf{r}_{\mathbf{x}_{\ell}}^{\ell})$. If $\mathbf{x}_{\ell} \in \mathcal{C}^{j+1}$, then we are done by reasoning as above. Otherwise, we have $\mathbf{x}_{\ell} \in \{\mathbf{x}_{c'}\}_{c'\in C^{\ell}}$, hence  $\mathbf{r}_{\mathbf{x}_{\ell}}^{\ell}=\gamma^{\ell}$. By Claim \ref{claim: weak closeness}, there exists $\mathbf{y} \in \mathcal{C}^{j+1}$ such that $|\mathbf{x}_{c_{\ell}}-\mathbf{y}|<2\gamma^{2}(\gamma^{\ell}+\mathbf{r}_{\mathbf{y}}^{j+1})$. It follows that 
        \begin{align}
            |\mathbf{z}-\mathbf{y}| \leq |\mathbf{z}-\mathbf{x}_{c_{\ell}}|+|\mathbf{x}_{c_{\ell}}-\mathbf{y}|<3\gamma^2\mathbf{r}_{\mathbf{x}_{\ell}}^{\ell}+2\gamma^{2}(\gamma^{\ell}+\mathbf{r}_{\mathbf{y}}^{j+1})  \leq 3\gamma(s+\mathbf{r}_{\mathbf{y}}^{j+1}).\label{eq better neck inclusion}
        \end{align}
        Thus \eqref{neck-n2-VinC} holds. 
    \end{proof}
    
    For any $\mathbf{x}\in \mathcal{C}^{\ell}$, we have $|\mathbf{x}-\mathbf{x}_{c}|<\gamma^{\ell-1}$ for some $c\in C^{\ell-1}$, so that
    \begin{align*}
        \mathcal{C}^\ell \subseteq P(\mathcal{C}^{\ell-1},\gamma^{\ell-1}).
    \end{align*}
    By Lemma \ref{Hausdorfflemma}, there is a closed subset $\mathcal{C}\subseteq \overline{P}(\mathbf{x}_a,20)$ such that $\lim_{\ell \to \infty} d_{\mathcal{P},H}(\mathcal{C},\mathcal{C}^\ell) =0$. For $\mathbf{x}\in \mathcal{C}\setminus \cup_\ell \mathcal{C}^\ell$, set $\mathbf{r}_{\mathbf{x}}\coloneqq0$, and set $\mathcal{N}\coloneqq P(\mathbf{x}_a,20) \setminus \overline{P}(\mathcal{C},\mathbf{r}_{\bullet})$.  Setting $B\coloneqq\cup_\ell B^\ell$, $D\coloneqq\cup_\ell D^\ell$, and $E\coloneqq\cup_\ell E^\ell$, we thus have
    \begin{align*}
        P(\mathbf{0},1) \subseteq P(\mathcal{C}\setminus \mathcal{C}^\ell,\gamma^{\ell}) \cup \mathcal{N}  \cup \bigcup_{b\in B}P(\mathbf{x}_b,r_b) \cup \bigcup_{d\in D} P(\mathbf{x}_d,r_d) \cup \bigcup_{e \in E} P(\mathbf{x}_e,r_e).
    \end{align*}
    Taking the intersection over all $\ell\in \mathbb{N}$ gives
    \begin{align*}
        P(\mathbf{0},1) \subseteq  (\mathcal{N} \cup \mathcal{C}_0)  \cup \bigcup_{b\in B}P(\mathbf{x}_b,r_b) \cup \bigcup_{d\in D} P(\mathbf{x}_d,r_d) \cup \bigcup_{e \in E} P(\mathbf{x}_e,r_e).
    \end{align*}
    We claim that $\mathcal{N}$ is an $(m,k,2\epsilon,C(\alpha)^{-\Lambda}\eta)$-neck region of scale $10$. For any $\mathbf{x}_1,\mathbf{x}_2\in \mathcal{C}_+$, there exists $j\in \mathbb{N}$ such that $\mathbf{x}_1,\mathbf{x}_2 \in \mathcal{C}^j$ and $\mathbf{r}_{\mathbf{x}_1}=\mathbf{r}_{\mathbf{x}_1}^{j}$, $\mathbf{r}_{\mathbf{x}_2}=\mathbf{r}_{\mathbf{x}_2}^{j}$, so $P(\mathbf{x}_1,\gamma^2\mathbf{r}_{\mathbf{x}_1})\cap P(\mathbf{x}_2,\gamma^2\mathbf{r}_{\mathbf{x}_2})=\emptyset$ by the fact that $\mathcal{C}^{j}$ is an $(m,k,\epsilon)$-neck region. Thus \ref{neck-vitali-covering} holds. To prove \ref{neck-frequency-pinching}, \ref{neck-k-sym}, it suffices to consider the case where $\mathbf{x}\in \mathcal{C}_0$. In this case, there is a sequence $\mathbf{x}_{c_j}$ with $c_j \in C^j$, $\mathbf{x}_{c_j} \to \mathbf{x}$. For any $s\in (0,\gamma^{-3}]$, $\mathbf{x}_{c_j}\in \mathcal{C}^j$ implies that for sufficiently large $j\in \mathbb{N}$, we have $|N_{\mathbf{x}_{c_{j}}}(s^2)-m|<\epsilon$. By the continuity of $\mathbf{y} \mapsto N_{\mathbf{y}}(s^2)$, we can therefore take $j\to \infty$ to obtain \ref{neck-frequency-pinching}. Similarly, $\mathbf{x}_{c_j} \in \mathcal{C}^j$ and the continuity of the integrals in Definition \ref{definition symmetry} as a function of basepoint yield \ref{neck-k-sym}. 
    
    Suppose $\mathbf{x} \in \mathcal{C}_+$ and $s\in [\mathbf{r}_{\mathbf{x}},\gamma^{-3}]$ satisfies $P(\mathbf{x},s)\subseteq P(\mathbf{x}_a,20)$. Then $\mathcal{C}^j \cap P(\mathbf{x},s)\subseteq P(\mathbf{x}+V,\epsilon s)$ for all $j\in \mathbb{N}$, so the Hausdorff convergence $\mathcal{C}^j \to \mathcal{C}$ implies that $\mathcal{C} \cap P(\mathbf{x},s)\subseteq P(\mathbf{x}+V,\frac{3}{2}\epsilon s)$, hence \eqref{neck-n2-CinV} holds with $\epsilon$ replaced by $\frac{3}{2}\epsilon$. The Hausdorff convergence $\mathcal{C}^j \to \mathcal{C}$ gives 
    $\mathcal{C}^j \subseteq P(\mathcal{C},\frac{1}{8}\gamma s)$ for sufficiently large $j\in \mathbb{N}$, so that if we choose $j$ large enough such that $\gamma^j \leq \frac{1}{8}\gamma s$, then
    \begin{align*}
        (\mathbf{x}+V)\cap P(\mathbf{x},s) \subseteq \bigcup_{\mathbf{y} \in \mathcal{C}^j} P(\mathbf{y},3\gamma (s+\mathbf{r}_{\mathbf{y}}^j)) \subseteq \bigcup_{\mathbf{y} \in \mathcal{C}} P(\mathbf{y},10\gamma(s+\mathbf{r}_{\mathbf{y}})),
    \end{align*}
    where we used \eqref{eq better neck inclusion} for the first inclusion, and for the last inclusion we used that $\mathbf{r}_{\mathbf{y}}^j = \mathbf{r}_{\mathbf{y}}$ for $\mathbf{y} \in \mathcal{C}^j \cap \mathcal{C}$, and $\mathbf{r}_{\mathbf{y}} \leq \gamma^j \leq \frac{\gamma}{4}s$ for any $\mathbf{y} \in \mathcal{C}^j \setminus \mathcal{C}$. Thus \eqref{neck-n2-VinC} holds for all $\mathbf{x}\in \mathcal{C}_+$.
    
    Suppose instead $\mathbf{x} \in \mathcal{C}_0$, and fix $s\in (0,\gamma^{-3}]$. As before, choose $c_i \in C^i$ such that $\mathbf{x}_{c_i} \to \mathbf{x}$ as $i\to \infty$. Then the arguments of the previous paragraph and the Hausdorff convergence $\mathcal{C}^j \to \mathcal{C}$ give
    \begin{align*}
        \mathcal{C} \cap P(\mathbf{x},s) \subseteq P(\mathcal{C}^j \cap P(\mathbf{x}_{c_i},s), 10^{-2}\epsilon s) \subseteq P(\mathbf{x}_c+V,\tfrac{5}{3}\epsilon s) \subseteq P(\mathbf{x}+V,2\epsilon s),
    \end{align*}
    hence \eqref{neck-n2-CinV} holds, and a similar argument yields \eqref{neck-n2-VinC}.
    
    The $k$-dimensional content estimate then follows from Theorem \ref{theorem-neck-structure}.

    \ref{c-ball decomposition-discrete} Choose $\ell\in \N_0$ such that $\gamma^{\ell+1} \leq r_\ast \leq \gamma^{\ell}$. Set 
    \begin{align*}
        \cC\coloneqq  \cC^{\ell}=\{ \mathbf{x}_b\}_{b\in B^{\ell}} \cup \{ \mathbf{x}_c\}_{c\in C^{\ell}} \cup \{ \mathbf{x}_d\}_{d\in D^{\ell}} \cup \{ \mathbf{x}_e\}_{e\in E^{\ell}}
    \end{align*}
    as in the proof of \ref{c-ball decomposition-continuous}. Set $B\coloneqq B^\ell$, $D\coloneqq D^\ell$, $E\coloneqq E^\ell$, and $F\coloneqq C^\ell$. Note that $r_c=\gamma^\ell$ when $c\in C^\ell$.
\end{proof}

\begin{remark}\label{remark improved frequency pinching for centers}
    By applying \itemref{$\infty$}{induction hypothesis iv}, we can assume that $\verts{N_{\mathbf{x}}(s^2)-m}\leq C(\alpha)^\Lambda \eta^{-\frac{1}{n}}\delta^{\frac{1}{2n}}$ for all $\mathbf{x}\in \cC$ and $s\in [10^3\gamma^{-1}\delta \mathbf{r}_\mathbf{x},\gamma^{-10}r]$. We also observe that $\mathbf{r}_{\bullet} \leq \gamma$ from our construction.
\end{remark}

In the proof of Proposition \ref{cballcovering}, we used the result concerning Hausdorff convergence of almost-monotone sequences of sets.
\begin{lemma} \label{Hausdorfflemma}
    Suppose $(X,d)$ is a locally compact metric space, and $(\mathcal{C}_{j})$ is a collection of closed subsets of $X$ satisfying
    \begin{align*}
        \mathcal{C}_{j+1}\subseteq B(\mathcal{C}_{j},2^{-j}).
    \end{align*}
    Define 
    \begin{align*}
        \mathcal{C}\coloneqq\bigcap_{k}\bigcup_{j\geq k}B(\mathcal{C}_{j},2^{-j}),
    \end{align*}
    so that, for any $x\in\mathcal{C}$, there exists a sequence $x_{i}\in\mathcal{C}_{j(i)}$ with $\lim_{i\to\infty}j(i)=\infty$ and $d(x,x_{i})<2^{-j(i)}$. Then $\mathcal{C}_{j}\to\mathcal{C}$ in the Hausdorff sense. 
\end{lemma}

\begin{proof}
    We need to show that $\mathcal{C}_{k+1}\subseteq B(\mathcal{C},\epsilon)$ and $\mathcal{C}\subseteq B(\mathcal{C}_{k},\epsilon)$ for sufficiently large $k\in\mathbb{N}$. For all $j\geq k$, we have $\mathcal{C}_{j}\subseteq B(\mathcal{C}_{j-1},2^{-(j-1)})\subseteq B(\mathcal{C}_{j-2},2^{-(j-2)}+2^{-(j-1)})\subseteq\cdots\subseteq B(\mathcal{C}_{k},2^{-k})$,
    hence $\bigcup_{j\geq k}B(\mathcal{C}_{j},2^{-j})\subseteq B(\mathcal{C}_{k},2^{-(k-2)})$
    for all $k\in\mathbb{N}$. In particular, for all $k\in\mathbb{N}$ satisfying $2^{-(k-2)}<\epsilon$, we have
    \begin{align*}
        \mathcal{C}\subseteq\bigcup_{j\geq k}B(\mathcal{C}_{j},2^{-j})\subseteq B(\mathcal{C}_{k},\epsilon).
    \end{align*}
    Because $\left(\bigcup_{j\geq k}\overline{B}(\mathcal{C}_{j},2^{-j})\right)_{k\in\mathbb{N}}$
    is a decreasing sequence of compact sets whose intersection is $\mathcal{C}$, it follows that
    \begin{align*}
        \lim_{k\to\infty}d_{H}\left(\mathcal{C},\bigcup_{j\geq k}\overline{B}(\mathcal{C}_{j},2^{-j})\right)=0.
    \end{align*}
    That is, there exists $k_{0}\in\mathbb{N}$ such that $\bigcup_{j\geq k}B(\mathcal{C}_{j},2^{-j})\subseteq B(\mathcal{C},\epsilon)$
    when $k\geq k_{0}$. In particular, $B(\mathcal{C}_{j},2^{-j})\subseteq B(\mathcal{C},\epsilon)$
    when $j\geq k_{0}$. 
\end{proof}

\subsection{Refined covering of \ref{c-balls}-balls}

In this section, we show how to refine the neck regions constructed in Proposition \ref{cballcovering} to arrive at strong neck regions, which is an essential step in applying Theorem \ref{thm:strongneckstructure}. The strategy for this covering comes from \cite[Section 7.4]{jiang-naber-2021-l2-curvature}.

Our refinement of neck regions requires proving the following $\epsilon$-regularity for the frequency based on the existence of symmetry.
\begin{lemma}\label{lemma epsilon regularity for frequency}
    Suppose $u$ is a caloric function with $N(10^5r^2)\leq \Lambda$. If $u$ is $(k,\epsilon,2r)$ symmetric at $\mathbf{x}_0$, then
    \begin{align*}
        N_{\mathbf{x}_0}(r^2)\leq \begin{cases}
            2\epsilon & \text{ if } k=n+2,\\
            1+C^{\Lambda}\epsilon & \text{ if } k=n+1.
        \end{cases}
    \end{align*}
\end{lemma}

\begin{proof}
    By translating and parabolic rescaling, we assume that $\mathbf{x}_0=\mathbf{x}_a$ and $r=1$. If $u$ is $(n+2,\epsilon,2)$-symmetric at $\mathbf{0}$ then
    \begin{align*}
        \int_{\R^n} \verts{\cd u}^2\,d\nu_{-4}\leq \epsilon \int_{\R^n} u^2\, d\nu_{-4},
    \end{align*}
    which implies that $N(1)\leq N(4)\leq 2\epsilon$.

    Now suppose  
    $u$ is $(n+1,\epsilon, 2)$-symmetric with respect to $V=L\times \mathbb{R}$ for some $(n-1)$-plane $L$ which we can suppose to be the span of the first $n-1$ coordinates of $\R^n$. We use the normalization $\int_{\R^n} u^2\, d\nu_{-1}=1$. Then 
    \begin{align}
        \begin{split}
            \int_{\R^n} \verts{\cd^2 u}^2\, d\nu_{-1}
            &\leq\  \int_{\R^n}\verts{\pd_{n}\pd_n u}^2\, d\nu_{-1}  +2\sum_{i\neq n}\int_{\R^n} \verts{\cd \pd_i u}^2 \, d\nu_{-1}\\
            &\leq\ 2\int_{\R^n} \verts{\Delta u}^2 \, d\nu_{-1} +4\sum_{i\neq n}\int_{\R^n} \verts{\cd \pd_i u}^2\, d\nu_{-1}\\
            &\leq\ C\int_{\R^n} \verts{\pdt u}^2\, d\nu_{-4} +C\sum_{i\neq n}\int_{\R^n} \verts{\pd_i u}^2\, d\nu_{-4}< C^{\Lambda}\epsilon.\label{eq n+1 symmetry implies hessian is small}
        \end{split}
    \end{align}

    On the other hand, integrating the $f$-Bochner formula \eqref{eq: f-Bochner formula} and integrating by parts, we get
    \begin{align*}
        C^{\Lambda}\epsilon&>2\int_{\R^n} \verts{\cd^2 u}^2 \,d\nu_{-1}=-2\int_{\R^n} \angles{\cd u, \cd \Delta_f u}\, d\nu_{-1}-\int_{\R^n} \verts{\cd u}^2 \, d\nu_{-1}\\
        &=2\int_{\R^n} (\Delta_f u)^2+u\Delta_f u\, d\nu_{-1}+\int_{\R^n} \verts{\cd u}^2 \, d\nu_{-1}\\
        &=2\int_{\R^n} (\Delta_f u+ \frac{1}{2} u)^2\, d\nu_{-1}-\frac{1}{2}\int_{\R^n}u^2\, d\nu_{-1}+\frac{N(1)}{2} \\
        &\geq \frac{1}{2}(N(1)-1).
    \end{align*}
\end{proof}

We now upgrade the decomposition in Proposition \ref{cballcovering} to include strong neck regions.

\begin{proposition}[Strong covering of \ref{c-balls}-balls] \label{prop:strongcballcovering}
    For any $\alpha, \epsilon \in (0,1]$, $\Lambda>0$, the following holds if $\delta \leq C(\alpha,\Lambda)^{-1} (\eta \epsilon)^{C}$ and $\eta \leq C(\alpha)^{-\Lambda}$. Let $u$ be a caloric function satisfying 
    \begin{align}
        \sup_{\mathbf{x}\in P(\mathbf{0},20r)} N_{\mathbf{x}}^u(\delta^{-2}r^2) \leq m+\delta,\label{eq: strong c ball decomposition frequency pinching}
    \end{align}
    for some integer $m_\ast\leq m\leq \Lambda$, where $m_\ast$ is defined in \eqref{eq def of mstar}. Assume that $\mathcal{V}\coloneqq \mathcal{V}^u_{\delta,m,r}(\mathbf{0})$ is a $(k,\alpha r)$-independent set and that $\mathcal{E}_{100r}^{k+1,1}(\mathbf{y})> \eta$ for all $\mathbf{y}\in P(\mathbf{0},2r)$. Then there is a decomposition
    \begin{align*}
        P(\mathbf{0},r)\subseteq(\mathcal{C}_{0}\cup\mathcal{N}) \cup\bigcup_{c\in C}P(\mathbf{x}_{c},r_{c})\cup\bigcup_{d\in D}P(\mathbf{x}_{d},r_{d})\cup\bigcup_{e\in E}P(\mathbf{x}_{e},r_{e}),
    \end{align*}
    where $\mathcal{N} = P(\mathbf{x}_a,5r)\setminus \ol{P}(\cC,\mathbf{r}_{\bullet})$ is a strong $(m,k,\epsilon,C(\alpha,\Lambda)^{-1}\eta)$-neck of scale $\frac{5}{2}r$, and $\mathcal{N}=\widehat{\cN}\cap P(\mathbf{x}_a,5r)$, where $\widehat{\cN}$ is a strong $(m,k,\epsilon,C(\alpha,\Lambda)^{-1}\eta)$-neck region of scale $10$. Furthermore, 
    \begin{align*}
        \mathcal{H}_{\mathcal{P}}^{k}(\mathcal{C}_{0})+\sum_{d\in D}r_{d}^{k}+\sum_{e\in E}r_{e}^{k}\leq Cr^{k},&&\sum_{c\in C}r_{c}^{k}\leq C\epsilon r^{k}.
    \end{align*}
\end{proposition}

\begin{proof}
    After parabolic rescaling, we may assume that $r=1$. Take $\delta'>0$ to be determined. Set $\delta=C(\alpha)^\Lambda \eta^{\frac{1}{n}} (\delta')^{10n^2}$. Thus Proposition \ref{cballcovering}\ref{c-ball decomposition-continuous} yields a cover
    \begin{align*}
        P(\mathbf{0},1)\subseteq(\widetilde{\mathcal{C}}_{0}\cup\widetilde{\mathcal{N}})\cup\bigcup_{b\in\widetilde{B}}P(\mathbf{x}_{b},r_{b})\cup\bigcup_{d\in\widetilde{D}}P(\mathbf{x}_{d},r_{d})\cup\bigcup_{e\in\widetilde{E}}P(\mathbf{x}_{e},r_{e}),
    \end{align*}
    where $\widetilde{\mathcal{N}}=P(\mathbf{x}_{\mathbf{a}},20)\setminus\bigcup_{\mathbf{x}\in\widetilde{\mathcal{C}}}\ol{P}(\mathbf{x},\widetilde{\mathbf{r}}_{\mathbf{x}})$ is a $(m,k,\delta',C(\alpha)^{-\Lambda}\eta)$-neck region of scale $10$, and 
    \begin{align}
        \mathcal{H}_{\mathcal{P}}^{k}(\widetilde{\mathcal{C}}_{0})+\sum_{b\in\widetilde{B}}r_{b}^{k}+\sum_{d\in\widetilde{D}}r_{d}^{k}+\sum_{e\in\widetilde{E}}r_{e}^{k}\leq C.\label{eq packing measure estimate for tilde neck}
    \end{align}
    For $\mathbf{x} \in P(\mathbf{x}_a,20)$, define 
    \begin{align*}
        \mathbf{r}_{\mathbf{x}}\coloneqq \begin{cases}
        \epsilon \gamma^3\widetilde{\mathbf{r}}_{\mathbf{w}} & \text{ if }\mathbf{x}\in P(\mathbf{w},\gamma^{3}\widetilde{\mathbf{r}}_{\mathbf{w}})\text{ for some }\mathbf{w}\in\widetilde{\mathcal{C}},\\
        \epsilon d(\mathbf{x},\widetilde{\mathcal{C}}) & \text{ if }\mathbf{x}\notin\bigcup_{\mathbf{w}\in\widetilde{\mathcal{C}}}P(\mathbf{w},\gamma^{3}\widetilde{\mathbf{r}}_{\mathbf{w}}).
        \end{cases} 
    \end{align*}
    
    If $\mathbf{x} \in P(\mathbf{w},\gamma^3 \widetilde{r}_{\mathbf{w}})$ for some $\mathbf{w} \in \widetilde{\mathcal{C}}$, then we use \ref{neck-vitali-covering} to deduce that for any $\mathbf{w}'\in \widetilde{\mathcal{C}}\setminus \{ \mathbf{w}\}$
    \begin{align*}
        |\mathbf{x}-\mathbf{w}'| \geq \gamma^2 (1-\gamma) \widetilde{\mathbf{r}}_{\mathbf{w}} > \gamma^3 \widetilde{\mathbf{r}}_{\mathbf{w}} >|\mathbf{x}-\mathbf{w}|, \qquad |\mathbf{x}-\mathbf{w}'|\geq \gamma^3 \widetilde{\mathbf{r}}_{\mathbf{w}'}.
    \end{align*}
    As a consequence, $\mathbf{r}_{\bullet}$ is well-defined and is $\epsilon$-Lipschitz. We also have
    \begin{align}
        \mathbf{r}_\bullet \geq \epsilon d_{\cP}(\cdot ,\widetilde{\cC}).\label{eq lower bound on r function}
    \end{align}
    
    For each $\mathbf{w}\in P(\mathbf{x}_a,20)$, we define
    \begin{align}
        \widetilde{\mathbf{w}}\in\arg\min_{\widetilde{\mathcal{C}}}d_{\mathcal{P}}(\mathbf{w},\cdot).\label{eq def of widetilde w}
    \end{align}
    For $\widetilde{\delta}>0$ to be determined, define
    \begin{align*}
        \check{\cC}\coloneqq \left\{\mathbf{w}\in P(\mathbf{x}_a,20)\text{ }:\text{ }\mathbf{w}\in P(\widetilde{\mathbf{w}}+V,\widetilde{\delta}^{2}\mathbf{r}_{\mathbf{w}})\cap P(\widetilde{\mathbf{w}},2\max\{d_{\mathcal{P}}(\mathbf{w},\widetilde{\mathcal{C}}),\widetilde{\mathbf{r}}_{\widetilde{\mathbf{w}}}\}) \right\},
    \end{align*}
    so that $\widetilde{\mathcal{C}}\subseteq\check{\cC}$.
    
    \begin{claim} \label{claim:strongcballpinching}
        For all $\mathbf{w}\in \check{\cC} \cap P(\mathbf{x}_a,20)$,
        we have 
        \begin{align*}
            \sup_{s\in [\epsilon^{-2}\gamma^{-6}\delta \mathbf{r}_{\mathbf{w}},\gamma^{-10}]}|N_{\mathbf{w}}(s^2)-m|<C(\alpha,\Lambda)(\widetilde{\delta}+\eta^{-\frac{1}{2n}}(\delta')^{\frac{1}{4n}}).
        \end{align*}
    \end{claim}
    
    \begin{proof} 
        By $\epsilon$-Lipschitz property of $\mathbf{r}_{\bullet}$ and \eqref{eq lower bound on r function}, we get for any $\mathbf{w}\in P(\mathbf{x}_{a},20)$,
        \begin{align}
            \widetilde{\mathbf{r}}_{\widetilde{\mathbf{w}}}=\epsilon^{-1}\gamma^{-3}\mathbf{r}_{\widetilde{\mathbf{w}}}\leq \epsilon^{-1}\gamma^{-3}(\mathbf{r}_{\mathbf{w}}+\epsilon|\mathbf{w}-\widetilde{\mathbf{w}}|)\leq2\epsilon^{-1}\gamma^{-3}\mathbf{r}_{\mathbf{w}}.\label{eq widetilde r and tilde r}
        \end{align}
        Thus, Remark \ref{remark improved frequency pinching for centers} implies that for any $\widetilde{\mathbf{w}}\in \widetilde{\cC}\cap P(\mathbf{x}_a,20)$
        \begin{align*}
            \sup_{s\in [2\gamma^{-4}\epsilon^{-1}\delta \mathbf{r}_{\mathbf{w}},\gamma^{-10}]}|N_{\widetilde{\mathbf{w}}}(s^2)-m| \leq C(\alpha)^{\Lambda} \eta^{-\frac{1}{n}}(\delta')^{\frac{1}{2n}}.
        \end{align*}
        We apply Lemma \ref{newlineup}\ref{plane of symmetry inside pinched points} with $\mathcal{C} \leftarrow \{ \widetilde{\mathbf{w}}\}$, $r_0 \leftarrow \gamma^{-10}$, $\mathbf{r}_{\widetilde{\mathbf{w}}} \leftarrow \mathbf{r}_{\mathbf{w}}$, $\kappa \leftarrow \epsilon^{-1}\delta\gamma^{-5}$, $s_0 \leftarrow 1$, $s\leftarrow \epsilon^{-1}\mathbf{r}_{\mathbf{w}}$, and $\mathbf{x}_0 \leftarrow \widetilde{\mathbf{w}}$ to obtain
        \begin{align*}
            (\widetilde{\mathbf{w}}+V)\cap P(\widetilde{\mathbf{w}},10 \epsilon^{-1}\mathbf{r}_{\mathbf{w}}) \subseteq \{ \mathbf{y} \in P(\widetilde{\mathbf{w}},10\epsilon^{-1}\mathbf{r}_{\mathbf{w}}) \colon |N_{\mathbf{y}}(\epsilon^{-4}\gamma^{-10}\delta^2 \mathbf{r}_{\mathbf{w}}^2)-m|<C(\alpha)^{\Lambda} \eta^{-\frac{1}{2n}}(\delta')^{\frac{1}{4n}}\}. 
            \label{eq nice pinching}
        \end{align*}
        By Lemma \ref{frequencydirectionalestimate} with $L \leftarrow \mathbb{R}^n$, we can argue as in Theorem \ref{thm: sym split equiv}\ref{thm: sym split equiv-2} to obtain
        \begin{align*}
            |N_{\mathbf{w}}(\epsilon^{-4}\gamma^{-12}\delta^2\mathbf{r}_{\mathbf{w}})-m|<C(\alpha,\Lambda)(\widetilde{\delta}+\eta^{-\frac{1}{2n}}(\delta')^{\frac{1}{4n}}).
        \end{align*}
        Then using \eqref{eq: strong c ball decomposition frequency pinching} and the monotonicity of $N$, we obtain the claim.
    \end{proof}
    
    \begin{claim} \label{claim:strongnecksymmetry}
        For all $\mathbf{x}\in \check{\cC} \cap P(\mathbf{x}_a,20)$, $u$ is $(k,\epsilon,s)$-symmetric at $\mathbf{x}$ but not $(k+1,C^{-\Lambda}\eta,s)$-symmetric for all $s\in [\mathbf{r}_{\mathbf{x}},\gamma^{-10}]$ if $\widetilde{\delta}\leq C(\alpha,\Lambda)^{-1}\epsilon$ and $\delta'\leq  \eta^2\widetilde{\delta}^{4n}$.
    \end{claim}
    
    \begin{proof}
        Fix $\mathbf{x}\in \check{\cC}$ and $s\in [\mathbf{r}_\mathbf{x},\gamma^{-3}]$. Taking $\delta\leq \gamma^6\epsilon^2$, we can apply Claim \ref{claim:strongcballpinching} and Lemma \ref{newlineup}\ref{existence of plane of symmetry} with $\cC\leftarrow \check{\cC}\cup \{\widetilde{\mathbf{x}}\}$ for some $\mathbf{x}_0\leftarrow \widetilde{\mathbf{x}}\in \widetilde{\cC}$, $r_0\leftarrow \gamma^{-10}$, $s_0\leftarrow 1$, $\kappa\leftarrow 1$, $\mathbf{r}\leftarrow \mathbf{r}$, $\delta\leftarrow \max\{\delta', C(\alpha,\Lambda)(\widetilde{\delta}+\eta^{-\frac{1}{2n}}(\delta')^{\frac{1}{4n}})\}$ and the fact that $u$ is $(k,\delta',1)$-symmetric at $\widetilde{\mathbf{x}}$, we get that $u$ is $(k,\epsilon,s)$-symmetric at $\mathbf{x}$ if $\max\{\delta', C(\alpha,\Lambda)(\widetilde{\delta}+\eta^{-\frac{1}{2n}}(\delta')^{\frac{1}{4n}})\}\leq C^{-\Lambda}\epsilon$.
    
        Further, we claim that $u$ is not $(k+1,C^{-\Lambda}\eta,s)$-symmetric at $\mathbf{x}$ for any $s\in [\mathbf{r}_{\mathbf{x}},\gamma^{-10}]$. Suppose not. By Claim \ref{claim:strongcballpinching}, if $\max\{\delta', C(\alpha,\Lambda)(\widetilde{\delta}+\eta^{-\frac{1}{2n}}(\delta')^{\frac{1}{4n}})\}\leq C^{-\Lambda}\eta$, we can use Lemma \ref{newlineup}\ref{existence of plane of symmetry} $\cC \leftarrow \{\mathbf{x},\widetilde{\mathbf{x}}\}$ with $\mathbf{x}_0\leftarrow \mathbf{x}$, $s_0\leftarrow s$, $r_0\leftarrow \gamma^{-10}$, $\kappa \leftarrow 1$, $\mathbf{r}\leftarrow \mathbf{r}$, $\delta \leftarrow C^{-\Lambda}\eta$ to prove that $u$ is $(k+1, \eta, 1)$-symmetric at $\widetilde{\mathbf{x}}$, which is a contradiction.
    \end{proof}
    
    Let $\mathcal{C}_{+}$ be a maximal subset of $\check{\cC}$ such that $\widetilde{\mathcal{C}}_{+}\subseteq\mathcal{C}_{+}$ and $\{P(\mathbf{x},\gamma^{2}\mathbf{r}_{\mathbf{x}})\}_{\mathbf{x}\in\mathcal{C}_{+}\cup \widetilde{\cC}_0}$ is pairwise disjoint. Set $\mathcal{C}\coloneqq \mathcal{C}_{+}\cup \widetilde{\mathcal{C}}_{0}$.
    \begin{claim} \label{claim:strongcballneck} 
        If $\widetilde{\delta}\leq \ol{\widetilde{\delta}}(\alpha,\epsilon,\eta,\Lambda)$ and $\delta'\leq \overline{\delta}'(\alpha,\widetilde{\delta},\epsilon,\eta,\Lambda)$, then $\cN\coloneqq P(\mathbf{x}_a,20)\setminus \bigcup_{\mathbf{x}\in \cC}\ol{P}(\mathbf{x},\mathbf{r_x})$ is a strong $(m,k,\epsilon,C^{-\Lambda}\eta)$-neck of scale $10$.
    \end{claim}
    
    \begin{proof}
        First, \ref{neck-vitali-covering} and \ref{neck-lipschitz} hold by construction. Second, since $\cC\subset \check{\cC}$, \ref{neck-frequency-pinching} and \ref{neck-k-sym} hold because of Claim \ref{claim:strongcballpinching} and Claim \ref{claim:strongnecksymmetry} if $\widetilde{\delta}\leq C(\alpha, \Lambda)^{-1}\epsilon$  and $\delta'\leq \eta^{2}\widetilde{\delta}^{4n}$. 
        
        Now we prove \eqref{neck-n2-CinV}. Fix $\mathbf{x}\in \check{\cC}$ and $s\in [\mathbf{r}_{\mathbf{x}},\gamma^{-3}10]$ such that $P(\mathbf{x},s)\subseteq P(\mathbf{x}_a,20)$. Now we apply Claim \ref{claim:strongcballpinching} and Lemma \ref{newlineup}\ref{containment of plane of symmetry} with $\cC\leftarrow \check{\cC}$, $\mathbf{x}\leftarrow \mathbf{x}$, $\mathbf{x}_0\leftarrow \widetilde{\mathbf{x}}\in \widetilde{\cC}$, $\delta\leftarrow C(\alpha,\Lambda)(\widetilde{\delta}+\eta^{-\frac{1}{2n}}(\delta')^{\frac{1}{4n}})$, $\kappa\leftarrow 1$, $s_0\leftarrow 1$, $r_0\leftarrow \gamma^{-4}$, $\eta \leftarrow C(\alpha,\Lambda)^{-1}\eta$, $\zeta \leftarrow C(\alpha,\Lambda)(\widetilde{\delta}^{\frac{1}{2}}+\eta^{-\frac{1}{4n}}(\delta')^{\frac{1}{8n}})$, $\mathbf{r}_{\bullet} \leftarrow \mathbf{r}_{\bullet}$ to obtain
        \begin{align*}
            \check{\cC}\cap P(\mathbf{x},s) \subseteq P(\mathbf{x}+V,C(\alpha,\Lambda)\eta^{-\frac{1}{n}}(\widetilde{\delta}^{\frac{1}{2n}}+\eta^{-\frac{1}{4n^2}}(\delta')^{\frac{1}{8n^2}})s).
        \end{align*}
        The claim follows if we take $\widetilde{\delta}\leq \ol{\widetilde{\delta}}(\alpha,\epsilon,\eta,\Lambda)$ and $\delta'\leq \overline{\delta}'(\alpha,\widetilde{\delta},\epsilon,\eta,\Lambda)$.
        
        It thus remains to prove \eqref{eq:improvedCinVinC}. This will follow from proving that for all $\mathbf{x} \in \mathcal{C}$ and $s\in [\mathbf{r}_{\mathbf{x}},\gamma^{-3}]$ such that $P(\mathbf{x},s)\subseteq P(\mathbf{x}_a,20)$, we want to show
        \begin{equation} \label{eq:n4primestrongcball}
            (\mathbf{x}+V)\cap P(\mathbf{x},s)\subseteq P(\mathcal{C},\tfrac{\gamma}{10} s).
        \end{equation}
        Fix $\mathbf{z}\in(\mathbf{x}+V)\cap P(\mathbf{x},s)$. Recalling \eqref{eq def of widetilde w}, if $|\mathbf{z}-\widetilde{\mathbf{z}}|<\frac{\gamma}{10} s$, then because $\widetilde{\mathcal{C}}\subseteq\mathcal{C}$, we are done. Suppose instead $|\mathbf{z}-\widetilde{\mathbf{z}}|\geq \frac{\gamma}{10}s$.
        By Claim \ref{claim:strongcballpinching} with $\mathbf{w}\leftarrow \mathbf{x}$, we can apply Lemma \ref{newlineup}\ref{plane of symmetry inside pinched points} with $\mathcal{C} \leftarrow \check{\cC}$, $\mathbf{x}_0 \leftarrow \widetilde{\mathbf{x}}$, $\mathbf{x}\leftarrow \mathbf{x}$, $\mathbf{r}_{\bullet} \leftarrow \mathbf{r}_{\bullet}$, $r_0 \leftarrow \gamma^{-10}$, $s_0 \leftarrow 1$, $s\leftarrow s$, $\delta \leftarrow C(\alpha,\Lambda)(\widetilde{\delta}+\eta^{-\frac{1}{2n}}(\delta')^{\frac{1}{4n}})$, $\kappa \leftarrow 1$, we have 
        \begin{align} \label{eq:zispinched}
            \sup_{s'\in [s,\gamma^{-10}]}|N_{\mathbf{z}}((s')^2)-m|\leq C(\alpha,\Lambda)((\widetilde{\delta})^\frac{1}{2}+\eta^{-\frac{1}{4n}}(\delta')^{\frac{1}{8n}}).
        \end{align}
        Using \eqref{eq widetilde r and tilde r}, the $\epsilon$-Lipschitz property of $\mathbf{r}_\bullet$, and $\mathbf{r_x}\leq s$ we get
        \begin{align} \label{eq:estforrsubz}
            \widetilde{\mathbf{r}}_{\widetilde{\mathbf{z}}}\leq2\epsilon^{-1}\gamma^{-3}\mathbf{r}_{\mathbf{z}}\leq2 \epsilon^{-1}\gamma^{-3}(\mathbf{r}_{\mathbf{x}}+\epsilon s)\leq4\gamma\epsilon^{-1}\gamma^{-3} s.
        \end{align}
        Thus, we can apply Lemma \ref{newlineup}\ref{containment of plane of symmetry} with $\mathcal{C} \leftarrow \widetilde{\mathcal{C}}$, $\mathbf{r}_{\bullet} \leftarrow \widetilde{\mathbf{r}}_{\bullet}$, $\eta \leftarrow C(\alpha,\Lambda)^{-1}\eta$, $\mathbf{x}_0 \leftarrow \widetilde{\mathbf{z}}$, $\mathbf{x} \leftarrow \widetilde{\mathbf{z}}$, $r_0 \leftarrow \gamma^{-10}$, $s_0 \leftarrow 1$, $\delta \leftarrow C(\alpha,\Lambda)\eta^{-\frac{1}{n}}(\delta')^{\frac{1}{2n}}$ $\kappa \leftarrow 1$, $s\leftarrow 4\epsilon^{-1}\gamma^{-3}s$, and $\zeta \leftarrow C(\alpha,\Lambda)((\widetilde{\delta})^{\frac{1}{2}}+\eta^{-\frac{1}{4n}}(\delta')^{\frac{1}{8n}})$ to obtain
        \begin{align*}
            & \hspace{-30 mm} \{\mathbf{y} \in P(\widetilde{\mathbf{z}},40\epsilon^{-1}\gamma^{-3}s) \colon \mathcal{E}_{400\gamma^{-3}\epsilon^{-1}s}(\mathbf{y})<C(\alpha,\Lambda)((\widetilde{\delta})^{\frac{1}{2}}+\eta^{-\frac{1}{4n}}(\delta')^{\frac{1}{8n}})\} \\
            &\subseteq P(\widetilde{\mathbf{z}}+V,C(\alpha,\Lambda)\eta^{-\frac{1}{n}}((\widetilde{\delta})^{\frac{1}{2n}}+\eta^{-\frac{1}{4n^2}}(\delta')^{\frac{1}{8n^2}})\epsilon^{-1}s).
        \end{align*}
        Combining this with \eqref{eq:zispinched} and $|\mathbf{z}-\widetilde{\mathbf{z}}| \leq 2\epsilon^{-1}s$ (which follows from \eqref{eq lower bound on r function} and \eqref{eq:estforrsubz}), we conclude the existence of $\mathbf{y} \in \widetilde{\mathbf{z}}+V$ such that 
        \begin{align}
            |\mathbf{y}-\mathbf{z}| \leq C(\alpha,\Lambda)\eta^{-\frac{1}{n}}((\widetilde{\delta})^{\frac{1}{2n}}+\eta^{-\frac{1}{4n^2}}(\delta')^{\frac{1}{8n^2}})\epsilon^{-1} s \leq \gamma^{10} s\label{eq:yminusz}
        \end{align}
        if we choose $\widetilde{\delta}\leq \overline{\widetilde{\delta}}(\alpha,\epsilon,\eta,\Lambda)$ and then $\delta' \leq \overline{\delta'}(\alpha,\epsilon,\eta,\Lambda)$. 
        We claim that $\mathbf{y}\in \check{\cC}$ or equivalently (because $\mathbf{y} \in \widetilde{\mathbf{z}}+V$)
        \begin{equation} \label{eq:annoyingestimate} 
            d_{\mathcal{P}}(\widetilde{\mathbf{z}},\widetilde{\mathbf{y}}+V) \leq \widetilde{\delta}^2\mathbf{r}_{\mathbf{y}}.
        \end{equation} 
        Using our assumption $\verts{\mathbf{z}-\widetilde{\mathbf{z}}}> \frac{\gamma}{10} s$ and recalling \eqref{eq lower bound on r function}, we estimate
        \begin{align*}
            s \leq 10\gamma^{-1}d_{\mathcal{P}}(\mathbf{z},\widetilde{\mathcal{C}}) \leq 10\gamma^{-1}(d_{\mathcal{P}}(\mathbf{y},\widetilde{\mathcal{C}})+|\mathbf{y}-\mathbf{z}|) \leq 10\gamma^{-1}\epsilon^{-1}\mathbf{r}_{\mathbf{y}}+10\gamma^{9}s,
        \end{align*}
        to obtain $s \leq 20\epsilon^{-1}\gamma^{-1}\mathbf{r}_{\mathbf{y}}$. We thus also have
        \begin{align} \label{eq:ytildeztildeclose}
        \begin{split}
            |\widetilde{\mathbf{y}}-\mathbf{\widetilde{z}}| &\leq d_{\mathcal{P}}(\mathbf{y},\mathcal{\widetilde{C}})+|\mathbf{y}-\mathbf{z}|+ d_{\mathcal{P}}(\mathbf{z},\mathcal{\widetilde{C}}) \leq 2|\mathbf{y}-\mathbf{z}|+2d_{\mathcal{P}}(\mathbf{y},\widetilde{\mathcal{C}})\\
            & \leq 2\gamma^{10}s+2\epsilon^{-1}\mathbf{r}_{\mathbf{y}} \leq 4\epsilon^{-1}\mathbf{r}_{\mathbf{y}}. 
        \end{split}
        \end{align}
        By \eqref{eq widetilde r and tilde r}, we have $\mathbf{r}_{\mathbf{y}} \geq \frac{1}{2} \epsilon \gamma^3\widetilde{\mathbf{r}}_{\widetilde{\mathbf{y}}}$, so we can apply Lemma \ref{newlineup}\ref{containment of plane of symmetry} with $\mathcal{C} \leftarrow \widetilde{\mathcal{C}}$, $r_0 \leftarrow \gamma^{-9}$, $\delta \leftarrow \delta'$ $s_0 \leftarrow 1$, $\mathbf{x}_0 \leftarrow \widetilde{\mathbf{y}}$, $\mathbf{x} \leftarrow \widetilde{\mathbf{y}}$, $s\leftarrow 2\epsilon^{-1}\gamma^{-3}\mathbf{r}_{\mathbf{y}}$, $\zeta \leftarrow C(\Lambda)(\delta')^{\frac{1}{2}}$, $\eta \leftarrow C(\alpha,\Lambda)\eta$, $\kappa \leftarrow 1$ to get
        \begin{align*}
            \mathbf{\widetilde{z}} \in P(\widetilde{\mathbf{y}}+V,C(\alpha,\Lambda)\eta^{-\frac{1}{n}}(\delta')^{\frac{1}{2n}}\epsilon^{-1}\mathbf{r}_{\mathbf{y}}),
        \end{align*}
        where we also used \eqref{eq:ytildeztildeclose}. If we choose $\delta' \leq \overline{\delta'}(\alpha,\widetilde{\delta},\epsilon,\eta,\Lambda)$, then \eqref{eq:annoyingestimate} follows, hence $\mathbf{y} \in \check{\cC}$. If $\mathbf{y} \in \widetilde{\mathcal{C}}_0 \subseteq \mathcal{C}$, then we are done by using \eqref{eq:yminusz}. Otherwise, $\check{\cC}\subseteq\widetilde{\mathcal{C}}_{0}\cup\bigcup_{\mathbf{w}\in\mathcal{C}_{+}}P(\mathbf{w},10\gamma^{2}\mathbf{r}_{\mathbf{w}})$ gives the existence of $\mathbf{w}\in\mathcal{C}_{+}$ such that $|\mathbf{y}-\mathbf{w}|<10\gamma^{2}\mathbf{r}_{\mathbf{w}}$.
        Because 
        \begin{align*}
            \mathbf{r}_{\mathbf{w}}\leq\mathbf{r}_{\mathbf{x}}+\epsilon(|\mathbf{x}-\mathbf{z}|+|\mathbf{z}-\mathbf{y}|+|\mathbf{y}-\mathbf{w}|)\leq s+\epsilon(s+\gamma^{10}s+10\gamma^2\mathbf{r_w})
        \end{align*}
        we have $\mathbf{r}_{\mathbf{w}}\leq 4s$, so that 
        \begin{align*}
            |\mathbf{z}-\mathbf{w}|\leq |\mathbf{z}-\mathbf{y}|+|\mathbf{y}-\mathbf{w}|\leq \gamma^{10}s+40\gamma^{2}s\leq \frac{\gamma}{10} s.
        \end{align*}
        This proves \eqref{eq:n4primestrongcball}, and the claim follows.
    \end{proof}
    
    Let $\{P(\mathbf{x}_b,r_b)\}_{b\in B}$ be the set of \ref{b-balls}-balls of $\mathcal{N}$.
    We claim that there are no \ref{b-balls}-balls in the neck decomposition.
    \begin{claim} \label{claim:nobball} 
        $B=\widetilde{B}=\emptyset$.
    \end{claim}
    
    \begin{proof} 
        Suppose by way of contradiction there exists $b\in \widetilde{B}\cup B$. By Theorem \ref{thm: sym split equiv}\ref{thm: sym split equiv-1}, $u$ is $(k+1,C(\alpha,\Lambda)\eta,100r_b)$-symmetric at some point $\mathbf{y}_b\in P(\mathbf{x}_b,4r_b)$. By Lemma \ref{lemma-comparison-of-caloric-energy} with $\mathbf{x}_1 \leftarrow \mathbf{y}_b$, $\mathbf{x}_0 \leftarrow \mathbf{x}_b$, $\sigma \leftarrow 1$, $\theta \leftarrow \frac{1}{4}$, $r \leftarrow 4r_b$, $\tau \leftarrow 384 r_b^2$, $u$ is $(k+1,C(\Lambda,\alpha)\eta,20r_b)$-symmetric at $\mathbf{x}_b$, where we also used $(1+\frac{1}{4})\tau \leq 10^4$ and the monotonicity of the $L^2$ norms. Thus Lemma \ref{lemma epsilon regularity for frequency} gives
        \begin{align*}
            N_{\mathbf{x}_b}(r_b^2)\leq 
            \begin{cases}
                2\eta & \text{ if } k=n+1,\\
                1+C(\alpha,\Lambda)\eta & \text{ if } k=n,
            \end{cases}
        \end{align*}
        which contradicts our assumption on $m$ if $\eta \leq \overline{\eta}(\Lambda,\alpha)$, in view of \ref{neck-frequency-pinching}.
    \end{proof}
    
    Now it remains to prove the content estimate for the $c$-balls appearing in the strong neck region in Claim \ref{claim:strongcballneck}:
    \begin{align*}
        \sum_{c\in C}r_{c}^{k}\leq C\epsilon.
    \end{align*}
    First, we prove the containment of $\check{\cC}$ inside a neighborhood of $\widetilde{\cC}$.
    \begin{claim} 
    \label{claim:strongneckcontainmentS} 
    $\check{\cC} \cap P(\mathbf{x}_a,5)\subseteq\widetilde{\mathcal{C}}_{0}\cup\bigcup_{\mathbf{z}\in\widetilde{\mathcal{C}}_{+}}P(\mathbf{z},50\gamma\widetilde{\mathbf{r}}_{\mathbf{z}})$. 
    \end{claim}
    
    \begin{proof}
        Suppose by way of contradiction
        that $\mathbf{y}\in (\check{\cC}  \cap P(\mathbf{x}_a,5))\setminus\left(\widetilde{\mathcal{C}}_{0}\cup\bigcup_{\mathbf{z}\in\widetilde{\mathcal{C}}_{+}}P(\mathbf{z},50\gamma\widetilde{\mathbf{r}}_{\mathbf{z}})\right)$.
        
        First, assume that $d_{\mathcal{P}}(\mathbf{y},\widetilde{\mathcal{C}})> 2\widetilde{\mathbf{r}}_{\widetilde{\mathbf{y}}}$. Note that $d_{\cP}(\mathbf{y},\widetilde{\cC})\leq 5$. Therefore, $P(\widetilde{\mathbf{y}},d_{\cP}(\mathbf{y},\widetilde{\cC}))\subset P(\mathbf{x}_a,20)$ so that we can
        apply \eqref{neck-n2-VinC} to $\widetilde{\mathcal{C}}$ with $s \leftarrow 2\verts{\mathbf{y}-\widetilde{\mathbf{y}}}$ and $\mathbf{x} \leftarrow \widetilde{\mathbf{y}}$ to obtain
        \begin{align*}
            (\widetilde{\mathbf{y}}+V)\cap P(\widetilde{\mathbf{y}},2|\mathbf{y}-\widetilde{\mathbf{y}}|) \subseteq \bigcup_{\mathbf{z}\in \widetilde{\mathcal{C}}}P(\mathbf{z},10\gamma (\widetilde{\mathbf{r}}_{\mathbf{z}}+2|\mathbf{y}-\widetilde{\mathbf{y}}|)).
        \end{align*}
        Thus, for some $\mathbf{z}\in \widetilde{\cC}$
        \begin{align*}
            \verts{\mathbf{y}-\mathbf{z}}\leq 10\gamma (\widetilde{\mathbf{r}}_{\mathbf{z}}+2|\mathbf{y}-\widetilde{\mathbf{y}}|)\leq  10\gamma (\widetilde{\mathbf{r}}_{\mathbf{z}}+2|\mathbf{y}-\mathbf{z}|),
        \end{align*}
        which implies that $\verts{\mathbf{y}-\mathbf{z}}< 20\gamma \widetilde{\mathbf{r}}_{\mathbf{z}}$ which contradicts $\mathbf{y}\notin P(\mathbf{z},50\gamma \widetilde{\mathbf{r}}_{\mathbf{z}})$.
        
        We now instead assume that $d_{\mathcal{P}}(\mathbf{y},\widetilde{\mathcal{C}})\leq 2\widetilde{\mathbf{r}}_{\widetilde{\mathbf{y}}}$. Since $2\widetilde{\mathbf{r}}_{\widetilde{\mathbf{y}}}\leq 40\gamma$, we have $P(\widetilde{\mathbf{y}}, 2\widetilde{\mathbf{r}}_{\widetilde{\mathbf{y}}})\subset P(\mathbf{x}_a,20)$. Therefore, we can apply \eqref{neck-n2-VinC} to $\widetilde{\mathcal{C}}$ with $\mathbf{x}\leftarrow \widetilde{\mathbf{y}}$ and $s\leftarrow 2\widetilde{\mathbf{r}}_{\widetilde{\mathbf{y}}}$ to obtain
        \begin{align*}
            (\widetilde{\mathbf{y}}+V)\cap P(\widetilde{\mathbf{y}},2\widetilde{\mathbf{r}}_{\widetilde{\mathbf{y}}})\subseteq\bigcup_{\mathbf{z}\in\widetilde{\mathcal{C}}}P(\mathbf{z},10\gamma(\widetilde{\mathbf{r}}_{\mathbf{z}}+2\widetilde{\mathbf{r}}_{\widetilde{\mathbf{y}}})).
        \end{align*}
        Therefore, because $\mathbf{y}\in \check{\cC}$ and $\mathbf{y}\in P(\widetilde{\mathbf{y}},2\widetilde{\mathbf{r}}_{\widetilde{\mathbf{y}}})$, there exists $\mathbf{z}\in\widetilde{\mathcal{C}}$ such that 
        \begin{align*}
            |\mathbf{y}-\mathbf{z}| \leq 10\gamma(\widetilde{\mathbf{r}}_{\mathbf{z}}+2\widetilde{\mathbf{r}}_{\widetilde{\mathbf{y}}})+\widetilde{\delta}^{2}\mathbf{r}_{\mathbf{y}}\leq 10\gamma\widetilde{\mathbf{r}}_{\mathbf{z}}+20\gamma \widetilde{\mathbf{r}}_{\widetilde{\mathbf{y}}} +\widetilde{\delta}^2 (\mathbf{r}_{\widetilde{\mathbf{y}}}+2\epsilon \widetilde{\mathbf{r}}_{\widetilde{\mathbf{y}}})\leq 10\gamma\widetilde{\mathbf{r}}_{\mathbf{z}}+30\gamma\widetilde{\mathbf{r}}_{\widetilde{\mathbf{y}}},
        \end{align*}
        where we used $\epsilon$-Lipschitz property of $\mathbf{r}$ in the second inequality, and chose $\widetilde{\delta} \leq \gamma \epsilon$. From $\mathbf{y}\notin P(\mathbf{z},50\gamma\widetilde{\mathbf{r}}_{\mathbf{z}})$,
        we have $50\gamma\widetilde{\mathbf{r}}_{\mathbf{z}}<|\mathbf{y}-\mathbf{z}|\leq10\gamma\widetilde{\mathbf{r}}_{\mathbf{z}}+30\gamma\widetilde{\mathbf{r}}_{\widetilde{\mathbf{y}}}$
        hence $\widetilde{\mathbf{r}}_{\mathbf{z}}\leq\widetilde{\mathbf{r}}_{\widetilde{\mathbf{y}}}$.
        It follows that 
        \begin{align*}
            |\mathbf{y}-\widetilde{\mathbf{y}}|\leq|\mathbf{y}-\mathbf{z}|\leq40\gamma\widetilde{\mathbf{r}}_{\widetilde{\mathbf{y}}},
        \end{align*}
        contradicting $\mathbf{y}\notin P(\widetilde{\mathbf{y}},50\gamma\widetilde{\mathbf{r}}_{\widetilde{\mathbf{y}}})$.
    \end{proof}
    
    Now we prove that the content of the new $c$-balls intersecting with balls around an old center is small compared to its scale. More precisely, for $\mathbf{z}\in\widetilde{\mathcal{C}}_{+}$, we define 
    \begin{align*}
        \mathcal{C}_{\mathbf{z}}\coloneqq \left\{ \mathbf{x}\in\mathcal{C}\cap P(\mathbf{z},50\gamma\widetilde{\mathbf{r}}_{\mathbf{z}})\colon P(\mathbf{x},\mathbf{r}_{\mathbf{x}})\text{ is a }c\text{-ball}\right\} .
    \end{align*}
    
    \begin{claim} \label{claim:strongcballcontent} 
        For any $\mathbf{z}\in\widetilde{\mathcal{C}}_{+}$,
        \begin{align*}
            \sum_{\mathbf{x} \in \mathcal{C}_{\mathbf{z}}}\mathbf{r}_{\mathbf{x}}^{k} \leq C\epsilon \widetilde{\mathbf{r}}_{\mathbf{z}}^{k}.
        \end{align*}
    \end{claim}
    
    \begin{proof} 
        Note that if $P(\mathbf{z},\widetilde{\mathbf{r}}_{\mathbf{z}})$ is an \ref{e-balls}-ball, then $\cC_{\mathbf{z}}=\emptyset$ because $\mathbf{x}\in \cC_{\mathbf{z}}$ is a center of a \ref{c-balls}-ball. By Claim \ref{claim:nobball}, it therefore suffices to consider the case when $P(\mathbf{z},\widetilde{\mathbf{r}}_{\mathbf{z}})$ is a \ref{d-balls}-ball.
        
        Any $\mathbf{x} \in \mathcal{C}_{\mathbf{z}}$ satisfies $\mathbf{r}_{\mathbf{x}} \leq \mathbf{r}_{\mathbf{z}}+\gamma \epsilon \widetilde{\mathbf{r}}_{\mathbf{z}} \leq 100\epsilon \gamma \widetilde{\mathbf{r}}_{\mathbf{z}}$, so that $P(\mathbf{x},4\mathbf{r}_{\mathbf{x}}) \subseteq P(\mathbf{z},\widetilde{\mathbf{r}}_{\mathbf{z}})$. Because $P(\mathbf{x},\mathbf{r}_{\mathbf{x}})$ is a \ref{c-balls}-ball, there exists $\mathbf{y} \in \mathcal{V}_{\delta,m,\mathbf{r}_{\mathbf{x}}}(\mathbf{x}) \subseteq \mathcal{V}_{\delta,m,\mathbf{\widetilde{\mathbf{r}}_{\mathbf{z}}}}(\mathbf{z}) \cap P(\mathbf{z},\widetilde{\mathbf{r}}_{\mathbf{z}})$. Because $\verts{\mathbf{x}-\mathbf{y}}\leq 4\mathbf{r_x}$, by triangle inequality and the previous inclusion 
        \begin{align}
            P(\mathbf{x},\epsilon\widetilde{\mathbf{r}}_{\mathbf{z}})\subseteq P(\mathcal{V}_{\delta,m,\widetilde{\mathbf{r}}_{\mathbf{z}}}(\mathbf{z}),2\epsilon \widetilde{\mathbf{r}}_{\mathbf{z}}).\label{eq-containment}
        \end{align}
        Let
        $\{\mathbf{x}_i \}_{i=1}^M$ be a maximal subset of 
        $\mathcal{C}_{\mathbf{z}}$ such that $|\mathbf{x}_i -\mathbf{x}_j|\geq 2\epsilon\widetilde{\mathbf{r}}_{\mathbf{z}}$. Using \eqref{eq-containment} and that $P(\mathbf{z},\widetilde{\mathbf{r}}_{\mathbf{z}})$ is a \ref{d-balls}-ball, there exists $W\in \operatorname{Gr}_{\mathcal{P}}(k-1)$ such that 
        \begin{align*}
            \bigcup_{i=1}^M P(\mathbf{x}_i,\epsilon\widetilde{\mathbf{r}}_{\mathbf{z}})\subseteq P(W,(2\epsilon+\alpha)\widetilde{\mathbf{r}}_{\mathbf{z}})\cap P(\mathbf{z},\widetilde{\mathbf{r}}_{\mathbf{z}}).
        \end{align*}
        It follows that 
        \begin{align*}
            M (\epsilon\widetilde{\mathbf{r}}_{\mathbf{z}})^{n+2} \leq C(n)(\epsilon+\alpha)^{n+2-k+1}\widetilde{\mathbf{r}}_{\mathbf{z}}^{n+2}.
        \end{align*}
        That is, $M\leq C\epsilon^{-k+1}$. Now by Ahlfors regularity  \eqref{eq:Ahlfors} and the maximality of $\mathbf{x}_i$
        \begin{align*}
            \sum_{\mathbf{x} \in \mathcal{C}_{\mathbf{z}}} \mathbf{r}_{\mathbf{x}}^{k} \leq \sum_{i=1}^M \mu(\mathcal{C}_{\mathbf{z}}\cap P(\mathbf{x}_i,10\epsilon\widetilde{\mathbf{r}}_{\mathbf{z}}))\leq CM (\epsilon\mathbf{\widetilde{r}}_{\mathbf{z}})^k \leq  C\epsilon \widetilde{\mathbf{r}}_{\mathbf{z}}^{k}.
        \end{align*}
    \end{proof}
    
    From \eqref{eq:n4primestrongcball} and Remark \ref{remark improved frequency pinching for centers}, we may apply Lemma \ref{lemma restriction of neck region} to conclude that $\mathcal{N}'\coloneqq \mathcal{N} \cap P(\mathbf{x}_a,5)$ is a strong $(m,k,\delta,C(\alpha,\Lambda)\eta)$-neck region. Moreover, setting 
    \begin{align*} 
        \mathcal{C}_c \coloneqq  \{ \mathbf{x} \in \mathcal{C}\cap P(\mathbf{x}_a,5) \colon P(\mathbf{x},\mathbf{r}_{\mathbf{x}})  \text{ is a \ref{c-balls}-ball} \},
    \end{align*}
    we can combine Claim \ref{claim:strongneckcontainmentS}, Claim \ref{claim:strongcballcontent} and \eqref{eq packing measure estimate for tilde neck} to obtain
    \begin{align*}
        \sum_{\mathbf{x} \in \mathcal{C}_c} r_{\mathbf{x}}^{k} \leq \sum_{\mathbf{z} \in \widetilde{\mathcal{C}}_+} \sum_{\mathbf{x} \in \mathcal{C}_{\mathbf{z}}}\mathbf{r}_{\mathbf{x}}^{k} \leq C\epsilon \sum_{\mathbf{z} \in \widetilde{\mathcal{C}}_+}\widetilde{\mathbf{r}}_{\mathbf{z}}^{k} \leq C\epsilon,
    \end{align*}
    from which the claim follows.
\end{proof}

\subsection{Covering of \ref{d-balls}-balls}

In this section, we show that \ref{d-balls}-balls can be decomposed into balls of other types, in such a way that the $k$-content of the resulting \ref{c-balls}-balls is small. The proof closely follows \cite[Section 7.5]{jiang-naber-2021-l2-curvature}.

\begin{proposition}[Covering of \ref{d-balls}-balls]\label{dballcovering}
    If $\alpha\le\bar\alpha$ and $\delta>0$, then the following statements hold. Let $u$ be a caloric function satisfying 
    \begin{align}
        \sup_{\mathbf{x}\in P(\mathbf{0},20r)} N_{\mathbf{x}}^u(\gamma^{-5}r^2) \leq m+\delta, \label{eq: d ball decomposition frequency pinching}
    \end{align}
    for some integer $m\le \Lambda$. Suppose $\mathcal{V}\coloneqq \mathcal{V}_{\delta,m,r}(\mathbf{0})\neq \emptyset$ is not $(k,\alpha r)$-independent. Then the following hold.
    \begin{enumerate}[label=(\arabic*)]
        \item \label{d ball covering-continuous} There exists a decomposition
        \begin{align*}
            P(\mathbf{0},r) \subseteq \bigcup_{b\in B} P(\mathbf{x}_b,r_b) \cup \bigcup_{c\in C} P(\mathbf{x}_c,r_c) \cup \bigcup_{e\in E} P(\mathbf{x}_e,r_e) \cup \mathcal{S}_d,
        \end{align*}
        where $\mathcal{H}_{\mathcal{P}}^k(\mathcal{S}_d)=0$, and we have the $k$-dimensional content estimates:
        \begin{align*}
            \sum_{c\in C} r_c^k \leq C\alpha r^k, && \sum_{b\in B}r_b^k + \sum_{e\in E}r_e^k\leq C\alpha^{k-n-2}r^k.
        \end{align*}

        \item \label{d ball covering-discrete} For any $r_\ast\in (0,1]$, there exists a decomposition
        \begin{align*}
            P(\mathbf{0},r) \subseteq \bigcup_{b\in B} P(\mathbf{x}_b,r_b) \cup \bigcup_{c\in C} P(\mathbf{x}_c,r_c) \cup \bigcup_{e\in E} P(\mathbf{x}_e,r_e) \cup \bigcup_{f\in F} P(\mathbf{x}_f,r_f),
        \end{align*}
        where $r_b,r_c,r_e \geq r_{\ast}$, $r_f \in [r_{\ast},10\alpha^{-1}r_{\ast}]$, and we have the $k$-dimensional content estimates:
        \begin{align*}
            \sum_{c\in C} r_c^k \leq C\alpha r^k, && \sum_{b\in B}r_b^k + \sum_{e\in E}r_e^k+\sum_{f\in F}r_f^k\leq C\alpha^{k-n-2}r^k.
        \end{align*}
    \end{enumerate}
\end{proposition}

\begin{proof} 
    \ref{d ball covering-continuous} By parabolic rescaling, we may assume $r=1$. Choose a maximal subset $\{\mathbf{x}_i\}_{i=1}^N$ of $P(\mathbf{0},1)$ satisfying $|\mathbf{x}_i-\mathbf{x}_j|\ge \frac{\alpha}{10}$ when $i\neq j$. 
    Then categorize the balls of this cover into types \ref{b-balls}-\ref{e-balls} in order to obtain
    \begin{align*}
        P(\mathbf{0},1)\subseteq \bigcup_{b\in B^1} P(\mathbf{x}_b,r_b) \cup \bigcup_{c\in C^1} P(\mathbf{x}_c,r_c) \cup \bigcup_{d\in D^1} P(\mathbf{x}_d,r_d) \cup \bigcup_{e\in E^1} P(\mathbf{x}_e,r_e),
    \end{align*}
    where $r_b\coloneqq r_c\coloneqq r_d\coloneqq r_e\coloneqq \frac{\alpha}{10}$.
    A simple covering argument gives
    \begin{align*}
        \sum_{b\in B^1} r_b^k + \sum_{e\in E^1} r_e^k \leq C\alpha^{k-n-2}\eqqcolon C_1(\alpha).
    \end{align*}
    Because $\mathcal{V}\coloneqq\mathcal{V}_{\delta,m,1}(\mathbf{0})\neq \emptyset$ is not $(k,\alpha)$-independent, there exists $V\in {\rm Aff}_{\mathcal{P}}(k-1)$ such that $\mathcal{V}\subseteq P(V,\alpha)\cap P(\mathbf{0},1)$. By the monotonicity of the frequency and \eqref{eq: d ball decomposition frequency pinching}, it follows that $\mathcal{V}_{\delta,m,r_c}(\mathbf{x}_c)\subseteq \mathcal{V}$ and $\mathcal{V}_{\delta,m,r_d}(\mathbf{x}_d)\subseteq \mathcal{V}$. Thus, $P(\mathbf{x}_c,r_c),P(\mathbf{x}_d,r_d) \subseteq P(V,5\alpha)\cap P(\mathbf{0},1)$ for any $c\in C^1$ and $d\in D^1$, which implies
    \begin{align*}
        |C^1|+|D^1| \leq C \alpha^{1-k}.
    \end{align*}
    In particular,
    \begin{align*}
        \sum_{c\in C^1}r_c^k+\sum_{d\in D^1} r_d^k \leq C\alpha.
    \end{align*}
    Because $P(\mathbf{x}_d,r_d)$ satisfies the same hypotheses for $d\in D^1$, we may apply the above procedure to each such ball, obtaining a cover
    \begin{align*}
        P(\mathbf{0},1) \subseteq \bigcup_{b\in B^2} P(\mathbf{x}_b,r_b) \cup \bigcup_{c\in C^2} P(\mathbf{x}_c,r_c) \cup \bigcup_{d\in D^2} P(\mathbf{x}_d,r_d) \cup \bigcup_{e\in E^2} P(\mathbf{x}_e,r_e),
    \end{align*}
    where $r_b \coloneqq r_c\coloneqq r_d\coloneqq r_e\coloneqq (\frac{\alpha}{10})^2$ for $b\in B^2\setminus B^1$, $c\in C^2\setminus C^1$, $e\in E^2\setminus E^1$, $d\in D^2$. Further, the following $k$-content estimates hold:
    \begin{align*}
        \sum_{b\in B^2} r_b^k +\sum_{e\in E^2} r_e^k \leq C_1(1+C\alpha), \qquad \sum_{c\in C^2} r_c^k \leq C\alpha(1+C\alpha), \qquad
        \sum_{d\in D^2} r_d^k  \leq (C\alpha)^{2}.
    \end{align*}
    By iterating this procedure, we therefore obtain a cover
    \begin{align}
        P(\mathbf{0},1) \subseteq \bigcup_{b\in B^\ell} P(\mathbf{x}_b,r_b) \cup \bigcup_{c\in C^\ell} P(\mathbf{x}_c,r_c) \cup \bigcup_{d\in D^\ell} P(\mathbf{x}_d,(\tfrac{\alpha}{10})^\ell) \cup \bigcup_{e\in E^\ell} P(\mathbf{x}_e,r_e),\label{eq-d-ball-covering-proof}
    \end{align}
    where we have the $k$-content estimates
    \begin{align*}
        \sum_{b\in B^{\ell}} r_b^k + \sum_{c\in C^{\ell}} r_c^k
        \le 2C_1, \qquad \sum_{c\in C^{\ell}} r_c^k \leq 
        C\alpha \sum_{i=0}^\ell (C\alpha)^{i}, \qquad
        \sum_{d\in D^{\ell}} r_d^k &\leq (C\alpha)^\ell,
    \end{align*}
    if we choose $\alpha<(2C)^{-1}$. If we set $B\coloneqq\cup_{\ell}B^{\ell}$, $C\coloneqq\cup_{\ell} C^{\ell}$, $E\coloneqq \cup_{\ell} E^{\ell}$, it follows that 
    \begin{align*}
        P(\mathbf{0},1) \subseteq \bigcup_{b\in B} P(\mathbf{x}_b,r_b) \cup \bigcup_{c\in C} P(\mathbf{x}_c,r_c) \cup \mathcal{S}_d \cup \bigcup_{e\in E} P(\mathbf{x}_e,r_e),
    \end{align*}
    where we observe that $\mathcal{S}_d \coloneqq \bigcap_{\ell} \bigcup_{d\in D^\ell} P(\mathbf{x}_d,(\tfrac{\alpha}{10})^\ell)$ satisfies $\mathcal{H}_{\mathcal{P}}^k(\mathcal{S}_d)=0$. 

    \ref{d ball covering-discrete} Fix $\ell\in \N_0$ such that $(\tfrac{\alpha}{10})^{\ell+1} \leq  r_\ast \leq (\tfrac{\alpha}{10})^{\ell}$. Set $B\coloneqq B^\ell$, $C\coloneqq C^\ell$, $E\coloneqq E^{\ell}$ and $F\coloneqq D^\ell$ in the decomposition \eqref{eq-d-ball-covering-proof}.
\end{proof}

\subsection{Proof of neck decomposition theorems}

We now apply Proposition \ref{cballcovering} and Proposition \ref{dballcovering} inductively to obtain the neck decomposition theorems, following \cite[Section 10.4]{cheeger-jiang-naber-2021-Sharp-quantitative}.

\begin{proof}[Proof of Theorem \ref{theorem-neck-decomposition}]
    By Lemma \ref{lemma-frequency-uniform-bound}, we have
    \begin{equation} \label{eq: all is well}
        \sup_{\mathbf{x}\in P(\mathbf{0},2)} N_{\mathbf{x}}(\gamma^{-5})\leq m
    \end{equation}
    for some integer $m\le C\Lambda$. Let $\alpha,\delta>0$ be the parameters in \ref{c-balls}, \ref{d-balls}, and \ref{e-balls} balls to be determined. Choose a Vitali cover of $P(\mathbf{0},1)$ by parabolic balls of radius $\gamma^{10}$, and categorize these as balls of type \ref{b-balls}-\ref{e-balls}, obtaining a cover
    \begin{align*}
        P(\mathbf{0},1)\subseteq \bigcup_{b\in B^0} P(\mathbf{x}_b,r_b) \cup \bigcup_{c \in C^0} P(\mathbf{x}_c,r_c) \cup \bigcup_{d\in D^0} P(\mathbf{x}_d,r_d) \cup \bigcup_{e\in E^0} P(\mathbf{x}_e,r_e)
    \end{align*}
    with $r_b=r_c=r_d=r_e=\gamma^{10}$, and
    \begin{align*}
        \sum_{b\in B^0} r_b^k + \sum_{c\in C^0} r_c^k +\sum_{d\in D^0} r_d^k +\sum_{e \in E^0} r_e^k \leq C.
    \end{align*}
    Let $C_0$ be the maximum of $C$ and the constants $C$ appearing in Proposition \ref{cballcovering} and Proposition \ref{dballcovering}. 
    
    For each $c\in C^0$, if $\delta \leq C(\alpha)^{\Lambda}\eta^{\frac{1}{n}}\epsilon^{10n^2}$, we can apply the \ref{c-balls}-ball covering (Proposition \ref{cballcovering}\ref{c-ball decomposition-continuous}) to $P(\mathbf{x}_c,r_c)$ for each $c\in C^0$ to obtain
    \begin{align*}
         P(\mathbf{0},1)\subseteq \bigcup_{a\in A^1} (\mathcal{N}^a \cup \mathcal{C}_{a,0}) \cup \bigcup_{b\in B^1} P(\mathbf{x}_b,r_b) \cup \bigcup_{d\in D^1} P(\mathbf{x}_d,r_d) \cup \bigcup_{e\in E^1} P(\mathbf{x}_e,r_e).
    \end{align*}
    Here $\mathcal{N}^a$ is an $(m,k,\epsilon, C(\alpha)^{-\Lambda})$-neck region and
    \begin{align*}
        \sum_{a\in A^1} r_a^k +\sum_{b\in B^1} r_b^k +\sum_{d\in D^1} r_d^k + \sum_{e\in E^1} r_e^k \leq C_0+C_0 \sum_{c\in C^0} r_c^k\leq 2C_0^2.
    \end{align*}
    For each $d\in D^1$, fixing $\alpha\le\bar\alpha$ to be determined later, we can apply the \ref{d-balls}-ball covering (Proposition \ref{dballcovering}\ref{d ball covering-continuous}) to obtain
    \begin{align*}
        P(\mathbf{0},1)\subseteq \bigcup_{a\in \ol{A}^1} (\mathcal{N}^a \cup \mathcal{C}_{a,0}) \cup \bigcup_{b\in \ol{B}^1} P(\mathbf{x}_b,r_b) \cup \bigcup_{c\in \ol{C}^1} P(\mathbf{x}_c,r_c) \cup \bigcup_{e\in \ol{E}^1} P(\mathbf{x}_e,r_e)\cup S_0^1,
    \end{align*}
    with the $k$-content estimates 
    \begin{align*}
        \sum_{b\in \ol{B}^1} r_b^k +\sum_{e\in \ol{E}^1} r_e^k \leq 2C_0^2+ C_1 \sum_{d\in D^1} r_d^k\leq 2C_0^2(1+ C_1),&&\sum_{c\in \ol{C}^1} r_c^k \leq C_0\alpha \sum_{d\in D^1} r_d^k  \leq  2\alpha C_0^3.
    \end{align*}
    Furthermore, $S_0^1 \coloneqq \cup_{d\in D^1} \mathcal{S}_d$ satisfies $\mathcal{H}_{\mathcal{P}}^k(S_0^1)=0$, where $C_1\coloneqq C\alpha^{k-n-2}$ appearing in Proposition \ref{dballcovering}\ref{d ball covering-continuous}. 
        
    We then apply the \ref{c-balls}-ball covering to $P(\mathbf{x}_c,r_c)$ with $c\in \ol{C}^1$, and then apply the \ref{d-balls}-ball covering to $P(\mathbf{x}_d,r_d)$ with $d\in D^2$. After the $\ell$-th iteration, we obtain
    \begin{align*}
        P(\mathbf{0},1) \subseteq  \bigcup_{a\in \ol{A}^\ell} (\mathcal{N}_a \cup \mathcal{C}_{a,0}) \cup \bigcup_{b\in \overline{B}^\ell} P(\mathbf{x}_b,r_b) \cup \bigcup_{c \in \ol{C}^\ell} P(\mathbf{x}_c,r_c) \cup \bigcup_{e\in \overline{E}^\ell} P(\mathbf{x}_e,r_e)\cup S_0^\ell,
    \end{align*}
    where $\mathcal{H}_{\mathcal{P}}^k(S_0^\ell)=0$,  and
    \begin{align*}
        \sum_{a\in \overline{A}^{\ell}} \left(r_a^k + \mathcal{H}_{\mathcal{P}}^k(\mathcal{C}_{a,0}) \right) +\sum_{b\in \overline{B}^\ell} r_b^k + \sum_{e\in \overline{E}^\ell} r_e^k \leq 2C_0^2(1+C_1) \sum_{j=0}^{\ell-1} (\alpha C_0^2)^{j},&&
        \sum_{c\in {\ol{C}}^{\ell}} r_c^k \leq 2C_0 (\alpha C_0^2)^{\ell}.
    \end{align*}
    Set $A\coloneqq \cup_{\ell}\overline{A}^{\ell}$, $B\coloneqq \cup_{\ell}\overline{B}^{\ell}$,
    $E\coloneqq \cup_{\ell}\overline{E}^{\ell}$, and
    \begin{align}
        \widetilde{\mathcal{C}}_m\coloneqq \bigcup_{\ell}S_{0}^{\ell}\cup\bigcap_{j}\bigcup_{\ell\geq j}\left(\bigcup_{c\in\overline{C}^{\ell}}P(\mathbf{x}_{c},r_{c})\right),\label{eq: residual set at step m}
    \end{align}
    so that $\mathcal{H}_{\mathcal{P}}^{k}(\widetilde{\mathcal{C}}_m)=0$ and
    \begin{align*}
        P(\mathbf{0},1)\subseteq\bigcup_{a\in A}(\mathcal{N}_{a}\cup\mathcal{C}_{a,0})\cup\bigcup_{b\in B}P(\mathbf{x}_{b},r_{b})\cup\bigcup_{e\in E}P(\mathbf{x}_{e},r_{e})\cup\widetilde{\mathcal{C}}_m.
    \end{align*}
    If $\alpha<\frac{1}{2C_{0}^{2}}$, then we have the $k$-content estimate
    \begin{align*}
        \sum_{a\in A}r_{a}^{k}+\sum_{b\in B}r_{b}^{k}+\sum_{e\in E}r_{e}^{k}\leq C(\alpha).
    \end{align*}
    Furthermore, by construction, we have
    \begin{align*}
        \sup_{e\in E} \sup_{\mathbf{x}\in P(\mathbf{x}_e,4r_e)} N_{\mathbf{x}}^u(\delta^2 r_e^2) < m-\delta.
    \end{align*}
    By Lemma \ref{lemma-refined-monotonicity-of-frequency}, we thus have
    \begin{align*}
        \sup_{e\in E} \sup_{\mathbf{x}\in P(\mathbf{x}_e,4r_e)} N_{\mathbf{x}}^u(\delta^4 r_e^2) < m-1+\delta.
    \end{align*}
    We can therefore cover each $P(\mathbf{x}_e,r_e)$ by a set $\{P(\mathbf{x}_{e'},r_{e'})\}_{e'\in E'_e}$ of $C \delta^{-2(n+2)}$ parabolic balls of radius $r_{e'}=\delta^2 r_{e}$ which satisfy the hypotheses of the \ref{c-balls}-ball and \ref{d-balls}-ball covering lemmas. For each $e'\in E'_e$, we can repeat this entire procedure in order to obtain a cover 
    \begin{align*}
        P(\mathbf{x}_{e'},r_{e'}) \subseteq \bigcup_{a \in A_{e'}} (\mathcal{N}_a \cup \mathcal{C}_{a,0}) \cup \bigcup_{b\in B_{e'}} P(\mathbf{x}_b,r_b) \cup \bigcup_{e''\in E_{e'}} P(\mathbf{x}_{e''},r_{e''}) \cup \widetilde{\mathcal{C}}_{m-1,e'}
    \end{align*}
    satisfying the $k$-content estimate
    \begin{align*}
        \sum_{a\in A_{e'}}r_a^k + \sum_{b\in B_{e'}}r_b^k + \sum_{e''\in E_{e'}}r_{e''}^k \leq C\delta^{2k} r_e^k,
    \end{align*}
    as well as $\mathcal{H}_{\mathcal{P}}^k(\widetilde{\mathcal{C}}_{m-1,e'})=0$. By taking the union over the $C\delta^{-2(n+2)}$ balls in this cover and then over all $e\in E$, we obtain a cover
    \begin{align*}
         P(\mathbf{0},1) \subseteq \bigcup_{a\in A^{(1)}} (\mathcal{N}_a \cup \mathcal{C}_{a,0}) \cup \bigcup_{b\in B^{(1)}} P(\mathbf{x}_b,r_b) \cup \bigcup_{e\in E^{(1)}}P(\mathbf{x}_e,r_e) \cup \widetilde{\mathcal{C}}_m \cup \widetilde{\mathcal{C}}_{m-1}
    \end{align*}
    satisfying the $k$-content estimate
    \begin{align*} 
        \sum_{a\in A^{(1)}} r_a^k + \sum_{b\in B^{(1)}}r_b^k + \sum_{e\in E^{(1)}}r_e^{k} &\leq \sum_{a\in A} r_a^k + \sum_{b\in B}r_b^k +\sum_{e\in E}\sum_{e'\in E_e'}\left(\sum_{a\in A_{e'}}r_a^k + \sum_{b\in B_{e'}}r_b^k + \sum_{e''\in E_{e'}}r_{e''}^k \right) \\ &\leq C+ C\delta^{-2(n+2-k)} \sum_{e\in E} r_e^k \\ & \leq C\delta^{-2(n+2-k)},
    \end{align*}
    as well as $\mathcal{H}_{\mathcal{P}}^k(\widetilde{\mathcal{C}}_{m-1})=0$, and assertions \ref{neckdecomposition:aballs}, \ref{neckdecomposition:bballs}, \ref{neckdecomposition:Minkowskiofcenters}. Moreover, for any $e\in E^{(1)}$, we have $\sup_{\mathbf{x}\in P(\mathbf{x}_e,4r_e)} N_{\mathbf{x}}^u(\delta^2r_e^2)<m-1-\delta$, so that $\sup_{\mathbf{x}\in P(\mathbf{x}_e,4r_e)} N_{\mathbf{x}}^u(\delta^4 r_e^2)<m-2+\delta$. Repeating the entire procedure at most $m\le C\Lambda$, we eventually obtain a cover
    \begin{align*}
        P(\mathbf{0},1) \subseteq \bigcup_{a\in A^{(m+1)}} (\mathcal{N}_a \cup \mathcal{C}_{a,0}) \cup \bigcup_{b\in B^{(m+1)}} P(\mathbf{x}_b,r_b) \cup \bigcup_{e\in E^{(m+1)}} P(\mathbf{x}_e,r_e)\cup \widetilde{\mathcal{C}},
    \end{align*}
    which satisfies assertions \ref{neckdecomposition:aballs}, \ref{neckdecomposition:bballs}, \ref{neckdecomposition:Minkowskiofcenters}. Moreover, for any $e\in E^{(m+1)}$, we have 
    \begin{align*}
        \sup_{\mathbf{x}\in P(\mathbf{x}_e,4r_e)} N_{\mathbf{x}}^u(\delta^4 r_e^2)<m-(m+1)+\delta<0,
    \end{align*}
    so that in fact $E^{(m+1)}=\emptyset$. By taking $\delta \coloneqq C(\alpha)^{\Lambda}\eta^{\frac{1}{n}}\epsilon^{10n^2}$, it follows that the cover satisfies the content estimate
    \begin{align*}
        \sum_{a\in A^{(m+1)}} r_a^k + \sum_{b\in B^{(m+1)}} r_b^k \leq C^{\Lambda}\delta^{-C\Lambda} \leq C^{\Lambda^2}(\eta\epsilon)^{-C\Lambda}.
    \end{align*}
\end{proof}

\begin{proof}[Proof of Theorem \ref{theorem-neck-decomposition2}]
    We can argue as in the proof of Theorem \ref{theorem-neck-decomposition} but using Proposition \ref{cballcovering}\ref{c-ball decomposition-discrete} and Proposition \ref{dballcovering}\ref{d ball covering-discrete} instead of Proposition \ref{cballcovering}\ref{c-ball decomposition-continuous} and Proposition \ref{dballcovering}\ref{d ball covering-continuous}. As a result, we obtain a decomposition for each $\ell\in \N$
    \begin{align*}
        P(\mathbf{0},1) \subseteq  \bigcup_{a\in A^\ell} \mathcal{N}_a \cup \bigcup_{b\in B^\ell} P(\mathbf{x}_b,r_b) \cup \bigcup_{c \in C^\ell} P(\mathbf{x}_c,r_c) \cup \bigcup_{e\in E^\ell} P(\mathbf{x}_e,r_e) \cup \bigcup_{f\in F^\ell} P(\mathbf{x}_f,r_f),
    \end{align*}
    where $\mathcal{N}_a = P(\mathbf{x}_a,r_a)\setminus P(\mathcal{C}_a,\mathbf{r}_{\bullet})$ is an $(m,k,\epsilon,C(\alpha)^{-\Lambda}\eta)$-neck region with $\mathbf{r}_{\bullet},r_b,r_c,r_e \geq r_{\ast}$, and $r_f \in [r_{\ast},(10\alpha^{-1}+\gamma^{-1})r_{\ast}]$. Because $r_c \leq \left(\gamma +\frac{\alpha}{10}\right)^{\ell}$ for any $c\in C^{\ell}$, we must have $C^{\ell}=\emptyset$ for sufficiently large $\ell \in \mathbb{N}$. Fix such $\ell$ and set $A\coloneqq A^\ell$, $B=B^\ell$, $E=E^\ell$, and $F=F^\ell$. We then have
    \begin{align*}
        \sum_{a\in A}r_{a}^{k}+\sum_{b\in B}r_{b}^{k}+\sum_{e\in E}r_{e}^{k} +\sum_{f\in F}r_{f}^{k}\leq C(\alpha).
    \end{align*}
    Now we proceed as in the proof of Theorem \ref{theorem-neck-decomposition} to cover \ref{e-balls}-balls for $e\in E$ with radius $\delta^2 r_e$. If $\delta^2 r_e\leq r_\ast \leq r_e$ such \ref{e-balls}-balls are considered as $f$-balls. The remaining \ref{e-balls}-balls can be decomposed as above. We repeat this procedure at most $m+1\leq C\Lambda$ times.
\end{proof}

\begin{proof}[Proof of Theorem \ref{theorem-neck-decomposition3}]
    We argue as in the proof of Theorem \ref{theorem-neck-decomposition} but using Proposition \ref{prop:strongcballcovering} instead of instead of Proposition \ref{cballcovering}\ref{c-ball decomposition-continuous} and we ask $m\geq m_\ast$ in Proposition \ref{dballcovering}\ref{d ball covering-continuous}. Note that as in Claim \ref{claim:nobball}, $m\geq m_\ast$ guarantees that there are no \ref{b-balls}-balls in the decomposition arising from Proposition \ref{dballcovering}\ref{d ball covering-continuous}. Thus
    \begin{align*}
        P(\mathbf{0},1) \subseteq  \bigcup_{a\in A_m} (\mathcal{N}_a \cup \mathcal{C}_{a,0})\cup \bigcup_{e\in E_m} P(\mathbf{x}_e,r_e) \cup \widetilde{\cC}_m,
    \end{align*}
    where $\mathcal{N}_a = P(\mathbf{x}_a,r_a)\setminus P(\mathcal{C}_a,\mathbf{r}_{\bullet})$ are strong $(m,k,\epsilon, C(\alpha,\Lambda)\eta)$-neck regions, $P(\mathbf{x}_e,r_e)$ are \ref{e-balls}-balls, and $\widetilde{\cC}_m$ is the residual set, which satisfies $\mathcal{H}_{\mathcal{P}}^k(\widetilde{\mathcal{C}}_m)=0$ and arises as in \eqref{eq: residual set at step m} when we apply Proposition \ref{cballcovering}\ref{c-ball decomposition-continuous} and  Proposition \ref{dballcovering}\ref{d ball covering-continuous} iteratively. Moreover, we have the $k$-content estimate
    \begin{align*}
        \sum_{a\in A_m} r_a^k + \sum_{e\in E}r_e^k \leq C(\Lambda)
    \end{align*}
    if $\epsilon\leq \ol{\epsilon}$, $\alpha\leq \ol{\alpha}$. Now we proceed as in the proof of Theorem \ref{theorem-neck-decomposition} to cover \ref{e-balls}-balls for $e\in E$ with radius $\delta^2 r_e$. We repeat the procedure at most $(m-2)\leq C\Lambda$ times so that we are left with a cover
    \begin{align*}
        P(\mathbf{0},1) \subseteq \bigcup_{a\in A} (\mathcal{N}_a \cup \widetilde{\mathcal{C}}_{a,0}) \cup \bigcup_{e\in E} P(\mathbf{x}_e,r_e) \cup \widetilde{\mathcal{C}},
    \end{align*}
    where $\mathcal{N}_{a}$ are strong $(m_a,k,\epsilon,C(\Lambda)\eta)$-neck regions for some $m_a \geq m_{\ast}$, 
    \begin{align*}
        \sup_{\mathbf{x} \in P(\mathbf{x}_e,4r_e)}N_{\mathbf{x}}(\delta^2 r_e^2) <m_\ast-\delta.
    \end{align*}
    By Lemma \ref{lemma-refined-monotonicity-of-frequency}, we thus have
    \begin{align*}
        \sup_{e\in E} \sup_{\mathbf{x}\in P(\mathbf{x}_e,4r_e)} N_{\mathbf{x}}^u(\delta^4 r_e^2) < m_\ast-1+\delta.
    \end{align*}
    For each \ref{e-balls}-ball $P(\mathbf{x}_e,r_e)$, we choose a Vitali cover by $\leq C(\epsilon \delta^2)^{-(n+2)}$ balls $\{P(\mathbf{x}_g,r_g)\}_{g\in G_e}$  of radius $r_g\coloneqq \delta^2\epsilon r_e$. Set $G\coloneqq \cup_{e \in E} G_e$. We then have
    \begin{align*}
        \sum_{a\in A}r_a^k + \sum_{g\in G}r_g^k \leq C(\Lambda) (\eta \epsilon)^{-C\Lambda},
    \end{align*}
    so the remaining claim follows from Theorem \ref{theorem-neck-structure}. 
\end{proof}

\section{Volume Estimates and Regularity}\label{volume-estimates-and-rectifiability}
In this section, we prove our main theorems on the structure of the quantitative nodal and singular sets. The volume estimates will be a consequence of a stronger estimate for quantitative strata defined below, following the strategy of \cite[Section 2.4]{cheeger-jiang-naber-2021-Sharp-quantitative}. Unlike the settings of \cite{cheeger-jiang-naber-2021-Sharp-quantitative}, we must use the finite resolution neck decomposition (Theorem \ref{theorem-neck-decomposition2}) to prove Minkowski estimates.

\begin{definition}[Quantitative stratum]\label{definition-quantative-strata}
    Given a caloric function $u$, the \textit{$(k,\epsilon,r_1,r_2)$ quantitative stratum} is defined as
    \begin{align*}
        \cS_{\epsilon, r_1,r_2}^{k}\coloneqq \cS_{\epsilon, r_1,r_2}^{k}(u)\coloneqq
    \left\{\mathbf{x}\in P(\mathbf{0},1)\colon \mathcal{E}^{k+1,1}_{r}(\mathbf{x})\ge \epsilon,\ \forall\, r\in [r_1,r_2]\right\}.
    \end{align*}
    We also set $\cS_{\epsilon,r}^{k}\coloneqq \cS_{\epsilon,r,1}^{k}$, and define $\cS_{\epsilon}^{k}\coloneqq \bigcap_{r>0}\cS_{\epsilon,r}^{k}$.
\end{definition}

\begin{remark} 
    In light of Theorem \ref{theorem-almost-frequency-cone-implies-unique-geometric-cone} and Theorem \ref{theorem-cone-splitting-inequality}, the set $\mathcal{S}_{\epsilon}^k$ roughly consists of points $\mathbf{x} \in P(\mathbf{0},1)$ such that the normalized leading term (with respect to parabolic degree) in the Taylor expansion of $u$ at $\mathbf{x}$ is $\epsilon$ away in $L^2(\mathbb{R}^n,d\nu_{-1})$-distance from the subset of caloric polynomials invariant in $k$ spatio-temporal directions (with temporal direction counted as two).
\end{remark}

The following theorem ensures that the strata $\mathcal{S}_{\epsilon}^k$ satisfy uniform estimates on the parabolic $k$-dimensional Minkowski content and are parabolically rectifiable in the following sense.
\begin{definition}[{\cite[Definition 3.7]{mattila-2022-parabolic-rectifiability}}]\label{definition-rectifiability-b}
    \begin{enumerate}
        \item A set $E\subseteq \mathbb{R}^{n+1}$ is \textit{parabolic $(k,\ell)$-rectifiable} if there are $(k,\ell)$-Lipschitz graphs $G_i$, $i\in\mathbb{N}$, such that
        \begin{align*}
            \mathcal{H}_{\mathcal{P}}^k\big(E\setminus \cup_{i} G_i\big) = 0.
        \end{align*}
        A set $E$ is \textit{horizontally} or \textit{vertically parabolic $(k,\ell)$-rectifiable} if the graphs are over $V\in \operatorname{Gr}_{\mathcal{P}}(k)$ of the form $W^k\times \{0\}$ or $W^{k-2}\times \mathbb{R}$, respectively.
        
        \item A set $E\subseteq \mathbb{R}^{n+1}$ is \textit{parabolic $k$-rectifiable} if $E\subseteq \mathbb{R}^{n+1}$ is parabolic $(k,\ell)$-rectifiable for each $\ell>0$. \textit{Horizontally} and \textit{vertically parabolic $k$-rectifiable} sets are defined analogously. 
    \end{enumerate}
\end{definition}

\begin{theorem}\label{thm: caloric vol part 1}
    Let $u$ be a caloric function such that $N(10^5)\leq \Lambda$. Then, for any $k \in \{1,\dots, n+1\}$, $\epsilon \in (0,1)$, and $r\in (0,1]$, we have
    \begin{align}
    \label{ineq: main content est}
        \cH_{\cP}^{n+2}(P(\cS_{\epsilon,r}^k,r)\cap P(\mathbf{0},1))\leq C^{\Lambda^2}\epsilon^{-C\Lambda^3} r^{n+2-k}.
    \end{align}
    Moreover, $\cS_\epsilon^k$ is parabolic $k$-rectifiable with $\mathcal{H}^k_{\mathcal{P}}(\cS_\epsilon^k\cap P(\mathbf{0},1))\le C(\epsilon)^{\Lambda^4}$. 
    Furthermore, if $k=n+1$ or $n$, then, for $\mathcal{H}^1$-almost every time $t\in(-1,1)$, $\cS^{k,t}_{\epsilon}\coloneqq \cS^k_{\epsilon}\cap (\R^n\times\{t\})$ is $(k-2)$-rectifiable with
    \begin{align}
        \mathcal{H}^{k-2}(\cS^{k,t}_{\epsilon})\le C^{\Lambda^2}\epsilon^{-C\Lambda^3},
        \label{ineq: main content est time slice} 
    \end{align}
    and for $\mathcal{H}^{\frac{1}{2}}$-almost every time $t\in (-1,1)$, $\mathcal{H}^{k-1}(\mathcal{S}_{\epsilon}^{k,t})=0$.
\end{theorem}
We defer the proof to Section \ref{proof of main theorem quantitative strata}.

\subsection{\texorpdfstring{$\epsilon$}{epsilon}-regularity for nodal and singular sets}\label{epsilon regularity}
In this subsection, using the cone-splitting inequality, we show that the quantitative nodal and singular sets are contained in the quantitative $(n+1)$-stratum and $n$-stratum, respectively. Moreover, they are contained in super-level sets of the frequency function. Therefore, estimating the volume of nodal and singular sets amounts to proving Theorem \ref{thm: caloric vol part 1}. The proof methods in the section are similar to those in \cite[Section 3.4]{naber-valtorta-2017-volume-estimtates-of-critical-sets-of-pde}.

Recall that $\mathscr{C}_r^k\coloneqq Z_r$ or $S_r$ and $\mathscr{C}^k\coloneqq Z$ or $S$ depending on whether $k=n+1$ or $k=n$.
\begin{proposition}\label{proposition-containment}
    Let $k\in \{n,n+1\}$. Suppose $u$ is a caloric function with $N(10^5r^2)\leq \Lambda$. If $\epsilon\le C^{-\Lambda}$, then, for any $r\in(0,1)$,
    \begin{align}
        \mathscr{C}^k_r(u)\subseteq \cS_{\epsilon,r}^{k}(u).
        \label{eq-inclusion-of-nodal-and-singular-set-in-quantitative-stratum}
    \end{align}
\end{proposition}

In the following two lemmata, we prove \textit{$\epsilon$-regularity} based on the existence of abundant symmetries.
\begin{lemma}\label{lemma-containment-of-nodal-sets}
    Suppose $u$ is spatially $(n, \epsilon, r)$-symmetric at $\mathbf{x}_0$ with $\epsilon\le\bar\epsilon$. Then 
    \begin{align}
        \inf_{P(\mathbf{x}_0,\tfrac{r}{16})} u^2> \frac{1}{8}\int_{\R^n} u^2 \, d\nu_{\mathbf{x}_0;t_0-r^2}.
    \end{align}
\end{lemma}

\begin{proof}
    After a space-time translation and parabolic scaling, we assume that $\mathbf{x}_0=\mathbf{0}$ and $r=1$. Additionally, we impose the normalization $\int_{\R^n} u^2 \, d\nu_{-1}=1$.

    Define $\ol{u} \coloneqq \int_{\R^n} u \, d\nu_{-1}$. Note that $\ol{u}$ is a constant. By applying the interior estimates in Lemma \ref{lemma-interior-estimates} and the Poincaré inequality, we have  
    \begin{align*}  
        \sup_{\mathbf{x} \in P(\mathbf{0}, \frac{1}{16})} \verts{u - \ol{u}}^2 \leq C\int_{\R^n} \verts{u - \ol{u}}^2 \, d\nu_{-1} \leq C \int_{\R^n} \verts{\cd u}^2 \, d\nu_{-1} < C\epsilon< \frac{1}{4},  
    \end{align*}  
    provided $\epsilon<\frac{1}{4C}$.
    Furthermore,  
    \begin{align*}  
        \ol{u}^2 = \int_{\R^n} \ol{u}^2 \, d\nu_{-1} = \int_{\R^n} u^2 \, d\nu_{-1} - \int_{\R^n} \verts{u - \ol{u}}^2 \, d\nu_{-1} > \frac{3}{4}.  
    \end{align*}
    Combining these estimates, we get the conclusion.
\end{proof}

\begin{lemma}\label{lemma-containment-of-singular-sets}
    Suppose $u$ is a caloric function with $N(10^5 r^2)\leq \Lambda$. If $u$ is $(n+1,\epsilon,2r)$-symmetric at $\mathbf{x}_0$ with $\epsilon\le C^{-\Lambda}$, then
    \begin{align*}
        \inf_{P(\mathbf{x}_0,\tfrac{r}{16})} \verts{u}^2+ 2r^2\verts{\cd u}^2 > \frac{1}{16} \int_{\R^n}u^2 \, d\nu_{\mathbf{x}_0;t_0-r^2}.
    \end{align*}
\end{lemma}

\begin{proof}
    After a space-time translation and parabolic scaling, we assume that $\mathbf{x}_0=\mathbf{0},$ $r=1$, and we normalize $\int_{\R^n} u^2 \, d\nu_{-1}=1$. Suppose  
    $u$ is $(n+1,\epsilon, 2)$-symmetric with respect to $V=L\times \mathbb{R}$ for some $(n-1)$-plane $L$. We may suppose $L$ to be the span of the first $n-1$ coordinates of $\R^n$.

    Set $\ol{u}\coloneqq \int_{\R^n} u^2 \, d\nu_{t}$, $\ol{\pd_iu}\coloneqq \int_{\R^n} \pd_iu \, d\nu_{t}$, and $\ol{\cd u}\coloneqq (\ol{\pd_1 u},\dots,\ol{\pd_n u})$. By the Poincar\'{e} inequality and the computation in \eqref{eq n+1 symmetry implies hessian is small} we have
    \begin{align*}
        \int_{\R^n} \verts{\cd u-\ol{\cd u}}^2 \, d\nu_{-1}\leq\ 2\int_{\R^n} \verts{\cd^2 u}^2\, d\nu_{-1}<C^{\Lambda}\epsilon.
    \end{align*}
    
    Now, using the interior estimates in Lemma \ref{lemma-interior-estimates},
    \begin{align*}
        \sup_{\mathbf{x}\in P(\mathbf{0},\frac{1}{16})} \verts{\cd u-\ol{\cd u}}^2&<C^\Lambda\epsilon,\\
        \sup_{\mathbf{x}\in P(\mathbf{0},\frac{1}{16})} \verts{u-\ol{u}-\ol{\cd u}\cdot x}^2&\leq C\int_{\R^n} \verts{u-\ol{u}-\ol{\cd u}\cdot x}^2 \, d\nu_{-1}\leq C\int_{\R^n} \verts{\cd u-\ol{\cd u}}^2 \, d\nu_{-1} <C^\Lambda \epsilon.
    \end{align*}
    Applying the Poincar\'e inequality and the above computations, if $\epsilon\le C^{-\Lambda}$, then
    \begin{align*}
        \ol{u}^2 +2\verts{\ol{\cd u}}^2 =\int_{\R^n} |\ol{u}+\ol{\nabla u}\cdot x|^2\,d\nu_{-1}\ge \frac{1}{2}\int_{\R^n} u^2\,d\nu_{-1}
        - \int_{\R^n} |u-(\ol{u}+\ol{\nabla u}\cdot x)|^2\, d\nu_{-1}
        > \frac{7}{16}.
    \end{align*}
    It follows that
    \begin{align*}
        \inf_{\mathbf{x}\in P(\mathbf{0},\frac{1}{16})}\verts{\ol{u}+\ol{\cd u}\cdot x}^2 +2\verts{\ol{\cd u}}^2
        \ge \frac{1}{2}\ol{u}^2 +(2- 2^{-8})|\ol{\nabla u}|^2 > \frac{7}{32}.
    \end{align*}
    Therefore, if $\epsilon\le C^{-\Lambda}$, 
    \begin{align*}
        \inf_{\mathbf{x}\in P(\mathbf{0},\frac{1}{16})} \verts{u}^2+ 2\verts{\cd u}^2 \ge \frac{7}{64}-C\epsilon > \frac{1}{16}.
    \end{align*}
\end{proof}

\begin{proof}[{Proof of Proposition \ref{proposition-containment}}]
    Suppose $\mathbf{x}\notin \cS_{\epsilon,r}^{n+1}(u)$, so that $\mathcal{E}^{n+2,1}_{s}(\mathbf{x})<\epsilon$ for some $s\in [r,1]$. By Theorem \ref{thm: sym split equiv}\ref{thm: sym split equiv-1}, $u$ is $(n+2,C^\Lambda\epsilon, s)$-symmetric at $\mathbf{x}$. By Lemma \ref{lemma-containment-of-nodal-sets}, $\mathbf{x}\notin Z_r(u)$ if $\epsilon\le C^{-\Lambda}$. We also have $S_r(u)\subseteq \mathcal{S}^{n}_{\epsilon, r}(u)$ by the same argument but using Lemma \ref{lemma-containment-of-singular-sets}.
\end{proof}

In the following lemma, we prove a second version of $\epsilon$-regularity that proves the containment of nodal and singular sets in the super-level sets of the frequency function.
\begin{lemma} \label{lemma pinching implies nodal}
    Let $k\in \{n,n+1\}$. The following holds if $\epsilon \leq C^{-(\Lambda+1)}$. Suppose $u$ is a caloric function with $N_{\mathbf{0}}(10^5)\leq \Lambda$. Then, for any $r\in (0,\epsilon]$, we have
    \begin{align*}
        \mathscr{C}_r^k(u)\subset \left\{\mathbf{x} \in P(\mathbf{0},1) \colon N_{\mathbf{x}}(\epsilon^{-2} r^2)\geq m_\ast-\frac{1}{2}\right\},
    \end{align*}
    where $m_\ast$ is defined as in \eqref{eq def of mstar}.
\end{lemma}

\begin{proof}
    By translation, we may assume $\mathbf{x}=\mathbf{0}$. Suppose $N_{\mathbf{0}}(\epsilon^{-2}r^2) <\frac{3}{2}$. By Lemma \ref{lem: pinched scale}, there exists $s\in [4r,\epsilon^{-\frac{1}{20}}r]$ such that $|N_{\mathbf{x}}(4s^2) - N_{\mathbf{x}}(16s^2)|< \epsilon^{\frac{1}{20}}$. By Theorem \ref{theorem-almost-frequency-cone-implies-unique-geometric-cone} with $\tau_1 \leftarrow \frac{s}{2}$, $\tau_2 \leftarrow 2s$, $\tau \leftarrow s$, $\delta \leftarrow  \epsilon^{\frac{1}{20}}$ to obtain
    \begin{align*}
        s^{2(j+2\ell)}\int_{\R^n} \verts{\pdt^{\ell}\cd^{j} (u-p_m)}^2\,d\nu_{-s^2}\leq \epsilon^{\frac{1}{20}} \int_{\R^n} u^2\,d\nu_{-16s^2} \leq C^{\Lambda}\epsilon^{\frac{1}{20}} \int_{\mathbb{R}^n} u^2 d\nu_{-s^2}
    \end{align*}
    for $j+\ell \leq 1$, where $p_m \in \mathcal{P}_m$ and $m\leq 1$. If $m=1$, then $p_1$ is $(n+1,0,s)$-symmetric at $\mathbf{0}$, hence $u$ is $(n+1,C^{\Lambda}\epsilon^{\frac{1}{20}},s)$-symmetric at $\mathbf{0}$. By Lemma \ref{lemma-containment-of-singular-sets}, if $\epsilon \leq \overline{\epsilon}(\Lambda)$, then $\mathbf{0} \notin S_r(u)$. If $m=0$, then $p_0$ is $(n+2,0,s)$-symmetric at $\mathbf{0}$, so $u$ is $(n+2,C^{\Lambda}\epsilon^{\frac{1}{20}},s)$-symmetric at $\mathbf{0}$. By Lemma \ref{lemma-containment-of-nodal-sets}, we have $\mathbf{0} \notin Z_r(u)$. 
\end{proof}

\subsection{Quantitative dimension reduction}\label{quantitative dimension reduction}
In the following proposition we prove that if $u$ is $k$-symmetric with respect to a plane $V$, then it is $(k+1)$-symmetric at points away from $V$, but only at a smaller scale. We will use the result in Lemma \ref{lem: cover large a-balls} in the proof of the volume estimate.
\begin{proposition}\label{prop: extra symmetry}
    For any $\epsilon, \eta \in (0,1]$, the following holds if $\delta \leq \eta^{C\Lambda}\epsilon^{C\Lambda^2}$. Let $u$ be a caloric function such that for some $\mathbf{x}_0 \in \mathbb{R}^{n+1}$ and $r>0$, 
    \begin{align*}
        N_{\mathbf{x}_0}^u(10^{10n} r^2) \leq \Lambda,\qquad \inf_{\tau \in [10^2r^2,10^{10n}]}\mathcal{E}_{\sqrt{\tau}}(\mathbf{x}_0)<\delta.
    \end{align*}
    Suppose $u$ is $(k,\delta,10r)$-symmetric at $\mathbf{x}_0$ with respect to some $V\in{\rm Gr}_{\mathcal{P}}(k)$. Then, for any $\mathbf{y}\in P(\mathbf{x}_0,2r)\setminus P(\mathbf{x}_0+V,\eta r)$, there exists $\beta=\beta(\mathbf{y})\ge \eta^2\epsilon^{C\Lambda}$ such that
    \begin{align*} \mathcal{E}^{k+1,1}_{\beta r}(\mathbf{y})<\epsilon, \qquad |N_{\mathbf{y}}^u(10^5 \beta^2)-N_{\mathbf{y}}^u(10^{-5}\beta^2)|<\epsilon^2.        
    \end{align*}
\end{proposition}

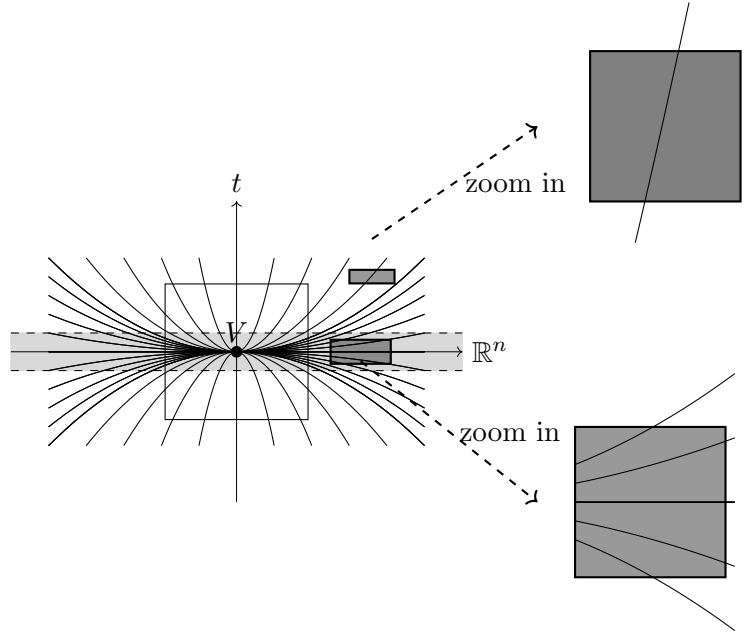
\begin{figure}[!htb]
    \centering
        \begin{tikzpicture}
        \foreach \x in {0.25}{
        \filldraw[gray!30] (-3,-\x) rectangle (3,\x);
        \draw[dashed] (-3,-\x) to (3,-\x) (-3,\x) to (3,\x);
        } 
        
        \foreach \x in {0.95}{
        \draw (-\x,-\x*\x) rectangle (\x,\x*\x);
        } 
    
        \begin{scope}[shift={(1.8,1)}]
            \foreach \x in {0.3}{
            \filldraw[fill=gray!80, thick] (-\x,-\x*\x) rectangle (\x,\x*\x);}
        \end{scope} 

        \begin{scope}[shift={(1.65,0)}]
            \foreach \x in {0.4}{
            \filldraw[fill=gray!90, thick] (-\x,-\x*\x) rectangle (\x,\x*\x);}
        \end{scope} 
    
        \foreach \x in {1,2,...,5} 
        \foreach \y in {5,-5} {
        \draw[domain=-0.5:0.5,smooth, variable=\t] plot({\x*\t},{\y*\t*\t});} 
    
        \foreach \x in {0,1,...,5} 
        \foreach \y in {5,-5} {
        \draw[domain=-0.5:0.5,smooth, variable=\t] plot ({\y*\t},{-\x*\t*\t}) plot({\y*\t},{\x*\t*\t});
        } 
    
        \draw[->, ultra thin] (0,-2) to (0,2) node[above] {$t$};
        \draw[->, ultra thin] (-3,0) to (3,0) node[right] {$\R^n$}; 
    
        \node[above] at (0,0) {$V$}; 
        \filldraw circle (2pt);
    
        \draw[->, dashed, thick] (1.8,1.5) to node[midway,right] {zoom in} (4,3); 
        
        \begin{scope}[shift={(-1.5,-13)}]
        \begin{scope}[shift={(1.8/0.25,1/0.0625)}]
            \foreach \x in {1}{
            \filldraw[fill=gray!100, thick] (-\x,-\x*\x) rectangle (\x,\x*\x);}
        \end{scope}
        \foreach \x in {4} 
        \foreach \y in {5} {
        \draw[domain=0.425:0.47,smooth, variable=\t] plot({\x*\t/0.25},{\y*\t*\t/0.0625});}
        \end{scope} 
    
        \draw[->, dashed, thick] (1.65,-0.125) to node[midway,right] {zoom in} (4,-2);  
        
        \begin{scope}[shift={(2,-2)}]
        \begin{scope}[shift={(1.4/0.4,0)}]
            \foreach \x in {1}{
            \filldraw[fill=gray!80,thick] (-\x,-\x*\x) rectangle (\x,\x*\x);}
        \end{scope}
        \foreach \x in {0,1,...,2} 
        \foreach \y in {5} {
        \draw[domain=0.2:0.37,smooth, variable=\t] plot({\y*\t/0.4},{-\x*\t*\t/0.16}) plot({\y*\t/0.4},{\x*\t*\t/0.16});}
        \end{scope} 
    \end{tikzpicture}
    \caption{Suppose the plane of symmetry is $V=L^k\times \{0\}$. Bottom right: If we are close to the $t=0$ time slice but away from $V$, then the characteristic curves appear horizontal under parabolic scaling. Top right: If we are away from $t=0$ and $V$, then the characteristic curves appear vertical under parabolic scaling.}
    \label{fig:quantitative dimension reduction}
\end{figure}

\begin{proof}
    By parabolic rescaling and translation, we may assume that $\mathbf{x}_0=\mathbf{0}$ and $r=1$. Let $\mathbf{y}=(y,s)$. For simplicity, we assume that $V=L^k\times \{0\}$, as the other case is similar. We may also assume without loss of generality that $\epsilon \leq C^{-\Lambda}$. 

    First suppose $\verts{s}\leq \eta^2\epsilon^{C \Lambda}$ so that $\verts{y}\geq \eta$. By Lemma \ref{lem: pinched scale}, there exists $\beta\in[C^{-\Lambda}\eta\epsilon^{4\Lambda+12},C^{-\Lambda}\eta\epsilon^2]$ such that $N_{\mathbf{y}}^u(10^{5}\beta^2)-N_{\mathbf{y}}^u(10^{-5}\beta^2)<\epsilon^2$. By Lemma \ref{lemma-refined-monotonicity-of-frequency}\ref{lemma: pinching int-2}, there exist positive integers $m,m_{\mathbf{y}} \leq C(\Lambda)$ and some $\tau_{\ast} \in [10^2,10^{10n}]$ such that 
    \begin{align} \label{eq:integersfordimred}
        \sup_{s\in [10^{-4}\beta^2,10^4 \beta^2]}|N_{\mathbf{y}}^u(s)-m_{\mathbf{y}}|<6\epsilon^2, \qquad \sup_{s\in [4^{-1}\tau_{\ast},4\tau_{\ast}]} |N^u(s)-m|<6\delta^2.
    \end{align}
    Set
    \begin{align*}
        v\coloneqq 2\tau \Delta_f u+mu, \qquad v_{\mathbf{y}}\coloneqq 2\tau_{\mathbf{y}} \Delta_{f_{\mathbf{y}}}u+m_{\mathbf{y}}u,
    \end{align*} 
    where $\tau_{\mathbf{y}}= s-t$ for $t<s$. For any $x\in \mathbb{R}^n$, we then have 
    \begin{align*}
        v(x,s-\beta^2)&=2(\beta^2-s) \Delta u(x,s-\beta^2)-\cd u(x,s-\beta^2)\cdot x+mu(x,s-\beta^2),\\
        v_{\mathbf{y}}(x,s-\beta^2)&=2\beta \Delta u(x,s-\beta^2)-\cd u(x,s-\beta^2)\cdot (x-y)+m_{\mathbf{y}}u(x,s-\beta^2).
    \end{align*}
    Combining equations yields
    \begin{align}
        \cd u(x,s-\beta^2) \cdot y&=v_{\mathbf{y}}(x,s-\beta^2)-v(x,s-\beta^2)-2s\Delta u(x,s-\beta^2)+(m-m_{\mathbf{y}})u(x,s-\beta^2).\label{eq extra symmetry}
    \end{align} 
    By \eqref{eq:integersfordimred} and Lemma \ref{lemma-detecing-homogeneity}, we have
    \begin{align}\label{eq:dimredestaty}
        \begin{split}
            \int_{\R^n} v_{\mathbf{y}}^2 \,d\nu_{\mathbf{y};s-\beta^2}&\leq C^{\Lambda} \epsilon^2 \int_{\R^n} u^2 \,d\nu_{\mathbf{y};s-\beta^2}, \\
            \int_{\R^n} v^2 \,d\nu_{-\tau_\ast}&\leq C^{\Lambda}\delta \int_{\R^n} u^2 \,d\nu_{-\tau_\ast}.
         \end{split}
    \end{align}
    By Lemma \ref{lemma-frequency-uniform-bound}, $N_{\mathbf{y}}^u(100)\le C\Lambda$. Using the monotonicity (Lemma \ref{lemma-monotonicity-formulae-for-the-energy-functionals}) and comparison of $L^2$-norms at different base points (Lemma \ref{lemma-comparison-of-caloric-energy} with $\theta \leftarrow \frac{1}{4},\tau \leftarrow 24$), we know that
    \begin{align} \label{eq:dimredestat0}
    \begin{split}
        \int_{\R^n} v^2 \,d\nu_{\mathbf{y};s-{\beta^2}}\leq& \int_{\R^n} v^2 \,d\nu_{\mathbf{y};s-24}
        \leq C\int_{\R^n} v^2 \,d\nu_{-100}\le\  C^{\Lambda} \delta \int_{\R^n} u^2\,d\nu_{-100}\le C^{\Lambda} \delta\int_{\R^n} u^2\,d\nu_{\mathbf{y};s-200}\\
        \le&\  \left(\frac{C}{\beta}\right)^{C\Lambda} \delta\int_{\R^n} u^2\,d\nu_{\mathbf{y};s-\beta^2}
        \le \frac{\delta}{\eta^{C\Lambda}\epsilon^{C\Lambda^2}}\int_{\R^n} u^2 \,d\nu_{\mathbf{y};s-{\beta^2}}.
    \end{split}
    \end{align}
    On the other hand, by $\beta^4 \geq C^{-\Lambda}\eta^4\epsilon^{20\Lambda}$, $|s|\leq \epsilon^{C\Lambda}\eta^2$, and the computation in Lemma \ref{lemma-interior-estimates} we know that
    \begin{align*}
        \int_{\R^n} s^2 \verts{\Delta u}^2\,d\nu_{\mathbf{y};s-\beta^2}\leq \frac{C^{\Lambda}s^2}{\beta^4} \int_{\R^n} u^2 \,d\nu_{\mathbf{y};s-\beta^2} \leq \epsilon^{C\Lambda}\int_{\mathbb{R}^n} u^2 d\nu_{\mathbf{y};s-\beta^2}.
    \end{align*}
    Combining this with \eqref{eq extra symmetry}, \eqref{eq:dimredestaty}, and \eqref{eq:dimredestat0} yields
    \begin{align*}
        \beta^2 \int_{\R^n} \verts{\cd u\cdot y}^2\,d\nu_{\mathbf{y};s-\beta^2}&\leq C^{-2\Lambda}\epsilon^4 \eta^2 \left( C^{\Lambda} \epsilon^2+ \frac{\delta}{\epsilon^{C\Lambda^2}}+ \epsilon^{C\Lambda} +2\Lambda^2 
        \right) \int_{\R^n} u^2 \,d\nu_{\mathbf{y};s-\beta^2}\\
        &\leq C^{-\Lambda} \epsilon^2 \eta^2\int_{\R^n} u^2\, d\nu_{\mathbf{y};s-\beta^2}
    \end{align*}
    if $\delta \leq \epsilon^{2C\Lambda^2}$. 
    Because $\verts{\pi_L^{\perp}y}\geq \eta$, it follows that $u$ is $(1,C^{-\Lambda}\epsilon^2,\beta)$-symmetric at $\mathbf{y}$ with respect to $\operatorname{span}(\pi_L^{\perp}y)$.
    Because $\partial_t u$ and the components of $\nabla u$ are caloric, we can apply Lemma \ref{lemma-comparison-of-caloric-energy} and then Lemma \ref{lemma: propagation of symmetry and pinching}\ref{propagation of symmetry} to conclude that $u$ is also $(k,\epsilon^{-C\Lambda^2}\delta,\beta)$-symmetric with respect to $V$ at $\mathbf{y}$. It follows that $u$ is $(k+1,C^{-\Lambda}\epsilon^2,\beta)$-symmetric with respect to $\mbox{span}(V,\mathbf{y})$, so by Theorem \ref{thm: sym split equiv}\ref{thm: sym split equiv-2}, we have $\mathcal{E}_{\frac{\beta}{10}}^{k+1,1}(\mathbf{y})< \epsilon$.

    Suppose instead that $\verts{s}\geq \epsilon^{C\Lambda}\eta^2$. Choose $\beta'\in[\epsilon^{4C\Lambda}\eta^2,\epsilon^{3C\Lambda}\eta^2]$ according to  Lemma \ref{lem: pinched scale}, so that $N_{\mathbf{y}}(10^{5}\beta'^2)-N_{\mathbf{y}}(10^{-5}\beta'^2)<\epsilon^2$. By Lemma \ref{lemma-refined-monotonicity-of-frequency}\ref{lemma: pinching int-2}, there exists a nonnegative integer $m_{\mathbf{y}}' \leq C(\Lambda)$ such that $\sup_{s'\in [10^{-4}(\beta')^2,10^4(\beta')^2]}|N_{\mathbf{y}}^u(s')-m_{\mathbf{y}}'|<\epsilon^2$. By Lemma \ref{lemma-detecing-homogeneity}, we then have $(x,s-(\beta')^2)$ we have
    \begin{align} \label{eq extra symmetry temporal}
        2s\pdt u(x,s-(\beta')^2)=(v_{\mathbf{y}}-v)(x,s-(\beta')^2)-\cd u(x,s-(\beta')^2) \cdot y+(m-m_{\mathbf{y}}')u(x,s-(\beta')^2)
    \end{align} 
    for all $x\in \mathbb{R}^n$. Again appealing to the proof of Lemma \ref{lemma-interior-estimates}, we get
    \begin{align*}
        (\beta')^2 \int_{\mathbb{R}^n} |\nabla u\cdot y|^2 d\nu_{\mathbf{y};s-(\beta')^2} \leq C^{\Lambda} \int_{\mathbb{R}^n} u^2 d\nu_{\mathbf{y};s-(\beta')^2},
    \end{align*}
    so we can combine this with \eqref{eq extra symmetry temporal}, \eqref{eq:dimredestaty}, and \eqref{eq:dimredestat0} to obtain
    \begin{align*}
        (\beta')^4 \int_{\mathbb{R}^n} \verts{\pdt u}^2 \, d\nu_{\mathbf{y};s-(\beta')^2}&\leq \frac{(\beta')^4}{s^2} \left( C^\Lambda \epsilon^2+ \frac{\delta}{\eta^{C\Lambda}\epsilon^{C\Lambda^2}}+\frac{C^\Lambda}{(\beta')^2}+2\Lambda^2\right)
        \int_{\R^n} u^2 \, d\nu_{\mathbf{y};s-(\beta')^2}
        \\
        &\leq C^{-\Lambda}\epsilon^2 \int_{\R^n}  u^2 \, d\nu_{\mathbf{y};s-(\beta')^2}.
    \end{align*}
    We can therefore argue as above to conclude that $u$ is $(k+1,C^{-\Lambda}\epsilon^2,\beta')$-symmetric with respect to $L \times \mathbb{R}$ at $\mathbf{y}$, hence $\mathcal{E}_{\beta'}^{k+1,1}(\mathbf{y})<\epsilon$.
\end{proof}

\subsection{Proof of main theorems in the caloric setting}\label{proof of main theorem quantitative strata}
\begin{proof}[Proof of Theorem \ref{thm: caloric vol part 1}]
    We first prove the content estimate \eqref{ineq: main content est}. Given $\epsilon>0$, let $\eta\le\bar\eta(\Lambda,\epsilon)$ and $\delta\le\bar\delta(\Lambda,\epsilon)$ be small constants to be determined later. 
    
    By the finite-resolution neck decomposition Theorem \ref{theorem-neck-decomposition2} with $r_\ast \leftarrow r$, $\epsilon \leftarrow \delta$, and $\eta\leftarrow \eta$, we have
    \begin{equation}
    \label{eq: cover by ab balls discrete}
        P(\mathbf{0},1)\subseteq \bigcup_a \mathcal{N}^a \cup \bigcup_b P(\mathbf{x}_b,r_b) \cup \bigcup_f P(\mathbf{x}_f,r_f),
    \end{equation}
    where $\mathcal{N}^a=P(\mathbf{x}_a,2r_a)\setminus \overline{P}(\mathcal{C}^a,\mathbf{r}^a)$ is an $(m_a,k,\delta, C^{-\Lambda}\eta)$-neck region for some $m_a\le C\Lambda$, and $\mathcal{E}^{k+1,1}_{100r_b}(\mathbf{y}_b)\le\eta$ for some $\mathbf{y}_b\in P(\mathbf{x}_b,4r_b)$, where we recall that $\gamma=10^{-10n}$. In addition,
    \begin{align}
        \sum_a r_a^k + \sum_b r_b^k+ \sum_f r_f^k \leq  C^{\Lambda^2} (\delta \eta)^{-C\Lambda},\label{eq neck decomposition for volume estimate}
    \end{align}
    $\mathbf{r}^a,r_a,r_b\geq r$ and $r_f\in [r, C^{\Lambda} \eta^{-\frac{1}{n}}\delta^{-10n^2}r]$.
    
    In the lemmata below, we estimate the contribution from neck regions and  \ref{b-balls}-balls with large radius $r_a$ and $r_b$. Because of the multiscale symmetry of a neck region defined in a ball $P(\mathbf{x}_a,r_a)$, we will see that the $k$-th strata won't be far away from the corresponding center set.
    \begin{lemma}\label{lem: cover large a-balls}
        For any $\epsilon>0$, if $m\le\Lambda$, $\delta\le \epsilon^{C\Lambda^2}$, $\chi\le \epsilon^{C\Lambda}$, and $\mathcal{N}^a\subseteq P(\mathbf{x}_a,r_a)$ is an $(m_a,k,\delta, C^{-\Lambda}\eta)$-neck region with $r_a\ge \chi^{-1}r$, then
        \begin{align*}
        \mathcal{H}^{n+2}_{\mathcal{P}} \left(P(\mathcal{S}^{k}_{\epsilon,r}\cap \mathcal{N}^a,r )\right) 
            & \le C\epsilon^{-C\Lambda}r_a^k r^{n+2-k}.
        \end{align*}
    \end{lemma}
    
    \begin{proof}
        We first prove that $k$-th strata is not that far from the center set.
        \begin{claim}\label{claim: neck to center}
            If $\mathbf{y}\in \mathcal{S}^{k}_{\epsilon,r}\cap \mathcal{N}^a$, then $d_{\mathcal{P}}(\mathbf{y},\mathcal{C}_a) < \chi^{-1}r$, where $\chi= \epsilon^{C\Lambda}$.
        \end{claim}
        
        \begin{proof}   
            Let $d\coloneqq d_{\cP}(\mathbf{y},\mathcal{C}_a)=|\mathbf{x}-\mathbf{y}|$ for some $\mathbf{x}\in \mathcal{C}_a$. By way of contradiction assume $d\ge \chi^{-1}r$. Note that $|\mathbf{y}-\mathbf{x}'|\ge \mathbf{r}_{\mathbf{x}'}$ for any $\mathbf{x}'\in \mathcal{C}$, since $\mathbf{y}\in\mathcal{N}^a$.
            
            Since $d\ge \mathbf{r}_{\mathbf{x}}$,
            by \ref{neck-k-sym},
            $u$ is $(k,\delta,d)$-symmetric {at $\mathbf{x}$} with respect to $V_a\in {\rm Gr}(k)$. We claim that $\mathbf{y}\notin P(\mathbf{x}+V_a,\frac{1}{2}d)$.
            Let $\mathbf{z}\in \mathbf{x}+V_a$ be the closest point to $\mathbf{y}$, and suppose by way of contradiction that $|\mathbf{y}-\mathbf{z}|< \tfrac{d}{2}$. By \ref{neck-Hausdorff}, there exists $\mathbf{x}'\in \mathcal{C}^a$ such that $\mathbf{z}\in P(\mathbf{x}',10\gamma (d+\mathbf{r}_{\mathbf{x}'}))$.
            Then
            \begin{align*}
                \mathbf{r}_{\mathbf{x}'}\le |\mathbf{y}-\mathbf{x}'|\le |\mathbf{y}-\mathbf{z}|+|\mathbf{z}-\mathbf{x}'|< \tfrac{1}{2}d + 10\gamma(d+\mathbf{r}_{\mathbf{x}'}),
            \end{align*}
            which implies $\mathbf{r}_{\mathbf{x}'} < \tfrac{3}{4}d$. However,
            \begin{align*}
                d\le |\mathbf{y}-\mathbf{x}'| < \tfrac{1}{2}d + 10\gamma(d+\mathbf{r}_{\mathbf{x}'})
                < \tfrac{1}{2}d + 20\gamma d < d,
            \end{align*}
            which is a contradiction. Thus, $\mathbf{y}\notin P(\mathbf{x}+V_a,\frac{1}{2}d)$.
            
            By Proposition \ref{prop: extra symmetry} with $r\leftarrow d$, there is $\underline{\beta}=\epsilon^{C\Lambda}$ and $\beta\ge \underline{\beta}$ such that $\mathcal{E}^{k+1,1}_{\beta d}(\mathbf{y})<\epsilon$ if $\delta \leq \epsilon^{C\Lambda^2}$ which is a contradiction to $\mathbf{y}\in \mathcal{S}^k_{\epsilon,r}$ if we set $\chi=\underline{\beta}$.
        \end{proof}
        
        For each $a$, let $\{\mathbf{y}_{a,j}\}_{j=1}^{K_a}\subseteq P(\mathbf{x}_a,r_a)$ be a maximal subset of $\mathcal{C}_a$ with $|\mathbf{y}_{a,i}-\mathbf{y}_{a,j}|\geq 2r$ for $1\leq i\neq j\leq K_a$. By Theorem \ref{theorem-neck-structure}, if $\delta \leq \overline{\delta}$, then there are $V_a \in \text{Gr}_{\mathcal{P}}(k)$ such that the projection $\pi_{V_a}: \mathbb{R}^n \times \mathbb{R} \to V_a$ satisfies $r \leq |\pi_{V_a}(\mathbf{y}_{a,i})-\pi_{V_a}(\mathbf{y}_{a,j})|$ for $1\leq i \neq j \leq K_a$, and such that $|\pi_{V_a}(\mathbf{y}_{a,j})-\pi_{V_a}(\mathbf{x}_a)|\leq r_a$ for $1\leq j\leq K_a$. Thus,
        \begin{align*}
            cr^k K_a \leq \sum_{j=1}^{K_a} \mathcal{H}_{\mathcal{P}}^k \left(V_a \cap P(\pi_{V_a}(\mathbf{y}_{a,j}),\tfrac{r}{2}) \right) \leq \mathcal{H}_{\mathcal{P}}^k(V_a\cap P(\pi_{V_a}(\mathbf{x}_a),r_a)) \leq Cr_a^k,
        \end{align*}  
        so that $K_a\le Cr_a^kr^{-k}$. From $\mathcal{C}_a \subseteq \bigcup_{j=1}^{K_a} P(\mathbf{y}_{a,j},\frac{1}{2}\chi^{-1}r)$ and Claim \ref{claim: neck to center}, 
        \begin{align*}
            \mathcal{S}^{k}_{\epsilon,r}\cap \mathcal{N}^a
            \subseteq P(\mathcal{C}^a,\chi^{-1}r)
            \subseteq \cup_{j} P(\mathbf{y}_j,\tfrac{3}{2}\chi^{-1}r).
        \end{align*}
        Thus,
        \begin{align*}
            \mathcal{H}^{n+2}_{\mathcal{P}}\left(P (\mathcal{S}^{k}_{\epsilon,r}\cap \mathcal{N}^a,r )\right) 
            \le \sum_{j=1}^{K_a} \mathcal{H}^{n+2}_{\mathcal{P}}(P(\mathbf{y}_j,2\chi^{-1}r))
            \le C\chi^{-n-2} r_a^kr^{n+2-k}.
        \end{align*}
    \end{proof} 

    For the rest of the proof, we set $\delta= \epsilon^{C\Lambda^2}$.

    \begin{lemma}\label{lem: strata b-ball}
        For any $\epsilon \in (0,\frac{1}{10}]$, if $\chi <\epsilon^{C\Lambda}$, $\eta \leq C^{-\Lambda} \epsilon^{C\Lambda^2}$,
        $P(\mathbf{x}_b,r_b)$ is a \ref{b-balls}-ball in the sense that $\mathcal{E}^{k+1,1}_{100r_b}(\mathbf{y}_b)\le\eta$ for some $\mathbf{y}_b\in P(\mathbf{x}_b,4r_b)$, and $r_b\ge \chi^{-1} r$, then
        \begin{align*}
            \mathcal{S}^k_{\epsilon,r}\cap P\left(\mathbf{x}_b,2r_b\right) = \emptyset.
        \end{align*}
    \end{lemma}

    \begin{proof}
        It suffices to show that for any $\mathbf{y}\in P(\mathbf{x}_b,4r_b)$, $\mathcal{E}^{k+1,1}_{s}(\mathbf{y})<\frac{\epsilon}{2}$ for some $s\ge \chi r_b$. By Lemma \ref{lem: pinched scale}, 
        if $\chi<\epsilon^{8\Lambda +20}$, then 
        \begin{align}
            N_{\mathbf{y}}(\epsilon^{-4}s^2)-N_{\mathbf{y}}(\epsilon^{4}s^2)<\epsilon^2\label{eq- pinched frequency}
        \end{align}
        for some $s\in[\frac{\chi r_b}{100}, \frac{r_b}{100}]$. Because $\mathcal{E}_{100 r_b}^{k+1,1}(\mathbf{y}_b) < \eta$, Theorem \ref{thm: sym split equiv}\ref{thm: sym split equiv-1} implies that $u$ is $(k+1,C^{\Lambda}\eta,100 r_b)$-symmetric at $\mathbf{y}_b\in P(\mathbf{x}_b,2r_b)$. By Lemma \ref{lemma-comparison-of-caloric-energy} applied to $\pdt u$ and components of $\cd u$, $u$ is $(k+1,C^{\Lambda} \eta,r_b)$-symmetric at $\mathbf{y}$. Taking $\eta\leq C^{-\Lambda} \chi^{2\Lambda} \epsilon^2$,  Lemma \ref{lemma: propagation of symmetry and pinching}\ref{propagation of symmetry} implies $u$ is $(k+1,\epsilon^2,100 s)$ symmetric at $\mathbf{y}$. This together with \eqref{eq- pinched frequency} and Theorem \ref{thm: sym split equiv}\ref{thm: sym split equiv-2} gives $\cE_s^{k+1,1}(\mathbf{y})<\epsilon$ if $\epsilon \leq \frac{1}{10}$.
    \end{proof}

    Fix $\chi=\epsilon^{C\Lambda}$ to be the minimum of the $\chi$ appearing in Lemma \ref{lem: cover large a-balls} and the $\chi$ appearing in Lemma \ref{lem: strata b-ball}. Thus by Lemma \ref{lem: strata b-ball}, we have
    \begin{align*}
        \mathcal{S}_{\epsilon, r}^k\subset \bigcup_{r_a >\chi^{-1}r} (\mathcal{S}_{\epsilon,r}^k \cap \mathcal{N}^a) \cup \bigcup_{r_a\leq\chi^{-1}r} P(\mathbf{x}_a,r_a) \cup \bigcup_{r_b\leq  \chi^{-1}r}P(\mathbf{x}_b,r_b) \cup \bigcup_f P(\mathbf{x}_f,r_f),
    \end{align*}
    so that
    \begin{align*}
        &P(\cS_{\epsilon, r}^k,r)\\
        &\subset \bigcup_{r_a >\chi^{-1}r} P(\mathcal{S}_{\epsilon,r}^k \cap \mathcal{N}^a ,r)\cup \bigcup_{r_a\leq\chi^{-1}r} P(\mathbf{x}_a,r_a+r) \cup \bigcup_{r_b\leq  \chi^{-1}r}P(\mathbf{x}_b,r_b+r) \cup \bigcup_f P(\mathbf{x}_f,r_f+r). 
    \end{align*}
    By Lemma \ref{lem: cover large a-balls} and \eqref{eq neck decomposition for volume estimate}, we have
    \begin{align*}
        \mathcal{H}_{\mathcal{P}}^{n+2}\left( \bigcup_{r_a>\chi^{-1}r}P(\mathcal{S}_{\epsilon,r}^k\cap \mathcal{N}^a,r) \right)\leq \sum_{r_a>\chi^{-1}r} C\epsilon^{-C\Lambda}r^{n+2-k}\sum_a r_a^k \leq C^{\Lambda^2}(\delta \eta)^{-C\Lambda}r^{n+2-k}.
    \end{align*}
    By construction, we have $r_a,r_b \geq r$ and $r_f\in [r, C^{\Lambda} \eta^{-\frac{1}{n}}\delta^{-10n^2}r]$, so that Theorem \ref{theorem-neck-decomposition2}\ref{neckdecomposition:contentestimate2} yields
    \begin{align*} 
        &\mathcal{H}_{\mathcal{P}}^{n+2}\left( \bigcup_{r_a\leq\chi^{-1}r} P(\mathbf{x}_a,r_a+r) \cup \bigcup_{r_b\leq  \chi^{-1}r}P(\mathbf{x}_b,r_b+r) \cup \bigcup_{f\in F} P(\mathbf{x}_f,r_f+r) \right) \\ 
        &\leq C \left(\sum_{r_a\leq \chi^{-1}r} r_a^{n+2}  +\sum_{r_b\leq \chi^{-1}r} r_b^{n+2}+ \sum_f r_f^{n+2}\right) \\
        &\leq C r^{n+2-k} \left(\chi^{-(n+2-k)}\sum_a r_a^k+ \chi^{-(n+2-k)}\sum_b r_b^k+C^{\Lambda}\delta^{-10n^3}\eta^{-2}\sum_f r_f^k\right)\\
        &\leq C^{\Lambda^2} (\delta \eta)^{-C\Lambda}r^{n+2-k}. 
    \end{align*}
    Combining estimates, taking $\eta = C^{-\Lambda}\epsilon^{C\Lambda^2}$ and recalling that $\delta= \epsilon^{C\Lambda^2}$ yields
    \begin{align*}
        \mathcal{H}_{\mathcal{P}}^{n+2}(P(\mathcal{S}_{\epsilon,r}^k,r))\leq C^{\Lambda^3}\epsilon^{-C\Lambda^3}r^{n+2-k}.
    \end{align*}
    
    Next, we prove estimates for $\mathcal{S}^k_{\epsilon}$. By the neck decomposition Theorem \ref{theorem-neck-decomposition}, there exists a decomposition
    \begin{equation}
    \label{eq: cover by ab balls}
        P(\mathbf{0},1)\subseteq \bigcup_a \mathcal{N}^a \cup \bigcup_b P(\mathbf{x}_b,r_b) \cup \left( \widetilde{\mathcal{C}} \cup \bigcup_a \mathcal{C}_{a,0} \right),
    \end{equation}
    where $\mathcal{N}^a=P(\mathbf{x}_a,2r_a)\setminus \overline{P}(\mathcal{C}^a,\mathbf{r}^a)$ is an $(m_a,k,\delta)$-neck region for some $m_a\le C\Lambda$, and $\mathcal{E}^{k+1,1}_{100r_b}(\mathbf{y}_b)\le\eta$ some $\mathbf{y}_b\in P(\mathbf{x}_b,4r_b)$, and
    \begin{align*}
        \sum_a r_a^k \leq  C^{\Lambda^2} (\delta\eta)^{-C\Lambda}, \qquad \mathcal{H}_{\mathcal{P}}^k(\widetilde{\mathcal{C}})=0.
    \end{align*}
    By Claim \ref{claim: neck to center} in Lemma \ref{lem: cover large a-balls}, for any $\mathbf{y}\in \mathcal{S}^k_{\epsilon}$, by choosing $r=\chi d_{\mathcal{P}}(\mathbf{y},\mathcal{C}_a)$, we have $\mathbf{y}\in \mathcal{S}^k_{\epsilon,r}$ and thus $\mathbf{y}\notin \mathcal{N}^a$.
    By Lemma \ref{lem: strata b-ball}, $P(\mathbf{x}_b,r_b)\cap \mathcal{S}^k_{\epsilon}\subseteq P(\mathbf{x}_b,r_b)\cap \mathcal{S}^k_{\epsilon,\chi r_b}=\emptyset$. By the covering \eqref{eq: cover by ab balls}, it follows that
    \begin{equation}\label{eq:classicalstrata}
        \mathcal{S}^k_{\epsilon}\cap P(\mathbf{0},1)\subseteq \widetilde{\mathcal{C}}\cup \bigcup_a\mathcal{C}_{a,0}.
    \end{equation}
    By Theorem \ref{theorem-neck-decomposition}\ref{neckdecomposition:contentestimate} and Theorem \ref{theorem-neck-structure}, we thus have 
    \begin{align*}
        \mathcal{H}_{\mathcal{P}}^k(\mathcal{S}_{\epsilon}^k \cap P(\mathbf{0},1)) \leq \sum_a \mathcal{H}_{\mathcal{P}}^k(\mathcal{C}_{a,0}) \leq C\sum_a r_a^k \leq C^{\Lambda^3}\epsilon^{-C\Lambda^3}.
    \end{align*}
    Moreover, each $\mathcal{C}_{a,0}$ is parabolic $(k,\ell)$, so that $\mathcal{S}^k_{\epsilon}$ is $(k,\ell)$-rectifiable if $C\delta<\ell$. Therefore, $\mathcal{S}_{\epsilon}^k$ is parabolic $k$-rectifiable. When $k\in \{n,n+1\}$, in the neck decomposition \eqref{eq: cover by ab balls}, all neck regions $\mathcal{N}_a=P(\mathbf{x}_a,2r_a)\setminus \mathcal{C}_{a}$ must be vertical. In fact, if $k=n$ and $\mathcal{N}_a$ is spatial, then $u$ is in fact $(n+2,C^{\Lambda}\delta,r_a)$-symmetric at any $\mathbf{x}\in\mathcal{C}_a$, and thus $P(\mathbf{x}_a,r_a)$ is a \ref{b-balls}-ball, a contradiction. Thus, $\mathcal{S}_{\epsilon}^k$ is vertically parabolic $k$-rectifiable if $k\in \{n,n+1\}$. 

    Finally, we prove estimates for time slices assuming $k \in \{n,n+1\}$. For all $t\in(-1,1)$, the restriction $\mathcal{C}_{a,0}^t\coloneqq\mathcal{C}_{a,0}\cap (\R^n\times\{t\})$ is a $(k-2,\delta)$-graph with $\mathcal{H}^{k-2}(\mathcal{C}_{a,0}^t)\le C r_a^{k-2}$. By Lemma \ref{lem:fubinilike} below, $\mathcal{H}^{k-2}(\widetilde{\mathcal{C}}^t)=0$ for $\mathcal{H}^1$-a.e. $t\in(-1,1)$, and $\mathcal{H}^{k-1}(\widetilde{\mathcal{C}}^t)=0$ for $\mathcal{H}^{\frac{1}{2}}$-a.e. $t\in (-1,1)$, where $\widetilde{\mathcal{C}}^t\coloneqq \widetilde{\mathcal{C}}\cap (\R^n\times\{t\})$. By intersecting \eqref{eq:classicalstrata} with $\mathbb{R}^n \times \{t\}$, we have $\mathcal{S}^{k,t}_{\epsilon}\subseteq \widetilde{\mathcal{C}}^t \cup_a \mathcal{C}_{a,0}^t$, so that
    \begin{align*}
        \mathcal{H}^{k-2}(\mathcal{S}_{\epsilon}^{k,t}) \leq \mathcal{H}^{k-2}(\widetilde{\mathcal{C}}^t)+ \sum_a \mathcal{H}^{k-2}(\mathcal{C}_{a,0}^t) \leq C\sum_a r_a^k \leq C^{\Lambda^3}\epsilon^{-C\Lambda^3}
    \end{align*}
    and
    \begin{align*}
        \mathcal{H}^{k-1}(\mathcal{S}_{\epsilon}^{k,t}) = \mathcal{H}^{k-1}(\widetilde{\mathcal{C}}^t)=0
    \end{align*}
    for $\mathcal{H}^{\frac{1}{2}}$-a.e. $t\in (-1,1)$. 
\end{proof}

In the following lemma, we disintegrate the measure of a subset $E\subset P(\mathbf{0},1)$ into its measures on time slices.
\begin{lemma} \label{lem:fubinilike}
    If $E\subseteq P(\mathbf{0},1)$ is Borel and $\sigma \in [0,\min(\frac{k}{2},1)]$, then $t\mapsto \mathcal{H}^{k-2\sigma}(E_t)$ is a measurable $[0,\infty]$-valued function, and
     \begin{align*}
         \int_{(-1,1)}\mathcal{H}^{k-2\sigma}(E_t) \,d\mathcal{H}^{\sigma}(t)\leq C\mathcal{H}_{\mathcal{P}}^k(E),
     \end{align*}
     where $E_t=E\cap (\R^n\times\{t\})$, and $\mathcal{H}^j$ denotes the $j$-dimensional Hausdorff measure with respect to the Euclidean metric. 
\end{lemma}

\begin{proof}
    By \cite[\S 2.10.26]{federer-2014-geometric-measure-theory}, the function $t\mapsto \mathcal{H}^{k-2\sigma}(E_t)$ is measurable. Moreover, appealing to \cite[Theorem 2.10.25]{federer-2014-geometric-measure-theory}, taking $X\leftarrow (\mathbb{R}^{n+1},d_{\mathcal{P}})$, $Y \leftarrow (\mathbb{R},d_{\mathcal{P}})$, and $f \leftarrow \pi$, where $\pi :\mathbb{R}^n \times \mathbb{R} \to \mathbb{R}$ is projection to the second factor, and noting that $C^{-1}\mathcal{H}^{2\sigma}\leq \mathcal{H}_{\mathcal{P}}^{\sigma}\leq C\mathcal{H}^{2\sigma}$ and $\mathcal{H}_{\mathcal{P}}^{k-2\sigma}(E_t)=\mathcal{H}^{k-2\sigma}(E_t)$, the claim follows.

    For completeness, we give the proof in the case $\sigma =1$ along the lines of \cite[Theorem 7.7]{MattilaTextbook}. The general argument is roughly the same but requires more technical justifications. We may assume $\mathcal{H}^k_{\mathcal{P}}(E)<\infty$. For $j\in \mathbb{N}$, let $\{P(\mathbf{x}_{j,i},r_{j,i})\}_{i\in I_j}$ be a covering of $E$ such that $r_{j,i}< \frac{1}{j}$ and
    \begin{align*}
        \omega_k\sum_{i\in I_j}r_{j,i}^k \le \mathcal{H}_{\mathcal{P},\frac{1}{j}}^{k} (E)+\frac{1}{j},
    \end{align*}
    where $\omega_k$ is the measure of $k$-dimensional sphere. Here, for $\delta>0$,
    \begin{align*}
        \mathcal{H}^k_{\mathcal{P},\delta}(E)\coloneqq \inf \sum_{i} \omega_k r_i^k,
    \end{align*}
    where the infimum is over all covering $E\subseteq \cup_{i}P(\mathbf{x}_{i},r_i)$ with $r_i<\delta$. Further, $\mathcal{H}^k_{\delta}$ is defined similarly but with coverings by Euclidean balls. By Fatou's Lemma,
    \begin{align*}
        \int_{-1}^1 \mathcal{H}^{k-2}(E_t)\,d\mathcal{H}^1(t)
        =&\int_{-1}^1 \lim_{j\to \infty}\mathcal{H}_{1/j}^{k-2}(E_t)\,d\mathcal{H}^1(t) \leq \liminf_{j\to \infty}\int_{-1}^1  \sum_{i\colon P(\mathbf{x}_{j,i},r_{j,i})_t \neq \emptyset }\omega_{k-2}r_{j,i}^{k-2}\,d\mathcal{H}^1(t)\\
        \le&\ \liminf_{j\to \infty}\sum_i \int_{t_{j,i}-r_{j,i}^2}^{t_{j,i}+r_{j,i}^2}  \omega_{k-2}r_{j,i}^{k-2}\,d\mathcal{H}^1(t)\leq \liminf_{j\to \infty} 2\omega_{k-2}\sum_i r_{j,i}^{k} \\ 
        \le & C\mathcal{H}^{k}_{\mathcal{P}}(E),
    \end{align*}
    where $\mathbf{x}_{j,i}=(x_{j,i},t_{j,i})$.
\end{proof}

\begin{proof}[Proof of Theorem \ref{main-theorem-caloric-setting} and Corollary \ref{theorem time slice}]
    Combine Proposition \ref{proposition-containment} and  Theorem \ref{thm: caloric vol part 1}. 
\end{proof}

Below, we prove an analog of Theorem \ref{thm:regularity-general} in the caloric case, which can be proven using real analyticity and Theorem \ref{main-theorem-caloric-setting}. We record it here because the proof in the general setting is similar, while real analyticity fails in that generality.
\begin{theorem} \label{thm:regularity}
    Let $u$ be a caloric function with $N(10^5)\leq \Lambda$. For any $k\in \{n,n+1\}$ and $\epsilon>0$, there exist $\mathbf{x}_i \in P(\mathbf{0},\frac{1}{4})$, vertical planes $V_i \in \operatorname{Aff}_{\mathcal{P}}(k)$, $r_i \in (0,1]$, and $\epsilon$-regular parabolic Lipschitz functions $F_i\colon V_i \to V_i^{\perp}$ such that 
    \begin{align*}
        \sum_{i} r_i^{k} \leq C(\epsilon,\Lambda),
    \end{align*}
    and such that $G_i \coloneqq \{ \mathbf{x}+ F_i(\mathbf{x})\colon \mathbf{x} \in V_i \cap P(\mathbf{x}_i,10 r_i)\}$ satisfy
    \begin{align*}
        \mathcal{H}_{\mathcal{P}}^{k}\left( (P(\mathbf{0},1)\cap \mathscr{C}^k(u))\setminus \bigcup_{i}G_i \right) =0.
    \end{align*}
\end{theorem}

\begin{proof}
    Given $\epsilon>0$, fix $\eta\leq \ol{\eta}(\Lambda)$ and $\delta\leq C(\eta,\Lambda)^{-1}\epsilon^2$ to be determined later. Let 
    \begin{align*}
        P(\mathbf{0},1)\subseteq \bigcup_{a \in A} \mathcal{N}^a 
        \cup  \bigcup_{g\in G} P(\mathbf{x}_g,r_g)\cup\left( \widetilde{\mathcal{C}} \cup \bigcup_{a\in A} \mathcal{C}_{a,0} \right),
    \end{align*}
    be the decomposition from Theorem \ref{theorem-neck-decomposition3} with $\epsilon \leftarrow \delta$ and $\eta\leftarrow \eta$ if $\eta\leq C(\Lambda)^{-1}$. Here $\cN^a=\widehat{\cN}^{a}\cap P(\mathbf{x}_a,2r_a)$ where $\widehat{\cN}^{a}$ is the associated strong neck region of scale $2r_a$.
    
    Recall that $\mathscr{C}_r^k\coloneqq Z_r$ or $S_r$ and $\mathscr{C}^k\coloneqq Z$ or $S$ depending on $k=n+1$ or $n$, respectively. Fix some $\epsilon_0<C^{-\Lambda}$ so that we can apply Proposition \ref{proposition-containment} to conclude that $\mathscr{C}^k(u)=\bigcap_{r>0}\mathscr{C}_r^k(u)\subset \cS_{\epsilon_0}^{k}$. Now we apply Claim \ref{claim: neck to center} with $\epsilon\leftarrow \epsilon_0$, $\delta\leftarrow \delta$, $\chi\leftarrow \epsilon_0^{C\Lambda}$ to obtain $\mathscr{C}^k(u) \cap \mathcal{N}^a = \emptyset$ for all $a\in A$.
    
    Further, we claim that $\mathscr{C}^k(u)\cap P(\mathbf{x}_g,r_g)=\emptyset$ for all $g\in G$. In fact, 
    \begin{align*}
        \sup_{\mathbf{x}\in P(\mathbf{x}_g,r_g)} N_{\mathbf{x}}(\delta^{-2}r_g^2)\leq m_\ast-1+ C(\Lambda)^{-1}\delta< m_\ast-\frac{3}{4},
    \end{align*}
    if $\delta\leq C^{-\Lambda}$, where $m_\ast$ is defined in \eqref{eq def of mstar}. However, Lemma \ref{lemma pinching implies nodal} implies that $\mathbf{x}\in \mathscr{C}^k(u)$ has $N_{\mathbf{x}}(0)\geq m_\ast-\frac{1}{2}$. Therefore, $\mathscr{C}^k\subset \widetilde{\cC}\cup \bigcup_{a\in A}\cC_{a,0}$. Recall that $\cH^{k}_{\cP}(\widetilde{\cC})=0$. Therefore, the claim about the structure of $\mathscr{C}^k$ follows from Theorem \ref{thm:strongneckstructure} applied to $\widehat{\cN}^a$ if $\delta \leq C(\eta,\Lambda)^{-1}\epsilon^2$.
\end{proof}

\section{Parabolic Inequalities with Lipschitz Coefficients}\label{more-general-parbolic-equations}
\def \re {\mathrm{e}}
We consider solutions to linear second order parabolic inequality in divergence form with Lipschitz leading coefficients:
\begin{align}
    \begin{split}
        &\verts{\pdt u- \cL u}\leq \lambda (\verts{u}+\verts{\cd u})\qquad \text{ on } P(\mathbf{0},10^2\lambda^{-1})),\\
        &(1+\lambda)^{-1}I\le a\leq (1+\lambda)I,\qquad \verts{a^{ij}(\mathbf{x})-a^{ij}(\mathbf{y})}\leq \lambda|\mathbf{x}-\mathbf{y}|,\label{part-2-eq-parabolic-equation}
    \end{split}
\end{align}
where $\cL u\coloneqq \div (a\cdot \cd u)$, $\lambda \in (0,1]$, and $a^{ij}=a^{ji}$. Although weak solutions to \eqref{part-2-eq-parabolic-equation} belong to $C_{\operatorname{loc}}^{1+\alpha,\frac{1+\alpha}{2}}(P(\mathbf{0},10^2\lambda^{-1}))$ for any $\alpha \in (0,1)$ (see \cite[Theorem 4.8]{lieberman-1996-second}), neither $\mathcal{L} u$ nor $\partial_tu$ are defined pointwise, so by \eqref{part-2-eq-parabolic-equation}, we really mean that $(\partial_t - \mathcal{L})u=g$ in the sense of distributions, where $g\in L^{\infty}(\mathbb{R}^n \times \mathbb{R})$ satisfies $|g|\leq \lambda(|u|+|\nabla u|)$ $\mathcal{H}_{\mathcal{P}}^{n+2}$-a.e. on $P(\mathbf{0},10^2\lambda^{-1})$. 

We also assume that $u$ satisfies the following doubling condition for all $\mathbf{x}_0=(x_0,t_0) \in P(\mathbf{0},\lambda^{-1})$ and $r\in (0,\lambda^{-1}]$:
\begin{align} \label{eq:strongerdoubling}
    \int_{B(x_0,4r)\times[t_0-4r^2,t_0]}u^2d\mathcal{H}_{\mathcal{P}}^{n+2} \leq \Lambda \int_{B(x_0,r)} u^2(x,t_0)dx.
\end{align}
The condition \eqref{eq:strongerdoubling} rules out non-trivial functions satisfying \eqref{part-2-eq-parabolic-equation} with support contained in a slab $\R^n\times [a,b]$ (see \cite{delsanto}). Furthermore, solutions to \eqref{part-2-eq-parabolic-equation} which satisfy \eqref{eq:strongerdoubling} have polynomial growth as made precise in Lemma \ref{lemma polynomial growth of general solutions} below. The polynomial growth enforces the frequency to be bounded at large scale, see Lemma \ref{lem:cutoff}.

For the rest of this section, we assume that $u$ is a solution to \eqref{part-2-eq-parabolic-equation} satisfying \eqref{eq:strongerdoubling}. 

\subsection{Generalized frequency}
In this subsection, we will define the energy functionals and prove their almost monotonicity as in Section \ref{parabolic frequency}.

Given a solution $u$ to \eqref{part-2-eq-parabolic-equation} satisfying \eqref{eq:strongerdoubling}, we first prove that parabolic rescalings at small scales satisfy an improved version of \eqref{part-2-eq-parabolic-equation}, in a sense that will be made precise below.  
\begin{definition} 
    For any $\mathbf{x} = (x,t) \in P(\mathbf{0},100)$ and $r\in (0,1]$, the \textit{parabolic rescaling} $u_{\mathbf{x};r}$ is defined by
    \begin{align*}
        u_{\mathbf{x};r}(y,s)\coloneqq u(x+r a^{-\frac{1}{2}}(\mathbf{x})y,t+r^2 s).
    \end{align*}
\end{definition}

\begin{remark}\label{remark pde for parabolic rescaling}
    Given $\mathbf{x}_0 = (x_0,t_0)$ and $r$, we define 
    \begin{align*}
        \widetilde{a}^{ij}(x,t)\coloneqq a^{-\frac{1}{2}}(\mathbf{x}_0)a(x_0+ra^{\frac{1}{2}}(\mathbf{x}_0)x, t_0+r^2 t)a^{-\frac{1}{2}}(\mathbf{x}_0)
    \end{align*}
    and $\widetilde{\mathcal{L}}v\coloneqq\div(\widetilde{a} \cdot \cd v)$. Then 
    \begin{align*} 
        |(\partial_t -\widetilde{\cL})u_{\mathbf{x}_0;r}| (x,t) &= r^2|(\partial_t - \mathcal{L})u|(x_0+ra^{\frac{1}{2}}(\mathbf{x}_0)x, t_0+r^2 t)\\ &\leq r^2 \lambda (|u|+|a(\mathbf{x}_0)\nabla u|)(x_0+ra^{\frac{1}{2}}(\mathbf{x}_0)x, t_0+r^2 t) \\
        &\leq \lambda r(r  |u_{\mathbf{x}_0;r}|+(1+C)|\nabla u_{\mathbf{x}_0;r}|)(x,t)
    \end{align*}
    for all $(x,t)\in P(\mathbf{0},5r^{-1})$. Moreover, we have $\widetilde{a}^{ij}(\mathbf{0})=\delta^{ij}$, $(1+\lambda)^{-2} I \leq \widetilde{a} \leq (1+\lambda)^2 I$, as well as
    \begin{align*}
        |\widetilde{a}(\mathbf{x})-\widetilde{a}(\mathbf{y})|\leq (1+\lambda)\lambda r\max\{ |a^{\frac{1}{2}}(\mathbf{x}_0)(x-y)|,|t-s|^{\frac{1}{2}} \} \leq 2\lambda r |\mathbf{x}-\mathbf{y}|
    \end{align*}
    for all $\mathbf{x}=(x,t)$ and $\mathbf{y}=(y,s)$ in $P(\mathbf{0},5r^{-1})$.
\end{remark}

Throughout this section, we apply the global computation in Section \ref{energy-functionals} to functions of the form $\chi u$, where $u$ is a local solution to \eqref{part-2-eq-parabolic-equation} satisfying \eqref{eq:strongerdoubling} and $\chi$ is a cutoff function defined below. Fix $\xi \in C^{\infty}(\mathbb{R})$ such that $\xi|_{(-\infty,1]}\equiv 1$ and $\xi|_{[2,\infty)}\equiv 0$, as well as $0\leq \xi \leq 1$ and $\xi'\leq 0$ everywhere. We define $\chi \in C^{\infty}(\mathbb{R}^n \times (-\infty,0))$ as
\begin{align*}
    \chi(x,t) \coloneqq \xi \left( \frac{\lambda^{\frac{\Upsilon}{2}}|x|}{|t|^{\Upsilon}}\right),
\end{align*}
where $\Upsilon \in (0,\frac{1}{2})$ is fixed throughout this section.

We will now define the fundamental almost-monotone quantities which generalize Definition \ref{definition-energy-functionals}. Note that these quantities are defined locally.
\begin{definition} 
    For $\mathbf{x}_0 \in P(\mathbf{0},5)$ and $r\in (0,1]$, we define
    \begin{align*}
        \widehat{H}_{\mathbf{x}_0}^u(\tau) \coloneqq \int_{\mathbb{R}^n} (\chi u_{\mathbf{x}_0;1} )^2\, d\nu_{-\tau}, && \widehat{E}_{\mathbf{x}_0}^u\coloneqq 2\tau \int_{\mathbb{R}^n} |\nabla (\chi u_{\mathbf{x}_0;1})|^2 \,d\nu_{-\tau},&&\widehat{N}_{\mathbf{x}_0}^u(\tau) \coloneqq \frac{\widehat{E}_{\mathbf{x}_0}(\tau)}{\widehat{H}_{\mathbf{x}_0}(\tau)}.
    \end{align*}
\end{definition}
When $\mathbf{x}_0=\mathbf{0}$ and $u$ is clear, we write $\widehat{H}(\tau)\coloneqq \widehat{H}_{\mathbf{x}_0}^u(\tau)$ and so on.

In the lemma below, we prove a version of an interior estimate similar to Lemma \ref{lemma-interior-estimates}. 
\begin{lemma}\label{lemma polynomial growth of general solutions}
    The following holds if $\lambda \leq \overline{\lambda}(\Lambda)$. For any $\mathbf{x}_0=(x_0,t_0)\in P(\mathbf{0},10)$, $r\in  (0,\tfrac{1}{C(\Lambda)\lambda})$, and $\tau \in (0,\tfrac{1}{C(\Lambda)\lambda}]$, we have 
    \begin{align*}
        \sup_{B(x_0,r)} (|u(\cdot, t_0-\tau)|^2 + \tau |\nabla u(\cdot, t_0-\tau)|^2) &\leq C(\Lambda) \left( 1+\frac{r^2}{\tau} \right)^{C(\Lambda)} \int_{ B(0,\sqrt{\tau})} u^2(x,-\tau)dx\\
        &\leq  C(\Lambda) \left(1+\frac{r^2}{\tau}\right)^{C(\Lambda)} \widehat{H}_{\mathbf{x}_0}^u(\tau).
    \end{align*}
\end{lemma}

\begin{proof} 
    By replacing $u$ with $u_{\mathbf{x}_0;1}$, we may assume that $\mathbf{x}_0=\mathbf{0}$ and $u_{\mathbf{x}_0;1}=u$. If $\lambda \leq \overline{\lambda}(\Lambda)$, then, for all $r\in (0,C(\Lambda)^{-1}\lambda^{-1}]$, we use \eqref{eq:strongerdoubling} and \cite[Theorem 3(1)]{escauriaza2006doubling} to get
    \begin{equation}\label{eq:applicationofEFV} 
        \int_{B(0,r) \times [-\tau-r^2,-\tau]} u^2\, d\cH_{\cP}^{n+2}\leq C(\Lambda)r^2 \int_{B(0,r)} u^2(x,-\tau)dx.
    \end{equation}
    By taking $r\leftarrow 2\sqrt{\tau}$ in \eqref{eq:applicationofEFV} and combining with interior parabolic estimates (see for instance \cite[Theorem 4.8]{lieberman-1996-second}), we obtain
    \begin{align*} 
        \sup_{B(\mathbf{0},\sqrt{\tau})} (u^2 +|\nabla u|^2)(\cdot,-\tau)\leq & \frac{C}{\tau^{\frac{n+2}{2}}}\int_{B(0,2\sqrt{\tau})\times [-5\tau,-\tau]} u^2\, d\cH_{\cP}^{n+2}
        \leq \frac{C(\Lambda)}{\tau^{\frac{n}{2}}} \int_{B(0,2\sqrt{\tau})}u^2(x,-\tau)dx \\
        \leq & C(\Lambda) \widehat{H}^{u}(\tau),
    \end{align*}
    which implies the claim if $r<\sqrt{\tau}$.
    
    Now suppose $r\in [\sqrt{\tau},C(\Lambda)^{-1}\lambda^{-1}]$, and choose $k\in \mathbb{N}$ such that $2^k \sqrt{\tau} \leq r\leq 2^{k+1}\sqrt{\tau}$. By \eqref{eq:strongerdoubling} and \cite[Theorem 2(1)]{escauriaza2006doubling}, we have
    \begin{align*}
       \int_{B(0,2s)}u^2(x,-\tau)dx \leq C(\Lambda) \int_{B(0,s)}u^2(x,-\tau)dx 
    \end{align*}
    for all $s\in (0,\lambda^{-1}]$. Iterating this estimate yields
    \begin{align*}
        \int_{B(0,r)} u^2 (x,-\tau)dx \leq \int_{B(0,2^{k+1}\sqrt{\tau})} u^2(x,-\tau)dx \leq C(\Lambda)^{k+1}\int_{B(0,\sqrt{\tau})} u^2(x,-\tau)dx,
    \end{align*}
    so combining this with \eqref{eq:applicationofEFV} and interior parabolic estimates yields
    \begin{align*} 
        \sup_{B(0,r)} (u^2 + r^2 |\nabla u|^2)(\cdot,-\tau) \leq &\frac{C}{r^{n+2}} \int_{B(0,2r) \times [-\tau-4r^2,-\tau]} u^2\,d\cH_{\cP}^{n+2}
        \leq \frac{C(\Lambda)}{r^n} \int_{B(0,2r) } u^2(x,-\tau)dx\\
        \leq & \frac{C(\Lambda)^{k+1}}{r^n}\int_{B(0,\sqrt{\tau})} u^2(x,-\tau)dx 
        \leq  C(\Lambda)^{k+1} \widehat{H}^{u}(\tau).
    \end{align*}
    The claim then follows from the observation that
    \begin{align*}
        C(\Lambda)^{k+1} = C(\Lambda)2^{k\log_2C(\Lambda)} \leq C(\Lambda) \left( \frac{r^2}{\tau} \right)^{\log_2C(\Lambda)}.
    \end{align*}
\end{proof}

In the lemma below, we show that the error arising from using a cutoff function while doing global computation is exponentially small. Moreover, we prove that the frequency is bounded.
\begin{lemma} \label{lem:cutoff} 
    The following holds if $\lambda \leq \overline{\lambda}(\Lambda,\Upsilon)$. For any $\mathbf{x}_0 \in P(\mathbf{0},10)$, $\alpha \in [0,1)$, and $\tau\in (0,1]$, we have
    \begin{align}
        \begin{split}
            &\int_{\mathbb{R}^n} (1_{\supp(\nabla \chi)}+|\partial_t \chi|+ |\nabla \chi |+ |\Delta \chi|)^2(|u_{\mathbf{x}_0;1}|^2+\tau|\nabla u_{\mathbf{x}_0;1}|^2)e^{\alpha f}\,d\nu_{-\tau} \\
            &\qquad\leq C({\Lambda,\alpha,\Upsilon}) \widehat{H}^u(\tau) \lambda^{\Upsilon}e^{-\frac{(1-\alpha)}{8\lambda^{\Upsilon}|t|^{1-2\Upsilon}}}.\label{eq-exponential decay outside of support}
        \end{split}
    \end{align}
    Further,
    \begin{align}
        \int_{\mathbb{R}^n} \left( (\chi u_{\mathbf{x}_0;1})^2 + \tau|\nabla (\chi u_{\mathbf{x}_0;1})|^2 \right) e^{\alpha f} \,d\nu_{-\tau} &\leq C(\alpha,\Upsilon, \Lambda)\widehat{H}_{\mathbf{x}_0}^u(\tau),\label{eq-weakfrequencybound}
    \end{align}
    and in particular, $\widehat{N}_{\mathbf{x}_0}^{u}(\tau) \leq C(\Lambda)$.
\end{lemma}

\begin{proof}
    By replacing $u$ with $u_{\mathbf{x}_0;1}$, we may assume that $\mathbf{x}_0=\mathbf{0}$ and $u_{\mathbf{x}_0;1}=u$. We have 
    \begin{align*}
        \supp(|\partial_t \chi|) \cup \supp(|\nabla \chi|) \cup \supp(|\nabla^2 \chi|) \subseteq \{(x,t)\in \mathbb{R}^n \times (-\infty,0] \colon |t|^{\Upsilon} \leq \lambda^{\frac{\Upsilon}{2}} |x| \leq 2|t|^{\Upsilon} \},
    \end{align*}
    as well as 
    \begin{align*}
        |\partial_t \chi| \leq C\lambda^{\frac{\Upsilon}{2}}|t|^{-(\Upsilon+1)}, \qquad |\nabla \chi| \leq C\lambda^{\frac{\Upsilon}{2}}|t|^{-\Upsilon}, \qquad |\Delta \chi| \leq C\lambda^{\frac{\Upsilon}{2}}|t|^{-2\Upsilon}.
    \end{align*}
    Combining these expressions and using Lemma \ref{lemma polynomial growth of general solutions}, we get
    \begin{align*} 
        &\int_{\mathbb{R}^n} (1_{\text{supp}(\nabla \chi)}+|\partial_t \chi|+ |\nabla \chi |+ |\Delta \chi|)^2(|u|^2+\tau|\nabla u|^2)e^{\alpha f} d\nu_t \\
        &\qquad \leq C(\Lambda)\widehat{H}^u(\tau)\lambda^{\Upsilon}|t|^{-4\Upsilon-2-\frac{n}{2}} \int_{B(0,2\lambda^{-\frac{\Upsilon}{2}}|t|^{\Upsilon})\setminus B(0,\lambda^{-\frac{\Upsilon}{2}}|t|^{\Upsilon})} (1+|x|)^{C\Lambda} e^{-(1-\alpha)\frac{|x|^2}{4|t|}} \,dx \\
        & \qquad\leq C({\Lambda,\alpha,\Upsilon}) \widehat{H}^u(\tau) \lambda^{\Upsilon}e^{-\frac{(1-\alpha)}{8\lambda^{\Upsilon}|t|^{1-2\Upsilon}}}.
    \end{align*} 
    This proves \eqref{eq-exponential decay outside of support}    
    To get \eqref{eq-weakfrequencybound}, we use Lemma \ref{lemma polynomial growth of general solutions} and \eqref{eq-exponential decay outside of support} as follows:
    \begin{align*} 
        &\int_{\mathbb{R}^n} \left( (\chi u)^2 + \tau|\nabla (\chi u)|^2 \right) e^{\alpha f} d\nu_{-1} \\
        &\leq  \frac{C(\Lambda)\widehat{H}^u(\tau)}{\tau^{\frac{n}{2}}} \int_{\mathbb{R}^n} (1+|y|)^{2\Lambda} e^{-\frac{(1-\alpha)|y|^2}{4\tau}}dy + 2 \int_{\mathbb{R}^n} \tau |\nabla \chi|^2 |u|^2 e^{\alpha f}d\nu_t \\
        &\leq  C(\Lambda)\widehat{H}^u(\tau) \int_{\mathbb{R}^n} (1+|z|^2)^{2\Lambda} e^{-\frac{(1-\alpha)|z|^2}{4}}dz + C(\alpha,\Upsilon,\Lambda) \widehat{H}^u(\tau)\leq C(\alpha,\Upsilon,\Lambda) \widehat{H}^u(\tau).
    \end{align*}
\end{proof}

The result below establishes almost monotonicity of $\log \widehat{H}$ when $\lambda$ is small as an analog of Lemma \ref{lemma-monotonicity-formulae-for-the-energy-functionals}.
\begin{lemma} \label{lem:weakestimates}
    The following holds if $\lambda \leq \overline{\lambda}(\Lambda,\Upsilon)$. There exists $C=C(\Lambda)$ such that for any $\mathbf{x}_0 \in P(\mathbf{0},10)$ with $0<\tau_1 < \tau_2\leq 1$, we have
    \begin{align}
        \left( \frac{\tau_1}{\tau_2} \right)^{C} \widehat{H}_{\mathbf{x}_0}^{u}(\tau_2)\leq  \widehat{H}_{\mathbf{x}_0}^{u}(\tau_1) \leq \left( \frac{\tau_2}{\tau_1} \right)^{\lambda^{\Upsilon} C} \widehat{H}_{\mathbf{x}_0}^u (\tau_2).\label{eq:Hoscillation} 
    \end{align}
\end{lemma}

\begin{proof} 
    By replacing $u$ with $u_{\mathbf{x}_0;1}$, we may assume that $\mathbf{x}_0=\mathbf{0}$ and $u_{\mathbf{x}_0;1}=u$. In light of Lemma \ref{lemma-monotonicity-formulae-for-the-energy-functionals}, we estimate the following error to compute $\frac{d}{d\tau}\log \widehat{H}^u:$
    \begin{align*}
        \int_{\R^n} (\chi u) \Box (\chi u)\, d\nu_{t}=\int_{\R^n} (\chi u)\bigg( (\pdt -\cL)u +\div ((a-I)\cdot \cd(\chi u))\bigg)\,d\nu_t.
    \end{align*}
    First, we integrate by parts and use Lemma \ref{lem:cutoff} to get
    \begin{align*}
        \verts*{\int_{\R^n} \chi u \div ((a-I) \cdot \cd (\chi u)) \, d\nu_t}&=\verts*{ \int_{\R^n} (\cd (\chi u)-\chi u\cd f)\cdot ((a-I) \cd (\chi u)) \,d\nu_t}\\
        &\leq C\lambda^{1-\frac{\Upsilon}{2}} \tau^\Upsilon\int_{\R^n} \verts{\cd (\chi u)}^2 + \verts{\chi u\cd f}\verts{\cd (\chi u)} \,d\nu_t\\
        &\leq C\lambda^{1-\frac{\Upsilon}{2}} \tau^{\Upsilon} \frac{\widehat{E}^{u}(\tau)+\widehat{H}^{u}(\tau)}{\tau},
    \end{align*}
    where we used Cauchy--Schwarz in the third line along with Lemma \eqref{lem:cutoff}.
    
    Next, we use 
    \begin{align}
        (\pdt-\cL) (\chi u)= \chi (\pdt-\cL) u+ u(\pdt-\cL)\chi -2(a \cd u)\cdot  \cd \chi\label{eq-product rule for pdt-L}
    \end{align}
    to estimate
    \begin{align*}
        \verts*{\int_{\R^n} \chi u(\pdt -\cL) (\chi u)\, d\nu_t}&\leq \int_{\R^n} \chi \verts{u} (\chi \lambda (\verts{u}+\verts{\cd u}) + \verts{u} \verts{(\pdt -\cL)\chi} + 2 (1+\lambda) \verts{\cd u} \verts{\cd \chi}) \, d\nu_t\\
        &\leq C\int_{\R^n} (\lambda \chi^2+\verts{\cd \chi}^2 +\verts{(\pdt-\cL)\chi}^2) (\frac{1}{\sqrt{\tau}}\verts{u}^2+\sqrt{\tau}\verts{\cd u}^2)\,d\nu_t \\
        &\leq C(\lambda+\lambda^{\Upsilon}e^{-\frac{1}{8\lambda^{\Upsilon} \tau^{1-2\Upsilon}}})\tau^{\frac{1}{2}} \bigg(\frac{\widehat{E}^{u}(\tau)+\widehat{H}^{u}(\tau)}{\tau}\bigg).
    \end{align*}
    Therefore,
    \begin{align}
        \frac{1}{\tau}(\widehat{N}^{u}(\tau)-C\lambda^{\Upsilon} \tau^{\Upsilon})\leq \frac{d}{d\tau} \log \widehat{H}^{u}(\tau) \leq \frac{1}{\tau} (\widehat{N}^{u}(\tau) +C\lambda^{\Upsilon}\tau^{\Upsilon}).\label{eq derivative of log H}
    \end{align}
    Combining this with \eqref{eq-weakfrequencybound} yields \eqref{eq:Hoscillation}.
\end{proof}

The following result shows that $\cd^2 (\chi u)$ and $\Delta_f(\chi u)$ are square integrable in weighted space-time. Furthermore, we show that $\Box (\chi u)$ is small in a weighted sense.
\begin{lemma} \label{lem:HessianL2} 
    The following holds if $\lambda \leq \ol{\lambda}(\Lambda,\Upsilon)$.
    Let $\alpha\in [0,\frac{1}{2n}]$ and $\theta\in (0,1]$. For any $\mathbf{x}_0 \in P(\mathbf{0},10)$ and $0<\theta\tau_2<\tau_1 \le \tau_2 \leq 1$, we have
    \begin{subequations}
        \begin{align}
            \int_{\tau_1}^{\tau_2} \frac{\tau}{\widehat{H}_{\mathbf{x}_0}^{u}(\tau)} \int_{\mathbb{R}^n}  |\nabla ^2 (\chi u_{\mathbf{x}_0;1})|^2 e^{\alpha f} \,d\nu_{-\tau} d\tau &\leq C(\alpha,\Lambda,\theta,\Upsilon),\label{eq integral hessian estimate}\\
            \int_{\tau_1}^{\tau_2} \frac{\tau}{\widehat{H}_{\mathbf{x}_0}^{u}(\tau)} \int_{\mathbb{R}^n} |\Box (\chi u_{\mathbf{x}_0;1})|^2 e^{\alpha f} \,d\nu_{-\tau} d\tau&\leq C(\alpha,\Lambda,\theta,\Upsilon)\lambda^{\Upsilon} \tau_2^{2\Upsilon},\label{eq integral box estimate}\\
            \int_{\tau_1}^{\tau_2} \frac{\tau}{\widehat{H}_{\mathbf{x}_0}^{u}(\tau)} \int_{\mathbb{R}^n} |\Delta_f(\chi u_{\mathbf{x}_0;1})|^2\,d\nu_{-\tau} d\tau &\leq C(\Lambda,\theta,\Upsilon). \label{eq integral deltaf bound}
        \end{align}
    \end{subequations}
\end{lemma}

\begin{proof}
    By replacing $u$ with $u_{\mathbf{x}_0;1}$, we may assume that $\mathbf{x}_0=\mathbf{0}$ and $u_{\mathbf{x}_0;1}=u$. Note that if $g,h \in C_c^{\infty}(\mathbb{R}^n \times [-\tau_2,-\tau_1])$, then
    \begin{align}
        \frac{d}{dt}\int_{\R^n} gh\, d\nu_t=\int_{\R^n} h\Box g\, d\nu_{t}-\int_{\R^n} g (\Box^* h+2 \cd f\cdot \cd h)\,d\nu_{t}.\label{eq time derivative of product} 
    \end{align}
    Further,
    \begin{align}
        \Box^* e^{\alpha f} + 2\nabla f\cdot \nabla e^{\alpha f}  = \left( 2(1-\alpha)f-n \right) \frac{\alpha}{2\tau} e^{\alpha f}.\label{eq box star for e alpha}
    \end{align}
    Using \eqref{eq time derivative of product} and \eqref{eq box star for e alpha}, we compute
    \begin{align*} 
        \frac{d}{dt} \int_{\mathbb{R}^n} \tau |\nabla(\chi  u)|^2 e^{\alpha f}\,d\nu_t =&  \int_{\mathbb{R}^n} e^{\alpha f} \Box (\tau|\nabla (\chi u)|^2)\, d\nu_t - \tau\int_{\mathbb{R}^n} |\nabla (\chi u)|^2 \left(\Box{^\ast} ( e^{\alpha f})+2 \nabla f\cdot \nabla e^{\alpha f}\right)\,d\nu_t \\
        =& -2\tau \int_{\mathbb{R}^n}
        |\nabla^2 (\chi u)|^2  e^{\alpha f} d\nu_t +2\tau  \int_{\mathbb{R}^n}  e^{\alpha f} \nabla (\chi u)\cdot \nabla \Box (\chi u)\, d\nu_t\\
        &+\int_{\mathbb{R}^n} |\nabla (\chi u)|^2 \left( \alpha(n-2(1-\alpha)f) -1 \right) e^{\alpha f}d\nu_t
        \eqqcolon I_1+I_2+I_3.
    \end{align*}
    Since $\alpha \leq \frac{1}{2n}$, it follows that
    \begin{align}
        I_3\leq -\frac{1}{2} \int_{\mathbb{R}^n} |\nabla (\chi u)|^2 e^{\alpha f}d\nu_t.\label{eq estimate I3}
    \end{align}
    We now estimate $I_2$. Using \eqref{eq-product rule for pdt-L}, we get
    \begin{align} \label{eq:pointwiseboxubound} 
    \begin{split}
        \verts{\Box (\chi u)}&\leq \verts{\cd (\chi u)\cdot\cd a}+ \verts{(a-I) \cd^2 (\chi u)}+ \verts{(\pdt -\cL) (\chi u)}\\
        &\leq C(\verts{u}+\verts{\nabla u})\cdot (\lambda \chi+ \verts{\partial_t \chi} + \verts{\Delta \chi}+ \verts{\nabla \chi}) + C\lambda^{1-\frac{\Upsilon}{2}} \tau^{\Upsilon} \verts{\nabla ^2 (\chi u)}.
        \end{split}
    \end{align}
    We can integrate the Bochner formula for $\Delta_{(1-\alpha)f}|\nabla (\chi u)|^2$ by parts to obtain
    \begin{align} 
        \tau \int_{\mathbb{R}^n} (\Delta_{(1-\alpha)f}(\chi u))^2 e^{\alpha f} d\nu_t = \int_{\mathbb{R}^n} \left( \tau |\nabla^2 (\chi u)|^2 + (1-\alpha)|\nabla (\chi u)|^2 \right) e^{\alpha f} \,d\nu_t.\label{eq shifted Bochner}
    \end{align}
    For $\epsilon>0$ to be determined, we integrate by parts and use \eqref{eq:pointwiseboxubound} and \eqref{eq shifted Bochner} to get
    \begin{align*}
        &\verts*{ 2\int_{\mathbb{R}^n} (\nabla (\chi u)\cdot \nabla \Box (\chi u))e^{\alpha f} \,d\nu_t} 
        \\
        &= \verts*{2\int_{\R^n} \Delta_{(1-\alpha) f} (\chi u)\Box(\chi u) e^{\alpha f} \,d\nu_{t}}\leq  \epsilon \int_{\mathbb{R}^n} (\Delta_{(1-\alpha)f} (\chi u))^2 e^{\alpha f}d\nu_t + \epsilon^{-1} \int_{\mathbb{R}^n} (\Box (\chi u))^2 e^{\alpha f}\,d\nu_t\\
        &\leq  \epsilon \int_{\mathbb{R}^n} |\nabla^2 (\chi u)|^2 e^{\alpha f}d\nu_t + \frac{\epsilon}{\tau} \int_{\mathbb{R}^n} |\nabla (\chi u)|^2 e^{\alpha f}\,d\nu_t \\ 
        & \quad+ C\epsilon^{-1} \int_{\R^n} (|u|^2+|\nabla u|^2)(\lambda \chi^2 + |\partial_t \chi| + |\nabla \chi|+|\Delta \chi|)^2 e^{\alpha f} \,d\nu_t + C \epsilon^{-1}\lambda^{1-\frac{\Upsilon}{2}}  \int_{\mathbb{R}^n} |\nabla^2 (\chi u)|^2 e^{\alpha f}d\nu_t\\
        &\leq  (\epsilon+C\epsilon^{-1}\lambda^{1-\frac{\Upsilon}{2}}) \int_{\mathbb{R}^n} |\nabla^2 (\chi u)|^2 e^{\alpha f}d\nu_t+ (\frac{\epsilon}{\tau} + C\epsilon^{-1}\lambda) \int_{\R^n} (|\chi u|^2 +|\nabla (\chi u)|^2) e^{\alpha f} d\nu_t \\
        &\quad + C \epsilon^{-1} \int_{\R^n} (|u|^2 + |\nabla u|^2) (|\partial_t \chi|+|\nabla \chi| + |\Delta \chi|)^2 e^{\alpha f} d\nu_t.
    \end{align*}
    Therefore, if $\epsilon\leq \ol{\epsilon}$ and $\lambda \leq \ol{\lambda}(\epsilon, \Lambda)$, we can use Lemma \ref{lem:cutoff} to get
    \begin{align*}
        I_1+I_2+I_3&\leq \tau(-2+\epsilon+C\epsilon^{-1}\lambda^{1-\frac{\Upsilon}{2}}) \int_{\mathbb{R}^n} |\nabla^2 (\chi u)|^2 e^{\alpha f}d\nu_t \\
        &\qquad + (\epsilon + C\epsilon^{-1}\lambda) \int_{\R^n} |\chi u|^2 e^{\alpha f} d\nu_t+(-\frac{1}{2}+\epsilon + C\epsilon^{-1}\lambda )\int_{\R^n}|\nabla (\chi u)|^2) e^{\alpha f} d\nu_t \\
        &\qquad + C \epsilon^{-1} \int_{\R^n} (|u|^2 + \tau|\nabla u|^2) (|\partial_t \chi|+|\nabla \chi| + |\Delta \chi|)^2 e^{\alpha f} d\nu_t\\
        &\leq -\tau\int_{\mathbb{R}^n} |\nabla^2 (\chi u)|^2 e^{\alpha f}d\nu_t+2\epsilon \widehat{H}^{u}(\tau) + C(\alpha,\Upsilon,\Lambda)\lambda^{\Upsilon}e^{-\frac{1-\alpha}{8\lambda^{\Upsilon}|t|^{1-2\Upsilon}}} \widehat{H}^{u}(\tau).
    \end{align*}
    We integrate $I_1+I_2+I_3$ against $\frac{1}{\widehat{H}^{u}(\tau)}$, integrate by parts, use \eqref{eq:Hoscillation}, \eqref{eq derivative of log H}, and Lemma \ref{lem:cutoff} to obtain
    \begin{align*} 
        &\int_{-\tau_2}^{-\tau_1} \frac{\tau}{\widehat{H}^{u}(\tau)} \int_{\mathbb{R}^n} |\nabla^2 (\chi u)|^2 e^{\alpha f} d\nu_t dt \\
        &\leq -\int_{-\tau_2}^{-\tau_1} \frac{1}{\widehat{H}^{u}(\tau)} \frac{d}{dt} \left( \int_{\mathbb{R}^n} \tau |\nabla (\chi u)|^2 e^{\alpha f}d\nu_t \right) dt + C(\alpha,\Upsilon,\Lambda)(\lambda^{\Upsilon}e^{-\frac{1-\alpha}{2\lambda^{\Upsilon}|t_2|^{1-2\Upsilon}}}+2\epsilon)|\tau_2-\tau_1| \\
        &\leq  \frac{\tau_2}{\widehat{H}^{u}(\tau_1)}  \int_{\mathbb{R}^n} |\nabla (\chi u)|^2 e^{\alpha f} d\nu_{t_2} - \int_{-\tau_2}^{-\tau_1} \frac{d}{dt} \log \widehat{H}^{u}(\tau)\frac{1}{\widehat{H}^{u}(\tau)} \int_{\mathbb{R}^n} \tau |\nabla (\chi u)|^2 e^{\alpha f}d\nu_t dt  \\ 
        & \quad+ C(\alpha,\Upsilon,\Lambda)(\lambda^{\Upsilon}e^{-\frac{1-\alpha}{8\lambda^{\Upsilon}|t_2|^{1-2\Upsilon}}}+2\epsilon)|\tau_2-\tau_1|\\ 
        &\leq C(\alpha,\Upsilon,\Lambda)\verts{t_2-t_1} +  C(\Lambda,\alpha) \log \left( \frac{\tau_2}{\tau_1} \right).
    \end{align*}
    This proves \eqref{eq integral hessian estimate}. We now integrate \eqref{eq:pointwiseboxubound} against $\tau \widehat{H}^u(\tau)^{-1}$ and apply Lemma \ref{lem:cutoff} and \eqref{eq integral hessian estimate} to obtain
    \begin{align*}
        &\int_{-\tau_2}^{-\tau_1} \frac{\tau}{\widehat{H}^u(\tau)} \int_{\mathbb{R}^n} |\Box (\chi u)|^2 e^{\alpha f} d\nu_{t} dt\\
        &\leq  C\lambda^2 \int_{-\tau_2}^{-\tau_1} \frac{\tau}{\widehat{H}^u(\tau)} \int_{\mathbb{R}^n} (|(\chi u)|^2 +|\nabla (\chi u)|^2) e^{\alpha f} d\nu_{t}dt \\
        &\qquad+ C \int_{-\tau_2}^{-\tau_1} \frac{\tau}{\widehat{H}^u(\tau)} \int_{\mathbb{R}^n} (|u|^2 +|\nabla u|^2 )( |\partial_t \chi|^2+|\Delta \chi|^2+|\nabla \chi|^2) e^{\alpha f} d\nu_{t}dt\\
        &\qquad+ C\lambda^{2-\Upsilon} \tau_2^{2\Upsilon} \int_{-\tau_2}^{-\tau_1} \frac{\tau}{\widehat{H}^u(\tau)} \int_{\mathbb{R}^n} |\nabla^2 (\chi u)|^2 d\nu_{t}dt\\
        &\leq C(\alpha,\Upsilon,\Lambda)\lambda^{\Upsilon}|\tau_2-\tau_1| + C\lambda^{2-\Upsilon} \tau_2^{2\Upsilon} \left(C(\alpha,\Upsilon,\Lambda) |\tau_2-\tau_1|+  C(\Lambda,\alpha) \log \left( \frac{\tau_2}{\tau_1} \right) \right) .
    \end{align*}
    This proves \eqref{eq integral box estimate}.
    
    Finally, we use \eqref{eq shifted Bochner} with $\alpha\leftarrow 0$, \eqref{eq integral hessian estimate} and \eqref{lem:cutoff} to get
    \begin{align*}
        \int_{-\tau_2}^{-\tau_1}\frac{\tau}{\widehat{H}^{u}(\tau)} \int_{\mathbb{R}^n} (\Delta_{f}(\chi u))^2 d\nu_t &=  \int_{-\tau_2}^{-\tau_1}\frac{\tau}{\widehat{H}^{u}(\tau)}\int_{\mathbb{R}^n} \left( |\nabla^2 (\chi u)|^2 + \frac{1}{\tau}|\nabla (\chi u)|^2 \right) \,d\nu_t dt\\
        &\leq C(\Lambda,\Upsilon)|\tau_2-\tau_1| +  C(\Lambda,\alpha) \log \left( \frac{\tau_2}{\tau_1} \right).
    \end{align*}
\end{proof}

In the following proposition, we prove almost monotonicity and rigidity of the generalized frequency when the scales are comparable.
\begin{proposition}\label{weakmonotonicitylemma}
    The following holds if $\lambda \leq \overline{\lambda}(\Lambda,\Upsilon)$. Fix $\theta\in (0,1]$.
    \begin{enumerate}[label=(\arabic*)]
        \item \label{weakmonotonicity} For all $\mathbf{x}_0 \in P(\mathbf{0},10)$ and $0<\theta \tau_2<\tau_1 < \tau_2 \leq 1$, we have
        \begin{align*}
            &\left| \widehat{N}_{\mathbf{x}_0}^{u}(\tau_2) - \widehat{N}_{\mathbf{x}_0}^{u}(\tau_1) - \int_{\tau_1}^{\tau_2} \frac{4\tau}{\widehat{H}_{\mathbf{x}_0}^{u}(\tau)} \int_{\mathbb{R}^n} \left( \Delta_f (\chi u_{\mathbf{x}_0;1}) + \frac{\widehat{N}_{\mathbf{x}_0}^{u}(\tau)}{2\tau} (\chi u_{\mathbf{x}_0;1})\right)^2 \,d\nu_{-\tau} d\tau \right| \\
            &\leq C(\Lambda,\theta,\Upsilon)\lambda^{\frac{\Upsilon}{2}}\tau_2^{\Upsilon}. 
        \end{align*}
         
        \item \label{caloricapproxoninterval} For any $\epsilon \in (0,1]$, if $4\tau_1 \leq \tau_2$ and $\widehat{N}_{\mathbf{x}_0}^{u}(\tau_1) >\widehat{N}_{\mathbf{x}_0}^{u}(\tau_2)-\epsilon$, then there exists $m\in \mathbb{N}_0$ such that 
        \begin{align*}
            &\sup_{\tau\in [\tau_1,\tau_2]} |\widehat{N}_{\mathbf{x}_0}^{u}(\tau)-m| +
            \int_{\tau_1}^{\tau_2} \frac{4\tau}{\widehat{H}_{\mathbf{x}_0}^{u}(\tau)} \int_{\mathbb{R}^n} \left( \Delta_f (\chi u_{\mathbf{x}_0;1}) + \frac{m}{2\tau} (\chi u_{\mathbf{x}_0;1})\right)^2 \,d\nu_{-\tau} d\tau\\
            &< C\epsilon + C(\Lambda,
            \theta,\Upsilon)\lambda^{\frac{\Upsilon}{2}}\tau_2^{\Upsilon}.
        \end{align*}
    \end{enumerate}
\end{proposition}

\begin{proof} 
    By replacing $u$ with $u_{\mathbf{x}_0;1}$, we may assume that $\mathbf{x}_0=\mathbf{0}$ and $u_{\mathbf{x}_0;1}=u$.
    
    \ref{weakmonotonicity} 
    By Lemma \ref{lemma-monotonicity-formulae-for-the-energy-functionals}, we can estimate 
    \begin{align*} 
        &\left| (\widehat{N}^{u})(\tau_2)- (\widehat{N}^{u})(\tau_1)- \int_{\tau_1}^{\tau_2}\frac{4 \tau}{\widehat{H}^{u}(\tau)} \int_{\mathbb{R}^n} \left( \Delta_f (\chi u) + \frac{\widehat{N}^{u}(\tau)}{2\tau} (\chi u) \right)^2 d\nu_{-\tau}d\tau \right|\\
        &=\verts*{\int_{\tau_1}^{\tau_2}\frac{4 \tau}{\widehat{H}^{u}(\tau)} \int_{\R^n} \left(\Delta_f (\chi u) + \frac{\widehat{N}^{u}(\tau)}{2\tau} (\chi u)\right)\Box (\chi u) \, d\nu_{-\tau}d\tau}\\
        &\leq  4 \left( \int_{\tau_1}^{\tau_2}  \frac{\tau}{\widehat{H}^{u}(\tau)}\int_{\mathbb{R}^n}\left( \Delta_f (\chi u) + \frac{\widehat{N}^{u}(\tau)}{2\tau} (\chi u) \right)^2 d\nu_{0\tau} d\tau \right)^{\frac{1}{2}} \left( \int_{2\tau_1}^{\tau_2}  \frac{\tau}{\widehat{H}^{u}(\tau)}\int_{\mathbb{R}^n} |\Box (\chi u)|^2 d\nu_{-\tau} d\tau\right)^{\frac{1}{2}}\\
        &\leq C(\Lambda,\theta,\Upsilon)\lambda^{\frac{\Upsilon}{2}}\tau_2^{\Upsilon}. 
    \end{align*}
    
    \ref{caloricapproxoninterval} By \ref{weakmonotonicity}, we have
    \begin{align*}
        &\int_{\tau_2}^{\tau_1} \frac{4\tau}{\widehat{H}^{u}(\tau)} \int_{\mathbb{R}^n} \left( \Delta_f(\chi u)+\frac{\widehat{N}^{u}(\tau)}{2\tau} (\chi u) \right)^2 d\nu_{-\tau}d\tau \leq C\epsilon + C(\Lambda,
            \theta,\Upsilon)\lambda^{\frac{\Upsilon}{2}}\tau_2^{\Upsilon}.
    \end{align*}
    Choose $\tau_0 \in [\tau_2,\tau_1]$ such that 
    \begin{align*}
        \frac{4\tau_0}{\widehat{H}^{u}(\tau_0)} \int_{\mathbb{R}^n} \left( \Delta_f (\chi u)+\frac{\widehat{N}^{u}(\tau_0)}{2\tau_0}(\chi u) \right)^2 d\nu_{-\tau_0}  \leq \frac{C\epsilon + C(\Lambda,\theta,\Upsilon)\lambda^{\frac{\Upsilon}{2}}\tau_2^{\Upsilon}}{\tau_2-\tau_1}.
    \end{align*}
    By Lemma \ref{lemma almost eigenvalue equation}, there exists an integer $m\leq C(\Lambda)$ such that
    \begin{align*} 
        |\widehat{N}^{u}(\tau_0)-m| \leq  \frac{80 \tau_0^2}{\widehat{H}^u(\tau_0)} \int_{\mathbb{R}^n} \left( \Delta_f u + \frac{\widehat{N}^{u}(\tau_0)}{2\tau_0} \right)^2 d\nu_{-\tau_0} \leq  \frac{\tau_2}{\tau_2-\tau_1}\left(C\epsilon + C(\Lambda,\theta,\Upsilon)\lambda^{\frac{\Upsilon}{2}}\tau_2^{\Upsilon}\right).
    \end{align*}
    On the other hand, for any $[t_1',t_2']\subseteq [t_1,t_2]$, we again use \ref{weakmonotonicity} to get
    \begin{align*}
        |\widehat{N}^{u}(\tau_2') -\widehat{N}^{u}(\tau_1')|\leq C\epsilon + C(\Lambda,\theta,\Upsilon)\lambda^{\frac{\Upsilon}{2}}\tau_2^{\Upsilon}.
    \end{align*}
    Therefore, \ref{caloricapproxoninterval} immediately follows. 
\end{proof}

Using the fact that the generalized frequency is pinched near integers, we improve its almost monotonicity at two comparable scales to two arbitrary scales.
\begin{corollary} \label{cor: strong monotonocity} 
    The following holds if $\lambda \leq \overline{\lambda}(\Lambda,\Upsilon)$. For any $\mathbf{x}_0\in P(\mathbf{0},10)$ and $0<\theta\tau_2<\tau_1 \leq \tau_2 \leq 1$, we have 
    \begin{align*}
        \widehat{N}_{\mathbf{x}_0}^u(\tau_2) \geq \widehat{N}_{\mathbf{x}_0}^u(\tau_1)-C(\Lambda,\Upsilon)\lambda^{\frac{\Upsilon}{2}}\tau_2^{\Upsilon}.
    \end{align*}
\end{corollary}

\begin{proof} 
    By replacing $u$ with $u_{\mathbf{x}_0;1}$, we may assume $\mathbf{x}_0=0$. For $B=B(\Lambda,\Upsilon)$ to be determined, suppose by way of contradiction that there exists $\tau_1 \in (0,\tau_2]$ such that $\tau_1 \leq \tau_2$ and $\widehat{N}^u(\tau_2)<\widehat{N}^u(\tau_1)-B\lambda^{\frac{\Upsilon}{2}}\tau_2^{\Upsilon}$. Choose $\epsilon' \in [\frac{B}{2}\lambda^{\frac{\Upsilon}{2}}\tau_2^{\Upsilon},B\lambda^{\frac{\Upsilon}{2}}\tau_2^{\Upsilon}]$ such that $|\widehat{N}^u(\tau_2)+\epsilon'-m|\geq \frac{B}{4}\lambda^{\frac{\Upsilon}{2}}\tau_2^{\Upsilon}$ for all $m\in \mathbb{N}_0$. Choose $\tau_{\ast} \in [\tau_1,\tau_2]$ such that $\widehat{N}^u(\tau_{\ast}) = \widehat{N}^u(\tau_2)+\epsilon'$ but $\widehat{N}^u(\tau)\leq \widehat{N}^u(\tau_2)+\epsilon'$ for all $\tau \in [\tau_{\ast},\tau_2]$. If $\tau_{\ast} \geq \frac{1}{4}\tau_2$, then Proposition \ref{weakmonotonicitylemma}\ref{weakmonotonicity} yields a contradiction if we take $B\geq \underline{B}(\Lambda,\Upsilon)$. Thus $\tau_{\ast} \leq \frac{1}{4}\tau_2$. If $\lambda \leq \overline{\lambda}(\Lambda,\Upsilon)$, then Proposition \ref{weakmonotonicitylemma}\ref{caloricapproxoninterval} with $\tau_1 \leftarrow \tau_{\ast}$, $\tau_2 \leftarrow 4\tau_{\ast}$, gives
    \begin{align*}
        \frac{B}{2}\lambda^{\frac{\Upsilon}{2}} \tau_2^{\Upsilon}\leq |\widehat{N}^u(\tau_2)+\epsilon'-m|=|\widehat{N}^{u}(\tau_{\ast})-m| \leq C(\Lambda,\Upsilon)\lambda^{\frac{\Upsilon}{2}}\tau_*^{\Upsilon} \leq C(\Lambda,\Upsilon)\lambda^{\frac{\Upsilon}{2}}\tau_2^{\Upsilon}
    \end{align*}
    for some $m\in \mathbb{N}_0$, since $\widehat{N}^u(\tau) \leq \widehat{N}^u(\tau_{\ast})$ for all $\tau \in [\tau_{\ast},4\tau_{\ast}]$. If $B\geq \underline{B}(\Lambda,\Upsilon)$, this yields a contradiction, and the claim follows.
\end{proof}

Using a discrete version of \eqref{ineq: N refined}, we prove the following analog of Lemma \ref{lemma-refined-monotonicity-of-frequency} where we prove a definite drop in the frequency at small scale given that the frequency at larger scale is a non-integer.
\begin{lemma} \label{lem:discreteloj}
    The following holds if $\lambda\leq \ol{\lambda}(\Lambda,\Upsilon)$. If $\widehat{N}(\tau)\le m+1-2\epsilon$ and $\lambda^{\frac{\Upsilon}{4}}\tau^{\frac{\Upsilon}{2}}<\epsilon$ for some $\epsilon\in(0,\frac{1}{4})$, then 
    \begin{align*}
        \widehat{N}\left(\left(\frac{\epsilon}{1-\epsilon}\right)^{60}\tau \right) \le m+\epsilon.
    \end{align*}
\end{lemma}

\begin{proof}   
    Define $\tau_j=4^{-j}\tau$. Using Lemma \ref{lemma almost eigenvalue equation} and Lemma \ref{weakmonotonicitylemma}\ref{weakmonotonicity} there exist a non-negative integer $m_j\leq C(\Lambda)$ and $\tau_j^*\in [\tau_{j+1},\tau_j]$ such that
    \begin{align}\label{eq pre recursion}
        \begin{split}
            (\widehat{N}(\tau_j^*)-m_j)(m_j+1-\widehat{N}(\tau_j^*))&\leq \frac{20(\tau_j^*)^2}{\widehat{H}^{u}(\tau_j^*)} \int_{\mathbb{R}^n} \left( \Delta_f (\chi u)+\frac{\widehat{N}(\tau_j^*)}{2\tau_j^*}(\chi u) \right)^2 d\nu_{-\tau_j^*} \\
            &\leq \frac{20}{3}\left(\widehat{N}(\tau_j)-\widehat{N}(\tau_{j+1}) +C(\Lambda,\Upsilon)\lambda^{\frac{\Upsilon}{2}}\tau_j^{\Upsilon}\right).
        \end{split}
    \end{align}
    We can assume that $m_j=m$ and $\widehat{N}(\tau_{j+1})-m-C(\Lambda,\Upsilon) \lambda^{\frac{\Upsilon}{2}}\tau_j^\Upsilon>\epsilon$. Using 
    \begin{align*}
        \widehat{N}(\tau_{j+1})- C(\Lambda,\Upsilon)\lambda^{\frac{\Upsilon}{2}}\tau_j^{\Upsilon}\leq \widehat{N}(\tau_j^*)\leq \widehat{N}(\tau_j)+C(\Lambda,\Upsilon)\lambda^{\frac{\Upsilon}{2}}\tau_j^{\Upsilon},
    \end{align*}
    and setting $\delta_j^2\coloneqq C(\Lambda,\Upsilon)\lambda^{\frac{\Upsilon}{2}} \tau_{j}^{\Upsilon}$ and $x_j\coloneqq \widehat{N}(\tau_j)-m$, \eqref{eq pre recursion} implies
    \begin{align*}
        (x_{j+1}-\delta_j^2)(1-x_j-\delta_j^2)\leq \frac{20}{3}(x_{j}-x_{j+1})+\delta_j^2,
    \end{align*}
    which implies
    \begin{align*}
        \frac{1}{x_j}\left(\frac{23}{3}-x_j\right)\leq \frac{20}{3}\frac{1}{x_{j+1}}+\frac{\delta_j^2}{\epsilon x_{j+1}}\leq \frac{20+\delta_j}{3}\frac{1}{x_{j+1}}.
    \end{align*}
    Set $y_j\coloneqq \frac{1}{x_j}$. Then
    \begin{align*}
        \frac{23}{20}\left(y_j-1\right)- \frac{\delta_j}{20}\frac{23}{20}\left(y_j-1\right)-\frac{\delta_j}{20}\leq (y_{j+1}-1),
    \end{align*}
    Applying this and $y_j-1=\frac{1}{\widehat{N}(\tau_j)-m}-1\geq \frac{1}{1-\epsilon}-1\geq \epsilon\geq \delta_j$, we get
    \begin{align*}
        \frac{21}{20}(y_j-1)\leq y_{j+1}-1.
    \end{align*}
    Using this linear relation, we get
    \begin{align*}
        \left(\frac{21}{20}\right)^j \frac{\epsilon}{1-\epsilon}+1\leq \left(\frac{21}{20}\right)^j \left(y_0-1\right)+1\leq y_{j+1},
    \end{align*}
    If $\left(\frac{21}{20}\right)^j \frac{\epsilon}{1-\epsilon}+1\geq \frac{1}{\epsilon}$, then $x_{j+1}\leq \epsilon$, which implies $\widehat{N}(\tau_{j+1})\leq m+\epsilon$. In particular, we have $\tau_j= 4^{-\log_\frac{21}{20} \left(\frac{1-\epsilon}{\epsilon}\right)^2}\tau=   \left(\frac{\epsilon}{1-\epsilon}\right)^{2\log_{\frac{21}{20}}4 }\tau$. Therefore, the claim follows.
\end{proof}

\begin{lemma} \label{lem:freqdroppart2}
    The following holds if $\lambda\leq \ol{\lambda}(\Lambda,\Upsilon)$. For any solution to \eqref{part-2-eq-parabolic-equation} satisfying \eqref{eq:strongerdoubling}, and for any $\mathbf{x}_0 \in P(\mathbf{0},1)$,  $\epsilon\in (0,\frac{1}{10})$, if $0<r_1\leq \epsilon^{C(\Lambda,\Upsilon)\Lambda}r_2$ and $r_2 \leq 1$, then there exists $s\in(r_1,r_2)$ such that
    \begin{align*}
        \widehat{N}(\epsilon^{-2} s^2)-\widehat{N}(\epsilon^2 s^2)<\epsilon.
    \end{align*}
\end{lemma}

\begin{proof} 
    This follows from the proof of Lemma \ref{lem: pinched scale}, using Lemma \ref{lem:discreteloj} instead of Lemma \ref{lemma-refined-monotonicity-of-frequency}.
\end{proof}

\subsection{Quantitative uniqueness of approximating caloric polynomials}

As an analog of Theorem \ref{theorem-almost-frequency-cone-implies-unique-geometric-cone}, in Theorem \ref{theorem quantitative uniqueness in general}, we prove the quantitative uniqueness of approximating caloric polynomials under the assumption that the frequency is pinched at two scales. We construct the approximating polynomial using dynamical rescaling and quantify the uniqueness using an $L^2$-distance defined below.

Given a function $w$ on $\R^n\times I$, define the \textit{renormalized parabolic dilation} $\widetilde{w}_{\mathbf{x}_0;s}$ at $\mathbf{x}_0 \in \R^n\times \R$ by 
\begin{align*}
    \widetilde{w}_{\mathbf{x}_0;s}(x)\coloneqq \frac{w(x_0+e^{-\frac{s}{2}} x,t_0-e^{-s})}{\sqrt{H^w_{\mathbf{x}_0}(e^{-\frac{s}{2}})}}.
\end{align*}
When $\mathbf{x}_0=\mathbf{0}$, we write $\widetilde{w}_s\coloneqq \widetilde{w}_{\mathbf{x}_0;s}$. Given $m\in \N_0$, we define the distance between a function and $\cP_m$ as follows. We denote $\nu\coloneqq \nu_{-1}$, and for any $0\leq s_2 <s_1 <\infty$, we define
\begin{align*}
    \Verts{\widetilde{w}-\widehat{\mathcal{P}}_m}_{[s_2,s_1]}\coloneqq \inf_{p\in\mathcal{P}_m}\Verts{\widetilde{w}-\widehat{p}}_{[s_2,s_1]}\coloneqq \inf_{p\in \mathcal{P}_m} \left( \int_{s_2}^{s_1} \Verts{\widetilde{w}_s - \widehat{p}}_{L^2(\mathbb{R}^n,\nu)}^2 \,ds \right)^{\frac{1}{2}},
\end{align*}
where $\widehat{p}(x)\coloneqq p(x,-1)$ and $\widehat{\cP}_m\coloneqq \{\widehat{p}\colon p\in \cP_m\}$.

\begin{remark}\label{remark subadditivity}
    We observe that $\Verts{\cdot -\widehat{\cP}_m}_{[\cdot, \cdot]}$ is sub-additive in the following sense
    \begin{align*}
        \Verts{\widetilde{w} -\widehat{\cP}_m}_{[s_0, s_0+3A]}\leq 10 (\Verts{\widetilde{w} -\widehat{\cP}_m}_{[s_0, s_0+2A]} + \Verts{\widetilde{w} -\widehat{\cP}_m}_{[s_0+A, s_0+3A]}).
    \end{align*}
    Suppose $\widehat{p}_1\in \arg\min \Verts{\widetilde{w}-\widehat{p}}_{[s_0,s_0+2A]}$ and $\widehat{p}_2\in \arg\min \Verts{\widetilde{w}-\widehat{p}}_{[s_0+A,s_0+3A]}$. Then
    \begin{align} \label{eq:consecutivepolynomials}
        \begin{split}
            \left(\int_{\mathbb{R}^n} (\widehat{p}_1-\widehat{p}_2)^2 \,d\nu \right)^{\frac{1}{2}} &= \left( \frac{1}{A} \int_{s_0+A}^{s_0+2A} \int_{\mathbb{R}^n} (\widehat{p}_1-\widehat{p}_2)^2 \,d\nu ds \right)^{\frac{1}{2}} \\
            &\leq \frac{1}{\sqrt{A}}(\Verts{\widetilde{w}-\widehat{\mathcal{P}}_m}_{[s_0,s_0+2A]}+\Verts{\widetilde{w}-\widehat{\mathcal{P}}_m}_{[s_0+A,s_0+3A]}).
        \end{split}
    \end{align}
    Therefore, using \eqref{eq:consecutivepolynomials}, we get
    \begin{align*}
         &\Verts{\widetilde{w} -\widehat{\cP}_m}_{[s_0, s_0+3A]}^2 \\
         &\leq  \int_{s_0}^{s_0+3A}\int_{\R^n} (\widetilde{w}_s-\widehat{p}_1)^2\, d\nu ds \leq \int_{s_0}^{s_0+2A}\int_{\R^n} (\widetilde{w}_s-\widehat{p}_1)^2\, d\nu ds+ \int_{s_0+A}^{s_0+3A}\int_{\R^n} (\widetilde{w}_s-\widehat{p}_1)^2\, d\nu ds\\
         &\leq \Verts{\widetilde{w}-\widehat{\cP}_m}_{[s_0,s_0+2A]}^2+ 2\int_{s_0+A}^{s_0+3A}\int_{\R^n} (\widetilde{w}_s-\widehat{p}_2)^2\, d\nu ds+2\int_{s_0+A}^{s_0+3A}\int_{\R^n} (\widehat{p}_2-\widehat{p}_1)^2\, d\nu ds\\
         &\leq 10\Verts{\widetilde{w}-\widehat{\cP}_m}_{[s_0,s_0+2A]}^2+10\Verts{\widetilde{w}-\widehat{\cP}_m}_{[s_0+A,s_0+3A]}^2.
    \end{align*}
\end{remark}

\textbf{Notation:} Throughout this subsection, we assume that $u$ is a solution to \eqref{part-2-eq-parabolic-equation} satisfying \eqref{eq:strongerdoubling}. By parabolic rescaling, we can assume that $\mathbf{x}_0=\mathbf{0}$ and $u_{\mathbf{x}_0;1}=u$. We will set $w\coloneqq \chi u$ throughout this section. Further, we will write $s\coloneqq e^{-\tau}$. 

\begin{theorem} \label{theorem quantitative uniqueness in general}
    For any $A \in [1,\infty)$, the following holds if $\lambda \leq \overline{\lambda}(A,\Lambda)$ and $\delta \leq \overline{\delta}(A,\Lambda,\Upsilon)$. Suppose
    \begin{align}
        \sup_{s\in [s_2-8A,s_1+8A]} |\widehat{N}^u(e^{-s})-m|=\delta\label{eq frequency pinched hypothesis for quantitative uniqueness}
    \end{align}
    for some $2A<s_2 < s_1-2$. Then there exists $\widehat{p}\in \widehat{\mathcal{P}}_m$ such that
    \begin{subequations}
        \begin{align}
            \Verts{\widetilde{w}-\widehat{p}}_{[s-A,s+A]}
            &<C(A,\Lambda, \Upsilon) \delta \left( e^{-\frac{\Upsilon}{4}(s_1-s)}+e^{-\frac{\Upsilon}{4}(s-s_2)} \right)+C\lambda^{\frac{\Upsilon}{4}}\left(e^{-\Upsilon s}+e^{-\frac{\Upsilon}{2}(s_1+s_2)}\right) \label{eq estimate for polynomial closeness} \\
            |\widehat{N}^u(e^{-s})-m| &<C(A,\Lambda, \Upsilon) \delta \left( e^{-\frac{\Upsilon}{4}(s_1-s)}+e^{-\frac{\Upsilon}{4}(s-s_2)} \right)+C\lambda^{\frac{\Upsilon}{4}}e^{-\Upsilon s}\label{eq estimate for frequency}
        \end{align}
    \end{subequations}
    for all $s\in [s_2,s_1]$.
\end{theorem}

As a first step towards proving Theorem \ref{theorem quantitative uniqueness in general}, we prove a weaker version of quantitative uniqueness when the frequency is pinched at comparable scales.

\begin{proposition} \label{prop:closeheatpolynomial}
    The following holds if $\lambda\leq \ol{\lambda}(\Lambda)$. Fix $\epsilon \in (0,1]$ and $B\in [1,\infty)$. Suppose $\mathbf{x}_0\in P(\mathbf{0},10)$ and $0\leq s_1-2B\leq s_2 \leq s_1-2$ such that
    \begin{align*}
        \widehat{N}_{\mathbf{x}_0}^{u}(\tfrac{1}{2}e^{-s_1}) >\widehat{N}_{\mathbf{x}_0}^{u}(2e^{-s_2})-\epsilon.
    \end{align*}
    Then there exists $m\in \N_0$ and $p\in \cP_m$ such that
    \begin{align*}
       \Verts{\widetilde{w}-\widehat{p}}_{[s_2,s_1]}^2\leq C(B, \Lambda,\Upsilon) \left( \epsilon+\lambda^{\frac{\Upsilon}{2}}e^{-\Upsilon s_2} \right).
    \end{align*}
\end{proposition}

\begin{proof} 
    By replacing $u$ with $u_{\mathbf{x}_0;1}$, we may assume that $\mathbf{x}_0=\mathbf{0}$ and $u_{\mathbf{x}_0;1}=u$. We set $s_i\coloneqq -\log(\tau_i)$ for $i=1,2$.
    
    Let $m\in \N_0$ be as in Proposition \ref{weakmonotonicitylemma}\ref{caloricapproxoninterval}. Let $h(\cdot, t)$ be the $L^2(\mathbb{R}^n,\nu_t)$-orthogonal projection of $(\chi u)(\cdot, t)$ onto $\mathcal{P}_m$. By Lemma \ref{lemma almost eigenvalue equation}, we have
    \begin{equation} \label{eq:closetoprojection} 
        \int_{\tau_1}^{\tau_2} \frac{1}{ \widehat{H}^{u}(\tau)}\int_{\mathbb{R}^n} (\chi u-h)^2 \,d\nu_{-\tau}\frac{d\tau}{\tau} \leq 20 \int_{\tau_1}^{\tau_2}\frac{\tau^2}{\widehat{H}^{u}(\tau)} \int_{\mathbb{R}^n} \left( \Delta_f (\chi u)+\frac{\widehat{N}^{u}(\tau)}{2\tau} \right)^2\, d\nu_{-\tau}\frac{d\tau}{\tau}.
    \end{equation}
    Note that $h$ is not necessarily in $\cP_m$. We now construct $p\in \cP_m$ that approximates $h$ in an optimal way as explained below. Let $\{ p_i\}_{i=1}^N$ be a basis of $\mathcal{P}_m$ satisfying
    $\int_{\mathbb{R}^n} p_i p_j \,d\nu_t = \delta_{ij}\tau^m$ and set $h_i(t)\coloneqq  \frac{1}{\tau^m}\int_{\mathbb{R}^n} \chi u p_i\,d\nu_t$ so that $h=\sum_{i=1}^N h_ip_i$. 
    
    Define for $1\leq j\leq N$
    \begin{align*}
        \ol{h}_j(s)\coloneqq  \sqrt{\frac{e^{-ms}}{\widehat{H}^u(e^{-s})}}h_j(-e^{-s}), \qquad a_j \coloneqq \fint_{s_1}^{s_2} \ol{h}_j(s)\,ds, \qquad p\coloneqq  \sum_{i=1}^N a_i p_i.
    \end{align*}
    Note that our choice of $a_j$ minimizes $\sum_{j=1}^N\int_{\tau_1}^{\tau_2} \left( \sqrt{\frac{\tau^m}{\widehat{H}^u(\tau)}}h_j(-\tau)-a_j\right)^2 \frac{d\tau}{\tau}$ so that
    \begin{align}
        \int_{s_1}^{s_2} (\ol{h}_j(s)-a_j)\,ds =0.\label{eq definition of aj}
    \end{align}
    Further, a direct computation yields
    \begin{align*} 
        \ol{h}_j'(s) = \frac{1}{2}(\widehat{N}^u(e^{-s})-m)\ol{h}_j(s)+\frac{e^{\frac{m}{2}s-s}}{\sqrt{\widehat{H}^u(e^{-s})}}\int_{\mathbb{R}^n} \Box{(\chi u)p_j \,d\nu_{-e^{-s}}},
    \end{align*}
    from which we can use H\"older's inequality to estimate
    \begin{align*} 
        |\ol{h}_i'(s)|\leq \frac{1}{2} |\widehat{N}(e^{-s})-m| + \frac{e^{-s}}{\sqrt{\widehat{H}^u(e^{-s})}} \left( \int_{\mathbb{R}^n} |\Box (\chi u)|^2 \,d\nu_{-e^{-s}} \right)^{\frac{1}{2}}.
    \end{align*}
    Integrating in time and combining with Lemma \ref{lem:HessianL2} yields 
    \begin{align} \label{eq L2 derivative of hj}
        \begin{split}
            \int_{s_2}^{s_1} |\ol{h}_i'(s)|^2 \,ds \leq & \int_{\tau_1}^{\tau_2} \frac{1}{\tau}(\widehat{N}(\tau)-m)^2 \,d\tau + 2\int_{\tau_1}^{\tau_2} \frac{\tau}{\widehat{H}^u(\tau)} \int_{\mathbb{R}^n} |\Box (\chi u)|^2 \,d\nu_{-\tau} d\tau\\
            \leq & \int_{\tau_1}^{\tau_2} \frac{1}{\tau}(\widehat{N}(\tau)-m)^2 \,d\tau + C(B,\Lambda,\Upsilon)\lambda^{\Upsilon} \tau_2^{2\Upsilon}.
        \end{split}
    \end{align}
    Therefore, using \eqref{eq definition of aj}, Poincaré inequality and \eqref{eq L2 derivative of hj}, we get 
    \begin{align*} 
        &\int_{\tau_1}^{\tau_2} \frac{1}{\widehat{H}^u(\tau)} \int_{\mathbb{R}^n} \left( h-\sqrt{\frac{\widehat{H}^u(\tau)}{\tau^m}}p \right)^2 d\nu_{-\tau} \frac{d\tau}{\tau} \\
        &=  \sum_{j=1}^N \int_{\tau_1}^{\tau_2} \left( \sqrt{\frac{\tau^m}{\widehat{H}^u(\tau)}}h_j(-\tau)-a_j\right)^2 \frac{d\tau}{\tau} = \sum_{j=1}^N \int_{s_2}^{s_1}(\ol{h}_j(s)-a_j)^2 \,ds \\
        &\leq  C(s_1-s_2)\sum_{j=1}^N \int_{s_2}^{s_1} (\ol{h}_j'(s))^2 \,ds\leq C(B)\int_{\tau_1}^{\tau_2} \frac{1}{\tau}(\widehat{N}(\tau)-m)^2 \,d\tau + C(B,\Lambda,\Upsilon)\lambda^{\Upsilon} \tau_2^{2\Upsilon}\\
        &\leq C(B)\epsilon^2 +C(B,\Lambda,\Upsilon)\lambda^{\Upsilon} \tau_2^{2\Upsilon}.
    \end{align*}
    This together with \eqref{eq:closetoprojection} gives
    \begin{align*}
        \Verts{\widetilde{w}-\widehat{p}}_{[s_2,s_1]}^2&=\int_{\tau_1}^{\tau_2}\int_{\mathbb{R}^n} \left(\frac{\chi u}{\sqrt{\widehat{H}^u(\tau)}}-\frac{p}{\sqrt{\tau^m}}\right)^2 \,d\nu_{-\tau}\frac{d\tau}{\tau}\\
        &\leq C\epsilon + C(B,\Lambda,\Upsilon) \lambda^{\frac{\Upsilon}{2}}\tau_2^{\Upsilon} +C(B)\epsilon^2 +C(B,\Lambda,\Upsilon)\lambda^{\Upsilon} \tau_2^{2\Upsilon}\leq 
        C(B,\Lambda,\Upsilon) \left( \epsilon+\lambda^{\frac{\Upsilon}{2}}\tau_2^{\Upsilon}\right).
    \end{align*}
\end{proof}

Note that the estimates in Proposition \ref{prop:closeheatpolynomial} worsen as $\frac{\tau_2}{\tau_1}\to \infty$. To improve the estimate and prove Theorem \ref{theorem quantitative uniqueness in general}, we view parabolic rescaling as a dynamical system. We then analyze the stability of caloric approximation of parabolic rescaling.

To this end, we observe that $\widetilde{w}_s$ satisfies the following evolution equation:
\begin{align}
    \partial_{s} \widetilde{w}_{s} = \Delta_f \widetilde{w}_s + \frac{m}{2}\widetilde{w}_{s} + \frac{1}{2}(N(\widetilde{w}_s)-m)\widetilde{w}_{s}+g_{s},\label{eq evolution equation for parabolic scaling}
\end{align}
where $m\in \R$ and
\begin{align} \label{eq:defofg}
    g_s \coloneqq \frac{(\widetilde{\tau\Box w})_s}{\sqrt{H^w(e^{-s})}}-\frac{\widetilde{w}_s}{2H^w(e^{-s})} \int_{\mathbb{R}^n} (\tau w \Box w) d\nu_{-e^{-s}}. 
\end{align}
Moreover, a change of variables implies that
\begin{align} 
    \int_{\mathbb{R}^n} \widetilde{w}_s^2 d\nu =\frac{1}{H^w(e^{-\frac{s}{2}})}\int_{\mathbb{R}^n} w^2(e^{-\frac{s}{2}}x,-e^{-s})\frac{e^{-\frac{|x|^2}{4}}}{(4\pi)^{\frac{n}{2}}}dx = \frac{1}{H^w(e^{-\frac{s}{2}})} \int_{\mathbb{R}^n} w^2 d\nu_{-e^{-s}}=1.\label{eq wtilde is normalized}
\end{align}

In the lemmata below, we show that if $u$ is a solution to \eqref{part-2-eq-parabolic-equation} satisfying \eqref{eq:strongerdoubling} and is close to a heat polynomial at some scale, then the behavior of the corresponding function $\widetilde{u}$ satisfying \eqref{eq evolution equation for parabolic scaling} is modeled on that of the corresponding linearization at $\mathcal{P}_m$:
\begin{align}
    \partial_{s} \widetilde{w}_{s} = \Delta_f \widetilde{w}_s + \frac{m}{2}\widetilde{w}_{s},
\end{align}
where $m\in \N_0$.

As a first step, we show that the frequency of a function and that of its caloric approximation are close.
\begin{lemma}\label{lemma closeness of frequency of u and caloric approximation}
    For any $\widehat{p}\in \widehat{\cP}_m$ and $w$ satisfying the mild growth condition \eqref{eq mild growth} and $\int_{\mathbb{R}^n} w^2\,d\nu=1$, we have
    \begin{align*}
        \verts{N^w(1)-N^{\widehat{p}}(1)}\leq 4\int_{\mathbb{R}^n} |\nabla (w-\widehat{p})|^2 \,d\nu + 8m\int_{\mathbb{R}^n} (w-\widehat{p})^2 \,d\nu.
    \end{align*}
\end{lemma}

\begin{proof}
    Let $\widehat{q} \in \widehat{\mathcal{P}}_m$ be the $L^2(\mathbb{R}^n,\nu)$ projection of $w$ onto $\widehat{\mathcal{P}}_m$. Then from
    \begin{align*}
        \int_{\mathbb{R}^n} (w-\widehat{q})\widehat{q}\, d\nu =0, \qquad \int_{\mathbb{R}^n} w^2 \,d\nu =1,
    \end{align*}
    we can estimate
    \begin{align}
        \int_{\R^n} \cd (w-\widehat{q})\cdot \cd \widehat{q}\,d\nu=  -\int_{\R^n} (w-\widehat{q})\Delta_f \widehat{q}\,d\nu=\frac{m}{2}\int_{\R^n} (w-\widehat{q})\widehat{q}\,d\nu=0.\label{eq vanishing of derivative}
    \end{align}
    Furthermore,
    \begin{align}
        1-\int_{\mathbb{R}^n} \widehat{q}^2 d\nu=\int_{\mathbb{R}^n} (w^2-\widehat{q}^2)d\nu = \int_{\mathbb{R}^n} (w-\widehat{q})^2 d\nu.\label{eq 1 minus norm of q}
    \end{align}
    In particular, $\widehat{h}\coloneqq  \frac{\widehat{q}}{\Verts{\widehat{q}}_{L^2(\mathbb{R}^n,\nu)}}$ satisfies 
    \begin{align}
        \Verts{\widehat{q}-\widehat{h}}_{L^2(\mathbb{R}^n,\nu)}= |\Verts{\widehat{q}}_{L^2(\mathbb{R}^n,\nu)}-1|= \frac{|\Verts{\widehat{q}}_{L^2(\mathbb{R}^n,\nu)}^2-1|}{\Verts{\widehat{q}}_{L^2(\mathbb{R}^n,\nu)}+1} \leq \Verts{w-\widehat{q}}^2_{L^2(\R^n,\nu)}.\label{eq hat q minus hat h}
    \end{align}
    Using \eqref{eq vanishing of derivative}, \eqref{eq 1 minus norm of q}, and \eqref{eq hat q minus hat h}, we obtain that for any $\widehat{p}\in \widehat{\cP}_m$,
    \begin{align*}
        |N^{w}(1)-N^{\widehat{p}}(1)| &=|N^{w}(1)-N^{\widehat{h}}(1)| = 2\left| \int_{\mathbb{R}^n}(|\nabla w|^2-|\nabla \widehat{h}|^2)d\nu \right| \\
        &\leq 2\left| \int_{\mathbb{R}^n}(|\nabla w|^2-|\nabla \widehat{q}|^2)d\nu \right| +2\left|\int_{\mathbb{R}^n} (|\nabla \widehat{q}|^2-|\nabla \widehat{h}|^2)\,d\nu \right|\\ 
        &= 2\left| \int_{\mathbb{R}^n} \nabla (w-\widehat{q})\cdot \nabla (w+\widehat{q})d\nu\right| + 2\left| \int_{\mathbb{R}^n} \nabla (\widehat{q}-\widehat{h})\cdot \nabla (\widehat{q}+\widehat{h})d\nu  \right|\\
        &= 2 \int_{\mathbb{R}^n} |\nabla (w-\widehat{q})|^2 d\nu +m\left|\int_{\R^n}(\widehat{q}^2-\widehat{h}^2)\,d\nu\right|\\
        &\leq  4\int_{\mathbb{R}^n} |\nabla (w-\widehat{p})|^2 d\nu + 4\int_{\mathbb{R}^n} |\nabla (\widehat{p}-\widehat{h})|^2 d\nu +m\int_{\R^n}(w-\widehat{q})^2\,d\nu\\
        &\leq  4\int_{\mathbb{R}^n} |\nabla (w-\widehat{p})|^2 d\nu + 2m\int_{\mathbb{R}^n} (\widehat{p}-\widehat{h})^2 d\nu +m\int_{\R^n}(w-\widehat{q})^2\,d\nu\\
        &\leq  4\int_{\mathbb{R}^n} |\nabla (w-\widehat{p})|^2 d\nu + 2m\int_{\mathbb{R}^n} (w-\widehat{p})^2 +2(w-\widehat{q})^2+(\widehat{q}-h)^2d\nu \,d\nu\\
        &\leq  4\int_{\mathbb{R}^n} |\nabla (w-\widehat{p})|^2 d\nu + 8m\int_{\mathbb{R}^n} (w-\widehat{p})^2 \,d\nu.
    \end{align*}
\end{proof}

Now we estimate the other terms in \eqref{eq evolution equation for parabolic scaling}.
\begin{lemma} \label{lem:smallerrorterms}
    For any $A \in [1,\infty)$ and any $\sigma \in (0,1]$ the following holds if $\lambda \leq \ol{\lambda}(\Lambda)$. There exists $C=C(A,\Lambda, \Upsilon)$ such that for any $s_0\geq 0$:
    \begin{subequations}
        \begin{align}
            \label{eq:smallerrorterms3}
            \int_{s_0}^{s_0+A} \int_{\mathbb{R}^n} g_s^2 \,d\nu ds &\leq C\lambda^{\Upsilon} e^{-2\Upsilon s_0},\\
            \int_{s_0}^{s_0+A}|\widehat{N}^u(e^{-s})-m|\,ds & \leq C\sigma^{-1} \Verts{\widetilde{w}-\widehat{\cP}_m}_{[s_0-\sigma,s_0+A]}^2+C\lambda^{\Upsilon}e^{-2\Upsilon s_0}.\label{eq:smallerrorterms1} 
        \end{align}
    \end{subequations}
\end{lemma}

\begin{proof}
    We first prove \eqref{eq:smallerrorterms3}. Suppose $0<s_2<s_1$. By combining \eqref{eq:defofg}, Cauchy's inequality, and Lemma \ref{lem:HessianL2} we get
    \begin{align*} 
        \int_{s_2}^{s_1} \int_{\mathbb{R}^n} g_s^2 \,d\nu ds &\leq  2\int_{s_2}^{s_1} \frac{1}{H(e^{-s})} \int_{\mathbb{R}^n} \widetilde{(\tau\Box w)_s^2} \,d\nu ds +2\int_{s_2}^{s_1} \frac{1}{H^w(e^{-s})} \left( \int_{\mathbb{R}^n} \widetilde{w}_s (\widetilde{\tau \Box w})_s  \,d\nu \right)^2 \,ds\\
        &\leq C\int_{e^{-s_1}}^{e^{-s_2}} \frac{\tau}{H(\tau)} \int_{\mathbb{R}^n} (\Box w)^2 \,d\nu_{-\tau} d\tau \leq C(A,\Lambda,\Upsilon)\lambda^{\Upsilon} e^{-2\Upsilon s_2}.
    \end{align*}
    
    We now prove \eqref{eq:smallerrorterms1}. Let $\eta \in C^{\infty}(\mathbb{R})$ be a cutoff function such that $\eta|_{(-\infty,s_0-\sigma]} \equiv0$, $\eta|_{[s_0,\infty)} \equiv 1$, and $|\eta'| \leq 10\sigma^{-1}$. 
    \begin{claim} \label{claim:closetheloop}
        The following holds:
        \begin{align*}
            &2 \int_{s_0-A}^{s_0+A}\int_{\mathbb{R}^n} \eta|\nabla (\widetilde{w}_s-\widehat{p})|^2 d\nu ds\\ 
            &\quad \leq C(A,\Lambda,\Upsilon) \int_{s_0-A}^{s_0+A} \int_{\mathbb{R}^n} (\widetilde{w}_s-\widehat{p})^2 d\nu ds + \frac{1}{20}\int_{s_0-A}^{s_0+A} \eta |\widehat{N}^u(e^{-s})-m|ds+C(A,\Lambda,\Upsilon)\lambda^{\Upsilon} e^{-2\Upsilon s_0}.
        \end{align*}
    \end{claim}
    
    \begin{proof}
        We use
        \begin{align} \label{eq:rescaledeqfordifference}
            \partial_s (\widetilde{w}_s-\widehat{p}) = \Delta_f (\widetilde{w}_s-\widehat{p}) + \frac{m}{2}(\widetilde{w}_s-\widehat{p}) + \frac{1}{2}(\widehat{N}^u(e^{-s})-m)\widetilde{w}_s+ g_s,
        \end{align}
        H\"older's inequality, and \eqref{eq wtilde is normalized}, to estimate
        \begin{align} \label{eq:rescaledthingtoIBP}
            \begin{aligned}
                \frac{d}{ds} \int_{\mathbb{R}^n}(\widetilde{w}_s-\widehat{p})^2 d\nu \leq & -2\int_{\mathbb{R}^n} |\nabla (\widetilde{w}_s-\widehat{p})|^2 d\nu + \left(\frac{m}{2}+C(A,\Lambda,\Upsilon) \right) \int_{\mathbb{R}^n} (\widetilde{w}_s-\widehat{p})^2 d\nu \\ 
                &+ \frac{1}{C(A,\Lambda,\Upsilon)}(\widehat{N}^u(e^{-s})-m)^2 + \int_{\mathbb{R}^n} g_s^2 \,d\nu.
            \end{aligned}
        \end{align}
         Then integrating \eqref{eq:rescaledthingtoIBP} in time against $\eta$ and applying Lemma \ref{lem:smallerrorterms} yields
         \begin{align} \label{eq:rescaledthingIBPed}
            \begin{aligned}
                &\sup_{s\in [s_0,s_0+A]} \int_{\mathbb{R}^n}(\widetilde{w}_s-\widehat{p})^2 d\nu +2 \int_{s_0-A}^{s_0+A}\int_{\mathbb{R}^n} \eta|\nabla (\widetilde{w}_s-\widehat{p})|^2 d\nu ds\\ 
                &\quad \leq2 \int_{s_0-A}^{s_0+A}\left( |\eta'|+  \left(\frac{m}{2}+20 \right) \eta \right) \int_{\mathbb{R}^n} (\widetilde{w}_s-\widehat{p})^2 d\nu ds+ \frac{1}{C(A,\Lambda, \Upsilon)}\int_{s_0-A}^{s_0+A} \eta |\widehat{N}^u(e^{-s})-m|^2ds \\
                &\qquad + 2\int_{s_0-A}^{s_0+A}\int_{\mathbb{R}^n} g_s^2 d\nu ds.
            \end{aligned}
        \end{align}
    \end{proof}
    
    For any $\widehat{p}\in \widehat{\cP}_m$, we can use Lemma \ref{lemma closeness of frequency of u and caloric approximation} with $w\leftarrow \widetilde{w}_s$ to obtain
    \begin{align}\label{eq pointwise freq}
        \begin{split}
            |\widehat{N}^u(e^{-s})-m| = |N^{\widetilde{w}_s}(1)-N^{\widehat{p}}(1)|
            \leq 4\int_{\mathbb{R}^n} |\nabla (\widetilde{w}_s-\widehat{p})|^2 d\nu + 8m\int_{\mathbb{R}^n} (\widetilde{w}_s-\widehat{p})^2 d\nu.
        \end{split}
    \end{align}
    Integrating this in time against $\eta$, applying Claim \ref{claim:closetheloop} and rearranging terms yields the desired estimate:
    \begin{align}\label{eq close the loop}
        \int_{s_0-A}^{s_0+A} \eta(s)|\widehat{N}^u(e^{-s})-m|ds \leq &C(A,\Lambda,\Upsilon) \sigma^{-1}\int_{s_0-\sigma}^{s_0+A} \int_{\mathbb{R}^n} (\widetilde{w}_s-\widehat{p})^2 d\nu ds + C(A,\Lambda,\Upsilon)\lambda^{\Upsilon} e^{-2\Upsilon s_0}.
    \end{align}
\end{proof}

In the lemma below, we improve the closeness to a polynomial in Proposition \ref{prop:closeheatpolynomial} up to higher-order derivatives. We need the following definition:
\begin{align*} 
    \Verts{v}_{H^{-1}(\mathbb{R}^n,\nu)} \coloneqq \sup \left\{ \int_{\mathbb{R}^n} vw\,d\nu \colon \int_{\mathbb{R}^n}(w^2 + |\nabla w|^2)d\nu \leq 1  \right\}. 
\end{align*} 
\begin{lemma} \label{lem:higherordercloseness}
    For any $A\in [1,\infty)$ and any $\sigma \in (0,1]$, the following holds if $\lambda \leq \ol{\lambda}(\Lambda)$. 
    \begin{enumerate}[label=(\arabic*)]
        \item There exists $C=C(A,\Lambda,\Upsilon)$ such that for any $p\in \mathcal{P}_m$ and $s_0\geq 0$, we have
        \begin{align*}
            \sup_{s\in [s_0,s_0+A]} &\int_{\mathbb{R}^n}(\widetilde{w}_s-\widehat{p})^2d\nu + \int_{s_0}^{s_0+A} \int_{\mathbb{R}^n} |\nabla (\widetilde{w}_s -\widehat{p})|^2 d\nu ds + \int_{s_0}^{s_0+A} \Verts{\partial_s(\widetilde{w}_s-\widehat{p})}_{H^{-1}(\mathbb{R}^n,\nu)}^2 ds \\&\leq C(A,\Lambda,\Upsilon)(\sigma^{-1}\Verts{\widetilde{w}-\widehat{p}}_{[s_0-\sigma,s_0+A]}^2 +\lambda^{\Upsilon}e^{-2\Upsilon s_0}).
        \end{align*}
    
        \item \label{lem:timederivative and hessian closeness} For any $\epsilon \in (0,1]$, if $s_0\geq 0$ and $\widehat{N}_{\mathbf{x}_0}^{u}(e^{-(s_0+A)}) >\widehat{N}_{\mathbf{x}_0}^{u}(e^{-s_0})-\epsilon$, then there exists $m\in \mathbb{N}_0$ such that for any $\widehat{p}\in \cP_m$ we have
        \begin{align}
            \begin{split}
                \int_{s_0}^{s_0+A}\int_{\mathbb{R}^n}|\partial_s (\widetilde{w}_s-\widehat{p})|^2+|\nabla^2 (\widetilde{w}_s-\widehat{p})|^2 \,d\nu ds\leq C\epsilon +C(A,\Lambda,\Upsilon)\left(\Verts{\widetilde{w}-\widehat{p}}_{[s_0,s_0+A]}^2 +\lambda^{\frac{\Upsilon}{2}}e^{-\Upsilon s_0}\right).\label{eq hessian estimate for orthogonal complement}
            \end{split}
        \end{align}
    \end{enumerate}
\end{lemma}

\begin{proof} 
    Using \eqref{eq:rescaledthingIBPed}, \eqref{eq close the loop} and recalling that $\eta \in C^{\infty}(\mathbb{R})$ is a cutoff function satisfying $\eta|_{(-\infty,s_0-\sigma]} \equiv0$, $\eta|_{[s_0,\infty)} \equiv 1$, and $|\eta'| \leq 10\sigma^{-1}$, we get the estimate for the first two terms. It remains to bound the third term.
    
    For any $\phi \in C_c^{\infty}(\mathbb{R}^n)$ with $\int_{\mathbb{R}^n}(\phi^2+|\nabla \phi|^2)\,d\nu \leq 1$, we use \eqref{eq:rescaledeqfordifference} to estimate
    \begin{align*} 
        \left| \int_{\mathbb{R}^n} \partial_{s} (\widetilde{w}_s-\widehat{p}) \phi \,d\nu \right| \leq & \left| \int_{\mathbb{R}^n} \nabla (\widetilde{w}_s-\widehat{p}) \cdot  \nabla \phi \,d\nu \right| + \frac{1}{2}|\widehat{N}^u(e^{-s})-m| \left| \int_{\mathbb{R}^n} \widetilde{w}_s \phi \,d\nu \right|\\
        &+\frac{m}{2}\left| \int_{\mathbb{R}^n} (\widetilde{w}_s-\widehat{p})\phi \,d\nu\right|+ \int_{\mathbb{R}^n}g_{s} \phi \,d\nu \\ 
        \leq & \left( \int_{\mathbb{R}^n} |\nabla (\widetilde{w}_s-\widehat{p})|^2  
        \,d\nu \right)^{\frac{1}{2}} + \frac{1}{2}|\widehat{N}^u(e^{-s})-m| \\ 
        &+ \frac{m}{2} \left(\int_{\mathbb{R}^n} (\widetilde{w}_s-\widehat{p})^2 d\nu \right)^{\frac{1}{2}}+ \int_{\mathbb{R}^n}g_s^2 \,d\nu.
    \end{align*}
    We then take the supremum over all such $\phi$, square both sides, integrate in time, and apply \eqref{eq:rescaledthingIBPed} and Lemma \ref{lem:smallerrorterms} to obtain
    so that
    \begin{align*} 
        \int_{s_0}^{s_0+A} \Verts{\widetilde{w}_s-\widehat{p}}_{H^{-1}(\mathbb{R}^n,\nu)}^2 \,ds \leq &C\int_{s_0}^{s_0+A} \int_{\mathbb{R}^n} |\nabla (\widetilde{w}_s-\widehat{p})|^2 \,d\nu ds + \frac{1}{2}\int_{s_0}^{s_0+A}(\widehat{N}^u(e^{-s})-m)^2\,ds \\
        &+C\int_{s_0}^{s_0+A}\int_{\mathbb{R}^n} g_s^2 \,d\nu ds + C(\Lambda)\int_{s_0}^{s_0+A} \int_{\mathbb{R}^n} (\widetilde{w}_s - \widehat{p})^2 \,d\nu ds \\ 
        \leq & C(A,\Lambda,\Upsilon)\sigma^{-1} \int_{s_0-\sigma}^{s_0+A} \int_{\mathbb{R}^n} (\widetilde{w}_s-\widehat{p})^2 \,d\nu ds + C(A,\Lambda,\Upsilon)\lambda^{\Upsilon}e^{-2\Upsilon s_0}.
    \end{align*}

    To prove \eqref{eq hessian estimate for orthogonal complement}, we use the $f$-Bochner formula \eqref{eq: f-Bochner formula} with $w\leftarrow \widetilde{w}-\widehat{p}$ and Proposition \ref{weakmonotonicitylemma}\ref{caloricapproxoninterval} to get
    \begin{align*} 
        \int_{s_0 }^{s_0+A}\int_{\mathbb{R}^n} |\nabla^2 (\widetilde{w}-p)|^2 \,d\nu ds &\leq  \int_{s_0}^{s_0+A}\int_{\mathbb{R}^n} \left( \Delta_f \widetilde{w}+\frac{m}{2}\widehat{p} \right)^2 \,d\nu ds \\
        &\leq  \int_{s_0}^{s_0+A}\int_{\mathbb{R}^n} \left( \Delta_f \widetilde{w}+\frac{m}{2}\widetilde{w} \right)^2 \,d\nu ds + 2m^2\int_{s_0}^{s_0+A}\int_{\mathbb{R}^n} (\widetilde{w} - \widehat{p})^2  \,d\nu ds \\
        &\leq C\epsilon +C(A,\Lambda,\Upsilon)\left(\Verts{\widetilde{w}-\widehat{p}}_{[s_0,s_0+A]}^2 +\lambda^{\frac{\Upsilon}{2}}e^{-\Upsilon s_0}\right).
    \end{align*}
    Then we combine this with \eqref{eq integral box estimate} in Lemma \ref{lem:HessianL2} to get \eqref{eq hessian estimate for orthogonal complement}.
\end{proof}

We now carry out the stability analysis of caloric approximation. Below, we prove that fast growth of closeness to $\cP_m$ propagates down the scale but fast decay of closeness to $\cP_m$ propagates up the scale. The strategy of our proof is similar to that of \cite[Lemma 5.31]{cheeger-tian-1994-cone-structure-at-infinity} (see also \cite[\S II.3]{simongeneral}).
\begin{proposition}[Growth or decay] \label{prop:growthordecay} 
    For any $A \in [1,\infty)$, the following holds when $\lambda \leq \overline{\lambda}(A,\Lambda)$ and $\epsilon \leq \overline{\epsilon}(A,\Lambda)$. If 
    \begin{align*}
       \lambda^{\frac{\Upsilon}{4}}e^{-\Upsilon s} \leq \Verts{\widetilde{w}-\widehat{\mathcal{P}}_m}_{[s+A,s+3A]}<\epsilon
    \end{align*}
    for some $s \geq 0$, then we have the following:
    \begin{enumerate}[label=(\arabic*)]
        \item \label{fastgrowth} 
        If $\Verts{\widetilde{w}-\widehat{\cP}_m}_{[s+A,s+3A]}\geq e^{\frac{1}{2}A} \Verts{\widetilde{w}-\widehat{\cP}_m}_{[s,s+2A]}$, then 
        \begin{align}
            \Verts{\widetilde{w}-\widehat{\cP}_m}_{[s+2A,s+4A]}\geq e^{\frac{1}{2}A} \Verts{\widetilde{w}-\widehat{\cP}_m}_{[s+A,s+3A]}.\label{eq:fastgrowth}
        \end{align}
    
        \item \label{fastdecay} 
        If $\Verts{\widetilde{w}-\widehat{\cP}_m}_{[s+2A,s+4A]}\leq e^{-\frac{1}{2}A} \Verts{\widetilde{w}-\widehat{\cP}_m}_{[s+A,s+3A]}$, then 
        \begin{align}
            \Verts{\widetilde{w}-\widehat{\cP}_m}_{[s+A,s+3A]}\leq e^{-\frac{1}{2}A} \Verts{\widetilde{w}-\widehat{\cP}_m}_{[s,s+2A]}.\label{eq:fastdecay}
        \end{align}
        
        \item \label{noneutralmode} Either \eqref{eq:fastgrowth} or \eqref{eq:fastdecay} holds. 
    \end{enumerate}
\end{proposition}

\begin{proof} 
    First we prove the Proposition in the linear setting.
    \begin{claim}\label{claim:growthordecay}
        For any $A \geq 1$ and $v\in C^{\infty}(\mathbb{R}^{n}\times (0,4A])$, such that 
        \begin{align}
            \partial_{s}v=\Delta_{f}v+\frac{m}{2}v, &&\int_{\mathbb{R}^{n}}v (\cdot,s)h\,d\nu=0 \text{ for all } h\in\widehat{\mathcal{P}}_{m},\label{eq linear equation that v satisfies}
        \end{align}
        and that $v(\cdot, s)\in L^2(\R^n,\nu)$ for almost every $s\in (0,4A]$ and $v\in L^2(\R^n\times (0,4A],d\nu ds)$, the conclusions of Proposition \ref{prop:growthordecay} hold.
    \end{claim}
    \begin{proof}[Proof of Claim \ref{claim:growthordecay}]
        Using Lemma \ref{lemma asymptotic expansion}, we can decompose solutions $v$ to \eqref{eq linear equation that v satisfies} as $v(\cdot,s) \coloneqq \sum_{k\neq m}e^{-\frac{(k-m)}{2}s}\widehat{p}_{k}$ into homogeneous caloric polynomials $\widehat{p}_k \in \widehat{\mathcal{P}}_k$, hence $\Verts{v(\cdot,s)-\widehat{\mathcal{P}}_m}_{L^2(\mathbb{R}^n,\nu)}=\Verts{v(\cdot,s)}_{L^2(\mathbb{R}^n,\nu)}$ for all $s \in [0,4A]$.  Write $v(\cdot,s)=v^+(\cdot,s)+v^{-}(\cdot,s)$, where
        \begin{align*}
            v^{+}(\cdot,s)\coloneqq \sum_{k<m} e^{-\frac{(k-m)}{2}s}\widehat{p}_{k},\quad v^{-}(\cdot,s)\coloneqq \sum_{k>m} e^{-\frac{(k-m)}{2}s}\widehat{p}_{k}.
        \end{align*}
        
        We first prove \ref{fastgrowth}. Assume that $\int_{A}^{3A}\int_{\R^n} v^2\,d\nu ds \geq e^{\frac{1}{2}A} \int_{0}^{2A}\int_{\R^n} v^2\,d\nu ds$, so that
        \begin{align*}
            \int_{A}^{3A} \int_{\R^n} (v^+)^2 \,d\nu ds &\geq  e^{\frac{1}{2}A}\int_{0}^{2A} \int_{\R^n} (v^-)^2 \,d\nu ds -\int_{A}^{3A} \int_{\R^n} (v^-)^2 \,d\nu ds \\
            &\geq (e^{\frac{3}{2}A}-1)\int_{A}^{3A} \int_{\R^n} (v^-)^2 \,d\nu ds.
        \end{align*}
        Using this, we can see that
        \begin{align*} 
            \int_{2A}^{4A}\int_{\R^n} v^2\,d\nu ds &\geq \int_{2A}^{4A} \int_{\R^n} (v^+)^2 \,d\nu ds\geq  e^{A}\int_{A}^{3A} \int_{\R^n} (v^+)^2 \,d\nu ds\\
            &\geq e^{\frac{1}{2}A} \int_{A}^{3A}\int_{\R^n} v^2\,d\nu ds,
        \end{align*}
        which proves \ref{fastgrowth}. The proof of \ref{fastdecay} is similar. 
        
        We now prove \ref{noneutralmode}.
        First consider the case where $\int_{A}^{3A}\int_{\mathbb{R}^{n}}(v^{+})^{2}\,d\nu ds\geq\frac{1}{2}\int_{A}^{3A}\int_{\mathbb{R}^{n}}v^{2}\,d\nu ds$. Then 
        \begin{align*}
            \int_{2A}^{4A}\int_{\mathbb{R}^{n}}v^{2}\,d\nu \,ds&\geq  \int_{2A}^{4A}\int_{\mathbb{R}^{n}}(v^{+})^{2}\,d\nu ds\geq e^A\int_{A}^{3A}\int_{\mathbb{R}^{n}}(v^+)^2\,d\nu ds\\
            &\geq e^{\frac{1}{2}A}\int_{A}^{3A}\int_{\mathbb{R}^{n}}v^{2}\,d\nu ds,
        \end{align*}
        hence \eqref{eq:fastgrowth} holds. If instead $\int_{A}^{3A}\int_{\mathbb{R}^{n}}(v^{-})^{2}\,d\nu ds\geq\frac{1}{2}\int_{A}^{3A}\int_{\mathbb{R}^{n}}v^{2}\,d\nu ds$,
        then we similarly have
        \begin{align*}
            \int_{0}^{2A}\int_{\mathbb{R}^{n}}v^{2}\,d\nu ds\geq\frac{1}{2}e^{A}\int_{A}^{3A}\int_{\mathbb{R}^{n}}v^{2}\,d\nu ds,
        \end{align*}
        so that \eqref{eq:fastdecay} holds. 
    \end{proof}

    Now we prove the general case. Suppose by way of contradiction that there are $\epsilon_i,\lambda_i \to 0$, solutions $u^i$ to \eqref{part-2-eq-parabolic-equation} (with $\lambda \leftarrow \lambda_i$) satisfying \eqref{eq:strongerdoubling} so that $(\widetilde{w}^i_s\coloneqq (\widetilde{\chi u^i})_s)_{s \in [0,\infty)}$ are solutions to \eqref{eq evolution equation for parabolic scaling}, along with $s_i \in [0,\infty)$ and
    \begin{align} 
        \lambda_i^{\frac{\Upsilon}{2}}e^{-\Upsilon s_i} \leq \Verts{\widetilde{w}^i - \widehat{\mathcal{P}}_m}_{[s_i+A,s_i+3A]}< \epsilon_i,
    \end{align}
    yet one of the following holds for all $i\in \mathbb{N}_0$:
    
    \begin{enumerate}
        \item \label{eq:fastgrowthcontradiction} 
        $\Verts{\widetilde{w}^i-\widehat{\mathcal{P}}_m}_{[s_i+A,s_i+3A]} \geq e^{\frac{1}{2}A}\Verts{\widetilde{w}^i-\widehat{\mathcal{P}}_m}_{[s_i,s_i+2A]}$, but
        \begin{align}
            \Verts{\widetilde{w}^i-\widehat{\mathcal{P}}_m}_{[s_i+2A,s_i+4A]} < e^{\frac{1}{2}A} \Verts{\widetilde{w}^i-\widehat{\mathcal{P}}_m}_{[s_i+A,s_i+3A]},\label{eq fast growth contradiction}
        \end{align}
        
        \item \label{eq:fastdecaycontradiction}  $\Verts{\widetilde{w}^i-\widehat{\mathcal{P}}_m}_{[s_i+2A,s_i+4A]} \leq e^{-\frac{1}{2}A} \Verts{\widetilde{w}^i-\widehat{\mathcal{P}}_m}_{[s_i+A,s_i+3A]}$, but 
        \begin{align}
            \Verts{\widetilde{w}^i-\widehat{\mathcal{P}}_m}_{[s_i+A,s_i+3A]} > e^{-\frac{1}{2}A} \Verts{\widetilde{w}^i-\widehat{\mathcal{P}}_m}_{[s_i,s_i+2A]},\label{eq fast decay contradiction}
        \end{align}   

        \item Both \eqref{eq fast growth contradiction} and \eqref{eq fast decay contradiction} hold.
    \end{enumerate}
    For each $i\in \mathbb{N}_0$, choose $\widehat{q}_i \in \widehat{\mathcal{P}}_m$ such that 
    \begin{align}
        \begin{split}
            &\kappa_i \coloneqq \Verts{\widetilde{w}^i-\widehat{\cP}_m}_{[s_i+A,s_i+3A]} = \left(\int_{s_i+A}^{s_i+3A} \Verts{\widetilde{w}^i_{s}-\widehat{q}_i}^2_{L^2(\mathbb{R}^n,\nu)}ds \right)^{\frac{1}{2}},\\
            &\lim_{i\to \infty} \kappa_i = 0, \quad \lim_{i\to \infty} \kappa_i^{-1} \lambda_i^{\frac{\Upsilon}{2}}e^{-\Upsilon s_i} = 0. \label{eq:hypothesessequence}
        \end{split}
    \end{align}
    Set $v^i_{s} \coloneqq \kappa_i^{-1}(\widetilde{w}_{s_i+s}^i-\widehat{q}^i)$, so that 
    \begin{align*}
        \partial_{s}v_{s}^{i}=\Delta_{f}v_{s}^{i}+\frac{m}{2}v_{s}^{i}+b_{s}^{i}\widetilde{w}_{s_i+s}^{i}+h_{s}^{i},
    \end{align*}
    where
    \begin{align*}
        b_{s}^{i}\coloneqq \frac{1}{2}\kappa_{i}^{-1}(N^{\widetilde{w}_{s_i+s}}(1)-m), \quad h_{s}^{i}\coloneqq \kappa_{i}^{-1}g_{s_{i}+s}^{i},
    \end{align*}
    as well as 
    \begin{align}
         \qquad \int_{A}^{3A} \int_{\mathbb{R}^n} (v_{s}^i)^2 \,d\nu ds =\Verts{v^i-\widehat{\mathcal{P}}_m}_{[A,3A]}=1.\label{eq norm 1}
    \end{align}
    If \eqref{eq:fastgrowthcontradiction} holds, then we have
    \begin{align} \label{eq norm 2}
        \Verts{v^i-\widehat{\mathcal{P}}_m}_{[0,2A]}\leq e^{-\frac{1}{2}A}, \qquad \Verts{v^i - \widehat{\mathcal{P}}_m}_{[2A,4A]}< e^{\frac{1}{2}A}.
    \end{align} 
    By \eqref{eq norm 1}, \eqref{eq norm 2}, and the argument of Remark \ref{remark subadditivity} (applied twice)
    \begin{align} \label{eq:handydoodad}
        \begin{split}
            \kappa_i^{-1}\Verts{\widetilde{w}^i - \widehat{q}^i}_{[s_i,s_i+4A]} =  \Verts{v^i}_{[0,4A]} \leq C(A).
        \end{split}
    \end{align}
    Therefore, after passing to a subsequence, we can get a weak convergence $v^i\rightharpoonup v$ for some $v\in L^2(\R^n\times [0,4A],d\nu ds)$.
    
    Furthermore, for any $\sigma\in(0,1]$, we can apply \eqref{eq:handydoodad} and Lemma \ref{lem:higherordercloseness} with $\widetilde{w} \leftarrow \widetilde{w}^i$, $s_0 \leftarrow s_i+\sigma$, and $A \leftarrow 4A-\sigma$, we obtain
    \begin{align}
        \begin{split}
            &\sup_{s \in [\sigma,4A]} \Verts{v_{s}^i}_{L^2(\mathbb{R}^n,\nu)}^2+\int_{\sigma}^{4A} \Verts{v_{s}^i}_{H^1(\mathbb{R}^n,\nu)}^2ds + \int_\sigma^{4A} \Verts{\partial_s v_{s}^i}_{H^{-1}(\mathbb{R}^n,\nu)}^2 ds  \\
            &= \kappa_i^{-2}\sup_{s \in [s_i+\sigma,s_i+4A]}\Verts{\widetilde{w}_{s}^i-\widehat{q}^i}_{L^2(\mathbb{R}^n,\nu)}^2+\kappa_i^{-2}\int_{s_i+\sigma}^{s_i+4A} \Verts{\widetilde{w}_{s}^i-\widehat{q}^i}_{H^1(\mathbb{R}^n,\nu)}^2ds \\
            &\qquad + \kappa_i^{-2}\int_{s_i+\sigma}^{s_i+4A} \Verts{\partial_s \widetilde{w}_{s}^i-\widehat{q}^i}_{H^{-1}(\mathbb{R}^n,\nu)}^2 ds \\ 
            &\leq \kappa_i^{-2}C(A,\Lambda,\Upsilon)(\sigma^{-1}\Verts{\widetilde{w}^i-\widehat{q}^i}_{[s_i,s_i+4A]}^2 + \lambda_i^{\Upsilon}e^{-2\Upsilon s_i}) \leq C(A,\Lambda,\Upsilon)(\sigma^{-1}+\kappa_i^{-2}\lambda_i^{\Upsilon}e^{-2\Upsilon s_i})\\
            &\leq C(A,\Lambda,\Upsilon)\sigma^{-1},\label{eq bound for vi}
        \end{split}
    \end{align}
    for sufficiently large $i$, where we used \eqref{eq:hypothesessequence} for the last inequality. Because $H^1(\mathbb{R}^n,\nu)\hookrightarrow L^2(\mathbb{R}^n,\nu)$ is compact (see \cite[Theorem 6.4.2]{BakLed}) and $L^2(\mathbb{R}^n,\nu)\hookrightarrow H^{-1}(\R^n,\nu)$ is continuous, it follows from the Aubin--Lions--Simon lemma (see \cite[Theorem II.5.16]{boyer-fabrie-2012-mathematical}, \cite{aubin-1963theoreme} or \cite[Theorem 5]{simon-1986-compact}) that $(v^i)_{i\in \mathbb{N}}$ has a subsequence converging (strongly) to $v$ in $L^2([\sigma,4A],L^2(\mathbb{R}^n,\nu))$.

    \begin{claim}\label{claim: limit satisfies pde}
        The limit $v\in C^\infty(\R^n\times (0,4A))\cap L^2(\R^n\times (0,4A))$ satisfies
        \begin{align*}
            \partial_s v_{s} = \Delta_f v_{s} + \frac{m}{2}v_{s}.
        \end{align*}
    \end{claim}
    
    \begin{proof}[Proof of Claim \ref{claim: limit satisfies pde}]
        By Lemma \ref{lem:smallerrorterms}, we have 
        \begin{align*} 
            \int_{0}^{4A} \int_{\mathbb{R}^n}(h_s^i)^2\,d\nu ds = \kappa_i^{-2} \int_{s_i}^{s_i+4A} \int_{\R^n} (g_s^i)^2\, d\nu ds \leq C(A,\Lambda, \Upsilon)\kappa_i^{-2} \lambda_i^{\Upsilon}e^{-2\Upsilon s_i},
        \end{align*}
        so because of \eqref{eq:hypothesessequence}, we have $\lim_{i\to \infty} \int_0^{4A}\int_{\mathbb{R}^n}h_s^i\varphi\,d\nu ds = 0$ for any $\varphi\in C_c^\infty (\R^n\times (0,4A))$. 
        
        Similarly, Lemma \ref{lem:smallerrorterms} and \eqref{eq:handydoodad} give
        \begin{align*} 
            \int_{\sigma}^{4A} |b_s^i| \,ds &= \frac{1}{2}\kappa_i^{-1} \int_{s_i+\sigma}^{s_i+4A} |\widehat{N}^{u^i}(e^{-s})-m|\, ds \leq C\kappa_i^{-1}  \Verts{\widetilde{w}^i-\widehat{\cP}_m}_{[s_i,s_i+4A]}^2 + C\lambda_i^{\Upsilon} \kappa_i^{-1}e^{-2\Upsilon s_i} \\
            &\leq C(A,\Lambda,\Upsilon)\sigma^{-1}\kappa_i^{-1} \Verts{\widetilde{w}^i-\widehat{q}^i}_{[s_i,s_i+4A]}^2 + C \kappa_i^{-1}\lambda_i^{\Upsilon}e^{-2\Upsilon s_i} \\
            &\leq C(A,\Lambda,\Upsilon)\sigma^{-1}\kappa_i +C(A,\Lambda,\Upsilon)\kappa_i^{-1}\lambda_i^{\Upsilon}e^{-2\Upsilon s_i}
        \end{align*}
        for any $\sigma>0$. By \eqref{eq:hypothesessequence}, we thus have $\lim_{i \to \infty} \int_\sigma^{4A} |b_s^i| \,ds =0$. Thus recalling that $\int_{\mathbb{R}^n}(\widetilde{w}_s^i)^2\,d\nu=1$ (see \eqref{eq wtilde is normalized}), we have
        \begin{align*} 
            \limsup_{i \to \infty} \int_{0}^{4A} \int_{\mathbb{R}^n} b_s^i \widetilde{w}_{s_i+s}^i \varphi \,d\nu ds \leq \Verts{\varphi}_{C^0(\mathbb{R}^n)}\limsup_{i\to \infty} \int_{\sigma}^{4A} |b_s^i|\,ds =0
        \end{align*}
        for any $\varphi\in C_c^\infty (\R^n\times (0,4A))$.
        
        Therefore, for any $\varphi\in C_c^\infty (\R^n\times (0,4A))$ we have
        \begin{align*}
            0&=\lim_{i\to \infty }\int_{0}^{4A}\int_{\R^n} v_{s}^{i}\partial_{s} \varphi +v_{s}^{i}\Delta_{f}\varphi+\left(\frac{m}{2}v_{s}^{i}+b_{s}^{i}\widetilde{w}_{s_i+s}^{i}+h_{s}^{i}\right)\varphi\,d\nu ds\\
            &=\int_{0}^{4A}\int_{\R^n} v_{s}\partial_{s} \varphi +v_{s}\Delta_{f}\varphi+\frac{m}{2}v_{s}\varphi\,d\nu ds.
        \end{align*}
        In particular, $v$ is a weak solution. By parabolic regularity, we know that $v$ is a strong solution and is smooth.
    \end{proof}

    \begin{claim}\label{claim: limit has bound}
    The limit $v$ satisfies
        \begin{align*}
            \esssup_{s\in [\sigma,4A]}\int_{\mathbb{R}^n} v^2(\cdot,s)\,d\nu \leq C(A,\Lambda,\Upsilon)\sigma^{-1}.
        \end{align*}
    \end{claim}
    
    \begin{proof}[Proof of \ref{claim: limit has bound}] 
        After passing to a subsequence, we can assume that $v^i(\cdot,s) \to v(\cdot,s)$ in $L^2(\mathbb{R}^n,\nu)$ for almost-every $s\in [\sigma,4A]$. The claim then follows from \eqref{eq bound for vi}.
    \end{proof}
    
    Because $v^i \to v$ strongly in $L^2(\mathbb{R}^n \times [A,3A],d\nu ds)$, we also have
    \begin{align}
        \Verts{v}_{[A,3A]}^2=1, \qquad \int_A^{3A} \int_{\mathbb{R}^n} v_s h \,d\nu ds = 0\label{eq bound and orthogonality for v}
    \end{align}
    for all $h\in \widehat{\mathcal{P}}_m$, hence $\Verts{v-\mathcal{P}_m}_{[A,3A]}=1$. 

    \begin{claim} \label{claim:semicontinuous}
        The limit $v$ satisfies
        \begin{align*}
            \Verts{v-\widehat{\mathcal{P}}_m}_{[0,2A]}\leq e^{-\frac{1}{2}A}, \qquad \Verts{v-\widehat{\mathcal{P}}_m}_{[2A,4A]} \leq   e^{\frac{1}{2}A}.
        \end{align*}
    \end{claim}
    
    \begin{proof}[Proof of Claim \ref{claim:semicontinuous}] 
        By \eqref{eq norm 2}, there exist $\widehat{p}^i \in \widehat{\mathcal{P}}_m$ such that 
        \begin{align*} 
            \int_0^{2A} \int_{\mathbb{R}^n} (v^i -\widehat{p}^i)^2 d\nu ds \leq e^{-\frac{1}{2}A}
        \end{align*}
        for all $i\to \infty$. Then $\Verts{\widehat{p}^i}_{L^2(\mathbb{R}^n)}^2 \leq A^{-1}\Verts{v^i}_{[A,2A]}^2+ e^{-\frac{1}{2}A}\leq A^{-1}$, so we can pass to a subsequence to ensure that $\widehat{p}^i \to \widehat{p} \in \widehat{\mathcal{P}}_m$. In particular, $v^i -\widehat{p}^i \to v-\widehat{p}$ weakly in $L^2(\mathbb{R}^n \times [0,2A],\nu ds)$, so 
        \begin{align*}
            e^{-\frac{A}{2}} \geq \limsup_{i\to \infty} \Verts{v^i -\widehat{p}^i}_{[0,2A]} \geq \Verts{v-\widehat{p}}_{[0,2A]} \geq \Verts{v-\widehat{\mathcal{P}}_m}_{[0,2A]}.
        \end{align*}
        The proof that $\Verts{v-\widehat{\mathcal{P}}_m}_{[2A,4A]} \leq  e^{\frac{1}{2}A}$ is similar, so we omit it.
    \end{proof}

    Then Claim \ref{claim: limit has bound}, Claim \ref{claim: limit satisfies pde}, \eqref{eq bound and orthogonality for v}, and Claim \ref{claim:semicontinuous} contradict Claim \ref{claim:growthordecay}.
\end{proof}

\begin{proof}[Proof of Theorem \ref{theorem quantitative uniqueness in general}]
   Choose $\epsilon \leq \overline{\epsilon}(A,\Lambda)$ and $\lambda \leq \ol{\lambda}(A,\epsilon, \Lambda)$ as in Proposition \ref{prop:growthordecay}. It suffices to prove the claim when $s_1=k_1A$ and $s_2=k_2A$ for some $k_1,k_2\in \mathbb{N}_0$. 

   Suppose $s\in [kA,(k+1)A]\subset [k_2A,k_1A]$. Then Proposition \ref{weakmonotonicitylemma} implies that 
   \begin{align*}
       |\widehat{N}^u(e^{-s})-m| & = (\widehat{N}^u(e^{-s})-m)_+ + (\widehat{N}^u(e^{-s})-m)_- \\
       &\leq (\widehat{N}^u(e^{-kA})-m)_++(\widehat{N}^u(e^{-(k+1)A})-m)_-+C(A,\Lambda, \Upsilon)\lambda^{\frac{\Upsilon}{2}}e^{-\Upsilon s}\\
       &\leq |\widehat{N}^u(e^{-kA})-m|+|\widehat{N}^u(e^{-(k+1)A})-m|+C(A,\Lambda, \Upsilon)\lambda^{\frac{\Upsilon}{2}}e^{-\Upsilon s}.
   \end{align*}
   By this and Remark \ref{remark subadditivity}, it suffices to prove the claim when $s=kA$, $s_1=k_1A$, $s_2=k_2A$ for some $k,k_1,k_2\in \mathbb{N}_0$.    
    \begin{claim} \label{claim:closenessdecay} 
        For all integers $k \in [k_2,k_1]$, we have
       \begin{align*}
           \Verts{\widetilde{w}-\widehat{\mathcal{P}}_m}_{[(k-1)A,(k+1)A]}< C(A,\Lambda,\Upsilon)\delta \left( e^{-\frac{\Upsilon}{2} A (k_1-k)} + e^{-\frac{\Upsilon}{2} A(k-k_2)}\right)+ \lambda^{\frac{\Upsilon}{4}}e^{-\Upsilon (k-2)A}.
       \end{align*}
   \end{claim}
   
   \begin{proof}[Proof of Claim \ref{claim:closenessdecay}] 
       Suppose, for contradiction, that for some $k \in [k_2,k_1]$, there exists $B(A,\Lambda, \Upsilon)$ to be determined such that
       \begin{align}
           \Verts{\widetilde{w}-\widehat{\mathcal{P}}_m}_{[(k-1)A,(k+1)A]}\geq B\delta \left( e^{-\frac{\Upsilon}{2} A (k_1-k)} + e^{-\frac{\Upsilon}{2} A(k-k_2)}\right)+ \lambda^{\frac{\Upsilon}{4}}e^{-\Upsilon (k-2)A}\geq \lambda^{\frac{\Upsilon}{4}}e^{-\Upsilon (k-2)A}.\label{eq upper bound for claim}
       \end{align} 
       By \eqref{eq frequency pinched hypothesis for quantitative uniqueness}, we have $\widehat{N}^u(e^{-(\ell-1)A})-\widehat{N}^u(e^{-(\ell+1)A})\leq 2\delta$ for any integer $\ell\in [k_2,k_1]$. Thus Proposition \ref{prop:closeheatpolynomial} gives
       \begin{align}
           \Verts{\widetilde{w}-\widehat{\mathcal{P}}_m}_{[(\ell-1)A,(\ell+1)A]}< C(A,\Lambda,\Upsilon)\left( \delta + \lambda^{\frac{\Upsilon}{2}} e^{-\ell\Upsilon A} \right) <\epsilon\label{eq lower bound for claim}
       \end{align}
       if we take $\lambda \leq \overline{\lambda}(A,\epsilon, \Upsilon)$ and $\delta \leq \overline{\delta}(A,\epsilon,\Upsilon)$. Therefore, using \eqref{eq upper bound for claim} and \eqref{eq lower bound for claim} and Proposition \ref{prop:growthordecay} with $s\leftarrow (k-2)A$, we can infer that either
        \begin{equation} \label{eq:consequenceoffastgrowth} 
           \Verts{\widetilde{w}-\widehat{\mathcal{P}}_m}_{[kA,(k+2)A]} \geq e^{\frac{1}{2}A} \Verts{\widetilde{w}-\widehat{\mathcal{P}}_m}_{[(k-1)A,(k+1)A]}
       \end{equation}
       or else
       \begin{equation}\label{eq:consequenceoffastdecay}
            \Verts{\widetilde{w}-\widehat{\mathcal{P}}_m}_{[(k-1)A,(k+1)A]} \leq e^{-\frac{1}{2}A} \Verts{\widetilde{w}-\widehat{\mathcal{P}}_m}_{[(k-2)A,kA]}.
        \end{equation}
       
        First assume \eqref{eq:consequenceoffastgrowth} holds. We claim that 
        \begin{align}
            \Verts{\widetilde{w}-\widehat{\mathcal{P}}_m}_{[(k+\ell)A,(k+\ell+2)A]} \geq e^{\frac{A}{2}} \Verts{\widetilde{w}-\widehat{\mathcal{P}}_m}_{[(k+\ell-1)A,(k+\ell+1)A]}\label{eq:growthinduction}
        \end{align}
        for any integer $\ell\in [0,k_1-k]$. Note that the claim is true when $\ell=0$ by \eqref{eq:consequenceoffastgrowth}. Suppose we have shown \eqref{eq:growthinduction} for all $\ell < \ell_0$, where $\ell_0 < k_1 -k$. Then we use \eqref{eq upper bound for claim} to get
        \begin{align*} 
            \Verts{\widetilde{w}-\widehat{\mathcal{P}}_m}_{[(k+\ell_0-1)A,(k+\ell_0+1)A]} \geq  e^{\frac{A}{2}\ell_0} \Verts{\widetilde{w}-\widehat{\mathcal{P}}_m}_{[(k-1)A,(k+1)A]} \geq e^{\frac{A}{2}\ell_0}\lambda^{\frac{\Upsilon}{4}}e^{-\Upsilon(k-2)A} \geq \lambda^{\frac{\Upsilon}{4}}e^{-\Upsilon(k+\ell_0)A}.
        \end{align*}
        We then apply Proposition \ref{prop:growthordecay}\ref{fastgrowth} with $s\leftarrow (k+\ell_0-2)A$ to get
        \begin{align}
            \Verts{\widetilde{w}-\widehat{\mathcal{P}}_m}_{[(k+\ell_0)A,(k+\ell_0+2)A]} \geq e^{\frac{A}{2}} \Verts{\widetilde{w}-\widehat{\mathcal{P}}_m}_{[(k+\ell_0-1)A,(k+\ell_0+1)A]}.
        \end{align}
        Thus \eqref{eq:growthinduction} holds by induction. 
        
        Applying \eqref{eq:growthinduction} for all integers $\ell\in [0,k_1-k]$ and \eqref{eq upper bound for claim} yield
        \begin{align*} 
            \Verts{\widetilde{w}-\widehat{\mathcal{P}}_m}_{[(k_1-1)A,(k_1+1)A]} &\geq e^{\frac{A}{2}(k_1-k)}   \Verts{\widetilde{w}-\widehat{\mathcal{P}}_m}_{[(k-1)A,(k+1)A]} \\
            &\geq  B\delta e^{\frac{A}{2}(k_1-k)-\frac{\Upsilon}{2}A(k_1-k)} + \lambda^{\frac{\Upsilon}{4}}e^{\frac{A}{2}(k_1-k)-\Upsilon(k-2)A} \geq B\delta +\lambda^{\frac{\Upsilon}{4}}e^{-\Upsilon k_1A}.
        \end{align*}
        If we choose $\lambda \leq \overline{\lambda}(A,\epsilon,\Lambda,\Upsilon)$ and $B\geq \underline{B}(A,\epsilon,\Lambda,\Upsilon)$, then \eqref{eq lower bound for claim} is contradicted. Therefore \eqref{eq:consequenceoffastdecay} holds.
       
        We claim that 
        \begin{align}
            \Verts{\widetilde{w}-\widehat{\mathcal{P}}_m}_{[(k-\ell-1)A,(k-\ell+1)A]} \leq e^{-\frac{A}{2}} \Verts{\widetilde{w}-\widehat{\mathcal{P}}_m}_{[(k-\ell-2)A,(k-\ell)A]}\label{eq:growthinduction2}
        \end{align}
        for any integer $\ell\in [0,k-k_2-1]$. Note that the claim holds for $\ell=0$ by \eqref{eq:consequenceoffastdecay}. Suppose we have shown \eqref{eq:growthinduction2} for all $\ell < \ell_0$, where $\ell_0 < k-k_2+1$. Then we use \eqref{eq upper bound for claim} to get
        \begin{align*} 
            \Verts{\widetilde{w}-\widehat{\mathcal{P}}_m}_{[(k-\ell_0-1)A,(k-\ell_0+1)A]} &\geq  e^{\frac{A}{2}\ell_0} \Verts{\widetilde{w}-\widehat{\mathcal{P}}_m}_{[(k-1)A,(k+1)A]} \geq e^{\frac{A}{2}\ell_0}\lambda^{\frac{\Upsilon}{4}}e^{-\Upsilon(k-2)A}\\
            &\geq \lambda^{\frac{\Upsilon}{4}}e^{-\Upsilon(k-\ell_0-2)A}.
        \end{align*}
        We may therefore apply Proposition \ref{prop:growthordecay}\ref{fastdecay} with $s\leftarrow (k-\ell_0-2)A$ and \eqref{eq:growthinduction2} with $\ell\leftarrow \ell_0-1$ to get
        \begin{align}
            \Verts{\widetilde{w}-\widehat{\mathcal{P}}_m}_{[(k-\ell_0-1)A,(k-\ell_0+1)A]} \leq e^{-\frac{A}{2}} \Verts{\widetilde{w}-\widehat{\mathcal{P}}_m}_{[(k-\ell_0-2)A,(k-\ell_0)A]}.
        \end{align}
        Thus \eqref{eq:growthinduction2} holds by induction. 
        
        Therefore, \eqref{eq:growthinduction2} for all integers $\ell \in [0,k-k_2-1]$ and \eqref{eq upper bound for claim} yield
        \begin{align*}
            \Verts{\widetilde{w}-\widehat{\mathcal{P}}_m}_{[(k_2-1)A,(k_2+1)A]} &\geq e^{\frac{A}{2}(k-k_2)}\Verts{\widetilde{w}-\widehat{\mathcal{P}}_m}_{[(k-1)A,(k+1)A]} \\
            &\geq e^{\frac{A}{2}(k-k_2)}B\delta e^{-\frac{\Upsilon}{2} A(k-k_2)}+ \lambda^{\frac{\Upsilon}{4}}e^{\frac{A}{2}(k-k_2)}e^{-\Upsilon (k-2)A}\geq  B\delta +\lambda^{\frac{\Upsilon}{4}}e^{-\Upsilon k_2 A}.
        \end{align*}
        This contradicts \eqref{eq lower bound for claim} if $B\geq \underline{B}(A,\epsilon,\Lambda,\Upsilon)$ and $\lambda \leq \overline{\lambda}(A,\epsilon,\Lambda,\Upsilon)$. Therefore, \eqref{eq upper bound for claim} fails and the claim holds.
    \end{proof}
  
    For each integer $k\in \{k_2,\dots,k_1\}$, we choose $\widehat{p}^k \in \widehat{\mathcal{P}}_m$ satisfying $\Verts{\widetilde{w}-\widehat{p}^k}_{[(k-1)A,(k+1)A]}\coloneqq \Verts{\widetilde{w}-\widehat{\cP}_m}_{[(k-1)A,(k+1)A]}$.
    By \eqref{eq:consecutivepolynomials} and Claim \ref{claim:closenessdecay}, we then have 
    \begin{align*} 
        \Verts{\widehat{p}^k - \widehat{p}^{k+1}}_{L^2(\mathbb{R}^n,\nu)} &\leq \Verts{\widetilde{w}-\widehat{p}^k}_{[(k-1)A,(k+1)A]}+\Verts{\widetilde{w}-\widehat{p}^k}_{[kA,(k+2)A]}\\
        & < C\delta \left( e^{-\frac{\Upsilon}{2} A (k_1-k)} + e^{-\frac{\Upsilon}{2} A(k-k_2)}\right)+ 2\lambda^{\frac{\Upsilon}{4}}e^{-\Upsilon (k-2)A},
    \end{align*}
    where $C=C(A,\Lambda,\Upsilon)$. Fix $k_0\coloneqq \floor{\frac{k_2+k_1}{2}}$. For any integer $k\in [k_2,k_0]$ we have
    \begin{align}
        \begin{split}
            \Verts{\widehat{p}^k-\widehat{p}^{k_0}}_{L^2(\mathbb{R}^n,\nu)} &\leq C\sum_{i=k}^{k_0} \left( \delta \left( e^{-\frac{\Upsilon}{2}A(k_1-i)} + e^{-\frac{\Upsilon}{2}(i-k_2)A} \right)+\lambda^{\frac{\Upsilon}{4}}e^{-\Upsilon (i-2)A} \right)\\
            &\leq C\delta\left( e^{-\frac{\Upsilon}{4}A(k_1-k_2)} + e^{-\frac{\Upsilon}{2}(k-k_2)A} \right) + C\lambda^{\frac{\Upsilon}{4}}e^{-\Upsilon k A}.\label{eq estimate for smaller k}
        \end{split}
    \end{align}
    On the other hand, for any $k\in [k_0,k_1]$, we have
    \begin{align}
        \begin{split}
            \Verts{\widehat{p}^k-\widehat{p}^{k_0}}_{L^2(\mathbb{R}^n,\nu)} &\leq C\sum_{i=k_0}^{k} \left( \delta \left( e^{-\frac{\Upsilon}{2}A(k_1-i)} + e^{-\frac{\Upsilon}{2}(i-k_2)A} \right)+\lambda^{\frac{\Upsilon}{4}}e^{-\Upsilon (i-2)A} \right)\\
            &\leq C\delta\left( e^{-\frac{\Upsilon}{2}A(k_1-k)} + e^{-\frac{\Upsilon}{4}(k_1-k_2)A} \right) + C\lambda^{\frac{\Upsilon}{4}}e^{-\frac{\Upsilon}{2} (k_1+k_2) A}.\label{eq estimate for larger k}
        \end{split}
    \end{align}
    Combining the estimates \eqref{eq estimate for smaller k} and \eqref{eq estimate for larger k}, for any integer $k\in [k_2,k_1]$, we get
    \begin{align*}
        \Verts{\widehat{p}^k-\widehat{p}^{k_0}}_{L^2(\mathbb{R}^n,\nu)}\leq C\delta \left( e^{-\frac{\Upsilon}{4}A(k_1-k)} + e^{-\frac{\Upsilon}{4}(k-k_2)A} \right) +C\lambda^{\frac{\Upsilon}{4}}\left(e^{-\Upsilon kA}+e^{-\frac{\Upsilon}{2}(k_1+k_2)A}\right).
    \end{align*}
    Therefore,
    \begin{align*}
        \Verts{\widetilde{w}-\widehat{p}^{k_0}}_{[(k-1)A,(k+1)A]}&\leq \Verts{\widetilde{w}-\widehat{p}^{k}}_{[(k-1)A,(k+1)A]} +\Verts{\widehat{p}^{k}-\widehat{p}^{k_0}}_{[(k-1)A,(k+1)A]}\\
        &\leq C\delta \left( e^{-\frac{\Upsilon}{4}A(k_1-k)} + e^{-\frac{\Upsilon}{4}(k-k_2)A} \right) +C\lambda^{\frac{\Upsilon}{4}}\left(e^{-\Upsilon kA}+e^{-\frac{\Upsilon}{2}(k_1+k_2)A}\right).
    \end{align*}
    This proves \eqref{eq estimate for polynomial closeness}.

    We now prove \eqref{eq estimate for frequency}. By Lemma \ref{lem:smallerrorterms}
    \begin{align*}
         \int_{kA}^{(k+1)A}|\widehat{N}^u(e^{-s})-m|\,ds & \leq C\Verts{\widetilde{w}-\widehat{\cP}_m}_{[kA-1,(k+1)A]}^2+C\lambda^{\Upsilon}e^{-2\Upsilon kA} \\
         &\leq C\Verts{\widetilde{w}-\widehat{\cP}_m}_{[(k-1)A,(k+1)A]}^2+C\lambda^{\Upsilon}e^{-2\Upsilon kA}.
    \end{align*}
    Using this along with Proposition \ref{weakmonotonicitylemma} and Claim \ref{claim:growthordecay}, we get 
    \begin{align*}
        N(e^{-kA})-m &\leq \fint_{kA}^{(k+1)A} (N(e^{-s})-m)\,ds + C(A,\Lambda,\Upsilon)\lambda^{\frac{\Upsilon}{2}}e^{-\Upsilon kA} \\
        &\leq C\Verts{\widetilde{w}-\widehat{\cP}_m}_{[(k-1)A,(k+1)A]}^2+C\lambda^{\frac{\Upsilon}{2}}e^{-\Upsilon kA} \\
        &\leq C(A,\Lambda,\Upsilon)\delta \left( e^{-\frac{\Upsilon}{2} A (k_1-k)} + e^{-\frac{\Upsilon}{2} A(k-k_2)}\right)+ C(A,\Lambda,\Upsilon)\lambda^{\frac{\Upsilon}{4}}e^{-\Upsilon kA}.
    \end{align*}
    Similarly,
    \begin{align*}
        -(N(e^{-kA})-m) &\leq -\fint_{(k-1)A}^{kA} (N(e^{-s})-m)\,ds + C(A,\Lambda,\Upsilon)\lambda^{\frac{\Upsilon}{2}}e^{-\Upsilon kA} \\
        &\leq C(A,\Lambda,\Upsilon)\delta \left( e^{-\frac{\Upsilon}{2} A (k_1-k)} + e^{-\frac{\Upsilon}{2} A(k-k_2)}\right)+ C(A,\Lambda,\Upsilon)\lambda^{\frac{\Upsilon}{4}}e^{-\Upsilon kA}.
    \end{align*}
    Thus \eqref{eq estimate for frequency} follows.
\end{proof}

\subsection{Symmetry and cone splitting}
Our definitions, methods, and results in this subsection are modifications of those in Section \ref{symmetry-and-cone-splitting}. These modifications are needed due to the lack of regularity of solutions to \eqref{part-2-eq-parabolic-equation}.

\begin{definition}[Frequency pinching]\label{definition general frequency pinching}
    Let $u$ be a solution to \eqref{part-2-eq-parabolic-equation} satisfying \eqref{eq:strongerdoubling}. We define
    \begin{align*}
        \widehat{\mathcal{E}}_{r}(\mathbf{x})\coloneqq \widehat{N}_{\mathbf{x}}^u(50r^2) - \widehat{N}_{\mathbf{x}}^u(\tfrac{1}{50}r^2) + \lambda^{\frac{\Upsilon}{4}}r^{\Upsilon},&&\widehat{\mathcal{E}}_{r}(\{\mathbf{x}_i \}_{i=0}^{k})\coloneqq \max_{i} \widehat{\mathcal{E}}_{r}(\mathbf{x}_i),
    \end{align*}
    so that if $\lambda \leq \overline{\lambda}(\Lambda,\Upsilon)$, then by Proposition \ref{weakmonotonicitylemma} we have
    \begin{align*}
     \widehat{N}_{\mathbf{x}}^u(50r^2) - \widehat{N}_{\mathbf{x}}^u(\tfrac{1}{50}r^2) \geq -\frac{1}{2}\lambda^{\frac{\Upsilon}{4}}r^{\Upsilon}.  
    \end{align*}
    Define the $(k,\alpha r)$-\textit{pinching} at $\mathbf{x}_0$ to be
    \begin{align*}
        \widehat{\mathcal{E}}^{k,\alpha}_{r}(\mathbf{x}_0) \coloneqq \inf \widehat{\mathcal{E}}_{r}(\{\mathbf{x}_i \}_{i=0}^{K}),
    \end{align*}
    where the infimum is over $(k,\frac{1}{20}\alpha r)$-independent subsets $\{\mathbf{x}_i\}_{i=0}^K \subseteq P(\mathbf{x}_0,\frac{1}{10}r)$. 
\end{definition}

\begin{definition}\label{part 2 definition symmetry}
    A function $u$ satisfying \eqref{part-2-eq-parabolic-equation} is weakly \textit{$(k,\delta,r)$-symmetric} with respect to $V$ at $\mathbf{x}_0=(x_0,t_0)$, if one of the following holds:
    \begin{enumerate}[label={(\arabic*)}]
        \item There is a $k$-plane $L \subseteq \mathbb{R}^n$ such that $V=L\times \{0\}$ and
        \begin{align*}
            \int_{r^2}^{e^2r^2} \frac{1}{\widehat{H}_{\mathbf{x}_0}^u(\tau)}\int_{\mathbb{R}^n} |\pi_{L} \nabla (\chi u_{\mathbf{x}_0;1})|^2 \, d\nu_{-\tau} d\tau +\lambda^\frac{\Upsilon}{4}r^\Upsilon<\delta,
        \end{align*}\label{part 2 def: spatial symmetry}
        
        \item There is a $(k-2)$-plane $L \subseteq \mathbb{R}^n$ such that $V=L\times \mathbb{R}$ and
        \begin{align*}
            \int_{r^2}^{e^2r^2} \frac{1}{\widehat{H}_{\mathbf{x}_0}^u(\tau)}\int_{\mathbb{R}^n} |\pi_{L} \nabla (\chi u_{\mathbf{x}_0;1})|^2 \, d\nu_{-\tau}d\tau + \int_{r^2}^{e^2r^2} \frac{\tau}{\widehat{H}_{\mathbf{x}_0}^u(\tau)}\int_{\mathbb{R}^n} |\partial_t (\chi u_{\mathbf{x}_0;1})|^2 \, d\nu_{-\tau}d\tau +\lambda^\frac{\Upsilon}{4}r^\Upsilon< \delta.
        \end{align*}
    \end{enumerate}
\end{definition}

\subsubsection{Propagation of symmetry}
As an analogue of Lemma \ref{lemma: propagation of symmetry and pinching}, we prove that symmetry propagates across scales.
\begin{proposition}\label{lemma: propagation of symmetry and pinching-2}
     If $\lambda \leq \ol{\lambda}(\Lambda)$, $\alpha \in [0,\frac{1}{2}]$, and $\beta,\tau_0\in (0,1]$, then there exists $C= C(\beta, \Lambda, \Upsilon)$ such that the following properties hold:
    \begin{enumerate}[label=(\arabic*)]
        \item \label{propagation of symmetry-2} For any $k$-plane $L\subset \R^n$ and $\tau \in [\beta^2 \tau_0, \tau_0]$, we have:
        \begin{subequations}
            \begin{align}
                \int_{\R^n} \tau \verts{\pi_L\cd (\chi u_{\mathbf{x}_0;1})}^2e^{-\alpha f}\,d\nu_{-\tau} &\leq C\int_{\R^n} \tau_0 \verts{\pi_L\cd (\chi u_{\mathbf{x}_0;1})}^2e^{-\alpha f}\,d\nu_{-\tau_0}+C\lambda^{\Upsilon}\tau_0^{2\Upsilon} \widehat{H}_{\mathbf{x}_0}(\tau_0),\label{eq needed monotonicity for spatial derivative}\\
                \int_{\beta^2 \tau_0}^{\tau_0}\int_{\R^n} \tau \verts{\partial_t (\chi u_{\mathbf{x}_0;1})}^2e^{-\alpha f}\,d\nu_{-\tau} &\leq C\int_{\frac{\tau_0}{2}}^{\tau_0} \int_{\R^n} \tau_0 \verts{\partial_t (\chi u_{\mathbf{x}_0;1})}^2e^{-\alpha f}\,d\nu_{-\tau_0}+C\lambda^{\Upsilon}\tau_0^{2\Upsilon} \widehat{H}_{\mathbf{x}_0}(\tau_0).\label{eq needed monotonicity for temporal derivative}
            \end{align}
        \end{subequations}
        In particular, if $u$ is weakly $(k,\delta,r)$-symmetric with respect to $V\in {\rm Gr}_{\cP}(k)$ at $\mathbf{x}$, then $u$ is weakly $\left(k,C\delta,s \right)$-symmetric with respect to $V$ at $\mathbf{x}$ for all $s\in [\beta r,r]$.

        \item \label{upwardpropagation-2} Suppose $0<e^2r_1 \leq r \leq 1$,
        \begin{align*}
            |\widehat{N}_{\mathbf{x}}(e^{-10}r_1^2)-\widehat{N}_{\mathbf{x}}(e^{10}r^2)|+\lambda^{\frac{\Upsilon}{4}}r^{\Upsilon}< \delta<C(\Lambda)^{-1},
        \end{align*}
        and $u$ is weakly $(k,\delta,r_2)$-symmetric at $\mathbf{x}$ with respect to $V\in{\rm Gr}_{\cP}(k)$ for some $r_2 \in [r_1,r]$. Then, for all $s\in [r_1,r]$, $u$ is weakly $(k,C\delta,s)$-symmetric at $\mathbf{x}$ with respect to $V$.
    \end{enumerate}
\end{proposition}

In the following lemma, we show that a solution to \eqref{part-2-eq-parabolic-equation} satisfying \eqref{eq:strongerdoubling} can be well-approximated near a given scale by a caloric function.
\begin{lemma} \label{lem:caloricapprox}
    Given $\mathbf{x}_0 \in \mathbb{R}^n \times \mathbb{R}$ and $\tau_0 \in (0,1]$, let $w\in C^{\infty}(\mathbb{R}^n \times (-\tau_0,0])\cap C(\mathbb{R}^n \times [-\tau_0,0])$ be the unique solution to
    \begin{align*}
        \Box w =0,\qquad \qquad w(\cdot,-\tau_0)  = \chi(\cdot,-\tau_0)u_{\mathbf{x}_0;1}(\cdot,-\tau_0)
    \end{align*}
    satisfying the mild growth condition \eqref{intro-mild-growth}. 
    Then, for any $\beta\in (0,\frac{1}{2})$, we have
    \begin{align*}
        &\sup_{\tau \in [\beta^2\tau_0,\tau_0]} \int_{\R^n} (w-\chi u_{\mathbf{x}_0;1})^2\,d\nu_{-\tau} + \int_{\beta^2\tau_0}^{\tau_0} \int_{\mathbb{R}^n} |\nabla (w-\chi u_{\mathbf{x}_0;1})|^2 d\nu_{-\tau}d\tau\leq C(\beta,\Lambda,\Upsilon) \lambda^{\Upsilon} \tau_0^{2\Upsilon}\widehat{H}_{\mathbf{x}_0}(\tau_0),\\
        &\sup_{\tau \in [\beta^2\tau_0,\frac{1}{2}\tau_0]} \int_{\mathbb{R}^n} \tau |\nabla (w-\chi u_{\mathbf{x}_0;1})|^2 d\nu_{-\tau} + \int_{\beta^2\tau_0}^{\frac{1}{2}\tau_0} \tau \int_{\mathbb{R}^n} (|\partial_t (w-\chi u_{\mathbf{x}_0;1})|^2 +  |\nabla^2 (w-\chi u_{\mathbf{x}_0;1})|^2 )\,d\nu_{-\tau} d\tau \\ 
        &\qquad \leq C(\beta,\Lambda,\Upsilon) \lambda^{\Upsilon} \tau_0^{2\Upsilon} \widehat{H}_{\mathbf{x}_0}(\tau_0).
    \end{align*}
\end{lemma}

\begin{proof}
    After replacing $u$ with $u_{\mathbf{x}_0;1}$, we may assume $\mathbf{x}_0=0$ and $u_{\mathbf{x}_0;1}=u$. We compute:
    \begin{align*}
        \frac{d}{dt} \int_{\mathbb{R}^n} (\chi u -w)^2 d\nu_t &\leq   -2\int_{\mathbb{R}^n} |\nabla (\chi u-w)|^2 d\nu_t + 2\int_{\mathbb{R}^n} |\chi u-w|\cdot |\Box (\chi u)|d\nu_t \\
        &\leq  -2\int_{\mathbb{R}^n} |\nabla (\chi u-w)|^2 d\nu_t + \frac{1}{\tau} \int_{\mathbb{R}^n} (\chi u -w)^2 d\nu_t + \tau\int_{\mathbb{R}^n} |\square (\chi u)|^2 d\nu_t.
    \end{align*}
    Thus we can integrate:
    \begin{align*}
        \frac{d}{dt} \left( \tau\int_{\mathbb{R}^n} (\chi u-w)^2 d\nu_t\right) + 2\tau \int_{\mathbb{R}^n} |\nabla (\chi u-w)|^2 d\nu_t \leq  \tau^{2}\int_{\mathbb{R}^n} |\square (\chi u)|^2 d\nu_{t}
    \end{align*}
    from $t=-\tau_0$ to $-\beta^2 \tau_0$ and applying Lemma \ref{lem:HessianL2} \eqref{eq integral box estimate} yields
    \begin{align*}
        \sup_{\tau\in [\beta^2 \tau_0,\tau_0]} \int_{\mathbb{R}^n} (\chi u-w)^2 d\nu_{-\tau} + \int_{\beta^2 \tau_0}^{\tau_0} \int_{\mathbb{R}^n} |\nabla (\chi u-w)|^2 d\nu_{-\tau}d\tau \leq C(\beta,\Lambda,\Upsilon)\lambda^{\Upsilon}\tau_0^{2\Upsilon}\widehat{H}(\tau_0). 
    \end{align*}
    Fix a cutoff function $\eta \in C^{\infty}(\mathbb{R})$, such that $\eta(-\tau_0)=0$, $\eta|_{[-\frac{1}{2}\tau_0,\infty)} \equiv 1$, and $0\leq \eta' \leq 10\tau_0^{-1}$. We then integrate
    \begin{align*} 
        \frac{d}{dt}\int_{\R^n} \tau \eta(t) \verts{\cd (\chi u-w)}^2\,d\nu_{t}=&(-\eta(t)+\tau\eta'(t))\int_{\R^n} \verts{\cd (\chi u-w)}^2\,d\nu_{t}-2\eta(t)\tau \int_{\R^n} \verts{\cd^2 (\chi u-w)}^2\,d\nu_{t}\\&-2\tau \eta(t) \int_{\mathbb{R}^n} \Delta_f (\chi u-w) \Box (\chi u) \,d\nu_t \\
        \leq &(-\eta(t)+\tau\eta'(t))\int_{\R^n} \verts{\cd (\chi u-w)}^2\,d\nu_{t}-2\eta(t)\tau \int_{\R^n} \verts{\cd^2 (\chi u-w)}^2\,d\nu_{t}\\
        &+2\tau \eta(t) \int_{\mathbb{R}^n} |\Delta_f (\chi u-w)|^2\,d\nu_t +2\tau \eta(t) \int_{\R^n}|\Box (\chi u)|^2 \,d\nu_t.
    \end{align*}
    Applying the $f$-Bochner formula \eqref{eq: f-Bochner formula} with $w\leftarrow \chi u-w$ and integrating by parts, we get
    \begin{align*}
        2\int_{\R^n} (\Delta_f (\chi u-w))^2 \,d\nu_t=2 \int_{\R^n} \verts*{\cd^2 (\chi u-w)}^2\,d\nu_t+ \frac{1}{\tau}\int_{\R^n} \verts{\cd (\chi u-w)}^2\,d\nu_t.
    \end{align*}
    Therefore, arguing as before,
    \begin{align*}
        &\sup_{\tau \in [\beta^2 \tau_0,\tau_0]}\int_{\R^n} \tau \eta(-\tau) \verts{\cd (\chi u-w)}^2\,d\nu_{-\tau}+\int_{\beta^2 \tau_0}^{\tau_0}\eta(-\tau)\tau \int_{\R^n} \verts{\cd^2 (\chi u-w)}^2\,d\nu_{-\tau}d\tau\\
        &\qquad \leq C(\beta,\Lambda,\Upsilon)\lambda^{\Upsilon}\tau_0^{2\Upsilon}\widehat{H}(\tau_0). 
    \end{align*}
    This together with Lemma \ref{lem:HessianL2} \eqref{eq integral box estimate} yields,
    \begin{align*}
        \int_{\beta^2 \tau_0}^{\tau_0} \eta(-\tau) \tau\int_{\R^n} \verts{\partial_t(\chi u-w)}^2\,d\nu_{-\tau} d\tau\leq  C(\beta,\Lambda,\Upsilon)\lambda^{\Upsilon}\tau_0^{2\Upsilon}\widehat{H}(\tau_0).
    \end{align*}
\end{proof}

The following will be used to prove a version of Lemma \ref{lemma comparability of H} in the non-caloric setting.
\begin{corollary} \label{cor:improvedexponent} 
    If $\lambda \leq \overline{\lambda}(\Lambda,\Upsilon)$, then, for any $\mathbf{x}_0 \in P(\mathbf{0},10)$ and $\tau_0 \in (0,1]$, the following estimates hold for any $B \in [2,10^{10n}]$
    \begin{align*} 
        \frac{\tau_0}{\widehat{H}_{\mathbf{x}_0}^{u}(\tau_0)} \int_{\mathbb{R}^n} |\pi_L \nabla(\chi u_{\mathbf{x}_0;1})|^2 e^{\frac{1}{8n}f} d\nu_{-\tau_0} &\leq C \int_{B\tau_0}^{Be^2 \tau_0}\frac{1}{\widehat{H}_{\mathbf{x}_0}^u(\tau)}\int_{\mathbb{R}^n} |\pi_L \nabla (\chi u_{\mathbf{x}_0;1})|^2 d\nu_{-\tau}d\tau \\
        &\qquad+ C(\Lambda,\Upsilon)\lambda^{\frac{\Upsilon}{2}} \tau_0^{\Upsilon},\\
        \int_{\tau_0}^{e^2 \tau_0} \frac{\tau}{\widehat{H}_{\mathbf{x}_0}^u(\tau)}\int_{\mathbb{R}^n} |\partial_t (\chi  u_{\mathbf{x}_0;1})|^2 e^{\frac{1}{8n} f}d\nu_{-\tau}d\tau&\leq C\int_{B\tau_0}^{Be^2\tau_0}\frac{\tau}{\widehat{H}_{\mathbf{x}_0}^u(\tau)}\int_{\mathbb{R}^n}|\partial_t (\chi u_{\mathbf{x}_0;1})|^2 d\nu_{-\tau}d\tau\\
        &\qquad + C(\Lambda,\Upsilon)\lambda^{\frac{\Upsilon}{2}}\tau_0^{\Upsilon}.
    \end{align*}
\end{corollary}

\begin{proof}
    By replacing $u$ with $u_{\mathbf{x}_0;1}$, we may assume $\mathbf{x}_0=\mathbf{0}$ and $u=u_{\mathbf{x}_0;1}$. Let $w$ be the caloric function from Lemma \ref{lem:caloricapprox} with $\tau_0 \leftarrow 2B \tau_0$. Then, for any $\tau \in [\tau_0,e^2\tau_0]$, Lemma \ref{lem:caloricapprox} gives
    \begin{align*}
        \tau\int_{\mathbb{R}^n} |\nabla (\chi u-w)|^2 e^{\frac{1}{8n} f}d\nu_{-\tau} &\leq   \lambda^{-\frac{\Upsilon}{2}}\tau^{-\Upsilon}\cdot \tau \int_{\mathbb{R}^n} |\nabla (\chi u-w)|^2 d\nu_{-\tau} \\&\qquad + \lambda^{\frac{\Upsilon}{2}}\tau^{\Upsilon} \cdot \tau \int_{\mathbb{R}^n} (|\nabla (\chi u)|^2 + |\nabla w|^2)e^{\frac{1}{4n} f}d\nu_{-\tau} \\
        &\leq C(\Lambda,\Upsilon)\lambda^{\frac{\Upsilon}{2}}\tau^{\Upsilon} \widehat{H}^u(\tau) + C(\Lambda,\Upsilon)\lambda^{\frac{\Upsilon}{2}}\tau^{\Upsilon} \left( \int_{\mathbb{R}^n} |\nabla w|^{\frac{8n}{4n-1}} d\nu_{-\tau}\right)^{\frac{4n-1}{4n}}\\
        &\leq C(\Lambda,\Upsilon)\lambda^{\frac{\Upsilon}{2}}\tau^{\Upsilon}\widehat{H}^u(\tau),
    \end{align*}
    where in the last line we applied Proposition \ref{proposition-hypercontractivity} and $\widehat{N}^w(2\tau) \leq C(\Lambda,\Upsilon)$. On the other hand, because the components of $\pi_L \nabla w$ are caloric, for any $\tau \in [B\tau_0,Be^2 \tau_0]$, we have
    \begin{align*}
        \tau_0 \int_{\mathbb{R}^n} |\pi_L \nabla w|^2 e^{\frac{1}{8n} f} d\nu_{-\tau_0} \leq C\tau_0 \left( \int_{\mathbb{R}^n} |\pi_L \nabla w|^{\frac{8n}{4n-1}} d\nu_{-\tau_0}\right)^{\frac{4n-1}{4n}} \leq C\tau \int_{\mathbb{R}^n} |\pi_L \nabla w|^2 d\nu_{-B\tau_0}.
    \end{align*}
    The first inequality therefore follows by combining expressions, and the second inequality is similar.
\end{proof}

We prove a version of \eqref{eq needed monotonicity for spatial derivative} and \eqref{eq needed monotonicity for temporal derivative} for caloric functions.
\begin{lemma}\label{lemma comparability of weaf}
    Suppose $w$ is a caloric function with mild growth \eqref{eq mild growth}. Then, for any $\alpha\in [0,\frac{1}{2}]$ and any $0<\tau_1\leq \tau_2$, we have
    \begin{align*}
        \int_{\R^n} w^2 e^{-\alpha f} \,d\nu_{-\tau_1}\leq \left(\frac{\tau_2}{\tau_1}\right)^{\frac{n\alpha}{2}} \int_{\R^n} w^2 e^{-\alpha f} \,d\nu_{-\tau_2}.
    \end{align*}
\end{lemma}

\begin{proof}
    Since $\Delta e^{-(1+\alpha)f}=e^{-(1+\alpha)f} \left(-\frac{(1+\alpha)n}{2\tau}+\frac{(1+\alpha)^2 f}{\tau}\right)$, we have
    \begin{align*}
        -\div (w^2 \cd e^{-(1+\alpha) f})=2(1+\alpha)\cd w\cdot \cd f we^{-(1+\alpha)f} +w^2e^{-(1+\alpha)f} \left(\frac{(1+\alpha)n}{2\tau}-\frac{(1+\alpha)^2 f}{\tau}\right).
    \end{align*}
    Integration by parts then yields
    \begin{align*}
        &\frac{(1+\alpha)}{\tau}\int_{\R^n} w^2 f e^{-\alpha f}\,d\nu_{t} \\
        &\leq \frac{n}{2\tau}\int_{\R^n} w^2 e^{-\alpha f}d\nu_t + \frac{2}{(1+\alpha)}\int_{\R^n} \verts{\cd w}^2 e^{-\alpha f}\,d\nu_t + \frac{(1+\alpha)}{2}  \int_{\R^n} \frac{1}{\tau}w^2 f e^{-\alpha f}d\nu_t.
    \end{align*}
    Rearranging this gives
    \begin{align}
        \int_{\R^n} w^2(2(1+\alpha)f-n) \frac{\alpha}{2\tau}e^{-\alpha f}d\nu_t\leq  \frac{4\alpha}{(1+\alpha)}\int_{\R^n} \verts{\cd w}^2 e^{-\alpha f}\,d\nu_t+ \frac{n\alpha}{2\tau}\int_{\R^n} w^2 e^{-\alpha f}\,d\nu_t.\label{eq integral of nice quantity}
    \end{align}
    Therefore, if $\alpha<\frac{1}{2}$, then
    \begin{align*}
        \frac{d}{dt}\int_{\R^n} w^2 e^{-\alpha f} \,d\nu_{t}&=-2\int_{\R^n}\verts{\cd w}^2e^{-\alpha f}\,d\nu_t - \int_{\R^n} w^2(\Box^* e^{-\alpha f} +2\cd f\cdot \cd e^{-\alpha f}) \,d\nu_t\\
        &=-2\int_{\R^n}\verts{\cd w}^2e^{-\alpha f}\,d\nu_t + \int_{\R^n} w^2(2(1+\alpha)f-n)\frac{\alpha}{2\tau} e^{-\alpha f} \,d\nu_t\\
        &\leq \frac{n \alpha}{2\tau} \int_{\R^n} w^2 e^{-\alpha f}\,d\nu_t.
    \end{align*}
    Thus the claim follows by Gronwall's inequality.
\end{proof}

\begin{proof}[Proof of Proposition \ref{lemma: propagation of symmetry and pinching-2}]
    By replacing $u$ with $u_{\mathbf{x};1}$, we can assume that $\mathbf{x}=0$ and $u_{\mathbf{x};1}=u$.
    
    \ref{propagation of symmetry-2}
    Suppose $h$ is a caloric approximation of $\chi u$ on $[-2\tau_0, -\beta^2 \tau_0]$. Then Lemma \ref{lem:caloricapprox} and Lemma \ref{lem:weakestimates} imply that
    \begin{align}
        \begin{split}
            &\sup_{\tau \in [\beta^2 \tau_0,\tau_0]} \int_{\mathbb{R}^n} \tau |\nabla (h-\chi u)|^2 d\nu_{-\tau} + \int_{\beta^2 \tau_0}^{\tau_0} \tau \int_{\mathbb{R}^n} |\partial_t (h-\chi u)|^2\,d\nu_{-\tau}d\tau\\
            &\qquad \leq C(\beta,\Lambda,\Upsilon) \lambda^{\Upsilon} \tau_0^{2\Upsilon} \widehat{H}^u(\tau_0).\label{eq lord caloric approximation}
        \end{split}
    \end{align}
    On the other hand, Lemma \ref{lemma comparability of weaf} applied twice with $w\leftarrow \partial_t h$ as well as $w\leftarrow \partial_i h$, we get
    \begin{subequations}
        \begin{align*}
            \sup_{\tau\in [\beta^2 \tau_0,\tau_0]}\int_{\R^n} |\pi_L\cd h|^2 e^{-\alpha f} \,d\nu_{-\tau}&\leq \beta^{n\alpha} \int_{\R^n} |\pi_L\cd h|^2 e^{-\alpha f} \,d\nu_{-\tau_0},\\
            \sup_{\tau\in [\beta^2 \tau_0,\tau_0]}\int_{\R^n} |\partial_th|^2 e^{-\alpha f} \,d\nu_{-\tau}&\leq \beta^{n\alpha} \int_{\R^n} |\partial_t h|^2 e^{-\alpha f} \,d\nu_{-\tau_0}.
        \end{align*}
    \end{subequations}
    Combining these with \eqref{eq lord caloric approximation}, we get \eqref{eq needed monotonicity for spatial derivative} and \eqref{eq needed monotonicity for temporal derivative}. As a consequence, if $u$ is weakly $(k,\delta,r)$-symmetric with respect to $V\in {\rm Gr}_{\cP}(k)$ at $\mathbf{x}$, then $u$ is weakly $(k,C\delta,s)$-symmetric with respect to $V$ at $\mathbf{x}$ for all $s\in [\beta r,r]$.

    \ref{upwardpropagation-2} We will assume that $u$ is weakly temporally symmetric, as the other case is similar. By Theorem \ref{theorem quantitative uniqueness in general} there exists $p\in \cP_m$ such that for any $\tau \in [r_1^2,e^2r^2]$ we have
    \begin{align*}
        \Verts{\widetilde{w}-\widehat{p}}_{[-\log\tau-1,-\log \tau+1]}\leq C(\Lambda, \Upsilon) \delta+C\lambda^{\frac{\Upsilon}{4}}\tau^{\frac{\Upsilon}{2}}\leq C(\Lambda,\Upsilon)\delta.
    \end{align*}
    Then Lemma \ref{lem:higherordercloseness} implies for all $\tau \in [r_1^2,e^2r^2]$
    \begin{align}
        \begin{split} \label{eq what we want}
            &\frac{1}{\widehat{H}^u(\tau)}\int_{\R^n} \left(\chi u-q \right)^2\,d\nu_{-\tau} + \int_{\frac{1}{2}\tau}^{2\tau} \frac{1}{\widehat{H}^u(s)}\int_{\mathbb{R}^n} \left( |\nabla (\chi u-q)|^2 + s|\partial_t (\chi u-q)|^2 \right) d\nu_{-s} ds\leq  C(\Lambda,\Upsilon)\delta,
        \end{split}
    \end{align}
    where we set $q\coloneqq \sqrt{\frac{\widehat{H}^u(s)}{s^m}}p$. Observe that:
    \begin{align} \label{eq:l2pabout1}
        \frac{1}{\tau^m}\int_{\R^n} p^2 \,d\nu_{-\tau}\geq \frac{1}{\widehat{H}^u(\tau)}\int_{\R^n} (\chi u)^2\,d\nu_{-\tau}-\frac{1}{\widehat{H}^u(\tau)} \int_{\R^n} (\chi u-q)^2\,d\nu_{-\tau} \geq 1-C(\Lambda, \Upsilon)\delta,
    \end{align}
    and similarly
    \begin{align} \label{eq:l2pabout12}
         \frac{1}{\tau^m}\int_{\R^n} p^2 \,d\nu_{-\tau}\leq 1+C(\Lambda,\Upsilon)\delta.
    \end{align}
    
    Using \eqref{eq derivative of log H}, we get
    \begin{align}
        \begin{split}\label{eq closeness of time derivative of q and p}
            \int_{\R^n}\verts*{\pdt q-\sqrt{\frac{\widehat{H}^u(s)}{s^m}}\partial_t p}^2\,d\nu_{-s}&\leq C\left(\frac{d}{ds}\log \widehat{H}^u(s)-\frac{m}{s}\right)^2 \int_{\mathbb{R}^n} q^2\,d\nu_{-s}\\
            &\leq C(\Lambda,\Upsilon)\lambda^{2\Upsilon}s^{2\Upsilon-2} \frac{\widehat{H}^u(s)}{s^m}
            \int_{\R^n} p^2\,d\nu_{-s}.
        \end{split}
    \end{align}
    Combining \eqref{lem:higherordercloseness} and \eqref{eq:l2pabout1} with $\tau\leftarrow 2r_2^2$ and $p\in \cP_m$, we get
    \begin{align*}
        &\log (2)\int_{\R^n}|\partial_t p|^2\,d\nu_{-1}\\
        &=\int_{\frac{\tau}{2}}^{\tau}\frac{1}{\widehat{H}^u(s)}
        \int_{\R^n} s \verts*{\sqrt{\frac{\widehat{H}^u(s)}{s^m}}\partial_tp}^2\,d\nu_{-s} ds\\
        &\leq  4\int_{\frac{\tau}{2}}^{\tau}\frac{1}{\widehat{H}^u(s)}
        \int_{\R^n} s \verts*{\sqrt{\frac{\widehat{H}^u(s)}{s^m}}\partial_tp-\partial_t q}^2\,d\nu_{-s} ds + 4\int_{\frac{\tau}{2}}^{\tau}\frac{1}{\widehat{H}^u(s)}
        \int_{\R^n} s \verts*{\partial_t (\chi u)-\partial_t q}^2\,d\nu_{-s} ds\\ &\qquad +4\int_{\frac{\tau}{2}}^{\tau}\frac{1}{\widehat{H}^u(s)}
        \int_{\R^n} s \verts*{\partial_t(\chi u)}^2\,d\nu_{-s} ds\\
        &\leq C(\Lambda, \Upsilon) \int_{\frac{\tau}{2}}^{\tau} \lambda^{2\Upsilon} s^{2\Upsilon-1}\frac{1}{s^m} \int_{\R^n}p^2\,d\nu_{-s}ds+C(\Lambda, \Upsilon)\delta \leq C(\Lambda,\Upsilon)\delta\int_{\mathbb{R}^n}p^2 d\nu_{-1}.
    \end{align*}
    Combining this along with $p\in \cP_m$, for any $\tau\in [r_1^2,r]$, we have
    \begin{align*}
        \int_{\tau}^{e^2\tau} \frac{s}{\widehat{H}^u(s)}\int_{\mathbb{R}^n} |\partial_t (\chi u)|^2 d\nu_{-s} ds &\leq4\int_{\tau}^{e^2\tau} \frac{s}{\widehat{H}^u(s)}\int_{\mathbb{R}^n} |\partial_t (\chi u-q)|^2 d\nu_{-s} ds \\
        &\qquad + 4\int_{\tau}^{e^2 \tau} \frac{s}{\widehat{H}^u(s)} \int_{\mathbb{R}^n} \verts*{\pdt q-\sqrt{\frac{\widehat{H}^u(s)}{s^m}}\partial_t p}^2\,d\nu_{-s} ds\\
        &\qquad + 4\int_{\tau}^{e^2 \tau} \frac{s}{s^m} \int_{\mathbb{R}^n} |\partial_t p|^2 d\nu_{-s}ds\\
        &\leq C(\Lambda,\Upsilon)(\delta+\lambda^{2\Upsilon}\tau^{2\Upsilon}+\delta)\leq C(\Lambda, \Upsilon)\delta,
    \end{align*}
    where we used \eqref{eq what we want}, \eqref{eq closeness of time derivative of q and p}, \eqref{eq:l2pabout12} in the second inequality.
\end{proof}

As an analog of Lemma \ref{lemma-comparison-of-caloric-energy}, the following result proves that energy functionals based at nearby points are comparable. It allows us to propagate symmetry at nearby points.
\begin{lemma}\label{lemma comparability of H}
    The following holds if $\lambda \leq \overline{\lambda}(\Lambda,\Upsilon)$. For any $r\in (0,1)$ and any $\mathbf{x}_0,\mathbf{x}_1 \in P(\mathbf{0},10)$ satisfying $|\mathbf{x}_1-\mathbf{x}_0|<\frac{1}{10n}r$:
    \begin{enumerate}[label=(\arabic*)]
        \item \label{comparability of widehat H}
        If $\tau_0,\tau_1 \in [r^2,e^{100}r^2]$, then
        \begin{align*}
        C(\Lambda)^{-1}\widehat{H}_{\mathbf{x}_0}^u(\tau_0)\leq \widehat{H}_{\mathbf{x}_1}^u(\tau_1)&\leq C(\Lambda)\widehat{H}_{\mathbf{x}_0}^u(\tau_0).
        \end{align*}
    
        \item \label{comparability of spatial derivative} For any $k$-plane $L\subseteq \mathbb{R}^n$ and $\tau \in [r^2,e^{100}r^2]$:
        \begin{align*}
            &\frac{\tau}{\widehat{H}_{\mathbf{x}_1}^u(\tau)} \int_{\R^n} \verts{\pi_L \cd (\chi u_{\mathbf{x}_1;1})}^2  \,d\nu_{-\tau}\leq C(\Lambda,\Upsilon)\frac{4\tau}{\widehat{H}_{\mathbf{x}_0}^u(4\tau)} \int_{\R^n} \verts{\pi_L \cd (\chi u_{\mathbf{x}_0;1})}^2 \,d\nu_{-4\tau} +C(\Lambda,\Upsilon)\lambda^{\frac{\Upsilon}{2}} \tau^{\Upsilon}.
        \end{align*}

        \item \label{comparability of time derivative} Finally,
        \begin{align*}
            &\int_{r^2}^{e^2 r^2}\frac{\tau}{\widehat{H}_{\mathbf{x}_1}^u(\tau)} \int_{\R^n} \verts{\partial_t (\chi u_{\mathbf{x}_1;1})}^2\,d\nu_{-\tau}d\tau\\
            &\qquad \leq C(\Lambda,\Upsilon)\int_{4r^2}^{4e^2r^2}\frac{\tau}{\widehat{H}_{\mathbf{x}_0}^u(\tau)} \int_{\R^n} \verts{\partial_t (\chi u_{\mathbf{x}_0;1})}^2\,d\nu_{-\tau}d\tau + C(\Lambda,\Upsilon)\lambda^{\frac{\Upsilon}{2}} r^{2\Upsilon}.
        \end{align*}
    \end{enumerate}
\end{lemma}

\begin{proof} 
    By replacing $u$ with $u_{\mathbf{x}_0;1}$, we may assume that $\mathbf{x}_0=\mathbf{0}$ and $u=u_{\mathbf{x}_0;1}$ and we write $\tau\coloneqq \tau_0$. Define
    \begin{align*}
        \chi_1(z,-s)\coloneqq \chi (a^{\frac{1}{2}}(\mathbf{x}_1)(z-x_1),-t_1-s).
    \end{align*}
    Then, for any $s\in [r^2,e^{100}r^2]$, we have 
    \begin{align} \label{eq:wherechisareequal}
        \chi_1(\cdot,s)|_{B(0,\frac{1}{10}s^{\Upsilon}\lambda^{-\frac{\Upsilon}{2}})} \equiv \chi(\cdot,s)|_{B(0,10s^{\Upsilon}\lambda^{-\frac{\Upsilon}{2}})} \equiv 1. 
    \end{align}
    
    \ref{comparability of widehat H} Using \eqref{eq:wherechisareequal} and Lemma \ref{lemma polynomial growth of general solutions}, we estimate
    \begin{align}\label{eq estimate for I3}
        \begin{split}
            \int_{\R^n} (\chi_{1} u-\chi u)^2 e^{\lambda f_1}\, d\nu_{\mathbf{x}_1;t_1-\tau}&\leq 2 \int_{B(x_1,10\lambda^{-\frac{\Upsilon}{2}}\tau^{\Upsilon})\setminus \overline{B}(x_1,10^{-1}\lambda^{-\frac{\Upsilon}{2}}\tau^{\Upsilon})} u^2 e^{\lambda f_1}\,d\nu_{\mathbf{x}_1;t_1-\tau} \\
            &\leq  C(\Lambda,\Upsilon) \left(1+ \frac{\lambda^{-\Upsilon}}{\tau^{1-2\Upsilon}}\right)^{C(\Lambda)} \frac{\lambda^{-\frac{n\Upsilon}{2}}}{\tau^{\frac{n}{2}(1-2\Upsilon)}} \exp \left( -\frac{\lambda^{-\Upsilon}}{400\tau^{1-2\Upsilon}} \right) \widehat{H}_{\mathbf{x}_1}^u(\tau)\\
            &\leq  C(\Lambda,\Upsilon)\exp \left( -\frac{1}{e^{60}\lambda^{\Upsilon}\tau^{1-2\Upsilon}} \right) \widehat{H}_{\mathbf{x}_1}^u(\tau)\leq \frac{1}{2}\widehat{H}_{\mathbf{x}_1}^u(\tau)
        \end{split}
    \end{align}
    if $\lambda \leq \ol{\lambda}(\Lambda,\Upsilon)$. On the other hand, we change variables $z=x_1+a^{\frac{1}{2}}(\mathbf{x}_1)y$ to obtain
    \begin{align*} 
        \widehat{H}_{\mathbf{x}_1}^u(\tau) &=  \int_{\mathbb{R}^n} \chi^2(y,-\tau) u^2(x_1+a^{\frac{1}{2}}(\mathbf{x}_1)y,t_1-\tau) d\nu_{-\tau}(y) \\
        &= \left(\sqrt{\det a^{\frac{1}{2}}(\mathbf{x}_1)}\right)^{-1}\int_{\mathbb{R}^n} \chi_1^2(z,t_1-\tau)u^2(z,t_1-\tau) \frac{1}{(4\pi \tau)^{\frac{n}{2}}}e^{-\frac{|a^{-\frac{1}{2}}(\mathbf{x}_1)(z-x_1)|^2}{4\tau}}dz \\
        &\leq  2\int_{\mathbb{R}^n} \chi_1^2 u^2 e^{\lambda f_1}\,d\nu_{\mathbf{x}_1;t_1-\tau}\leq 4 \int_{\mathbb{R}^n} \chi^2 u^2 e^{\lambda f_1}\,d\nu_{\mathbf{x}_1;t_1-\tau} + 4\int_{\mathbb{R}^n} (\chi_1 u-\chi u)^2 e^{\lambda f_1}\,d\nu_{\mathbf{x}_1;t_1-\tau}\\
        &\leq C \int_{\mathbb{R}^n} \chi^2 u^2 e^{\frac{1}{2} f}\,d\nu_{t_1-\tau} + \frac{1}{2} \widehat{H}_{\mathbf{x}_1}^u(\tau)\leq C \widehat{H}^u(\tau-t_1) + \frac{1}{2} \widehat{H}_{\mathbf{x}_1}^u(\tau)\\
        &\leq C(\Lambda,\Upsilon)\widehat{H}^u(\tau)+\frac{1}{2} \widehat{H}_{\mathbf{x}_1}^u(\tau)
    \end{align*}
    if $\lambda \leq \ol{\lambda}(\Lambda,\Upsilon)$, where we used \eqref{part-2-eq-parabolic-equation} in the third line. Further, in the fourth line, we used Lemma \ref{lemma-change-of-base-point} with $\mathbf{x}_1 \leftarrow \mathbf{0}$, $\mathbf{x}_0\leftarrow \mathbf{x}_1$, $\tau \leftarrow \tau-t_1$, $\alpha_0 \leftarrow \lambda$, $\alpha \leftarrow \frac{1}{2}$ and $\sigma \leftarrow \frac{1}{100}r^2$ as well as Lemma \ref{lem:cutoff}. Finally, in the last line, we used Lemma \ref{lem:weakestimates}. Rearranging the computation and using Lemma \ref{lem:weakestimates} again, we get
    \begin{align*}
        \widehat{H}_{\mathbf{x}_1}^u(\tau_1)\leq C(\Lambda)\widehat{H}^u(\tau).
    \end{align*}
    By symmetry, we get \ref{comparability of widehat H}.
    
    \ref{comparability of spatial derivative} Fix $v\in L$ with $|v|=1$. Using \eqref{eq:wherechisareequal}  and Lemma \ref{lemma polynomial growth of general solutions}, we have
    \begin{align} \label{eq:chiminuschi1}
        \begin{split}
            \int_{\mathbb{R}^n} |\nabla(\chi_1u-\chi u)|^2 e^{\frac{1}{8n}f} \,d\nu_{t_1-\tau}&\leq \frac{C \lambda^{\frac{\Upsilon}{2}}}{\tau^{\Upsilon}} \int_{B(0,20\tau^{\Upsilon}\lambda^{-\frac{\Upsilon}{2}})\setminus \overline{B}(0,\frac{1}{20}\tau^{\Upsilon}\lambda^{-\frac{\Upsilon}{2}})} \left( |u|^2 +|\nabla u|^2\right) e^{\frac{1}{8n}f}\,d\nu_{t_1-\tau} \\
            &\leq C(\Lambda,\Upsilon)e^{-\frac{1}{e^{70}r^{2-4\Upsilon}\lambda^{\Upsilon}}}\widehat{H}(\tau),
        \end{split}
    \end{align}
    where we used Lemma \ref{lem:HessianL2} and \ref{comparability of widehat H} in the last inequality. Using a coordinate change together with Lemma \ref{lem:cutoff} and $\verts{v-a^{-\frac{1}{2}}(\mathbf{x}_1)v} \leq \lambda|\mathbf{x}_1| \leq \lambda r$, we estimate
    \begin{align*}
        \int_{\mathbb{R}^n} |v\cdot \nabla (\chi u_{\mathbf{x}_1;1})|^2 d\nu_{-\tau}&= \int_{\mathbb{R}^n} |a^{-\frac{1}{2}}(\mathbf{x}_1)v\cdot \nabla (\chi_1 u)(z,t_1-\tau)|^2  \frac{e^{-\frac{(z-x_1)\cdot a(\mathbf{x}_1)(z-x_1)}{4\tau}}}{(4\pi \tau)^{\frac{n}{2}}\sqrt{\det a(\mathbf{x}_1)}}\,dz \\
        &\leq C\int_{\mathbb{R}^n} |v\cdot \nabla (\chi_1 u)|^2 e^{\frac{1}{16n}f_{\mathbf{x}_1}}d\nu_{\mathbf{x}_1;t_1-\tau} + \lambda^2 r^2 \int_{\mathbb{R}^n} |\nabla (\chi_1 u)|^2 e^{\frac{1}{16n}f_{\mathbf{x}_1}}d\nu_{\mathbf{x}_1;t_1-\tau}\\
        &\leq C  \int_{\mathbb{R}^n} |v\cdot \nabla  (\chi_1 u)|^2 e^{\frac{1}{8n}f}d\nu_{t_1-\tau}+C\lambda^2 r^2 \int_{\mathbb{R}^n} |\nabla (\chi_1 u)|^2 e^{\frac{1}{8n}f}d\nu_{t_1-\tau}.
    \end{align*}
    where for the last inequality we applied Lemma \ref{lemma-change-of-base-point} with $\mathbf{x}_0 \leftarrow \mathbf{x}_1$ $\mathbf{x}_1 \leftarrow \mathbf{0}$, $\tau \leftarrow t_1-\tau$, $\alpha_0 \leftarrow \frac{1}{16n}$, $\alpha \leftarrow \frac{1}{8n}$ and $\sigma \leftarrow \frac{1}{100n^2}r^2$. Combining this estimate with \eqref{eq:chiminuschi1} gives
    \begin{align*}
        r^2\int_{\mathbb{R}^n} |v\cdot \nabla (\chi u_{\mathbf{x}_1;1})|^2 d\nu_{-\tau} \leq C r^2 \int_{\mathbb{R}^n} |v\cdot \nabla  (\chi_1 u)|^2 e^{\frac{1}{8n}f}d\nu_{t_1-\tau}+C(\Lambda,\Upsilon)\lambda^2 r^2 \widehat{H}(\tau),
    \end{align*}
    and \ref{comparability of spatial derivative} then follows by Corollary \ref{cor:improvedexponent}.
    
    \ref{comparability of time derivative} Using \eqref{eq:wherechisareequal}  and Lemma \ref{lemma polynomial growth of general solutions}, we have
    \begin{align}\label{eq:chiminuschi2}
        \begin{split} 
            &\int_{e^{-1}r^2}^{e^3r^2} \frac{s}{\widehat{H}_{\mathbf{x}_1}^u(s)}  \int_{\mathbb{R}^n} |\partial_t (\chi_1u-\chi u)|^2 e^{\frac{1}{8n}f} \,d\nu_{-s}ds \\
            &\leq \frac{C}{r^4} \int_{e^{-1}r^2}^{e^3r^2} \frac{s e^{- \frac{1}{e^{60}s^{1-2\Upsilon}\lambda^{\Upsilon}}}}{\widehat{H}_{\mathbf{x}_1}^u(s)}  \int_{B(0,20 s^{\Upsilon}\lambda^{-\frac{\Upsilon}{2}})\setminus \overline{B}(0,\frac{1}{10}s^{\Upsilon}\lambda^{-\frac{\Upsilon}{2}})} \left( |u|^2 +|\partial_t u|^2\right) \,d\nu_{-s} ds \\
            &\leq C(\Lambda,\Upsilon)e^{-\frac{1}{e^{70}r^{2-4\Upsilon}\lambda^{\Upsilon}}},
        \end{split}
    \end{align}
    where we used Lemma \ref{lem:HessianL2} in the last inequality. On the other hand, \eqref{part-2-eq-parabolic-equation} yields
    \begin{align*}
        \int_{r^2}^{e^2r^2}\int_{\mathbb{R}^n} |\partial_t (\chi u_{\mathbf{x}_1;1})|^2 d\nu_{-\tau}d\tau&= \int_{r^2-t_1}^{e^2r^2-t_1} \int_{\mathbb{R}^n} |\partial_t (\chi_1 u)(z,-s)|^2  \frac{e^{-\frac{(z-x_1)\cdot a(\mathbf{x}_1)(z-x_1)}{4(s+t_1)}}}{(4\pi (s+t_1))^{\frac{n}{2}}\sqrt{\det a(\mathbf{x}_1)}}\,dz ds \\
        &\leq C\int_{r^2-t_1}^{e^2r^2-t_1} \int_{\mathbb{R}^n} |\partial_t (\chi_1 u)|^2 e^{\frac{1}{16n}f_{\mathbf{x}_1}}d\nu_{\mathbf{x}_1;-s}ds\\
        &\leq C \int_{r^2-t_1}^{e^2r^2-t_1} \int_{\mathbb{R}^n} |\partial_t (\chi_1 u)|^2 e^{\frac{1}{8n}f}d\nu_{-s}ds,
    \end{align*}
    where for the last inequality, we applied Lemma \ref{lemma-change-of-base-point} with $\mathbf{x}_0 \leftarrow \mathbf{x}_1$ $\mathbf{x}_1 \leftarrow \mathbf{0}$, $\tau \leftarrow s$, $\alpha_0 \leftarrow \frac{1}{16n}$, $\alpha \leftarrow \frac{1}{8n}$ and $\sigma \leftarrow \frac{1}{100}r$.
    We now combine this with \eqref{eq:chiminuschi2}, apply Corollary \ref{cor:improvedexponent}, and then apply Lemma \ref{lemma: propagation of symmetry and pinching-2} to obtain the claim.
\end{proof}

We also have a version of the gradient estimates similar to Lemma \ref{frequencydirectionalestimate}. We define $\widehat{H}_s(\mathbf{x})\coloneqq H_{s}^{\chi u_{\mathbf{x};1}}(\mathbf{x})$. It allows us to propagate symmetry along the plane of symmetry.
\begin{lemma}\label{lemma harnack for doubling}
    The following holds if $\lambda \leq \ol{\lambda}(\Lambda, \Upsilon)$. For any $\beta, r \in (0,1]$, $\mathbf{x}_0\in P(\mathbf{0},10)$, and any $k$-plane $L\subset \R^n$, there exists $C\coloneqq C(\beta,\Lambda,\Upsilon)$ such that for all $\mathbf{x}_1 \in ((L\times \mathbb{R})+\mathbf{x}_0) \cap P(\mathbf{x}_0,\frac{1}{10}r)$ and $s\in [\beta^2 r^2,r^2]$: 
    \begin{align*}
        \left|\log \left( \frac{\widehat{H}_s(\mathbf{x}_1)}{\widehat{H}_{s}(\mathbf{x}_0)}\right) \right|^2 &\leq C \frac{|x_1-x_0|^2}{r^2} \int_{r^2}^{e^2 r^2}\frac{1}{\widehat{H}_\tau(\mathbf{x}_0)}\int_{\mathbb{R}^n} |\pi_L\nabla (\chi u_{\mathbf{x}_0;1})|^2 \,d\nu_{-\tau}d\tau \\ 
        &\quad +C \frac{|t_1-t_0|}{r^2} \int_{r^2}^{e^2r^2} \frac{\tau}{\widehat{H}_\tau(\mathbf{x}_0)}\int_{\mathbb{R}^n} |\partial_t(\chi u_{\mathbf{x}_0;1})|^2 \,d\nu_{-\tau}d\tau+C\lambda^2 r^2.
    \end{align*}
\end{lemma}

\begin{proof}
    By replacing $u$ with $u_{\mathbf{x}_0;1}$, we may assume that $\mathbf{x}_0=\mathbf{0}$ and $u=u_{\mathbf{x}_0;1}$. After a rotation, we can write $v\cdot \cd w$ as $e_1\cdot \cd w=\partial_{x_1}w$ for any $v\in L$. At $\mathbf{x}=(x,t)$, we compute
    \begin{align*}
        \pd_{x_1} \widehat{H}_s(\mathbf{x})&= 2\int_{\R^n} (\chi u_{\mathbf{x};1})(y,-s) \pd_{x_1}(\chi(y,-s) u(x + a^{-\frac{1}{2}}(\mathbf{x})y,t-s))\,d\nu_{-s}(y)\\
        &= 2\int_{\R^n}( \chi^2 u_{\mathbf{x};1})(y,-s)\left( (e_1 \cdot \nabla u)_{\mathbf{x};1}(y,-s) + y\cdot (\partial_{x_1}a^{-\frac{1}{2}}(\mathbf{x}))(\nabla u)_{\mathbf{x};1}(y,-s)\right)
        \,d\nu_{-s}(y)\\
        &= 2\int_{\mathbb{R}^n} (u_{\mathbf{x};1}\chi^2)(y,-s)e_1 \cdot (I- a^{-\frac{1}{2}}(\mathbf{x}))(\nabla u)_{\mathbf{x};1}(y,-s)\,d\nu_{-s}(y)\\
        &\qquad -2\int_{\mathbb{R}^n} u_{\mathbf{x};1}^2(y,-s)\chi(y,s)\partial_{y_1} \chi(y,-s) \,d\nu_{-s}(y)\\
        &\qquad +2\int_{\R^n}( \chi^2 u_{\mathbf{x};1})(y,-s) y\cdot (\partial_{x_1}a^{-\frac{1}{2}}(\mathbf{x}))(\nabla u)_{\mathbf{x};1}(y,-s)
        \,d\nu_{-s}(y)\\
        &\qquad +2 \int_{\mathbb{R}^n} (u_{\mathbf{x};1}\chi)(y,-s)\partial_{y_1}(u_{\mathbf{x};1}\chi)(y,-s)\,d\nu_{-s}(y)\\
        &\eqqcolon K_1+K_2+K_3+K_4.
    \end{align*}
    Using the $\lambda$-Lipschitz property of $a$ in \eqref{part-2-eq-parabolic-equation} and Lemma \ref{lem:cutoff}, we get
    \begin{align*}
        K_1\leq 2\lambda \verts{\mathbf{x}} \int_{\mathbb{R}^n} (u_{\mathbf{x};1}\chi^2)(y,-s)\verts*{(\nabla u)_{\mathbf{x};1}(y,-s)}\,d\nu_{-s}(y)\leq C(\Lambda,\Upsilon) \lambda \verts{\mathbf{x}}\widehat{H}_s(\mathbf{x}).
    \end{align*}
    We can use Lemma \ref{lem:cutoff} again to get
    \begin{align*}
        K_2\leq C(\Lambda,\Upsilon) \lambda^{\Upsilon} e^{-\frac{1}{8\lambda^{\Upsilon}s^{1-2\Upsilon}}}\widehat{H}_s(\mathbf{x}).
    \end{align*}
    Furthermore, we use \eqref{part-2-eq-parabolic-equation} to get
    \begin{align*}
        K_3&\leq 2\lambda \sqrt{s} \int_{\R^n}( \chi^2 u_{\mathbf{x};1})(y,-s) \frac{\verts{y}}{\sqrt{s}}\verts{(\nabla u)_{\mathbf{x};1}(y,-s)}
        \,d\nu_{-s}(y)\\
        &\leq 2\lambda \sqrt{s} \int_{\R^n}( \chi^2 u_{\mathbf{x};1})(y,-s) \verts{(\nabla u)_{\mathbf{x};1}(y,-s)}
        e^{\frac{1}{2}f}\,d\nu_{-s}(y)\leq C(\Lambda,\Upsilon) \lambda \sqrt{s}\widehat{H}_s(\mathbf{x}),
    \end{align*}
    where we use Lemma \ref{lem:cutoff} in the last line.
    Finally, using H\"older's inequality, we get
    \begin{align*}
        K_4&\leq 2 \left(\int_{\mathbb{R}^n} (u_{\mathbf{x};1}\chi)^2(y,-s) e^{\frac{1}{2}f}\,d\nu_{-s}(y)\right)^{\frac{1}{2}}\left(\int_{\R^n}\verts{\partial_{y_1}(u_{\mathbf{x};1}\chi)(y,-s)}^2e^{-\frac{1}{2}f}\,d\nu_{-s}(y)\right)^{\frac{1}{2}}\\
        &\leq C(\Lambda, \Upsilon) \left(\widehat{H}_s(\mathbf{x})\int_{\R^n}\verts{\partial_{y_1}(u_{\mathbf{x};1}\chi)(y,-s)}^2e^{-\frac{1}{2}f}\,d\nu_{-s}(y)\right)^{\frac{1}{2}},
    \end{align*}
    where we use Lemma \ref{lem:cutoff} in the last line. On the other hand, using Proposition \ref{lemma: propagation of symmetry and pinching-2}\ref{propagation of symmetry-2}, Lemma \ref{lemma comparability of H}\ref{comparability of widehat H}\ref{comparability of spatial derivative}, and Lemma \ref{lem:weakestimates}, we get
    \begin{align*}
        &\int_{\R^n}\verts{\partial_{y_1}(\chi u_{\mathbf{x};1})(y,-s)}^2e^{-\frac{1}{2}f}\,d\nu_{-s}(y)\\
        &\leq C(\beta,\Lambda,\Upsilon)\int_{\R^n}\verts{\partial_{y_1}(\chi u_{\mathbf{x};1})(y,-r^2)}^2e^{-\frac{1}{2}f}\,d\nu_{-r^2}(y) +C\lambda^{\Upsilon}r^{4\Upsilon}\widehat{H}_{r^2}({\mathbf{x}})\\
        &\leq C(\beta,\Lambda,\Upsilon)\int_{\R^n}\verts{\partial_{y_1}(\chi u)(y,-r^2)}^2\,d\nu_{-r^2}(y) +C\lambda^{\Upsilon}r^{4\Upsilon}\widehat{H}_{r^2}(\mathbf{0}).
    \end{align*}
    Therefore, combining estimates and using Lemma \ref{lemma comparability of H}\ref{comparability of widehat H}, we get
    \begin{align*}
        |\partial_{x_1}\log \widehat{H}_s(\mathbf{x})|^2\leq C(\beta,\Lambda,\Upsilon) \frac{1}{\widehat{H}_{r^2}(\mathbf{0})}\int_{\R^n} \verts{\partial_{y_1}(\chi u)}^2\, d\nu_{-r^2}(y)+ C(\Lambda, \Upsilon) \lambda^2 r^2.
    \end{align*}
    Integrating this from $\mathbf{0}$ to $(x_1,0)$ and then applying Lemma \ref{lemma comparability of H}\ref{comparability of widehat H} and \eqref{eq needed monotonicity for spatial derivative}, we obtain
    \begin{align*}
        \verts*{\log \left(\frac{ \widehat{H}_s((x_1,0))}{\widehat{H}_s(\mathbf{0})}\right)}^2\leq &C(\beta,\Lambda,\Upsilon) \frac{|x_1|^2}{\widehat{H}_{r^2}(\mathbf{0})}\int_{\R^n} \verts{\partial_{y_1}(\chi u)}^2\, d\nu_{-r^2}(y)+ C(\Lambda, \Upsilon) \lambda^2 r^2 \\
        \leq & C(\beta,\Lambda,\Upsilon)\frac{|x_1|^2}{r^2} \int_{r^2}^{e^2 r^2}\frac{1}{\widehat{H}_{\tau}(\mathbf{0})}\int_{\R^n} \verts{\partial_{y_1}(\chi u)}^2\, d\nu_{-\tau}(y)d\tau+ C(\Lambda, \Upsilon) \lambda^2 r^2.
    \end{align*}

    Next, we estimate how $\widehat{H}_s$ changes in the time direction. Let $\{\mathbf{x}_j'\coloneqq (x_1,t_j')\}_{j=0}^M$ with $t_0'=0$ and $t_M'=t_1$, where $M\leq C(\beta)$ and $|t_j'-t_{j+1}'|\leq \frac{\beta^2 r^2}{100}$. Then
    \begin{align*}
        &\widehat{H}_s(\mathbf{x}_{j-1}')-\widehat{H}_s(\mathbf{x}_j')\\
        &= \int_{\R^n}\left(  (\chi u_{\mathbf{x}_{j-1}';1})^2-(\chi u_{\mathbf{x}_j';1})^2 \right)\,d\nu_{-s}\\
        &=\int_{\R^n} \left((\chi u_{\mathbf{x}_{j-1}';1})^2(y,-s)- (\chi u_{\mathbf{x}_{j-1}';1})^2(y,t_j'-t_{j-1}'-s)\right)\,d\nu_{-s}\\
        &\quad- \int_{\R^n} \left(\chi^2(y,-s)-\chi^2(y,t_j'-t_{j-1}'-s)\right) u_{\mathbf{x}_{j-1}';1}^2(y,t_j'-t_{j-1}'-s)\,d\nu_{-s}\\
        &\quad +\int_{\R^n} \chi^2(y,-s)\left(u_{\mathbf{x}_{j-1}';1}^2(y,t_j'-t_{j-1}'-s)- u_{\mathbf{x}_{j-1}';1}^2(a^{\frac{1}{2}}(\mathbf{x}_{j-1}')a^{-\frac{1}{2}}(\mathbf{x}_j')y,t_j'-t_{j-1}'-s)\right)\,d\nu_{-s}\\
        &\eqqcolon L_1+L_2+L_3.
    \end{align*}
    We first estimate $L_1$. Using H\"older's inequality, Lemma \ref{lem:cutoff}, and Lemma \ref{lem:weakestimates}, we get
    \begin{align*}
        \verts{L_1}&=2\left|\int_{\R^n} \left(\int_{t_j'-t_{j-1}'}^0(\chi u_{\mathbf{x}_{j-1}';1})(y,\ol{s}-s)\partial_{t}(\chi u_{\mathbf{x}_{j-1}';1})(y,\ol{s}-s)\,d\ol{s}\right)d\nu_{-s}\right|\\
        &\leq C(\beta,\Lambda,\Upsilon) \left|t_1 \widehat{H}_{r^2}(\mathbf{x}_{j-1}') \int_{t_j'-t_{j-1}'}^0 \int_{\mathbb{R}^n} \verts*{\partial_{t}(\chi u_{\mathbf{x}_{j-1}';1})(y,\overline{s}-s)}^2 e^{-\frac{1}{2}f}\,d\nu_{-s}d\overline{s}\right|^{\frac{1}{2}}\\
        &\leq C(\beta,\Lambda, \Upsilon)\left|t_1  \widehat{H}_s(\mathbf{x}_{j-1}') \int_{s-t_j'+t_{j-1}'}^{s}\int_{\R^n} \verts{\partial_t(\chi u_{\mathbf{x}_{j-1}';1}) (y,-\tau)}^2 e^{-\frac{1}{4}f}\,d\nu_{-\tau}d\tau\right|^{\frac{1}{2}}\\
        & \leq  C(\beta,\Lambda, \Upsilon)\left| t_1  \widehat{H}_s(\mathbf{x}_{j-1}') \int_{\frac{r^2}{100}}^{\frac{r^2}{10}} \int_{\mathbb{R}^n} \verts{\partial_t(\chi u_{\mathbf{x}_{j-1}';1})}^2e^{-\frac{1}{4}f}d\nu_{-\tau}d\tau\right|^\frac{1}{2}\\
        & \leq C(\beta,\Lambda, \Upsilon)\left( \verts{t_1}  \widehat{H}_s(\mathbf{x}_{j-1}') \int_{\frac{r^2}{20}}^{\frac{r^2}{10}} \int_{\mathbb{R}^n} \verts{\partial_t(\chi u)}^2d\nu_{-\tau}d\tau \right)^\frac{1}{2}\\
        &\leq C(\beta,\Lambda,\Upsilon)\widehat{H}_s(\mathbf{x}_{j-1}')\left(\frac{\verts{t_1}}{r^2} \int_{r^2}^{e^2r^2} \frac{\tau}{\widehat{H}_\tau(\mathbf{0})}\int_{\mathbb{R}^n} \verts{\partial_t(\chi u)}^2d\nu_{-\tau}d\tau\right)^\frac{1}{2},
    \end{align*}
    where we used Lemma \ref{lemma-change-of-base-point} in the third line. Further, we used Proposition \ref{lemma: propagation of symmetry and pinching-2}\ref{propagation of symmetry-2} in the fourth line. Moreover, we used Lemma \ref{lemma comparability of H}\ref{comparability of time derivative} in the fifth line. Finally, we used Proposition \ref{lemma: propagation of symmetry and pinching-2}\ref{propagation of symmetry-2} and Lemma \ref{lemma comparability of H}\ref{comparability of widehat H} in the last line.

    We now estimate $L_2$ using Lemma \ref{lem:cutoff} and Lemma \ref{lemma comparability of H}\ref{comparability of widehat H}:
    \begin{align*}
        L_2&=\int_{\R^n} \left(\int_{t_j'-t_{j-1}'}^0\partial_{\ol{s}}\chi^2(y,\ol{s}-s)\,d\ol{s}\right) \,u^2(y,t_j'-t_{j-1}'-s)\,d\nu_{-s}\\
        &\leq C(\beta,\Lambda,\Upsilon) \lambda^{\Upsilon} \verts{t_1}e^{-\frac{1}{10\lambda^{\Upsilon}s^{1-2\Upsilon}}}\widehat{H}_s(\mathbf{x}_{j-1}').
    \end{align*}
    Finally, we estimate $L_3$. Define $A(\ell)\coloneqq \ell a^{\frac{1}{2}}(\mathbf{x}_{j-1}')a^{-\frac{1}{2}}(\mathbf{x}_j')+(1-\ell)I$ for $\ell=[0,1]$. Note that $\verts{A'(\ell)}\leq \lambda \sqrt{\verts{t}_1}$. Then
    \begin{align*}
        &\left|\frac{d}{d\ell} \int_{\R^n} \chi^2(y,-s) u_{\mathbf{x}_{j-1}';1}^2(A(\ell)y,t_j'-t_{j-1}'-s)\,d\nu_{-s}\right|\\
        &=\left|\int_{\R^n}\chi^2(y,-s)2u_{\mathbf{x}_{j-1}';1}(A(\ell)y,t_j'-t_{j-1}'-s)(\cd u_{\mathbf{x}_{j-1}';1}(A(\ell)y,t_j'-t_{j-1}'-s))\cdot (A'(\ell)y)\,d\nu_{-s} \right| \\
        &\leq C\verts{A'(\ell)}\int_{\R^n}\chi^2(A(\ell)^{-1}y,-s)|u_{\mathbf{x}_{j-1}';1}(y,t_j'-t_{j-1}'-s)||(\cd u_{\mathbf{x}_{j-1}';1}(y,t_j'-t_{j-1}'-s))|| y|e^{\frac{1}{4}f}\,d\nu_{-s}\\ 
        &\leq \lambda \sqrt{\verts{t_1}} C(\beta,\Lambda,\Upsilon)\widehat{H}_s(\mathbf{x}_{j-1}'),
    \end{align*}
    where we changed variables $y\leftarrow A(\ell)^{-1}y$ in the third line and in the last line we used a computation similar to that in Lemma \ref{lem:cutoff}. Integration from $\ell=0$ to $\ell=1$ yields
    \begin{align*}
        |L_3| \leq \lambda \sqrt{|t_1|}C(\beta,\Lambda,\Upsilon)\widehat{H}_s(\mathbf{x}_{j-1}').
    \end{align*}
    Combining the estimates for $L_1$, $L_2$, and $L_3$, we get
    \begin{align}
    \begin{split}
        &\verts*{\frac{\widehat{H}_s(\mathbf{x}_j')-\widehat{H}_s(\mathbf{x}_{j-1}')}{\widehat{H}_s(\mathbf{x}_{j-1}')}}\\
        &\leq C(\Lambda, \Upsilon)\left(\frac{\verts{t_1}}{r^2}\int_{r^2}^{e^2 r^2} \frac{\tau}{\widehat{H}_{\tau}(\mathbf{0})}\int_{\R^n} \verts{\partial_t(\chi u) (y,-\tau)}^2 \,d\nu_{-\tau}d\tau \right)^{\frac{1}{2}}+C(\beta,\Lambda,\Upsilon)\lambda \sqrt{\verts{t_1}}.\label{eq elemetary estimate for logs}
        \end{split}
    \end{align}
    Since $\left| \log \left( \frac{\widehat{H}_s(\mathbf{x}_1)}{\widehat{H}_s(\mathbf{0})} \right) \right|\leq C(\Lambda)$ by Lemma \ref{lemma comparability of H}\ref{comparability of widehat H}, we can assume that the right-hand side of \eqref{eq elemetary estimate for logs} is at most $\frac{1}{4}$. In that case, we can use the following elementary estimate
    \begin{align*}
        \left|\log  \left( \frac{\widehat{H}_s(\mathbf{x}_{j}')}{\widehat{H}_s(\mathbf{x}_{j-1}')} \right) \right| = \left|\log  \left(1+ \frac{\widehat{H}_s(\mathbf{x}_{j}')-\widehat{H}_s(\mathbf{x}_{j-1}')}{\widehat{H}_s(\mathbf{x}_{j-1}')} \right) \right|\leq 4\verts*{\frac{\widehat{H}_s(\mathbf{\mathbf{x}_j'})-\widehat{H}_s(\mathbf{x}_{j-1}')}{\widehat{H}_s(\mathbf{x}_{j-1}')}}
    \end{align*}
    to obtain the claim.
\end{proof}

\subsubsection{Symmetry and pinching}
As a modification of Theorem \ref{thm: sym split equiv}, we prove that $\widehat{\mathcal{E}}_r^{k,\alpha}(\mathbf{x})$ is small if and only if $\mathbf{x}$ is frequency-pinched and almost $k$-symmetric at scale $r$.
\begin{theorem} \label{thm:part2-sym split equiv}
   If $\lambda \leq \overline{\lambda}(\Lambda)$, then the following hold for any $\epsilon \in(0,1]$. Let $u$ satisfy \eqref{part-2-eq-parabolic-equation} and \eqref{eq:strongerdoubling}, $r\in (0,1]$, and $\mathbf{x}_0 \in P(\mathbf{0},10)$.
    \begin{enumerate}[label={(\arabic*)}]
    
        \item Then $u$ is weakly $(k,C(\Lambda,\Upsilon)\alpha^{-n}\widehat{\mathcal{E}}^{k,\alpha}_{r}(\mathbf{x}_0),r)$-symmetric at $\mathbf{x}_0$.\label{thm:part2-sym split equiv-1}
        
        \item \label{thm:part2-sym split equiv-2} Suppose $u$ is weakly $(k,\epsilon,10^2r)$-symmetric at $\mathbf{x}$ with respect to $V\in{\rm Gr}_{\mathcal{P}}(k)$ and
        \begin{align*}
            |\widehat{N}_{\mathbf{x}}(10^{2}r^2)-\widehat{N}_{\mathbf{x}}(10^{-2} \kappa^2 r^2)| + \lambda^{\frac{\Upsilon}{4}}r^{\Upsilon} \leq \epsilon.
        \end{align*}
        Then
        \begin{align}\label{freqatnearbypoints part 2}
            \verts*{\widehat{N}_{\mathbf{v}}(50r^2)-\widehat{N}_{\mathbf{v}}\left(50^{-1}\kappa^2r^2\right)}+\lambda^{\frac{\Upsilon}{4}}r^{\Upsilon}\leq C(\kappa,\Lambda,\Upsilon)\sqrt{\epsilon}
        \end{align}
        for any $\mathbf{v} \in (\mathbf{x}+V)\cap P(\mathbf{x},10r)$. As a consequence, $\widehat{\mathcal{E}}^{k,1}_s(\mathbf{v})\le C(\kappa,\Lambda,\Upsilon)\sqrt{\epsilon}$, and
        $u$ is weakly $(k,C(\kappa,\Lambda,\Upsilon)\sqrt{\epsilon},s)$-symmetric at $\mathbf{v}$ with respect to $V$ for all $s\in [\kappa r,r]$.
    \end{enumerate}
\end{theorem}

\begin{proof}[Proof of Theorem \ref{thm:part2-sym split equiv}\ref{thm:part2-sym split equiv-2}] 
    Using \eqref{eq derivative of log H} and Theorem \ref{weakmonotonicitylemma}\ref{weakmonotonicity}, we have
    \begin{align}
        \log \left( \frac{\widehat{H}_{\tau}(\mathbf{x})}{\widehat{H}_{\frac{1}{2}\tau}(\mathbf{x})} \right)- C\lambda^{\Upsilon} \tau^{\Upsilon} \leq \log(2)\widehat{N}_{\mathbf{x}}(\tau)\leq \log \left(\frac{\widehat{H}_{2\tau}(\mathbf{x})}{\widehat{H}_{\tau}(\mathbf{x})}\right)+C\lambda^{\Upsilon}\tau^{\Upsilon}\label{eq comparability of frequency and doubling}
    \end{align}
    for all $\mathbf{x} \in P(\mathbf{0},10)$ and $\tau \in (0,1]$. This together with Lemma \ref{lemma harnack for doubling} implies that for any $\mathbf{v}\in (\mathbf{x}+V)\cap P(\mathbf{x},10r)$ and $\tau \in (0,10^2 r^2]$, we have
    \begin{align*}
        \log(2)\widehat{N}_{\mathbf{x}}(\tau)-C\lambda^{\Upsilon}\tau^{\Upsilon}&\leq \log \left(\frac{\widehat{H}_{2\tau}(\mathbf{x})}{\widehat{H}_{\tau}(\mathbf{x})}\right)\leq \log \left(\frac{\widehat{H}_{2\tau}(\mathbf{v})}{\widehat{H}_{\tau}(\mathbf{v})}\right)+ C\sqrt{\epsilon}+C\lambda r\\
        &\leq \log (2)\widehat{N}_{\mathbf{v}}(2\tau) +C\sqrt{\epsilon}+C \lambda^{\Upsilon}\tau^{\Upsilon},
    \end{align*}
    so that 
    \begin{align}
        \widehat{N}_{\mathbf{x}}(\tau) \leq \widehat{N}_{\mathbf{v}}(2\tau)+C\sqrt{\epsilon}+C\lambda^{\frac{\Upsilon}{4}}r^{\Upsilon}\label{eq upper bound for nx}
    \end{align}
    when $\tau \in [50^{-1}\kappa^2 r^2,50r^2]$. Similarly, we have
    \begin{align}
        \widehat{N}_{\mathbf{x}}(\tau) \geq \widehat{N}_{\mathbf{v}}(\frac{1}{2}\tau) - C\sqrt{\epsilon} - C\lambda^{\frac{\Upsilon}{4}}r^{\Upsilon}.\label{eq lower bound for nx}
    \end{align}
    Combining \eqref{eq upper bound for nx} and \eqref{eq lower bound for nx}, we get \eqref{freqatnearbypoints part 2}. The remaining claims follow from the proof of \ref{thm:part2-sym split equiv-1}.
\end{proof}

We will prove Theorem \ref{thm:part2-sym split equiv}\ref{thm:part2-sym split equiv-1} as a consequence of a cone-splitting inequality in Theorem \ref{theorem cone splitting inequality in general}. While proving the cone splitting inequality, we will need the following lemma which establishes the equivalence of two weighted $L^2$-norms of polynomials.
\begin{lemma}\label{lemma polynomials are good}
    Fix $\alpha>0$. Suppose $p$ is a polynomial on $\mathbb{R}^n$ of degree at most $m$. Then
    \begin{align*}
       \int_{\mathbb{R}^n} p^2\, d\nu_{\mathbf{x}_0;t}\leq e^{C\alpha(1+m)}\int_{\R^n} p^2 e^{-\alpha f_{\mathbf{x}_0}} \,d\nu_{\mathbf{x}_0;t} 
    \end{align*} 
    for all $\mathbf{x}_0 =(x_0,t_0)\in \mathbb{R}^n \times \mathbb{R}$ and $t<t_0$.
\end{lemma}

\begin{proof}
    By a change of coordinate $x\leftarrow \frac{x-x_0}{\sqrt{t-t_0}}$, we assume that $t=-1$ and $\mathbf{x}=0$.
    
    Since $p$ has degree at most $m$, we can write $p=\sum_{k=0}^m \widehat{p}_k$ where $\widehat{p}_k\in \widehat{\cP}_m$. Therefore, integration by parts gives
    \begin{align}
        \int_{\R^n} \verts{\cd p}^2\,d\nu= -\sum_{k=0}^m\int_{\R^n} \widehat{p}_k\Delta_f \widehat{p}_k\,d\nu =\sum_{k=0}^m\frac{k}{2}\int_{\R^n} \widehat{p}_k^2\,d\nu\leq \frac{m}{2} \int_{\R^n} p^2\,d\nu.\label{eq gradient norm is bounded by l2 norm}
    \end{align}
    Using \eqref{eq gradient norm is bounded by l2 norm} and \eqref{eq integral of nice quantity} with $\alpha\leftarrow 0$, we get
    \begin{align}
        \int_{\R^n} fp^2 d\nu \leq C\int_{\R^n} (p^2+ \verts{\cd p}^2)\,d\nu\leq C(1+m)\int_{\R^n} p^2\,d\nu. \label{eq nice bound for fp2}
    \end{align}
    Fix $R^2 = 2C(1+m)$. Then
    \begin{align*}
        \int_{\R^n} p^2 \,d\nu &= \int_{\R^n\setminus B_R} p^2\,d\nu +\int_{B_R} p^2\,d\nu \leq \frac{4}{R^2} \int_{\R^n\setminus B_R} fp^2 \,d\nu + e^{\frac{\alpha R^2}{4}} \int_{B_R} p^2 e^{-\alpha f}\,d\nu\\
        &\leq \frac{C(1+m)}{R^2} \int_{\R^n} p^2\,d\nu +e^{\frac{\alpha R^2}{4}}\int_{\R^n}p^2 e^{-\alpha f}\,d\nu.
    \end{align*}
    Rearranging the inequality gives the desired result.
\end{proof}

We now prove the cone splitting inequality as an analog of Proposition \ref{proposition-spatial-splitting-two-points}.
\begin{theorem}\label{theorem cone splitting inequality in general}
    The following holds if $\lambda \leq \overline{\lambda}(\Lambda)$. Let $u$ satisfy \eqref{part-2-eq-parabolic-equation} and \eqref{eq:strongerdoubling}, and suppose $r\in (0,1]$. For any $\mathbf{x}_0,\mathbf{x}_1 \in P(\mathbf{0},10)$ satisfying $|\mathbf{x}_1-\mathbf{x}_0|<\frac{1}{10}r$, 
    \begin{align*}
        &\int_{e^{-2}r^2}^{e^2r^2} \frac{1}{\widehat{H}_{\mathbf{x}_j}^u(\tau)}\int_{\mathbb{R}^n} \left( |(x_1 -x_0)\cdot \nabla (\chi u_{\mathbf{x}_j;1})|^2 + |(t_1-t_0)\cdot \partial_t (\chi u_{\mathbf{x}_j;1})|^2 \right) d\nu_{-\tau} \frac{d\tau}{\tau} \\
        &\qquad \leq C(\Lambda,\Upsilon)\widehat{\mathcal{E}}_{r}(\{\mathbf{x}_1,\mathbf{x}_0\})
    \end{align*}
    for $j=0,1$.
\end{theorem}

\begin{proof} 
    We will prove the statement with $j=0$. By replacing $u$ with $u_{\mathbf{x}_0;1}$ we can assume that $\mathbf{x}_0=\mathbf{0}$ and $u_{\mathbf{x}_0;1}=u$. By Lemma \ref{lem:higherordercloseness} and Lemma \ref{prop:closeheatpolynomial}, there exist $p_i \in\mathcal{P}_{m_i}$ such that
    \begin{align}\label{eq:whatweshouldhavedone}
        \begin{split} 
            &\sup_{\tau\in [\frac{1}{40}r^2,40r^2]} \frac{1}{\widehat{H}_{\mathbf{x}_i}^u(\tau)}\int_{\R^n} \left(\chi u_{\mathbf{x}_i;1}-\sqrt{\frac{\widehat{H}^u_{\mathbf{x}_i}(\tau)}{\tau^{m_i}}}p_i \right)^2\,d\nu_{-\tau}\\
            &\qquad \leq C(\Lambda,\Upsilon)(\Verts{\widetilde{w}_{\mathbf{x}_i}-\widehat{p}_i}_{[-\log(50)-2\log(r),\log(50)-2\log(r)]}+\lambda^{\frac{\Upsilon}{2}} r^{2\Upsilon}) \leq C(\Lambda,\Upsilon) \widehat{\mathcal{E}}_{r}(\mathbf{x}_i)
        \end{split}
    \end{align}
    for all $i\in \{0,1\}$.  For ease of notation, we define 
    \begin{align*}
        q_i(x,t)\coloneqq \sqrt{\frac{\widehat{H}_{\mathbf{x}_i}^u(\tau_i)}{\tau_i^{m_i}}} p_i(a^{\frac{1}{2}}(\mathbf{x}_i)(x-x_i),t-t_i),\qquad \chi_i(x,t)\coloneqq \chi(a^{\frac{1}{2}}(\mathbf{x}_i)(x-x_i),t-t_i).
    \end{align*}
    Note that for all $i,j\in \{0,1\}$ we have
    \begin{align*}
        \supp \chi_j (\cdot, t) \subset B(x_i, 10\lambda^{-\frac{\Upsilon}{2}}\tau_i^\Upsilon), \qquad  \chi_j(\cdot,t)\equiv 1 \text{ in } B(x_i, 10^{-1}\lambda^{-\frac{\Upsilon}{2}} \tau_i^\Upsilon).
    \end{align*}
    
    Let $\overline{q}_i$ be the $L^2(\mathbb{R}^n,\nu_{\mathbf{x}_i;t_i-\tau})$-projection of $q_i$ onto the $m_i$-eigenspace of $-2\tau_i \Delta_{f_i}$. For $\tau\in [\frac{1}{10}r^2,10r^2]$, we have
    \begin{align} \label{eq I defs}
        \begin{split}
            \int_{\mathbb{R}^n} (q_0-\overline{q}_i)^2 e^{-\frac{1}{8} f_i} d\nu_{\mathbf{x}_i;-\tau}&\leq  
            C\int_{\mathbb{R}^n} (q_i-\overline{q}_i)^2  d\nu_{\mathbf{x}_i;-\tau}+
            C \int_{\mathbb{R}^n} (q_i- \chi_i u)^2 e^{-\lambda f_i} d\nu_{\mathbf{x}_i;-\tau} \\
            &\qquad + C \int_{\mathbb{R}^n} (\chi_i u-\chi u)^2 d\nu_{\mathbf{x}_i;-\tau} +C\int_{\mathbb{R}^n} (\chi u -q_0)^2 e^{-\frac{1}{8}f_i} d\nu_{\mathbf{x}_i;-\tau}\\
            &\eqqcolon I_1+I_2+I_3+I_4.
        \end{split}
    \end{align}
    We first estimate $I_1$. Since $p_i$ is a polynomial of degree at most $m_i$, we can argue as in \eqref{eq gradient norm is bounded by l2 norm} to get
    \begin{align}
        \int_{\R^n} \tau^2 \verts{\cd^2 p_i}^2\,d\nu_{-\tau_i}\leq C(m)\int_{\R^n} p_i^2\,d\nu_{-\tau_i}.\label{eq estimate for hessian of p}
    \end{align}
    Furthermore, 
    \begin{align}
        \int_{\R^n} \verts{x\otimes \cd p_i}^2\,d\nu_{-\tau_i}\leq \int_{\R^n} \tau f \verts{\cd p}^2\,d\nu_{-\tau_i}\leq C\int_{\R^n} p_i^2\,d\nu_{-\tau_i},\label{eq estimate for x times p}
    \end{align}
    where in the last inequality we used \eqref{eq nice bound for fp2} for the components of $\cd p$ and argued as in \eqref{eq gradient norm is bounded by l2 norm}. Since $q_i$ are polynomials of degree at most $m_i$, we can use Lemma \ref{lemma polynomials are good} to get 
    \begin{align*}
        &\int_{\R^n}(2\tau_i\Delta_{f_i} q_i+m_iq_i)^2 \,d\nu_{\mathbf{x}_i;t}\\
        &\leq C \int_{\R^n}(2\tau_i\Delta_{f_i} q_i+m_iq_i)^2 e^{-10\lambda f_i}\,d\nu_{\mathbf{x}_i;t}\\
        &=\frac{C\widehat{H}_{\mathbf{x}_i}^u(\tau_i)}{\tau_i^{m_i}}\int_{\R^n }\left(2\tau \tr\left((a(\mathbf{x}_i)-I) \left(\cd^2 p_i - \frac{x\otimes \cd p_i}{2\tau }\right)\right)\right)^2(a^{\frac{1}{2}}(\mathbf{x}_i)(x-x_i),-\tau_i) e^{-10\lambda f_i}\,d\nu_{\mathbf{x}_i;t}(x)\\
        &\leq \frac{C\widehat{H}_{\mathbf{x}_i}^u(\tau_i)}{\tau_i^{m_i}}\lambda^2 r^2\int_{\R^n} \tau^2\left|\cd^2 p_i - \frac{x\otimes \cd p_i}{2\tau}\right|^2(a^{\frac{1}{2}}(\mathbf{x}_i)(x-x_i),-\tau_i)e^{-10\lambda f_i}\,d\nu_{\mathbf{x}_i;t}(x)\\
        &\leq \frac{C\widehat{H}_{\mathbf{x}_i}^u(\tau_i)}{\tau_i^{m_i}}\lambda^2 r^2 \int_{\R^n}\tau^2\left|\cd^2 p_i - \frac{x\otimes \cd p_i}{2\tau}\right|^2e^{-2\lambda f}\,d\nu_{-\tau_i}\\
        &\leq \frac{C\widehat{H}_{\mathbf{x}_i}^u(\tau_i)}{\tau_i^{m_i}}\lambda^2 r^2 \int_{\R^n}\tau^2\left|\cd^2 p_i\right|^2 + \left|x\otimes \cd p_i\right|^2\,d\nu_{-\tau_i}\leq \frac{C\widehat{H}_{\mathbf{x}_i}^u(\tau_i)}{\tau_i^{m_i}}\lambda^2 r^2 \int_{\R^n} p_i^2\,d\nu_{-\tau_i}\leq C\lambda^2r^2 \widehat{H}_{\mathbf{x}_i}^u(\tau_i),
    \end{align*}
    where we used \eqref{eq estimate for hessian of p} and \eqref{eq estimate for x times p} in the second-to-last inequality.
    Therefore, Lemma \ref{lemma almost eigenvalue equation} implies that
    \begin{align}\label{eq estimate for I1}
        I_1=C\int_{\R^n} (q_i-\ol{q}_i)^2\,d\nu_{\mathbf{x}_i;-\tau}\leq C\lambda^2r^4 \widehat{H}_{\mathbf{x}_i}^u(\tau_i).
    \end{align}
    Second, a coordinate change in \eqref{eq:whatweshouldhavedone} gives
    \begin{align}\label{eq estimate for I2}
        I_2\leq C(\Lambda,\Upsilon) \widehat{\mathcal{E}}_{r}(\mathbf{x}_i)\widehat{H}_{\mathbf{x}_i}^u(\tau_i).
    \end{align}
    Third, the computation in \eqref{eq estimate for I3} gives
    \begin{align*}
        I_3 \leq C(\Lambda)\exp\left( -\frac{1}{e^{60}\lambda^{\Upsilon}\tau} \right)\widehat{H}_{\mathbf{x}_i}^u(\tau_i).
    \end{align*}
    Finally, by Lemma \ref{lemma-change-of-base-point}, we change bases $e^{-\frac{1}{8} f_i}d\nu_{\mathbf{x}_i;t}\leq Cd\nu_{t}$ in \eqref{eq:whatweshouldhavedone} with $i\leftarrow 0$ to get
    \begin{align}\label{eq estimate for I4}
        I_4\leq C(\Lambda,\Upsilon) \widehat{\mathcal{E}}_{r}(\mathbf{0})\widehat{H}^u(\tau).
    \end{align}
    Combining Lemma \ref{lemma polynomials are good}, \eqref{eq estimate for I1}, \eqref{eq estimate for I2}, \eqref{eq estimate for I3}, \eqref{eq estimate for I4}, \eqref{eq I defs}, and Lemma \ref{lemma comparability of H}\ref{comparability of widehat H}, we get
    \begin{align*}
        \int_{\mathbb{R}^n} (q_0-\overline{q}_i)^2 d\nu_{\mathbf{x}_i;t}\leq C\int_{\mathbb{R}^n} (q_0-\overline{q}_i)^2 e^{-\frac{1}{8} f_i} d\nu_{\mathbf{x}_i;t}\leq C(\Lambda,\Upsilon)\widehat{\mathcal{E}}_{r}(\{\mathbf{0},\mathbf{x}_1\})\widehat{H}_{\mathbf{x}_i}^u(\tau).
    \end{align*}
    Lemma \ref{lemma closeness of frequency of u and caloric approximation} then implies that
    \begin{align}
        |N^{p_0}_{\mathbf{x}_i}(\tau_i)-m_i|=|N^{q_0}_{\mathbf{x}_i}(\tau_i)-m_i|\leq C(\Lambda,\Upsilon) \widehat{\mathcal{E}}_{r}(\{\mathbf{0},\mathbf{x}_1\}).\label{eq frequency close to m for p0}
    \end{align}
    for any $\tau_i\in [\frac{1}{8}r^2,8r^2]$. Since $p_0$ is caloric, Corollary \ref{cor: nearby same m} implies that $m_i=m$. Then Proposition \ref{proposition-spatial-splitting-two-points} implies that  
    \begin{align}
        \int_{\R^n} |x_1\cdot\nabla p_0|^2+|t_1\cdot\partial_t p_0|^2\,d\nu_{-\tau}
        \le C(\Lambda,\Upsilon) \widehat{\mathcal{E}}_{r}(\{\mathbf{0},\mathbf{x}_1\}) \int_{\R^n} p_0^2\,d\nu_{-\tau}\label{eq splitting for p0}
    \end{align}
    for any $\tau\in [\frac{1}{8}r^2,8r^2]$. 
    Using \eqref{eq derivative of log H} together with \eqref{eq frequency close to m for p0} and \eqref{eq splitting for p0}, we get
    \begin{align*}
        \int_{\R^n}|t_1\pdt q_0|^2\,d\nu_{-\tau} & \leq \frac{C\widehat{H}^u(\tau)}{\tau^{m}} \int_{\R^n}\verts{t_1\pdt p_0}^2\,d\nu_{-\tau}+\verts{t_1}^2C\left(\frac{d}{d\tau}\log \widehat{H}^u-\frac{m}{\tau}\right)^2q_0^2\,d\nu_{-\tau}\\
        & \leq C(\Lambda,\Upsilon)\widehat{\mathcal{E}}_{r}(\{\mathbf{0},\mathbf{x}_1\}) \int_{\R^n} q_0^2\,d\nu_{-\tau}.
    \end{align*}
    Using this and \eqref{eq splitting for p0} again, we obtain that for any $\tau\in [\frac{1}{8}r^2,8r^2]$ 
    \begin{align*}
        \int_{\R^n} |x_1\cdot\nabla q_0|^2+|t_1\cdot\partial_t q_0|^2\,d\nu_{-\tau}
        \le C(\Lambda,\Upsilon) \widehat{\mathcal{E}}_{r}(\{\mathbf{0},\mathbf{x}_1\})H^{q_0}(\tau)\leq C(\Lambda,\Upsilon)\widehat{\mathcal{E}}_{r}(\{\mathbf{0},\mathbf{x}_1\}) \widehat{H}^u(\tau),
    \end{align*}
    where we used \eqref{eq:whatweshouldhavedone} in the second inequality. On the other hand, Lemma \ref{lem:higherordercloseness} and Lemma \ref{prop:closeheatpolynomial} give
    \begin{align} 
        &\int_{e^{-2}r^2}^{e^2 r^2} \frac{1}{\widehat{H}^u(\tau)} \int_{\mathbb{R}^n}  \left| \nabla \left( \chi u - q_0 \right) \right|^2 +\tau\left| \pd_t \left( \chi u - q_0 \right) \right|^2 \,d\nu_{-\tau}d\tau\leq C(\Lambda,\Upsilon)\widehat{\mathcal{E}}_{r}(\{\mathbf{0},\mathbf{x}_1\}).
    \end{align}
    In particular, we get our claim:
    \begin{align*} 
        &\int_{e^{-2}r^2}^{e^2 r^2} \frac{1}{\widehat{H}^u(\tau)} \int_{\mathbb{R}^n}  \left| x_1\cdot \cd(\chi u) \right|^2 +\left| t_1\pd_t \chi u \right|^2 \,d\nu_{-\tau}\frac{d\tau}{\tau}\\
        & \leq C\int_{e^{-2}r^2}^{e^2 r^2} \frac{1}{\widehat{H}^u(\tau)} \int_{\mathbb{R}^n}  \left|\cd(\chi u-q_0) \right|^2 +\tau\left| \pd_t (\chi u-q_0) \right|^2 \,d\nu_{-\tau}d\tau\\
        &\qquad + C\int_{e^{-2}r^2}^{e^2 r^2} \frac{1}{\widehat{H}^u(\tau)} \int_{\mathbb{R}^n}  \left| x_1\cdot \cd q _0 \right|^2 +\left| t_1\pd_t q_0 \right|^2 \,d\nu_{-\tau}\frac{d\tau}{\tau}\\
        &\leq C(\Lambda,\Upsilon)\widehat{\mathcal{E}}_{r}(\{\mathbf{0},\mathbf{x}_1\}).
    \end{align*}
\end{proof}

\begin{proof}[Proof of Theorem \ref{thm:part2-sym split equiv}\ref{thm:part2-sym split equiv-1}]
    The proof follows from Lemma \ref{lemma-linear-independence} and Theorem \ref{theorem cone splitting inequality in general} as in the proof of Theorem \ref{theorem-cone-splitting-inequality}.
\end{proof}

\subsubsection{Existence of a plane of symmetry}
As an analog of Lemma \ref{newlineup}, we prove that there exists a plane of symmetry that is uniform across scales and locations.
\begin{lemma} \label{newlineup2}
    Suppose $\Lambda \in [1,\infty)$. The following holds if $\lambda\leq \ol{\lambda}(\Lambda,\Upsilon)$ and $\delta \leq C(\Lambda)^{-1}$.
    Fix any $\eta,\kappa,r_0 \in (0,1]$ and any solution $u$ of \eqref{part-2-eq-parabolic-equation} satisfying \eqref{eq:strongerdoubling}. Suppose $\mathcal{C} \subseteq P(\mathbf{0},r_0)$ and $\mathbf{r}_{\bullet}:\mathcal{C} \to [0,r_0]$ is a function satisfying the following for some $m\in \mathbb{N}$:
    \begin{enumerate}[label=(\alph*)]
        \item \label{hypothesisplaneofsymmetrya2} there exists $\mathbf{x}_0 \in \mathcal{C}$ and $s_0 \in [\mathbf{r}_{\mathbf{x}_0},r_0]$ such that $u$ is weakly $(k,\delta,s_0)$-symmetric with respect to $V$; \label{plane of symmetry a2}
        
        \item \label{hypothesisplaneofsymmetryb2}
        $\sup_{\mathbf{x}\in \mathcal{C}} \sup_{s\in [10^{-6}\kappa \mathbf{r}_{\mathbf{x}},10^6 \kappa^{-1}r_0]} 
        (|\widehat{N}_{\mathbf{x}}(s^2)-m|+\lambda^{\frac{\Upsilon}{4}}s^{\frac{\Upsilon}{2}})<\delta$.
    \end{enumerate}
    Then the following statements hold for any $\mathbf{x} \in \mathcal{C}$:
    \begin{enumerate}[label={(\arabic*)}]
        \item \label{existence of plane of symmetry2} for any $s\in [\kappa \mathbf{r}_{\mathbf{x}},\kappa^{-1} r_0]$, $u$ is weakly $(k,C(\Lambda,\Upsilon)\delta,s)$-symmetric with respect to $V$ at $\mathbf{x}$;
    
        \item \label{plane of symmetry inside pinched points2} for any $s\in [\mathbf{r}_{\mathbf{x}},r_0]$,
        \begin{align*}
            (\mathbf{x}+V)\cap P(\mathbf{x},10s) \subset \{\mathbf{y}\in P(\mathbf{x}, 10s)\colon |\widehat{N}_{\mathbf{y}}(\kappa^2 s^2)-m|<C(\kappa,\Lambda,\Upsilon)\sqrt{\delta}\};
        \end{align*}
        
        \item \label{containment of plane of symmetry2} if in addition we assume that $u$ is not weakly $(k+1,\eta,s_0)$-symmetric at $\mathbf{x}_0$, then, for any $s\in [\mathbf{r}_{\mathbf{x}},10^{-2}r_0]$ and $\zeta \geq C(\Lambda,\Upsilon)\sqrt{\delta}$,
        \begin{align*}
            \{\mathbf{y} \in P(\mathbf{x},10s): \widehat{\mathcal{E}}_{100s}(\mathbf{y})< \zeta\}\subseteq  P(\mathbf{x}+ V , C(\Lambda,\Upsilon) \eta^{-\frac{1}{n}}\zeta^{\frac{1}{n}} s).
        \end{align*}
    \end{enumerate}
\end{lemma}

\begin{proof} 
    By parabolic rescaling, we will assume that $r_0=1$.
    
    \ref{existence of plane of symmetry2} By Lemma \ref{lemma: propagation of symmetry and pinching-2}\ref{upwardpropagation-2} with $r \leftarrow 100$, $r_1 \leftarrow \mathbf{r}_{\mathbf{x}_0}$, $r_2 \leftarrow s_0$, $u$ is weakly $(k,C(\Lambda,\Upsilon)\delta,1)$-symmetric at $\mathbf{x}_0$ with respect to $V$. This together with Lemma\ref{lemma comparability of H}\ref{comparability of spatial derivative}\ref{comparability of time derivative} imply that for any $\mathbf{x}\in \cC$
    \begin{align}
        &\int_{10^3}^{e^2\cdot 5\cdot 10^3} \frac{1}{\widehat{H}^u_{\mathbf{x}}(\tau)} \int_{\mathbb{R}^n}  \left| \pi_L\nabla \left( \chi u_{\mathbf{x};1}\right) \right|^2 +\tau\left| \pd_t \left( \chi u_{\mathbf{x};1}\right) \right|^2 e^{-\frac{1}{4}f}\,d\nu_{-\tau}d\tau\leq C(\Lambda,\Upsilon)\delta.\label{eq estimate with eaf}
    \end{align}
    We claim that the above estimate holds without the weight $e^{-\frac{1}{4}f}$.
    
    \begin{claim}\label{claim quantitative uniqueness saves us}
        $u$ is weakly $(k,C(\Lambda,\Upsilon)\delta,10^{\frac{3}{2}})$-symmetric with respect to $V$ at any $\mathbf{x} \in \cC$.
    \end{claim}
    \begin{proof}[Proof of Claim \ref{claim quantitative uniqueness saves us}]
        We consider the case where $V=L \times \mathbb{R}$ for some $k$-plane $L\subseteq \mathbb{R}^n$, as the case $V=L \times \{0\}$ is similar. On the other hand, Lemma \ref{lem:higherordercloseness} and Lemma \ref{prop:closeheatpolynomial} imply the existence of $p_{\mathbf{x}}\in \cP_m$ for every $\mathbf{x}\in \cC$ such that if $q_{\mathbf{x}}\coloneqq \sqrt{\frac{\widehat{H}_{\mathbf{x}}(\tau)}{\tau^m}}p_{\mathbf{x}}$, then
        \begin{align} 
            &\int_{10^3}^{e^2\cdot 5\cdot 10^3} \frac{1}{\widehat{H}^u_{\mathbf{x}}(\tau)} \int_{\mathbb{R}^n}  \left| \nabla \left( \chi u_{\mathbf{x};1} - q_{\mathbf{x}} \right) \right|^2 +\tau\left| \pd_t \left( \chi u_{\mathbf{x};1} - q_{\mathbf{x}} \right) \right|^2 \,d\nu_{-\tau}d\tau\leq C(\Lambda,\Upsilon)\delta.\label{eq starting estimate for line up lemma}
        \end{align}
        This together with \eqref{eq estimate with eaf} implies
        \begin{align*}
            &\int_{10^3}^{e^2\cdot 5\cdot 10^3} \frac{1}{\widehat{H}^u_{\mathbf{x}}(\tau)} \int_{\mathbb{R}^n}  \left| \pi_L\nabla q_{\mathbf{x}} \right|^2 +\tau\left| \pd_t q_{\mathbf{x}}\right|^2 e^{-\frac{1}{4}f}\,d\nu_{-\tau}d\tau\leq C(\Lambda,\Upsilon)\delta.
        \end{align*}
        Since $q_{\mathbf{x}}$ is a polynomial of degree at most $m$, we use Lemma \ref{lemma polynomials are good} to get
        \begin{align}
            &\int_{10^3}^{e^2\cdot 5\cdot 10^3} \frac{1}{\widehat{H}^u_{\mathbf{x}}(\tau)} \int_{\mathbb{R}^n}  \left| \pi_L\nabla q_{\mathbf{x}} \right|^2 +\tau\left| \pd_t q_{\mathbf{x}}\right|^2 \,d\nu_{-\tau}d\tau\leq C(\Lambda,\Upsilon)\delta.\label{eq improved estimates for qx}
        \end{align}
        Thus \eqref{eq starting estimate for line up lemma} and \eqref{eq improved estimates for qx} imply the claim.
    \end{proof}
    The conclusion then follows by applying Claim \ref{claim quantitative uniqueness saves us} and Lemma \ref{lemma: propagation of symmetry and pinching-2}\ref{upwardpropagation-2} to each point $\mathbf{x}\in\mathcal{C}$. 

    \ref{plane of symmetry inside pinched points2} 
    Apply \ref{existence of plane of symmetry2} and  Theorem \ref{thm:part2-sym split equiv}\ref{thm:part2-sym split equiv-2}.
    
    \ref{containment of plane of symmetry2} Suppose $\mathbf{x} \in \mathcal{C}$, $s\in [\mathbf{r}_{\mathbf{x}},10^{-2}]$, $\mathbf{y} \in P(\mathbf{x},10s)$, and $\widehat{\cE}_{100s}(\mathbf{y})<\zeta$ for some $\zeta \geq C(\Lambda,\Upsilon)\sqrt{\delta}$. Set $\epsilon\coloneqq  C(\Lambda,\Upsilon)\eta^{-\frac{1}{n}}\zeta^{\frac{1}{n}}$, and assume by way of contradiction that $\mathbf{y} \notin P(\mathbf{x}+V,\epsilon s)$. By applying \ref{plane of symmetry inside pinched points} twice with $s \leftarrow s$ and $s\leftarrow 100s$, there exists a $(k,100 s)$-independent set $\{\mathbf{v}_i\}_{i \in I} \subseteq (\mathbf{x}+V)\cap P(\mathbf{x},10s)$ containing $\mathbf{x}$ and satisfying $\widehat{\mathcal{E}}_{100s}(\mathbf{v}_i)< C(\Lambda,\Upsilon)\sqrt{\delta} \leq \zeta$ for all $i\in I$. Thus $\widehat{\mathcal{E}}_{100s}^{k+1,\epsilon}(\mathbf{x})<\zeta$, so by Theorem \ref{thm:part2-sym split equiv}\ref{thm:part2-sym split equiv-1} and Lemma \ref{lemma: propagation of symmetry and pinching-2}\ref{upwardpropagation-2}, $u$ is weakly $(k+1,C(\Lambda,\Upsilon)^{-1}\eta,s)$-symmetric at $\mathbf{x}$. By Lemma \ref{lemma: propagation of symmetry and pinching-2}\ref{upwardpropagation-2}, it follows that $u$ is weakly $(k+1,C(\Lambda,\Upsilon)^{-1}\eta,100)$-symmetric at $\mathbf{x}$, hence by Claim \ref{claim quantitative uniqueness saves us}, $u$ is also weakly $(k+1,C(\Lambda,\Upsilon)^{-1}\eta,1)$-symmetric at $\mathbf{x}_0$. Another application of Lemma \ref{lemma: propagation of symmetry and pinching-2}\ref{upwardpropagation-2} finally yields that $u$ is weakly $(k+1,\eta,s_0)$-symmetric at $\mathbf{x}_0$, a contradiction. 
\end{proof}

\subsubsection{$L^2$-subspace approximation}
Let $\mu$ be a probability measure supported on $\overline{P}(\mathbf{0},1)$. Recall the definitions of $x_{\operatorname{cm}}$, $Q$ from \eqref{def:cmandvariance}.

We now state the analog of Lemma \ref{lemma eigenvalue estimate}. 
\begin{lemma}\label{lemma eigenvalue estimate part2}
    Let $v_j$ be the eigenvector of the covariance matrix $Q$ with eigenvalue $\lambda_j$. Let $u$ be a solution to \eqref{part-2-eq-parabolic-equation} satisfying \eqref{eq:strongerdoubling}. Then, for any $r\in (0,1]$, we have
    \begin{align*}
        \int_{100n^2 r^2}^{100e^2 n^2 r^2}\frac{\lambda_j}{\widehat{H}_{\mathbf{x}_0}^u(100n^2r^2)} \int_{\R^n} ( v_j \cdot \nabla (u\chi_{\mathbf{0};1}))^2\,d\nu_{-\tau}d\tau
        \le C(\Lambda,\Upsilon) \int_{P(\mathbf{0},1)}\widehat{\cE}_{20nr}(\mathbf{x})\,d\mu(\mathbf{x}).
    \end{align*}
\end{lemma}

\begin{proof}
    Taking $u \leftarrow \chi u_{\mathbf{0};1}$ in \eqref{eq:L2bestplane1} and integrating in time against $(\widehat{H}_{\mathbf{0}}^u)^{-1}$ yields
    \begin{align} \label{eq:L2bestplane3}
        \begin{split}
            \lambda_j & \int_{100n^2 r^2}^{100 e^2n^2 r^2} \frac{1}{\widehat{H}^u(\tau)}\int_{\R^n}  ( v_j \cdot \nabla (\chi u_{\mathbf{0};1})(z,-\tau))^2\,d\nu_{-\tau}(z) d\tau \\
            &\leq  \int_{100n^2 r^2}^{100 e^2 n^2 r^2} \frac{1}{\widehat{H}^u(\tau)}\int_{\R^n}  \int_{P(\mathbf{0},1)}\int_{P(\mathbf{0},1)} \left( \nabla (\chi u_{\mathbf{0};1})(z,-\tau)\cdot (x-y) \right)^2  \,d\mu(\mathbf{x})d\mu(\mathbf{y})  d\nu_{-\tau}(z) d\tau.
        \end{split}
    \end{align}
    Moreover, for any fixed $\mathbf{x},\mathbf{y}\in P(\mathbf{0},1)$, by Lemma \ref{lemma comparability of H}\ref{comparability of spatial derivative} and Proposition \ref{theorem cone splitting inequality in general}, we obtain
    \begin{align*}
       &\int_{100n^2r^2}^{100e^2n^2 r^2}\frac{1}{\widehat{H}_{\mathbf{0}}^u(\tau)}\int_{\R^n} (\nabla (\chi u_{\mathbf{0};1})(z,-\tau)\cdot ( x-y ) )^2 \,d\nu_{-\tau}(z)d\tau \\
        &\le C(\Lambda,\Upsilon)\int_{400n^2r^2}^{400 e^2 n^2 r^2}\frac{1}{\widehat{H}_{\mathbf{x}}^u(\tau)}\int_{\R^n} (\nabla (\chi u_{\mathbf{x};1})(z,-\tau)\cdot ( x-y ) )^2 \,d\nu_{-\tau}(z)d\tau + C(\Lambda,\Upsilon)\lambda^{\frac{\Upsilon}{2}}r^{2\Upsilon}
        \\ &\le C(\Lambda,\Upsilon)\widehat{\cE}_{20nr}(\{\mathbf{x},\mathbf{y}\}).
    \end{align*}
    The lemma then follows by integrating over $\mathbf{x},\mathbf{y} \in P(\mathbf{0},1)$ and combining with \eqref{eq:L2bestplane3}.
\end{proof}

The following is a modification of Lemma \ref{lem:beta}.
\begin{lemma} \label{lem:betapart2}
    Let $k\in\{n,n+1\}$. For any finite measure $\mu$ supported in $P(\mathbf{0},1)$ on $\mathbb{R}^{n+1}$ and any solution $u$ of \eqref{part-2-eq-parabolic-equation} satisfying \eqref{eq:strongerdoubling}, if $u$ is weakly $(k,\epsilon, 3r)$-symmetric but not weakly $(k+1,\eta,r)$-symmetric at $\mathbf{0}$, then, for any $P(\mathbf{x},r)\subseteq P(\mathbf{0},2)$, if $\epsilon \leq C^{-1}(\Lambda,\Upsilon) \eta$, we have
    \begin{align*}
        \beta_{\mathcal{P},k}^{2}(\mathbf{x},r;\mu)\leq \frac{C^\Lambda}{r^k \eta}\left(\int_{P(\mathbf{x},r)}\widehat{\cE}_{20nr}(\mathbf{y})d\mu(\mathbf{y})\right).
    \end{align*}
\end{lemma}

\begin{proof} 
    Given Lemma \ref{lemma eigenvalue estimate part2} and the proof of Lemma \ref{lem:beta}, it remains to consider the case where $k=n$ and the symmetry is spatial. Let $w$ be the caloric function with $w(\cdot,20r^2)=(\chi u_{\mathbf{0};1})(\cdot,20r^2)$, so that  
    \begin{align} \label{eq:closenessforbeta}
        \begin{split} 
            \sup_{\tau \in [r^2,10r^2]}& \int_{\mathbb{R}^n} \tau |\nabla (w-\chi u_{\mathbf{x}_0;1})|^2 d\nu_{-\tau} + \int_{r^2}^{10r^2} \tau \int_{\mathbb{R}^n} |\partial_t (w-\chi u_{\mathbf{x}_0;1})|^2\,d\nu_{-\tau} d\tau \\ 
            &\qquad \leq C(\Lambda,\Upsilon) \lambda^{\Upsilon} \tau_0^{2\Upsilon} \widehat{H}_{\mathbf{x}_0}(\tau_0). 
        \end{split}
    \end{align}
    It follows that
    \begin{align*}
        r^2\int_{\mathbb{R}^n} |\nabla w|^2 d\nu_{-9r^2} \leq (C\epsilon + C(\Lambda,\Upsilon) \lambda^{\Upsilon} r^{4\Upsilon})\int_{\mathbb{R}^n} w^2 d\nu_{-9r^2},
    \end{align*}
    so that integrating
    $\frac{d}{dt}\int_{\mathbb{R}^n}|\nabla w|^2 d\nu_t \leq -2\int_{\mathbb{R}^n}|\nabla^2 w|^2 d\nu_t$ from $t=-9r^2$ to $t=-e^2 r^2$ yields
    \begin{align*}
        r^2\int_{-e^2r^2}^{-r^2} \int_{\mathbb{R}^n} |\partial_t w|^2 d\nu_{t}dt \leq Cr^2 \int_{-e^2r^2}^{-r^2} \int_{\mathbb{R}^n} |\nabla^2 w|^2 d\nu_{t}dt \leq  (C\epsilon + C(\Lambda,\Upsilon) \lambda^{\Upsilon} 4^{4\Upsilon})\int_{\mathbb{R}^n} w^2 d\nu_{-9r^2}.
    \end{align*}
    Along with \eqref{eq:closenessforbeta}, this implies that $u$ is weakly $(k+1,C\epsilon + C(\Lambda,\Upsilon) \lambda^{\Upsilon} 4^{4\Upsilon},1)$-symmetric, a contradiction if $\epsilon \leq C(\Lambda,\Upsilon)^{-1}\eta$.
\end{proof}

\subsection{\texorpdfstring{$\epsilon$}{epsilon}-regularity}
We now prove $\epsilon$-regularity results in the general setting, analogous to those in Section \ref{epsilon regularity}. 

We define effective versions of nodal and singular sets in the general setting:
\begin{definition}
    Suppose $u$ is a solution to \eqref{part-2-eq-parabolic-equation} satisfying \eqref{eq:strongerdoubling}. For $r>0$, define
    \begin{subequations}
        \begin{align*}
            \widehat{Z}_r(u)&\coloneqq \left\{ \mathbf{x} \in P(\mathbf{0},1)\colon \inf_{P(\mathbf{x},10^{-6}s)} |u|^2 \leq \frac{1}{8} \int_{\mathbb{R}^n} |u_{\mathbf{x};1}|^2 \, d\nu_{-s^2} 
            \text{ for all }s\in[r,1]\right\},\\
            \widehat{S}_r(u)&\coloneqq \left\{ \mathbf{x}\in P(\mathbf{0},1)\colon  \inf_{P(\mathbf{x},10^{-6}s)} \left(\verts{u}^2 +2s^2\verts{\cd u}^2 \right) \leq \frac{1}{16} \int_{\R^n} \verts{u_{\mathbf{x};1}}^2 \, d\nu_{-s^2} \text{ for all }s\in[r,1]\right\}.
        \end{align*}
    \end{subequations}
    We define
    \begin{align*}
        \widehat{\mathscr{C}}^k(u)\coloneqq \begin{cases}
            Z(u) & \text{ when }k=n+1,\\
            S(u) & \text{ when }k=n,
        \end{cases}
    \end{align*}
   and $\widehat{\mathscr{C}}^k_r$ analogously.
\end{definition}

\begin{definition}[Quantitative stratum]\label{definition-quantative-strata-2}
    Suppose $u$ is a solution to \eqref{part-2-eq-parabolic-equation} satisfying \eqref{eq:strongerdoubling}. The \textit{$(k,\epsilon,r_1,r_2)$ quantitative stratum} is
    \begin{align*}
        \widehat{\cS}_{\epsilon, r_1,r_2}^{k}\coloneqq \widehat{\cS}_{\epsilon, r_1,r_2}^{k}(u)\coloneqq
        \left\{\mathbf{x}\in P(\mathbf{0},1)\colon \widehat{\mathcal{E}}^{k+1,1}_{r}(\mathbf{x})\ge \epsilon,\ \forall\, r\in [r_1,r_2]\right\}.
    \end{align*}
    We also set $\widehat{\cS}_{\epsilon,r}^{k}\coloneqq \widehat{\cS}_{\epsilon,r,1}^{k}$, and define $\widehat{\cS}_{\epsilon}^{k}\coloneqq \bigcap_{r>0}\widehat{\cS}_{\epsilon,r}^{k}$.
\end{definition}

The following is a version of Proposition \ref{proposition-containment}, where we prove the containment of $\widehat{\mathscr{C}}_r^k$ in the $k$-th quantitative strata.
\begin{proposition}\label{proposition-containment-2}
    The following holds if $\lambda\leq \ol{\lambda}(\Lambda, \Upsilon)$. Suppose $u$ is a solution to \eqref{part-2-eq-parabolic-equation} satisfying \eqref{eq:strongerdoubling}. Let $k\in \{n,n+1\}$. If $\epsilon\le C(\Lambda)^{-1}$, then, for any $r\in(0,1)$,
    \begin{align}
        \widehat{\mathscr{C}}^k_r(u)\subseteq \widehat{\cS}_{\epsilon,r}^{k}(u).
        \label{eq-inclusion-of-nodal-and-singular-set-in-quantitative-stratum-2}
    \end{align}
\end{proposition}

As in the caloric setting (Lemma \ref{lemma-containment-of-nodal-sets}), we first prove that only constant functions can attain maximal symmetry.
\begin{lemma}\label{lemma-containment-of-nodal-sets-2}
    The following holds if $\lambda\leq \ol{\lambda}(\Lambda, \Upsilon)$. Let $u$ be a solution  of \eqref{part-2-eq-parabolic-equation} satisfying \eqref{eq:strongerdoubling}. Suppose $u$ is spatially weakly $(n, \epsilon, r)$-symmetric at $\mathbf{x}_0$ with $\epsilon\le\bar\epsilon(\Lambda)$. Then 
    \begin{align*}
        \inf_{P(\mathbf{x}_0,10^{-6}r)} |u|^2> \frac{1}{8} \int_{\R^n} (\chi u_{\mathbf{x}_0;1})^2\,d\nu_{-r^2}.
    \end{align*}
\end{lemma}

\begin{proof}
    By replacing $u$ with $u_{\mathbf{x}_0;1}$, we can assume that $\mathbf{x}_0=\mathbf{0}$ and $u=u_{\mathbf{x}_0;1}$. Then
    \begin{align*}
        \log(2)\widehat{N}^u(r^2)\leq \log\frac{\widehat{H}(4r^2)}{\widehat{H}(r^2)}+C\lambda^{\Upsilon}r^{2\Upsilon} \leq \int_{r^2}^{e^2r^2} \frac{\widehat{N}^u(\tau)}{\tau}\,d\tau +C\lambda^{\Upsilon}r^{2\Upsilon}{\tau}d\tau <\epsilon,
    \end{align*}
    where in the last inequality we use that $u$ is spatially weakly $(n,\epsilon,r)$-symmetric at $\mathbf{0}$. 
    Then, for any $\mathbf{x}\in P(\mathbf{0},10^{-3}r)$, Lemma \ref{lemma comparability of H}\ref{comparability of widehat H}\ref{comparability of spatial derivative} and Lemma \ref{lem:weakestimates} give
    \begin{align*}
        \int_{\R^n} \verts{\cd (\chi u_{\mathbf{x};1})}^2e^{-\frac{1}{4}f} \,d\nu_{-\frac{1}{2}r^2}\leq C(\Lambda,\Upsilon) \epsilon \widehat{H}^u_{\mathbf{x}}(\frac{1}{2}r^2).
    \end{align*}
    Thus for any $\mathbf{x}\in P(\mathbf{0},10^{-3}r)$ and any $\tau\leq \frac{1}{2}r^2$, Corollary \ref{cor: strong monotonocity} implies
    \begin{align*}
        \widehat{N}_{\mathbf{x}}\left(\tau\right)\leq C(\Lambda,\Upsilon) \sqrt{\epsilon}.
    \end{align*}
    Now arguing as in Lemma \ref{lemma-containment-of-nodal-sets}, for any $\mathbf{x}\in P(\mathbf{0},10^{-3}r)$, we have
    \begin{align*}
        C^{-1}\frac{1}{\tau^{\frac{n}{2}}}\int_{B(x,\sqrt{\tau})}|u(y,t-\tau)-1|^2 \,dy \leq \int_{\R^n} \verts{\chi u_{\mathbf{x};1}-1}^2\,d\nu_{-\tau}\leq C(\Lambda,\Upsilon)\sqrt{\epsilon}\widehat{H}(r^2),
    \end{align*}
    where in the last inequality we use Lemma \ref{lemma comparability of H}\ref{comparability of widehat H} and Lemma \ref{lem:weakestimates}. Taking $\tau \coloneqq 10^{-7}r^2$, $\mathbf{x} \leftarrow (0,t)$ for $t\in (-10^{-6}r^2,10^{-6}r^2)$, we then integrate over $t$ to obtain
    \begin{align*}
        \frac{1}{r^{n+2}}\int_{P(\mathbf{0},10^{-4}r)} |u-1|^2 d\mathcal{H}_{\mathcal{P}}^{n+2} \leq C(\Lambda,\Upsilon)\sqrt{\epsilon} \widehat{H}(r^2).
    \end{align*}
    The claim follows from standard parabolic interior estimates.
\end{proof}

As in Lemma \ref{lemma-containment-of-singular-sets}, we now prove that only linear functions can possess $n+1$ symmetries.
\begin{lemma}\label{lemma-containment-of-singular-sets-2}
    The following holds if $\lambda\leq \ol{\lambda}(\Lambda, \Upsilon)$. Let $u$ be a solution to \eqref{part-2-eq-parabolic-equation} satisfying \eqref{eq:strongerdoubling}. Suppose $u$ is weakly $(n+1,\epsilon,r)$-symmetric at $\mathbf{x}_0$ with $\epsilon\leq \ol{\epsilon}$. Then
    \begin{align*}
        \inf_{P(\mathbf{x}_0,10^{-6}r)} \verts{u}^2+ 2r^2\verts{\cd u}^2 > \frac{1}{16} \int_{\R^n}(\chi u_{\mathbf{x}_0})^2 \, d\nu_{-r^2}.
    \end{align*}
\end{lemma}

\begin{proof}
    By replacing $u$ with $u_{\mathbf{x}_0;1}$, we can assume that $\mathbf{x}_0=\mathbf{0}$ and $u=u_{\mathbf{x}_0;1}$. By rotating coordinates, we may assume that $u$ is weakly symmetric with respect to $(\R^{n-1}\times \{0\})\times \R$ so that
    \begin{align*}
        \sum_{i=1}^{n-1}\int_{r^2}^{e^2r^2}\frac{1}{\widehat{H}(\tau)}\int_{\R^n} \verts{\partial_i (\chi u)}^2\,d\nu_{-\tau}d\tau+\int_{r^2}^{e^2r^2}\frac{\tau}{\widehat{H}(\tau)}\int_{\R^n} \verts{\partial_t (\chi u)}^2\,d\nu_{-\tau}d\tau\leq \epsilon.
    \end{align*}
    By propagation of symmetry in Lemma \ref{lemma: propagation of symmetry and pinching-2} and  change of base point in Lemma \ref{lemma comparability of H}, we know that for any $\mathbf{x}\in P(\mathbf{0},\frac{1}{100}r)$
    \begin{align*}
        \sum_{i=1}^{n-1}\int_{r^2}^{e^2r^2}\frac{1}{\widehat{H}_{\mathbf{x}}(\tau)}\int_{\R^n} \verts{\partial_i (\chi u_{\mathbf{x};1})}^2\,d\nu_{-\tau}d\tau+\int_{r^2}^{e^2r^2}\frac{\tau}{\widehat{H}_{\mathbf{x}}(\tau)}\int_{\R^n} \verts{\partial_t (\chi u_{\mathbf{x};1})}^2\,d\nu_{-\tau}d\tau\leq C\sqrt{\epsilon}.
    \end{align*}
    If $w_{\mathbf{x}}$ is the caloric approximation of $\chi u_{\mathbf{x};1}$, then we know that
    \begin{align*}
        \sum_{i=1}^{n-1}\int_{r^2}^{e^2r^2}\frac{1}{\widehat{H}_{\mathbf{x}}(\tau)}\int_{\R^n} \verts{\partial_i (w_{\mathbf{x}})}^2\,d\nu_{-\tau}d\tau+\int_{r^2}^{e^2r^2}\frac{\tau}{\widehat{H}_{\mathbf{x}}(\tau)}\int_{\R^n} \verts{\partial_t (w_{\mathbf{x}})}^2\,d\nu_{-\tau}d\tau\leq C\sqrt{\epsilon}.
    \end{align*}
    This implies $w_{\mathbf{x}}$ is $(n+1,C\sqrt{\epsilon},r)$-symmetric. Since $w_{\mathbf{x}}$ is caloric, Lemma \ref{lemma-containment-of-singular-sets} gives pointwise lower bounds for $w_{\mathbf{x}}$ and using standard parabolic regularity and the argument of Lemma \ref{lemma-containment-of-nodal-sets-2}, we get the pointwise estimate
    \begin{align*}
        \inf_{\mathbf{x}\in P(\mathbf{x}_0,10^{-6}r)} \left( \verts{u}^2+2r^2 \verts{\cd u}^2 \right)\geq \frac{1}{16} \int_{\R^n}(\chi u)^2\,d\nu_{-r^2},
    \end{align*}
    from which the claim follows. 
\end{proof}

\begin{proof}[Proof of Proposition \ref{proposition-containment-2}]
    The proof is similar to the proof of Theorem \ref{proposition-containment} after we replace Theorem \ref{thm: sym split equiv}\ref{thm: sym split equiv-1}, Lemma \ref{lemma-containment-of-nodal-sets}, Lemma \ref{lemma-containment-of-singular-sets} with Theorem \ref{thm:part2-sym split equiv}\ref{thm:part2-sym split equiv-1}, Lemma \ref{lemma-containment-of-nodal-sets-2}, Lemma \ref{lemma-containment-of-singular-sets-2}.
\end{proof}

This lemma establishes that nodal and singular sets are contained in the super-level sets of the frequency function, see Lemma \ref{lemma pinching implies nodal}.
\begin{lemma}\label{lemma pinching implies nodal-2}
    The following holds if $\lambda\leq \ol{\lambda}(\Lambda,\Upsilon)$. Let $k\in \{n,n+1\}$. If $\epsilon\leq C(\Lambda)^{-1}$, then, for any $\mathbf{x}_0 \in P(\mathbf{0},10)$ and $r\in (0,\epsilon]$, we have
    \begin{align*}
        \widehat{\mathscr{C}}_r^k(u)\subset \left\{\mathbf{x}\in P(\mathbf{0},1)\colon \widehat{N}(\epsilon^{-2}r^2)\geq m_\ast-\frac{2}{3}\right\},
    \end{align*}
    where $m_\ast$ is defined as in \eqref{eq def of mstar}.
\end{lemma}

\begin{proof}
    We focus on $k=n$, since the case $k=n+1$ is analogous. Suppose $\widehat{N}_{\mathbf{x}_0}^u(\epsilon^{-2}r^2)<\frac{4}{3}$. Let $w$ be the caloric function with $w(\cdot,-2\epsilon^{-2}r^2)=(\chi u_{\mathbf{x}_0;1})(\cdot,-2\epsilon^{-2}r^2)$ of $\chi u_{\mathbf{x}_0;1}$ as in Lemma \ref{lem:caloricapprox}. If $\lambda \leq \overline{\lambda}(\Lambda,\Upsilon)$, we have
    \begin{align*}
        &2\tau\int_{\R^n}\verts{\cd w}^2\,d\nu_{-\epsilon^{-2}r^2}\\
        &\leq 2\lambda^{-\frac{\Upsilon}{2}}r^{-2\Upsilon}\int_{\R^n}\verts{\cd(w-\chi u_{\mathbf{x}_0;1})}^2\,d\nu_{-\epsilon^{-2}r^2}+(1+\lambda^{\frac{\Upsilon}{2}}r^{2\Upsilon})\int_{\R^n}\verts{\cd(\chi u_{\mathbf{x}_0;1})}^2\,d\nu_{-\epsilon^{-2}r^2}\\&\leq (C\lambda^\Upsilon r^{2\Upsilon}+\frac{4}{3}) \widehat{H}_{\mathbf{x}_0}^u(\epsilon^{-2}r^2)\leq (C\lambda^{\Upsilon}r^{2\Upsilon}+\frac{4}{3})H^w(\epsilon^{-2}r^2) \leq \frac{3}{2}H^w(\epsilon^{-2}r^2).
    \end{align*}
    By the proof of Lemma \ref{lemma pinching implies nodal}, $w$ is $(n+1,C\epsilon^{\frac{1}{C(\Lambda,\Upsilon)}},s)$-symmetric at $\mathbf{0}$. If $\lambda \leq \overline{\lambda}(\Lambda,\Upsilon,\epsilon)$, then $u$ is weakly $(n+1,C\epsilon^{\frac{1}{C(\Lambda,\Upsilon)}},s)$-symmetric at $\mathbf{x}_0$, so if we choose $\epsilon \leq \overline{\epsilon}(\Lambda,\Upsilon)$, then Lemma \ref{lemma-containment-of-singular-sets-2} implies the claim. 
\end{proof}

The following is an epsilon regularity of the frequency based on symmetry similar to Lemma \ref{lemma epsilon regularity for frequency}.
\begin{lemma}\label{lemma epsilon regularity for frequency-2}
    The following holds if $\lambda\leq \ol{\lambda}(\Lambda,\Upsilon)$. If $u$ is weakly $(k,\epsilon,r)$-symmetric at $\mathbf{x}_0$, then
    \begin{align*}
        \widehat{N}_{\mathbf{x}_0}(r^2)\leq \begin{cases}
            4\epsilon & \text{ if } k=n+2,\\
            1+C(\Lambda,\Upsilon) \epsilon& \text{ if } k=n+1.
        \end{cases}
    \end{align*}
\end{lemma}

\begin{proof}
    When $k=n+2$, the claim follows from the proof of Lemma \ref{lemma-containment-of-nodal-sets-2}.

    We now suppose $k=n+1$. By replacing $u$ with $u_{\mathbf{x}_0;1}$, we can assume that $\mathbf{x}_0=\mathbf{0}$ and $u=u_{\mathbf{x}_0;1}$. Let $w$ be a caloric approximation of $\chi u$. As in the proof of Lemma \ref{lemma-containment-of-singular-sets-2}, we can prove that $w$ is $(n+1,C\epsilon,r)$-symmetric. Since $w$ is caloric, the proof of Lemma \ref{lemma epsilon regularity for frequency} implies that
    \begin{align*}
        \int_{r^2}^{e^2r^2}\frac{\tau}{\widehat{H}^u(\tau)}\int_{\R^n} \verts{\cd^2 w}^2\,d\nu_{-\tau}d\tau\leq C(\Lambda,\Upsilon)\epsilon.
    \end{align*}
    In particular, Lemma \ref{lem:caloricapprox} implies that
    \begin{align*}
        \int_{r^2}^{e^2r^2} \frac{\tau}{\widehat{H}^u(\tau)} \int_{\R^n} \verts{\cd^2(\chi u)}^2\,d\nu_{-\tau}d\tau &\leq \int_{r^2}^{e^2r^2} \frac{\tau}{\widehat{H}^u(\tau)} \int_{\R^n}\verts{\cd^2 w}^2 \,d\nu_{-\tau}d\tau  \\
        &\qquad +\int_{r^2}^{e^2r^2} \frac{\tau}{\widehat{H}^u(\tau)} \int_{\R^n}\verts{\cd^2 (w-\chi u)}^2  \,d\nu_{-\tau}d\tau \\ 
        &\leq C(\Lambda,\Upsilon) \epsilon.
    \end{align*}
    Therefore,
    \begin{align*}
        C(\Lambda,\Upsilon)\epsilon&\geq \int_{r^2}^{e^2r^2} \frac{\tau}{\widehat{H}^u(\tau)} \int_{\R^n}\verts{\cd^2 \chi u}^2 \,d\nu_{-\tau}d\tau\\
        &=\int_{r^2}^{e^2r^2} \frac{\tau}{\widehat{H}^u(\tau)} \int_{\R^n} \left(\Delta_f (\chi u)+\frac{1}{2\tau}\chi u\right)^2 \,d\nu_{-\tau} d\tau+\int_{r^2}^{e^2r^2} \frac{1}{2\tau}\left(\widehat{N}(\tau)-1\right) \,d\tau\\
        &\geq \int_{r^2}^{e^2r^2} \frac{1}{2\tau}\left(\widehat{N}(\tau)-1\right) \,d\tau\geq \widehat{N}(r^2)-1.
    \end{align*}
    This proves the second claim.
\end{proof}

\subsection{Quantitative dimension reduction}
We prove a result concerning quantitative dimension reduction similar to Proposition \ref{prop: extra symmetry}. In particular, $u$ is $k$-symmetric with respect to a plane $V$, then it is $(k+1)$-symmetric at points away from $V$, but at a smaller scale.
\begin{proposition}\label{prop: extra symmetry2}
    For any $\epsilon>0$, there exists $\beta_0 = \beta_0(\epsilon,\Lambda,\Upsilon)>0$ such that following holds if $\lambda \leq \overline{\lambda}(\Lambda,\Upsilon)$ and $\delta \leq \overline{\delta}(\epsilon,\Lambda,\Upsilon)$. Let $u$ be a solution  of \eqref{part-2-eq-parabolic-equation} satisfying \eqref{eq:strongerdoubling}, such that for some $r \in (0,1]$ and $\mathbf{x}_0 \in P(\mathbf{0},10)$, we have
    \begin{align} \label{eq:quantred2assumption}  
        \inf_{\tau \in [10^5 n^2 r^2,10^8 n^2 r^2]}\widehat{\mathcal{E}}_{\sqrt{\tau}}(\mathbf{x}_0)<\delta.
    \end{align}
    Suppose $u$ is weakly $(k,\delta,10^{5}nr)$-symmetric at $\mathbf{x}_0$ with respect to some $V\in{\rm Gr}_{\mathcal{P}}(k)$.
    Then, for any $\mathbf{y}\in P(\mathbf{x}_0,2r)\setminus P(\mathbf{x}_0+V,r)$, there exists $\beta(\mathbf{y})\ge \beta_0$ such that
    \begin{align*}
        \widehat{\mathcal{E}}^{k+1,1}_{\beta(\mathbf{y}) r}(\mathbf{y})<\epsilon.
    \end{align*}
\end{proposition}

\begin{proof} 
    We may assume without loss of generality that $\epsilon \leq \overline{\epsilon}(\Lambda,\Upsilon)$. For $\beta>0$ to be determined, let $w$ be the caloric approximation to $u$ at $\mathbf{y}$ on scales $[\beta^2r^2,10^{12}n^2r^2]$ from Lemma \ref{lem:caloricapprox}, so that
    \begin{align} \label{eq:closenessfordimreduction}
        \begin{split} 
            \sup_{\tau \in [\beta^2 r^2,10^{11}n^2r^2]} &\sum_{j=0}^1\int_{\mathbb{R}^n} \tau^j |\nabla^j (w-\chi u_{\mathbf{y};1})|^2 \,d\nu_{-\tau} + \int_{\beta^2 r^2}^{10^{11}n^2r^2} \tau \int_{\mathbb{R}^n} |\partial_t (w-\chi u_{\mathbf{y};1})|^2 \,d\nu_{-\tau} d\tau \\&\leq C(\Lambda,\Upsilon) \lambda^{\Upsilon} r^{4\Upsilon} \widehat{H}_{\mathbf{y}}(r^2).
        \end{split}
    \end{align}
    Writing $\mathbf{y} = (y,s)$, we set $\mathbf{y}'\coloneqq  (y',-s)\coloneqq  (-a^{\frac{1}{2}}(\mathbf{y})y,-s)$ and
    \begin{align*}
        \widehat{w}(z,-\tau)&\coloneqq   w(z+y',-\tau -s), \qquad \check{w}(z,-\tau)\coloneqq w(a^{\frac{1}{2}}(\mathbf{y})a^{-\frac{1}{2}}(\mathbf{0})z+y',-\tau-s), \\  \check{\chi}(z,-\tau)&\coloneqq\chi(a^{\frac{1}{2}}(\mathbf{y})a^{-\frac{1}{2}}(\mathbf{0})z+y',-\tau-s),
    \end{align*}
    so that $H^{\widehat{w}}(\tau)=H_{\mathbf{y'}}^w(\tau)$ for all $\tau \in [r^2,10^{12}n^2r^2]$. For any $\tau \in [r^2,10^{9}n^2r^2]$, we then have
    \begin{align*}
        \operatorname{supp}((\chi-\check{\chi})(\cdot,-\tau)) \subseteq B(0,10^{8}\lambda^{-\frac{\Upsilon}{2}}r^{2\Upsilon})\setminus \overline{B}(0,10^{-8}\lambda^{-\frac{\Upsilon}{2}}r^{2\Upsilon}).
    \end{align*}
    Combining this with the change of variables $x=a^{\frac{1}{2}}(\mathbf{y})a^{-\frac{1}{2}}(\mathbf{0})z+y'$, the ellipticity condition of \eqref{part-2-eq-parabolic-equation}, and Lemma \ref{lemma-change-of-base-point} with $\alpha_0 \leftarrow 10\lambda$, $\alpha \leftarrow \frac{1}{50n}$, $\sigma \leftarrow 4r^2$, and $\tau \leftarrow \tau \in [10^{4}n^2r^2,10^{11}n^2r^2]$, we obtain 
    \begin{align} \label{eq:quantreduction1}
        \begin{split}
            \int_{\mathbb{R}^n} |\check{w}-\chi  u_{\mathbf{0};1}|^2 d\nu_{-\tau} &= \int_{\mathbb{R}^n} |\check{w}(z,-\tau)-\chi(z,-\tau)u_{\mathbf{y};1}(a^{\frac{1}{2}}(\mathbf{y})a^{-\frac{1}{2}}(\mathbf{0})z+y',-\tau-s)|^2 d\nu_{-\tau}(z) \\
            &\leq 2\sqrt{\frac{\det a(\mathbf{0})}{\det a(\mathbf{y})}}\int_{\mathbb{R}^n} |w-\chi u_{\mathbf{y};1}|^2(x,-\tau-s) \frac{e^{-\frac{|a^{\frac{1}{2}}(\mathbf{0})a^{-\frac{1}{2}}(\mathbf{y})(x-y')|^2}{4\tau}}}{(4\pi \tau)^{\frac{n}{2}}}dx\\
            &\qquad + 2\int_{\mathbb{R}^n} |\chi(z,-\tau)- \check{\chi}(z,-\tau) |u_{\mathbf{0};1}^2(z,-\tau)d\nu_{-\tau}(z) \\
            &\leq C \int_{\mathbb{R}^n} |w-\chi u_{\mathbf{y};1}|^2(x,-\tau-s) e^{10\lambda f_{\mathbf{y}'}}d\nu_{\mathbf{y}';-\tau-s}(x)\\
            &\qquad + 4\int_{B(0,10^{8}\lambda^{-\frac{\Upsilon}{2}}r^{2\Upsilon})\setminus \overline{B}(0,10^{-8}\lambda^{-\frac{\Upsilon}{2}}r^{2\Upsilon})} u_{\mathbf{0};1}^2(z,-\tau)d\nu_{-\tau}(z) \\
            &\leq C\int_{\mathbb{R}^n} |w-\chi u_{\mathbf{y};1}|^2 e^{\frac{1}{50n} f}d\nu_{-\tau-s} + C(\Lambda,\Upsilon)e^{-\frac{1}{10^{8}\lambda^{\Upsilon}r^{2(1-2\Upsilon)}}}\widehat{H}_{r^2}^u(\mathbf{y}).
        \end{split}
    \end{align}
    Next, we set $A(\ell)\coloneqq(1-\ell)a^{\frac{1}{2}}(\mathbf{y})a^{-\frac{1}{2}}(\mathbf{0})+\ell I$ and estimate
    \begin{align*}
        |A'(\ell)| = |a^{\frac{1}{2}}(\mathbf{y})a^{-\frac{1}{2}}(\mathbf{0})-I| \leq 2 |a^{-\frac{1}{2}}(\mathbf{y})-a^{-\frac{1}{2}}(\mathbf{0})| \leq C\lambda |\mathbf{y}| \leq C\lambda r,
    \end{align*}
    using the ellipticity and Lipschitz assumptions of \eqref{part-2-eq-parabolic-equation}. Using this and Lemma \ref{lemma-change-of-base-point}, along with the change of variables $x=A(\ell)z+y'$, we obtain the following for all $\tau \in [10^4n^2 r^2,10^{9} n^2r^2]$
    \begin{align*} 
        &\left| \frac{d}{d \ell} \int_{\mathbb{R}^n} w^2(A(\ell)z+y',-\tau-s)d\nu_{-\tau}(z) \right| \hspace{-50 mm} \\
        &= 2\left| \int_{\mathbb{R}^n} w(A(\ell)z+y',-\tau-s)\nabla w(A(\ell)z+y',-\tau-s)\cdot A'(\ell)zd\nu_{-\tau}(z) \right| \\
        &\leq \frac{2|A'(\ell)|\cdot |A(\ell)^{-1}|}{\det A(\ell)} \int_{\mathbb{R}^n} |w(x,-\tau-s)|\cdot |\nabla w(x,-\tau-s)| \cdot |x-y'| \frac{e^{-\frac{|A(\ell)^{-1}(x-y')|^2}{4\tau}}}{(4\pi \tau)^{\frac{n}{2}}}dx \\
        &= C\lambda r \sqrt{\tau} \int_{\mathbb{R}^n} |w|\cdot |\nabla w|e^{\frac{1}{20n}f_{\mathbf{y}'}} \,d\nu_{\mathbf{y}';-s-\tau}\\
        &\leq C\lambda r  \left( \int_{\mathbb{R}^n} |w|^2 e^{\frac{1}{4n}f_{\mathbf{y}'}} \,d\nu_{\mathbf{y}';-\tau-s}\right)^{\frac{1}{2}} \left( \int_{\mathbb{R}^n} \tau |\nabla w|^2 \,d\nu_{\mathbf{y}';-\tau-s} \right)^{\frac{1}{2}} \\
        &\leq C(\Lambda,\Upsilon)\lambda r\int_{\mathbb{R}^n} w^2 \,d\nu_{\mathbf{y}';-s-2\tau}\\
        &\leq C(\Lambda,\Upsilon)\lambda r \int_{\mathbb{R}^n} w^2 \,d\nu_{\mathbf{y}';-s-\tau},
    \end{align*}
    where we used Proposition \ref{proposition-hypercontractivity}, and the fact that $N_{\mathbf{y}'}^w(\tau) \leq CN^w(10^2\tau) \leq C(\Lambda,\Upsilon)$ by \eqref{eq:closenessfordimreduction} and Lemma \ref{lemma-frequency-uniform-bound}. We integrate from $\ell=0$ to $\ell=1$ to obtain
    \begin{align*}
        C(\Lambda)\lambda r \int_{\mathbb{R}^n} w^2 d\nu_{\mathbf{y}';-s-\tau} \geq \left| \int_{\mathbb{R}^n}w^2 d\nu_{\mathbf{y}';-s-\tau} - \int_{\mathbb{R}^n}\check{w}^2 d\nu_{-\tau}\right|,
    \end{align*}
    which implies
    \begin{align} \label{eq:quantreduction2}
        (1-C(\Lambda)\lambda r)\int_{\mathbb{R}^n}w^2 d\nu_{\mathbf{y}';-s-\tau} \leq \int_{\mathbb{R}^n} \check{w}^2 d\nu_{-\tau} \leq (1+C(\Lambda)\lambda r)\int_{\mathbb{R}^n} w^2 d\nu_{\mathbf{y}';-s-\tau}.
    \end{align}
    Combining \eqref{eq:quantreduction1} and \eqref{eq:quantreduction2} yields
    \begin{align*} 
        (1-C(\Lambda,\Upsilon)\lambda^{\frac{\Upsilon}{2}}\tau^{\frac{\Upsilon}{2}})\widehat{H}^u(\tau) \leq H_{\mathbf{y}'}^w(\tau) \leq (1+C(\Lambda,\Upsilon)\lambda^{\frac{\Upsilon}{2}}\tau^{\frac{\Upsilon}{2}})\widehat{H}^u(\tau) 
    \end{align*}
    for all $\tau \in [10^{4}n^2r^2,10^{9}n^2r^2]$. We can therefore argue as in \eqref{eq comparability of frequency and doubling} to obtain
    \begin{align*}
        \widehat{N}^u(\tau) &\leq \frac{1}{\log(2)}\log \left( \frac{\widehat{H}^u(2\tau)}{\widehat{H}^u(\tau)} \right) + C(\Lambda,\Upsilon)\lambda^{\Upsilon}\tau^{\Upsilon} \leq \frac{1}{\log(2)}\log \left( \frac{\widehat{H}^w(2\tau)}{\widehat{H}^w(\tau)} \right) + C(\Lambda,\Upsilon)\lambda^{\frac{\Upsilon}{2}}\tau^{\frac{\Upsilon}{2}} \\
        &\leq N_{\mathbf{y}'}^w(2\tau)+C(\Lambda,\Upsilon)\lambda^{\frac{\Upsilon}{2}}\tau^{\frac{\Upsilon}{2}}.
    \end{align*}
    Similarly, we have $N_{\mathbf{y}'}^w(\tau) \leq \widehat{N}^u(2\tau) + C(\Lambda,\Upsilon)\lambda^{\frac{\Upsilon}{2}}\tau^{\frac{\Upsilon}{2}}$, so that \eqref{eq:quantred2assumption} implies 
    \begin{align*}
        \inf_{\tau \in [10^5 n^2 r^2,10^8 n^2 r^2]} \mathcal{E}_{\sqrt{\tau}}(\mathbf{y}';w) < C(\Lambda,\Upsilon)(\delta+ \lambda^{\frac{\Upsilon}{2}}\tau^{\frac{\Upsilon}{2}}) \leq C(\Lambda,\Upsilon)\delta.
    \end{align*}
    Moreover, because $u$ is weakly $(k,\delta,10^{5}nr)$-symmetric at $\mathbf{0}$, it follows from Lemma \ref{lemma comparability of H} that $u$ is weakly $(k,C(\Lambda,\Upsilon)\delta,10^3r)$-symmetric at $\mathbf{y}$ with respect to $V$. Combining this with \eqref{eq:closenessfordimreduction}, we conclude that $w$ is $(k,C(\Lambda,\Upsilon)\delta,10^2r)$-symmetric at $\mathbf{0}$. By Lemma \ref{lemma-comparison-of-caloric-energy}, it follows that $w$ is $(k,C(\Lambda,\Upsilon)\delta,10r)$-symmetric at $\mathbf{y}'$ with respect to $V$. Moreover, we have
    \begin{align*}
        d_{\mathcal{P}}(\mathbf{0},\mathbf{y}'+V) = |\pi_V^{\perp}(\mathbf{y}')| \geq |\pi_V^{\perp}(\mathbf{y})|-|\pi_V^{\perp}(y-a^{\frac{1}{2}}(\mathbf{y})y,0)| = (1-\lambda r) d_{\mathcal{P}}(\mathbf{y},V).
    \end{align*}
    We can therefore apply Proposition \ref{prop: extra symmetry} with $\mathbf{x}_0 \leftarrow \mathbf{y}'$, $\mathbf{y} \leftarrow \mathbf{0}$, $\eta \leftarrow \frac{1}{2}$, $u \leftarrow w$, $\epsilon \leftarrow \epsilon^3$, $\delta \leftarrow C(\Lambda,\Upsilon)\delta$ to obtain $\beta = \beta(\mathbf{y}) \geq \epsilon^{C(\Lambda,\Upsilon)}$ such that 
    \begin{align*}
        \mathcal{E}_{\beta(\mathbf{y}) r}^{k+1,1}(\mathbf{0};w) < \epsilon^3, \qquad |N^w(10^5 \beta^2)-N^w(10^{-5}\beta^2)| < \epsilon^6,
    \end{align*}
    if we take $\delta \leq \overline{\delta}(\epsilon,\Lambda,\Upsilon)$. From Proposition \ref{thm: sym split equiv}\ref{thm: sym split equiv-1}, it follows that $w$ is $(k+1,C(\Lambda,\Upsilon)\epsilon^{3},\beta(\mathbf{y})r)$-symmetric. From \eqref{eq:closenessfordimreduction}, it follows that $u$ is  weakly $(k+1,C(\Lambda,\Upsilon)\epsilon^{3},\beta(\mathbf{y})r)$-symmetric at $\mathbf{y}$, and (arguing as in \eqref{eq comparability of frequency and doubling}) 
    \begin{align*}
        |\widehat{N}_{\mathbf{y}}^u(10^4 \beta(\mathbf{y})^2)-\widehat{N}_{\mathbf{y}}^u(10^{-4}\beta(\mathbf{y})^2)| < \epsilon^5.
    \end{align*}
    Then Proposition \ref{thm:part2-sym split equiv}\ref{thm:part2-sym split equiv-2} yields $\widehat{\mathcal{E}}_{\beta(\mathbf{y})r}^{k+1,1}(\mathbf{y}) < \epsilon$.
\end{proof}

\subsection{Proof of main theorems in the general setting}

Given the results we have developed thus far, the proof of Theorem \ref{main-theorem-general-setting} closely follows that of the caloric setting. Therefore, we only state the necessary modifications.

We first define neck regions corresponding to solutions to \eqref{part-2-eq-parabolic-equation}.
\begin{definition} \label{def:neckregionpart2}
    Given a $r\in (0,1]$, a closed set $\mathcal{C} \subseteq P(\mathbf{x}_0,2r)$ with $\mathbf{x}_0 \in \mathcal{C}$, and a function $\mathbf{r}_{\bullet}:\mathcal{C} \to [0,\gamma r]$, we say that
    \begin{align*}
        \mathcal{N} \coloneqq  P(\mathbf{x}_0,2r)\setminus \overline{P}(\mathcal{C},\mathbf{r}_{\bullet})
    \end{align*}
    is an \textit{$(m,k,\delta,\eta)$-neck region (of scale $r$)} if it satisfies \ref{neck-vitali-covering}, \ref{neck-Hausdorff} of Definition \ref{definition neck region}, as well as the following:

    \begin{enumerate}[label={(n\arabic*)}]\setcounter{enumi}{1}
        \item $\sup_{s\in [\mathbf{r}_{\mathbf{x}},\gamma^{-3}r]} (|\widehat{N}_{\mathbf{x}}(s^2)-m|+\gamma^{\frac{\Upsilon}{4}}s^{\Upsilon})< \delta$ for all $\mathbf{x} \in \mathcal{C}$; \label{neck-frequency-pinching-part2}
        
        \item For all $s\in[\mathbf{r}_{\mathbf{x}},\gamma^{-3}r]$, $u$ is weakly $(k,\delta,s)$-symmetric at $\mathbf{x}$ with respect to $V$ but not weakly $(k+1,\eta,s)$-symmetric.
        \label{neck-k-sym-part2}
    \end{enumerate}
    We say $\mathcal{N}$ is a \textit{strong $(m,k,\delta,\eta)$-region} if, in addition, it satisfies \eqref{eq:improvedCinVinC} and \ref{neck-lipschitz} of Definition \ref{definition neck region}.
\end{definition}

Similarly, we define the packing measure as well as
\begin{align*}
    \widehat{\mathcal{V}}_{\delta,m,r}^u(\mathbf{x})\coloneqq \left\{ \mathbf{y}\in P(\mathbf{x},4r)\colon |\widehat{N}_{\mathbf{y}}^u(s^2)-m|+\lambda^{\frac{\Upsilon}{4}}r^\Upsilon\leq \delta \text{ for any } s\in [\delta r,\delta^{-1}r] \right\}.
\end{align*}
By making appropriate modifications to the definition of different types of balls, we can prove the following neck structure theorems.
\begin{theorem}[Neck structure theorems]\label{thm:strongneckstructure-part2} 
    Suppose $u$ is a solution to \eqref{part-2-eq-parabolic-equation} satisfying \eqref{eq:strongerdoubling}. Then the conclusions of Theorem \ref{theorem-neck-structure} and Theorem \ref{thm:strongneckstructure} hold.
\end{theorem}

\begin{proof}
    The proof is the same, except that we replace Lemma \ref{lem:carlesoncondition} with the following lemma.
\end{proof}

Recall the definition of $\widehat{\beta}$ from \ref{eq hat calculus}. We the Carleson condition for $\widehat{\beta}$ as in Lemma \ref{lem:carlesoncondition} taking into the error term in $\widehat{\cE}$.
\begin{lemma} \label{lem:carlesonconditionpart2} 
    Suppose $P(\mathbf{0},2)\setminus \bigcup_{\mathbf{y}\in \mathcal{C}}\overline{P}(\mathbf{y},\mathbf{r}_{\mathbf{y}})$ is a strong $(m,k,\delta,\eta)$-neck region modeled on a vertical $k$-plane $V\in \operatorname{Gr}_{\mathcal{P}}(k)$. Let $\mu$ be the packing measure, let $\pi:\mathcal{C}\to V$ be the projection map, and set $\widehat{\mu}\coloneqq \pi_{\ast}\mu$. Then, for any $\mathbf{x}\in P(\mathbf{0},\frac{3}{2})$ and $r>0$ such that $P(\mathbf{x},20r)\subset P(\mathbf{0},\frac{7}{4})$,
    \begin{align*}
        \int_{P^V(\mathbf{x},10r)} \int_{\frac{1}{8}\widehat{\mathbf{r}}_{\mathbf{y}}}^r\widehat{\beta}_{\mathcal{P},k}^2(\mathbf{y},10^5s)\,\frac{ds}{s}d\widehat{\mu}(\mathbf{y}) \leq C\delta r^{k}.
    \end{align*}
\end{lemma}

\begin{proof} 
    Set $r_{i}\coloneqq 2^{-i}$ for $i \in \mathbb{N}$. We argue as in Lemma \ref{lem:carlesoncondition} to obtain
    \begin{align*} 
        \int_{P(\mathbf{x},10r)} \int_{\frac{1}{8}\widehat{\mathbf{r}}_{\mathbf{y}}}^r \widehat{\beta}_{\mathcal{P},k}^2(\mathbf{y},10^5s)\,\frac{ds}{s}d\widehat{\mu}(\mathbf{y}) 
        &\leq  C\sum_{i\in \mathbb{N}} \int_{P(\mathbf{x},20r)} \1_{\{ \frac{1}{16}\widehat{\mathbf{r}}_{\mathbf{z}}\leq r_i \leq 10^5r\} } \mathcal{\widehat{\mathcal{E}}}_{20nr_i}(\mathbf{z})d\widehat{\mu}(\mathbf{z}) \\
        &\leq C\sum_{i\in \mathbb{N}} \int_{P(\mathbf{x},20r)} \left( \widehat{N}_{\mathbf{z}}(2\cdot 10^4 r_i^2)-\widehat{N}_{\mathbf{z}}(8r_i^2) + \lambda^{\frac{\Upsilon}{4}}r_i^{\Upsilon} \right)d\widehat{\mu}(\mathbf{z})\\
        &\leq C\int_{P(\mathbf{x},20r)} \left( \widehat{N}_{\mathbf{z}}(\gamma^{-2}r_i^2)-\widehat{N}_{\mathbf{z}}(\widehat{\mathbf{r}}_{\mathbf{z}}^2) + \lambda^{
        \frac{\Upsilon}{4}}r^{\Upsilon}
         \right) d\widehat{\mu}(\mathbf{z}),
    \end{align*}
    where in the first inequality, we use Lemma \ref{lem:betapart2}.
\end{proof}

We also have neck decomposition theorems in the general case.
\begin{theorem}[Neck Decomposition] \label{theorem-neck-decomposition-part2}
    For any $\epsilon,\eta \in (0,1]$, the following holds if $r \in (0,1]$ and $\lambda \leq \overline{\lambda}(\epsilon,\eta,\Lambda,\Upsilon)$. Suppose $u$ is a solution to \eqref{part-2-eq-parabolic-equation} satisfying \eqref{eq:strongerdoubling}.  there exists a decomposition
    \begin{align*}
        P(\mathbf{0},r)\subseteq \bigcup_a \mathcal{N}^a\cup \bigcup_b P(\mathbf{x}_b,r_b) \cup \left( \widetilde{\mathcal{C}} \cup \bigcup_a \mathcal{C}_{a,0} \right),
    \end{align*}
    where the following hold:
    \begin{enumerate}[label=(\alph*)]
        \item \label{neckdecomposition:aballs-part2} $\mathcal{N}^a = P(\mathbf{x}_a,2r_a)\setminus \overline{P}(\mathcal{C}_a,\mathbf{r}_a)$ is an $(m_a,k,\epsilon ,C^{-1}(\Lambda,\Upsilon)\eta)$-neck region for some positive integer $m_a \leq C(\Lambda,\Upsilon)$, where $\cC_a$ is the center set associated to $\cN^a$;
    
        \item \label{neckdecomposition:bballs-part2} $\widehat{\cE}_{100 r_b}^{k+1,1}(\mathbf{y}_b) \leq \eta$ for some $\mathbf{y}_b\in P(\mathbf{x}_b,4r_b)$;

        \item \label{neckdecomposition:contentestimate-part2} A $k$-dimensional parabolic Minkowski content estimate holds:
        \begin{align*}
            \sum_a r_a^k + \sum_b r_b^k\leq C(\epsilon,\eta,\Lambda,\Upsilon)r^k;
        \end{align*}

        \item \label{neckdecomposition:Minkowskiofcenters-part2} Further, $\mathcal{H}_{\mathcal{P}}^k(\widetilde{\mathcal{C}})=0$. 
    \end{enumerate}
    Similarly modified statements of Theorems \ref{theorem-neck-decomposition2} and Theorem \ref{theorem-neck-decomposition3} hold.
\end{theorem}

\begin{proof} 
    The proof of Theorem \ref{theorem-neck-decomposition} relied only on covering arguments, Lemma \ref{newlineup}, and Theorem \ref{thm: sym split equiv}. Hence, the proof can be adapted by replacing these with Lemma \ref{newlineup2} and Theorem \ref{thm:part2-sym split equiv}, noting that the hypotheses hold under our assumption $r \leq \overline{r}(\epsilon,\eta,\Lambda,\Upsilon)$.
    
    For the last claim, we use Lemma \ref{lemma epsilon regularity for frequency-2} instead of Lemma \ref{lemma epsilon regularity for frequency}.
\end{proof}

Similarly, we have volume estimates for the quantitative strata.
\begin{theorem} \label{thm: caloric vol part 2}
    For any $\epsilon \in (0,1]$, the following holds if $\lambda \leq \overline{\lambda}(\epsilon,\Lambda,\Upsilon)$ and $r \in (0,1]$. Suppose $u$ is a solution to \eqref{part-2-eq-parabolic-equation} satisfying \eqref{eq:strongerdoubling}. Then, for any $k \in \{1,\dots, n+1\}$, 
    \begin{align}
    \label{ineq: main content est-2}
        \cH_{\cP}^{n+2}(P(\widehat{\cS}_{\epsilon,r}^k,r)\cap P(\mathbf{0},1))\leq C(\epsilon,\Lambda,\Upsilon) r^{n+2-k}.
    \end{align}
    Moreover, $\widehat{\cS}_\epsilon^k$ is parabolic $k$-rectifiable with $\mathcal{H}^k_{\mathcal{P}}(\widehat{\cS}_\epsilon^k\cap P(\mathbf{0},1))\le C(\epsilon,\Lambda,\Upsilon)$. 
    Furthermore, if $k=n+1$ or $n$, for $\mathcal{H}^1$-almost every time $t\in(-1,1)$, $\widehat{\cS}^{k,t}_{\epsilon}\coloneqq \widehat{\cS}^k_{\epsilon}\cap (\R^n\times\{t\})$ is $(k-2)$-rectifiable with
    \begin{align}
        \mathcal{H}^{k-2}(\widehat{\cS}^{k,t}_{\epsilon})\le C(\epsilon,\Lambda,\Upsilon),
        \label{ineq: main content est time slice-2} 
    \end{align}
    and for $\mathcal{H}^{\frac{1}{2}}$-almost every time $t\in (-1,1)$, $\mathcal{H}^{k-1}(\widehat{\mathcal{S}}_{\epsilon}^{k,t})=0$.
\end{theorem}

\begin{proof}
    The proof is verbatim as in the proof of Theorem \ref{thm: caloric vol part 1} after we replace Theorem \ref{thm: sym split equiv}, Lemma \ref{lemma: propagation of symmetry and pinching}, Proposition \ref{prop: extra symmetry}, Theorem \ref{theorem-neck-structure} Theorem \ref{theorem-neck-decomposition} with Theorem \ref{thm:part2-sym split equiv}, Proposition \ref{lemma: propagation of symmetry and pinching-2}, Proposition \ref{prop: extra symmetry2}, Theorem \ref{theorem-neck-decomposition-part2} and Theorem \ref{thm:strongneckstructure-part2}, respectively.
\end{proof}

We now show that for any function satisfying the hypotheses of our main theorems, parabolic rescaling at sufficiently small scales will satisfy the recurring hypotheses of Section \ref{more-general-parbolic-equations}.
\begin{lemma} \label{lem:ellipsoid}
    Under the hypotheses of Theorem \ref{thm:regularity-general} and Theorem \ref{main-theorem-general-setting}, there exists $\Lambda = \Lambda(M,\Theta)$ such that for any $\lambda>0$, the following holds if $r\leq \overline{r}(\lambda,M)$. For any $\mathbf{x} \in P(\mathbf{0},\frac{1}{4})$, $u_{\mathbf{x};r}$ satisfies \eqref{part-2-eq-parabolic-equation} and \eqref{eq:strongerdoubling}.
\end{lemma}

\begin{proof} 
    We can assume that $M\geq 1$. By Remark \ref{remark pde for parabolic rescaling} with $\lambda \leftarrow M$ and $r\leftarrow \frac{\lambda}{200^2M}$, imply that $u_{\mathbf{x};r}$ satisfies \eqref{part-2-eq-parabolic-equation} if $r \leq \frac{\lambda}{200^2M}$. Fix $\mathbf{y}=(y,s) \in P(\mathbf{0},\frac{1}{2})$. It remains to verify \eqref{eq:strongerdoubling}. From 
    \begin{align*}
        \int_{B(y,4) \times [s-16,s]}u^2 d\mathcal{H}_{\mathcal{P}}^{n+2} \leq \int_{B(0,5)\times [-25,s]} u^2 d\mathcal{H}_{\mathcal{P}}^{n+2} \leq \Theta \int_{B(0,\frac{1}{2})}u^2(x,s)dx \leq \Theta \int_{B(\mathbf{y},1)} u^2(x,s)dx.
    \end{align*}
    By \cite[Theorem 3]{escauriaza2006doubling}, for any $\mathbf{y}_0 \in P(\mathbf{0},\frac{1}{2})$ and any $r'\leq \ol{r}'(M,\Theta)$, we have
    \begin{align}
        \int_{B(y_0,4r')\times [s_0-(4r')^2,s_0]}u^2 d\mathcal{H}_{\mathcal{P}}^{n+2} \leq C(M,\Theta)(r')^2 \int_{B(y_0,r')}u^2(y,s_0)dy.
    \end{align}
    In terms of the rescaling $u_{\mathbf{x};r}$, this is
    \begin{align} \label{eq: ellipsoid}
        \begin{split}
            &\int_{r^{-1}a^{\frac{1}{2}}(\mathbf{x})(B(y_0,4r')-x) \times r^{-2}[s_0-(4r')^2-t,s_0-t]}u_{\mathbf{x};r}^2 d\mathcal{H}_{\mathcal{P}}^{n+2} \\
            &\leq C(M,\Theta)\left(\frac{r'}{r}\right)^2 \int_{r^{-1}a^{\frac{1}{2}}(\mathbf{x})(B(y_0,r')-x)}u_{\mathbf{x};r}^2(y,r^{-2}(s_0-t))dy.
        \end{split}
    \end{align}
    
    Fix $\mathbf{y} = (y,s) \in P(\mathbf{0},\lambda^{-1})$ and $\rho \in (0,M\lambda^{-1}]$. If $r \leq \overline{r}(\lambda,M)$, then we can apply \eqref{eq: ellipsoid} with $y_0 \leftarrow x+ra^{\frac{1}{2}}(\mathbf{x})y$, $s_0 \leftarrow t+r^2s$, and $r' \leftarrow \rho r$ to obtain
    \begin{align*}
        \int_{(y+a^{\frac{1}{2}}(\mathbf{x})B(0,4\rho)) \times [s-(4\rho)^2,s]} u_{\mathbf{x};r}^2 d\mathcal{H}_{\mathcal{P}}^{n+2} \leq C(M,\Theta)\rho^2 \int_{y+a^{\frac{1}{2}}(\mathbf{x})B(0,\rho)}u_{\mathbf{x};r}^2(z,s)dz.
    \end{align*}
    If $r\leq \overline{r}(\lambda,\Lambda)$, then we can iterate \cite[Theorem 2(1)]{escauriaza2006doubling} and use $(2M)^{-1}I\leq a(\mathbf{x})\leq 2MI$ to obtain
    \begin{align*}
        \int_{(B(y,\frac{4\rho}{M})) \times [s-\left(\frac{4\rho}{M}\right)^2,s]} u_{\mathbf{x};r}^2 d\mathcal{H}_{\mathcal{P}}^{n+2}&\leq \int_{(y+a^{\frac{1}{2}}(\mathbf{x})B(0,2M\rho)) \times [s-(4\rho)^2,s]} u_{\mathbf{x};r}^2 d\mathcal{H}_{\mathcal{P}}^{n+2}\\
        & \leq C(M,\Theta)\rho^2 \int_{y+a^{\frac{1}{2}}(\mathbf{x})B(0,(2M)^{-1}\rho)}u_{\mathbf{x};r}^2(z,s)dz\\
        &\leq C(M,\Theta)\rho^2 \int_{B(y,\frac{\rho}{M})}u_{\mathbf{x};r}^2(z,s)dz.
    \end{align*}
    The claim then follows by replacing $\rho \leftarrow M\rho$.
\end{proof}

\begin{proof}[Proof of Theorem \ref{thm:regularity-general},  Theorem \ref{main-theorem-general-setting}, Corollary \ref{corollary-general-setting}]
    Choose $\epsilon\leq \ol{\epsilon}(\Lambda)$ such that Proposition \ref{proposition-containment-2} holds. We now choose $\lambda \leq \ol{\lambda}(\epsilon,\Lambda,\Upsilon)$ as in Theorem \ref{thm: caloric vol part 2}. Due to Lemma \ref{lem:ellipsoid}, we can apply Theorem \ref{thm: caloric vol part 2} to $u_{\mathbf{x},r_0}$, where $r_0 \leq \overline{r}_0(\lambda, M)$ for $\mathbf{x}\in P(\mathbf{0},\frac{1}{4})$. Then by Proposition \ref{proposition-containment-2}, we have 
    \begin{align*} 
        \widehat{\mathscr{C}}^k (u)=\widehat{\mathscr{C}}^k (u_{\mathbf{x};r_0})\subset \widehat{\mathscr{C}}^k_{r}(u_{\mathbf{x};r_0})\subset \widehat{\cS}_{\epsilon,r}^k(u_{\mathbf{x};r_0}).
    \end{align*}
    Therefore, proof of the Theorem \ref{main-theorem-general-setting} follows from a covering argument. The remaining statements are similarly deduced as in the Proof of Corollary \ref{theorem time slice} and Theorem \ref{thm:regularity}, by making the following replacements:
    \begin{enumerate}
        \item Theorem \ref{theorem-neck-decomposition2} and Theorem \ref{theorem-neck-decomposition3} with Theorem \ref{theorem-neck-decomposition-part2},

        \item Proposition \ref{proposition-containment} with Proposition \ref{proposition-containment-2},

        \item Proposition \ref{lemma pinching implies nodal} with Proposition \ref{lemma pinching implies nodal-2},

        \item Theorem \ref{thm:strongneckstructure} with Theorem \ref{thm:strongneckstructure-part2}.
    \end{enumerate}    
\end{proof}

\begin{proof}[Proof of Corollary \ref{cor:qualitative}]
    Fix $\mathbf{x}_0 =(x_0,t_0)\in \Omega$ and $r>0$ such that $P(\mathbf{x}_0,10r)\subseteq \Omega$. Set 
    \begin{align*}
        \Theta \coloneqq \sup_{t\in [t_0-2r,t_0+2r]} \frac{\int_{B(x_0,5r)\times [-25,t]}u^2\,d\mathcal{H}_{\mathcal{P}}^{n+2}}{\int_{B(x_0,\frac{1}{2}r)}u^2(x,t)\,dx}.
    \end{align*}
    Because $u$ is continuous, it follows that $t\mapsto \int_{B(x_0,5r)}u^2(x,t)dx$ is continuous, so that either $\Theta <\infty$ or else $u(\cdot,t)|_{B(x_0,5r)}\equiv 0$ for some $t\in [t_0-2r,t_0+2r]$. In the latter case, it follows from either of \cite[Theorem 1]{AlVe}, \cite[Theorem 3]{Fernandez} that $u(\cdot,t)$ vanishes on the connected component of $\Omega \cap (\mathbb{R}^n \times \{t\})$ containing $x_0$, yielding a contradiction. We may therefore apply Theorem \ref{main-theorem-general-setting} to any locally finite cover $\{P(x_i,r_i)\}_{i \in \mathbb{N}}$ of $\Omega$ with $P(x_i,10r_i)\subseteq \Omega$ in order to obtain the claim.
\end{proof}

\appendix

\section{Integration by Parts and Asymptotic Expansions}
In this appendix, we provide justification for integration by parts used in Section \ref{parabolic frequency} as well as an $L^2$-expansion used in Claim \ref{claim:growthordecay}.

Suppose $u\in C^{\infty}(\mathbb{R}^n \times (t_0,t_1))$ is a caloric function satisfying
\begin{align*}
    C_{\epsilon}\coloneqq\sup_{t\in [t_0+\epsilon,t_1-\epsilon]} \int_{\mathbb{R}^n} u^2 d\nu_{0,t_1;t}<\infty
\end{align*}
for all $\epsilon>0$. For any $t\in [t_0+\epsilon,t_1-\epsilon]$ and $\alpha >0$, we can then estimate
\begin{align*}
    C_{\epsilon} \geq \int_{\mathbb{R}^n}u^2 d\nu_{0,t_1;t} \geq \frac{1}{(4\pi(t_1-t))^{\frac{n}{2}}}e^{-\frac{(|x|-1)_+^2}{4(t_1-t)}}\int_{B(x,1)} u^2(z,t)dz,
\end{align*}
so that interior parabolic estimates yield
\begin{align}\label{eq:coursehigherderivatives}
    |\partial_t^k \nabla^{\ell} u(x,t)|\leq C(\alpha,\epsilon,k,\ell)e^{(1+\alpha)\frac{|x|^2}{4(t_1-t)}}.
\end{align}
An elementary consequence of \eqref{eq:coursehigherderivatives} is that for any $k,\ell\in \mathbb{N}$, $\mathbf{y}=(y,s)\in \mathbb{R}^n \times [t_0+\epsilon,t_1-\epsilon]$, and $t\in [t_0+\epsilon,s)$, there exists $\alpha=\alpha(\epsilon)>0$ such that
\begin{align*}
    \int_{\mathbb{R}^n} (1+|\partial_t^k \nabla^{\ell}u|)^2 e^{\alpha(\epsilon)f_{\mathbf{y}}}d\nu_{\mathbf{y};t} <C(\alpha,\epsilon,k,\ell).
\end{align*}
In particular, this justifies differentiation under the integral sign when computing time derivatives of quantities of the form
\begin{align*}
    \frac{d}{dt}\int_{\mathbb{R}^n} \Phi \,d\nu_{\mathbf{x}_0;t},
\end{align*}
where $\Phi$ is any quantity we consider in Section \ref{parabolic frequency} to Section \ref{volume-estimates-and-rectifiability}. Moreover, we can freely integrate by parts any such quantities and can differentiate the expansion \eqref{eq spectral decomposition} termwise by virtue of \eqref{eq:coursehigherderivatives} and the following lemma.

\begin{lemma} \label{IBPlemma} 
    Suppose a vector field $X\in C^{\infty}(\mathbb{R}^n,\mathbb{R}^n)$ such that for any $\delta>0$, and $t_0 < t < s < t_1$
    \begin{align*}
        |X|(y,t) + |\nabla X|(y,t)\leq e^{\frac{|y|^2}{(4+\delta)(s-t)}}.
    \end{align*}
    Then for all $y \in \mathbb{R}^n$, $\mathbf{y}\coloneqq (y,s)$ satisfies
    \begin{align*}
        \int_{\mathbb{R}^n} \operatorname{div}_{f_{\mathbf{y}}}(X)\,d\nu_{\mathbf{y};t}=0.
    \end{align*}
\end{lemma}

\begin{proof} 
    We first observe that $\operatorname{div}_{f_{\mathbf{y}}}(X) \in L^1(\mathbb{R}^n,d\nu_{\mathbf{y};t})$ by our assumptions on $X$ and the fact that $|\nabla f_{\mathbf{y}}|$ has a linear growth. Because 
    \begin{align*}
        \left| \int_{B(y,r)} \operatorname{div}_{f_{\mathbf{y}}}(X)d\nu_{\mathbf{y};t} \right| &= \frac{1}{(4\pi(s-t))^{\frac{n}{2}}} \left| \int_{\partial B(y,r)} \left\langle X(x), \frac{x-y}{|x-y|} \right\rangle e^{-f_{\mathbf{y}}(x,t)}d\mathcal{H}^{n-1}(x) \right| \\ &\leq \frac{1}{(4\pi(s-t))^{\frac{n}{2}}} r^{n-1}e^{-\frac{r^2}{4(s-t)}} e^{\frac{(r+|y|)^2}{4(1+\delta)(s-t)}},
    \end{align*}
    the claim follows by taking $r\to \infty$. 
\end{proof}

The lemma below proves that under certain regularity assumption solutions to drift heat equation admit a series expansion in terms of homogeneous caloric polynomials. The following result is well-known, but we provide a proof for completeness.
\begin{lemma} \label{lemma asymptotic expansion}
    Suppose $u\in C^{\infty}(\mathbb{R}^n \times (0,4A])$ satisfies 
    \begin{align*}
        (\partial_s - \Delta_f)u = \frac{m}{2}u
    \end{align*}
    such that for any $\epsilon>0$ there exists $C_{\epsilon}<\infty$ satisfying
    \begin{align*} 
        \esssup_{s\in [\epsilon,4A]} \int_{\mathbb{R}^n} u^2(\cdot,s) \,d\nu \leq C_{\epsilon} \text{ and }\int_{0}^{4A} \int_{\mathbb{R}^n} u^2 \,d\nu ds \leq C.
    \end{align*}
    Then there exist $\widehat{p}_k \in \widehat{\mathcal{P}}_k$ such that 
    \begin{align*}
        \lim_{N\to \infty} \int_{0}^{4A} \int_{\mathbb{R}^n} \left|u- \sum_{k=0}^N e^{-\frac{(k-m)}{2}s}\widehat{p}_k \right|^2 d\nu ds =0. 
    \end{align*}
\end{lemma}
\begin{proof} 
    Define a test function $\varphi(x,s)\coloneqq \psi(s)\zeta(x)$ where $\psi(s)\in C_c^\infty ((0,4A))$ and $\zeta \in C_c^{\infty}(\mathbb{R}^n)$. Then by the weak formulation, we know that
    \begin{align*} 
        0 &=\int_0^{4A} \int_{\mathbb{R}^n} (-u\partial_s \varphi - u \Delta_f \varphi -\frac{m}{2}u\varphi )\,d\nu ds \\
        &=-\int_0^{4A} \partial_s \psi \left(\int_{\mathbb{R}^n}u \zeta d\nu \right)ds - \int_{0}^{4A} \psi \left( \int_{\mathbb{R}^n} u\left(\Delta_f \zeta+\frac{m}{2}\zeta\right) \,d\nu \right) \,ds.
    \end{align*}
    Fix any $\widehat{q}_k\in \widehat{\cP}_k$. Choose a sequence $\zeta_j \in C_c^{\infty}(\mathbb{R}^n)$ such that 
    \begin{align*}
        \lim_{j\to \infty} \int_{\mathbb{R}^n} \left( |\zeta_j-\widehat{q}_k|^2 + |\nabla (\zeta_j - \widehat{q}_k)|^2 + |\Delta_f(\zeta_j - \widehat{q}_k)|^2 \right) \,d\nu=0.
    \end{align*}
    Because $\esssup_{s\in \operatorname{supp}(\psi)}\Verts{u(\cdot,s)}_{L^2(\mathbb{R}^n,\nu)} <\infty$, the dominated convergence theorem then implies
    \begin{align*}
        0= -\int_0^{4A} \widehat{a}_k(s) \partial_s\psi(s) ds - \int_0^{4A} \psi(s) \left(  \frac{m-k}{2} \widehat{a}_k(s) \right) \,ds,
    \end{align*}
    where $\widehat{a}_k(s)\coloneqq \int_{\R^n} u\widehat{q}_k\,d\nu$. Thus the following holds in a weak sense
    \begin{align}
        \partial_s\widehat{a}_k(s)=\frac{m-k}{2} \widehat{a}_k(s).\label{eq ode for coefficients}
    \end{align}
    Since $\widehat{a}_k(s)\in L^2((0,4A))$, we get $\partial_s\widehat{a}_k(s)\in L^2((0,4A))$ which implies that $\widehat{a}_k(s)\in H^1((0,4A))\subset C((0,4A))$. Therefore, \eqref{eq ode for coefficients} holds in a classical sense. In particular, $\widehat{a}_k(s)=e^{\frac{m-k}{2}s}\widehat{a}_k(0)$. Thus letting $b_k(\cdot ,s)\in \cP_k$ denote the $L^2(\nu)$-projection of $u(\cdot, s)$ onto $\widehat{\mathcal{P}}_k$, we have $b_k(\cdot,s)=e^{\frac{m-k}{2}s}b_k(\cdot,0)$. Moreover, $u(\cdot,s) \in L^2(\mathbb{R}^n,\nu)$ implies
    \begin{align*} 
        \lim_{N\to \infty} \int_{\mathbb{R}^n} \left| u(\cdot,s)-\sum_{j=0}^N b_j(\cdot,s) \right|^2 \,d\nu =0,    
    \end{align*}
    as well as 
    \begin{align*} 
        \int_{\mathbb{R}^n} \left| u(\cdot,s)-\sum_{j=0}^N b_j(\cdot,s) \right|^2 \,d\nu \leq \int_{\mathbb{R}^n}u^2(\cdot,s)\,d\nu<\infty.    
    \end{align*}
    The dominated convergence theorem and our assumption $\int_0^{4A}\int_{\mathbb{R}^n} u^2 \,d\nu ds<\infty$ therefore imply
    \begin{align*}
        \lim_{N\to \infty} \int_{0}^{4A} \int_{\mathbb{R}^n} \left|u- \sum_{k=0}^N e^{-\frac{(k-m)}{2}s}\widehat{p}_k \right|^2 d\nu ds =0, 
    \end{align*}
    where $\widehat{p}_k\coloneqq b_k(\cdot,0)$.
\end{proof}

\section{Criterion for Regularity of Parabolic Lipschitz Functions}
For the sake of completeness, we provide a well-known sufficient criterion for the graph of a parabolic Lipschitz function to be regular in the sense of Definition \ref{def:regularlipschitz}, see \cite{hofman-caloric}.

Fix a vertical plane $V\in {\rm Gr}_{\cP}(k)$. Suppose $f\colon V \to \mathbb{R}$ is a function supported in $P(\mathbf{0},5)$. For $\mathbf{x}\in \R^{n-1}\times \R$ and $r>0$, recall that
\begin{align*}
    \kappa^2(\mathbf{x},r)=\kappa^2_k(\mathbf{x},r;f)=\inf_{\ell}\frac{1}{r^{k}}\int_{V\cap P(\mathbf{x},r)}\left(\frac{|f(\mathbf{y})-\ell(\mathbf{y})|}{r}\right)^{2}\, d\cH_{\cP}^k(\mathbf{y}),
\end{align*}
where the infimum is taken over all affine functions $\ell \colon V \to\mathbb{R}$ satisfying $\partial_{t}\ell\equiv0$. 

The following proposition shows that a Lipschitz function whose $\kappa^2$ satisfies a Carleson estimate has half-time derivative in BMO, as defined in \eqref{eq def of BMO}. The proof is modeled closely on \cite[proof of Theorem 1]{HLN}.
\begin{proposition} \label{prop:criterionforregularity}
    Suppose $V\in \operatorname{Gr}_{\mathcal{P}}(k)$ is vertical, $f\colon V \to\mathbb{R}$ is supported in $V\cap P(\mathbf{0},5)$, that $f$ has parabolic Lipschitz constant at most $\delta$ and
    \begin{align*}
        \int_{0}^{r} \int_{V\cap P(\mathbf{x},r)}\kappa^{2}(\mathbf{y},s)\,d\mathcal{H}_{\mathcal{P}}^{k}(\mathbf{y})\frac{ds}{s}<\delta r^{k}
    \end{align*}
    for all $\mathbf{x}\in V$ and all $r\in (0,100]$. Then $\Verts{\partial_{t}^{\frac{1}{2}}f}_{{\rm BMO}_{\mathcal{P}}(V)} \leq C\sqrt{\delta}$. 
\end{proposition}

\begin{proof} 
    By a rotation, we may assume that $V = (\{\mathbf{0}^2\}\times \mathbb{R}^{k-2}) \times \mathbb{R} \cong \mathbb{R}^{k-2} \times \mathbb{R}$. Using this identification, we lose no generality in assuming $n=k-2$. Given $g\colon \R^{k-2}\times \R\to \R$ and $\Omega\subset \R^{k-2}\times \R$, we write
    \begin{align*}
        g_\Omega\coloneqq \frac{1}{\cH_{\cP}^{k}(\Omega)}\int_\Omega g\, d\cH_{\cP}^{k}.
    \end{align*}
    
    Because the assumptions and conclusion of the theorem are invariant
    under parabolic rescaling, it suffices to show that if $Q\coloneqq P(\mathbf{0},1)$, then
    \begin{align}
        \int_{Q}|\partial_{t}^{\frac{1}{2}}f-(\partial_{t}^{\frac{1}{2}}f)_{Q}| \, d\cH_{\cP}^{k}\leq C\sqrt{\delta}.\label{eq reduction of the proof for bmo}
    \end{align}
    Fix $\varphi\in C^{\infty}(P(\mathbf{0},2))$ such that $\varphi|_{P(\mathbf{0},\frac{3}{2})}\equiv1$, $0\leq \varphi \leq 1$, and $\verts{\cd^\ell \varphi} + \verts{\pd_t^\ell \varphi}\leq c(\ell)$. Write
    \begin{align*}
        f=f_{1}+f_{2}\eqqcolon  \varphi(f-f_{Q})+\left((1-\varphi)f+ \varphi f_{Q}\right).
    \end{align*}
    
    We first prove \eqref{eq reduction of the proof for bmo} after mollifying $f_1$. We handle the general case by a limiting argument. To this end, fix a rotationally invariant function $\chi\in C_{c}^{\infty}(P(\mathbf{0},1))$ satisfying $\int_{\mathbb{R}^{k-2}\times \mathbb{R}}\chi \,dxdt=1$. For any $\epsilon\in(0,1)$,
    set $\chi_{\epsilon}(x,t)\coloneqq \epsilon^{-k}\chi(\epsilon^{-1}x,\epsilon^{-2}t)$. Letting $H$ denote the Hilbert transform, for any Schwartz functions $\psi_1,\psi_2 \in \mathcal{S}(\mathbb{R})$, we have the following integration-by-parts formula:
    \begin{align*}
        \int_{\mathbb{R}} (\partial_t^{\frac{1}{2}}\psi_1)(t)(\partial_t^{\frac{1}{2}}\psi_2)(t)dt &= \int_{\mathbb{R}}(|\tau|^{\frac{1}{2}}\widehat{\psi}_1(\tau))(|\tau|^{\frac{1}{2}}\widehat{\psi}_2(\tau))d\tau = \int_{\mathbb{R}} (\operatorname{sgn}(\tau)\widehat{\psi}_1(\tau))(\tau \widehat{\psi}_2(\tau))d\tau \\
        & = \int_{\mathbb{R}}H(\psi_1)(t)\partial_t \psi_2(t)dt,
    \end{align*}
    where $\widehat{\psi}$ denotes the Fourier transform of $\psi$. By taking $\psi_1 \coloneqq \partial_{\lambda}(\chi_{\lambda}\ast f_1)(x,\cdot)$ and $\psi_2 \coloneqq (\chi_{\lambda}\ast f_1)(\cdot)$, and then integrating in $x$ and $\lambda$, we obtain
    \begin{align*}
        \int_{\mathbb{R}^{k-2}\times \mathbb{R}}(\partial_{t}^{\frac{1}{2}}(\chi_{\epsilon}\ast f_{1}))^{2}\, d\cH_{\cP}^{k}= & \int_{\mathbb{R}^{k-2}\times \mathbb{R}}(\partial_{t}^{\frac{1}{2}}(\chi_{1}\ast f_{1}))^{2}\, d\cH_{\cP}^{k} \\
        &-2\int_{\epsilon}^{1}\int_{\mathbb{R}^{k-2}\times \mathbb{R}}H\left(\partial_{\lambda}(\chi_{\lambda}\ast f_{1})\right)\partial_{t}(\chi_{\lambda}\ast f)\, d\cH_{\cP}^{k}d\lambda\\
        = & I_{1}+I_{2}.
    \end{align*}
    
    For any $\mathbf{x} =(x,t)\in P(\mathbf{0},2)$, 
    \begin{equation} \label{eq:f1bounded}
    |f_{1}(\mathbf{x})|\leq|f(\mathbf{x})-f_{Q}|\leq C\left|\fint_{Q}(f(\mathbf{x})-f(\mathbf{y}))\,d\mathcal{H}_{\mathcal{P}}^{k}(\mathbf{y})\right|\leq C\delta,
    \end{equation}
    so that also $|\chi_{1}\ast f_1|\leq C\delta$. For any $\mathbf{x},\mathbf{y}\in \mathbb{R}^{k-2}\times \mathbb{R}$, we have 
    \begin{align*}
        |f_{1}(\mathbf{x})-f_{1}(\mathbf{y})|\leq & |f(\mathbf{x})-f(\mathbf{y})|\varphi(\mathbf{x})+|f(\mathbf{y})-f_{Q}|\cdot|\varphi(\mathbf{x})-\varphi(\mathbf{y})| \leq  C \delta |\mathbf{x}-\mathbf{y}|,
    \end{align*}
    
    \begin{claim} \label{claim:mollifiedlipschitz}
        For any $x\in \mathbb{R}^{k-2}$ and $s,t\in \mathbb{R}$, we have
        \begin{align*}
            |(\chi_{1}\ast f_{1})(x,t_1)-(\chi_{1}\ast f_{1})(x,t_2)|\leq C\delta |t_1-t_2|.
        \end{align*}
    \end{claim}
    \begin{proof} 
        If $|t_1 - t_2| \geq \frac{1}{4}$, then 
        \begin{align*}
            |(\chi_1 \ast f_1)(x,t)-(\chi_1 \ast f_1)(x,s)|\leq C\delta \sqrt{|t_1-t_2|} \leq C\delta |t_1-t_2|.
        \end{align*}
        Suppose instead $|t_1-t_2|\leq \frac{1}{4}$. For any $t\in [t_1,t_2]$, 
        \begin{align*}
        \left| \frac{d}{dt}(f_1\ast \chi_{1})(\mathbf{x}) \right|&= \left| \int_{\mathbb{R}} \partial_t \chi(\mathbf{x}-\mathbf{y})(f_1(y,s)-f_1(y,t_1))\, d\cH_{\cP}^{k}(\mathbf{y})\right| \\ & \leq C\delta,
        \end{align*}
        where we used that $\operatorname{supp}(\chi(x-y,t-\cdot))\subseteq [t_1-2,t_1+2]$. Integration from $t_1$ to $t_2$ yields the claim.
    \end{proof}
    
    Combining \eqref{eq:f1bounded} and Claim \ref{claim:mollifiedlipschitz} yields
    \begin{align}
        |\partial_{t}^{\frac{1}{2}}(\chi_{1}\ast f_{1})|(x,t)=c\left|\int_{\mathbb{R}}\frac{(\chi_{1}\ast f_{1})(x,t)-(\chi_{1}\ast f_{1})(x,s)}{|t-s|^{\frac{3}{2}}}\,ds\right|\leq  C\delta\int_{\mathbb{R}}\frac{\min\{|t-s|,1\}}{|t-s|^{\frac{3}{2}}}\,ds\leq C\delta.\label{eq half-derivative inside}
    \end{align}
    If in addition we have $|t|\geq100$, then because $\text{supp}(\chi_{1}\ast f_{1})\subseteq P(\mathbf{0},4)$,
    we obtain
    \begin{align}
        |\partial_{t}^{\frac{1}{2}}(\chi_{1}\ast f_{1})|(x,t)\leq\frac{C\delta}{|t|^{\frac{3}{2}}}.\label{eq half-derivative outside}
    \end{align}
    Combining estimates \eqref{eq half-derivative inside} and \eqref{eq half-derivative outside} gives
    \begin{align*}
        |\partial_{t}^{\frac{1}{2}}(\chi_{1}\ast f_{1})|(x,t)\leq\frac{C\delta}{(1+|t|)^{\frac{3}{2}}}
    \end{align*}
    for all $(x,t)\in \mathbb{R}^{k-2}\times \mathbb{R}$, so that
    \begin{align*}
        I_{1}=\int_{Q}|\partial_{t}^{\frac{1}{2}}(\chi_{1}\ast f_{1})|^{2}\,d\mathcal{H}_{\mathcal{P}}^{k}\leq C\delta.
    \end{align*}
    Next, we use the $L^{2}$-boundedness of the Hilbert transform to
    get
    \begin{align*}
        \int_{\epsilon}^{1}\lambda^{-1}\int_{\mathbb{R}^{k-2}\times \mathbb{R}}\left|H\left(\partial_{\lambda}(\chi_{\lambda}\ast f_{1})\right)\right|^{2}\, d\cH_{\cP}^{k}d\lambda\leq C\int_{\epsilon}^{1}\lambda^{-1}\int_{\mathbb{R}^{k-2}\times \mathbb{R}}|\partial_{\lambda}(\chi_{\lambda}\ast f_{1})|^{2}\, d\cH_{\cP}^{k}d\lambda,
    \end{align*}
    hence
    \begin{align*}
        |I_{2}|\leq & 2\int_{\epsilon}^{1}\int_{\mathbb{R}^{k-2}\times \mathbb{R}}H\left(\partial_{\lambda}(\chi_{\lambda}\ast f_{1})\right)\partial_{t}(\chi_{\lambda}\ast f_{1})\, d\cH_{\cP}^{k}d\lambda\\
        \leq & C\left(\int_{\epsilon}^{1}\lambda^{-1}\int_{\mathbb{R}^{k-2}\times \mathbb{R}}|\partial_{\lambda}(\chi_{\lambda}\ast f_{1})|^{2}\, d\cH_{\cP}^{k}d\lambda\right)^{\frac{1}{2}}\left(\int_{\epsilon}^{1}\lambda\int_{\mathbb{R}^{k-2}\times \mathbb{R}}|\partial_{t}(\chi_{\lambda}\ast f_{1})|^{2}\, d\cH_{\cP}^{k}d\lambda\right)^{\frac{1}{2}}.
    \end{align*}
    For any $\mathbf{x}=(x,t)\in \mathbb{R}^{k-2}\times \mathbb{R}$ and $\mathbf{y}=(y,s)\in \mathbb{R}^{k-2}\times \mathbb{R}$ we have
    \begin{align*}
        \verts*{\varphi(\mathbf{y})-\varphi(\mathbf{x})-\langle\nabla\varphi(\mathbf{x}),y-x\rangle} &\leq\verts*{\int_{t}^{s}\partial_{\tau}\varphi(y,\tau)\,d\tau}+\verts{(\varphi(y,t)-\varphi(\mathbf{x}))-\nabla\varphi(\mathbf{x}) \cdot (y-x)} \\
        &\leq C|\mathbf{x}-\mathbf{y}|^{2}.
    \end{align*}
    Because $\int_{\mathbb{R}^{k-2}\times \mathbb{R}}\partial_{t}\chi_{\lambda}(\mathbf{y})\, d\cH_{\cP}^{k} (\mathbf{y})=0$ and $\int_{\mathbb{R}^{k-2}\times \mathbb{R}}\partial_{t}\chi_{\lambda}(\mathbf{y})y\, d\cH_{\cP}^{k}(\mathbf{y})=0$ (the second identity follows from the rotational invariance of $\chi_{\lambda}$),
    this implies
    \begin{align*}
        \partial_{t}(\chi_{\lambda}\ast f_{1})(\mathbf{x})= & \int_{\mathbb{R}^{k-2}\times \mathbb{R}}\partial_{t}\chi_{\lambda}(\mathbf{x}-\mathbf{y})(f(\mathbf{y})-f_{Q})\varphi(\mathbf{y})\,d\mathcal{H}_{\mathcal{P}}^k(\mathbf{y})\\
        = & \int_{\mathbb{R}^{k-2}\times \mathbb{R}}\partial_{t}\chi_{\lambda}(\mathbf{x}-\mathbf{y})(f(\mathbf{y})-f_{Q})\left(\varphi(\mathbf{y})-\varphi(\mathbf{x})-\nabla\varphi(\mathbf{x})\cdot(y-x)\right)d\mathcal{H}_{\mathcal{P}}^k(\mathbf{y})\\
        & +\varphi(\mathbf{x})\int_{\mathbb{R}^{k-2}\times \mathbb{R}}\partial_{t}\chi_{\lambda}(\mathbf{x}-\mathbf{y})f(\mathbf{y})\,d\mathcal{H}_{\mathcal{P}}^k(\mathbf{y})\\
        &+\nabla\varphi(\mathbf{x}) \cdot\int_{\mathbb{R}^{k-2}\times \mathbb{R}}\partial_{t}\chi_{\lambda}(\mathbf{x}-\mathbf{y})(f(\mathbf{y})-f(\mathbf{x}))(\mathbf{y}-\mathbf{x})\,d\mathcal{H}^k(\mathbf{y}),
    \end{align*}
    hence
    \begin{align*}
        |\partial_{t}(\chi_{\lambda}\ast f_{1})-\varphi\partial_{t}(\chi_{\lambda}\ast f)|(\mathbf{x})\leq C\delta\lambda^{-(n+3)}\int_{P(\mathbf{x},\lambda)}|\mathbf{x}-\mathbf{y}|^{2}\, d\cH_{\cP}^{k}(\mathbf{y})\leq C\delta.
    \end{align*}
    Similarly, because $\int_{\mathbb{R}^{k-2}\times \mathbb{R}}\partial_{\lambda}\chi_{\lambda}(\mathbf{y})\,d\mathcal{H}_{\mathcal{P}}^k(\mathbf{y})=0$
    and $\int_{\mathbb{R}^{k-2}\times \mathbb{R}}\partial_{\lambda}\chi_{\lambda}(\mathbf{y})y\,d\mathcal{H}_{\mathcal{P}}^k(\mathbf{y})=0$ (where the second identity again follows by rotational invariance), we
    have
    \begin{align*}
        \partial_{\lambda}(\chi_{\lambda}\ast f_{1})(\mathbf{x})= & \int_{\mathbb{R}^{k-2}\times \mathbb{R}}\partial_{\lambda}\chi_{\lambda}(\mathbf{x}-\mathbf{y})(f(\mathbf{y})-f_{Q})\left(\varphi(\mathbf{y})-\varphi(\mathbf{x})-\nabla\varphi(\mathbf{x})\cdot (y-x)\right)\, d\cH_{\cP}^k(\mathbf{y})\\
         & +\varphi(\mathbf{x})\int_{\mathbb{R}^{k-2}\times \mathbb{R}}\partial_{\lambda}\chi_{\lambda}(\mathbf{x}-\mathbf{y})f(\mathbf{y})\, d\cH_{\cP}^k(\mathbf{y})\\
         & +\nabla\varphi(\mathbf{x}) \cdot \int_{\mathbb{R}^{k-2}\times \mathbb{R}}\partial_{\lambda}\chi_{\lambda}(\mathbf{x}-\mathbf{y})(f(\mathbf{y})-f(\mathbf{x}))(y-x)\,d\mathcal{H}_{\mathcal{P}}^k(\mathbf{y}).
    \end{align*}
    Using 
    \begin{align*}
        |\partial_{\lambda}\chi_{\lambda}(x-y,t-s)|\leq & \verts{-(n+1)\lambda^{-(n+2)}\chi(\lambda^{-1}(x-y),\lambda^{-2}(t-s))} \\&+|\lambda^{-(n+3)}\langle\nabla\chi(\lambda^{-1}(x-y),\lambda^{-2}(t-s)),x-y\rangle|\\
         & +|t-s|\cdot|\lambda^{-(n+4)}\partial_{t}\chi(\lambda^{-1}(x-y),\lambda^{-2}(t-s))|\\
        \leq & C\lambda^{-(n+2)},
    \end{align*}
    we get
    \begin{equation} \label{eq:I2estimate}
        |\partial_{\lambda}(\chi_{\lambda}\ast f_{1})-\varphi\partial_{\lambda}(\chi_{\lambda}\ast f)|\leq C\delta\lambda.
    \end{equation}
    Combining expressions, we thus obtain 
    \begin{align*}
        |I_{2}|&\leq C\left(\int_{\epsilon}^{1}\lambda^{-1}\int_{(\mathbb{R}^{k-2}\times \mathbb{R})\cap P(\mathbf{0},10)}|\partial_{\lambda}(\chi_{\lambda}\ast f)|^{2}\, d\cH_{\cP}^{k}d\lambda+C\delta\right)^{\frac{1}{2}}\times\\
        &\quad \quad\left(\int_{\epsilon}^{1}\lambda\int_{(\mathbb{R}^{k-2}\times \mathbb{R})\cap P(\mathbf{0},10)}|\partial_{t}(\chi_{\lambda}\ast f)|^{2}\, d\cH_{\cP}^{k}d\lambda+C\delta\right)^{\frac{1}{2}}.
    \end{align*}
    
    \begin{claim} \label{claim:applicationofbeta}
        The following holds:
        \begin{align*}
            &\int_{0}^{1}\lambda^{-1}\int_{ P(\mathbf{0},10)}|\partial_{\lambda}(\chi_{\lambda}\ast f)|^{2}\,d\cH_{\cP}^{k}d\lambda+\int_{0}^{1}\lambda\int_{P(\mathbf{0},10)}|\partial_{t}(\chi_{\lambda}\ast f)|^{2}\,d\cH_{\cP}^{k}d\lambda\\
            &\leq C\int_{0}^1 \int_{P(\mathbf{0},10)} \kappa^2(\mathbf{y},\lambda) \, d\cH_{\cP}^{k}\frac{d\lambda}{\lambda}<C\delta.
        \end{align*}
    \end{claim}
    \begin{proof} 
        For any $\theta\in C_{c}^{\infty}(P(\mathbf{0},\lambda))$
        satisfying $\int_{\R^{k-2}\times \mathbb{R}}\theta(\mathbf{y})\,d\cH_{\cP}^{k}(\mathbf{y})=0$ and $\int_{\R^{k-2}\times \mathbb{R}}\theta(\mathbf{y})y\,d\cH_{\cP}^{k}(\mathbf{y})=0$,
        and for any affine function $\ell:\R^{k-2}\times \R\to\mathbb{R}$ with $\partial_{t}\ell=0$,
        we have
        \begin{align*}
            (\theta *  \ell)(\mathbf{x})=\int_{\R^n}\theta(\mathbf{y})\ell(\mathbf{x}-\mathbf{y})\,d\cH_{\cP}^{k}(\mathbf{y})=0
        \end{align*}
        for all $\mathbf{x}\in \R^{k-2}\times \R$. Thus
        \begin{align*}
            (\theta\ast f)^{2}(\mathbf{x})= & \left(\theta\ast(f-\ell)\right)^{2}(\mathbf{x})=\left(\int_{\R^n}\theta(\mathbf{x}-\mathbf{y})(f(\mathbf{y})-\ell(\mathbf{y}))\,d\cH_{\cP}^{k}(\mathbf{y})\right)^{2}\\
            \leq & C\lambda^{2k}\Verts{\theta}_{\infty}^{2}\left(\frac{1}{\lambda^{k}}\int_{P(\mathbf{x},\lambda)}(f(\mathbf{y})-\ell(\mathbf{y}))\,d\cH_{\cP}^{k}(\mathbf{y})\right)^{2}\\
            \leq & C\lambda^{2k}\Verts{\theta}_{\infty}^{2}\frac{1}{\lambda^{k}}\int_{P(\mathbf{x},\lambda)}(f(\mathbf{y})-\ell(\mathbf{y}))^{2}\,d\cH_{\cP}^{k}(\mathbf{y}).
        \end{align*}
        Taking infimum over $\ell$, we obtain
        \begin{align*}
            (\theta\ast f)^{2}(\mathbf{x})\leq C\lambda^{2k+2}\Verts{\theta}_{\infty}^{2}\kappa^{2}(\mathbf{x},\lambda).
        \end{align*}
        
        Now set $\theta\coloneqq \partial_{\lambda}\chi_{\lambda}$. Recalling that $\Verts{\partial_{\lambda}\chi_{\lambda}}_{\infty}\leq C\lambda^{-(k+1)}$, we integrate in spacetime and then against $\frac{d\lambda}{\lambda}$ to obtain
        \begin{align*}
            \int_{0}^{1}\int_{P(\mathbf{0},10)}((\partial_{\lambda}\chi_{\lambda})\ast f)^{2}(\mathbf{y})\,d\cH_{\cP}^{k}(\mathbf{y})\frac{d\lambda}{\lambda}\leq & C\int_{0}^{1}\int_{ P(\mathbf{0},10)}\lambda^{2k+2}\lambda^{-2(k+1)}\kappa^{2}(\mathbf{y},\lambda)\,d\cH_{\cP}^{k}(\mathbf{y})\frac{d\lambda}{\lambda}.
        \end{align*}
        Similarly, we can take $\theta\coloneqq \partial_{t}\chi_{\lambda}$ and recall that $\Verts{\partial_{t}\chi_{t}}_{\infty}\leq C\lambda^{-(k+2)}$ to obtain
        \begin{align*}
            \int_{0}^{1}\int_{ P(\mathbf{0},10)}\lambda((\partial_{t}\chi_{\lambda})\ast f)^{2}(\mathbf{y})\,d\cH_{\cP}^{k}(\mathbf{y})d\lambda\leq & C\int_{0}^{1}\int_{ P(\mathbf{0},10)}\lambda^{-2(k+2)}\lambda^{2k+4}\kappa^{2}(\mathbf{y},\lambda)\,d\cH_{\cP}^{k}(\mathbf{y})\frac{d\lambda}{\lambda}.
        \end{align*}
    \end{proof}
    Combining Claim \ref{claim:applicationofbeta} with \eqref{eq:I2estimate} yields $|I_{2}|\leq C\delta$,
    and further combining this with the estimate $|I_{1}|\leq C\delta$
    gives
    \begin{align*}
        \int_{\R^{k-2}\times \R}(\partial_{t}^{\frac{1}{2}}(\chi_{\epsilon}\ast f_{1}))^{2}(\mathbf{x})\,d\cH_{\cP}^{k}(\mathbf{x})\leq C\delta
    \end{align*}
    for all $\epsilon\in(0,1]$. We can thus find a sequence $\epsilon_{j}\searrow0$
    such that 
    \begin{align*}
        \partial_{t}^{\frac{1}{2}}(\chi_{\epsilon_{j}}\ast f)\rightharpoonup\zeta
    \end{align*}
    weakly in $L^{2}(\R^{k-2}\times \R)$ to some $\zeta\in L^{2}(\R^{k-2}\times \R)$ as $j\to\infty$, which satisfies $\int_{\R^{k-2}\times \R}\zeta^{2}\,d\cH_{\cP}^{k}\leq C\delta$. Define $I_{\frac{1}{2}}(s)\coloneqq |s|^{-\frac{1}{2}}$
    for $s\neq 0$. Then the Fourier transform of $I_{\frac{1}{2}}$ is $\cF{I}_{\frac{1}{2}}(\tau)=\frac{1}{c}|\tau|^{-\frac{1}{2}}$. For any Schwartz function $g$ on $\mathbb{R}$, we have
    \begin{align*}
        c\cF\left(I_{\frac{1}{2}}\ast\partial_{t}^{\frac{1}{2}}g\right)(\tau)=c\cF {I}_{\frac{1}{2}}(\tau)|\tau|^{\frac{1}{2}}\cF g(\tau)=\cF g(\tau),
    \end{align*}
    so that $cI_{\frac{1}{2}}\ast(\partial_{t}^{\frac{1}{2}}g)=g$. Applying this identity to $g\coloneqq (\chi_{\epsilon_{j}}\ast f_{1})(x,\cdot)$ for $x\in \mathbb{R}^{k-2}$
    yields $cI_{\frac{1}{2}}\ast\left(\partial_{t}^{\frac{1}{2}}(\chi_{\epsilon}\ast f)\right)=\chi_{\epsilon}\ast f$.
    For any Schwartz function $g \in \mathscr{S}(\mathbb{R}^{k-2}\times \mathbb{R})$, 
    \begin{align*}
        \int_{\R^{k-2}\times \R}g(\mathbf{x})(\chi_{\epsilon}\ast f_{1})(\mathbf{x})\, d\cH_{\cP}^{k}(\mathbf{x})= & c\int_{\mathbb{R}^{k-2}\times \R}g(x,\cdot)\left(I_{\frac{1}{2}}\ast\left(\partial_{t}^{\frac{1}{2}}(\chi_{\epsilon}\ast f_{1})(x,\cdot)\right)\right)(t)\, d\cH_{\cP}^{k}(\mathbf{x})\\
        = & c\int_{\mathbb{R}^{k-2}\times \R}(I_{\frac{1}{2}}\ast g (x,\cdot))(t)\partial_{t}^{\frac{1}{2}}(\chi_{\epsilon}\ast f_{1})(x,\cdot)\, d\cH_{\cP}^{k}(\mathbf{x}),
    \end{align*}
    so that taking $\epsilon\to0$ and using the fact that $\chi_{\epsilon}\ast f_{1}\to f_{1}$
    in $L^{2}(\mathbb{R}^{k-2} \times \mathbb{R})$, we obtain
    \begin{align*}
        \int_{\mathbb{R}^{k-2}\times \mathbb{R}}g(\mathbf{x})f_{1}(\mathbf{x})\,d\mathcal{H}_{\mathcal{P}}^k(\mathbf{x})&=c\int_{\mathbb{R}^{k-2} \times \mathbb{R}}(I_{\frac{1}{2}}\ast g(x,\cdot))(t)\zeta(\mathbf{x})\,d\mathcal{H}_{\mathcal{P}}^k(\mathbf{x})\\
        &=c\int_{\mathbb{R}^{k-2}}g(\mathbf{x})(I_{\frac{1}{2}}\ast\zeta(x,\cdot))(t)\,d\mathcal{H}_{\mathcal{P}}^k(\mathbf{x}).
    \end{align*}
    That is, $f_{1}(x,\cdot)=c(I_{\frac{1}{2}}\ast\zeta(x,\cdot))$. A computation similar to the above yields $\partial_{t}^{\frac{1}{2}}(I_{\frac{1}{2}}\ast\zeta)=\zeta$,
    so that $\partial_{t}^{\frac{1}{2}}f_{1}=\zeta$ exists and satisfies
    \begin{align*}
        \int_{\mathbb{R}^{k-2}\times \mathbb{R}}(\partial_{t}^{\frac{1}{2}}f_{1})^{2}(\mathbf{x})\, d\cH_{\cP}^{k}(\mathbf{x})\leq C\delta.
    \end{align*}
    
    We now handle $f_{2}$. Because $f_{2}$ is constant on $\mathbb{R}^{k-2}\times \mathbb{R}\cap P(\mathbf{0},\frac{3}{2})$,
    for any $\mathbf{x}=(x,t)\in Q$, we have
    \begin{align*}
        (\partial_{t}^{\frac{1}{2}}f_{2})(\mathbf{x})= & c\int_{|s|\geq\frac{9}{4}}\frac{f_{2}(x,s)-f_{2}(x,t)}{|t-s|^{\frac{3}{2}}}ds.
    \end{align*}
    Set
    \begin{align*}
        \alpha\coloneqq c\int_{|s|\geq\frac{9}{4}}\frac{f(0,s)-f(0,0)}{|s|^{\frac{3}{2}}}ds,
    \end{align*}
    and recall that $|f-f_{Q}|\leq C\delta$ on $\mathbb{R}^{k-2}\times \mathbb{R}\cap P(\mathbf{0},2)$,
    so that using
    \begin{align*}
        f_{2}(\mathbf{x})=\chi f_{Q}+(1-\chi)f=\chi(f_{Q}-f)+f,
    \end{align*}
    we obtain
    \begin{align*}
        |\partial_{t}^{\frac{1}{2}}f_{2}(\mathbf{x})-\alpha|\leq & c\int_{|s|\geq\frac{9}{4}}\left|\frac{f(x,s)-f(x,t)}{|t-s|^{\frac{3}{2}}}-\frac{f(0,s)-f(0,0)}{|s|^{\frac{3}{2}}}\right|ds\\
        & +C\left(\sup_{\mathbb{R}^{k-2}\times \mathbb{R}\cap P(\mathbf{0},2)}|f-f_{Q}|\right)\int_{|s|\geq\frac{9}{4}}\frac{1}{|s|^{\frac{3}{2}}}ds\\
        \le & C\int_{|s|\geq\frac{9}{4}}\left|\frac{f(x,s)-f_{Q}}{|t-s|^{\frac{3}{2}}}-\frac{f(0,s)-f_{Q}}{|s|^{\frac{3}{2}}}\right|ds+C\delta\\
        \leq & C\int_{|s|\geq\frac{9}{4}}\left|\frac{f(x,s)-f(0,s)}{|t-s|^{\frac{3}{2}}}+(f(0,s)-f_{Q})\left(\frac{1}{|t-s|^{\frac{3}{2}}}-\frac{1}{|s|^{\frac{3}{2}}}\right)\right|ds+C\delta\\
        \leq & C\delta\int_{|s|\geq\frac{9}{4}}|s|\frac{|t|}{|s|^{\frac{5}{2}}}ds+C\delta\\
        \leq & C\delta.
    \end{align*}
    Combining estimates yields that $\partial_{t}^{\frac{1}{2}}f$ exists
    and is integrable on $Q$, and 
    \begin{align*}
        \int_{Q}|\partial_{t}^{\frac{1}{2}}f(\mathbf{x})-\alpha|\, d\cH_{\cP}^{k}(\mathbf{x})\leq & \int_{Q}|\partial_{t}^{\frac{1}{2}}f_{1}(\mathbf{x})|\, d\cH_{\cP}^{k}(\mathbf{x})+\int_{Q}|\partial_{t}^{\frac{1}{2}}f_{2}(\mathbf{x})-\alpha|\, d\cH_{\cP}^{k}(\mathbf{x})\leq  C\sqrt{\delta},
    \end{align*}
    so the claim follows. 
\end{proof}


\bibliographystyle{amsalpha}
\bibliography{references}

@article{angenent1991nodal,
  title={Nodal properties of solutions of parabolic equations},
  author={Angenent, Sigurd},
  journal={The Rocky Mountain Journal of Mathematics},
  volume={21},
  number={2},
  pages={585--592},
  year={1991},
  publisher={JSTOR}
}

@article{arya2023space,
  title={Space-like quantitative uniqueness for parabolic operators},
  author={Arya, Vedansh and Banerjee, Agnid},
  journal={Journal de Math{\'e}matiques Pures et Appliqu{\'e}es},
  volume={177},
  pages={214--259},
  year={2023},
  publisher={Elsevier}
}

@article{arya2025sharp,
  title={Sharp Order of Vanishing for Parabolic Equations, Nodal Set Estimates and Landis Type Results},
  author={Arya, Vedansh and Banerjee, Agnid and Garofalo, Nicola},
  journal={Archive for Rational Mechanics and Analysis},
  volume={249},
  number={6},
  pages={67},
  year={2025},
  publisher={Springer}
}

@article{aubin-1963theoreme,
  title={Un th{\'e}oreme de compacit{\'e}},
  author={Aubin, Jean-Pierre},
  journal={CR Acad. Sci. Paris},
  volume={256},
  number={24},
  pages={5042--5044},
  year={1963}
}

@book {BakLed,
    AUTHOR = {Bakry, Dominique and Gentil, Ivan and Ledoux, Michel},
     TITLE = {Analysis and geometry of {M}arkov diffusion operators},
    SERIES = {Grundlehren der mathematischen Wissenschaften [Fundamental
              Principles of Mathematical Sciences]},
    VOLUME = {348},
 PUBLISHER = {Springer, Cham},
      YEAR = {2014},
     PAGES = {xx+552},
      ISBN = {978-3-319-00226-2; 978-3-319-00227-9},
   MRCLASS = {60J25 (58J65 60J35 60J60)},
  MRNUMBER = {3155209},
MRREVIEWER = {Ming\ Liao},
       DOI = {10.1007/978-3-319-00227-9},
       URL = {https://doi-org.proxy.libraries.rutgers.edu/10.1007/978-3-319-00227-9},
}

@book{boyer-fabrie-2012-mathematical,
  title={{Mathematical Tools for the Study of the Incompressible Navier-Stokes Equations and Related Models}},
  author={Boyer, Franck and Fabrie, Pierre},
  volume={183},
  year={2012},
  publisher={Springer Science \& Business Media}
}

@article {cheeger-jiang-naber-2021-Sharp-quantitative,
    AUTHOR = {Cheeger, Jeff and Jiang, Wenshuai and Naber, Aaron},
     TITLE = {Rectifiability of singular sets of noncollapsed limit spaces
              with {R}icci curvature bounded below},
   JOURNAL = {Ann. of Math. (2)},
  FJOURNAL = {Annals of Mathematics. Second Series},
    VOLUME = {193},
      YEAR = {2021},
    NUMBER = {2},
     PAGES = {407--538},
      ISSN = {0003-486X,1939-8980},
   MRCLASS = {53B20 (35A21 53C23)},
  MRNUMBER = {4226910},
MRREVIEWER = {Daniele\ Semola},
       DOI = {10.4007/annals.2021.193.2.2},
       URL = {https://doi.org/10.4007/annals.2021.193.2.2},
}

@article{cheeger-naber-2013-lower,
  title={{Lower bounds on Ricci curvature and quantitative behavior of singular sets}},
  author={Cheeger, Jeff and Naber, Aaron},
  journal={Inventiones mathematicae},
  volume={191},
  number={2},
  pages={321--339},
  year={2013},
  publisher={Springer}
}

@article{cheeger-naber-valtorta-2015-critical,
  title = {{Critical Sets of Elliptic Equations}},
  author = {Cheeger, Jeff and Naber, Aaron and Valtorta, Daniele},
  year = {2015},
  month = feb,
  journal = {Communications on Pure and Applied Mathematics},
  volume = {68},
  number = {2},
  pages = {173--209},
  issn = {00103640},
  doi = {10.1002/cpa.21518}
}

@article {cheeger-tian-1994-cone-structure-at-infinity,
    AUTHOR = {Cheeger, Jeff and Tian, Gang},
     TITLE = {On the cone structure at infinity of {R}icci flat manifolds
              with {E}uclidean volume growth and quadratic curvature decay},
   JOURNAL = {Invent. Math.},
  FJOURNAL = {Inventiones Mathematicae},
    VOLUME = {118},
      YEAR = {1994},
    NUMBER = {3},
     PAGES = {493--571},
      ISSN = {0020-9910,1432-1297},
   MRCLASS = {53C21 (53C25 53C55)},
  MRNUMBER = {1296356},
MRREVIEWER = {Xiao\ Wei\ Peng},
       DOI = {10.1007/BF01231543},
       URL = {https://doi.org/10.1007/BF01231543},
}

@article {colding-mninicozzi-2011-lower,
    AUTHOR = {Colding, Tobias H. and Minicozzi, II, William P.},
     TITLE = {Lower bounds for nodal sets of eigenfunctions},
   JOURNAL = {Comm. Math. Phys.},
  FJOURNAL = {Communications in Mathematical Physics},
    VOLUME = {306},
      YEAR = {2011},
    NUMBER = {3},
     PAGES = {777--784},
      ISSN = {0010-3616,1432-0916},
   MRCLASS = {58J50 (28A78 35P15 35P20)},
  MRNUMBER = {2825508},
MRREVIEWER = {Julie\ Rowlett},
       DOI = {10.1007/s00220-011-1225-x},
       URL = {https://doi.org/10.1007/s00220-011-1225-x},
}

@misc{colding-minicozzi-2020-parabolic,
  title = {{Parabolic Frequency on Manifolds}},
  author = {Colding, Tobias Holck and Minicozzi II, William P.},
  year = {2020},
  month = feb,
  number = {arXiv:2002.11015},
  eprint = {2002.11015},
  primaryclass = {math},
  publisher = {{arXiv}},
  archiveprefix = {arxiv}
}

@article {DavidSemmes,
    AUTHOR = {David, G. and Semmes, S.},
     TITLE = {Singular integrals and rectifiable sets in {${\bf R}^n$}:
              {B}eyond {L}ipschitz graphs},
   JOURNAL = {Ast\'erisque},
  FJOURNAL = {Ast\'erisque},
    NUMBER = {193},
      YEAR = {1991},
     PAGES = {152},
      ISSN = {0303-1179,2492-5926},
   MRCLASS = {42B20 (42B25)},
  MRNUMBER = {1113517},
MRREVIEWER = {Stephen\ Buckley},
}

@article {delsanto,
    AUTHOR = {Del Santo, Daniele and Prizzi, Martino},
     TITLE = {Backward uniqueness for parabolic operators whose coefficients
              are non-{L}ipschitz continuous in time},
   JOURNAL = {J. Math. Pures Appl. (9)},
  FJOURNAL = {Journal de Math\'ematiques Pures et Appliqu\'ees. Neuvi\`eme
              S\'erie},
    VOLUME = {84},
      YEAR = {2005},
    NUMBER = {4},
     PAGES = {471--491},
      ISSN = {0021-7824},
   MRCLASS = {35K15},
  MRNUMBER = {2133125},
MRREVIEWER = {Sergey\ G.\ Pyatkov},
       DOI = {10.1016/j.matpur.2004.09.004},
       URL = {https://doi-org.proxy.libraries.rutgers.edu/10.1016/j.matpur.2004.09.004},
}

@article {dong-1992-nodal,
    AUTHOR = {Dong, Rui-Tao},
     TITLE = {Nodal sets of eigenfunctions on {R}iemann surfaces},
   JOURNAL = {J. Differential Geom.},
  FJOURNAL = {Journal of Differential Geometry},
    VOLUME = {36},
      YEAR = {1992},
    NUMBER = {2},
     PAGES = {493--506},
      ISSN = {0022-040X,1945-743X},
   MRCLASS = {58G25 (35P99)},
  MRNUMBER = {1180391},
MRREVIEWER = {Stig\ I.\ Andersson},
       URL = {http://projecteuclid.org/euclid.jdg/1214448750},
}

@article {donnelly-fefferman-1988-nodal-sets,
    AUTHOR = {Donnelly, Harold and Fefferman, Charles},
     TITLE = {Nodal sets of eigenfunctions on {R}iemannian manifolds},
   JOURNAL = {Invent. Math.},
  FJOURNAL = {Inventiones Mathematicae},
    VOLUME = {93},
      YEAR = {1988},
    NUMBER = {1},
     PAGES = {161--183},
      ISSN = {0020-9910,1432-1297},
   MRCLASS = {58G25 (35B60 35P05)},
  MRNUMBER = {943927},
MRREVIEWER = {P.\ G\"unther},
       DOI = {10.1007/BF01393691},
       URL = {https://doi.org/10.1007/BF01393691},
}

@article {donnelly-fefferman-1990-nodal,
    AUTHOR = {Donnelly, Harold and Fefferman, Charles},
     TITLE = {Nodal sets for eigenfunctions of the {L}aplacian on surfaces},
   JOURNAL = {J. Amer. Math. Soc.},
  FJOURNAL = {Journal of the American Mathematical Society},
    VOLUME = {3},
      YEAR = {1990},
    NUMBER = {2},
     PAGES = {333--353},
      ISSN = {0894-0347,1088-6834},
   MRCLASS = {58G25 (35P05)},
  MRNUMBER = {1035413},
MRREVIEWER = {H.-B.\ Rademacher},
       DOI = {10.2307/1990956},
       URL = {https://doi.org/10.2307/1990956},
}

@article{escauriaza2006doubling,
  title={Doubling properties of caloric functions},
  author={Escauriaza, Luis and Fern{\'a}ndez, FJ and Vessella, Sergio},
  journal={Applicable analysis},
  volume={85},
  number={1-3},
  pages={205--223},
  year={2006},
  publisher={Taylor \& Francis}
}

@book {evans-2010-pde,
    AUTHOR = {Evans, Lawrence C.},
     TITLE = {Partial differential equations},
    SERIES = {Graduate Studies in Mathematics},
    VOLUME = {19},
   EDITION = {Second},
 PUBLISHER = {American Mathematical Society, Providence, RI},
      YEAR = {2010},
     PAGES = {xxii+749},
      ISBN = {978-0-8218-4974-3},
   MRCLASS = {35-01},
  MRNUMBER = {2597943},
MRREVIEWER = {Diego\ M.\ Maldonado},
       DOI = {10.1090/gsm/019},
       URL = {https://doi.org/10.1090/gsm/019},
}

@article{fang-li-2025volume,
  title={Volume estimates for the singular sets of mean curvature flows},
  author={Fang, Hanbing and Li, Yu},
  journal={arXiv preprint arXiv:2504.09811},
  year={2025},
  pages={1-32}
}

@misc{fang-li-RF,
      title={Singular sets in noncollapsed Ricci flow limit spaces}, 
      author={Hanbing Fang and Yu Li},
      year={2025},
      eprint={2510.26317},
      archivePrefix={arXiv},
      primaryClass={math.DG},
      url={https://arxiv.org/abs/2510.26317}, 
}

@article {federer-1970-singular-sets,
    AUTHOR = {Federer, Herbert},
     TITLE = {The singular sets of area minimizing rectifiable currents with
              codimension one and of area minimizing flat chains modulo two
              with arbitrary codimension},
   JOURNAL = {Bull. Amer. Math. Soc.},
  FJOURNAL = {Bulletin of the American Mathematical Society},
    VOLUME = {76},
      YEAR = {1970},
     PAGES = {767--771},
      ISSN = {0002-9904},
   MRCLASS = {28.80 (26.00)},
  MRNUMBER = {260981},
MRREVIEWER = {J.\ E.\ Brothers},
       DOI = {10.1090/S0002-9904-1970-12542-3},
       URL = {https://doi.org/10.1090/S0002-9904-1970-12542-3},
}

@book{federer-2014-geometric-measure-theory,
  title={Geometric measure theory},
  author={Federer, Herbert},
  year={2014},
  publisher={Springer}
}

@article{fu2025stratification,
  title={Stratification and Rectifiability of Harmonic Map Flows via Tangent Measures},
  author={Fu, Haotong and Wang, Wei and Wu, Ke and Zhang, Zhifei},
  journal={arXiv preprint arXiv:2504.14880},
  pages={1-41},
  year={2025}
}

@misc{gianniotis-diameter,
      title={Diameter bounds in 3d Type I Ricci flows}, 
      author={Panagiotis Gianniotis},
      year={2025},
      eprint={2510.14019},
      archivePrefix={arXiv},
      primaryClass={math.DG},
      url={https://arxiv.org/abs/2510.14019}, 
}

@misc{gianniotis-L1,
      title={$L^1$ curvature bounds for Type I Ricci flows}, 
      author={Panagiotis Gianniotis and Konstantinos Leskas},
      year={2025},
      eprint={2510.22660},
      archivePrefix={arXiv},
      primaryClass={math.DG},
      url={https://arxiv.org/abs/2510.22660}, 
}

@article{gross-1975-log-sobolev,
    AUTHOR = {Gross, Leonard},
     TITLE = {Logarithmic {S}obolev inequalities},
   JOURNAL = {Amer. J. Math.},
  FJOURNAL = {American Journal of Mathematics},
    VOLUME = {97},
      YEAR = {1975},
    NUMBER = {4},
     PAGES = {1061--1083},
      ISSN = {0002-9327,1080-6377},
   MRCLASS = {46E35 (81.47)},
  MRNUMBER = {420249},
MRREVIEWER = {R.\ H\o egh-Krohn},
       DOI = {10.2307/2373688},
       URL = {https://doi.org/10.2307/2373688},
}

@article {han-hardt-lin-1998-geometric,
    AUTHOR = {Han, Qing and Hardt, Robert and Lin, Fanghua},
     TITLE = {Geometric measure of singular sets of elliptic equations},
   JOURNAL = {Comm. Pure Appl. Math.},
  FJOURNAL = {Communications on Pure and Applied Mathematics},
    VOLUME = {51},
      YEAR = {1998},
    NUMBER = {11-12},
     PAGES = {1425--1443},
      ISSN = {0010-3640,1097-0312},
   MRCLASS = {35J15 (35D99 49Q20)},
  MRNUMBER = {1639155},
MRREVIEWER = {Giovanni\ Leoni},
       DOI =
              {10.1002/(SICI)1097-0312(199811/12)51:11/12<1425::AID-CPA8>3.3.CO;2-V},
       URL =
              {https://doi.org/10.1002/(SICI)1097-0312(199811/12)51:11/12<1425::AID-CPA8>3.3.CO;2-V},
}

@article {han-lin-1004-on-the-geometric,
    AUTHOR = {Han, Qing and Lin, Fang-Hua},
     TITLE = {On the geometric measure of nodal sets of solutions},
   JOURNAL = {J. Partial Differential Equations},
  FJOURNAL = {Journal of Partial Differential Equations},
    VOLUME = {7},
      YEAR = {1994},
    NUMBER = {2},
     PAGES = {111--131},
      ISSN = {1000-940X,2079-732X},
   MRCLASS = {35J60 (35B99)},
  MRNUMBER = {1280175},
MRREVIEWER = {Harold\ Donnelly},
}

@article{han-lin-1994-nodal-sets-ii,
    AUTHOR = {Han, Qing and Lin, Fang-Hua},
     TITLE = {Nodal sets of solutions of parabolic equations. {II}},
   JOURNAL = {Comm. Pure Appl. Math.},
  FJOURNAL = {Communications on Pure and Applied Mathematics},
    VOLUME = {47},
      YEAR = {1994},
    NUMBER = {9},
     PAGES = {1219--1238},
      ISSN = {0010-3640,1097-0312},
   MRCLASS = {35B05 (28A78 35K10)},
  MRNUMBER = {1290401},
MRREVIEWER = {Robert\ McOwen},
       DOI = {10.1002/cpa.3160470904},
       URL = {https://doi.org/10.1002/cpa.3160470904},
}

@article {hardt-hoffman-nadirashvili-1999-critical-sets,
    AUTHOR = {Hardt, R. and Hoffmann-Ostenhof, M. and Hoffmann-Ostenhof, T.
              and Nadirashvili, N.},
     TITLE = {Critical sets of solutions to elliptic equations},
   JOURNAL = {J. Differential Geom.},
  FJOURNAL = {Journal of Differential Geometry},
    VOLUME = {51},
      YEAR = {1999},
    NUMBER = {2},
     PAGES = {359--373},
      ISSN = {0022-040X,1945-743X},
   MRCLASS = {35J15 (35B65)},
  MRNUMBER = {1728303},
MRREVIEWER = {Rolando\ Magnanini},
       URL = {http://projecteuclid.org/euclid.jdg/1214425070},
}

@article {hardt-simon-1989-nodal,
    AUTHOR = {Hardt, Robert and Simon, Leon},
     TITLE = {Nodal sets for solutions of elliptic equations},
   JOURNAL = {J. Differential Geom.},
  FJOURNAL = {Journal of Differential Geometry},
    VOLUME = {30},
      YEAR = {1989},
    NUMBER = {2},
     PAGES = {505--522},
      ISSN = {0022-040X,1945-743X},
   MRCLASS = {58E05 (35J99)},
  MRNUMBER = {1010169},
MRREVIEWER = {Fang\ Hua\ Lin},
       URL = {http://projecteuclid.org/euclid.jdg/1214443599},
}

@article {hezari-sogge-2012-a-natural,
    AUTHOR = {Hezari, Hamid and Sogge, Christopher D.},
     TITLE = {A natural lower bound for the size of nodal sets},
   JOURNAL = {Anal. PDE},
  FJOURNAL = {Analysis \& PDE},
    VOLUME = {5},
      YEAR = {2012},
    NUMBER = {5},
     PAGES = {1133--1137},
      ISSN = {2157-5045,1948-206X},
   MRCLASS = {35P15 (35R01 58C40)},
  MRNUMBER = {3022851},
MRREVIEWER = {He-Jun\ Sun},
       DOI = {10.2140/apde.2012.5.1133},
       URL = {https://doi.org/10.2140/apde.2012.5.1133},
}

@article {HLN,
    AUTHOR = {Hofmann, Steve and Lewis, John L. and Nystr\"om, Kaj},
     TITLE = {Existence of big pieces of graphs for parabolic problems},
   JOURNAL = {Ann. Acad. Sci. Fenn. Math.},
  FJOURNAL = {Annales Academi\ae\ Scientiarum Fennic\ae. Mathematica},
    VOLUME = {28},
      YEAR = {2003},
    NUMBER = {2},
     PAGES = {355--384},
      ISSN = {1239-629X,1798-2383},
   MRCLASS = {35K05 (35K10)},
  MRNUMBER = {1996443},
MRREVIEWER = {Sergey\ G.\ Pyatkov},
}

@article {hofman-caloric,
    AUTHOR = {Hofmann, Steve and Lewis, John L. and Nystr\"om, Kaj},
     TITLE = {Caloric measure in parabolic flat domains},
   JOURNAL = {Duke Math. J.},
  FJOURNAL = {Duke Mathematical Journal},
    VOLUME = {122},
      YEAR = {2004},
    NUMBER = {2},
     PAGES = {281--346},
      ISSN = {0012-7094,1547-7398},
   MRCLASS = {35K05 (42B20)},
  MRNUMBER = {2053754},
MRREVIEWER = {Xuan\ Thinh\ Duong},
       DOI = {10.1215/S0012-7094-04-12222-5},
       URL = {https://doi-org.proxy.libraries.rutgers.edu/10.1215/S0012-7094-04-12222-5},
}

@article {hofmanPDE,
    AUTHOR = {Bortz, Simon and Hofmann, Steve and Martell, Jos\'e{} Mar\'ia
              and Nystr\"om, Kaj},
     TITLE = {Solvability of the {${\rm L}^p$} {D}irichlet problem for the
              heat equation is equivalent to parabolic uniform
              rectifiability in the case of a parabolic {L}ipschitz graph},
   JOURNAL = {Invent. Math.},
  FJOURNAL = {Inventiones Mathematicae},
    VOLUME = {239},
      YEAR = {2025},
    NUMBER = {1},
     PAGES = {165--217},
      ISSN = {0020-9910,1432-1297},
   MRCLASS = {35K05 (35K20 35R35 42B25 42B37)},
  MRNUMBER = {4841778},
MRREVIEWER = {Mariusz\ Ciesielski},
       DOI = {10.1007/s00222-024-01300-1},
       URL = {https://doi-org.proxy.libraries.rutgers.edu/10.1007/s00222-024-01300-1},
}

@article{huang-jiang-2023-volume,
  title={{Volume Estimates for Singular sets and Critical Sets of Elliptic Equations with H\"older Coefficients}},
  author={Huang, Yiqi and Jiang, Wenshuai},
  journal={arXiv preprint arXiv:2309.08089},
  pages={1-47},
  year={2023}
}

@article{huang-jiang-2024nodal,
  title={{The Nodal Sets of Solutions to Parabolic Equations}},
  author={Huang, Yiqi and Jiang, Wenshuai},
  journal={arXiv preprint arXiv:2406.05877},
  pages={1-44},
  year={2024}
}

@misc{huang-jiang-MCF,
      title={The Bounded Diameter Conjecture and Sharp Geometric Estimates for Mean Curvature Flow}, 
      author={Yiqi Huang and Wenshuai Jiang},
      year={2025},
      eprint={2510.17060},
      archivePrefix={arXiv},
      primaryClass={math.DG},
      url={https://arxiv.org/abs/2510.17060}, 
}

@article{jiang-naber-2021-l2-curvature,
  title={{$L^2$ curvature bounds on manifolds with bounded Ricci curvature}},
  author={Jiang, Wenshuai and Naber, Aaron},
  journal={Annals of Mathematics},
  volume={193},
  number={1},
  pages={107--222},
  year={2021},
  publisher={Department of Mathematics, Princeton University Princeton, New Jersey, USA}
}

@article {kenig-zhu-zhuge-2022-doubling,
    AUTHOR = {Kenig, Carlos E. and Zhu, Jiuyi and Zhuge, Jinping},
     TITLE = {Doubling inequalities and nodal sets in periodic elliptic
              homogenization},
   JOURNAL = {Comm. Partial Differential Equations},
  FJOURNAL = {Communications in Partial Differential Equations},
    VOLUME = {47},
      YEAR = {2022},
    NUMBER = {3},
     PAGES = {549--584},
      ISSN = {0360-5302,1532-4133},
   MRCLASS = {35A02 (35B27 35J15)},
  MRNUMBER = {4387203},
       DOI = {10.1080/03605302.2021.1989699},
       URL = {https://doi.org/10.1080/03605302.2021.1989699},
}

@book{lieberman-1996-second,
  title={Second order parabolic differential equations},
  author={Lieberman, Gary M},
  year={1996},
  publisher={World scientific}
}

@article {lin-1991-nodal-sets,
    AUTHOR = {Lin, Fang-Hua},
     TITLE = {Nodal sets of solutions of elliptic and parabolic equations},
   JOURNAL = {Comm. Pure Appl. Math.},
  FJOURNAL = {Communications on Pure and Applied Mathematics},
    VOLUME = {44},
      YEAR = {1991},
    NUMBER = {3},
     PAGES = {287--308},
      ISSN = {0010-3640,1097-0312},
   MRCLASS = {58G11 (35J05 35K05 58G03)},
  MRNUMBER = {1090434},
MRREVIEWER = {Robert\ McOwen},
       DOI = {10.1002/cpa.3160440303},
       URL = {https://doi.org/10.1002/cpa.3160440303},
}

@article{lin-liu2016-betti-numbers-of-level-sets,
    AUTHOR = {Lin, Fanghua and Liu, Dan},
     TITLE = {On the {B}etti numbers of level sets of solutions to elliptic
              equations},
   JOURNAL = {Discrete Contin. Dyn. Syst.},
  FJOURNAL = {Discrete and Continuous Dynamical Systems. Series A},
    VOLUME = {36},
      YEAR = {2016},
    NUMBER = {8},
     PAGES = {4517--4529},
      ISSN = {1078-0947,1553-5231},
   MRCLASS = {35J25 (35B05 35B35 82D55)},
  MRNUMBER = {3479524},
MRREVIEWER = {Siegfried\ Carl},
       DOI = {10.3934/dcds.2016.36.4517},
       URL = {https://doi.org/10.3934/dcds.2016.36.4517},
}

@article {lin-shen-2019-nodal-sets,
    AUTHOR = {Lin, Fanghua and Shen, Zhongwei},
     TITLE = {Nodal sets and doubling conditions in elliptic homogenization},
   JOURNAL = {Acta Math. Sin. (Engl. Ser.)},
  FJOURNAL = {Acta Mathematica Sinica (English Series)},
    VOLUME = {35},
      YEAR = {2019},
    NUMBER = {6},
     PAGES = {815--831},
      ISSN = {1439-8516,1439-7617},
   MRCLASS = {35B27 (35J15)},
  MRNUMBER = {3952693},
MRREVIEWER = {Taras\ A.\ Mel\cprime nyk},
       DOI = {10.1007/s10114-019-8228-5},
       URL = {https://doi.org/10.1007/s10114-019-8228-5},
}

@article {liu-tian-yang-2024-measure,
    AUTHOR = {Liu, Fang and Tian, Long and Yang, Xiaoping},
     TITLE = {Measure upper bounds of nodal sets of {R}obin eigenfunctions},
   JOURNAL = {Math. Z.},
  FJOURNAL = {Mathematische Zeitschrift},
    VOLUME = {306},
      YEAR = {2024},
    NUMBER = {1},
     PAGES = {Paper No. 14, 14},
      ISSN = {0025-5874,1432-1823},
   MRCLASS = {35J05 (35P05 58E10)},
  MRNUMBER = {4675263},
       DOI = {10.1007/s00209-023-03409-0},
       URL = {https://doi.org/10.1007/s00209-023-03409-0},
}

@article{logunov-2018-nodal,
  title={{Nodal sets of Laplace eigenfunctions: polynomial upper estimates of the Hausdorff measure}},
  author={Logunov, Alexander},
  journal={Annals of Mathematics},
  volume={187},
  number={1},
  pages={221--239},
  year={2018},
  publisher={Department of Mathematics, Princeton University Princeton, New Jersey, USA}
}

@article {logunov-2018-nodal-sets,
    AUTHOR = {Logunov, Alexander},
     TITLE = {Nodal sets of {L}aplace eigenfunctions: proof of
              {N}adirashvili's conjecture and of the lower bound in {Y}au's
              conjecture},
   JOURNAL = {Ann. of Math. (2)},
  FJOURNAL = {Annals of Mathematics. Second Series},
    VOLUME = {187},
      YEAR = {2018},
    NUMBER = {1},
     PAGES = {241--262},
      ISSN = {0003-486X,1939-8980},
   MRCLASS = {58J50 (35J05 35P15 35P20 35R01)},
  MRNUMBER = {3739232},
MRREVIEWER = {Leonid\ Friedlander},
       DOI = {10.4007/annals.2018.187.1.5},
       URL = {https://doi.org/10.4007/annals.2018.187.1.5},
}

@incollection {logunov-malinnikova-2018-nodal-sets,
    AUTHOR = {Logunov, Alexander and Malinnikova, Eugenia},
     TITLE = {Nodal sets of {L}aplace eigenfunctions: estimates of the
              {H}ausdorff measure in dimensions two and three},
 BOOKTITLE = {50 years with {H}ardy spaces},
    SERIES = {Oper. Theory Adv. Appl.},
    VOLUME = {261},
     PAGES = {333--344},
 PUBLISHER = {Birkh\"auser/Springer, Cham},
      YEAR = {2018},
      ISBN = {978-3-319-59077-6; 978-3-319-59078-3},
   MRCLASS = {35R01 (31B05 35P05 58J50)},
  MRNUMBER = {3792104},
MRREVIEWER = {Sugata\ Mondal},
}

@incollection {logunov-malinnikova-2020-review,
    AUTHOR = {Logunov, Alexander and Malinnikova, Eugenia},
     TITLE = {Review of {Y}au's conjecture on zero sets of {L}aplace
              eigenfunctions},
 BOOKTITLE = {Current developments in mathematics 2018},
     PAGES = {179--212},
 PUBLISHER = {Int. Press, Somerville, MA},
      YEAR = {[2020] \copyright 2020},
      ISBN = {978-1-57146-387-6},
   MRCLASS = {58J50 (35P30)},
  MRNUMBER = {4363378},
}

@article {logunov-malinnikova-2021-the-sharp,
    AUTHOR = {Logunov, A. and Malinnikova, E. and Nadirashvili, N. and
              Nazarov, F.},
     TITLE = {The sharp upper bound for the area of the nodal sets of
              {D}irichlet {L}aplace eigenfunctions},
   JOURNAL = {Geom. Funct. Anal.},
  FJOURNAL = {Geometric and Functional Analysis},
    VOLUME = {31},
      YEAR = {2021},
    NUMBER = {5},
     PAGES = {1219--1244},
      ISSN = {1016-443X,1420-8970},
   MRCLASS = {35B05 (31B05 35J05)},
  MRNUMBER = {4356702},
       DOI = {10.1007/s00039-021-00581-5},
       URL = {https://doi.org/10.1007/s00039-021-00581-5},
}

@book {MattilaTextbook,
    AUTHOR = {Mattila, Pertti},
     TITLE = {Geometry of sets and measures in {E}uclidean spaces},
    SERIES = {Cambridge Studies in Advanced Mathematics},
    VOLUME = {44},
      NOTE = {Fractals and rectifiability},
 PUBLISHER = {Cambridge University Press, Cambridge},
      YEAR = {1995},
     PAGES = {xii+343},
      ISBN = {0-521-46576-1; 0-521-65595-1},
   MRCLASS = {28A75 (49Q20)},
  MRNUMBER = {1333890},
MRREVIEWER = {Harold\ Parks},
       DOI = {10.1017/CBO9780511623813},
       URL = {https://doi.org/10.1017/CBO9780511623813},
}

@article{mattila-2022-parabolic-rectifiability,
  title = {{Parabolic Rectifiability, Tangent Planes and Tangent Measures}},
  author = {Mattila, Pertti},
  year = {2022},
  month = jun,
  journal = {Annales Fennici Mathematici},
  volume = {47},
  number = {2},
  pages = {855--884},
  issn = {2737-114X},
  doi = {10.54330/afm.119821},
  copyright = {Copyright (c) 2022 Annales Fennici Mathematici},
  langid = {english}
}

@article{metafune-Pallara-Priola-2002-spectrum,
  title={{Spectrum of Ornstein-Uhlenbeck operators in Lp spaces with respect to invariant measures}},
  author={Metafune, Giorgio and Pallara, Diego and Priola, Enrico},
  journal={Journal of Functional Analysis},
  volume={196},
  number={1},
  pages={40--60},
  year={2002},
  publisher={Elsevier}
}

@misc{naberreifenbergnotes,
      title={Lecture Notes on Rectifiable Reifenberg for Measures}, 
      author={Aaron Naber},
      year={2018},
      eprint={1812.07564},
      archivePrefix={arXiv},
      primaryClass={math.AP},
      url={https://arxiv.org/abs/1812.07564}, 
}

@article{naber-valtorta-2017-volume-estimtates-of-critical-sets-of-pde,
  title={{Volume estimates on the critical sets of solutions to elliptic PDEs}},
  author={Naber, Aaron and Valtorta, Daniele},
  journal={Communications on Pure and Applied Mathematics},
  volume={70},
  number={10},
  pages={1835--1897},
  year={2017},
  publisher={Wiley Online Library}
}

@article{naber-valtorta-2017-rectifiable-for-harmonic-maps,
    AUTHOR = {Naber, Aaron and Valtorta, Daniele},
     TITLE = {Rectifiable-{R}eifenberg and the regularity of stationary and
              minimizing harmonic maps},
   JOURNAL = {Ann. of Math. (2)},
  FJOURNAL = {Annals of Mathematics. Second Series},
    VOLUME = {185},
      YEAR = {2017},
    NUMBER = {1},
     PAGES = {131--227},
      ISSN = {0003-486X,1939-8980},
   MRCLASS = {58E20 (53C43)},
  MRNUMBER = {3583353},
MRREVIEWER = {Andreas\ Gastel},
       DOI = {10.4007/annals.2017.185.1.3},
       URL = {https://doi.org/10.4007/annals.2017.185.1.3},
}

@article {naber-valtorta-YM,
    AUTHOR = {Naber, Aaron and Valtorta, Daniele},
     TITLE = {Energy identity for stationary {Y}ang {M}ills},
   JOURNAL = {Invent. Math.},
  FJOURNAL = {Inventiones Mathematicae},
    VOLUME = {216},
      YEAR = {2019},
    NUMBER = {3},
     PAGES = {847--925},
      ISSN = {0020-9910,1432-1297},
   MRCLASS = {58E15 (53C07)},
  MRNUMBER = {3955711},
MRREVIEWER = {Graeme\ Wilkin},
       DOI = {10.1007/s00222-019-00854-9},
       URL = {https://doi-org.proxy.libraries.rutgers.edu/10.1007/s00222-019-00854-9},
}

@article{naber-valtorta-2024energy,
  title={{Energy Identity for Stationary Harmonic Maps}},
  author={Naber, Aaron and Valtorta, Daniele},
  journal={arXiv preprint arXiv:2401.02242},
  pages={1-101},
  year={2024}
}

@article {parabolicwhitney,
    AUTHOR = {Bortz, Simon and Hoffman, John and Hofmann, Steve and Luna
              Garc\'ia, Jos\'e{} Luis and Nystr\"om, Kaj},
     TITLE = {Carleson measure estimates for caloric functions and parabolic
              uniformly rectifiable sets},
   JOURNAL = {Anal. PDE},
  FJOURNAL = {Analysis \& PDE},
    VOLUME = {16},
      YEAR = {2023},
    NUMBER = {4},
     PAGES = {1061--1088},
      ISSN = {2157-5045,1948-206X},
   MRCLASS = {35K10 (28A75 28A78 42B37)},
  MRNUMBER = {4605204},
       DOI = {10.2140/apde.2023.16.1061},
       URL = {https://doi.org/10.2140/apde.2023.16.1061},
}

@article{poon-1996-unique-continuation,
    AUTHOR = {Poon, Chi-Cheung},
     TITLE = {{Unique continuation for parabolic equations}},
   JOURNAL = {Comm. Partial Differential Equations},
  FJOURNAL = {Communications in Partial Differential Equations},
    VOLUME = {21},
      YEAR = {1996},
    NUMBER = {3-4},
     PAGES = {521--539},
      ISSN = {0360-5302,1532-4133},
   MRCLASS = {35K10 (35B60)},
  MRNUMBER = {1387458},
MRREVIEWER = {H.\ J.\ Kuiper},
       DOI = {10.1080/03605309608821195},
       URL = {https://doi.org/10.1080/03605309608821195},
}

@article{simon-1986-compact,
  title={{Compact sets in the space $L^p (O, T; B)$}},
  author={Simon, Jacques},
  journal={Annali di Matematica pura ed applicata},
  volume={146},
  number={1},
  pages={65--96},
  year={1986},
  publisher={Springer}
}

@incollection {simongeneral,
    AUTHOR = {Simon, Leon},
     TITLE = {Isolated singularities of extrema of geometric variational
              problems},
 BOOKTITLE = {Harmonic mappings and minimal immersions ({M}ontecatini,
              1984)},
    SERIES = {Lecture Notes in Math.},
    VOLUME = {1161},
     PAGES = {206--277},
 PUBLISHER = {Springer, Berlin},
      YEAR = {1985},
      ISBN = {3-540-16040-X},
   MRCLASS = {58E15 (58E20)},
  MRNUMBER = {821971},
MRREVIEWER = {Harold\ Parks},
       DOI = {10.1007/BFb0075139},
       URL = {https://doi-org.proxy.libraries.rutgers.edu/10.1007/BFb0075139},
}

@article {sogge-zelditch-2012-lower-bounds,
    AUTHOR = {Sogge, Christopher D. and Zelditch, Steve},
     TITLE = {Lower bounds on the {H}ausdorff measure of nodal sets {II}},
   JOURNAL = {Math. Res. Lett.},
  FJOURNAL = {Mathematical Research Letters},
    VOLUME = {19},
      YEAR = {2012},
    NUMBER = {6},
     PAGES = {1361--1364},
      ISSN = {1073-2780,1945-001X},
   MRCLASS = {58C40 (28A78 35P15 35R01)},
  MRNUMBER = {3091613},
MRREVIEWER = {Nelia\ Charalambous},
       DOI = {10.4310/MRL.2012.v19.n6.a14},
       URL = {https://doi.org/10.4310/MRL.2012.v19.n6.a14},
}

@article{wongkew-1993-volumes,
  title={Volumes of tubular neighbourhoods of real algebraic varieties},
  author={Wongkew, Richard},
  journal={Pacific Journal of Mathematics},
  volume={159},
  number={1},
  pages={177--184},
  year={1993},
  publisher={Mathematical Sciences Publishers}
}

@incollection {zelditch-2013-eigenfunctions-and-nodal-sets,
    AUTHOR = {Zelditch, Steve},
     TITLE = {Eigenfunctions and nodal sets},
 BOOKTITLE = {Surveys in differential geometry. {G}eometry and topology},
    SERIES = {Surv. Differ. Geom.},
    VOLUME = {18},
     PAGES = {237--308},
 PUBLISHER = {Int. Press, Somerville, MA},
      YEAR = {2013},
      ISBN = {978-1-57146-269-5},
   MRCLASS = {58J50 (35P05 35R01 53Dxx)},
  MRNUMBER = {3087922},
       DOI = {10.4310/SDG.2013.v18.n1.a7},
       URL = {https://doi.org/10.4310/SDG.2013.v18.n1.a7},
}

@article {zelditch-2015-hausdorff,
    AUTHOR = {Zelditch, Steve},
     TITLE = {Hausdorff measure of nodal sets of analytic {S}teklov
              eigenfunctions},
   JOURNAL = {Math. Res. Lett.},
  FJOURNAL = {Mathematical Research Letters},
    VOLUME = {22},
      YEAR = {2015},
    NUMBER = {6},
     PAGES = {1821--1842},
      ISSN = {1073-2780,1945-001X},
   MRCLASS = {35R01 (35J05 35P15 58C40)},
  MRNUMBER = {3507264},
       DOI = {10.4310/MRL.2015.v22.n6.a15},
       URL = {https://doi.org/10.4310/MRL.2015.v22.n6.a15},
}

@article {kukavicaparabolic,
    AUTHOR = {Kukavica, Igor},
     TITLE = {Hausdorff measure of level sets for solutions of parabolic
              equations},
   JOURNAL = {Internat. Math. Res. Notices},
  FJOURNAL = {International Mathematics Research Notices},
      YEAR = {1995},
    NUMBER = {13},
     PAGES = {671--682},
      ISSN = {1073-7928,1687-0247},
   MRCLASS = {35K35 (35B99 47D06)},
  MRNUMBER = {1383945},
MRREVIEWER = {Paolo\ Acquistapace},
       DOI = {10.1155/S1073792895000389},
       URL = {https://doi-org.proxy.libraries.rutgers.edu/10.1155/S1073792895000389},
}

@article {AlVe,
    AUTHOR = {Alessandrini, Giovanni and Vessella, Sergio},
     TITLE = {Remark on the strong unique continuation property for
              parabolic operators},
   JOURNAL = {Proc. Amer. Math. Soc.},
  FJOURNAL = {Proceedings of the American Mathematical Society},
    VOLUME = {132},
      YEAR = {2004},
    NUMBER = {2},
     PAGES = {499--501},
      ISSN = {0002-9939,1088-6826},
   MRCLASS = {35B60 (35K10)},
  MRNUMBER = {2022375},
       DOI = {10.1090/S0002-9939-03-07142-9},
       URL = {https://doi-org.proxy.libraries.rutgers.edu/10.1090/S0002-9939-03-07142-9},
}

@article {Fernandez,
    AUTHOR = {Fernandez, F. J.},
     TITLE = {Unique continuation for parabolic operators. {II}},
   JOURNAL = {Comm. Partial Differential Equations},
  FJOURNAL = {Communications in Partial Differential Equations},
    VOLUME = {28},
      YEAR = {2003},
    NUMBER = {9-10},
     PAGES = {1597--1604},
      ISSN = {0360-5302,1532-4133},
   MRCLASS = {35K55 (35B05 35B60)},
  MRNUMBER = {2001174},
MRREVIEWER = {W.\ Watzlawek},
       DOI = {10.1081/PDE-120024523},
       URL = {https://doi-org.proxy.libraries.rutgers.edu/10.1081/PDE-120024523},
}
\end{document}